\documentclass[12pt]{amsart}   
 
\usepackage{graphicx,amsfonts,amssymb,stmaryrd,amscd,amsmath,latexsym,amsbsy,hyperref,eufrak}
\usepackage{tikz}

\newtheorem{theorem}{Theorem}[section]
\newtheorem{lemma}[theorem]{Lemma}
\newtheorem{proposition}[theorem]{Proposition}
\newtheorem{corollary}[theorem]{Corollary}
\theoremstyle{definition}
\newtheorem{definition}[theorem]{Definition}

\newtheorem{example}[theorem]{Example}
\newtheorem{exercise}[theorem]{Exercise}

\newtheorem{remark}[theorem]{Remark}

\newcommand{\bchi}{{\mbox{\boldmath$\chi$}}}

\newcommand{\Ext}{\text{Ext}}

\newcommand{\Tr}{\text{Tr}}
\newcommand{\id}{\text{id}}

\newcommand{\End}{\text{End}}
\newcommand{\Hom}{\text{Hom}}
\newcommand{\Ad}{\text{Ad}}

\newcommand{\Aut}{\text{Aut}}

\newcommand{\Rep}{\text{Rep}}

\newcommand{\diag}{\text{diag}}

\renewcommand{\O}{\mathcal{O}}

\newcommand{\g}{\mathfrak{g}}
\newcommand{\h}{\mathfrak{h}}
\newcommand{\n}{\mathfrak{n}}
\renewcommand{\b}{\mathfrak{b}}

\newcommand{\ben}{\begin{enumerate}}
\newcommand{\een}{\end{enumerate}}
\newcommand{\ad}{{\text{Ad}}}

\newcommand{\Vect}{{\text{Vect}}}

\theoremstyle{plain}

\newtheorem*{sol}{Solution}

\theoremstyle{definition}

\theoremstyle{remark}

\newcommand{\solu}[1]{\begin{sol}{\bf (\ref{#1})}}

\def\g{\mathfrak{g}}
\def\R{\mathbb{R}}
\def\u{\mathfrak{u}}

\def\Aut{\mathop{\mathrm{Aut}}\nolimits}
\def\h{\mathfrak{h}}

\def\ad{\mathop{\mathrm{ad}}\nolimits}
\def\Ad{\mathop{\mathrm{Ad}}\nolimits}
\def\Vect{\mathrm{Vect}}

\def\z{\mathfrak{z}}
\def\a{\mathfrak{a}}
\def\End{\mathrm{End}}
\def\Hom{\mathrm{Hom}}

\def\n{\mathfrak{n}}

\def\GL{\mathop{GL}}

\def\p{\mathfrak{p}}
\def\b{\mathfrak{b}}
\def\m{\mathfrak{m}}
\def\O{\mathcal{O}}

\def\Ext{\mathop{\mathrm{Ext}}\nolimits}

\def\k{\mathbf{k}}

\def\Rep{\mathop{\mathrm{Rep}}\nolimits}

\numberwithin{equation}{section}

\begin{document}

\title{Lie groups and Lie algebras}

\author{Pavel Etingof}

\maketitle 

\centerline{\bf To my father Ilya Etingof on his 95th birthday with admiration}

\tableofcontents

\section*{Introduction}

The purpose of {\bf group theory} is to give a mathematical treatment of {\bf symmetries}. For example, symmetries of a set of $n$ elements form the symmetric group $S_n$, and symmetries of a regular $n$-gon -- the dihedral group $D_n$. Likewise, {\bf Lie group theory} serves to give a mathematical treatment of {\bf continuous symmetries}, i.e., families of symmetries continuously depending on several real parameters. 

The theory of Lie groups was founded in the second half of the 19th century by the Norwegian mathematician {\bf Sophus Lie}, after whom it is named. It was then developed by many mathematicians over the last 150 years, and has numerous applications in mathematics and science, especially physics. 

A prototypical example of a Lie group is the group $SO(3)$ of rotational symmetries of the 2-dimensional sphere; in this case the parameters are the Euler angles $\phi,\theta,\psi$.  

It turns out that unlike ordinary parametrized curves and surfaces, Lie groups are determined by their linear approximation at the identity element. This leads to the notion of the {\bf Lie algebra} of a Lie group. This notion allows one to reformulate the theory of continuous symmetries in purely algebraic terms, which provides an extremely effective way of studying such symmetries. The goal of these notes is to give a detailed study of Lie groups and Lie algebras and interactions between them, with numerous examples. 

These notes are based on a year-long introductory course on Lie groups and Lie algebras given by the author at MIT in 2020-2021 (in particular, they contain no original material). The first half (Sections 1-26) corresponds to the first semester and follows rather closely the excellent book ``An introduction to Lie groups and Lie algebras" by A. Kirillov Jr. (\cite{K}), but also discusses some additional topics. Namely, after a brief review of geometry and topology of manifolds, it covers the basic theory of Lie groups and Lie algebras, including the three fundamental theorems of Lie theory (except the proof of the third theorem, which is given in the second half). Then it proceeds to nilpotent and solvable Lie algebras, theorems of Lie and Engel, representations of $\mathfrak{sl}_2$, enveloping algebras and the Poincar\'e-Birkhoff Witt theorem, free Lie algebras, the Baker-Campbell-Hausdorff formula, and concludes with a detailed study and classification of complex semisimple Lie algebras, their representations, and the Weyl character formula. 

The second half (starting with Section 27) covers representation theory of $GL_n$ and other classical groups, minuscule representations, spin representations and spin groups, representation theory of compact Lie groups (again following \cite{K}) and, more generally, compact topological groups, including existence of the Haar measure and the Peter-Weyl theorem. Then it discusses applications to quantum mechanics (a fairly complete treatment of the hydrogen atom) and proceeds to real forms of semisimple Lie algebras and groups, discussing the classification of such forms in terms of Vogan diagrams, maximal tori and maximal compact subgroups, the polar and Cartan decompositions, and classification of connected compact Lie groups and complex reductive groups. Then we discuss topology of Lie groups and homogeneous spaces (in particular, their cohomology rings), cohomology of Lie algebras, prove the third fundamental theorem of Lie theory and Ado's theorem on the existence of a faithful representation for a finite 
dimensional Lie algebra, and conclude with the study of Borel and parabolic subgroups, the flag manifold of a complex semisimple group and the Iwasawa decomposition for real groups. 

Some other sources covering the same material are \cite{E,FH,Hu,Kn}.

Each section roughly corresponds to one 80-minute lecture. Part I consists of 26 sections, which corresponds to a 1-semester course. Part II consists of 25 sections, 
to allow for a review of Part I. Also, a lot of material is contained in exercises, which are often provided with detailed hints. These exercises were assigned as homework problems.\footnote{During the first semester and at the beginning of the second one homework problems were also assigned 
from \cite{K}.} 

Finally, we note that Lie theory is an inherently synthetic subject. While the main 
technical tools ultimately boil down to various parts of algebra (notably linear algebra and the theory of noncommutative rings and modules, and, at more advanced stages, algebraic geometry), Lie theory 
also relies in important ways on analysis, differential equations, differential geometry and topology. Thus, while we try to recall basic notions from these subjects along the way, the reader 
will need some degree of dexterity with them, which increases as we dig deeper into the material. 

{\bf Acknowledgments.} I'd like to thank David Vogan for inspiring me to write these notes and useful comments, and the students of the MIT courses ``Lie groups and Lie algebras, I,II" for feedback. I am especially grateful to Frank Wang and Atticus Wang for careful reading and many corrections to parts I and II, respectively. This work was partially supported by the NSF grant DMS-2001318.
The latest version was cleaned up using ChatGPT-5.5. 

\newpage

\centerline{\Large\bf Lie groups and Lie algebras, I}

\section{\bf Manifolds} 

\subsection{Topological spaces and groups} Recall that the mathematical notion responsible for describing continuity is that of a {\bf topological space}. Thus, to describe continuous symmetries, we should put this notion together with the notion of a group. This leads to the concept of a {\bf topological group}.

Recall:

$\bullet$ A {\bf topological space} is a set $X$, certain subsets of which (including $\emptyset$ and $X$) are declared to be {\bf open}, so that an arbitrary union and finite intersection of open sets is open. 

$\bullet$ The collection of open sets in $X$ is called {\bf the topology} of $X$. 

$\bullet$ A subset $Z\subset X$ of a topological space $X$ is {\bf closed} if its complement is open.

$\bullet$ If $X,Y$ are topological spaces then the Cartesian product $X\times Y$ has a natural {\bf product topology} in which open sets are (possibly infinite) unions of products $U\times V$, where $U\subset X,V\subset Y$ are open. 

$\bullet$ Every subset $Z\subset X$ of a topological space $X$ carries a natural {\bf induced topology}, in which open sets are intersections of open sets in $X$ with $Z$. 

$\bullet$ A map $f: X\to Y$ between topological spaces is {\bf continuous} if for every open set $V\subset Y$, the preimage $f^{-1}(V)$ is open in $X$.  

For example, the open sets of the usual topology of the real line $\Bbb R$ 
are (disjoint) unions of open intervals $(a,b)$, where $-\infty\le a<b\le \infty$. 

\begin{definition} A {\bf topological group} is a group $G$ which is also a topological space, so that the multiplication map $m: G\times G\to G$ and the inversion map $\iota: G\to G$ are continuous. 
\end{definition} 

For example, the group $(\Bbb R,+)$ of real numbers with the operation of addition and the usual topology of $\Bbb R$ is a topological group, since the functions $(x,y)\mapsto x+y$ and $x\mapsto -x$ are continuous. Also a subgroup of a topological group is itself a topological group, so another example is rational numbers with addition, $(\Bbb Q,+)$. This last example is not a very good model for continuity, however, and shows that general topological groups are not very well behaved. Thus, we will focus on a special class of topological groups called {\bf Lie groups}. 

Lie groups are distinguished among topological groups by the property that as topological spaces they belong to a very special class called {\bf topological manifolds.} So we need to start with reviewing this notion. 

\subsection{Topological manifolds}

Recall:

$\bullet$ A {\bf neighborhood} of a point $x\in X$ in a topological space $X$ is an open set containing $x$. 

$\bullet$ A {\bf base} for a topological space $X$ is a collection $\mathcal B$ of open sets in $X$ such that for every neighborhood $U$ of a point $x\in X$ there exists a neighborhood $V\subset U$ of $x$ which belongs to $\mathcal B$. Equivalently, every open set in $X$ is a union of members of $\mathcal B$. 

For example, open intervals form a base of the usual topology of $\Bbb R$. 
Moreover, we may take only intervals whose endpoints have rational coordinates,
which gives a {\it countable} base for $\Bbb R$. Also if $X,Y$ are topological spaces with bases $\mathcal B_X,\mathcal B_Y$ then products $U\times V$, where $U\in \mathcal B_X,V\in \mathcal B_Y$, form a base of the product topology of $X\times Y$. 
Thus if $X$ and $Y$ have countable bases, so does $X\times Y$; in particular, $\Bbb R^n$ with its usual (product) topology has a countable base (boxes whose vertices have rational coordinates). 

$\bullet$ $X$ is {\bf Hausdorff} if any two distinct points 
have disjoint neighborhoods. 

$\bullet$ If $X$ is Hausdorff, we say that a sequence of points $x_n\in X, n\in \Bbb N$ {\bf converges} to $x\in X$ as $n\to \infty$ (denoted $x_n\to x$) if every neighborhood of $x$ contains almost all terms of this sequence. Then one also says that the {\bf limit} of $x_n$ is $x$ 
and writes $$\lim_{n\to \infty}x_n=x.$$ It is easy to show that the limit is unique when it exists. 
In a Hausdorff space with a countable base, a closed set is one that is closed under taking limits of sequences. 

$\bullet$ A Hausdorff space $X$ is {\bf compact} if every open cover $\lbrace U_\alpha,\alpha\in A\rbrace$ of $X$ (i.e., $U_\alpha\subset X$ for all $\alpha\in A$ and $X=\cup_{\alpha\in A}U_\alpha$) has a finite subcover. 

$\bullet$ A continuous map $f: X\to Y$ is a {\bf homeomorphism} if it is a bijection and $f^{-1}: Y\to X$ is continuous. 

\begin{definition} A Hausdorff topological space $X$ is said to be an {\bf $n$-dimensional topological manifold} if it has a countable base and is {\bf locally Euclidean}, i.e., for every $x\in X$ there is a neighborhood $U\subset X$ of $x$ and a continuous map $\phi: U\to \Bbb R^n$ such that $\phi: U\to \phi(U)$ is a homeomorphism and $\phi(U)\subset \Bbb R^n$ is open. 
\end{definition} 

 The second property is often formulated as the condition that $X$ is {\bf locally homeomorphic to $\Bbb R^n$}. 

It is true (although not immediately obvious) that if a nonempty open set in $\Bbb R^n$ is homeomorphic to one in $\R^m$ then $n=m$. Therefore, the number $n$ is uniquely determined by $X$ as long as $X\ne \emptyset$. It is called {\bf the dimension} of $X$.  
(By convention, $\emptyset$ is a manifold of any integer dimension). 

\begin{example}\label{e1} 1. Obviously $X=\Bbb R^n$ is an $n$-dimensional topological manifold: we can take $U=X$ and $\phi={\rm Id}$. 

2. An open subset of a topological manifold is itself a topological manifold of the same dimension. 

3. The circle $S^1\subset \Bbb R^2$ defined by the equation 
$x^2+y^2=1$ is a topological manifold: for example, the point $(1,0)$ has a neighborhood 
$U=S^1\setminus \lbrace (-1,0)\rbrace$ and a map $\phi: U\to \Bbb R$ given by the stereographic projection: 
$$
\phi(\theta)=\tan(\tfrac{\theta}{2}),\ -\pi<\theta<\pi 
$$
 and similarly for every other point. More generally, the sphere \linebreak $S^n\subset \Bbb R^{n+1}$ defined by the equation $x_0^2+...+x_n^2=1$ is a topological manifold, for the same reason. The stereographic projection for the 2-dimensional sphere is shown in the following picture. 

\includegraphics[scale=0.5]{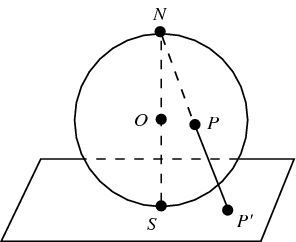}

4. The curve {\Huge$ \infty$}
is not a manifold, since it is not locally homeomorphic to $\Bbb R$ at the self-intersection point (show it!)
 \end{example}
 
A pair $(U,\phi)$ with the above properties is called a {\bf local chart}. An {\bf atlas} 
of local charts is a collection of charts $(U_\alpha,\phi_\alpha),\alpha\in A$ 
such that 
$\cup_{\alpha\in A}U_\alpha=X;$ 
i.e., $\lbrace U_\alpha,\alpha\in A\rbrace$ is an open cover of $X$. Thus any topological manifold $X$ admits an atlas labeled by points of $X$. There are also much smaller atlases. For instance, an open set in $\Bbb R^n$ has an atlas with just one chart, 
while the sphere $S^n$ has an atlas with two charts. Very often $X$ admits an atlas with finitely many charts. For example, if $X$ is compact  then there is a finite atlas, since every atlas has a finite subatlas. Moreover, there is always a countable atlas, due to the following lemma: 

\begin{lemma}\label{counsub} 
If $X$ is a topological space with a countable base then every open cover of $X$ 
has a countable subcover. 
\end{lemma}

\begin{proof}
Let $\lbrace V_i,i\in \Bbb N\rbrace $ be a countable base of $X$. If $\lbrace U_\alpha\rbrace$ 
is an open cover of $X$ then for each $x\in X$ pick indices $i(x)$ and $\alpha(x)$ such that $x\in V_{i(x)}\subset U_{\alpha(x)}$. Let $I\subset \Bbb N$ 
be the image of the map $i$. For each $j\in I$ pick $x\in X$ such that 
$i(x)=j$ and set $\alpha_j:=\alpha(x)$. Then $\lbrace U_{\alpha_j},j\in I\rbrace$ 
is a countable subcover of $\lbrace U_\alpha\rbrace$. 
\end{proof} 

Now let $(U,\phi)$ and $(V,\psi)$ be two charts such that $V\cap U\ne \emptyset$. 
Then we have the {\bf transition map} 
$$
\phi\circ \psi^{-1}: \psi(U\cap V)\to \phi(U\cap V),
$$
which is a homeomorphism between open subsets in $\Bbb R^n$. 
For example, consider the atlas of two charts for the circle $S^1$ (Example \ref{e1}(3)), one missing the point $(-1,0)$ and the other missing the point $(1,0)$. Then $\phi(\theta)=\tan({\theta\over 2})$ and 
$\psi(\theta)=\cot({\theta\over 2})$, $\phi(U\cap V)=\psi(U\cap V)=\Bbb R\setminus 0$, 
and $(\phi\circ \psi^{-1})(x)=\frac{1}{x}$. 

\subsection{$C^k$, real analytic and complex analytic manifolds}

The notion of topological manifold is too general for us, since continuous functions on which it is based in general do not admit a linear approximation. To develop the theory of Lie groups, we need more regularity. So we make the following definition. 

\begin{definition} An atlas on $X$ is said to be {\bf of regularity class $C^k$}, $1\le k\le \infty$, if all transition maps between its charts are of class $C^k$ ($k$ times continuously differentiable). An atlas of class $C^\infty$ is called {\bf smooth}. Also an atlas is said to be {\bf real analytic} if all transition maps are real analytic. Finally, if 
$n=2m$ is even, so that $\Bbb R^n=\Bbb C^m$, then an atlas is called {\bf complex analytic} if all its transition maps are complex analytic (i.e., holomorphic).  
\end{definition} 

\begin{example} The two-chart atlas for the circle $S^1$ defined by stereographic projections (Example \ref{e1}(3)) is real analytic, since the function $f(x)=\frac{1}{x}$ is analytic. The same applies to the sphere $S^n$ for any $n$. For example, for $S^2$ it is easy to see that the transition map $\Bbb R^2\setminus 0\to \Bbb R^2\setminus 0$ is given by the formula 
$$
f(x,y)=\left(\frac{x}{x^2+y^2},\frac{y}{x^2+y^2}\right).
$$
Using the complex coordinate $z=x+iy$, we get 
$$
f(z)=z/|z|^2=1/\overline{z}.
$$
So this atlas is not complex analytic. But it can be easily made complex analytic by 
replacing one of the stereographic projections ($\phi$ or $\psi$) by its complex conjugate. Then we will have $f(z)=\frac{1}{z}$. On the other hand,
it is known (although hard to prove) that $S^n$ admits no complex structure for even $n\neq 2,6$; the case $n=6$ is the famous open problem.
\end{example} 
 
\begin{definition} Two $C^k$, real analytic, or complex analytic atlases $U_\alpha,V_\beta$ are said to be {\bf compatible} if the transition maps between $U_\alpha$ and $V_\beta$ are of the same class ($C^k$, real analytic, or complex analytic). 
\end{definition} 

It is clear that compatibility is an equivalence relation. 

\begin{definition} A {\bf $C^k$, real analytic, or complex analytic structure} on a topological manifold $X$ is an equivalence class of $C^k$, real analytic, or complex analytic atlases. If $X$ is equipped with such a structure, it is said to be a {\bf $C^k$, real analytic, or complex analytic manifold}. Complex analytic manifolds are also called {\bf complex manifolds}, and a $C^\infty$-manifold is also called {\bf smooth}. A {\bf diffeomorphism} (or {\bf isomorphism}) between such manifolds is a homeomorphism which respects the corresponding classes of atlases. 
\end{definition} 

\begin{remark} This is really a {\bf structure} and not a {\bf property}. For example, consider $X=\Bbb C$ and $Y=D\subset \Bbb C$ the open unit disk, with the usual complex coordinate $z$. It is easy to see that $X,Y$ are isomorphic as real analytic manifolds. But they are not isomorphic as complex analytic manifolds: a complex isomorphism would be a holomorphic function $f: \Bbb C\to D$, hence bounded, but by Liouville's theorem any bounded holomorphic function on $\Bbb C$ is a constant. Thus we have two different complex structures on $\Bbb R^2$ (Riemann showed that there are no others). Also, it is true, but much harder to show, that there are uncountably many different smooth structures on $\Bbb R^4$, and there are 28 (oriented) smooth structures on $S^7$. 
\end{remark} 

Note that the Cartesian product $X\times Y$ of manifolds $X,Y$ is naturally a manifold (of the same regularity type) of dimension $\dim X+\dim Y$.  

\begin{exercise}\label{exe1} Let $f_1,...,f_m$ be functions $\Bbb R^n\to \Bbb R$ which are $C^k$ or real analytic. Let $X\subset \Bbb R^n$ be the set of points $P$ such that $f_i(P)=0$ for all $i$ and $df_i(P)$ are linearly independent. Use the implicit function theorem to show that $X$ is a topological manifold of dimension $n-m$ and equip it with a natural $C^k$, respectively real analytic structure. Prove the analogous statement for holomorphic functions $\Bbb C^n\to \Bbb C$, namely that in this case $X$ is naturally a complex manifold of (complex) dimension $n-m$. 
\end{exercise}  

\subsection{Regular functions} 
Now let $P\in X$ and $(U,\phi)$ be a local chart such that $P\in U$ and $\phi(P)=0$. Such a chart is called a {\bf coordinate chart} around $P$. In particular, 
we have {\bf local coordinates} \linebreak $x_1,...,x_n: U\to \Bbb R$ (or $U\to \Bbb C$ for complex manifolds), which are just the components of $\phi$, i.e.,  $\phi(Q)=(x_1(Q),...,x_n(Q))$. Note that $x_i(P)=0$, and $x_i(Q)$ determine $Q$ if $Q\in U$. 

\begin{definition} A {\bf regular function} on an open set $V\subset X$ in a $C^k$, real analytic, or complex analytic manifold $X$ is a function \linebreak $f: V\to\Bbb R,\Bbb C$ such that $f\circ \phi_\alpha^{-1}: \phi_\alpha(V\cap U_\alpha)\to\Bbb R,\Bbb C$ is of the corresponding regularity class, for some (and then any) 
atlas $(U_\alpha,\phi_\alpha)$ defining the corresponding structure on $X$.\footnote{More precisely, for $C^k$ and real analytic manifolds regular functions will
be assumed real-valued, unless specified otherwise. In the complex analytic case
there is, of course, no choice, and regular functions are automatically complex-valued.}
\end{definition} 

In other words, $f$ is regular if it is expressed as a regular function in local coordinates near every point of $V$. Clearly, this is independent of  the choice of coordinates. 

The space (in fact, algebra) of regular functions on $V$ will be denoted by $O(V)$. 

\begin{definition} Let $V,U$ be neighborhoods of $P\in X$. 
Let us say that $f\in O(V)$, $g\in O(U)$ are {\bf equal near $P$} if there exists 
a neighborhood $W\subset U\cap V$ of $P$ such that $f|_W=g|_W$. 
\end{definition} 

It is clear that this is an equivalence relation. 

\begin{definition} A {\bf germ} of a regular function at $P$ is an equivalence class of 
regular functions defined on neighborhoods of $P$ which are equal near $P$. 
\end{definition} 

The algebra of germs of regular functions at $P$ is denoted by $O_P$. Thus we have 
$O_P=\underrightarrow{\lim}\ O(U)$, where the direct limit is taken over neighborhoods of $P$. 

\subsection{Tangent spaces} 
From now on we will only consider smooth, real analytic and complex analytic manifolds. By a {\bf derivation at $P$} we will mean a linear map 
$D: O_P\to \Bbb R$ in the smooth and real analytic case and $D: O_P\to \Bbb C$ in the complex analytic case, satisfying the Leibniz rule 
\begin{equation}\label{leib}
D(fg)=D(f)g(P)+f(P)D(g). 
\end{equation}
Note that for any such $D$ we have $D(1)=0$. 

Let $T_PX$ be the space of all such derivations. 
Thus $T_PX$ is a real vector space for smooth and real analytic manifolds and 
a complex vector space for complex manifolds. 

\begin{lemma}\label{l1} Let $x_1,...,x_n$ be local coordinates at $P$. Then $T_PX$ 
has basis $D_1,...,D_n$, where 
$$
D_i(f):=\frac{\partial f}{\partial x_i}(0).
$$
\end{lemma} 

\begin{proof} We may assume $X=\Bbb R^n$ or $\Bbb C^n$, $P=0$. Clearly, $D_1,...,D_n$ is a linearly independent set in $T_PX$. 
Also let $D\in T_PX$, $D(x_i)=a_i$, and consider $D_*:=D-\sum_i a_iD_i$. 
Then $D_*(x_i)=0$ for all $i$. Now given a regular function $f$ near $0$, 
for small $x_1,...,x_n$ by the fundamental theorem of calculus and the chain rule we have: 
$$
f(x_1,...,x_n)=f(0)+\int_0^1 \frac{df(tx_1,...,tx_n)}{dt}dt=f(0)+\sum_{i=1}^nx_ih_i(x_1,...,x_n), 
$$
where 
$$
h_i(x_1,...,x_n):=\int_0^1(\partial_if)(tx_1,...,tx_n)dt
$$
are regular near $0$. So by the Leibniz rule 
$$
D_*(f)=\sum_i D_*(x_i)h_i(0,...,0)=0,
$$ 
hence $D_*=0$. 
\end{proof} 

\begin{definition} The space $T_PX$ is called the {\bf tangent space} 
to $X$ at $P$. Elements $v\in T_PX$ are called {\bf tangent vectors} to $X$ at $P$.
\end{definition} 

Observe that every tangent vector $v\in T_PX$ defines a derivation 
$\partial_v: O(U)\to \Bbb R,\Bbb C$ for every neighborhood $U$ of $P$, satisfying \eqref{leib}. The number $\partial_vf$ is called the {\bf derivative of $f$ along $v$}. For usual curves and surfaces in $\Bbb R^3$ these coincide with the familiar notions from calculus. 

\subsection{Regular maps} 

\begin{definition} A continuous map $F: X\to Y$ between manifolds (of the same regularity class) is {\bf regular} if for any regular function $h$ on an open set $U\subset Y$ the function $h\circ F$ on $F^{-1}(U)$ is regular. In other words, $F$ is regular if it is expressed by regular functions in local coordinates.
\end{definition} 

It is easy to see that the composition of regular maps is regular, and that a homeomorphism $F$ such that $F,F^{-1}$ are both regular is the same thing as a diffeomorphism (=isomorphism). 
 
Let $F: X\to Y$ be a regular map and $P\in X$. Then we can define the {\bf differential}  
of $F$ at $P$, $d_PF$, which is a linear map $T_PX\to T_{F(P)}Y$. 
Namely, for $f\in O_{F(P)}$ and $v\in T_PX$, the vector $d_PF\cdot v$ is defined by the formula 
$$
(d_PF\cdot v)(f):=v(f\circ F). 
$$
The differential of $F$ is also denoted by $F_*$; namely, for $v\in T_PX$ 
one writes $dF_P\cdot v=F_*v$. 

Moreover, if $G: Y\to Z$ is another regular map, then we have the usual {\bf chain rule}, 
$$
d(G\circ F)_P=dG_{F(P)}\circ dF_P. 
$$
In particular, if $\gamma: (a,b)\to X$ is a regular {\bf parametrized curve} then 
for $t\in (a,b)$ we can define the {\bf velocity vector} $\gamma'(t)\in T_{\gamma(t)}X$ by
$$
\gamma'(t):=d_t\gamma\cdot 1   
$$
(where $1\in \Bbb R=T_t(a,b)$). 
 
\subsection{Submersions and immersions, submanifolds}  

\begin{definition} A regular map of manifolds $F: X\to Y$ is a {\bf submersion} if 
$dF_P: T_PX\to T_{F(P)}Y$ is surjective for all $P\in X$.    
\end{definition} 

The following proposition is a version of the implicit function theorem for manifolds. 

\begin{proposition}\label{ift} If $F$ is a submersion then for any $Q\in Y$, $F^{-1}(Q)$ 
is a manifold of dimension $\dim X-\dim Y$. 
\end{proposition} 

\begin{proof} This is a local question, so it reduces to the case when $X,Y$ are open subsets in Euclidean spaces. In this case it reduces to Exercise \ref{exe1}. 
\end{proof} 

\begin{definition} A regular map of manifolds $F: X\to Y$ is an {\bf immersion} if $d_PF: T_PX\to T_{F(P)}Y$ is injective for all $P\in X$. 
\end{definition} 

\begin{example} The inclusion of the sphere $S^n$ into $\Bbb R^{n+1}$ is an immersion. The map 
$F: S^1\to \Bbb R^2$ given by 
\begin{equation}\label{lemn}
x(\theta)=\frac{\cos \theta}{1+\sin^2 \theta},\ y(\theta)=\frac{\sin\theta\cos\theta}{1+\sin^2\theta}
\end{equation}
is also an immersion; its image is the lemniscate (shaped as {\Huge$ \infty$}). This shows that an immersion need not be injective. On the other hand, the map $F: \Bbb R\to \Bbb R^2$ given by $F(t)=(t^2,t^3)$ parametrizing a semicubic parabola {\Huge $\prec$} is injective, but not an immersion, since $F'(0)=(0,0)$. 
\end{example} 

\begin{definition} An immersion $F: X\to Y$ is an {\bf embedding} if the map $F: X\to F(X)$ is a homeomorphism (where $F(X)$ is equipped with the induced topology from $Y$). In this case, $F(X)\subset Y$ is said to be an {\bf (embedded) submanifold.}\footnote{Recall that a subset $Z$ of a topological space $X$ 
is called {\bf locally closed} if it is a closed subset in an open subset $U\subset X$. It is clear that embedded submanifolds are locally closed. For this reason they are often called locally closed (embedded) submanifolds.}
\end{definition} 

\begin{example} The immersion of $S^n$ into $\Bbb R^{n+1}$ and of $(0,1)$ into $\Bbb R$ are embeddings, but the parametrization of the lemniscate by the circle given by \eqref{lemn} is not. The parametrization of the curve {\Huge $\rho$} by $\Bbb R$ is also not an embedding; it is injective but the inverse is not continuous. 
\end{example} 

\begin{definition} An embedding $F: X\to Y$ of manifolds is {\bf closed} if
$F(X)\subset Y$ is a closed subset. In this case we say that $F(X)$ is a {\bf closed (embedded) submanifold} of $Y$. 
\end{definition} 

\begin{example} The embedding of $S^n$ into $\Bbb R^{n+1}$ is closed but of $(0,1)$ into $\Bbb R$ is not.  Also in Proposition \ref{ift}, $f^{-1}(Q)$ is a closed submanifold of $X$. 
\end{example}

\section{\bf Lie groups, I} 

\subsection{The definition of a Lie group} 
\begin{definition} A $C^k$, real or complex analytic {\bf Lie group} is a manifold $G$ of the same class, with a group structure such that the multiplication map $m: G\times G\to G$ is regular. 
\end{definition} 

Thus, in a Lie group $G$ for any $g\in G$ the left and right translation maps $L_g,R_g: G\to G$, $L_g(x):=gx,R_g(x):=xg$, are diffeomorphisms. 

\begin{proposition} In a Lie group $G$, the inversion map $\iota: G\to G$ is a diffeomorphism, and $d\iota_1=-{\rm Id}$. 
\end{proposition}

\begin{proof} For the first statement it suffices to show that $\iota$ is regular near $1$, the rest follows by translation. So let us pick a coordinate chart near $1\in G$ and 
write the map $m$ in this chart in local coordinates. Note that in these coordinates, $1\in G$ corresponds to $0\in \Bbb R^n$. 
Since $m(x,0)=x$ and $m(0,y)=y$, the linear approximation of $m(x,y)$ at $0$ is $x+y$. Thus by the implicit function theorem, the equation $m(x,y)=0$ is solved near $0$ by a regular function $y=\iota(x)$ with $d\iota(0)=-{\rm Id}$. This proves the proposition.  
\end{proof} 

\begin{remark} A {\bf $C^0$ Lie group} is a topological group which is a topological manifold. The {\bf Hilbert 5th problem} was to show that any such group is actually a real analytic Lie group (i.e., the regularity class does not matter). This problem is solved by the deep {\bf Gleason-Yamabe theorem}, proved in 1950s. So from now on we will not pay attention to regularity class and consider only real and complex Lie groups. 
\end{remark} 

Note that any complex Lie group of dimension $n$ 
can be regarded as a real Lie group of dimension $2n$. Also 
the Cartesian product of real (complex) Lie groups is a real (complex) Lie group. 

\subsection{Homomorphisms} 

\begin{definition} A {\bf homomorphism of Lie groups} $f: G\to H$ is a group homomorphism which is also a regular map. An {\bf isomorphism of Lie groups} is a homomorphism $f$ which is a group isomorphism, such that $f^{-1}: H\to G$ is regular. 
\end{definition}

We will see later that the last condition is in fact redundant. 

\subsection{Examples} \label{exalie} 

\begin{example} 1. $(\Bbb R^n,+)$ is a real Lie group and $(\Bbb C^n,+)$ is a complex Lie group (both $n$-dimensional). 

2. $(\Bbb R^\times,\times)$, $(\Bbb R_{>0},\times)$ are real Lie groups, 
$(\Bbb C^\times,\times)$ is a complex Lie group (all $1$-dimensional).

3. $S^1=\lbrace z\in \Bbb C: |z|=1\rbrace$ is a 1-dimensional real Lie group under multiplication of complex numbers. 

Note that $\Bbb R^\times\cong \Bbb R_{>0}\times \Bbb Z/2$,  $\Bbb C^\times \cong \Bbb R_{>0}\times S^1$ as real Lie groups (trigonometric form of a complex number) and $(\Bbb R,+)\cong (\Bbb R_{>0},\times)$ via $x\mapsto e^x$. 

4. The groups of invertible $n$ by $n$ matrices: $GL_n(\Bbb R)$ is a real Lie group and $GL_n(\Bbb C)$ is a complex Lie group. These are open sets in the corresponding spaces of all matrices and have dimension $n^2$. 

5. $SU(2)$, the special unitary group of size $2$. This is the set of complex $2$-by-$2$ matrices $A$ such that 
$$
AA^\dagger=\bold 1,\ \det A=1.
$$
So writing 
$$
A=\left(\begin{matrix} a& b\\ c& d\end{matrix}\right),\quad A^\dagger=\left(\begin{matrix} \overline a& \overline c\\ \overline b& \overline d\end{matrix}\right), 
$$ 
we get 
$$
a\overline a+b\overline b=1,\ a\overline c+b\overline d=0,\ c\overline c+d\overline d=1.
$$
The second equation implies that $(c,d)=\lambda(-\overline b,\overline a)$. 
Then we have 
$$
1=\det A=ad-bc=\lambda(a\overline a+b\overline b)=\lambda,
$$
so $\lambda=1$. Thus $SU(2)$ is identified with the set of $(a,b)\in \Bbb C^2$ 
such that $a\overline a+b\overline b=1$. Writing $a=x+iy$, $b=z+it$, we have 
$$
SU(2)=\lbrace (x,y,z,t)\in \Bbb R^4: x^2+y^2+z^2+t^2=1\rbrace.
$$
Thus $SU(2)$ is a 3-dimensional real Lie group which as a manifold is the 3-dimensional sphere $S^3\subset \Bbb R^4$.

6. Any countable group $G$ with {\bf discrete topology} (i.e., such that every set is open)  
is a  (real and complex) Lie group. 
\end{example} 

\subsection{The connected component of $1$} 

Recall: 

$\bullet$ A topological space $X$ is {\bf path-connected} if for any $P,Q\in X$ there is a continuous map $x: [0,1]\to X$ such that $x(0)=P,x(1)=Q$ (such $x$ is called  {\bf a path connecting $P$ to $Q$}).  

$\bullet$ If $X$ is any topological space, then for $P\in X$ we can define its {\bf path-connected component} to be the set $X_P$ of $Q\in X$ for which there is a path connecting $P$ to $Q$. Then $X_P$ is the largest path-connected subset of $X$ containing $P$. Clearly, 
the relation that $Q$ belongs to $X_P$ is an equivalence relation, which 
splits $X$ into equivalence classes called {\bf path-connected components.}
The set of such components is denoted $\pi_0(X)$. 

$\bullet$ A topological space $X$ is {\bf connected} if the only subsets of $X$ that are both open and closed are $\emptyset$ and $X$. For $P\in X$, the {\bf connected component} 
of $X$ is the union $X^P$ of all connected subsets of $X$ containing $P$, which is obviously connected itself (so it is the largest connected subset of $X$ containing $P$). A path-connected space $X$ is always connected but not vice versa (the classic counterexample is the graph of the function $y=\sin(\frac{1}{x})$ together with the interval $[-1,1]$ of the $y$-axis); however, a connected {\it manifold} is path-connected (show it!), so for manifolds the notions of connected component and path-connected component coincide.  

$\bullet$ If $Y$ is a topological space, $X$ is a set and $p: Y\to X$ is a surjective map (i.e., $X=Y/\sim$ is the quotient of $Y$ by an equivalence relation) then $X$ acquires a topology called the {\bf quotient topology}, in which open sets are subsets $V\subset X$ such that $p^{-1}(V)$ is open. 

Now let $G$ be a real or complex Lie group, and $G^\circ$ the connected component of $1\in G$. Then the connected component of any $g\in G$ is $gG^\circ$. 

\begin{proposition} (i) $G^\circ$ is a normal subgroup of $G$.

(ii) $\pi_0(G)=G/G^\circ$ with quotient topology is a discrete and countable group. 
\end{proposition} 

\begin{proof} (i) Let $g\in G$, $a\in G^\circ$, and $x: [0,1]\to G$ be a path connecting $1$ to $a$. Then $gxg^{-1}$ is a path connecting $1$ to $gag^{-1}$, so $gag^{-1}\in G^\circ$, hence $G^\circ$ is normal. 

(ii) Since $G$ is a manifold, for any $g\in G$, there is a neighborhood of $g$ contained 
in $G_g=gG^\circ$. This implies that any coset of $G^\circ$ in $G$ is open, hence $G/G^\circ$ is discrete. Also $G/G^\circ$ is countable since $G$ has a countable base.  
\end{proof} 

Thus we see that any Lie group is an extension of a discrete countable group by a connected Lie group. This essentially reduces studying Lie groups to studying connected Lie groups. In fact, one can further reduce to simply connected Lie groups, which is done in the next subsections. 

\section{\bf Lie groups, II} 

\subsection{A crash course on coverings} Now we need to review some more topology. Let $X,Y$ be Hausdorff topological spaces, and \linebreak $p: Y\to X$ a continuous map. Then $p$ is called a {\bf covering} if every point $x\in X$ has a neighborhood $U$ such that $p^{-1}(U)$ is a union of disjoint open sets (called {\bf sheets} of the covering) each of which is mapped homeomorphically onto $U$ by $p$: 

\begin{center}
\includegraphics[scale=0.5]{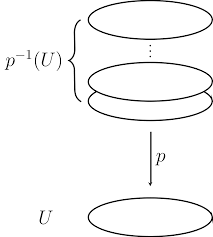}
\end{center}  

In other words, there exists a homeomorphism $h: U\times F\to p^{-1}(U)$ for some discrete space $F$ with $(p\circ h)(u,f)=u$ for all $u\in U$, $f\in F$. I.e., informally speaking, a covering is a map that locally on $X$ looks like the projection $X\times F\to X$ for some discrete $F$. 

We will consider only coverings with countable fibers, and just call them coverings. 
It is clear that a covering of a manifold ($C^k$, real or complex analytic) is a manifold 
of the same type, and the covering map is regular. 

\begin{center}
\includegraphics[scale=0.5]{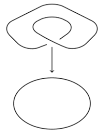}
\end{center}

Two paths $x_0,x_1: [0,1]\to X$ such that $x_i(0)=P, x_i(1)=Q$ are said to be {\bf homotopic} if  there is a continuous map $$x: [0,1]\times [0,1]\to X,$$ called a {\bf homotopy} between $x_0$ and $x_1$, such that $x(t,0)=x_0(t)$ and $x(t,1)=x_1(t)$, $x(0,s)=P,x(1,s)=Q$. See a movie here:

\noindent \url{https://commons.wikimedia.org/wiki/File:Homotopy.gif#/media/File:HomotopySmall.gif}

For example, if $x(t)$ is a path and $g: [0,1]\to [0,1]$ is a change of parameter with $g(0)=0$, $g(1)=1$ then the paths $x_1(t)=x(t)$ and $x_2(t)=x(g(t))$ are clearly homotopic. 

A path-connected Hausdorff space $X$ is said to be {\bf simply connected} if for any $P,Q\in X$, any paths $x_0,x_1: [0,1]\to X$ such that $x_i(0)=P, x_i(1)=Q$ are homotopic. 

\begin{example} $S^1$ is not simply connected but $S^n$ is simply connected for $n\ge 2$. 
\end{example} 
  
It is easy to show that any covering has a {\bf homotopy lifting property}: if $b\in X$ 
and $\widetilde b\in p^{-1}(b)\subset Y$ then any path $\gamma$ starting at $b$ 
admits a unique lift to a path $\widetilde \gamma$ starting at $\widetilde b$, i.e., 
$p(\widetilde \gamma)=\gamma$. Moreover, if $\gamma_1,\gamma_2$ are homotopic 
paths on $X$ then $\widetilde \gamma_1,\widetilde \gamma_2$ are homotopic on $Y$ 
(in particular, have the same endpoint). Thus, if $Z$ is a simply connected space with a point $z$ then any continuous map $f: Z\to X$ with $f(z)=b$ lifts to a unique continuous map $\widetilde f: Z\to Y$ satisfying $\widetilde f(z)=\widetilde b$; i.e., $p\circ \widetilde f=f$. Namely, to compute $\widetilde f(w)$, pick a path $\beta$ from $z$ to $w$, let $\gamma=f(\beta)$ 
and consider the path $\widetilde \gamma$. Then the endpoint of $\widetilde \gamma$ is 
$\widetilde f(w)$, and it does not depend on the choice of $\beta$. 

If $Z,X$ are manifolds (of any regularity type), $Z$ is simply connected, and 
$f: Z\to X$ is a regular map then the lift $\widetilde f: Z\to Y$ is also regular. Indeed, if we introduce local coordinates on $Y$ using the homeomorphism between sheets of the covering and their images then $\widetilde f$ and $f$ will be locally expressed by the same functions.

A covering $p: Y\to X$ of a path-connected space $X$ is called {\bf universal} if $Y$ is simply connected. 

\begin{center}  
\includegraphics[scale=0.5]{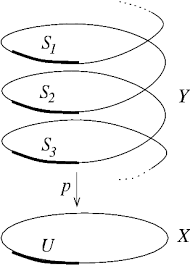}  
\end{center}

If $X$ is a sufficiently nice space, e.g., a manifold, its universal covering can be constructed as follows. Fix $b\in X$ and let $\widetilde X_b$ be the set of homotopy classes of paths on $X$ starting at $b$. We have a natural map $p: \widetilde X_b\to X$, $p(\gamma)=\gamma(1)$. If $U\subset X$ is a small ball around a point $x\in X$ then $U$ is simply connected, so we have a natural identification $h: U\times F\to p^{-1}(U)$ with $(p\circ h)(u,f)=u$, where $F=p^{-1}(x)$ is the set of homotopy classes of paths from $b$ to $x$; namely, $h(u,f)$ is the concatenation of $f$ with any path connecting $x$ with $u$ inside $U$. Here the {\bf concatenation} $\gamma_1\circ \gamma_2$ of paths $\gamma_1,\gamma_2:[0,1]\to X$ with $\gamma_2(1)=\gamma_1(0)$ is the path $\gamma=\gamma_1\circ \gamma_2: [0,1]\to X$ such that $\gamma(t)=\gamma_2(2t)$ for $t\le 1/2$ and $\gamma(t)=\gamma_1(2t-1)$ for $t\ge 1/2$. 

The topologies on all such $p^{-1}(U)$ induced by these identifications glue together into a topology on $\widetilde X_b$, and the map $p: \widetilde X_b\to X$ is then a covering. Moreover, the homotopy lifting property implies that $\widetilde X_b$ is simply connected, so this covering is universal. 

It is easy to see that a universal covering $p: Y\to X$ covers any path-connected covering $p': Y'\to X$, i.e., there is a covering $q: Y\to Y'$ such that $p=p'\circ q$; this is why it is called universal. Therefore a universal covering is unique up to an isomorphism (indeed, if 
$Y,Y'$ are universal then we have coverings $q_1: Y\to Y'$ and $q_2: Y'\to Y$ and $q_1\circ q_2=q_2\circ q_1={\rm Id}$). 

\begin{example} 1. The map $z\mapsto z^n$ defines an $n$-sheeted covering 
$S^1\to S^1$. 

2. The map $x\to e^{ix}$ defines the universal covering $\Bbb R\to S^1$. 
\end{example}  

Now denote by $\pi_1(X,x)$ the set of homotopy classes of {\it closed} paths on a path-connected space $X$, starting and ending at $x$. Then $\pi_1(X,x)$ is a group under concatenation of paths (concatenation is associative since the paths $a(bc)$ and $(ab)c$ differ only by parametrization and are hence homotopic). This group is called the {\bf fundamental group} of $X$ relative to the point $x$. 
It acts on the fiber $p^{-1}(x)$ for every covering $p: Y\to X$ (by lifting $\gamma\in \pi_1(X,x)$ to $Y$), which is called the action by {\bf deck transformations}. This action is transitive iff $Y$ is path-connected and moreover free iff Y is universal. 

Finally, the group $\pi_1(X,x)$ does not depend on $x$ up to an isomorphism. More precisely, conjugation by any path from $x_1$ to $x_2$ defines an isomorphism $\pi_1(X,x_1)\to \pi_1(X,x_2)$ (although two non-homotopic paths may define different isomorphisms 
if $\pi_1$ is non-abelian). 

\begin{example} 1. $\pi_1(S^1)=\Bbb Z$.

2. $\pi_1(\Bbb C\setminus \lbrace z_1,...,z_n\rbrace)={F}_n$ is a free group in $n$ generators. 

3. We have a 2-sheeted universal covering $S^n\to \Bbb R\Bbb P^n$ (real projective space) for $n\ge 2$. Thus $\pi_1(\Bbb R\Bbb P^n)=\Bbb Z/2$ for $n\ge 2$. 
\end{example}

\begin{exercise} Make sure you can fill all the details in this subsection!
\end{exercise} 

\subsection{Coverings of Lie groups} Let $G$ be a connected (real or complex) Lie group and $\widetilde G=\widetilde G_1$ be the universal covering of $G$, consisting of homotopy classes of paths $x: [0,1]\to G$ with $x(0)=1$. Then $\widetilde G$ is a group via $(x\cdot y)(t)=x(t)y(t)$, and also a manifold. 

\begin{proposition}\label{abeli} (i) $\widetilde G$ is a simply connected Lie group. 
The covering $p: \widetilde G\to G$ is a homomorphism of Lie groups. 

(ii) ${\rm Ker}(p)$ is a central subgroup of $\widetilde{G}$ naturally isomorphic to $\pi_1(G)=\pi_1(G,1)$. Thus, $\widetilde G$ is a central extension of $G$ by $\pi_1(G)$. 
In particular, $\pi_1(G)$ is abelian.  
\end{proposition} 

\begin{proof} We will only prove (i). 
We only need to show that $\widetilde{G}$ is a Lie group, i.e., 
that the multiplication map $\widetilde m: \widetilde G\times \widetilde G\to \widetilde G$ is regular. But $\widetilde G\times \widetilde G$ is simply connected, and $\widetilde m$ is a lifting of the map 
$$
m':=m\circ (p\times p): \widetilde G\times \widetilde G\to G\times G\to G,
$$
so it is regular.  
In other words, $\widetilde m$ is regular since in local coordinates it is defined 
by the same functions as $m$. 
\end{proof} 

\begin{exercise} Prove Proposition \ref{abeli}(ii). 
\end{exercise} 

\begin{remark} The same argument shows that more generally, the fundamental group of any path-connected topological group is abelian. 
\end{remark} 

\begin{example} 1. The map $z\mapsto z^n$ defines an $n$-sheeted covering 
of Lie groups $S^1\to S^1$. 

2. The map $x\to e^{ix}$ defines the universal covering of Lie groups $\Bbb R\to S^1$. 
\end{example}

\begin{exercise}\label{beltexe} Consider the action of $SU(2)$ on the 3-dimensional real vector space of traceless Hermitian 2-by-2 matrices by conjugation. 

(i) Show that this action preserves the positive inner product $(A,B)={\rm Tr}(AB)$ and has determinant $1$. Deduce that it defines a homomorphism $\phi: SU(2)\to SO(3)$.

(ii) Show that $\phi$ is surjective, with kernel $\pm 1$, and is a universal covering map (use that $SU(2)=S^3$ is simply connected). Deduce that $\pi_1(SO(3))=\Bbb Z/2$ and that $SO(3)\cong \Bbb R\Bbb P^3$ as a manifold. 
\end{exercise}

This is demonstrated by the famous {\bf Dirac belt trick}, which illustrates the notion of a {\bf spinor}; namely, spinors are vectors in $\Bbb C^2$ acted upon by matrices from $SU(2)$. Here are some videos of the belt trick: 

\url{https://www.youtube.com/watch?v=17Q0tJZcsnY}

\url{https://www.youtube.com/watch?v=Vfh21o-JW9Q}
  
\subsection{Closed Lie subgroups}

\begin{definition} A {\bf closed Lie subgroup} of a (real or complex) Lie group $G$ is a subgroup which is also an embedded submanifold. 
\end{definition} 

This terminology is justified by the following lemma. 

\begin{lemma}\label{l2} A closed Lie subgroup of $G$ is closed in $G$. 
\end{lemma} 

\begin{exercise} Prove Lemma \ref{l2}. 
\end{exercise} 

We also have 

\begin{theorem}\label{closedlie} Any closed subgroup of a real Lie group $G$ is a closed Lie subgroup. 
\end{theorem} 

This theorem is rather nontrivial, and we will not prove it at this time (it will be proved much later in Exercise \ref{closedlieex}), but we will soon prove a weaker version which suffices for our purposes. 

\begin{example}\label{wind} 1. $SL_n(\Bbb K)$ is a closed Lie subgroup of $GL_n(\Bbb K)$ for $\Bbb K=\Bbb R,\Bbb C$. 
Indeed, the equation $\det A=1$ defines a smooth hypersurface in the space of matrices (show it!). 

2. Let $\phi: \Bbb R\to S^1\times S^1$ be the irrational torus winding given by the formula 
$\phi(x)=(e^{ix},e^{ix\sqrt{2}})$:

\begin{center}
\includegraphics[scale=0.5]{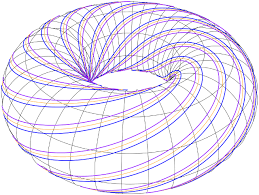}
\end{center} 

\noindent Then $\phi(\Bbb R)$ is a subgroup of $S^1\times S^1$ 
but not a closed Lie subgroup, since it is not an embedded submanifold: although $\phi$ is an immersion, the map $\phi^{-1}: \phi(\Bbb R)\to \Bbb R$ is not continuous. 
\end{example} 

\subsection{Generation of connected Lie groups by a neighborhood of the identity} 

\begin{proposition}\label{generati} (i) If $G$ is a connected Lie group and $U$ a neighborhood of $1$ in $G$ then $U$ generates $G$. 

(ii) If $f: G\to K$ is a homomorphism of Lie groups, $K$ is connected, 
and $df_1: T_1G\to T_1K$ is surjective, then $f$ is surjective.   
\end{proposition} 

\begin{proof} (i) Let $H$ be the subgroup of $G$ generated by $U$. Then $H$ is open in $G$ since $H=\cup_{h\in H}hU$. Thus $H$ is an embedded submanifold of $G$, hence a closed Lie subgroup. Thus by Lemma \ref{l2} $H\subset G$ is closed. So $H=G$ since $G$ is connected. 

(ii)  Since $df_1$ is surjective, by the implicit function theorem $f(G)$ contains some neighborhood of $1$ in $K$. Thus it contains the whole $K$ by (i). 
\end{proof} 

\section{\bf Homogeneous spaces, Lie group actions} 

\subsection{Homogeneous spaces} 
A regular map of manifolds $p: Y\to X$ is said to be a {\bf locally trivial fibration} (or {\bf fiber bundle}) with {\bf base} $X$, {\bf total space} $Y$ and {\bf fiber} being a manifold $F$ if every point $x\in X$ has a neighborhood $U$ such that there is a diffeomorphism $h: U\times F\cong p^{-1}(U)$ with $(p\circ h)(u,f)=u$. In other words, locally $p$ looks like the projection $X\times F\to X$ (the trivial fiber bundle with fiber $F$ over $X$), but not necessarily globally so. This generalizes the notion of a covering, in which  case $F$ is $0$-dimensional (discrete). 

\begin{theorem} (i) Let $G$ be a Lie group of dimension $n$ and $H\subset G$ a closed Lie subgroup of dimension $k$. Then the {\bf homogeneous space} $G/H$ has a natural structure of an $n-k$-dimensional manifold, and the canonical map $p: G\to G/H$, $g\mapsto \overline g$ is a locally trivial fibration with fiber $H$. 

(ii) If moreover $H$ is normal in $G$ then $G/H$ is a Lie group. 

(iii) We have a natural isomorphism $T_{\overline 1}(G/H)\cong T_1G/T_1H$. 
\end{theorem} 

\begin{proof} Let $\overline g\in G/H$ and $g\in p^{-1}(\overline g)$. Then $gH\subset G$ is an embedded submanifold (image of $H$ under left translation by $g$). Pick a sufficiently small transversal submanifold $U$ passing through $g$ (i.e., $T_gG=T_g(gH)\oplus T_gU$). 

\includegraphics[scale=0.6]{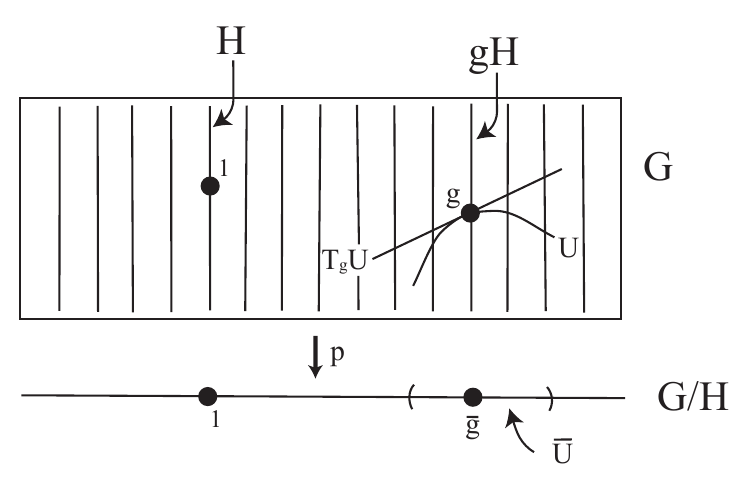}

By the inverse function theorem, the set $UH$ is open in $G$, and the multiplication map 
$U\times H\to UH$ is a diffeomorphism. Let $\overline U$ be the image of $UH$ in $G/H$. Since $p^{-1}(\overline U)=UH$ is open, $\overline U$ is open in the quotient topology. Also it is clear that $p: U\to \overline U$ is a homeomorphism. This defines a local chart near $\overline g\in G/H$, and it is easy to check that transition maps between such charts are regular. So $G/H$ acquires the structure of a manifold, which is easily checked to be independent of  the choices we made. Since the multiplication map $U\times H\to UH$ is a diffeomorphism, the map $p: G\to G/H$ is a locally trivial fibration with fiber $H$, which yields (i). If $H$ is normal then $G/H$ carries a natural group structure with regular multiplication map, so it is a Lie group, which proves (ii). Finally, we have a surjective linear map $T_gG\to T_{\overline g}G/H$ whose kernel is $T_g(gH)$.  So in particular for $g=1$ we get $T_{\overline{1}}(G/H)\cong T_1G/T_1H$, as claimed in (iii). 
\end{proof} 

Recall that a sequence of group homomorphisms $d_i: C^i\to C^{i+1}$ 
is a {\bf complex} if for all $i$, $d_{i}\circ d_{i-1}$ is the trivial homomorphism 
$C^{i-1}\to C^{i+1}$. (One may consider finite complexes, 
semi-infinite to the left or to the right, or infinite in both directions). 
In this case ${\rm Im}(d_{i-1})\subset {\rm Ker}(d_i)$ is a subgroup.
The $i$-th {\bf cohomology} $H^i(C^\bullet)$ of the complex $C^\bullet$ is the quotient ${\rm Ker}(d_i)/{\rm Im}(d_{i-1})$. In general it is just a set 
but if $C^i$ are abelian groups, it is also an abelian group. 
Also recall that a complex $C^\bullet$ is called {\bf exact} in the $i$-th term if 
${\rm Ker}(d_i)={\rm Im}(d_{i-1})$, i.e., if $H^i(C^\bullet)$ is trivial (consists of one element). A complex exact in all its terms (except possibly first and last, where this condition makes no sense) is called an {\bf exact sequence}. 

\begin{corollary}\label{c1} Let $H\subset G$ be a closed Lie subgroup.  

(i) If $H$ is connected then the map $p_0: \pi_0(G)\to \pi_0(G/H)$ is a bijection.

(ii) If also $G$ is connected then there is an exact sequence
$$
\pi_1(H)\to \pi_1(G)\to \pi_1(G/H)\to 1. 
$$
\end{corollary} 

\begin{proof} This follows from the theory of covering spaces using that 
\linebreak $p: G\to G/H$ is a fibration.  
\end{proof} 

\begin{exercise} Fill in the details in the proof of Corollary \ref{c1}. 
\end{exercise} 

\begin{remark}\label{longexx} The sequence in Corollary \ref{c1}(ii) is the end portion of the infinite 
{\bf long exact sequence of homotopy groups of a fibration}, 
$$
...\to\pi_i(H)\to \pi_i(G)\to \pi_i(G/H)\to \pi_{i-1}(H)\to...,
$$ 
where $\pi_i(X)$ is the $i$-th homotopy group of $X$. 
\end{remark} 

\subsection{Lie subgroups} We will call the image of an injective immersion of manifolds {\bf an immersed submanifold}; it has a manifold structure coming from the source of the immersion. 

\begin{definition} A {\bf Lie subgroup} of a Lie group $G$ is a subgroup $H$ which is also an {\it immersed} submanifold (but need not be an {\it embedded} submanifold, nor a closed subset). 
\end{definition} 

It is clear that in this case $H$ is still a Lie group and the inclusion $H\hookrightarrow G$ is a homomorphism of Lie groups. 

\begin{example} 1. The winding of a torus in Example \ref{wind}(2) realizes $\Bbb R$ as a Lie subgroup of $S^1\times S^1$ which is not closed.  

2. Any countable subgroup of $G$ is a $0$-dimensional Lie subgroup, but not always a closed one (e.g., $\Bbb Q\subset \Bbb R$). 
\end{example} 

\begin{proposition}\label{firsthom} Let $f: G\to K$ be a homomorphism of Lie groups. Then 
$H:={\rm Ker}f$ is a closed normal Lie subgroup in $G$ and ${\rm Im}f$ is a Lie subgroup in $K$, closed if and only if it is an embedded submanifold. In the latter case, we have an isomorphism of Lie groups $G/H\cong {\rm Im}f$. 
\end{proposition} 

We will prove Proposition \ref{firsthom} in Subsection \ref{proofs}. 

\subsection{Actions and representations of Lie groups} 

Let $X$ be a manifold, $G$ a Lie group, and $a: G\times X\to X$ a set-theoretical left action of $G$ on $X$. 

\begin{definition} This action is called {\bf regular} if the map $a$ is regular.  
\end{definition} 

From now on, by an action of $G$ on $X$ we will always mean a regular action. 

\begin{example} 1. Any Lie subgroup of $GL_n(\Bbb R)$ acts on $\Bbb R^n$ by linear transformations. Likewise, any Lie subgroup of $GL_n(\Bbb C)$ acts on $\Bbb C^n$. 

2. $SO(3)$ acts on $S^2$ by rotations. 
\end{example} 

\begin{definition} A (real analytic) {\bf finite dimensional representation} of a {\it real} Lie group $G$ is a linear action of $G$ on a finite dimensional vector space $V$ over $\Bbb  R$ or $\Bbb C$. Similarly, a (complex analytic) finite dimensional representation of a {\it complex} Lie group $G$ is a linear action of $G$ on a finite dimensional vector space $V$ over $\Bbb C$. 
\end{definition} 

In other words, a representation is a homomorphism of Lie groups $\pi_V: G\to GL(V)$. 

\begin{definition} 
A {\bf (homo)morphism of representations} (or {\bf intertwining operator}) $A: V\to W$ 
is a linear map which commutes with the $G$-action, i.e., 
$A\pi_V(g)=\pi_W(g)A$, $g\in G$. In particular, if $V=W$, such $A$ is called an {\bf endomorphism} of $V$. 
\end{definition} 

As usual, an {\bf isomorphism of representations} is an invertible morphism. With these definitions, finite dimensional representations of $G$ form a {\it category}. 

Note also that we have the operations of {\bf dual and tensor product on representations.} 
Namely, given a representation $V$ of $G$, we can define its representation  on the dual space $V^*$ by 
$$
\pi_{V^*}(g)=\pi_V(g^{-1})^*,
$$
and if $W$ is another representation of $G$ then we can define a representation of $G$ on $V\otimes W$ (the tensor product of vector spaces) by 
$$
\pi_{V\otimes W}(g)=\pi_V(g)\otimes \pi_W(g).
$$

Also if $V\subset W$ is a {\bf subrepresentation} (i.e., a subspace invariant under $G$) then 
$W/V$ is also a representation of $G$, called the {\bf quotient representation}.  

\subsection{Orbits and stabilizers} 

As in ordinary group theory, if $G$ acts on $X$ and $x\in X$ then we can define the {\bf orbit} $Gx\subset X$ of $x$ as the set of $gx$, $g\in G$, and the {\bf stabilizer}, or {\bf isotropy group} 
$G_x\subset G$ to be the group of $g\in G$ such that $gx=x$. 

\begin{proposition}\label{orstab} (The orbit-stabilizer theorem for Lie group actions) The stabilizer $G_x\subset G$ is a closed Lie subgroup, and the natural map $G/G_x\to X$ is an injective immersion whose image is $Gx$.  
\end{proposition} 

Proposition \ref{orstab} will be proved in Subsection \ref{proofs}.

\begin{corollary} The orbit $Gx\subset X$ is an immersed submanifold, and we have a natural isomorphism $T_x(Gx)\cong T_1G/T_1G_x$. If $Gx$ is an embedded submanifold then the map $G/G_x\to Gx$ is a diffeomorphism. 
\end{corollary} 

\begin{remark}Note that $Gx$ need not be closed in $X$. E.g., let $\Bbb C^\times$ act on $\Bbb C$ by multiplication. The orbit of $1$ is $\Bbb C^\times \subset \Bbb C$, which is not closed. 
\end{remark} 

\begin{example} Suppose that $G$ acts on $X$ transitively. Then we get 
that $X\cong G/G_x$ for any $x\in X$, i.e., $X$ is a {\bf homogeneous space}. 
\end{example} 

\begin{corollary} If $G$ acts transitively on $X$ then the map $p: G\to X$ given by $p(g)=gx$ is a locally trivial fibration with fiber $G_x$. 
\end{corollary}

\begin{example} 1. $SO(3)$ acts transitively on $S^2$ by rotations, $G_x=S^1=SO(2)$, so 
$S^2=SO(3)/S^1$. Thus $SO(3)=\Bbb R\Bbb P^3$ fibers over $S^2$ with fiber $S^1$. 

2. $SU(2)$ acts on $S^2=\Bbb C\Bbb P^1$, and the stabilizer is $S^1=U(1)$. 
Thus $SU(2)/S^1=S^2$, and $SU(2)=S^3$ fibers over $S^2$ with fiber $S^1$ (the {\bf Hopf fibration}). Here is D. Richter's keyring model of the Hopf fibration: 

\begin{center}
\includegraphics[scale=0.4]{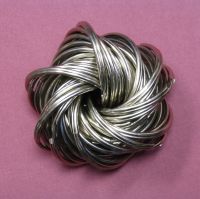}
\end{center}

3. Let $\Bbb K=\Bbb R$ or $\Bbb C$ and $\mathcal{F}_n(\Bbb K)$ the set of flags $0\subset V_1\subset...\subset V_n=\Bbb K^n$ ($\dim V_i=i$). Then $G=GL_n(\Bbb K)$ acts transitively on $\mathcal{F}_n(\Bbb K)$ (check it!). Also 
let $P\in \mathcal{F}_n(\Bbb K)$ be the flag for which $V_i=\Bbb K^i$ is the subspace of vectors whose all coordinates but the first $i$ are zero. Then $G_P$ is the subgroup $B_n(\Bbb K)\subset GL_n(\Bbb K)$ of invertible upper triangular matrices. Thus $\mathcal{F}_n(\Bbb K)=GL_n(\Bbb K)/B_n(\Bbb K)$ is a homogeneous space of $GL_n(\Bbb K)$, in particular, a $\Bbb K$-manifold. It is called the {\bf flag manifold}. 
\end{example} 

\subsection{Left translation, right translation, and adjoint action} 
Recall that a Lie group $G$ acts on itself by left translations $L_g(x)=gx$ and right translations $R_{g^{-1}}(x)=xg^{-1}$ (note that both are left actions). 

\begin{definition} The {\bf adjoint action} ${\rm Ad}_g: G\to G$ is the action 
${\rm Ad}_g=L_g\circ R_{g^{-1}}=R_{g^{-1}}\circ L_g$; i.e., ${\rm Ad}_g(x)=gxg^{-1}$. 
\end{definition} 

Note this is an action by (inner) automorphisms. Also since ${\rm Ad}_g(1)=1$, 
we have a linear map $d_1{\rm Ad}_g: \g\to \g$, where $\g=T_1G$. We will abuse notation and denote this map just by ${\rm Ad}_g$. This defines a representation of $G$ on $\g$ 
called the {\bf adjoint representation}.  

\section{\bf Tensor fields} 

\subsection{A crash course on vector bundles} Let $X$ be a real manifold. A {\bf vector bundle} on $X$ is, informally speaking, a (locally trivial) fiber bundle on $X$ whose fibers are finite dimensional vector spaces. In other words, it is a family of vector spaces parametrized by $x\in X$ and varying regularly with $x$. More precisely, we have the following definition. 

Let $\Bbb K=\Bbb R$ or $\Bbb C$.  

\begin{definition} A $\Bbb K$-{\bf vector bundle} of rank $n$ on $X$ is a manifold $E$ with a surjective regular map $p: E\to X$ and a $\Bbb K$-vector space structure on each fiber $p^{-1}(x)$ such that every $x\in X$ has a neighborhood $U$ admitting a diffeomorphism $g: U\times \Bbb K^n\to p^{-1}(U)$ with the following properties: 

(i) $(p\circ g)(u,v)=u$, and 

(ii) the map $g$ is $\Bbb K$-linear on the second factor. 
\end{definition} 

In other words, locally on $X$, $E$ is isomorphic to $X\times \Bbb K^n$, but 
not necessarily globally so.  

As for ordinary fiber bundles, $E$ is called the {\bf total  space} and $X$ the {\bf base} of the bundle. 

Note that even if $X$ is a complex manifold and $\Bbb K=\Bbb C$, $E$ need not be a complex manifold. 

\begin{definition} A complex vector bundle $p: E\to X$ on a complex manifold $X$ is said to be {\bf holomorphic} if $E$ is a complex manifold and the diffeomorphisms $g_U$ can be chosen holomorphic.  
\end{definition} 

 From now on, unless specified otherwise, all complex vector bundles on complex manifolds we consider will be holomorphic. 
 
It follows from the definition that if $p: E\to X$ is a vector bundle then 
$X$ has an open cover $\lbrace U_\alpha\rbrace$ such that $E$ trivializes on each $U_\alpha$, i.e., there is a diffeomorphism 
$g_\alpha: U_\alpha\times \Bbb K^n\to p^{-1}(U_\alpha)$ as above. In this case we have {\bf clutching functions} 
$$
h_{\alpha\beta}: U_\alpha\cap U_\beta\to GL_n(\Bbb K)
$$ 
(holomorphic if $E$ is a holomorphic bundle), defined by the formula
$$
(g_\alpha^{-1}\circ g_\beta)(x,v)=(x,h_{\alpha\beta}(x)v)
$$ 
which satisfy the {\bf consistency conditions} 
$$
h_{\alpha\beta}(x)=h_{\beta\alpha}(x)^{-1} 
$$
and 
$$
h_{\alpha\beta}(x)\circ h_{\beta\gamma}(x)=h_{\alpha\gamma}(x)
$$
for $x\in U_\alpha\cap U_\beta\cap U_\gamma$. Moreover, 
the bundle can be reconstructed from this data, starting from the disjoint union 
$\sqcup_\alpha U_\alpha\times \Bbb K^n$ and identifying (gluing) points according 
to 
$$
h_{\alpha\beta}: (x,v)\in U_\beta\times \Bbb K^n\sim (x,h_{\alpha\beta}(x)v)\in U_\alpha\times \Bbb K^n. 
$$ 
The consistency conditions ensure that the relation $\sim$ 
is symmetric and transitive, so it is an equivalence relation, 
and we define $E$ to be the space of equivalence classes 
with the quotient topology. Then $E$ has a natural structure of a vector bundle 
on $X$. 

This can also be used for constructing vector bundles. Namely, the above construction defines a $\Bbb K$-vector bundle on $X$ once we are given a cover $\lbrace U_\alpha\rbrace$ on $X$ and a collection of clutching functions 
$$
h_{\alpha\beta}: U_\alpha\cap U_\beta\to GL_n(\Bbb K)
$$ 
satisfying the consistency conditions. 

\begin{remark} All this works more generally for non-linear fiber bundles if we drop the linearity conditions along fibers.
\end{remark} 

\begin{example} 1. The {\bf trivial bundle} $p: E=X\times \Bbb K^n\to X$, $p(x,v)=x$. 

2. The {\bf tangent bundle} is the vector bundle $p: TX\to X$ constructed as follows. 
For the open cover we take an atlas of charts $(U_\alpha,\phi_\alpha)$ with transition maps 
$$
\theta_{\alpha\beta}=\phi_\alpha\circ \phi_\beta^{-1}: \phi_\beta(U_\alpha\cap U_\beta)\to 
\phi_\alpha(U_\alpha\cap U_\beta),
$$ 
and we set 
$$
h_{\alpha\beta}(x):=d_{\phi_\beta(x)}\theta_{\alpha\beta}. 
$$
(Check that these maps satisfy consistency conditions!) 

Thus the tangent bundle $TX$ is a vector bundle of rank $\dim X$ whose fiber $p^{-1}(x)$ is  
naturally the tangent space $T_xX$ (indeed, the tangent vectors transform under coordinate changes exactly by multiplication by $h_{\alpha\beta}(x)$). In other words, it formalizes 
the idea of ``the tangent space $T_xX$ varying smoothly with $x\in X$". 
\end{example} 

\begin{definition} A {\bf section} of a map $p: E\to X$ is a map $s: X\to E$ such that 
$p\circ s={\rm Id}_X$. 
\end{definition} 

\begin{example} If $p:E=X\times Y\to X$, $p(x,y)=x$ is the trivial bundle then a section 
$s: X\to E$ is given by $s(x)=(x,f(x))$ where $y=f(x)$ is a function $X\to Y$, and the image of $s$ is the graph of $f$. So the notion of a section is a generalization of the notion of a function. 
\end{example} 

In particular, we may consider sections of a vector bundle $p: E\to X$ over an open set $U\subset X$. These sections form a vector space denoted $\Gamma(U,E)$. 

\begin{exercise}\label{basi} Show that a vector bundle $p: E\to X$ is trivial (i.e., globally isomorphic to 
$X\times \Bbb K^n$) if and only if it admits sections $s_1,...,s_n$ which form a basis in every fiber $p^{-1}(x)$. 
\end{exercise} 

\subsection{Vector fields} \label{vefi}

\begin{definition} A {\bf vector field} on $X$ is a section of the tangent bundle $TX$. 
\end{definition} 

Thus in local coordinates a vector field looks like 
$$
\bold v=\sum_i v_i\frac{\partial}{\partial x_i},
$$
$v_i=v_i(\bold x)$, and if $x_i\mapsto x_i'$ is a change of local coordinates then 
the expression for $\bold v$ in the new coordinates is 
$$
\bold v=\sum_i v_i'\frac{\partial}{\partial x_i'}
$$
where 
$$
v_i'=\sum_j \frac{\partial x_i'}{\partial x_j}v_j, 
$$
i.e., the clutching function is the {\bf Jacobi matrix} of the change of variable. 
Thus, every vector field $\bold v$ on $X$ defines a derivation of the 
algebra $O(U)$ for every open set $U\subset X$ compatible with restriction 
maps $O(U)\to O(V)$ for $V\subset U$;\footnote{In other words, using a fancier language, $\bold v$ defines a derivation of the {\bf sheaf} of regular functions on $X$.} in particular, a derivation $O_x\to O_x$ 
for all $x\in X$. Conversely, it is easy to see that such a collection of derivations 
gives rise to a vector field, so this is really the same thing. 

A manifold $X$ is called {\bf parallelizable} if its tangent bundle is trivial. 
By Exercise \ref{basi}, this is equivalent to having a collection of vector fields 
$\bold v_1,...,\bold v_n$ which form a basis in every tangent space (such 
a collection is called a {\bf frame}). For example, the circle $S^1$ 
and hence the torus $S^1\times S^1$ are parallelizable. On the other hand, the sphere 
$S^2$ is not parallelizable, since it does not even have a single nowhere vanishing vector field (the {\bf Hairy Ball theorem}, or {\bf Hedgehog theorem}). The same is true for any even-dimensional sphere $S^{2m}$, $m\ge 1$. 

\subsection{Tensor fields, differential forms}\label{diffor} Since vector bundles are basically just smooth families of vector spaces varying over some base manifold $X$, we can do with them the same things we can do with vector spaces - duals, tensor products, symmetric and exterior powers, etc. E.g., the {\bf cotangent bundle} $T^*X$ is dual to the tangent bundle $TX$. 

More generally, we make the following definition.  

\begin{definition} A {\bf tensor field} of rank $(k,m)$ on a manifold $X$ is a section 
of the tensor product $(TX)^{\otimes k}\otimes  (T^*X)^{\otimes m}$. 
\end{definition} 

For example, a tensor field of rank $(1,0)$ is a vector field. Also, a skew-symmetric tensor field of rank $(0,m)$ is called a {\bf differential $m$-form} on $X$. In other words, a differential $m$-form is a section of the vector bundle $\Lambda^mT^*X$. 

For instance, if $f\in O(X)$ then we have a differential $1$-form $df$ on $X$, called {\bf the differential of $f$} (indeed, recall that $d_xf: T_xX\to \Bbb K$). A general $1$-form can therefore be written in local coordinates as 
$$
\omega=\sum_i a_idx_i. 
$$
where $a_i=a_i(\bold x)$. If coordinates are changed as $x_i\mapsto x_i'$, 
then in new coordinates 
$$
\omega=\sum_i a_i'dx_i'
$$
where 
$$
a_i'=\sum_j \frac{\partial x_j}{\partial x_i'}a_j.
$$
Thus the clutching function is the {\bf inverse of the Jacobi matrix} of the change of variable. 
For instance,
$$
df=\sum_i \frac{\partial f}{\partial x_i}dx_i.
$$

More generally, a differential $m$-form in local coordinates looks like 
$$
\omega=\sum_{1\le i_1<...<i_m\le n}a_{i_1...i_m}(x)dx_{i_1}\wedge...\wedge dx_{i_m}. 
$$ 

\subsection{Left and right invariant tensor fields on Lie groups}

Note that if a Lie group $G$ acts on a manifold $X$, then it automatically acts 
on the tangent bundle $TX$ and thus on vector and, more generally, tensor fields on $X$. 
In particular, $G$ acts on tensor fields on itself by left and right translations; we will denote this action by $L_g$ and $R_g$, respectively. We say that a tensor field $T$ on $G$ is {\bf left invariant} if $L_gT=T$ 
for all $g\in G$, and {\bf right invariant} if $R_gT=T$ for all $g\in G$. 

\begin{proposition}\label{spre} (i) For any $\tau\in \g^{\otimes k}\otimes \g^{*\otimes m}$ 
there exists a unique left invariant tensor field ${\bold L}_\tau$
and a unique right invariant tensor field ${\bold R}_\tau$ whose value at $1$ is $\tau$. 
Thus, the spaces of such tensor fields are naturally isomorphic to 
$\g^{\otimes k}\otimes \g^{*\otimes m}$. 

(ii) ${\bold L}_\tau$ is also right invariant iff ${\bold R}_\tau$ is also left invariant iff $\tau$ 
is invariant under the adjoint representation ${\rm Ad}_g$. 
\end{proposition} 

\begin{proof} We only prove (i). Consider the tensor fields ${\bold L}_\tau(g):=L_g\tau$,\linebreak ${\bold R}_\tau(g):=R_{g^{-1}}\tau$
(i.e., we ``spread" $\tau$ from $1\in G$ to other points $g\in G$ by left/right translations). By construction, $R_{g^{-1}}\tau$ is right invariant, while $L_g\tau$ is left invariant, both with value $\tau$ at $1$, 
and it is clear that these are unique. 
\end{proof} 

\begin{exercise} Prove Proposition \ref{spre}(ii). 
\end{exercise}

\begin{corollary} A Lie group is parallelizable. 
\end{corollary} 

\begin{proof} Given a basis $e_1,...,e_n$ of $\g=T_1G$, the vector fields
$L_ge_1,...,L_ge_n$ form a frame. 
\end{proof} 

\begin{remark} 
In particular, $S^1$ and $SU(2)=S^3$ are parallelizable. It turns out that $S^n$ 
for $n\ge 1$ is parallelizable if and only if $n=1,3,7$ (a deep theorem in differential topology). So spheres of other dimensions don't admit a Lie group structure. The sphere $S^7$ does not admit one either, although it admits a weaker structure of a ``homotopy Lie group", or $H$-space (arising from octonions) which suffices for parallelizability. Thus the only spheres admitting a Lie group structure are $S^0=\lbrace{1,-1\rbrace}$, $S^1$ and $S^3$. This result is fairly elementary and will be proved in Section \ref{top2}. 
\end{remark} 

\section{\bf Classical Lie groups} 

\subsection{First examples of classical groups} \label{clagrou}
Roughly speaking, {\bf classical groups} are groups of matrices arising from linear algebra. 
More precisely, classical groups are the following subgroups of the {\bf general linear group} $GL_n(\Bbb K)$: 
$GL_n(\Bbb K)$, $SL_n(\Bbb K)$ (the {\bf special linear group}), 
$O_n(\Bbb K)$, $SO_n(\Bbb K)$, $Sp_{2n}(\Bbb K)$,
$O(p,q)$, $SO(p,q)$,  $U(p,q)$, $SU(p,q)$, \linebreak
$Sp(2p,2q):=Sp_{2n}(\Bbb C)\cap U(2p,2q)$ for $p+q=n$ 
(and also some others we'll consider later). 

Namely, 

$\bullet$ The {\bf orthogonal group} 
$O_n(\Bbb K)$ is the group of matrices preserving the nondegenerate 
quadratic form in $n$ variables, $Q=x_1^2+...+x_n^2$ (or, equivalently, the corresponding bilinear form $x_1y_1+...+x_ny_n$);

$\bullet$ The {\bf symplectic group} $Sp_{2n}(\Bbb K)$ is the group of matrices 
preserving a nondegenerate skew-symmetric form 
in $2n$ variables; 

$\bullet$ The {\bf pseudo-orthogonal group}  
$O(p,q)$, $p+q=n$ is the group of real matrices preserving a nondegenerate quadratic form 
of signature $(p,q)$, $Q=x_1^2+...+x_p^2-x_{p+1}^2-...-x_n^2$ (or, equivalently, the corresponding bilinear form); 

$\bullet$ The {\bf pseudo-unitary group} $U(p,q)$, $p+q=n$ is the group of complex matrices 
preserving a nondegenerate Hermitian quadratic form 
of signature $(p,q)$, $Q=|x_1|^2+...+|x_p|^2-|x_{p+1}|^2-...-|x_n|^2$ (or, equivalently, the corresponding sesquilinear form);

$\bullet$ The {\bf special pseudo-orthogonal, pseudo-unitary, and orthogonal groups} 
$SO(p,q)\subset O(p,q)$, $SU(p,q)\subset U(p,q)$, $SO_n(\Bbb K)\subset O_n(\Bbb K)$ are the subgroups of matrices of determinant $1$.

Note that the groups don't change under switching $p,q$ and that 
$(S)O_n(\Bbb R)=(S)O(n,0)$; it is also denoted $(S)O(n)$. Also $(S)U(n,0)$ is denoted by $(S)U(n)$. 

\begin{exercise} Show that the special (pseudo)orthogonal groups are index $2$ subgroups of the (pseudo)orthogonal groups. 
\end{exercise} 

Let us show that they are all Lie groups. For this purpose we'll use the {\bf exponential map} for matrices. Namely, recall from linear algebra that we have an analytic function ${\rm exp}: {\mathfrak{gl}}_n(\Bbb K)\to GL_n(\Bbb K)$ given by the formula 
$$
\exp(a)=\sum_{n=0}^\infty \frac{a^n}{n!}, 
$$
and the matrix-valued analytic function log near $1\in GL_n(\Bbb K)$, 
$$
\log(A)=-\sum_{n=1}^\infty \frac{(1-A)^n}{n}.
$$
Namely, this is well defined if the spectral radius of $1-A$ is $<1$ (i.e., all eigenvalues are in the open unit disk). These maps have the following properties:

(1) They are mutually inverse. 

(2) They are conjugation-invariant. 

(3) $d\exp_0=d\log_1={\rm Id}$.

(4) If $xy=yx$ then $\exp(x+y)=\exp(x)\exp(y)$. If $XY=YX$ 
then $\log(XY)=\log(X)+\log(Y)$ (for $X,Y$ sufficiently close to $1$).  

(5) For $x\in {\mathfrak{gl}}_n(\Bbb K)$ the map $t\mapsto \exp(tx)$ is a homomorphism of Lie groups $\Bbb K\to \GL_n(\Bbb K)$. 

(6) $\det \exp(a)=\exp ({\rm Tr}\ a)$, $\log (\det A)={\rm Tr}(\log A)$. 

Now we can look at classical groups and see what happens to the equations defining them when we apply $\log$. 

1. $G=SL_n(\Bbb K)$. We already showed that it is a Lie group  
in Example \ref{wind}(1) but let us re-do it by a different method. 
The group $G$ is defined by the equation $\det A=1$. So for $A$ close to $1$ we have 
$\log(\det A)=0$, i.e., ${\rm Tr}\log(A)=0$. So $\log(A)\in \mathfrak{sl}_n(\Bbb K)=\mathfrak{g}$, 
the space of matrices with trace $0$. This defines a local chart near $1\in G$, showing that $G$ is a manifold (namely, local charts near other points are obtained by translation), hence a Lie group (multiplication is inherited from the ambient matrix group, hence regular in the induced manifold structure).  

2. $G=O_n(\Bbb K)$. The equation is $A^T=A^{-1}$, thus $\log(A)^T=-\log(A)$, so $\log(A)\in \mathfrak{so}_n(\Bbb K)=\g$, the space of skew-symmetric matrices. 

3. $G=U(n)$. The equation is $\overline{A}^T=A^{-1}$, thus $\overline{\log(A)}^T=-\log(A)$, so $\log(A)\in \mathfrak{u}_n=\g$, the space of skew-Hermitian matrices. 

\begin{exercise} Do the same for all classical groups listed above. 
\end{exercise} 

We see that the logarithm map identifies the neighborhood of $1$ in the group $G$ with a neighborhood of $0$ in a finite-dimensional vector space. Thus we 
obtain

\begin{proposition} Every classical group $G$ from the above list is a Lie group, with 
$\g=T_1G\subset {\mathfrak{gl}}_n(\Bbb K)$. Moreover, if 
$\mathfrak{u}\subset {\mathfrak{gl}}_n(\Bbb K)$ is a small neighborhood of $0$ 
and $U=\exp(\mathfrak{u})$ then $\exp$ and $\log$ define mutually inverse 
diffeomorphisms between $\mathfrak{u}\cap \g$ and $U\cap G$. 
\end{proposition} 

\begin{exercise} Which of these groups are complex Lie groups? 
\end{exercise} 

\begin{exercise} Use this proposition to compute the dimensions of classical groups: 
$\dim SL_n=n^2-1$, $\dim O_n=n(n-1)/2$, $\dim Sp_{2n}=n(2n+1)$, 
$\dim SU_n=n^2-1$, etc. (Note that for complex groups we give the dimension over $\Bbb C$). 
\end{exercise} 

\subsection{Quaternions} An important role in the theory of Lie groups is played by the {\bf algebra of quaternions}, which is the only noncommutative finite dimensional division algebra over $\Bbb R$, discovered in the 19th century by W. R. Hamilton. 

\begin{definition} The {\bf algebra of quaternions} is the $\Bbb R$-algebra with basis 
$1,\bold i,\bold j,\bold k$ and multiplication rules 
$$
\bold i\bold j=-\bold j\bold i=\bold k,\  \bold j\bold k=-\bold k\bold j=\bold i,\ \bold k\bold i=-\bold i\bold k=\bold j, \bold i^2=\bold j^2=\bold k^2=-1.
$$
\end{definition} 

This algebra is associative but not commutative.

Given a quaternion 
$$
\bold q=a+b\bold i+c\bold j+d\bold k,\ a,b,c,d\in \Bbb R,
$$ 
we define the {\bf conjugate quaternion} by the formula 
$$
\overline{\bold q}=a-b\bold i-c\bold j-d\bold k.
$$ 
Thus 
$$
\bold q\overline{\bold q}=|\bold q|^2=a^2+b^2+c^2+d^2\in \Bbb R,
$$
where $|\bold q|$ is the length of $\bold q$ as a vector in $\Bbb R^4$. So if $\bold q\ne 0$ then it is invertible and 
$$
\bold q^{-1}=\frac{\overline {\bold q}}{|\bold q|^2}.
$$

Thus $\Bbb H$ is a {\bf division algebra} (i.e., a skew-field). One can show that the only 
finite dimensional associative division algebras over $\Bbb R$ are $\Bbb R$, $\Bbb C$ and $\Bbb H$. (See Exercise \ref{quate}).

 In particular, we can do linear algebra over $\Bbb H$ in almost the same way as we do over ordinary fields. 
Namely, every (left or right) module over $\Bbb H$ is free and has a basis; such a module 
is called a (left or right) {\bf quaternionic vector space}. In particular, any 
(say, right) quaternionic vector space of dimension $n$ (i.e., with basis of $n$ elements) 
is isomorphic to $\Bbb H^n$. Moreover, $\Bbb H$-linear maps between such spaces are given by left multiplication by quaternionic matrices. Finally, it is easy to see that Gaussian elimination works the same way as over ordinary fields; in particular, every invertible square matrix over $\Bbb H$ is a product of elementary matrices of the form $1+(\bold q-1)E_{ii}$ and $1+\bold q E_{ij}$, $i\ne j$, where $\bold q\in \Bbb H$ is nonzero.    

Also it is easy to show that 
$$
\overline{\bold q_1\bold q_2}=\overline{\bold q_2}\cdot \overline{\bold q_1},\ |\bold q_1\bold q_2|=|\bold q_1|\cdot |\bold q_2|
$$
(check this!). So quaternions are similar to complex numbers, except they are non-commutative. Finally, note that $\Bbb H$ contains a copy of $\Bbb C$ 
spanned by $1,\bold i$; however, this does not make $\Bbb H$ a $\Bbb C$-algebra since $\bold i$ is not a central element. 

\begin{proposition} The group of unit quaternions 
$\lbrace \bold q\in \Bbb H: |\bold q|=1\rbrace$ under multiplication 
is isomorphic to $SU(2)$ as a Lie group. 
\end{proposition} 

\begin{proof} We can realize $\Bbb H$ as $\Bbb C^2$, where $\Bbb C\subset \Bbb H$ 
is spanned by $1,\bold i$; namely, $(z_1,z_2)\mapsto z_1+\bold j z_2$. 
Then left multiplication by quaternions on $\Bbb H=\Bbb C^2$
commutes with right multiplication by $\Bbb C$, i.e., is $\Bbb C$-linear. 
So it is given by complex $2$-by-$2$ matrices. It is easy to compute 
that the corresponding matrix is 
$$
z_1+\bold jz_2\mapsto \begin{pmatrix}z_1& -\overline{z_2}\\  z_2& \overline{z_1}\end{pmatrix}, 
$$
and we showed in Example \ref{exalie}(5) that such matrices (with $|z_1|^2+|z_2|^2=1$) are exactly the matrices from $SU(2)$. 
\end{proof} 

This is another way to see that $SU(2)\cong S^3$ as a manifold (since the set of unit quaternions is manifestly $S^3$). 

\begin{corollary} The map $\bold q\mapsto (\frac{\bold q}{|\bold q|},|\bold q|)$ 
is an isomorphism of Lie groups $\Bbb H^\times\cong SU(2)\times \Bbb R_{>0}$. 
\end{corollary} 

This is the quaternionic analog of the trigonometric form of complex numbers, except the ``phase" factor $\frac{\bold q}{|\bold q|}$ is now not in $S^1$ but in $S^3=SU(2)$. 

\begin{exercise}\label{quate} Let $D$ be a finite dimensional division algebra 
over $\Bbb R$. 

(i) Show that if $D$ is commutative then $D=\Bbb R$ or $D=\Bbb C$.  

(ii) Assume that $D$ is not commutative. Take $\bold q\in D$, $\bold q\notin \Bbb R$. Show that there exist $a,b\in \Bbb R$ such that $\bold i:=a+b\bold q$ satisfies $\bold i^2=-1$. 

(iii) Decompose $D$ into the eigenspaces $D_\pm$ of the operator of conjugation by $\bold i$ 
with eigenvalues $\pm 1$ and show that $1,\bold i$ is a basis of $D_+$, i.e., $D_+\cong \Bbb C$. 

(iv) Pick $\bold q\in D_-$, $\bold q\ne 0$, and show that $D_-=D_+\bold q$, so 
$\lbrace 1,\bold i,\bold q,\bold i\bold q\rbrace$ is a basis of $D$ over $\Bbb R$. Deduce that 
$\bold q^2$ is a central element of $D$. 

(v) Conclude that $\bold q^2=-\lambda$ where 
$\lambda\in  \Bbb R_{>0}$ and deduce that $D\cong \Bbb H$.    
\end{exercise}  

\subsection{More classical groups} \label{morecla} 
Now we can define a new classical group $GL_n(\Bbb H)$, a real Lie group 
of dimension $4n^2$, called the {\bf quaternionic general linear group}. 
For example, as we just showed, 
$GL_1(\Bbb H)=\Bbb H^\times\cong SU(2)\times \Bbb R_{>0}$. 

For $A\in GL_n(\Bbb H)$, let $\det A$ be the determinant of $A$ 
as a linear operator on $\Bbb C^{2n}=\Bbb H^n$. 

\begin{lemma} We have $\det A>0$. 
\end{lemma} 

\begin{proof} For $n=1$, $A=\bold q\in \Bbb H^\times$ and 
$\det \bold q=|\bold q|^2>0$. It follows that 
$\det (1+(\bold q-1)E_{ii})=|\bold q|^2>0$. Also 
it is easy to see that $\det(1+\bold qE_{ij})=1$ for $i\ne j$.
It then follows by Gaussian elimination that for any $A$ we have 
$\det(A)>0$. 
\end{proof} 

Let $SL_n(\Bbb H)\subset GL_n(\Bbb H)$ be the subgroup of matrices $A$ with $\det A=1$, called the {\bf quaternionic special linear group}.   

\begin{exercise} Show that $SL_n(\Bbb H)\subset GL_n(\Bbb H)$ is a normal subgroup, and $GL_n(\Bbb H)\cong SL_n(\Bbb H)\times \Bbb R_{>0}$.  
\end{exercise} 

Thus $SL_n(\Bbb H)$ is a real Lie group of dimension $4n^2-1$. 

We can also define groups of quaternionic matrices preserving various sesquilinear forms. 
Namely, let $V\cong \Bbb H^n$ be a right quaternionic vector space. 

\begin{definition} A {\bf sesquilinear form} on $V$ is a biadditive function 
$(,): V\times V\to \Bbb H$ such that 
$$
(\bold x \alpha,\bold y \beta)=\overline{\alpha}(\bold x,\bold y)\beta,\ \bold x,\bold y\in V,\ \alpha,\beta\in \Bbb H.
$$

Such a form is called {\bf Hermitian} if $(\bold x,\bold y)=\overline{(\bold y,\bold x)}$ 
and {\bf skew-Hermitian} if $(\bold x,\bold y)=-\overline{(\bold y,\bold x)}$. 
\end{definition}

Note that the order of factors is important here!

\begin{proposition}\label{sesq} (i) Every nondegenerate Hermitian form on $V$ 
in some basis takes the form 
$$
(\bold x,\bold y)=\overline{x_1}y_1+...+\overline{x_p}y_p-\overline{x_{p+1}}y_{p+1}-...-\overline{x_n}y_n
$$
for a unique pair $(p,q)$ with $p+q=n$. 

(ii) Every nondegenerate skew-Hermitian form on $V$ 
in some basis takes the form 
$$
(\bold x,\bold y)=\overline{x_1}\bold jy_1+...+\overline{x_n}\bold jy_n. 
$$
\end{proposition} 

\begin{exercise} Prove Proposition \ref{sesq}. 
\end{exercise} 

In (i), the pair $(p,q)$ is called the {\bf signature} of the quaternionic Hermitian form. 

\begin{exercise}\label{quatgrou} Show that a nondegenerate quaternionic Hermitian form of signature $(p,q)$ can be written as 
$$
(\bold x,\bold y)=B_1(\bold x,\bold y)+\bold jB_2(\bold x,\bold y), 
$$
with $B_1,B_2$ taking values in $\Bbb C=\Bbb R+\Bbb R\bold i\subset \Bbb H$, 
where $B_1$ is a usual nondegenerate Hermitian form of signature $(2p,2q)$ and 
$B_2$ is a nondegenerate skew-symmetric bilinear form on $V$ as a ($2n$-dimensional) $\Bbb C$-vector space. Show that $B_2(\bold x,\bold y)=B_1(\bold x\bold j,\bold y)$. 
Deduce that any complex linear transformation preserving $B_1$ and $B_2$ 
is $\Bbb H$-linear. 
\end{exercise} 

Thus the group of symmetries of a nondegenerate quaternionic Hermitian form of signature $(p,q)$ is $Sp(2p,2q)=Sp_{2n}(\Bbb C)\cap U(2p,2q)$. It is called the {\bf quaternionic pseudo-unitary group}. 

One also sometimes uses the notation $U(p,q,\Bbb R)=O(p,q)$, $U(p,q,\Bbb C)=U(p,q)$, $U(p,q,\Bbb H)=Sp(2p,2q)$, and $U(n,0,\Bbb K)=U(n,\Bbb K)$ for $\Bbb  K=\Bbb R,\Bbb C,\Bbb H$.  

\begin{exercise} Show that a nondegenerate quaternionic skew-Hermitian form
can be written as 
$$
(\bold x,\bold y)=B_1(\bold x,\bold y)+\bold jB_2(\bold x,\bold y), 
$$
with $B_1,B_2$ taking values in $\Bbb C=\Bbb R+\Bbb R\bold i\subset \Bbb H$, 
where $B_1$ is an ordinary skew-Hermitian form, while $B_2$ is a symmetric bilinear form 
(both nondegenerate). Show that $B_2(\bold x,\bold y)=B_1(\bold x\bold j,\bold y)$.
Deduce that any complex linear transformation preserving $B_1$ and $B_2$ 
is $\Bbb H$-linear. Also show that the signature of the Hermitian form $iB_1$ is necessarily $(n,n)$. 
\end{exercise} 

Thus the group of symmetries of a nondegenerate quaternionic skew-Hermitian form
is $O_{2n}(\Bbb C)\cap U(n,n)$. This group is denoted by $O^*(2n)$ and called the {\bf quaternionic orthogonal group}. Every element of $O^*(2n)$ has determinant $1$, so this group can also be denoted by $SO^*(2n)$. 

All of these groups are Lie groups, which is shown similarly to Subsection \ref{clagrou}, using the exponential map. 

\begin{exercise} Compute the dimensions of all classical groups introduced above. 
\end{exercise} 

\section{\bf The exponential map of a Lie group} 

\subsection{The exponential map} We will now generalize the exponential and logarithm maps from matrix groups to arbitrary Lie groups. 

Let $G$ be a real Lie group, $\g=T_1G$.
 
\begin{proposition} Let $x\in \g$. There is a unique morphism of Lie groups $\gamma=\gamma_x: \Bbb R\to G$ such that $\gamma'(0)=x$. 
\end{proposition} 

\begin{proof} For such a morphism we should have
$$
\gamma(t+s)=\gamma(t)\gamma(s),\ t,s\in \Bbb R,
$$
so differentiating by $s$ at $s=0$, we get\footnote{For brevity for $g\in G$, $x\in \g$ we denote $L_gx$ by $gx$ and $R_gx$ by $xg$.}
$$
\gamma'(t)=\gamma(t)x.
$$
Thus $\gamma(t)$ is a solution of the ODE defined by the left-invariant vector field $\bold L_x$ corresponding to $x\in \g$ with initial condition $\gamma(0)=1$. By the existence and uniqueness theorem for solutions of ODE, this equation has a unique solution with this initial condition defined 
for $|t|<\varepsilon$ for some $\varepsilon>0$. 
Moreover, if $|s|+|t|<\varepsilon$, both $\gamma_1(t):=\gamma(s+t)$ and $\gamma_2(t):=\gamma(s)\gamma(t)$ satisfy this differential equation with initial condition $\gamma_1(0)=\gamma_2(0)=\gamma(s)$, so $\gamma_1=\gamma_2$. 
Thus
$$
\gamma(s+t)=\gamma(s)\gamma(t),\ |s|+|t|<\varepsilon;
$$
hence $\gamma(t)x=x\gamma(t)$ for $|t|<\varepsilon$. 

We claim that the solution $\gamma(t)$ extends to all values of $t\in \Bbb R$. Indeed, let us prove that it extends to $|t|<2^n\varepsilon$ for all $n\ge 0$ by induction in $n$. The base of induction ($n=0$) is already known, so we only need to justify the induction step from $n-1$ to $n$. Given $t$ with $|t|<2^n\varepsilon$, we define 
$$
\gamma(t):=\gamma(\tfrac{t}{2})^2. 
$$
This agrees with the previously defined solution for $|t|<2^{n-1}\varepsilon$, and we have  
$$
\gamma'(t)=\tfrac{1}{2}(\gamma'(\tfrac{t}{2})\gamma(\tfrac{t}{2})+\gamma(\tfrac{t}{2})\gamma'(\tfrac{t}{2}))=
\tfrac{1}{2}\gamma(\tfrac{t}{2})x\gamma(\tfrac{t}{2})+\tfrac{1}{2}\gamma(\tfrac{t}{2})^2x=\gamma(\tfrac{t}{2})^2x=\gamma(t)x,
$$
as desired. 

Thus, we have a regular map $\gamma: \Bbb R\to G$ with $\gamma(s+t)=\gamma(s)\gamma(t)$ and $\gamma'(0)=x$, which is unique by the uniqueness of solutions of ODE. 
\end{proof} 

\begin{definition} The {\bf exponential map} 
$\exp: \g\to G$ is defined by the formula 
$\exp(x)=\gamma_x(1)$.
\end{definition} 

Thus $\gamma_x(t)=\exp(tx)$. So we have 

\begin{proposition} The flow defined by the right-invariant vector field $\bold R_x$ 
is given by $g\mapsto \exp(tx)g$, and the flow defined by the left-invariant vector field $\bold L_x$ is given by $g\mapsto g\exp(tx)$. 
\end{proposition} 

\begin{example} 1. Let $G=\Bbb K^n$. Then $\exp(x)=x$.  

2. Let $G=GL_n(\Bbb K)$ or its Lie subgroup. Then 
$\gamma_x(t)$ satisfies the matrix differential equation 
$$
\gamma'(t)=\gamma(t)x
$$
with $\gamma(0)=1$, so 
$$
\gamma_x(t)=e^{tx},
$$
the matrix exponential. For example, if $n=1$, this is the usual exponential function.  
\end{example} 

The following theorem describes the basic properties of the exponential map. Let $G$ be a real or complex Lie group.  

\begin{theorem} (i) $\exp: \g\to G$ is a regular map which is a diffeomorphism of a neighborhood of $0\in \g$ onto a neighborhood of $1\in G$, with $\exp(0)=1$, 
$\exp'(0)={\rm Id}_\g$. 

(ii) $\exp((s+t)x)=\exp(sx)\exp(tx)$ for $x\in \g$, $s,t\in \Bbb K$. 

(iii) For any morphism of Lie groups $\phi: G\to K$ and $x\in T_1G$ we have 
$$
\phi(\exp(x))=\exp( \phi_*x);
$$
 i.e., the exponential map commutes with morphisms. 

(iv) For any $g\in G$, $x\in \g$, we have 
$$
g\exp(x)g^{-1}=\exp({\rm Ad}_gx).
$$ 
\end{theorem} 
 
\begin{proof} (i) The regularity of $\exp$ follows from the fact that if a differential equation depends regularly on parameters then so do its solutions. Also $\gamma_0(t)=1$ so $\exp(0)=1$. We have $\exp'(0)x=\frac{d}{dt}\exp(tx)|_{t=0}=x$, so $\exp'(0)={\rm Id}$. 
By the inverse function theorem this implies that $\exp$ is a diffeomorphism near the origin. 

(ii) Holds since $\exp(tx)=\gamma_x(t)$. 

(iii) Both $\phi(\exp(tx))$ and  $\exp(\phi_*(tx))$ 
satisfy the equation $\gamma'(t)=\gamma(t)\phi_*(x)$ 
with the same initial conditions.  

(iv) is a special case of (iii) with $\phi: G\to G$, $\phi(h)=ghg^{-1}$. 
\end{proof} 

Thus $\exp$ has an inverse $\log: U\to \g$ defined on a neighborhood $U$ of 
$1\in G$ with $\log(1)=0$. This map is called  the {\bf logarithm}. For $GL_n(\Bbb K)$
and its Lie subgroups it coincides with the matrix logarithm. The logarithm map defines a canonical coordinate chart on $G$ near $1$, so a choice of a basis 
of $\g$ gives a local coordinate system.   

\begin{proposition} Let $G$ be a connected Lie group and $\phi: G\to K$ a morphism of Lie groups. Then $\phi$ is completely determined by the linear map $\phi_*: T_1G\to T_1K$.
\end{proposition} 

\begin{proof} We have $\phi(\exp(x))=\exp(\phi_*(x))$, so 
since $\exp$ is a diffeomorphism near $0$, $\phi$ is determined 
by $\phi_*$ on a neighborhood of $1\in G$. This completely determines $\phi$ 
since this neighborhood generates $G$ by Proposition \ref{generati}. 
\end{proof} 

\begin{exercise} (i) Show that a connected compact complex Lie group is abelian. ({\bf Hint:} consider the adjoint representation and use that a holomorphic function on a compact complex manifold is constant, by the maximum principle.)

(ii) Classify such Lie groups of dimension $n$ up to isomorphism (Show that they are compact complex tori whose isomorphism classes are bijectively labeled by elements of the set $GL_n(\Bbb C)\backslash GL_{2n}(\Bbb R)/GL_{2n}(\Bbb Z)$.) 

(iii) Work out the classification explicitly in the 1-dimensional case (this is the classification of complex elliptic curves). Namely, show that isomorphism classes are labeled by points of $\Bbb H/\Gamma$, where $\Bbb H$ is the upper half-plane
and $\Gamma=SL_2(\Bbb Z)$ acting on $\Bbb H$ by M\"obius transformations 
$\tau\mapsto \frac{a\tau+b}{c\tau+d}$ (where ${\rm Im}(\tau)>0$). 
\end{exercise} 

\subsection{The commutator} 

In general (say, for $G=GL_n(\Bbb K)$, $n\ge 2$),
$\exp(x+y)\ne \exp(x)\exp(y)$. So let us consider the map 
$$
(x,y)\mapsto \mu(x,y)=\log(\exp(x)\exp(y))
$$
which maps $U\times U\to \g$, where $U\subset \g$ is a neighborhood of $0$. 
This map expresses the product in $G$ in the coordinate chart coming from the logarithm 
map. We have $\mu(x,0)=\mu(0,x)=x$ and $\mu_*(x,y)=x+y$. So, since $\mu$ is regular, we have the second Taylor approximation 
$$
\mu(x,y)=x+y+\tfrac{1}{2}\mu_2(x,y)+...
$$
where $\mu_2=d^2\mu_{(0,0)}$ is the quadratic part and $...$ are higher terms. 
Moreover, $\mu_2(x,0)=\mu_2(0,y)=0$, hence 
$\mu_2$ is a bilinear map $\g\times \g\to \g$. It is easy to see that 
$\mu(x,-x)=0$, hence $\mu_2$ is skew-symmetric. 

\begin{definition} The map $\mu_2$ is called the {\bf commutator} and denoted by $x,y\mapsto [x,y]$. 
\end{definition} 

Thus we have 
\begin{equation}\label{prodexp}
\exp(x)\exp(y)=\exp(x+y+\tfrac{1}{2}[x,y]+...).
\end{equation}

\begin{example} Let $G=GL_n(\Bbb K)$. Then 
$$
\exp(x)\exp(y)=(1+x+\tfrac{x^2}{2}+...)(1+y+\tfrac{y^2}{2}+...)=
1+x+y+\tfrac{x^2}{2}+xy+\tfrac{y^2}{2}+...=
$$
$$
1+(x+y)+\tfrac{(x+y)^2}{2}+\tfrac{xy-yx}{2}+...=\exp(x+y+\tfrac{xy-yx}{2}+...)
$$
Thus 
$$
[x,y]=xy-yx.
$$ 
This justifies the term ``commutator": it measures the failure of $x$ and $y$ to commute. 
\end{example} 

\begin{corollary} If $G\subset GL_n(\Bbb K)$ is a Lie subgroup then $\g=T_1G\subset \mathfrak{gl}_n(\Bbb K)$ is closed under the commutator $[x,y]=xy-yx$, 
which coincides with the commutator of $G$.  
\end{corollary} 

For $x\in \g$ define the linear map ${\rm ad} x: \g\to \g$ 
by 
$$
{\rm ad}x(y)=[x,y].
$$
\begin{proposition}\label{commprop} (i) Let $G,K$ be Lie groups and $\phi: G\to K$ a morphism of Lie groups. 
Then $\phi_*: T_1G\to T_1K$ preserves the commutator: 
$$
\phi_*([x,y])=[\phi_*(x),\phi_*(y)].
$$

(ii) The adjoint action preserves the commutator.  

(iii) We have 
$$
\exp(x)\exp(y)\exp(x)^{-1}\exp(y)^{-1}=\exp([x,y]+...)
$$
where $...$ denotes cubic and higher terms. 

(iv) Let $X(t),Y(s)$ be parametrized curves on $G$ such that $X(0)=Y(0)=1$, $X'(0)=x,Y'(0)=y$. Then 
we have 
$$
[x,y]=\lim_{s,t\to 0}\frac{\log(X(t)Y(s)X(t)^{-1}Y(s)^{-1})}{ts}.
$$
In particular, 
$$
[x,y]=\lim_{s,t\to 0}\frac{\log(\exp(tx)\exp(sy)\exp(tx)^{-1}\exp(sy)^{-1})}{ts}
$$
and 
$$
[x,y]=\tfrac{d}{dt}|_{t=0} {\rm Ad}_{X(t)}(y).
$$
Thus ${\rm ad}={\rm Ad}_*$, the differential of ${\rm Ad}$ at $1\in G$. 

(v) If $G$ is commutative (=abelian) then $[x,y]=0$ for all $x,y$. 
\end{proposition} 

\begin{proof} (i) Follows since $\phi$ commutes with the exponential map. 

(ii) Follows from (i) by setting $\phi={\rm Ad}_g$. 

(iii) By \eqref{prodexp}, modulo cubic and higher terms we have 
$$
\log(\exp(x)\exp(y))=\log(\exp(y)\exp(x))+[x,y]+...,
$$
which implies the statement by exponentiation. 

(iv) Let $\log X(t)=x(t),\ \log Y(s)=y(s)$. Then by (iii) we have 
$$
\log(X(t)Y(s)X(t)^{-1}Y(s)^{-1})=
$$
$$
\log(\exp(x(t))\exp(y(s))\exp(x(t))^{-1}\exp(y(s))^{-1})=
ts([x,y]+o(1)),\ t,s\to 0. 
$$
This implies the first two statements. The last statement follows 
by taking the limit in $s$ first, then in $t$. 

(v) follows from (iii). 
\end{proof} 

\section{\bf Lie algebras} 

\subsection{The Jacobi identity} 

The matrix commutator $[x,y]=xy-yx$ obviously satisfies the identity 
$$
[[x,y],z]+[[y,z],x]+[[z,x],y]=0
$$
called the {\bf Jacobi identity}. Thus it is satisfied for any Lie subgroup of $GL_n(\Bbb K)$. 

\begin{proposition} The Jacobi identity holds for any Lie group $G$. 
\end{proposition} 

\begin{proof} Let $\g=T_1G$. The Jacobi identity is equivalent to ${\rm ad}x$ being a derivation 
of the commutator: 
$$
{\rm ad}x([y,z])=[{\rm ad}x(y),z]+[y,{\rm ad}x(z)],\ x,y,z\in \g. 
$$
To show that it is indeed a derivation, let $g(t)=\exp(tx)$, then 
$$
{\rm Ad}_{g(t)}([y,z])=[{\rm Ad}_{g(t)}(y),{\rm Ad}_{g(t)}(z)].
$$
The desired identity is then obtained by differentiating this equality by $t$ at $t=0$ and using the Leibniz rule and Proposition \ref{commprop}(iv). 
\end{proof} 

\begin{corollary} We have 
${\rm ad}[x,y]=[{\rm ad}x,{\rm ad}y]$.
\end{corollary}

\begin{proof} This is also equivalent to the Jacobi identity. 
\end{proof} 

\begin{proposition} For $x\in \g$ one has $\exp({\rm ad}x)={\rm Ad}_{\exp(x)}\in GL(\g)$. 
\end{proposition} 

\begin{proof} We will show that $\exp(t{\rm ad}x)={\rm Ad}_{\exp(tx)}$ for $t\in \Bbb R$. 
Let $\gamma_1(t)=\exp(t{\rm ad}x)$ and $\gamma_2(t)={\rm Ad}_{\exp(tx)}$. Then 
$\gamma_1,\gamma_2$ both satisfy the differential equation $\gamma'(t)=\gamma(t){\rm ad}x$ and 
equal $1$ at $t=0$. Thus $\gamma_1=\gamma_2$. 
\end{proof} 

\subsection{Lie algebras} 

\begin{definition} A {\bf Lie algebra} over a field ${\bf k}$ is a vector space $\g$ over ${\bf k}$
equipped with bilinear operation $[,]:  \g\times \g\to \g$, called the {\bf commutator} or {\bf (Lie) bracket}  which satisfies the following identities: 

(i) $[x,x]=0$ for all $x\in \g$; 

(ii) the Jacobi identity: $[[x,y],z]+[[y,z],x]+[[z,x],y]=0$. 

A {\bf (homo)morphism of Lie algebras} is a linear map between Lie algebras that preserves the commutator. 
\end{definition} 

\begin{remark} If ${\bf k}$ has characteristic $\ne 2$ then the condition $[x,x]=0$ is equivalent to skew-symmetry $[x,y]=-[y,x]$, but in characteristic 2 it is stronger. 
\end{remark} 

\begin{example} Any subspace of ${\mathfrak{gl}}_n({\bf k})$ closed 
under $[x,y]:=xy-yx$ is a Lie algebra. 
\end{example} 

\begin{example} The map ${\rm ad}: \g\to {\rm End}(\g)$ is a morphism of Lie algebras.  
\end{example} 

Thus we have 

\begin{theorem} If $G$ is a $\Bbb K$-Lie group (for $\Bbb K=\Bbb R,\Bbb C$) then 
$\g:=T_1G$ has a natural structure of a Lie algebra over $\Bbb K$. Moreover, 
if $\phi: G\to K$ is a morphism of Lie groups then $\phi_*: T_1G\to T_1K$ 
is a morphism of Lie algebras.
\end{theorem}

We will denote the Lie algebra $\g=T_1G$ by ${\rm Lie}G$ or ${\rm Lie}(G)$ and call it the {\bf Lie algebra of $G$}. We see that the assignment $G\mapsto {\rm Lie}G$ is a functor from the category 
of Lie groups to the category of Lie algebras. Thus we have a map 
$\Hom(G,K)\to \Hom({\rm Lie} G,{\rm Lie}K)$, which is injective if $G$ is connected.  

Motivated by Proposition \ref{commprop}(v), a Lie algebra $\g$ is said to be {\bf commutative} or {\bf abelian} if $[x,y]=0$ 
for all $x,y\in \g$. 

\subsection{Lie subalgebras and ideals} 

A {\bf Lie subalgebra} of a Lie algebra $\g$ is a subspace $\h\subset \g$ closed under the commutator. It is called a {\bf Lie ideal} if moreover $[\g,\h]\subset \h$. 

\begin{proposition} Let $H\subset G$ be a Lie subgroup. Then: 

(i) ${\rm Lie}H\subset {\rm Lie}G$ is a Lie subalgebra; 

(ii) If $H$ is normal then ${\rm Lie}H$ is a Lie ideal in ${\rm Lie}G$; 

(iii) If $G,H$ are connected and ${\rm Lie}H\subset {\rm Lie}G$ is a Lie ideal then 
$H$ is normal in $G$. 
\end{proposition} 

\begin{proof} (i) If $x,y\in \h$ then $\exp(tx),\exp(sy)\in H$, so by Proposition \ref{commprop}(iv)
$$
[x,y]=\lim_{t,s\to 0}\frac{\log(\exp(tx)\exp(sy)\exp(-tx)\exp(-sy))}{ts}\in \h. 
$$

(ii) We have $ghg^{-1}\in H$ for $g\in G$ and $h\in H$. Thus, 
taking $h=\exp(sy)$, $y\in \h$ and taking the derivative in $s$ at zero, 
we get ${\rm Ad}_g(y)\in \h$. Now taking $g=\exp(tx)$, $x\in \g$ and 
taking the derivative in $t$ at zero, by Proposition \ref{commprop}(iv) we get $[x,y]\in \h$, i.e., $\h$ is a Lie ideal. 

(iii) If $x\in \g$, $y\in \h$ are small then 
$$
\exp(x)\exp(y)\exp(x)^{-1}=
$$
$$
\exp({\rm Ad}_{\exp(x)}y)=\exp(\exp({\rm ad}x)y)=
\exp(\sum_{n=0}^{\infty}\tfrac{({\rm ad}x)^ny}{n!})\in H
$$
since $\sum_{n=0}^{\infty}\tfrac{({\rm ad}x)^ny}{n!}\in \h$. So $G$ acting on itself by conjugation maps a small neighborhood of $1$ in $H$ into $H$ (as $G$ is generated by its neighborhood of $1$ by Proposition \ref{generati}, since it is connected). But $H$ is also connected, so is generated by its neighborhood of $1$, again by Proposition \ref{generati}. Hence $H$ is normal. 
\end{proof} 

\subsection{The Lie algebra of vector fields} 

Recall that a vector field on a manifold $X$ is a compatible family of derivations 
$\bold v: O(U)\to O(U)$ for open subsets $U\subset X$. 

\begin{proposition} If $\bold v,\bold w$ are derivations of an algebra $A$ then so is $[\bold v,\bold w]:=\bold v\bold w-\bold w\bold v$. 
\end{proposition} 

\begin{proof} We have 
$$
(\bold v\bold w-\bold w\bold v)(ab)=\bold v (\bold w(a)b+a\bold w(b))-\bold w (\bold v(a)b+a\bold v(b))=
$$
$$
\bold v\bold w(a)b+\bold w(a)\bold v(b)+\bold v(a)\bold w(b)+a\bold v\bold w(b)
$$
$$
-\bold w\bold v(a)b-\bold v(a)\bold w(b)-\bold w(a)\bold v(b)-a\bold w\bold v(b)=
$$
$$
(\bold v\bold w-\bold w\bold v)(a)b+a(\bold v\bold w-\bold w\bold v)(b).
$$
\end{proof} 

Thus, the space ${\rm Vect}(X)$ of vector fields on $X$ is a Lie algebra 
under the operation 
$$
\bold v,\bold w\mapsto [\bold v,\bold w],
$$ 
called the {\bf Lie bracket of vector fields}.\footnote{Note that this Lie algebra is infinite dimensional for all real manifolds and many (but not all) complex manifolds of positive dimension.}

In local coordinates we have 
$$
\bold v=\sum_i v_i\frac{\partial}{\partial x_i},\ \bold w=\sum w_j\frac{\partial}{\partial x_j},
$$
so
$$
[\bold v,\bold w]=\sum_i\left(\sum_j (v_j\tfrac{\partial w_i}{\partial x_j}- w_j\tfrac{\partial v_i}{\partial x_j})\right)\tfrac{\partial}{\partial x_i}.
$$
This implies that if vector fields $\bold v,\bold w$ are tangent to a $k$-dimensional submanifold $Y\subset X$ then so is their Lie bracket 
$[\bold v,\bold w]$. Indeed, in local coordinates $Y$ is given by equations 
$x_{k+1}=...=x_n=0$, and in such coordinates a vector field is tangent to $Y$ iff it does not contain terms with $\tfrac{\partial}{\partial x_j}$ for $j>k$. 

\begin{exercise} Let $U\subset \Bbb R^n$ be an open subset, $\bold v,\bold w\in \Vect(U)$ and $g_t,h_t$ be the associated flows, defined in 
a neighborhood of every point of $U$ for small $t$. 
Show that for any $\bold x\in U$
$$
\lim_{t,s\to 0}\frac{g_th_sg_t^{-1}h_s^{-1}(\bold x)-\bold x}{ts}=-[\bold v,\bold w](\bold x).
$$
\end{exercise} 

Now let $G$ be a Lie group and $\Vect_L(G),\Vect_R(G)\subset \Vect(G)$ be the subspaces of left and right invariant vector fields. 

\begin{proposition} $\Vect_L(G),\Vect_R(G)\subset \Vect(G)$ are Lie subalgebras which are both canonically isomorphic to $\g={\rm Lie}G$. 
\end{proposition}  

\begin{proof} The first statement is obvious, so we prove only the second statement. 
Let $\bold x, \bold y\in \Vect_L(G)$. 
Then $\bold x=\bold L_x$, $\bold y=\bold L_y$ for 
$x=\bold x(1),y=\bold y(1)\in \g$, where $\bold L_z$ 
denotes the vector field on $G$ obtained by left translations of $z\in \g$. 
Then $[\bold L_x,\bold L_y]=\bold L_z$, where 
$z=[\bold L_x,\bold L_y](1)$. So let us compute $z$. 

Let $f$ be a regular function on a neighborhood of $1\in G$. We have shown that for $u\in \g$ 
$$
(\bold L_uf)(g)=\tfrac{d}{dt}|_{t=0}f(g\exp(tu)). 
$$
Thus, 
$$
z(f)=x(\bold L_yf)-y(\bold L_xf)=x(\tfrac{\partial}{\partial s}|_{s=0}f(\bullet \exp(sy)))-y(\tfrac{\partial}{\partial t}|_{t=0}f(\bullet\exp(tx)))=
$$
$$
\tfrac{\partial}{\partial t}|_{t=0}\tfrac{\partial}{\partial s}|_{s=0}f(\exp(tx)\exp(sy))-\tfrac{\partial}{\partial s}|_{s=0}\tfrac{\partial}{\partial t}|_{t=0}
f(\exp(sy)\exp(tx))=
$$
$$
\tfrac{\partial^2}{\partial t\partial s}|_{t=s=0} (F(tx+sy+\tfrac{1}{2}ts[x,y]+...)-F(tx+sy-\tfrac{1}{2}ts[x,y]+...)), 
$$
where $F(u):=f(\exp(u))$. It is easy to see by using Taylor expansion that this expression 
equals to $[x,y](f)$. Thus $z=[x,y]$, i.e., the map $\g\to \Vect_L(G)$ given by 
$x\mapsto \bold L_x$ is a Lie algebra isomorphism. Similarly, the map 
$\g\to \Vect_R(G)$ given by $x\mapsto -\bold R_x$ is a Lie algebra isomorphism, as claimed.   
\end{proof} 

\section{\bf Fundamental theorems of Lie theory}

\subsection{Proofs of Theorem \ref{closedlie}, Proposition \ref{orstab}, Proposition \ref{firsthom}} \label{proofs}

Let $G$ be a Lie group with Lie algebra $\g$ and $X$ be a manifold with an action $a: G\times X\to X$. 
Then for any $z\in \g$ we have a vector field $a_*(z)$ on $X$ given by 
$$
(a_*(z)f)(x)=\tfrac{d}{dt}|_{t=0}f(\exp(-tz)x),
$$
where $t\in \Bbb R$, $f\in O(U)$ for some open set $U\subset X$ and $x\in U$. 

\begin{proposition}\label{linea} The map $a_*$ is linear and we have 
$$
a_*([z,w])=[a_*(z),a_*(w)].
$$
In other words, the map $a_*:\g\to \Vect(X)$ is a homomorphism of Lie algebras. 
\end{proposition} 

\begin{exercise} Prove Proposition \ref{linea}.
\end{exercise}

This motivates the following definition. 

\begin{definition} An action of a Lie algebra $\g$ on a manifold $X$ is a homomorphism 
of Lie algebras $\g\to \Vect(X)$. 
\end{definition} 

Thus an action of a Lie group $G$ on $X$ induces an action of the Lie algebra $\g={\rm Lie}G$ on $X$. 

Now let $x\in X$. Then we have a linear map $a_{*x}: \g\to T_xX$ given by 
$a_{*x}(z):=a_*(z)(x)$.

\begin{theorem}\label{closedlie1} (i) The stabilizer $G_x$ is a closed subgroup of $G$ with Lie algebra 
$$
\g_x:={\rm Ker}(a_{*x}). 
$$

(ii) The map $G/G_x\to X$ given by $g\mapsto gx$ is an immersion. So the orbit $Gx$ is an immersed submanifold of $X$, and 
$$
T_x(Gx)\cong {\rm Im}(a_{*x})\cong \g/\g_x. 
$$ 
\end{theorem} 

Part (i) of Theorem \ref{closedlie1} is the promised weaker version of Theorem \ref{closedlie} 
sufficient for our purposes. Also, part (ii) implies Proposition \ref{orstab}. 

\begin{proof} (i) It is clear that $G_x$ is closed in $G$, but we need to show it is a Lie subgroup 
and compute its Lie algebra.\footnote{Although we claimed in Theorem \ref{closedlie} that a closed subgroup of a Lie group is always a Lie subgroup, we did not prove it, so we need to prove it in this case.} It suffices to show that for some neighborhood $U$ of $1$ in $G$, $U\cap G_x$ 
is a (closed) submanifold of $U$ such that $T_1(U\cap G_x)=\g_x$. 

Note that $\g_x\subset \g$ is a Lie subalgebra, since the commutator of vector fields vanishing at $x$ also vanishes at $x$ (by the formula for commutator in local coordinates). Also, for any $z\in \g_x$, $\exp(-tz)x$ is a solution of the ODE $\gamma'(t)=a_{*\gamma(t)}(z)$ with initial condition $\gamma(0)=x$, and $\gamma(t)=x$ is such a solution, so by uniqueness of ODE solutions 
$\exp(-tz)x=x$, thus $\exp(-tz)\in G_x$. 

Now choose a complement $\mathfrak{u}$ of $\g_x$ in $\g$, so that $\g=\g_x\oplus \mathfrak{u}$. 
Then $a_{*x}: \mathfrak{u}\to T_xX$ is injective. By the implicit function theorem, the map $\mathfrak{u}\to X$ given by $u\mapsto\exp(u)x$ is injective for small $u$, so $\exp(u)\in G_x$ for small $u\in \mathfrak{u}$ if and only if $u=0$. 

But in a small neighborhood $U$ of $1$ in $G$, any element $g$ can be uniquely written as $g=\exp(u)\exp(z)$, where $u\in \mathfrak{u}$ and $z\in \g_x$; this follows from the inverse function theorem applied to the map $\mathfrak u\times \g_x \to G,\ (u,z)\mapsto \exp(u)\exp(z)$,
whose differential at $(0,0)$ is an isomorphism. So we see that $g\in G_x$ iff $u=0$, i.e., $\log(g)\in \g_x$. This shows that $U\cap G_x$ coincides with $U\cap \exp(\g_x)$, as desired. 

(ii) The same proof shows that we have an isomorphism
$T_1(G/G_x)\cong \g/\g_x=\mathfrak{u}$, so the injectivity of $a_{*x}: \mathfrak{u}\to T_xX$ 
implies that the map $G/G_x\to X$ given by $g\mapsto gx$ is an immersion, as claimed. 
\end{proof} 

\begin{corollary} (Proposition \ref{firsthom}) 
Let $\phi: G\to K$ be a morphism of Lie groups and $\phi_*: {\rm Lie}G\to {\rm Lie}K$ 
be the corresponding morphism of Lie algebras. Then $H:={\rm Ker}(\phi)$ is a closed normal Lie subgroup with Lie algebra $\h:={\rm Ker}(\phi_*)$, and the map $\overline{\phi}: G/H\to K$ is an immersion. Moreover, if ${\rm Im}\overline{\phi}$ is a submanifold of $K$ then it is a closed Lie subgroup, and 
we have an isomorphism of Lie groups $\overline{\phi}: G/H\cong {\rm Im}\overline{\phi}$. 
\end{corollary} 

\begin{proof} Apply Theorem \ref{closedlie1} to the action of $G$ on $X=K$ via $g\circ k=\phi(g)k$, and 
take $x=1$. 
\end{proof} 

\begin{corollary} Let $V$ be a finite dimensional representation of a Lie group $G$, and $v\in V$. Then 
the stabilizer $G_v$ is a closed Lie subgroup of $G$ with Lie algebra $\g_v:=\lbrace z\in \g: zv=0\rbrace$. 
\end{corollary} 

\begin{example} 
Let $A$ be a finite dimensional algebra (not necessarily associative, e.g. a Lie algebra). 
Then the group $G={\rm Aut}(A)\subset GL(A)$ is a closed Lie subgroup 
with Lie algebra ${\rm Der}(A)\subset \End(A)$ of derivations of $A$, i.e., linear maps 
$d: A\to A$ such that 
$$
d(ab)=d(a)\cdot b+a\cdot d(b).
$$
Indeed, consider the action of $GL(A)$ on $\Hom(A\otimes A,A)$. 
Then $G=G_\mu$ where $\mu: A\otimes A\to A$ is the multiplication  map. 
Also, if $g_t$ is a smooth family of automorphisms of $A$ 
such that $g_0={\rm id}$ (i.e., $g_t(ab)=g_t(a)g_t(b)$) and $d=\frac{d}{dt}|_{t=0}g_t$ then 
$d(ab)=d(a)\cdot b+a\cdot d(b)$,
and conversely, if $d$ is a derivation then $g_t:=\exp(td)$ is an automorphism. 
\end{example} 

\subsection{The center of $G$ and $\g$} 

Let $G$ be a Lie group with Lie algebra $\g$ and $Z=Z(G)$ the center of $G$, i.e. the set of $z\in G$ such that $zg=gz$ for all $g\in G$. 
Also let $\mathfrak{z}=\mathfrak{z}(\g)$ be the set of $x\in \g$ such that $[x,y]=0$ for all $y\in \g$; it is called the {\bf center} of $\g$. 

\begin{proposition} If $G$ is connected then $Z$ is a closed (normal, commutative) Lie subgroup of $G$ with Lie algebra $\mathfrak{z}$. 
\end{proposition} 

\begin{proof} Since $G$ is connected, an element $g\in G$ belongs to $Z$ iff it commutes with $\exp(tu)$ for all $u\in \g$, i.e., iff ${\rm Ad}_g(u)=u$. Thus $Z={\rm Ker}({\rm Ad})$, where 
${\rm Ad}: G\to GL(\g)$ is the adjoint representation. 
Thus by Proposition \ref{firsthom}, $Z\subset G$ is a closed Lie subgroup with 
Lie algebra ${\rm Ker}({\rm ad})$, as claimed.  
\end{proof} 

\begin{remark} In general (when $G$ is not necessarily connected), it is easy to show that $G/G^\circ$ acts on $\mathfrak{z}$, and $Z$ is a closed Lie subgroup of $G$ with Lie algebra $\mathfrak{z}^{G/G^\circ}$ 
(the subspace of invariant vectors). 
\end{remark}

\begin{definition} For a connected Lie group $G$, the group $G/Z(G)$ is called the {\bf adjoint group} of $G$.
\end{definition} 

It is clear that $G/Z(G)$ is naturally isomorphic to the image of the adjoint representation 
${\rm Ad}: G\to GL(\g)$, which motivates the terminology. 

\subsection{The statements of the fundamental theorems of Lie theory} 

\begin{theorem}\label{first} (First fundamental theorem of Lie theory) For a Lie group $G$, there is a bijection between connected Lie subgroups $H\subset G$ and Lie subalgebras $\h\subset \g={\rm Lie}G$, given by 
$\h={\rm Lie}H$. 
\end{theorem} 

\begin{theorem}\label{second} (Second fundamental theorem of Lie theory)
If $G$ and $K$ are Lie groups with $G$ simply connected 
then the map 
$$
\Hom(G,K)\to \Hom({\rm Lie}G,{\rm Lie} K)
$$ 
given by $\phi\mapsto \phi_*$ is a bijection. 
\end{theorem} 

\begin{theorem} \label{third} (Third fundamental theorem of Lie theory) Any finite dimensional Lie algebra is the Lie algebra of a Lie group. 
\end{theorem}

These theorems hold for real as well as complex Lie groups. Thus we have 

\begin{corollary} For $\Bbb K=\Bbb R,\Bbb C$, the assignment $G\mapsto {\rm Lie}G$ is an equivalence between the category of simply connected $\Bbb K$-Lie groups and the category of finite dimensional $\Bbb K$-Lie algebras. 
Moreover, any connected Lie group $K$ has the form $G/\Gamma$ where $G$ is simply connected and $\Gamma\subset G$ is a discrete central subgroup. 
\end{corollary}

\begin{proof} The second fundamental theorem says that the functor $G\mapsto {\rm Lie}G$ is fully faithful, and the third fundamental theorem says that 
it is essentially surjective. Thus it is an equivalence of categories. The last statement follows from Proposition \ref{abeli} ($G$ is the universal covering of $K$). 
\end{proof} 

We will discuss proofs of the fundamental theorems of Lie theory in Subsection \ref{prooffund}. 
The third theorem is the hardest one, and we will give its complete proof only in Section \ref{thirdlie}. 

\subsection{Complexification of real Lie groups and real forms of complex Lie groups} 

Let $\mathfrak{k}$ be a real Lie algebra. Then $\mathfrak{k}_{\Bbb C}:=\mathfrak{k}\otimes_{\Bbb R}\Bbb C$ is a complex Lie algebra. We say that $\g:=\mathfrak{k}_{\Bbb C}$ is the {\bf complexification} of $\mathfrak{k}$, and $\mathfrak{k}$ is a {\bf real form} of $\g$.  
Thus a real form of $\g$ is a real Lie subalgebra $\mathfrak{k}\subset \g$ 
such that the natural map $\mathfrak{k}\otimes_{\Bbb R}\Bbb C\to \g$ is an isomorphism. 

In this case we have an antilinear involution $\sigma: \g\to \g$ given by $\sigma(a+ib)=a-ib$ for $a,b\in \mathfrak{k}$, and 
$\mathfrak{k}:=\g^\sigma$ is the set of fixed points of $\sigma$. 
Conversely, it is easy to see that if $\sigma$ is an antilinear involution of a complex Lie algebra $\g$ (i.e., an automorphism as a real Lie algebra such that $\sigma^2=1$ and $\sigma(\lambda a)=\overline \lambda \sigma(a)$ for $a\in \g,\lambda\in \Bbb C$), then $\mathfrak{k}:=\g^\sigma\subset \g$ is a real form of $\g$.  
Thus real forms of a complex Lie algebra are in natural bijection 
with its antilinear involutions. 

Note that two non-isomorphic real Lie algebras can have isomorphic complexifications; in other words, the same complex Lie algebra can have non-isomorphic real forms. For example, 
$$
\mathfrak{u}(n)_{\Bbb C}\cong \mathfrak{gl}_n(\Bbb R)_{\Bbb C}\cong \mathfrak{gl}_n(\Bbb C)
$$
while for $n>1$, 
$$
\mathfrak{u}(n)\ncong \mathfrak{gl}_n(\Bbb R),
$$
since in the first algebra any element $x$ with nilpotent ${\rm ad}x$ must satisfy ${\rm ad}x=0$, while in the second one it does not have to. 

Let us now discuss real forms of complex Lie groups. By analogy with the case of Lie algebras, we make the following definition. 

\begin{definition} Let $G$ be a complex Lie group with Lie algebra $\g$ and $\sigma: G\to G$ be an involutive automorphism of $G$ as a real Lie group such that the induced map $\sigma: \g\to \g$ is antilinear (i.e., $\sigma$ is antiholomorphic).   
Then the fixed point subgroup $K:=G^\sigma$ is called 
a {\bf real form} of $G$ and $G$ is called a {\bf complexification} of $K$.\footnote{Note that this definition is not quite equivalent to Definition 3.51 in \cite{K} of the same notion, which is less conventional. For example, according to the definition of \cite{K}, every complex elliptic curve has a real form, which does not agree with the definition from algebraic geometry (cf. Example \ref{ellcu}).} 
\end{definition} 

Note that a real Lie group $K$ may not admit a complexification. For example, Exercise \ref{unico} shows that this happens if $K^\circ\cong \widetilde{SL_2(\Bbb R)}$, the universal cover of $SL_2(\Bbb R)$. 
On the other hand, 
Example \ref{ellcu} shows that $K$ may admit several
(in fact, infinitely many) non-isomorphic complexifications. 

For example, both $U(n)$ and $GL_n(\Bbb R)$ are real forms of $GL_n(\Bbb C)$, with $\sigma(g)=(\overline{g}^T)^{-1}$ and $\sigma(g)=\overline g$ respectively. Note that $GL_n(\Bbb R)$ is not connected, so a real form of a connected Lie group may be disconnected. 

We see that every real form (i.e., antilinear involution) of $\g$ defines at most one such form for $G$. However, it could be none 
since the involution $\sigma: \g\to \g$ may not lift to $G$. This is demonstrated by the following example. 

\begin{example}\label{ellcu} Let $\Lambda\subset \Bbb C$ be a lattice generated by $1$ and $\tau\in \Bbb C$ with ${\rm Im}\tau>0$, $-\frac{1}{2}<{\rm Re}\tau\le \frac{1}{2}$, $|\tau|\ge1$, and 
let $E:=\Bbb C/\Lambda$ be the corresponding complex elliptic curve
(a 1-dimensional complex Lie group). This normalization is achieved by rescaling the lattice and changing its basis. We have ${\rm Lie}E=\Bbb C$, whose antilinear involutions are $\sigma_u(z)=u\overline z$, $|u|=1$. The condition for such an involution to lift to $E$ is $u\overline\Lambda=\Lambda$. Hence $E$ has a real form if and only if
\[
{\rm Re}\tau=0,\qquad {\rm Re}\tau=\tfrac12,
\qquad\text{or}\qquad |\tau|=1.
\]
Indeed, write $u=a+c\tau$ with $a,c\in\Bbb Z$. Since ${\rm Im}\tau\ge\sqrt3/2$ and $|c|{\rm Im}\tau\le|u|=1$, we have $c=0,\pm1$. If $c=0$, then $u=\pm1$ and $2{\rm Re}\tau\in\Bbb Z$; if $|c|=1$, the normalization and $|a+c\tau|=1$ force $|\tau|=1$. Conversely, take $u=1$ in the first two cases and $u=\tau$ in the third. This coincides with the definition of a real elliptic curve in algebraic geometry saying that $E$ can be defined by a Weierstrass equation $y^2=P(x)$ where $P$ is a cubic polynomial with real coefficients (check it!). For $\sigma(z)=\overline z$, there are two cases: $\tau\in i\Bbb R$ ($P$ has three real roots) and $\tau\in \frac{1}{2}+i\Bbb R$ ($P$ has one real root). In the first case the corresponding real group $E^\sigma$ is $\Bbb Z/2\times\Bbb R/\Bbb Z$ (the two components are the images of $\Bbb R$ and $\Bbb R+\frac{1}{2}\tau$), while in the second case it is $\Bbb R/\Bbb Z$ (the image of $\Bbb R$). For $|\tau|=1$, the fixed group of $\sigma_\tau(z)=\tau\overline z$ is $\Bbb R/\Bbb Z$, parametrized by $t\mapsto t(1+\tau)\bmod\Lambda$. 
\end{example}  
 
However, if $G$ is a simply connected complex Lie group, then every 
real form of $\g$ necessarily defines one for $G$. Indeed, in this case
by the second fundamental theorem of Lie theory (for real Lie groups), the antilinear involution $\sigma:\g\to \g$ lifts to an antiholomorphic involution $G\to G$. 

\begin{exercise} (i) Classify complex Lie algebras of dimension at most 3, up to isomorphism. 

(ii) Classify real Lie algebras of dimension at most 3.

(iii) Classify connected complex and real Lie groups of dimension at most 3. 
\end{exercise}

\section{\bf Proofs of the fundamental theorems of Lie theory} 

\subsection{Distributions and the Frobenius theorem} 

The proofs of the fundamental theorems of Lie theory are based on the notion of an integrable distribution in differential geometry, and the Frobenius theorem about such distributions. 

\begin{definition} A $k$-dimensional {\bf distribution} on a manifold $X$ is a rank $k$ subbundle $D\subset TX$. 
\end{definition} 

This means that in every tangent space $T_xX$ we fix a $k$-dimensional subspace $D_x$ which varies regularly with $x$. In other words, on some neighborhood $U\subset X$ of every $x\in X$, $D$ is spanned by vector fields $\bold v_1,...,\bold v_k$ linearly independent at every point of $U$. 

\begin{definition} A $k$-dimensional distribution $D$ is {\bf integrable} if every point $x\in X$ has a neighborhood $U$ and local coordinates $x_1,\dots,x_n$ on $U$ such that 
$D$ is defined at every point of $U$ by the equations 
$$
dx_{k+1}=\dots=dx_n=0,
$$
i.e., it is spanned by the vector fields 
$$
\partial_i=\tfrac{\partial}{\partial x_i},\qquad i=1,\dots,k.
$$
\end{definition}  

By definition, every $x\in X$ is contained in a $k$-dimensional disk $B_x\subset X$ tangent to $D$ at all its points. 

Now let $D$ be a $k$-dimensional integrable distribution on a manifold $X$. For $x,y\in X$, let us say that 
$x\sim_D y$ if $x$ can be connected to $y$ by a piecewise smooth curve whose tangent vector at every point where it is defined belongs to $D$. This is clearly an equivalence relation. Denote the equivalence class of $x$ under this relation by $S_x$. 

Choose a countable cover $\{U_i\}_{i\in \Bbb N}$ of $X$ by sufficiently small coordinate balls for $D$. 
Then for every $i$, the set $U_i\cap S_x$ is a union of pairwise disjoint $k$-dimensional disks tangent to $D$ (in the chosen coordinates these are the sets with fixed values of $x_{k+1},\dots,x_n$). Call these disks the {\it plaques} of $S_x$ in $U_i$. 

We endow $S_x$ with the topology for which all plaques are open and carry their usual topology. Then the inclusion $S_x\hookrightarrow X$ is continuous (so this topology may be stronger than the induced topology). Then $S_x$ is Hausdorff since so is $X$. 

\begin{lemma} This endows $S_x$ with the structure of a smooth manifold. Thus $S_x\subset X$ is an immersed submanifold. 
\end{lemma} 

\begin{proof} 
The plaques are $k$-dimensional disks, so they define local coordinate charts on $S_x$. If two plaques meet, the transition map between them is the restriction of the corresponding change of coordinates on $X$, hence is smooth. Thus $S_x$ is locally Euclidean of dimension $k$ (with smooth transition maps). So it remains to show that $S_x$ has a countable base. 

For this, it suffices to show that for every $i$, the set $U_i\cap S_x$ contains at most countably many plaques. Fix $i_0$ such that $U_{i_0}\cap S_x\ne \emptyset$, and let $B_0\subset U_{i_0}\cap S_x$ be one of the plaques. Let $B\subset U_i\cap S_x$ be another plaque. Pick points $y_0\in B_0$ and $y\in B$, and let $\gamma:[0,1]\to X$ be a piecewise smooth curve from $y_0$ to $y$ tangent to $D$. By compactness of $[0,1]$, there exist numbers
$$
0=t_0<t_1<\dots<t_N=1
$$
and indices $m_0=i_0,m_1,\dots,m_N=i$ such that
$$
\gamma([t_{r-1},t_r])\subset U_{m_r},\qquad r=1,\dots,N,
$$
and 
$$
\gamma(t_r)\in U_{m_r}\cap U_{m_{r+1}},\qquad r=0,\dots,N-1.
$$
We call the sequence $m_0,\dots,m_N$ an {\it itinerary} of $\gamma$. 

Since each segment $\gamma([t_{r-1},t_r])$ is tangent to $D$ and lies in the distinguished chart $U_{m_r}$, it lies in a single plaque of $U_{m_r}\cap S_x$. Moreover, a plaque $P$ in $U_{m_r}$ meets at most countably many plaques in $U_{m_{r+1}}$: the open subset $P\cap U_{m_{r+1}}$ of the disk $P$ has countably many connected components, and the transverse coordinates are constant on each component. Therefore, starting from $B_0$, each itinerary permits at most countably many final plaques $B$. 

Since there are only countably many possible itineraries, it follows that $U_i\cap S_x$ contains at most countably many plaques. Hence the union of countable bases for all plaques in all $U_i$ is a countable base of the topology of $S_x$. So $S_x$ is a smooth manifold, and the inclusion $S_x\hookrightarrow X$ is an immersion. 
\end{proof} 

The immersed submanifold $S_x\subset X$ is called the {\bf integral submanifold} for $D$ through $x$. 

\begin{remark} An integrable distribution is also called a {\bf foliation}, and the integral submanifolds $S_x$ are called the {\bf leaves} of the foliation. Thus the manifold $X$ falls into a disjoint union of such leaves. But note that the leaves need not be closed (think of the irrational torus winding!).  
\end{remark}

\begin{example} A $1$-dimensional distribution is the same thing as a {\bf direction field.} It is always integrable, as follows from the existence theorem 
for ODE, and its integral submanifolds are called {\bf integral curves}. 
They are geometric realizations of solutions of the corresponding ODE. 
\end{example} 

However, for $k\ge 2$ a distribution is not always integrable. 

\begin{theorem} (The Frobenius theorem) A distribution $D$ is integrable if and only if for every two vector fields $\bold v,\bold w$ contained in $D$, 
their commutator $[\bold v,\bold w]$ is also contained in $D$.
\end{theorem}

\begin{example} Let $\bold v=\partial_x$, $\bold w=x\partial_y+\partial_z$ in $\Bbb R^3$, and $D$ be the 2-dimensional distribution spanned by 
$\bold v,\bold w$. Then $[\bold v,\bold w]=\partial_y\notin D$. So $D$ is not integrable. 
\end{example} 

\begin{proof} If $D$ is integrable, a vector field is contained in $D$ iff  
it is tangent to integral submanifolds of $D$. But the commutator of two vector fields tangent to a submanifold is itself tangent to this submanifold. This establishes the ``only if" part. 
 
It remains to prove the ``if " part. The proof is by induction in the rank $k$ of $D$. The base case $k=0$ is trivial, so it suffices to establish the inductive step.  
The question is local, so we may work in a neighborhood $U$ of $P\in X$. 
Suppose that $\bold v_1,...,\bold v_k\in \Vect(U)$ is a basis of $D$ in $U$ (on every tangent space). By local existence and uniqueness of solutions of ODE, 
in some local coordinates $x_1,...,x_n=z$, the vector field $\bold v_k$ equals $\partial_z$. By subtracting from $\bold v_i, i<k$ a suitable multiple of $\bold v_k$ we can make sure that $\bold v_i$ has no $\partial_z$-component. Then 
$$
\bold v_i=\sum_{j=1}^{n-1}a_{ij}(x_1,...,x_{n-1},z)\partial_{x_j}.
$$
Thus, since by assumption $[\partial_z,\bold v_i]=[\bold v_k,\bold v_i]$ is a linear combination of $\bold v_m$ with functional coefficients, we have 
$$
[\partial_z,\bold v_i]=\sum_{m=1}^{k-1} b_{im}(x_1,...,x_{n-1},z)\bold v_m
$$
($\bold v_k$ does not occur since there is no $\partial_z$ component on the left hand side). Hence
$$
\partial_z a_{ij}(x_1,...,x_{n-1},z)=\sum_{m=1}^{k-1} b_{im}(x_1,...,x_{n-1},z)a_{mj}(x_1,...,x_{n-1},z).
$$
So, setting $A=(a_{mj}(x_1,...,x_{n-1},z))$ (a $(k-1)\times (n-1)$-matrix) and $B=(b_{im}(x_1,...,x_{n-1},z))$ (a $(k-1)\times (k-1)$ matrix), we have 
$$
\partial_z A=BA. 
$$
Let $A_0$ be the solution of this linear ODE in $(k-1)\times (k-1)$ matrices with $A_0(x_1,...,x_{n-1},0)=1$. Then $A=A_0C$, where $C=C(x_1,...,x_{n-1})$ 
is a $(k-1)\times (n-1)$-matrix which does not depend on $z$. 
So we have a new basis of $D$ given by $\bold w_k=\partial_z$ and
$$
\bold w_i=\sum_j c_{ij}(x_1,...,x_{n-1})\partial_{x_j},\ 1\le i\le k-1.
$$
Thus there is a neighborhood $U$ of $P$ which can be represented as $U=(-a,a)\times U'$, where $\dim U'=n-1$, so that $D=\Bbb R\oplus D'$, where $D'$ is a $k-1$-dimensional 
distribution on $U'$ spanned by $\bold w_i$, $1\le i\le k-1$. It is clear that for any two vector fields $\bold v,\bold w$ on $U'$ contained
in $D'$, so is $[\bold v,\bold w]$. Hence $D'$ is integrable by the induction assumption. Therefore, so is $D$, justifying the inductive step.  
  
\end{proof} 

\subsection{Proofs of the fundamental theorems of Lie theory} \label{prooffund} 

\subsubsection{Proof of Theorem \ref{first}}

Let $G$ be a Lie group with Lie algebra $\g$. Let $\h\subset \g$ be a Lie subalgebra. We need to show that there is a unique (not necessarily closed) connected Lie subgroup $H\subset G$ with Lie algebra $\h$. The proof of existence of $H$ is based on the Frobenius theorem. 

Define the distribution $D$ on $G$ by left-translating $\h\subset \g=T_1G$, i.e., $D_g=L_g\h$. 
So any vector field contained in $D$ is of the form 
$$
\bold v=\sum f_i\bold L_{a_i},
$$
where $a_i$ is a basis of $\h$ and $f_i$ are regular functions. 
Now if 
$$
\bold w=\sum g_j\bold L_{a_j}
$$
is another such field then 
$$
[\bold v,\bold w]=\sum_{i,j}(f_i\bold L_{a_i}(g_j)\bold L_{a_j}-g_j\bold L_{a_j}(f_i)\bold L_{a_i}+f_ig_j[\bold L_{a_i},\bold L_{a_j}]).
$$
But $[a_i,a_j]=\sum_k c_{ij}^ka_k$,
so 
$$
[\bold L_{a_i},\bold L_{a_j}]=\sum_k c_{ij}^k\bold L_{a_k}.
$$
Thus if $\bold v,\bold w$ are contained in $D$ then so is $[\bold v,\bold w]$.
Hence by the Frobenius theorem, $D$ is integrable. 

Now consider the integral (immersed) submanifold $H$ of $D$ going through $1\in G$. We claim that $H$ is a Lie subgroup of $G$ with Lie algebra $\h$. Indeed, it suffices to show that $H$ is a subgroup of $G$. But this is clear since $H$ is the collection of elements of $G$ of the form 
$$
g=\exp(a_1)...\exp(a_m),
$$
where $a_i\in \h$. 

Moreover, $H$ is unique since it has to be generated by the image of the 
exponential map $\exp: \h\to G$. 

\subsubsection{Proof of Theorem \ref{second}}
We need to show that the natural map $\Hom(G,K)\to \Hom({\rm Lie}G,{\rm Lie}K)$ 
is a bijection if $G$ is simply connected. 

We know this map is injective so we only need to establish surjectivity. For any morphism $\psi: {\rm Lie}G\to {\rm Lie}K$, consider the morphism 
$$
\theta=({\rm id},\psi): {\rm Lie}G\to {\rm Lie}(G\times K)={\rm Lie}G\oplus {\rm Lie}K
$$ 
The previous proposition implies that there is a connected Lie subgroup $H\subset G\times K$ whose Lie algebra 
is ${\rm Im}\theta$. We have projection homomorphisms 
$p_1: H\to G$, $p_2: H\to K$, and $(p_1)_*={\rm id}$, so by Proposition \ref{generati}(ii)
$p_1$ is a covering. Since $G$ is simply connected, $p_1$ is an isomorphism, so we can define $\phi:=p_2\circ p_1^{-1}: G\to K$, and it is easy to see that $\psi=\phi_*$.

\subsubsection{Proof of Theorem \ref{third}} 
Finally, let us discuss a proof of Theorem \ref{third}, stating that any finite dimensional Lie algebra $\g$ over $\Bbb K=\Bbb R$ or $\Bbb C$ is the Lie algebra of a Lie group. We will deduce it from the following purely algebraic {\it Ado's theorem}.

\begin{theorem} Any finite dimensional Lie algebra over $\Bbb K$ is a Lie subalgebra of $\mathfrak{gl}_n(\Bbb K)$. 
\end{theorem}  

Ado's theorem in fact holds over any ground field, but it is rather nontrivial and we won't prove it now. A proof can be found, for example, in \cite{J}. But Ado's theorem immediately implies Theorem \ref{third}. Indeed, using Theorem \ref{first}, Ado's theorem implies the following even stronger statement: 

\begin{theorem} Any finite dimensional $\Bbb K$-Lie algebra is the Lie algebra of a Lie subgroup of $GL_n(\Bbb K)$ for some $n$. 
\end{theorem} 

This implies 

\begin{corollary} Any simply connected Lie group is the universal covering of a linear Lie group, i.e., of a Lie subgroup of $GL_n(\Bbb K)$.  
\end{corollary} 

However, it is not true that any Lie group is isomorphic to a Lie subgroup of $GL_n(\Bbb K)$, see Exercise \ref{unico}.

One can also prove Theorem \ref{third} directly and then deduce Ado's theorem as a corollary. We will do this in Sections \ref{thirdlie} and \ref{adothm}. 
We note that Theorem \ref{third} will not be used in proofs of other results 
until that point. 

\section{\bf Representations of Lie groups and Lie algebras} 

\subsection{Representations} 

We have previously defined (finite dimensional) representations of Lie groups and (iso)morphisms between them. We can do the same for Lie algebras: 

\begin{definition} A {\bf representation of a Lie algebra} $\g$ over a field $\bold k$ (or a $\g$-module) is a vector space $V$ over $\bold k$ equipped with a homomorphism of Lie algebras $\rho=\rho_V: \g\to \mathfrak{gl}(V)$. A {\bf (homo)morphism of representations} $A: V\to W$ (also called an {\bf intertwining operator}) is a linear map which commutes with the $\g$-action: $A\rho_V(b)=\rho_W(b)A$ for $b\in \g$. Such $A$ is an isomorphism if it is an isomorphism of vector spaces. 
\end{definition} 

The first and second fundamental theorems of Lie theory imply: 

\begin{corollary} Let $G$ be a Lie group and $\g={\rm Lie}G$. 

(i) Any finite dimensional representation $\rho: G\to GL(V)$
gives rise to a Lie algebra representation 
$\rho_*: \g\to \mathfrak{gl}(V)$, and any morphism of $G$-representations is also a morphism of $\g$-representations. 

(ii) If $G$ is connected then any morphism of $\g$-representations is a morphism of $G$-representations. 

(iii) If $G$ is simply connected then the assignment $\rho\mapsto \rho_*$ 
is an equivalence of categories $\Rep G\to \Rep \g$ between the corresponding categories of finite dimensional representations. In particular, any finite dimensional representation of the Lie algebra $\g$ can be uniquely exponentiated to the group $G$. 
\end{corollary}  

\begin{example} 1. The trivial representation: $\rho(g)=1, g\in G$, $\rho_*(x)=0$, $x\in \g$. 

2. The adjoint representation: $\rho(g)={\rm Ad}_g, \rho_*(x)={\rm ad}x$. 
\end{example} 

\begin{exercise} Let $\g$ be a complex Lie algebra regarded as a real one. Show that $\g_\Bbb C\cong \g\oplus \overline\g$, where $\overline\g$ is the conjugate complex Lie algebra. Deduce that if $G$ is a simply connected complex Lie group then 
$\Rep_{\Bbb R} G\cong \Rep(\g\oplus \overline\g)$, where $\Rep_{\Bbb R} G$ is the category of finite dimensional complex representations of $G$ regarded as a real Lie group. 
\end{exercise} 

As usual, a {\bf subrepresentation} of a representation $V$ is a subspace $W\subset V$ invariant under the $G$-action (resp. $\g$-action). In this case the quotient space 
$V/W$ has a natural structure of a representation, called the {\bf quotient representation}. The notion of {\bf direct sum} of representations is defined in an obvious way: 
$$
\rho_{V\oplus W}=\rho_V\oplus \rho_W. 
$$
Also we have the notion of {\bf dual representation}: 
$$
\rho_{V^*}(g)=\rho_V(g^{-1})^*, g\in G;\ \rho_{V^*}(x)=-\rho_V(x)^*, x\in \g,
$$
and {\bf tensor product}: 
$$
\rho_{V\otimes W}(g)=\rho_V(g)\otimes \rho_W(g),\ \rho_{V\otimes W}(x)=\rho_V(x)\otimes 1_W+1_V\otimes \rho_W(x).
$$
Thus we have the notion of {\bf symmetric and exterior powers} $S^mV,\wedge^m V$ of a representation $V$, which can be defined either as quotients or (over a field of characteristic zero) as subrepresentations of $V^{\otimes m}$. Also for representations $V,W$, $\Hom(V,W)$ is a representation 
via
$$
g\circ A=\rho_W(g)A\rho_V(g^{-1}),\ x\circ A=\rho_W(x)A-A\rho_V(x), 
$$ 
so if $V$ is finite dimensional then $\Hom(V,W)\cong V^*\otimes W$. 
Finally, for every representation $V$ we have the notion of invariants:
$$
V^G=\lbrace v\in V: gv=v\ \forall g\in G\rbrace,\ V^\g=\lbrace v\in V: xv=0\ \forall x\in \g\rbrace.
$$
Thus $V^G\subset V^\g$ and $V^G=V^\g$ for connected $G$ (in general, $V^G=(V^\g)^{G/G^\circ}$). Also $\Hom(V,W)^G\cong \Hom_G(V,W)$ 
and $\Hom(V,W)^\g=\Hom_\g(V,W)$, the spaces of intertwining operators. 
Note that in all cases the formula for Lie algebras is determined by the formula for groups by the requirement that these definitions should be consistent with the assignment $\rho\mapsto \rho_*$. 

\begin{definition} A representation $V\ne 0$ of $G$ or $\g$ is {\bf irreducible} if any subrepresentation $W\subset V$ is either $0$ or $V$ and is {\bf indecomposable} if for any decomposition $V\cong V_1\oplus V_2$, we have $V_1=0$ or $V_2=0$. 
\end{definition} 

It is clear that any finite dimensional representation is isomorphic to a direct sum of indecomposable representations (in fact, uniquely so up to order of summands by the {\it Krull-Schmidt theorem}).
However, not any $V$ is a direct sum of irreducible representations, e.g. 
$$
\rho: \Bbb C\to GL_2(\Bbb C),\ \rho(x)=\begin{pmatrix} 1& x\\ 0& 1\end{pmatrix}.
$$ 

\begin{definition}  
A representation $V$ is called {\bf completely reducible} if it is isomorphic to a direct sum of irreducible representations. 
\end{definition} 

Some of the main problems of representation theory are: 

1) Classify irreducible representations; 

2) If $V$ is a completely reducible representation, find its decomposition into irreducibles. 

3) For which $G$ are all representations completely reducible?  

\begin{example}\label{decoo} Let $V$ be a finite dimensional $\Bbb C$-representation of $\g$ or $G$  and $A: V\to V$ be a homomorphism of representations (e.g., defined by a central element). Then we have a decomposition of representations $V=\oplus_\lambda V(\lambda)$, where $V(\lambda)$ is the generalized eigenspace of $A$ with eigenvalue $\lambda$. 
\end{example} 

\begin{example} Let $V$ be the vector representation of $GL(V)$. Then $V$ is irreducible, and more generally so are $S^mV,\wedge^nV$ (show it!). Thus $V\otimes V$ is completely reducible: $V\otimes V\cong S^2V\oplus \wedge^2V$. 
\end{example} 

\subsection{Schur's lemma} 

\begin{lemma} (Schur's lemma) 
Let $V,W$ be irreducible finite dimensional complex representations of $G$ or $\g$. 
Then $\Hom_{G,\g}(V,W)=0$ if $V,W$ are not isomorphic, and every endomorphism 
of the representation $V$ is a scalar. 
\end{lemma} 

\begin{proof}  Let $A: V\to W$ be a nonzero morphism of representations. Then ${\rm Im}(A)\subset W$ 
is a nonzero subrepresentation, hence ${\rm Im}(A)=W$. Also ${\rm Ker}(A)\subset V$ is a proper subrepresentation, so ${\rm Ker}(A)=0$. Thus $A$ is an isomorphism, i.e., we may assume that $W=V$.
In this case, let $\lambda$ be an eigenvalue of $A$. Then $A-\lambda \cdot {\rm Id}: V\to V$ is a morphism of representations but not an isomorphism, hence it must be zero, so $A=\lambda \cdot {\rm Id}$.  
\end{proof} 

Note that the second statement of 
Schur's lemma (unlike the first one) does not hold over $\Bbb R$. For example, consider 
the rotation group $SO(2)$ (or any of its finite subgroups of order $>2$) acting on $V=\Bbb R^2$ by rotations. Then $\End(V)=\Bbb C\ne \Bbb R$. Similarly, if $V$ is the representation of $SU(2)$ on $\Bbb H$ defined by right multiplication by unit quaternions then $V$ is an irreducible real representation but 
$\End(V)=\Bbb H\ne \Bbb R$. For this reason, in representation theory of Lie groups 
and Lie algebras one usually considers complex representations. Thus from now on all representations we consider will be assumed complex unless specified otherwise.\footnote{An exception is the adjoint representation of a real Lie group and associated tensor representations, which are real.}

\begin{corollary} The center of $G,\g$ acts on an irreducible representation by a scalar. In particular, 
if $G$ or $\g$ is abelian then every irreducible representation of $G$ or $\g$
is 1-dimensional. 
\end{corollary}

\begin{example} Irreducible representations of $\Bbb R$ are $\chi_s$ given by $\chi_s(a)=\exp(sa)$, $s\in \Bbb C$. Irreducible representations of $\Bbb R^\times=\Bbb R_{>0}\times \Bbb Z/2$ 
are $\chi_{s,+}(a)=|a|^s$, $\chi_{s,-}(a)=|a|^s{\rm sign}(a)$. Irreducible representations 
of $S^1$ are $\chi_n(z)=z^n$, $n\in \Bbb Z$. Irreducible representations of the real group $\Bbb C^\times=\Bbb R_{>0}\times S^1$ are 
$\chi_{s,n}(z)=|z|^s(z/|z|)^n$, $s\in \Bbb C$, $n\in \Bbb Z$.  
\end{example} 

\begin{corollary} Let $V_i$ be irreducible and $V=\oplus_i n_iV_i, W=\oplus_i m_iV_i$ be completely reducible complex representations 
of $G$ or $\g$. Then we have a natural linear isomorphism
$$
\Hom_{G,\g}(V,W)\cong \oplus_i {\rm Mat}_{m_i,n_i}(\Bbb C).
$$ 
Moreover, if $V=W$ then this is an isomorphism of algebras. 
\end{corollary} 

\subsection{Unitary representations} 

A finite dimensional representation $V$ of $G$ is said to be {\bf unitary} if it is equipped with a positive definite Hermitian inner product $B(,)$ invariant under $G$, i.e., $B(gv,gw)=B(v,w)$ for
$v, w\in V$, $g\in G$.  

\begin{proposition} Any unitary representation can be written as an orthogonal direct sum of irreducible unitary representations. In particular, it is completely reducible. 
\end{proposition} 

\begin{proof} If $W\subset V$ is a subrepresentation of a unitary representation $V$ then 
let $W^\perp$ be its orthogonal complement under $B$. Then $W^\perp$ is also a subrepresentation since $B$ is invariant, and $V=W\oplus W^\perp$ since $B$ is positive definite. 

Now we can prove that $V$ is an orthogonal direct sum of irreducible unitary representations by induction in $\dim V$. The base $\dim V=1$ is clear so let us make the inductive step. Pick an irreducible $W\subset V$. Then $V=W\oplus W^\perp$, and $W^\perp$ 
is a unitary representation of dimension smaller than $\dim V$, so is an orthogonal direct sum of irreducible unitary representations by the induction assumption. 
\end{proof} 

\begin{proposition} Any finite dimensional representation $V$ of a finite group $G$ is unitary.  
Moreover, if $V$ is irreducible, the unitary structure is unique up to a positive factor. 
\end{proposition} 

\begin{proof} Let $B$ be any positive definite inner product on $V$. 
Let 
$$
\widehat{B}(v,w):=\sum_{g\in G}B(gv,gw). 
$$
Then $\widehat{B}$ is positive definite and invariant, so $V$ is unitary. 

If $V$ is irreducible and $B_1,B_2$ are two unitary structures on $V$ then 
$B_1(v,w)=B_2(Av,w)$ for some homomorphism $A: V\to V$. Thus by Schur's lemma 
$A=\lambda\cdot {\rm Id}$, and $\lambda>0$ since $B_1,B_2$ are positive definite.  
\end{proof} 

\begin{corollary} Every finite dimensional complex representation of a finite group $G$ is completely reducible. 
\end{corollary} 

\subsection{Representations of $\mathfrak{sl}_2$} \label{sl2rep} 

The Lie algebra $\mathfrak{sl}_2=\mathfrak{sl}_2(\Bbb C)$ has basis 
$$
e=\begin{pmatrix} 0& 1\\ 0& 0\end{pmatrix},\ 
h=\begin{pmatrix} 1& 0\\ 0& -1\end{pmatrix},\
f=\begin{pmatrix} 0& 0\\ 1& 0\end{pmatrix}
$$
with commutator 
$$
[e,f]=h,\ [h,e]=2e,\ [h,f]=-2f. 
$$
Since 2-by-2 matrices act on variables $x,y$, they also act on the space $V=\Bbb C[x,y]$ of polynomials 
in $x,y$. Namely, this action is given by the formulas
$$
e=x\partial_y,\ f=y\partial_x,\ h=x\partial_x-y\partial_y. 
$$
This infinite-dimensional representation has the form 
$V=\oplus_{n\ge 0}V_n$, where $V_n$ 
is the space of polynomials of degree $n$. 
The space $V_n$ is invariant under $e,f,h$, so 
it is an $n+1$-dimensional representation of $\mathfrak{sl}_2$. 
It has basis $v_{pq}=x^py^q$, such that 
$$
hv_{pq}=(p-q)v_{pq},\ ev_{pq}=qv_{p+1,q-1},\ fv_{pq}=pv_{p-1,q+1}.
$$
Thus $V_0$ is the trivial representation, and $V_1$ is the tautological representation by 2-by-2 matrices. 
Also it is easy to see that $V_2$ is the adjoint representation. 

\begin{theorem}\label{sl2repth} (i) $V_n$ is irreducible. 

(ii) If $V\ne 0$ is a finite dimensional representation of $\mathfrak{sl}_2$ then $e|_V$ and $f|_V$ are nilpotent, so $U:={\rm Ker}(e)\ne 0$. Moreover, $h$ preserves $U$ and acts diagonalizably on it, with nonnegative integer eigenvalues.  

(iii) Any irreducible finite dimensional representation $V$ of $\mathfrak{sl}_2$ is isomorphic to $V_n$ for some $n$. 

(iv) Any finite dimensional representation $V$ of $\mathfrak{sl}_2$ is 
completely reducible. 
\end{theorem} 

\begin{proof} (i) Let $W\subset V_n$ be a nonzero subrepresentation. Since it is $h$-invariant, it must be spanned by vectors $v_{p,n-p}$ for $p$ from a nonempty subset $S\subset [0,n]$. Since $W$ is $e$-invariant and $f$-invariant, if $m\in S$ then so are $m+1,m-1$ (if they are in $[0,n]$). Thus $S=[0,n]$ and $W=V_n$. 

(ii) Let $V$ be a finite dimensional representation of $\mathfrak{sl}_2$. 
We can write $V$ as a direct sum of generalized eigenspaces of $h$: 
$V=\oplus_\lambda V(\lambda)$. Since $he=e(h+2)$, $hf=f(h-2)$, we have $e: V(\lambda)\to V(\lambda+2)$, $f: V(\lambda)\to V(\lambda-2)$. 
Thus $e|_V$, $f|_V$ are nilpotent, so $U\ne 0$. 

If $v\in U$ then $e(hv)=(h-2)ev=0$, so $hv\in U$, i.e., $U$ is $h$-invariant. 

Given $v\in U$, consider the vector $v_m:=e^mf^mv$. We have 
\begin{gather}\label{efm}
ef^mv=fef^{m-1}v+hf^{m-1}v=fef^{m-1}v+f^{m-1}(h-2(m-1))v=...\\ \nonumber
=
f^{m-1}m(h-m+1)v. 
\end{gather}
Thus 
$$
v_m=e^{m-1}f^{m-1}m(h-m+1)v=m(h-m+1)v_{m-1}. 
$$
Hence
$$
v_m=m! h(h-1)...(h-m+1)v. 
$$
But for large enough $m$, $v_m=0$, since $f$ is nilpotent, so 
$$
h(h-1)...(h-m+1)v=0.
$$ 
Thus $h$ acts diagonalizably on $U$ with nonnegative integer eigenvalues. 

(iii) Let $v\in U$ be an eigenvector of $h$, i.e., 
$hv=\lambda v$. Let $w_m=f^mv$. Then 
$$
fw_m=w_{m+1}, hw_m=(\lambda-2m)w_m. 
$$
Also, it follows from \eqref{efm} that 
$$
ew_m=m(\lambda-m+1)w_{m-1}.
$$
Thus if $w_m\ne 0$ and $\lambda\ne m$ then $w_{m+1}\ne 0$. 
Also the nonzero vectors $w_m$ are linearly independent since they have different eigenvalues of $h$.  
Thus $\lambda=n$ must be a nonnegative integer (as also follows from (ii)), and $w_{n+1}=0$. So $V$, being irreducible, has a basis $w_m$, $m=0,...,n$. Now it is easy to see 
that $V\cong V_n$, via the assignment 
$$
w_m\mapsto n(n-1)...(n-m+1)x^{n-m}y^m.
$$ 

(iv) Consider the {\bf Casimir operator} 
$$
C=2fe+\frac{h^2}{2}+h.
$$
It is easy to check that $[C,e]=[C,f]=[C,h]=0$, so $C: V\to V$ is a homomorphism. 
Thus $C|_{V_n}=\frac{n(n+2)}{2}$ (it is a scalar by Schur's lemma, and acts with such eigenvalue on 
$v_{n0}\in V_n$); note that these are different for different $n$. For a general representation, we have $V=\oplus_c V_c$, the direct sum of generalized eigenspaces of $C$. 

Assume $V$ is indecomposable. Then by Example \ref{decoo} 
$C$ has a single eigenvalue $c$ on $V$. 
Fix a {\bf Jordan-H\"older filtration} on $V$, i.e. a filtration 
$$
0=F_0V\subset F_1V\subset...\subset F_mV=V
$$ 
such that $Y_i:=F_{i}V/F_{i-1}V$ are irreducible for all $i$. 
By (iii), for each $i$ we have $Y_i\cong V_n$ for some $n$, so $c=\frac{n(n+2)}{2}$ and thus this $n$ is the same for all $i$. Thus $V(k)$ has dimension $m$, with $h$ acting on it by $k\cdot {\rm Id}$ for $k=n,n-2,...,-n$ and $V(k)=0$ otherwise, by (ii); in particular, $\dim V=m(n+1)$. Let $u_1,...,u_m$ be a basis of $V(n)$. As in (iii), we define subrepresentations $W_i\subset V$ generated by $u_i$. It is easy to see that $W_i\cong V_n$ and the natural morphism $W_1\oplus...\oplus W_m\to V$ is injective. Hence it is an isomorphism by dimension count, i.e., $V$ is completely reducible. 
\end{proof} 

\begin{corollary} (The Jacobson-Morozov lemma for $GL(V)$) Let $V$ be a finite dimensional complex 
vector space and $N: V\to V$ be a nilpotent operator. Then there is, up to an isomorphism, a unique 
 action of $\mathfrak{sl}_2$ on $V$ for which $e$ acts by $N$. 
\end{corollary} 

\begin{proof} This follows from Theorem \ref{sl2repth} and the Jordan normal form theorem for operators on $V$. 
\end{proof} 

For a representation $V$ define its {\bf character} by 
$$
\chi_V(z)={\rm Tr}_V(z^h)=\sum_m \dim V(m)z^m.
$$
Thus 
$$
\chi_{V_n}(z)=z^n+z^{n-2}+...+z^{-n}=\frac{z^{n+1}-z^{-n-1}}{z-z^{-1}}.
$$
 It is easy to see that 
$$
\chi_{V\oplus W}=\chi_V+\chi_W, \chi_{V\otimes W}=\chi_V\chi_W. 
$$
Since the functions $\chi_{V_n}$ are linearly independent, we see that a finite dimensional representation of $\mathfrak{sl}_2$ is determined by its character.  

\begin{theorem} (The Clebsch-Gordan rule) We have 
$$
V_m\otimes V_n\cong \oplus_{i=0}^{\min(m,n)}V_{|m-n|+2i}.
$$
\end{theorem} 
 
\begin{proof} It suffices to note that we have the corresponding character identity: 
 $$
\chi_{V_m}\chi_{V_n}=\sum_{i=0}^{\min(m,n)}\chi_{V_{|m-n|+2i}}.
$$
\end{proof} 

\begin{exercise} Show that $V_n$ has an invariant nondegenerate inner product (i.e., such that $(av,w)+(v,aw)=0$ for $a\in \mathfrak{sl}_2$, $v,w\in V_n$) which is symmetric for even $n$ and skew-symmetric for odd $n$. In particular, $V_n^*\cong V_n$. 
\end{exercise} 

\begin{exercise}\label{unico} Let $G$ be the universal cover of $SL_2(\Bbb R)$. Show that 
$G$ is not isomorphic to a Lie subgroup of $GL_n(\Bbb R)$ for any $n$ and that moreover, the only quotients of $G$ that are such subgroups are $SL_2(\Bbb R)$ and $PSL_2(\Bbb R)$. 
\end{exercise} 

\section{\bf The universal enveloping algebra of a Lie algebra} 

\subsection{The definition of the universal enveloping algebra} 

Let $V$ be a vector space over a field $\bf k$. 
Recall that the {\bf tensor algebra} of $V$ is the $\Bbb Z$-graded associative algebra $TV:=\oplus_{n\ge 0}V^{\otimes n}$ (with $\deg(V^{\otimes n})=n$),  with multiplication given by 
$a\cdot b=a\otimes b$ for $a\in V^{\otimes m}$ and $b\in V^{\otimes n}$. 
If $\lbrace x_i\rbrace$ is a basis of $V$ then 
$TV$ is just the free algebra with generators 
$x_i$ (i.e., without any relations). Its basis consists 
of various words in the letters $x_i$.  

Let $\g$ be a Lie algebra over ${\bf k}$. 

\begin{definition} The {\bf universal enveloping algebra} of $\g$, denoted $U(\g)$, 
is the quotient of $T\g$ by the ideal $I$ generated by 
the elements $xy-yx-[x,y]$, $x,y\in \g$. 
\end{definition}  

Recall that any associative algebra $A$ is also a Lie algebra with operation 
$[a,b]:=ab-ba$. The following proposition follows immediately from the definition of $U(\g)$. 

\begin{proposition} (i) Let $J\subset T\g$ be an ideal, and $\rho: \g\to T\g/J$ 
the natural linear map. Then $\rho$ is a homomorphism of Lie algebras if and only if 
$J\supset I$, so that $T\g/J$ is a quotient of $T\g/I=U(\g)$. In other words, 
$U(\g)$ is the largest quotient of $T\g$ for which $\rho$ is a homomorphism of Lie algebras. 

(ii) (universal property of $U(\g)$) 
Let $A$ be any associative algebra over ${\bf k}$. Then the map 
$$
\Hom_{\rm associative}(U(\g),A)\to \Hom_{\rm Lie}(\g,A)
$$ 
given by $\phi\mapsto \phi\circ \rho$ is a bijection. 
\end{proposition} 

Part (ii) of this proposition implies that any Lie algebra map \linebreak $\psi: \g\to A$ 
can be uniquely extended to an associative algebra map $\phi: U(\g)\to A$ 
so that $\psi=\phi\circ\rho$. This is the universal property of $U(\g)$ which 
justifies the term ``universal enveloping algebra". 

In particular, it follows that a representation of $\g$ on a vector space $V$ is the same thing as an algebra map $U(\g)\to {\rm End}(V)$ (i.e., a representation of $U(\g)$ on $V$). Thus, to understand the representation theory of $\g$, it is helpful to understand the structure of $U(\g)$; for example, every central element 
$C\in U(\g)$ gives rise to a morphism of representations $V\to V$ (note that this has already come in handy in studying representations of $\mathfrak{sl}_2$).  

In terms of the basis $\lbrace x_i\rbrace$ of $\g$, we can write the bracket 
as 
$$
[x_i,x_j]=\sum_k c_{ij}^kx_k,
$$
where $c_{ij}^k\in {\bf k}$ are the {\bf structure constants}. Then 
the algebra $U(\g)$ can be described as the quotient of 
the free algebra ${\bf k}\langle \lbrace x_i\rbrace\rangle$ by 
the relations 
$$
x_ix_j-x_jx_i=\sum_k c_{ij}^kx_k.
$$

\begin{example} 1. If $\g$ is abelian (i.e., $c_{ij}^k=0$) then 
$U(\g)=S\g={\bf k}[\lbrace x_i\rbrace]$ is the symmetric algebra 
of $\g$, $S\g=\oplus_{n\ge 0}S^n\g$, which in terms of the basis 
is the polynomial algebra in $x_i$. 

2. $U(\mathfrak{sl}_2(\bf k))$ is generated by $e,f,h$ with defining relations
$$
he-eh=2e,\ hf-fh=-2f,\ ef-fe=h.
$$
\end{example} 

Recall that $\g$ acts on $T\g$ by derivations via the adjoint action. 
Moreover, using the Jacobi identity, we have 
$$
{\rm ad}z(xy-yx-[x,y])=[z,x]y+x[z,y]-[z,y]x-y[z,x]-[z,[x,y]]=
$$
$$
([z,x]y-y[z,x]-[[z,x],y])+(x[z,y]-[z,y]x-[x,[z,y]]).
$$
Thus ${\rm ad}z(I)\subset I$, and hence the action of $\g$ on $T\g$ 
descends to its action on $U(\g)$ by derivations (also called the adjoint action). It is easy to see that these derivations are in fact inner: 
$$
{\rm ad}z(a)=za-az
$$
for $a\in U(\g)$ (although this is not so for $T\g$). Indeed, it suffices to note that this holds for $a\in \g$ 
by the definition of $U(\g)$. 

Thus we get 

\begin{proposition} The center $Z(U(\g))$ of $U(\g)$ coincides 
with the subalgebra of invariants $U(\g)^{{\rm ad}\g}$.
\end{proposition} 

\begin{example} The Casimir operator $C=2fe+\frac{h^2}{2}+h$ which we used 
to study representations of $\g=\mathfrak{sl}_2$ is in fact a central element of $U(\g)$. 
\end{example} 

\subsection{Graded and filtered algebras} 

Recall that a $\Bbb Z_{\ge 0}$-{\bf filtered} algebra is an algebra $A$ 
equipped with a filtration 
$$
0=F_{-1}A\subset F_0A\subset F_1A\subset...\subset F_nA\subset...
$$
such that $1\in F_0A$, $\cup_{n\ge 0} F_nA=A$ and 
$F_iA\cdot F_jA\subset F_{i+j}A.$
In particular, if $A$ is generated by $\lbrace{x_\alpha\rbrace}$ 
then a filtration on $A$ can be obtained by declaring $x_\alpha$ to be of degree $1$; i.e.,  $F_nA=(F_1A)^n$ is the span of all words in $x_\alpha$ of degree $\le n$. 

If $A=\oplus_{i\ge 0}A_i$ 
is $\Bbb Z_{\ge 0}$-graded then we can define a filtration on 
$A$ by setting $F_nA:=\oplus_{i=0}^n A_i$; however, not any filtered algebra is obtained in this way, and having a filtration is a weaker condition than having a grading. Still, if $A$ is a filtered algebra, we can define its {\bf associated graded algebra} ${\rm gr}(A):=\oplus_{n\ge 0}{\rm gr}_n(A)$ (also denoted ${\rm gr}A$), where ${\rm gr}_n(A):=F_nA/F_{n-1}A$. The multiplication in ${\rm gr}(A)$ is given by the ``leading terms" of multiplication in $A$: for $a\in {\rm gr}_i(A),\ b\in {\rm gr}_j(A)$, pick their representatives 
$\widetilde a\in F_iA,\ \widetilde b\in F_jA$ and let 
$ab$ be the projection of $\widetilde a\widetilde b$ to ${\rm gr}_{i+j}(A)$. 

\begin{proposition}\label{doma} If ${\rm gr}(A)$ is a domain (has no zero divisors) then 
so is $A$. 
\end{proposition} 

\begin{exercise} Prove Proposition \ref{doma}. 
\end{exercise} 

\begin{example}\label{mapphi}  Let $\g$ be a Lie algebra over $\bold k$. Define a filtration\footnote{The grading on $T\g$ does not descend to $U(\g)$, in general, since the relation $xy-yx=[x,y]$ is not homogeneous: the right hand side has degree 1 while the left hand side has degree 2. So $U(\g)$ is not graded but is only filtered.} on $U(\g)$
by setting $\deg(\g)=1$. Thus $F_nU(\g)$ is the image of $\oplus_{i=0}^n\g^{\otimes i}\subset T\g$. Note that since 
$$
xy-yx=[x,y],\ x,y\in \g,
$$ 
we have $[F_iU(\g),F_jU(\g)]\subset F_{i+j-1}U(\g)$. Thus, ${\rm gr}U(\g)$ is commutative; in other words, we have a surjective algebra morphism 
$$
\phi: S\g\to {\rm gr}U(\g).  
$$
\end{example} 

\subsection{The coproduct of $U(\g)$} \label{copro} 

For a vector space $\g$ define the algebra homomorphism $\Delta: T\g\to T\g\otimes T\g$ given for $x\in \g\subset T\g$ by $\Delta(x)=x\otimes 1+1\otimes x$ (it exists and is unique since $T\g$ is freely generated by $\g$). 

\begin{lemma} If $\g$ is a Lie algebra then 
the kernel $I$ of the map $T\g\to U(\g)$ satisfies the property 
$\Delta(I)\subset I\otimes T\g+T\g\otimes I\subset T\g\otimes T\g$. 
Thus $\Delta$ descends to an algebra homomorphism 
$U(\g)\to U(\g)\otimes U(\g)$. 
\end{lemma} 

\begin{proof} For $x,y\in \g$ and $a=a(x,y):=xy-yx-[x,y]$ we have 
$\Delta(a)=a\otimes 1+1\otimes a$. The lemma follows since the ideal $I$
is generated by elements of the form $a(x,y)$.  
\end{proof} 

The homomorphism $\Delta$ is called the {\bf coproduct} (of $T\g$ or $U(\g)$). 

\begin{example} Let $\g=V$ be abelian (a vector space). Then $U(\g)=SV$, which for $\dim V<\infty$ can be viewed as the algebra of polynomial functions on $V^*$. Similarly, $SV\otimes SV$ is the algebra of polynomial functions on $V^*\times V^*$. In terms of this identification, we have $\Delta(f)(x,y)=f(x+y)$. 
\end{example} 

\subsection{Differential operators on manifolds and Lie groups}\label{diffope}

We have seen in Subsection \ref{vefi} that a vector field on a manifold $X$  
is the same thing as a derivation of the algebra $O(U)$ for every open set $U\subset X$ compatible with restriction maps $O(U)\to O(V)$ for $V\subset U$; in particular, 
on every $U$ we have $[\bold v,m_f]=m_{\bold v(f)}$ where $f\in O(U)$ and 
$m_f: O(U)\to O(U)$ is the operator of multiplication by $f\in O(U)$. Thus if also $g\in O(U)$ then $[[\bold v,m_f],m_g]=0$. Conversely, if $A$ is an endomorphism of the space $O(U)$ for every open $U\subset X$ compatible with restriction maps and $[[A,m_f],m_g]=0$ for any $f,g\in O(U)$ 
then $A=\bold v+m_h$ for a unique vector field $\bold v$ and regular function $h$ on $X$ (check this!). This gives rise to the following generalization of the notion of a vector field. 

\begin{definition} (Grothendieck) A {\bf differential operator} of order $\le N$ on $X$ 
is an endomorphism of the space $O(U)$ for every open set $U\subset X$ compatible with restriction maps $O(U)\to O(V)$ for $V\subset U$ such that for any $f_0,...,f_N\in O(U)$ one has 
$$
[...[[A,m_{f_0}],m_{f_1}],...,m_{f_N}]=0.
$$
\end{definition} 

It is easy to show that the latter condition is equivalent to the classical condition for a differential operator of order $\le N$: in local coordinates $(x_i)$ on a chart $U\subset X$ the operator $A$ looks like 
$$
A=\sum_{k=0}^N \sum_{i_1\le...\le i_k}F_{i_1,...,i_k}\frac{\partial^k}{\partial x_{i_1}...\partial x_{i_k}},
$$
where $F_{i_1,...,i_k}\in O(U)$ (check this!). The space of such operators 
is denoted by $D_N(X)$. Thus we have a nested sequence of spaces 
$$
O(X)=D_0(X)\subset D_1(X)\subset...\subset D_N(X)\subset...
$$
The nested union $\cup_{N\ge 0}D_N(X)$ is a filtered associative algebra called the 
{\bf algebra of differential operators on $X$} and denoted by $D(X)$. 

Now suppose that a Lie group $G$ with Lie algebra $\g$ acts on $X$. Then we have a homomorphism of Lie algebras $\g\to {\rm Vect}(X)$, which can be viewed 
as a Lie algebra homomorphism $\g\to D(X)$. Thus by the universal property 
of the universal enveloping algebra, we obtain an associative algebra 
homomorphism $\xi: U(\g)\to D(X)$. Moreover, this homomorphism 
preserves filtrations. 

For example, if $X=G$ and $G$ acts by right translations, then the corresponding map $\g\to {\rm Vect}(G)$ identifies 
$\g$ with the Lie algebra ${\rm Vect}_L(G)$ of left-invariant vector fields on $G$. Thus 
the map $\xi: U(\g)\to D(G)$ lands in the subalgebra $D_L(G)$ 
of left-invariant differential operators on $G$. 

\begin{exercise}\label{isomoenvel} Show that the map $\xi: U(\g)\to D_L(G)$ is a filtered algebra isomorphism. 
\end{exercise} 

\section{\bf The Poincar\'e-Birkhoff-Witt theorem} 

\subsection{The statement of the Poincar\'e-Birkhoff-Witt theorem} 

Let $\g$ be a Lie algebra over a field $\bold k$. 
Recall from Example \ref{mapphi} that we have a surjective algebra homomorphism 
$$
\phi: S\g\to {\rm gr}U(\g).  
$$

\begin{theorem}\label{PBWth} (Poincar\'e-Birkhoff-Witt theorem) The homomorphism $\phi$ is an isomorphism. 
\end{theorem} 

We will prove Theorem \ref{PBWth} in Subsection \ref{proofPBW}. Now let us discuss its reformulation in terms of a basis and corollaries. 

Given a basis $\lbrace x_i\rbrace$ of $\g$, fix an ordering on this basis and consider 
ordered monomials $\prod_i x_i^{n_i}$, where the product is ordered according to the ordering of the basis. The statement that $\phi$ is surjective is equivalent to saying that ordered monomials span $U(\g)$. This is also easy to see directly: any monomial can be ordered using the commutation relations at the cost of an error of lower degree, so proceeding recursively, we can write any monomial as a linear combination of ordered ones. Thus the PBW theorem can be formulated as follows: 

\begin{theorem}\label{PBWth1} The ordered monomials are linearly independent, hence form a basis of $U(\g)$. 
\end{theorem} 

For instance, if $\bold k=\Bbb R$ or $\Bbb C$ and $\g={\rm Lie}(G)$ where $G$ is a Lie group, this theorem is easy to deduce from Exercise \ref{isomoenvel} (do this!).

\begin{corollary}\label{coroPBW} The map $\rho: \g\to U(\g)$ is injective. Thus $\g\subset U(\g)$. 
\end{corollary} 

\begin{remark} Let $\g$ be a vector space equipped with a bilinear map
$[,]: \g\times \g\to \g$. Then one can define the algebra $U(\g)$ as above. 
However, if the map $\rho: \g\to U(\g)$ is injective then 
we clearly must have $[x,x]=0$ for $x\in \g$ and the Jacobi identity, i.e., $\g$ has to be a Lie algebra. Thus the PBW theorem and even Corollary \ref{coroPBW} fail without the axioms of a Lie algebra. 
\end{remark} 

\begin{corollary}\label{tenpr} Let $\g_i$, $1\le i\le n$, be Lie subalgebras of $\g$ such that 
$\g=\oplus_i \g_i$ as a vector space (but $[\g_i,\g_j]$ need not be zero). Then the multiplication map 
$\otimes_i U(\g_i)\to U(\g)$ in any order is a linear isomorphism. 
\end{corollary} 

\begin{proof} The corollary follows immediately from the PBW theorem by choosing a basis of each $\g_i$. 
\end{proof} 

\begin{remark} 1. Corollary \ref{tenpr} applies to the case of infinitely many $\g_i$ 
if we understand the tensor product accordingly: the span of tensor products 
of elements of $U(\g_i)$ where almost all of these elements are equal to $1$.  

2. Note that if $\dim \g_i=1$, this recovers the PBW theorem itself, so Corollary \ref{tenpr} is in fact a generalization of the PBW theorem. 
\end{remark} 

Let ${\rm char}(\bold k)=0$. Define the {\bf symmetrization map} $\sigma: S\g\to U(\g)$ given by 
$$
\sigma(y_1\otimes...\otimes y_n)=\frac{1}{n!}\sum_{s\in S_n}y_{s(1)}...y_{s(n)}.
$$
It is easy to see that this map commutes with the adjoint action of $\g$. 

\begin{corollary} $\sigma$ is an isomorphism. 
\end{corollary} 

\begin{proof} It is easy to see that ${\rm gr}\sigma$ (the induced map on the associated graded algebra) coincides with $\phi$, so the result follows from the PBW theorem.  
\end{proof} 

Let $Z(U(\g))$ denote the center of $U(\g)$. 

\begin{corollary} The map $\sigma$ defines a filtered vector space isomorphism  $\sigma_0: (S\g)^{{\rm ad}\g}\to Z(U(\g))$ whose associated graded is 
the algebra isomorphism $\phi|_{(S\g)^{{\rm ad}\g}}: (S\g)^{{\rm ad}\g}\to {\rm gr} Z(U(\g))$. 
\end{corollary} 

In the case when $\g={\rm Lie}G$ for a connected Lie group $G$, we thus obtain
a filtered vector space isomorphism of the center of $U(\g)$ with $(S\g)^{{\rm Ad} G}$. 

\begin{remark} The map $\sigma_0$ is not, in general, an algebra homomorphism; however, 
a nontrivial theorem of M. Duflo says that if $\g$ is finite dimensional then there exists
a canonical filtered {\it algebra isomorphism} $\eta: Z(U(\g))\to (S\g)^{{\rm ad}\g}$ 
(a certain twisted version of $\sigma_0$) whose associated graded is $\phi|_{Z(U(\g))}$. A construction of the Duflo isomorphism can be found in \cite{CR}. 
\end{remark} 

\begin{example} Let $\g=\mathfrak{sl}_2=\mathfrak{so}_3$. Then 
$\g$ has a basis $x,y,z$ with $[x,y]=z$, $[y,z]=x$, $[z,x]=y$, and $G=SO(3)$ 
acts on these elements by ordinary rotations of the $3$-dimensional space. 
So the only $G$-invariant polynomials of $x,y,z$ are polynomials 
of $r^2=x^2+y^2+z^2$. Thus we get that $Z(U(\g))=\Bbb C[x^2+y^2+z^2]$. 
In terms of $e,f,h$, we have 
$$
x^2+y^2+z^2=-fe-\frac{h^2+2h}{4}=-\frac{C}{2},
$$
where $C$ is the Casimir element. 
\end{example} 

\subsection{Proof of the PBW theorem} \label{proofPBW}

The proof of Theorem \ref{PBWth} is based on the following key lemma. 

\begin{lemma}\label{pbwlemma} There exists a unique linear map 
$\varphi: T\g\to S\g$ such that 

(i) for an {\bf ordered} monomial $X:=x_{i_1}...x_{i_m}\in \g^{\otimes m}$ one has \linebreak $\varphi(X)=X$; 

(ii) one has $\varphi(I)=0$; in other words, $\varphi$ descends to a linear map 
$\overline\varphi: U(\g)\to S\g$. 
\end{lemma} 

\begin{remark} The map $\varphi$ is not canonical and depends on the choice of the ordered basis $x_i$ of $\g$.  
\end{remark} 

Note that Lemma \ref{pbwlemma} immediately implies the PBW theorem, since by this lemma the images of ordered monomials 
under $\varphi$ are linearly independent in $S\g$, implying that these monomials themselves are linearly independent in $U(\g)$. 

\begin{proof} It is clear that $\varphi$ is unique if exists since ordered monomials span $U(\g)$. We will construct $\varphi$ by defining it inductively on $F_nT\g$ for $n\ge 0$. 

Suppose $\varphi$ is already defined on $F_{n-1}T\g$ and let us extend it to $F_nT\g=F_{n-1}T\g\oplus \g^{\otimes n}$. So we should define $\varphi$ on $\g^{\otimes n}$. Since $\varphi$ is already 
defined on ordered monomials $X$ (by $\varphi(X)=X$), we need to extend this definition 
to all monomials. 

Namely, let $X$ be an ordered monomial of degree $n$, and let us define $\varphi$ 
on monomials of the form $s(X)$ for $s\in S_n$, where 
$$
s(y_1...y_n):=y_{s^{-1}(1)}...y_{s^{-1}(n)}.
$$ 
To this end, fix a decomposition $D$ of $s$ into a product of 
transpositions of neighbors: 
$$
s=s_{j_r}...s_{j_1},
$$
and define $\varphi(s(X))$ by the formula 
$$
\varphi(s(X)):=X-\Phi_D(s,X),
$$
where 
$$
\Phi_D(s,X):=\sum_{m=0}^{r-1} \varphi([,]_{j_{m+1}}(s_{j_m}...s_{j_1}(X))),
$$
and 
$$
[,]_j(y_1...y_jy_{j+1}...y_n):=y_1...[y_j,y_{j+1}]...y_n.
$$

We need to show that $\varphi(s(X))$ is well defined, i.e., $\Phi_D(s,X)$ does not really 
depend on the choice of $D$ and $s$ but only on $s(X)$. We first show that $\Phi_D(s,X)$ is 
independent of  $D$. 

To this end, recall that the symmetric group $S_n$ is generated by $s_j,1\le j\le n-1$
with defining relations 
$$
s_j^2=1;\ s_js_k=s_ks_j, |j-k|\ge 2;\ s_js_{j+1}s_j=s_{j+1}s_js_{j+1}. 
$$
Thus any two decompositions of $s$ into a product of transpositions of neighbors 
can be related by a sequence of applications of 
these relations somewhere inside the decomposition. 

Now, the first relation does not change 
the outcome by the identity $[x,y]=-[y,x]$.

For the second relation, suppose that $j<k$ and 
we have two decompositions $D_1,D_2$ of $s$ given by
$s=ps_js_kq$ and $s=ps_ks_jq$, where $q$ is a product of $m$ transpositions of neighbors. Let $q(X)=YabZcdT$ where $a,b,c,d\in \g$ 
stand in positions $j,j+1,k,k+1$. Let $\Phi_1:=\Phi_{D_1}(s,X)$, $\Phi_2:=\Phi_{D_2}(s,X)$. Then 
the sums defining $\Phi_1$ and $\Phi_2$ differ only in the $m$-th and $m+1$-th term, so we get  
$$
\Phi_1-\Phi_2=
$$
$$
\varphi(YabZ[c,d]T)+\varphi(Y[a,b]ZdcT)-\varphi(Y[a,b]ZcdT)-\varphi(YbaZ[c,d]T),
$$
which equals zero by the induction assumption. 

For the third relation, suppose that we have two decompositions $D_1,D_2$ of $s$ given by
$s=ps_js_{j+1}s_jq$ and $s=ps_{j+1}s_js_{j+1}q$, where $q$ is a product of $k$ transpositions of neighbors. Let $q(X)=YabcZ$ where $a,b,c\in \g$ 
stand in positions $j,j+1,j+2$. Let $\Phi_1:=\Phi_{D_1}(s,X)$, $\Phi_2:=\Phi_{D_2}(s,X)$. Then 
the sums defining $\Phi_1$ and $\Phi_2$ differ only in the $k$-th, $k+1$-th, and $k+2$-th terms, so we get  
$$
\Phi_1-\Phi_2=  
$$
$$
\left(\varphi(Y[a,b]cZ)+\varphi(Yb[a,c]Z)+\varphi(Y[b,c]aZ)\right)-
$$
$$
\left(\varphi(Ya[b,c]Z)+\varphi(Y[a,c]bZ)+\varphi(Yc[a,b]Z)\right).
$$
So the Jacobi identity 
$$
[[b,c],a]+[b,[a,c]]+[[a,b],c]=0
$$
combined with property (ii) in degree $n-1$
implies that $\Phi_1-\Phi_2=0$, i.e., $\Phi_1=\Phi_2$, as claimed. 
Thus we will denote $\Phi_D(s,X)$ just by $\Phi(s,X)$. 

It remains to show that $\Phi(s,X)$ does not depend on the choice of $s$ and only depends on $s(X)$. Let $X=x_{i_1}...x_{i_n}$; then $s(X)=s'(X)$ if and only if $s=s't$, where $t$ is the product of transpositions 
$s_k$ for which $i_k=i_{k+1}$. Thus, it suffices to show that $\Phi(s,X)=\Phi(ss_k,X)$ for such $k$. But this follows from the fact that $[x,x]=0$. 

Now, it follows from the construction of $\varphi$ that for any monomial $X$ of degree $n$ (not necessarily ordered), $\varphi(s_j(X))=\varphi(X)-\varphi([,]_j(X))$. Thus $\varphi$ satisfies property (ii) in degree $n$. This concludes the proof of 
Lemma \ref{pbwlemma} and hence Theorem \ref{PBWth}. 
\end{proof} 

\section{\bf Free Lie algebras, the Baker-Campbell-Hausdorff formula} 

\subsection{Primitive elements}\label{primi} Let $\g$ be a Lie algebra over a field $\bold k$. Let us say that $x\in U(\g)$ is {\bf primitive} if $\Delta(x)=x\otimes 1+1\otimes x$. 
It is clear that if $x\in\g\subset U(\g)$ then $x$ is primitive.  

\begin{lemma}\label{primel} If the ground field $\bold k$ has characteristic zero 
then every primitive element of $U(\g)$ is contained in $\g$. 
\end{lemma} 

\begin{proof} Let $0\ne f\in U(\g)$ be a primitive element. Suppose that the filtration degree 
of $f$ is $n$. Let $f_0\in S^n\g$ be the leading term of $f$ (it is well defined by the PBW Theorem). Then $f_0$ is primitive in $S\g$, 
and in fact in $SV$ for some finite dimensional subspace $V\subset \g$. So 
$f_0(x+y)=f_0(x)+f_0(y)$, $x,y\in V^*$. In particular, $2^nf_0(x)=f_0(2x)=2f_0(x)$, so 
$2^n-2=0$, which implies that $n=1$ as ${\rm char}(\bold k)=0$. Thus $f=c+f_0$ where $f_0\in \g$, $c\in \bold k$ and $c=0$ since $f$ is primitive. 
\end{proof} 

\begin{remark} Note that the assumption of characteristic zero is essential. 
Indeed, if the characterictic of $\bold k$ is $p>0$ and $x\in \g$ then $x^{p^i}\in U(\g)$
is primitive for all $i$.  
\end{remark}

\subsection{Free Lie algebras} 

Let $V$ be a vector space over a field $\bold k$. The {\bf free Lie algebra} $L(V)$ generated by $V$ is the Lie subalgebra of $TV$ generated by $V$. Note that $L(V)$ 
is a $\Bbb Z_{>0}$-graded Lie algebra: $L(V)=\oplus_{m\ge 1}L_m(V)$, 
with grading defined by $\deg V=1$; thus $L_m(V)$ is spanned by commutators of $m$-tuples of elements of $V$ inside $TV$. 

\begin{example} The free Lie algebra $FL_2=L({\bold k}^2)$ in two generators $x,y$ is generated by $x,y$ with $FL_2[1]$ having basis $x,y$, $FL_2[2]$ having basis $[x,y]$, $FL_2[3]$ having basis $[x,[x,y]]$, $[y,[x,y]]$, etc. 
Similarly, $FL_3=L({\bold k}^3)$ is generated by $x,y,z$ with $FL_3[1]$ having basis $x,y,z$, 
$FL_3[2]$ having basis $[x,y],[x,z],[y,z]$, $FL_3[3]$ having basis $[x,[x,y]]$, $[y,[x,y]]$, $[y,[y,z]]$, $[z,[y,z]]$, $[x,[x,z]]$, $[z,[x,z]]$, $[x,[y,z]]$, $[y,[z,x]]$ (note that $[z,[x,y]]$ expresses in terms of the last two using the Jacobi identity). 
\end{example} 

The Lie algebra embedding $L(V)\hookrightarrow TV$ 
gives rise to an associative algebra homomorphism 
$\psi: U(L(V))\to TV$. 

\begin{proposition}\label{freelieprop} (i) $\psi$ is an isomorphism, so $U(L(V))\cong TV$. 

(ii) $\psi$ preserves the coproduct.  

(iii) (The universal property of free Lie algebras) 
If $\g$ is any Lie algebra over $\bold k$ then restriction to $V$ 
defines an isomorphism 
$$
{\bf res}: \Hom_{\rm Lie}(L(V),\g)\cong \Hom_{\bold k}(V,\g).
$$ 
\end{proposition} 

\begin{proof} (i) By definition, $U(L(V))$ is generated by $V$ as an associative algebra, so 
$U(L(V))=TV/J$ for some 2-sided ideal $J$. Moreover, the map 
$\psi: TV/J\to TV$ restricts to the identity on the space $V$ of generators. 
Thus $J=0$ and $\psi={\rm Id}$. 

(ii) is clear since the two coproducts agree on generators.

(iii) Let $a: V\to \g$ be a linear map. Then $a$ can be viewed as a linear map $V\to U(\g)$. 
So it extends to a map of associative algebras $\widetilde a: TV\to U(\g)$
which restricts to a Lie algebra map $\widehat a: L(V)\to U(\g)$. Moreover, since $\widehat a(V)\subset \g\subset U(\g)$ and $L(V)$ is generated by $V$ as a Lie algebra, we obtain that 
$\widehat a: L(V)\to \g$. It is easy to see that the assignment $a\mapsto \widehat a$ 
is inverse to ${\bf res}$, implying that ${\bf res}$ is an isomorphism. 
\end{proof} 

\begin{exercise} Let $\dim V=n$ and $d_m(n)=\dim L_m(V)$. Use the PBW theorem 
to show that $d_m(n)$ are uniquely determined from the identity 
$$
\prod_{m=1}^\infty (1-q^m)^{d_m(n)}=1-nq. 
$$
\end{exercise} 

\subsection{The Baker-Campbell-Hausdorff formula} \label{bchf}

We have defined the commutator $[x,y]$ on $\g={\rm Lie}G$ as the quadratic part of 
$\mu(x,y)=\log(\exp(x)\exp(y))$. So one may wonder if taking higher order terms in the Taylor explansion of $\mu(x,y)$,
\begin{equation}\label{CHser}
\mu(x,y)\sim \sum_{n=1}^\infty\frac{\mu_n(x,y)}{n!}
\end{equation} 
would yield new operations on $\g$. It turns out, however, 
that all these operations express via the commutator. Namely, we have 

\begin{theorem} For each $n\ge 1$, $\mu_n(x,y)$ may be written as a $\Bbb Q$-Lie polynomial of $x,y$ (i.e., a $\Bbb Q$-linear combination of Lie monomials, obtained by taking successive commutators of $x,y$), which is universal (i.e., independent of  $G$). 
\end{theorem} 

\begin{proof}Expansion \eqref{CHser} is equivalent to the equality 
\begin{equation}\label{CHser1}
\exp(tx)\exp(ty)=\exp\left(\sum_{n=1}^\infty\frac{t^n\mu_n(x,y)}{n!}\right)
\end{equation}
inside $U(\g)[[t]]\subset D(G)[[t]]$ for $x,y\in \g$ (see Subsection \ref{diffope}).
Let $T\Bbb C^2=\Bbb C\langle x,y\rangle$ be the free noncommutative algebra in the letters $x,y$. The series $X=\exp(tx):=\sum_{n=0}^\infty \frac{t^nx^n}{n!}$ can be viewed as 
an element of $\Bbb C\langle x,y\rangle[[t]]$, and similarly for $Y:=\exp(ty)$. Thus we may define 
$$
\mu:=\log(XY)\in \Bbb C\langle x,y\rangle[[t]],
$$
where 
$$
\log A:=-\sum_{n=1}^{\infty}\frac{(1-A)^n}{n}.
$$
Then $\mu=\sum_{n=1}^\infty\frac{t^n\mu_n}{n!}$ where $\mu_n\in \Bbb C\langle x,y\rangle$ 
is homogeneous of degree $n$. These $\mu_n$ are the desired universal expressions, and 
it remains to show that they are Lie polynomials, i.e., can be expressed solely in terms of commutators.  

To this end, note that since $\Delta(x)=x\otimes 1+1\otimes x$, the element $X$ is {\bf grouplike}, i.e., $\Delta(X)=X\otimes X$ (where we extend the coproduct to the completion by continuity). The same property is shared by $Y$ and hence by $Z:=XY$, i.e., we have 
$\Delta(Z)=Z\otimes Z$. Thus 
$$
\Delta(\log Z)=\log \Delta(Z)=\log(Z\otimes Z)=\log((Z\otimes 1)(1\otimes Z))
$$
$$
=\log Z\otimes 1+1\otimes \log Z. 
$$
Thus $\mu=\log Z$ is primitive, hence so is $\mu_n$ for each $n$. Thus by 
Lemma \ref{primel}, $\mu_n\in FL_2=L(\Bbb C^2)$, where $FL_2\subset \Bbb C\langle x,y\rangle$ is the free Lie algebra generated by $x,y$. This implies the statement. 
\end{proof} 

\begin{example} 
$$
\mu_3(x,y)=\tfrac{1}{2}([x,[x,y]]+[y,[y,x]]).
$$
Thus 
$$
\mu(x,y)=x+y+\tfrac{1}{2}[x,y]+\tfrac{1}{12}([x,[x,y]]+[y,[y,x]])+...
$$
\end{example} 

\begin{remark} 1. The universal expressions $\mu_n$ are unique, see Example \ref{uniqness} below. 

2. E. Dynkin derived an explicit formula for $\mu(x,y)$ 
making it apparent that it expresses solely in terms of commutators. 
Several proofs of this formula may be found in the expository paper \cite{Mu}.
\end{remark} 

\section{\bf Solvable and nilpotent Lie algebras, theorems of Lie and Engel} 

\subsection{Ideals and commutant} 

Let $\g$ be a Lie algebra. Recall that an ideal in $\g$ is a subspace $\h$ such that $[\g,\h]\subset \h$. 
If $\h\subset \g$ is an ideal then $\g/\h$ has a natural structure of a Lie algebra. 
Moreover, if $\phi: \g_1\to \g_2$ is a homomorphism of Lie algebras then 
${\rm Ker}\phi$ is an ideal in $\g_1$, ${\rm Im}\phi$ is a Lie subalgebra in $\g_2$, and $\phi$ 
induces an isomorphism $\g_1/{\rm Ker}\phi\cong {\rm Im}\phi$ (check it!). 

\begin{lemma}\label{ideals} If $I_1,I_2\subset \g$ are ideals then so are $I_1\cap I_2,I_1+I_2$ and $[I_1,I_2]$ (the set of linear combinations of $[a_1,a_2]$, $a_m\in I_m, m=1,2$). 
\end{lemma}

\begin{exercise} Prove Lemma \ref{ideals}. 
\end{exercise} 

\begin{definition} The commutant of $\g$ is the ideal $[\g,\g]$. 
\end{definition} 

\begin{lemma}\label{ideals1} The quotient $\g/[\g,\g]$ is abelian; moreover, if $I\subset \g$ is an ideal such that $\g/I$ is abelian then $I\supset [\g,\g]$. 
\end{lemma} 

\begin{exercise} Prove Lemma \ref{ideals1}. 
\end{exercise} 

\begin{example} The commutant of $\mathfrak{gl}_n(\bf k)$ is $\mathfrak{sl}_n(\k)$ (check it!). 
\end{example} 

\begin{exercise} (i) Prove that if $G$ is a connected Lie group with Lie algebra $\g$ then the group commutant $[G,G]$ (the subgroup of $G$ generated by elements $ghg^{-1}h^{-1}$, $g,h\in G$) is a Lie subgroup of $G$ with Lie algebra $[\g,\g]$. 

(ii) Let $\widetilde G=\Bbb R\times H$, where $H$ is the {\bf Heisenberg group}
of real matrices of the form 
$$
M(a,b,c):=\begin{pmatrix} 1&a&b\\ 0&1&c\\ 0&0&1\end{pmatrix},\ a,b,c\in \Bbb R.
$$
Let $\Gamma\cong \Bbb Z^2\subset \widetilde G$ be the (closed) central subgroup generated by the pairs $(1,M(0,0,0)={\rm Id})$ and $(\sqrt{2},M(0,1,0))$. 
Let $G=\widetilde G/\Gamma$. Show that $[G,G]$ is not closed in $G$ 
(although by (i) it is a Lie subgroup).  

(iii) Does $[G,G]$ have to be closed in $G$ if $G$ is simply connected? (Consider ${\rm Hom}(G,\Bbb R)$ and apply the second fundamental theorem of Lie theory).  
\end{exercise} 

\subsection{Solvable Lie algebras} 

For a Lie algebra $\g$ define its {\bf derived series} recursively by the formulas 
$D^0(\g)=\g$, $D^{n+1}(\g)=[D^n(\g),D^n(\g)]$. This is a descending sequence of ideals in $\g$. 

\begin{definition} A Lie algebra $\g$ is said to be {\bf solvable} if $D^n(\g)=0$ for some $n$. 
\end{definition} 

\begin{proposition} The following conditions on $\g$ are equivalent: 

(i) $\g$ is solvable; 

(ii) There exists a sequence of ideals 
$\g=\g_0\supset \g_1\supset...\supset \g_m=0$ 
such that $\g_i/\g_{i+1}$ is abelian.  
\end{proposition} 

\begin{proof} It is clear that (i) implies (ii), since we can take $\g_i=D^i\g$. Conversely, 
by induction we see that $D^i\g\subset \g_i$, as desired.  
\end{proof} 

\begin{proposition}\label{ideals2} (i) Any Lie subalgebra or quotient of a solvable Lie algebra is solvable. 

(ii) If $I\subset \g$ is an ideal and $I,\g/I$ are solvable then $\g$ is solvable. 
\end{proposition} 

\begin{exercise} Prove Proposition \ref{ideals2}. 
\end{exercise} 

\subsection{Nilpotent Lie algebras}

For a Lie algebra $\g$ define its {\bf lower central series} recursively by the formulas 
$D_0(\g)=\g$, $D_{n+1}(\g)=[\g,D_n(\g)]$. This is a descending sequence of ideals in $\g$. 

\begin{definition} A Lie algebra $\g$ is said to be {\bf nilpotent} if $D_n(\g)=0$ for some $n$. 
\end{definition} 

\begin{proposition} The following conditions on $\g$ are equivalent: 

(i) $\g$ is nilpotent; 

(ii) There exists a sequence of ideals 
$\g=\g_0\supset \g_1\supset...\supset \g_m=0$ 
such that $[\g,\g_i]\subset \g_{i+1}$.  
\end{proposition} 

\begin{proof} It is clear that (i) implies (ii), since we can take $\g_i=D_i\g$. Conversely, 
by induction we see that $D_i\g\subset \g_i$, as desired.  
\end{proof} 

\begin{remark} Any nilpotent Lie algebra is solvable since $[\g,\g_i]\subset \g_{i+1}$ implies 
$[\g_i,\g_i]\subset \g_{i+1}$, hence $\g_i/\g_{i+1}$ is abelian.  
\end{remark} 

\begin{proposition}\label{ideals3} Any Lie subalgebra or quotient of a nilpotent Lie algebra is nilpotent. 
\end{proposition} 

\begin{exercise} Prove Proposition \ref{ideals3}. 
\end{exercise} 

\begin{example} (i) The Lie algebra of upper triangular matrices of size $n$ 
is solvable, but it is not nilpotent for $n\ge 2$. 

(ii) The Lie algebra of strictly upper triangular matrices is nilpotent.

(iii) The Lie algebra of all matrices of size $n\ge 2$ is not solvable. 
\end{example} 

\subsection{Lie's theorem} 

One of the main technical tools of the structure theory of finite dimensional Lie algebras is {\bf Lie's theorem} for solvable Lie algebras. Before stating and proving this theorem, we will prove the following auxiliary lemma, which will be used several times. 

\begin{lemma}\label{liethlem} 
Let $\g={\bold k}x\oplus \h$ be a Lie algebra 
over a field $\bold k$ in which $\h$ is an ideal (but $[x,\h]$ need not be $0$). 
Let $V$ be a finite dimensional $\g$-module 
and $v\in V$ a common eigenvector of $\h$: 
$$
av=\lambda(a)v,\ a\in \h
$$
where $\lambda: \h\to \bold k$ is a character. Then:

(i) $W:=\bold k[x]v$ is a $\g$-submodule of $V$
on which $a-\lambda(a)$ is nilpotent for all $a\in \h$.

(ii) If in addition $\lambda$ vanishes on $[\g,\h]$ (i.e., $\lambda([a,x])=0$ for all $a\in \h$) 
then every $a\in \h$ acts on $W$ by the scalar $\lambda(a)$. Thus the common eigenspace $V_\lambda\subset V$ of $\h$ is a $\g$-submodule. 

(iii) The assumption (hence the conclusion) of (ii) always holds if 
${\rm char}({\bold k})=0$. 
\end{lemma} 

\begin{proof} (i) For $a\in \h$ we have 
\begin{equation}\label{recur}
ax^iv=xax^{i-1}v+[a,x]x^{i-1}v.
\end{equation}
Therefore, it follows by induction in $i$ that $ax^iv$ is a linear combination of $v,xv,...,x^iv$, hence $W\subset V$ is a submodule.   

Let $n$ be the smallest integer such that $x^nv$ is a linear combination of $x^iv$ with $i<n$. 
Then $v_i:=x^{i-1}v$ for $i=1,...,n$ is a basis of $W$ and $\dim W=n$. It follows from \eqref{recur} that the element $a$ acts in this basis by an upper triangular matrix with all diagonal entries equal $\lambda(a)$, as claimed. 

(ii) It follows from \eqref{recur} by induction in $i$  that for every 
$a\in \h$, $ax^iv=\lambda(a)x^iv$, as desired. 

(iii) By (i), ${\rm Tr}(a|_W)=n\lambda(a)$ for all $a\in \h$. 
On the other hand, if $a\in [\g,\g]$ then ${\rm Tr}(a|_W)=0$, thus $n\lambda(a)=0$ in ${\bf k}$. 
Since ${\rm char}({\bf k})=0$, this implies that $\lambda(a)=0$.
\end{proof} 

\begin{theorem}\label{liethm} (Lie's theorem) Let ${\bf k}$ be an algebraically closed field of characteristic zero, and $\g$ a finite dimensional solvable Lie algebra over ${\bf k}$. Then any irreducible 
finite dimensional representation of $\g$ is 1-dimensional. 
\end{theorem} 

\begin{proof} Let $V$ be a finite dimensional representation of $\g$. It suffices to show that 
$V$ contains a common eigenvector of $\g$. The proof is by induction in $\dim\g$. 
The base is trivial so let us justify the induction step. Since $\g$ is solvable, $\g\ne [\g,\g]$, so fix a subspace $\h\subset \g$ of codimension $1$ containing $[\g,\g]$. Since $\g/[\g,\g]$ is abelian, $\h$ is an ideal in $\g$, hence solvable. Thus by the induction assumption, 
there is a nonzero common eigenvector $v\in V$ for $\h$, i.e., 
there is a linear functional $\lambda: \h\to {\bf k}$ such that 
$av=\lambda(a)v$ for all $a\in \h$. 

Let $x\in \g$ be an element not belonging to $\h$ 
and $W$ be the subspace of $V$ spanned by $v,xv,x^2v,...$. 
By Lemma \ref{liethlem}(i), $W$ is a $\g$-submodule of $V$ and $a-\lambda(a)$ 
is nilpotent on $W$. 
Thus by Lemma \ref{liethlem}(ii),(iii) every $a\in \h$ acts 
on $W$ by $\lambda(a)$, in particular $[\g,\g]$ acts by zero. Hence 
$W$ is a representation of the abelian Lie algebra 
$\g/[\g,\g]$. Now the statement follows since every finite dimensional representation of an abelian Lie algebra has a common eigenvector. 
\end{proof} 

\begin{remark}\label{remm} Lemma \ref{liethlem}(iii) and Lie's theorem do not hold in characteristic $p>0$. Indeed, let $\g$ be the Lie algebra 
with basis $x,y$ and $[x,y]=y$, and let $V$ be the space with basis $v_0,...,v_{p-1}$ 
and action of $\g$ given by 
$$
xv_i=iv_i,\ yv_i=v_{i+1},
$$
where $i+1$ is taken modulo $p$. It is easy to see that $V$ is irreducible. 
\end{remark} 

Here is another formulation of Lie's theorem: 

\begin{corollary}\label{Liethm2} Every finite dimensional representation $V$ of a finite dimensional solvable Lie algebra 
$\g$ over an algebraically closed field ${\bf k}$ of characteristic zero has a basis 
in which all elements of $\g$ act by upper triangular matrices. 
In other words, there is a sequence of subrepresentations 
$0=V_0\subset V_1\subset...\subset V_n=V$ such that $\dim(V_{k+1}/V_k)=1$. 
\end{corollary} 

In the case $\dim\g=1$, this recovers the well known theorem in linear algebra that 
any linear operator on a finite dimensional ${\bf k}$-vector space is upper triangular in some basis (which is actually true in any characteristic). 

\begin{proof} The proof is by induction in $\dim V$ (where the base is obvious). 
By Lie's theorem, there is a common eigenvector $v_0\in V$ for $\g$. 
Let $V':=V/{\bf k}v_0$. Then by the induction assumption $V'$ has a basis 
$v_1',...,v_n'$ in which $\g$ acts by upper triangular matrices. 
Let $v_1,...,v_n$ be any lifts of $v_1',...,v_n'$ to $V$. Then $v_0,v_1,...,v_n$ is a basis of $V$ 
in which $\g$ acts by upper triangular matrices.     
\end{proof} 

\begin{corollary}\label{coroo} Over an algebraically closed field of characteristic zero, the following hold. 

(i) A solvable finite dimensional Lie algebra $\g$ admits a sequence of ideals 
$0=I_0\subset I_1\subset...\subset I_n=\g$ such that $\dim(I_{k+1}/I_k)=1$. 

(ii) A finite dimensional Lie algebra $\g$ is solvable if and only if $[\g,\g]$ is nilpotent. 
\end{corollary} 

\begin{proof} (i) Apply Corollary \ref{Liethm2} to the adjoint representation of $\g$. 

(ii) If $[\g,\g]$ is nilpotent then it is solvable and $\g/[\g,\g]$ is abelian, so $\g$ is solvable. 
Conversely, if $\g$ is solvable then by Corollary \ref{Liethm2} elements of $[\g,\g]$ act on $\g$, hence on $[\g,\g]$ by strictly upper triangular matrices, which implies the statement. 
\end{proof} 

\begin{example}\label{semdi} Let $\g,V$ be as in Remark \ref{remm} and $\h=\g\ltimes V$ be the semidirect product, 
i.e. $\h=\g\oplus V$ as a space with 
$$
[(g_1,v_1),(g_2,v_2)]=([g_1,g_2],g_1v_2-g_2v_1). 
$$
Then $\h$ is a counterexample to Corollary \ref{coroo} both (i) and (ii) in characteristic $p>0$. 
\end{example} 

\subsection{Engel's theorem} 

Another key tool of the structure theory of finite dimensional Lie algebras is {\bf Engel's theorem}. Before stating and proving this theorem, we prove an auxiliary result.  

\begin{theorem}\label{auxth} Let $V\ne 0$ be a finite dimensional vector space 
over any field ${\bf k}$, and $\g\subset {\mathfrak {gl}}(V)$ be a Lie algebra consisting of nilpotent operators. Then there exists a nonzero vector $v\in V$ such that $\g v=0$. 
\end{theorem} 

\begin{proof} The proof is by induction on the dimension of $\g$. The base case $\g=0$ is trivial and we assume the dimension of $\g$ is positive.

First we find an ideal $\h$ of codimension one in $\g$. Let $\h$ be a maximal (proper) subalgebra of $\g$, which exists by finite-dimensionality of $\g$. We claim that $\h\subset \g$ is an ideal and has codimension one. 

Indeed,  for each $a\in \h$, the operator ${\rm ad}a$ induces a linear operator $\g / \h \to \g / \h$, and this operator is nilpotent (since $a$ acts nilpotently on $V$, 
it also acts nilpotently on $\mathfrak{gl}(V)=V\otimes V^*$, hence the operator 
${\rm ad}a: \g\to \g$ is nilpotent). Thus, by the inductive hypothesis, there exists a nonzero element $\overline x$ in $\g / \h$ such that ${\rm ad}a\cdot \overline x  = 0$  for each $a \in \h$. Let $x$ be a lift of $\overline x$ to $\g$. Then $[a,x]\in \h$ for all $a\in \h$. Let $\h'$ be the span of $\h$ and $x$. Then $\h'\subset \g$ is a Lie subalgebra in which $\h$ is an ideal. Hence, by maximality, $\h' = \g$. This proves the claim.

Now let $W=V^\h\subset V$. By the inductive hypothesis, $W\ne 0$. Also 
by Lemma \ref{liethlem}(ii) (with $\lambda=0$), $W$ is a $\g$-subrepresentation of $V$. 

Now take $w\ne 0$ in $W$. Let $k$ be the smallest positive integer such that $x^kw=0$; it exists since $x$ acts nilpotently on $V$. Let $v=x^{k-1}w\in W$. Then $v\ne 0$ but $\h v=0,xv=0$, so $\g v=0$, as desired. 
\end{proof} 

\begin{definition} An element $x\in \g$ is said to be {\bf nilpotent} if the operator ${\rm ad}x: \g\to \g$ is nilpotent.   
\end{definition} 

\begin{corollary} (Engel's theorem)  A finite dimensional Lie algebra $\g$ is nilpotent if and only if every 
element $x\in \g$ is nilpotent. 
\end{corollary} 

\begin{proof} The ``only if" direction is easy. To prove the ``if" direction, note that by Theorem \ref{auxth}, 
in some basis $v_i$ of $\g$ all elements ${\rm ad}x$ act by strictly upper triangular matrices. Let $I_m$ be the subspace of $\g$ spanned by the vectors $v_1,...,v_m$. Then $I_m\subset I_{m+1}$ and $[\g,I_{m+1}]\subset I_m$, hence $\g$ is nilpotent. 
\end{proof} 

\section{\bf Semisimple and reductive Lie algebras, the Cartan criteria}

\subsection{Semisimple and reductive Lie algebras, the radical} 

Let $\g$ be a finite dimensional Lie algebra over a field $\k$. 

\begin{proposition} The sum of all solvable ideals of $\g$ is a solvable ideal. 
\end{proposition} 

\begin{definition} This ideal is 
called {\bf the radical} of $\g$ and denoted ${\rm rad}(\g)$.  
\end{definition} 

\begin{proof} Let $I,J$ be solvable ideals of $\g$. Then $I+J\subset \g$ 
is an ideal, and $(I+J)/I=J/(I\cap J)$ is solvable, so $I+J$ is solvable. 
Thus the sum of finitely many solvable ideals is solvable. Hence 
the sum of all solvable ideals in $\g$ is a solvable ideal, as desired. 
\end{proof} 

\begin{definition} (i) $\g$ is called {\bf semisimple} if ${\rm rad}(\g)=0$, i.e., $\g$ does not contain 
nonzero solvable ideals.

(ii) A non-abelian $\g$ is called {\bf simple} if it contains no ideals other than $0,\g$. In other words, a non-abelian $\g$ is simple if its adjoint representation is irreducible (=simple). 
\end{definition} 

Thus if $\g$ is both solvable and semisimple then $\g=0$. 

\begin{proposition} (i) We have ${\rm rad}(\g\oplus \h)={\rm rad}(\g)\oplus {\rm rad}(\h)$. In particular, 
the direct sum of semisimple Lie algebras is semisimple. 

(ii) A simple Lie algebra is semisimple. Thus a direct sum of simple Lie algebras is semisimple. 
\end{proposition} 

\begin{proof} (i) The images of ${\rm rad}(\g\oplus \h)$ in $\g$ and in $\h$ are solvable, hence contained in ${\rm rad}(\g)$, respectively ${\rm rad}(\h)$. Thus 
$$
{\rm rad}(\g\oplus \h)\subset {\rm rad}(\g)\oplus {\rm rad}(\h).
$$
But ${\rm rad}(\g)\oplus {\rm rad}(\h)$ 
is a solvable ideal  in $\g\oplus \h$, so 
$$
{\rm rad}(\g\oplus \h)={\rm rad}(\g)\oplus {\rm rad}(\h).
$$

(ii) The only nonzero ideal in $\g$ is $\g$, and $[\g,\g]=\g$ since $\g$ is not abelian. Hence $\g$ is not solvable. Thus $\g$ is semisimple. 
\end{proof} 

\begin{example} The Lie algebra $\mathfrak{sl}_2(\k)$ is simple if ${\rm char}(\k)\ne 2$. 
Likewise, $\mathfrak{so}_3(\k)$ is simple. 
\end{example} 

\begin{theorem}\label{wld} (weak Levi decomposition) The Lie algebra $\g_{\rm ss}=\g/{\rm rad}(\g)$ is semisimple. Thus any $\g$ can be included 
in an exact sequence 
$$
0\to {\rm rad}(\g)\to \g\to \g_{\rm ss}\to 0,
$$
where ${\rm rad}(\g)$ is a solvable ideal and $\g_{\rm ss}$ is semisimple. Moreover, if $\h\subset \g$ 
is a solvable ideal such that $\g/\h$ is semisimple then $\h={\rm rad}(\g)$. 
\end{theorem} 

\begin{proof} Let $I\subset \g_{\rm ss}$ be a solvable ideal, and let $\widetilde I$ be its preimage in 
$\g$. Then $\widetilde I$ is a solvable ideal in $\g$. Thus $\widetilde I={\rm rad}(\g)$ and $I=0$.  
\end{proof} 

In fact, in characteristic zero there is a stronger statement, which says that the extension in Theorem \ref{wld} splits.
Namely, given a Lie algebra $\h$ and another Lie algebra $\a$ acting on $\h$ by derivations, we may form the {\bf semidirect product} Lie algebra
$\a\ltimes \h$ which is $\a\oplus \h$ as a vector space with 
commutator defined by 
$$
[(a_1,h_1),(a_2,h_2)]=([a_1,a_2],a_1\circ h_2-a_2\circ h_1+[h_1,h_2]).
$$
Note that a special case of this construction has already appeared in Example \ref{semdi}. 

\begin{theorem}\label{levi} (Levi decomposition) If ${\rm char}({\bf k})=0$ then 
we have $\g\cong {\rm rad}(\g)\oplus \g_{\rm ss}$ as vector spaces, 
where $\g_{\rm ss}\subset \g$ is a semisimple subalgebra (but not necessarily an ideal); i.e., 
$\g$ is isomorphic to the semidirect product $\g_{\rm ss}\ltimes {\rm rad}(\g)$. 
In other words, the projection $p: \g\to \g_{\rm ss}$ admits an (in general, 
non-unique) splitting $q: \g_{\rm ss}\to \g$, i.e., a Lie algebra map such that $p\circ q={\rm Id}$. 
\end{theorem} 

Theorem \ref{levi} will be proved in Subsection \ref{levide}. 

\begin{example} Let $G$ be the group of motions of the Euclidean space $\Bbb R^3$ (generated by rotations and translations). Then $G=SO_3(\Bbb R)\ltimes \Bbb R^3$, so 
$\g={\rm Lie}G={\mathfrak{so}}_3(\Bbb R)\ltimes \Bbb R^3$, hence ${\rm rad}(\g)=\Bbb R^3$ (abelian Lie algebra) and $\g_{\rm ss}={\mathfrak{so}}_3(\Bbb R)$.  
\end{example} 

\begin{proposition} Let ${\rm char}(\k)=0$, $\k$ algebraically closed, and $V$ be an irreducible representation of $\g$. Then ${\rm rad}(\g)$ acts on $V$ by scalars, and 
$[\g,{\rm rad}(\g)]$ by zero. 
\end{proposition} 

\begin{proof} By Lie's theorem, there is a nonzero $v\in V$ and $\lambda\in {\rm rad}(\g)^*$ such that $av=\lambda(a)v$ for $a\in {\rm rad}(\g)$. Let $x\in \g$ and $\g_x\subset \g$ be the Lie subalgebra spanned by ${\rm rad}(\g)$ and $x$. Let $W$ be the span of $x^nv$ for $n\ge 0$. By Lemma \ref{liethlem}(i), $W$ is a $\g_x$-subrepresentation of $V$ on which $a\in {\rm rad}(\g)$ has the only eigenvalue $\lambda(a)$. 
 Thus by Lemma \ref{liethlem}(iii), for $a\in {\rm rad}(\g)$ we have $\lambda([x,a])=0$, so the $\lambda$-eigenspace $V_\lambda$ of ${\rm rad}(\g)$ in $V$ is a $\g$-subrepresentation of $V$, 
which implies that $V_\lambda=V$ since $V$ is irreducible. 
\end{proof} 

\begin{definition} $\g$ is called {\bf reductive} if ${\rm rad}(\g)$ coincides with the center $\mathfrak{z}(\g)$ of $\g$. 
\end{definition} 

In other words, $\g$ is reductive if $[\g,{\rm rad}(\g)]=0$. 

The Levi decomposition theorem implies that a reductive Lie algebra in characteristic zero is a direct sum of a semisimple Lie algebra and an abelian Lie algebra (its center). We will also prove this in Corollary \ref{redualg}. 

\subsection{Invariant inner products} 

Let $B$ be a bilinear form on a Lie algebra $\g$. Recall that $B$ is invariant if 
$B([x,y],z)=B(x,[y,z])$ for any $x,y,z\in \g$. 

\begin{example} If $\rho: \g\to {\mathfrak{gl}}(V)$ is a finite dimensional representation of $\g$ 
then the form
$$
B_V(x,y):={\rm Tr}(\rho(x)\rho(y))
$$
is an invariant symmetric bilinear form on $\g$. Indeed, the symmetry is obvious and 
$$
B_V([x,y],z)=B_V(x,[y,z])={\rm Tr}|_V(\rho(x)\rho(y)\rho(z)-\rho(x)\rho(z)\rho(y)). 
$$
\end{example} 

\begin{proposition}\label{bilform} If $B$ is a symmetric invariant bilinear form on $\g$ and $I\subset \g$ is an ideal then 
the orthogonal complement $I^\perp\subset \g$ is also an ideal. In particular, $\g^\perp={\rm Ker}(B)$ is an ideal in $\g$. 
\end{proposition} 

\begin{exercise} Prove Proposition \ref{bilform}. 
\end{exercise} 

\begin{proposition}\label{reducri} If $B_V$ is nondegenerate for some $V$ then $\g$ is reductive. 
\end{proposition} 

\begin{proof} Let $V_1,...,V_n$ be the simple composition factors of $V$; i.e., $V$ has a filtration 
by subrepresentations such that $F_{i}V/F_{i-1}V=V_i$, $F_0V=0$ and $F_nV=V$. Then 
$B_V(x,y)=\sum_i B_{V_i}(x,y)$. Now, if $x\in [\g,{\rm rad}(\g)]$ then 
$x|_{V_i}=0$, so $B_{V_i}(x,y)=0$ for all $y\in \g$, hence $B_V(x,y)=0$. 
\end{proof} 

\begin{example} It is clear that if $\g={\mathfrak{gl}}_n(\bf k)$ and $V={\bf k}^n$ 
then the form $B_V$ is nondegenerate, as $B_V(E_{ij},E_{kl})=\delta_{il}\delta_{jk}$. 
Thus $\g$ is reductive. Also if $n$ is not divisible by the characteristic of ${\bf k}$ 
then ${\mathfrak{sl}}_n(\bf k)$ is semisimple, since it is orthogonal to scalars under $B_V$ (hence reductive), and has trivial center. In fact, it is easy to show that in this case ${\mathfrak{sl}}_n(\bf k)$ is a simple Lie algebra (another way to see that it is semisimple). 
\end{example}  

In fact, we have the following proposition. 

\begin{proposition} All classical Lie algebras over $\Bbb K=\Bbb R$ and $\Bbb C$ are reductive. 
\end{proposition} 

\begin{proof} Let $\g$ be a classical Lie algebra and $V$ its standard matrix representation. 
It is easy to check that the form $B_V$ on $\g$ is nondegenerate, which implies that 
$\g$ is reductive. 
\end{proof} 

For example, the Lie algebras $\mathfrak{so}_n(\Bbb K)$, $\mathfrak{sp}_{2n}(\Bbb K)$, ${\mathfrak{su}}(p,q)$  have trivial center and therefore are semisimple. 

\subsection{The Killing form and the Cartan criteria} 

\begin{definition} The {\bf Killing form} of a Lie algebra $\g$ is the form 
$B_\g(x,y)={\rm Tr}({\rm ad} x \cdot {\rm ad} y)$.  
\end{definition} 

The Killing form is denoted by $K_\g(x,y)$ or shortly by $K(x,y)$. 

\begin{theorem}\label{cartsolv} (Cartan criterion of solvability) A Lie algebra 
$\g$ over a field $\bf k$ of characteristic zero 
is solvable if and only if $[\g,\g]\subset {\rm Ker}(K)$. 
\end{theorem} 

\begin{theorem}\label{cartsem} (Cartan criterion of semisimplicity) A Lie algebra 
$\g$ over a field $\bf k$ of characteristic zero is semisimple if and only if 
its Killing form is nondegenerate. 
\end{theorem} 

Theorems \ref{cartsolv} and \ref{cartsem} will be proved in the next section. 

\begin{corollary}\label{uniscal} On a complex simple Lie algebra, the Killing form is the unique  invariant bilinear form up to scaling.  
\end{corollary} 

\begin{proof} Let $\g$ be a simple Lie algebra. 
Then the Killing form is a nonzero (in fact, nondegenerate) invariant bilinear form on $\g$. Also any invariant bilinear form $B$ on $\g$ can be viewed as a homomorphism of representations $B: \g\to \g^*$. 
Thus by Schur's lemma it is unique up to scaling. 
\end{proof} 

\subsection{Jordan decomposition} To prove the Cartan criteria, we will use the Jordan decomposition of a square matrix. Let us recall it. 

\begin{proposition}\label{jdec} A square matrix $A\in {\mathfrak{gl}}_N({\bf k})$ over a field ${\bf k}$ of characteristic zero can be uniquely written as $A_s+A_n$, where $A_s\in {\mathfrak{gl}}_N({\bf k})$ is semisimple (i.e. diagonalizes over the algebraic closure of ${\bf k}$) and $A_n\in {\mathfrak{gl}}_N({\bf k})$ is nilpotent in such a way that $A_sA_n=A_nA_s$. 
Moreover, $A_s=P(A)$ for some $P\in {\bf k}[x]$. 
\end{proposition} 

\begin{proof} 
By the Chinese remainder theorem, 
there exists a polynomial $P\in \overline {\bf k}[x]$ 
such that for every eigenvalue $\lambda$ of $A$
we have $P(x)=\lambda$ modulo $(x-\lambda)^N$, 
i.e., 
$$
P(x)-\lambda=(x-\lambda)^N Q_\lambda(x)
$$ 
for some polynomial $Q_\lambda$. 
Then on the generalized eigenspace $V(\lambda)$ for $A$, we have   
$$
P(A)-\lambda=(A-\lambda)^NQ_\lambda(A)=0,
$$ 
so 
$A_s:=P(A)$ is semisimple and $A_n=A-P(A)$ is nilpotent, with $A_nA_s=A_sA_n$. 
If $A=A_s'+A_n'$ is another such decomposition then $A_s',A_n'$ commute with $A$, hence with $A_s$ and $A_n$. Also we have 
$$
A_s-A_s'=A_n'-A_n.
$$
Thus this matrix is both semisimple and nilpotent, so it is zero. Finally, since $A_s,A_n$ are unique, they are invariant under the Galois group of $\overline{\bf k}$ over ${\bf k}$ 
and therefore have entries in ${\bf k}$. 
\end{proof}    

\begin{remark} 1. If ${\bf k}$ is algebraically closed, then $A$ admits a basis in which it is upper triangular, and $A_s$ is the diagonal part while $A_n$ is the off-diagonal part of $A$. 

2. Proposition \ref{jdec} holds with the same proof in characteristic $p$ if 
the field ${\bf k}$ is perfect, i.e., the Frobenius map $x\to x^p$ is surjective on ${\bf k}$. 
However, if ${\bf k}$ is not perfect, the proof fails: the fact that $A_s$, $A_n$ are Galois invariant does not imply that their entries are in ${\bf k}$. Also the statement fails: 
if ${\bf k}=\Bbb F_p(t)$ and $Ae_i=e_{i+1}$ for $i=1,..,p-1$ 
while $Ae_p=te_1$ then $A$ has only one eigenvalue $t^{1/p}$, so $A_s=t^{1/p}\cdot {\rm Id}$, i.e., does not have entries in ${\bf k}$. 
\end{remark} 

\section{\bf Proofs of the Cartan criteria, properties of semisimple Lie algebras} 

\subsection{Proof of the Cartan solvability criterion} 

It is clear that $\g$ is solvable if and only if so is $\g\otimes_{\bf k}\overline{\bf k}$, 
so we may assume that ${\bf k}$ is algebraically closed. 

For the ``only if" part, note that by Lie's theorem, $\g$ has a basis in which the operators ${\rm ad}x$, $x\in \g$, are upper triangular. Then $[\g,\g]$ acts in this basis by strictly upper triangular matrices, so $K(x,y)=0$ for $x\in [\g,\g]$ and $y\in \g$.

To prove the ``if" part, let us prove the following lemma. 

\begin{lemma}\label{lee} Let $\g\subset {\mathfrak{gl}}(V)$ be a Lie subalgebra such that 
for any $x\in [\g,\g]$ and $y\in \g$ we have ${\rm Tr}(xy)=0$. Then $\g$ is solvable. 
\end{lemma} 

\begin{proof} Let $x\in [\g,\g]$. Let $\lambda_i, i=1,...,m$, be the distinct eigenvalues 
of $x$. Let $E\subset \bf k$ be a $\Bbb Q$-span of $\lambda_i$. Let $b: E\to \Bbb Q$ be a linear functional. There exists an interpolation polynomial $Q\in {\bf k}[t]$ such that 
$Q(\lambda_i-\lambda_j)=b(\lambda_i-\lambda_j)=b(\lambda_i)-b(\lambda_j)$ for all $i,j$. 

By Proposition \ref{jdec}, we can write $x$ as $x=x_s+x_n$. Then the operator ${\rm ad} x_s$ is diagonalizable with eigenvalues $\lambda_i-\lambda_j$. So 
$$
Q({\rm ad}x_s)={\rm ad}b,
$$ 
where $b: V\to V$ is the operator acting by $b(\lambda_j)$ on the generalized 
$\lambda_j$-eigenspace of $x$. 

Also we have 
$$
{\rm ad}x={\rm ad}x_s+{\rm ad}x_n
$$
a sum of commuting semisimple and nilpotent operators. Thus 
$$
{\rm ad}x_s=({\rm ad}x)_s=P({\rm ad}x),
$$
and $P(0)=0$ since $0$ is an eigenvalue of ${\rm ad}x$. 
Thus 
$$
{\rm ad}b=R({\rm ad}x),
$$
where $R(t)=Q(P(t))$ and $R(0)=0$. 

Let $x=\sum_j [y_j,z_j]$, $y_j,z_j\in \g$, and $d_j$ be the dimension of the generalized $\lambda_j$-eigenspace of $x$. Then 
$$
\sum_j d_jb(\lambda_j)\lambda_j={\rm Tr}(b x)=
$$
$$
{\rm Tr}(\sum_j b[y_j,z_j])={\rm Tr}(\sum_j [b,y_j]z_j)={\rm Tr}(\sum_j R({\rm ad}x)(y_j)z_j). 
$$
Since $R(0)=0$, we have $R({\rm ad}x)(y_j)\in [\g,\g]$, so by assumption we get 
$$
\sum_j d_jb(\lambda_j)\lambda_j=0.
$$
Applying $b$, we get 
$\sum_j d_jb(\lambda_j)^2=0$. Thus $b(\lambda_j)=0$ 
for all $j$. Hence $b=0$, so $E=0$. 

Thus, the only eigenvalue of $x$ is $0$, i.e., $x$ is nilpotent. But then by Engel's theorem, $[\g,\g]$ is nilpotent. Thus $\g$ is solvable. Thus proves the lemma. 
\end{proof} 
 
Now the ``if" part of the Cartan solvability criterion follows easily by applying Lemma \ref{lee} to the faithful representation $V=\g$ of the Lie algebra $\g/\mathfrak{z}(\g)$. 

\subsection{Proof of the Cartan semisimplicity criterion} 

Assume that $\g$ is semisimple, and let $I={\rm Ker}(K_\g)$, an ideal in $\g$. Then $K_I=(K_\g)|_I=0$. Thus by Cartan's solvability criterion $I$ is solvable. Hence $I=0$.  

Conversely, suppose $K_\g$ is nondegenerate. Then $\g$ is reductive. 
Moreover, the center of $\g$ is contained in the kernel of $K_\g$, so it must be trivial.  
Thus $\g$ is semisimple. 

\subsection{Properties of semisimple Lie algebras} 

\begin{proposition} Let ${\rm char}({\bf k})=0$ and $\g$ be a finite dimensional Lie algebra over $\k$. Then $\g$ is semisimple iff $\g\otimes_{\bf \k} \overline{\bf k}$ is semisimple. 
\end{proposition} 

\begin{proof} 
Immediately follows from Cartan's criterion of semisimplicity. Here is another proof (of the nontrivial direction): if $\g$ is semisimple and $I$ is a nonzero solvable ideal in $\g\otimes_{\bf \k} \overline{\bf k}$ then it has a finite Galois orbit $I_1,...,I_n$ and $I_1+...+I_n$ is a 
Galois invariant solvable ideal, so it comes from a solvable ideal in $\g$. 
\end{proof} 

\begin{remark} This theorem fails if we replace the word ``semisimple'' by ``simple":
e.g., if $\g$ is a simple complex Lie algebra regarded as a real Lie algebra then 
$\g_{\Bbb C}\cong \g\oplus \g$ is semisimple but not simple. 
\end{remark} 

\begin {theorem}\label{dirsu} Let $\g$ be a semisimple Lie algebra and $I \subset \g$ an ideal. Then
there is an ideal $J\subset \g$ such that $\g = I \oplus J$.
\end{theorem} 

\begin{proof} Let $I^\perp$ be the orthogonal complement of $I$ with respect to the Killing form, an ideal in $\g$. Consider the intersection $I \cap I^\perp$. It is an ideal
in $\g$ with the zero Killing form (as the Killing form of an ideal in $\g$ is the restriction of the Killing of $\g$). Thus, by the Cartan solvability criterion, it is
solvable. By definition of a semisimple Lie algebra, this means that $I \cap I^\perp$ = 0,
so we may take $J= I^\perp$. 
\end{proof} 

We will see below (in Proposition \ref{dirsu2}) 
that $J$ is in fact unique and must equal $I^\perp$. 

\begin{corollary}\label{dirsu1} A Lie algebra $\g$ is semisimple iff it is a direct sum of simple Lie
algebras.
\end{corollary} 

\begin{proof} We have already shown that a direct sum of simple Lie algebras is semisimple. 
The opposite direction easily follows by induction from Theorem \ref{dirsu}.
\end{proof} 

\begin{corollary}\label{gperf} If $\g$ is a semisimple Lie algebra, then $[\g, \g] = \g$.
\end{corollary} 

\begin{proof} For a simple Lie algebra it is clear because $[\g, \g]$ is an ideal in $\g$ which
cannot be zero (otherwise, $\g$ would be abelian). So the result follows from Corollary \ref{dirsu1}.\end{proof} 

\begin{proposition}\label{dirsu2} Let $\g = \g_1\oplus...\oplus \g_k$ be a semisimple Lie algebra, with $\g_i$
being simple. Then any ideal $I$ in $\g$ is of the form $I = \oplus_{i\in S} \g_i$ 
for some subset $S \subset \lbrace 1,..., k\rbrace$.
\end{proposition} 

\begin{proof} The proof goes by induction in $k$. Let $p_k : \g\to \g_k$ be the projection.
Consider $p_k(I) \subset \g_k$. Since $\g_k$ is simple, either $p_k(I)=0$, in which case
$I\subset \g_1\oplus...\oplus \g_{k-1}$  and we can use the induction assumption, 
or $p_k (I ) = \g_k$.
Then $[\g_k , I] = [\g_k , p_k (I )] = \g_k$. Since $I$ is an ideal, $I \supset \g_k$, so 
$I =I'\oplus \g_k$ for some subspace $I'\subset \g_1 \oplus· · ·\oplus \g_{k-1}$. 
It is immediate that then $I'$ is an ideal in
$\g_1\oplus· · ·\oplus \g_{k-1}$ and the result again follows from the induction assumption.\end{proof} 

\begin{corollary} Any ideal in a semisimple Lie algebra is semisimple. Also, any
quotient of a semisimple Lie algebra is semisimple.
\end{corollary} 

Let ${\rm Der}\g$ be the Lie algebra of derivations of a Lie algebra $\g$. We have a homomorphism
${\rm ad}: \g\to {\rm Der}\g$ whose kernel is the center $\mathfrak{z}(\g)$. 
Thus if $\g$ has trivial center (e.g., is semisimple) then the map ${\rm ad}$ is injective and identifies $\g$ with a Lie subalgebra of ${\rm Der}\g$.  
Moreover, for $d\in {\rm Der}\g$ and $x\in \g$, we have 
$$
[d,{\rm ad}x](y)=d[x,y]-[x,dy]=[dx,y]={\rm ad}(dx)(y). 
$$
Thus $\g\subset {\rm Der}\g$ is an ideal.  

\begin{proposition}\label{g=derg} If $\g$ is semisimple then $\g={\rm Der}\g$. 
\end{proposition} 

\begin{proof} 
Consider the invariant symmetric bilinear form 
$$
K(a,b)={\rm Tr}|_\g(ab)
$$
on ${\rm Der}\g$. This is an extension of the Killing form of $\g$ to ${\rm Der}\g$, so its restriction to $\g$ is nondegenerate. Let $I=\g^\perp$ be the orthogonal complement of $\g$ in ${\rm Der}\g$ under $K$. 
It follows that $I$ is an ideal, $I\cap \g=0$, and $I\oplus \g={\rm Der}\g$. Since both $I$ and $\g$ are ideals, we have $[\g,I]=0$. Thus for $d\in I$ and $x\in \g$, $[d,{\rm ad}x]={\rm ad}(dx)=0$, so 
$dx$ belongs to the center of $\g$. Thus $dx=0$, i.e., $d=0$. It follows that $I=0$, as claimed.  
\end{proof} 

\begin{corollary}\label{auto} Let $\g$ be a real or complex semisimple Lie algebra, and $G={\rm Aut}(\g)\subset GL(\g)$. Then $G$ is a Lie group with 
$$
{\rm Lie}G={\rm Der}\g\cong \g.
$$ 
Thus $G$ acts on $\g$ by the adjoint action. 
\end{corollary} 

\begin{proof} It is easy to show that for any finite dimensional real or complex Lie algebra $\g$, ${\rm Aut}(\g)$ is a Lie group with Lie algebra ${\rm Der}(\g)$, so the statement follows from Proposition \ref{g=derg}. 
\end{proof}  

\section{\bf Extensions of representations, Whitehead's theorem, complete reducibility}

\subsection{Extensions} \label{extee}
Let $\g$ be a Lie algebra and $U,W$ be representations of $\g$. 
We would like to classify all representations $V$ which fit into a short exact sequence 
\begin{equation}\label{seseq}
0\to U\to V\to W\to 0,
\end{equation} 
i.e., $U\subset V$ is a subrepresentation such that the surjection $p: V\to W$ has kernel 
$U$ and thus defines an isomorphism $V/U\cong W$. In other words, $V$ is endowed with a 2-step filtration with $F_0V=U$ and $F_1V=V$ such that $F_1V/F_0V=W$, so ${\rm gr}(V)=U\oplus W$.
To do so, pick a splitting of this sequence as a sequence of vector spaces, i.e. an injection $i: W\to V$ (not a homomorphism of representations, in general) such that $p\circ i={\rm Id}_W$. This defines a linear isomorphism $\widetilde i: U\oplus W\to V$ given by $(u,w)\mapsto u+i(w)$, which allows us to rewrite the action of $\g$ on $V$ as an action on $U\oplus W$. Since $\widetilde i$ is not in general a morphism of representations, this action is given by 
$$
\rho(x)(u,w)=(xu+a(x)w,xw)
$$ 
where $a: \g\to {\rm Hom}_{\bf k}(W,U)$ is a linear map, 
and $\widetilde i$ is a morphism of representations iff $a=0$. 

What are the conditions on $a$ to give rise to a representation? We compute: 
$$
\rho([x,y])(u,w)= ([x,y]u+a([x,y])w,[x,y]w),
$$
$$
[\rho(x),\rho(y)](u,w)=([x,y]u+([x,a(y)]+[a(x),y])w,[x,y]w).
$$
Thus the condition to give a representation is the Leibniz rule
$$
a([x,y])=[x,a(y)]+[a(x),y]=[x,a(y)]-[y,a(x)].
$$
In general, if $E$ is a representation of $\g$ then a linear function 
$a: \g\to E$ such that 
$$
a([x,y])=x\circ a(y)-y\circ a(x)
$$
is called a ${\bf 1-cocycle}$ of $\g$ with values in $E$. 
The space of 1-cocycles is denoted by $Z^1(\g,E)$. 

\begin{example} We have $Z^1(\g,{\bf k})=(\g/[\g,\g])^*$ 
and $Z^1(\g,\g)={\rm Der}\g$. 
\end{example} 

Thus we see that in our setting $a: \g\to {\rm Hom}_{\bf k}(W,U)$ 
defines a representation if and only if $a\in Z^1(\g, {\rm Hom}_{\bf k}(W,U))$. 
Denote the representation $V$ attached to such $a$ by $V_a$. 
Then we have a natural short exact sequence 
$$
0\to U\to V_a\to W\to 0.
$$
It may, however, happen that some $a\ne 0$ defines a trivial extension $V\cong U\oplus W$, 
i.e., $V_a\cong V_0$, and more generally $V_a\cong V_b$ for $a\ne b$. Let us determine when this happens. More precisely, let us look for isomorphisms $f: V_a\to V_b$ preserving the structure of the short exact sequences, i.e., such that ${\rm gr}(f)={\rm Id}$. Then 
$$
f(u,w)=(u+Aw,w)
$$
where $A: W\to U$ is a linear map. Then we have 
$$
xf(u,w)=x(u+Aw,w)=(xu+xAw+b(x)w,xw)
$$
and 
$$
fx(u,w)=f(xu+a(x)w,xw)=(xu+a(x)w+Axw,xw),
$$
so we get that $xf=fx$ iff
$$
[x,A]=a(x)-b(x). 
$$
In particular, setting $b=0$, we see that $V$ is a trivial extension if and only if 
$a(x)=[x,A]$ for some $A$. 

More generally, if $E$ is a $\g$-module, the linear function $a: \g\to E$ given by $a(x)=xv$ for some $v\in E$ is called the {\bf 1-coboundary} of $v$, and one writes $a=dv$. The space of 1-coboundaries is denoted by $B^1(\g,E)$; it is easy to see that it is a subspace of $Z^1(\g,E)$, i.e., a 1-coboundary is always a 1-cocycle. Thus in our setting $f: V_a\to V_b$ is an isomorphism of representations iff 
$$
a-b=dA, 
$$ 
i.e., there is an isomorphism $f: V_a\cong V_b$ with ${\rm gr}(f)={\rm Id}$ 
if and only if $a=b$ in the quotient space 
$$
{\rm Ext}^1(W,U):=Z^1(\g,{\rm Hom}_{\bf k}(W,U))/B^1(\g,{\rm Hom}_{\bf k}(W,U)).
$$
The notation is justified by the fact that this space parametrizes extensions 
of $W$ by $U$. More precisely, every short exact sequence \eqref{seseq} gives rise to a class $[V]\in {\rm Ext}^1(W,U)$, and the extension defined by this sequence is trivial iff $[V]=0$. 

More generally, for a $\g$-module $E$ the space 
$$
H^1(\g,E):=Z^1(\g,E)/B^1(\g,E)
$$
is called the {\bf first cohomology} of $\g$ with coefficients in $E$. 
Thus, 
$$
\Ext^1(W,U)=H^1(\g, {\rm Hom}_{\bf k}(W,U)).
$$

\begin{lemma}\label{longex} A short exact sequence $0\to U\to V\to W\to 0$ 
gives rise to an exact sequence 
$$
H^1(\g,U)\to H^1(\g,V)\to H^1(\g,W).
$$ 
\end{lemma} 

\begin{exercise} Prove Lemma \ref{longex}.  
\end{exercise}

\subsection{Whitehead's theorem} \label{Whi}

We have shown in Corollary \ref{gperf} and Proposition \ref{g=derg} 
that for a semisimple $\g$ over a field of characteristic zero, 
$H^1(\g,{\bf k})=(\g/[\g,\g])^*=0$, and $H^1(\g,\g)={\rm Der}\g/\g=0$. In fact, these are special cases of a more general theorem. 

\begin{theorem}\label{vanish} (Whitehead) If $\g$ is semisimple over a field $\k$ characteristic zero then for every finite dimensional representation
$V$ of $\g$, $H^1(\g,V)=0$. 
\end{theorem} 

\subsection{Proof of Theorem \ref{vanish}}\label{vani} By extension of scalars, we may assume without loss of generality that $\k$ is algebraically closed. We will use the following lemma, which actually holds over any field. 

\begin{lemma}\label{leem1} Let $E$ be a representation of a Lie algebra $\g$ and 
$C\in U(\g)$ be a central element which acts 
by $0$ on the trivial representation of $\g$ and by some scalar $\lambda\ne 0$
on $E$. Then $H^1(\g,E)=0$. 
\end{lemma} 
 
\begin{proof} We have seen that $H^1(\g,E)={\rm Ext}^1({\bf k},E)$, so 
our job is to show that any extension 
$$
0\to E\to V\to {\bf k}\to 0
$$
splits. Let $p: V\to {\bf k}$ be the projection. We claim that 
there exists a unique vector $v\in V$ such that $p(v)=1$ and 
$Cv=0$. Indeed, pick some $w\in V$ with $p(w)=1$. 
Then $Cw\in E$, so set $v=w-\lambda^{-1}Cw$. Since $C^2w=\lambda Cw$, 
we have $Cv=0$. Also if $v'$ is another such vector then $v-v'\in  E$ so $C(v-v')=\lambda(v-v')=0$, hence $v=v'$. 

Thus ${\bf k}v\subset V$ is a $\g$-invariant complement to $E$ (as $C$ is central), which implies the statement. 
\end{proof} 

It remains to construct a central element of $U(\g)$ for a semisimple Lie algebra $\g$ to which we can apply Lemma \ref{leem1}. This can be done as follows. 
Let $a_i$ be a basis of $\g$ and $a^i$ the dual basis under an invariant inner product on $\g$ (for example, the Killing form). Define the {\bf (quadratic) Casimir element} 
$$
C:=\sum_i a_ia^i.
$$
It is easy to show that $C$ is independent of  the choice of the basis (although it depends on the choice of the inner product). 
Also $C$ is central: for $y\in \g$, 
$$
[y,C]=\sum_i ([y,a_i]a^i+a_i[y,a^i])=0
$$
since 
$$
\sum_i ([y,a_i]\otimes a^i+a_i\otimes [y,a^i])=0
$$
(this is seen by taking the inner product of the first tensorand with $a^j$ and using the invariance of the inner product).
Finally, note that for $\g=\mathfrak{sl}_2$, $C$ is proportional to the Casimir element 
$2fe+\frac{h^2}{2}+h=ef+fe+\frac{h^2}{2}$ considered previously, 
as the basis $f,e,\frac{h}{\sqrt{2}}$ is dual to the basis 
$e,f,\frac{h}{\sqrt{2}}$ under an invariant inner product of $\g$. 

The key lemma used in the proof of Theorem \ref{vanish} is the following. 

\begin{lemma}\label{leem2} Let $\g$ be semisimple and 
$V$ be a nontrivial finite dimensional irreducible $\g$-module. Then 
there is a central element $C\in U(\g)$ such that $C|_{\bf k}=0$ and $C|_V\ne 0$.  
\end{lemma} 

\begin{proof} Consider the invariant symmetric bilinear form on $\g$
$$
B_V(x,y)={\rm Tr}|_V(xy). 
$$
We claim that $B_V\ne 0$. Indeed, let $\bar\g\subset \mathfrak{gl}(V)$ be the image of $\g$. 
By Lemma \ref{lee}, if $B_V=0$ then $\overline \g$ is solvable, so, being the quotient of a semisimple Lie algebra $\g$, it must be zero, 
hence $V$ is trivial, a contradiction. 

Let $I={\rm Ker}(B_V)$. Then $I\subset \g$ is an ideal, so by Proposition \ref{dirsu2}, 
$\g=I\oplus \g'$ 
for some semisimple Lie algebra $\g'$, and $B_V$ is nondegenerate on $\g'$. 
Let $C$ be the Casimir element of $U(\g')$ corresponding to the inner product $B_V$. Then ${\rm Tr}_V(C)=\sum_i B_V(a_i,a^i)=\dim\g'$, so 
$C|_V=\frac{\dim \g'}{\dim V}\ne 0$. Also it is clear that $C|_\k=0$, so the lemma follows. 
\end{proof} 

\begin{corollary}\label{coo1} For any irreducible finite dimensional representation $V$ of a semisimple Lie algebra $\g$ over ${\bf k}$, we have $H^1(\g,V)=0$. 
\end{corollary} 

\begin{proof} If $V$ is nontrivial, this follows from Lemmas \ref{leem1} and \ref{leem2}. 
On the other hand, if $V={\bf k}$ then $H^1(\g,V)=(\g/[\g,\g])^*=0$.  
\end{proof} 

Now we can prove Theorem \ref{vanish}. By Lemma \ref{longex}, it suffices to prove the theorem for irreducible $V$, which is guaranteed by Corollary \ref{coo1}.

\subsection{Complete reducibility of representations of semisimple Lie algebras} 

\begin{theorem}\label{comred} Every finite dimensional representation of a semisimple Lie algebra $\g$
over a field of characteristic zero is completely reducible, i.e., isomorphic to a direct sum of irreducible representations.  
\end{theorem} 

\begin{proof} Theorem \ref{vanish} implies that 
for any finite dimensional representations $W,U$ of $\g$ one has 
${\rm Ext}^1(W,U)=0$. Thus any short exact sequence 
$$
0\to U\to V\to W\to 0
$$
splits, which implies the statement. 
\end{proof} 

\begin{corollary}\label{redualg} A reductive Lie algebra $\g$ in characteristic zero is uniquely a direct sum of a semisimple and abelian Lie algebra. 
\end{corollary} 

\begin{proof} Consider the adjoint representation of $\g$. It is a representation of $\g'=\g/\mathfrak{z}(\g)$, which fits into a short exact sequence 
$$
0\to \mathfrak{z}(\g)\to \g\to \g'\to 0.
$$
By complete reducibility, this sequence splits, i.e. we have a decomposition 
$\g=\g'\oplus \mathfrak{z}(\g)$ as a direct sum of ideals, and it is 
clearly unique. 
\end{proof} 

\section{\bf Structure of semisimple Lie algebras, I} 

\subsection{Semisimple elements} 
Let ${\bf k}$ be an algebraically closed field, and $\g$ be a finite dimensional Lie algebra over ${\bf k}$. 
Let $x\in \g$. Let $\g_\lambda\subset \g$ be the generalized eigenspace of ${\rm ad}x$ with eigenvalue $\lambda$. 
Then $\g=\oplus_\lambda \g_\lambda$. 

\begin{lemma}\label{gra} We have $[\g_\lambda,\g_\mu]\subset \g_{\lambda+\mu}$.
\end{lemma} 

\begin{proof} Let $y\in \g_\lambda,z\in \g_\mu$. We have 
$$
({\rm ad}x-\lambda-\mu)^N([y,z])=
$$
$$
\sum_{p+q+r+s=N}(-1)^{r+s}\frac{N!}{p!q!r!s!}\lambda^r\mu^s[({\rm ad}x)^p(y),({\rm ad}x)^q(z)]=
$$
$$
\sum_{k+\ell=N}\frac{N!}{k!\ell!}[({\rm ad}x-\lambda)^k(y),({\rm ad} x-\mu)^\ell(z)].
$$
Thus if $({\rm ad}x-\lambda)^n(y)=0$ and $({\rm ad}x-\mu)^m(z)=0$ then 
$$
({\rm ad}x-\lambda-\mu)^{m+n}([y,z])=0,
$$ 
so $[y,z]\in \g_{\lambda+\mu}$.
\end{proof} 

\begin{definition} An element $x$ of a Lie algebra $\g$ is called {\bf semisimple} if the operator ${\rm ad}x$ is semisimple and {\bf nilpotent} if this operator is nilpotent. 
\end{definition} 

It is clear that any element which is both semisimple and nilpotent is central, so for a semisimple Lie algebra it must be zero. Note also that for $\g=\mathfrak{sl}_n({\bf k})$ this coincides with the usual definition.  

\begin{proposition}\label{jordecLie} (Jordan decomposition in a semisimple Lie algebra) 
Let $\g$ be a semisimple Lie algebra and ${\rm char}(\k)=0$. Then every element 
$x\in \g$ has a unique decomposition as $x=x_s+x_n$, where $x_s$ is semisimple,
$x_n$ is nilpotent and $[x_s,x_n]=0$. Moreover, if $y\in \g$ and $[x,y]=0$ then $[x_s,y]=[x_n,y]=0$.  
\end{proposition} 

\begin{proof} Recall that $\g\subset {\mathfrak{gl}}(\g)$ via the adjoint representation. So 
we can consider the Jordan decomposition $x=x_s+x_n$, with $x_s,x_n\in {\mathfrak{gl}}(\g)$. We have $x_s(y)=\lambda y$ for $y\in \g_\lambda$. Thus $y\mapsto x_s(y)$ is a derivation of $\g$ by Lemma \ref{gra}. But by Proposition \ref{g=derg} every derivation of $\g$ is inner, which implies that $x_s\in \g$, hence $x_n\in \g$. It is clear that $x_s$ is semisimple, $x_n$ is nilpotent, and $[x_s,x_n]=0$. Also if $[x,y]=0$ then ${\rm ad}y$ preserves $\g_\lambda$ for all $\lambda$, hence $[x_s,y]=0$ as linear operators on $\g$ and thus as elements of $\g$. This also implies that the decomposition is unique since 
if $x=x_s'+x_n'$ then $[x_s,x_s']=[x_n,x_n']=0$, so $x_s-x_s'=x_n'-x_n$ is both semisimple and nilpotent, hence zero. 
\end{proof} 

\begin{corollary} Any semisimple Lie algebra $\g\ne 0$ over a field of characteristic zero contains nonzero semisimple elements.
\end{corollary} 

\begin{proof} Otherwise, by Proposition \ref{jordecLie}, every element $x\in \g$ is nilpotent, which by Engel's theorem would imply that $\g$ is nilpotent, hence solvable, hence zero. 
\end{proof} 

\subsection{Toral subalgebras} From now on we assume that ${\rm char}(\bold k)=0$ unless specified otherwise.  

\begin{definition} An abelian Lie subalgebra $\h\subset \g$ 
is called a {\bf toral subalgebra} if it consists of semisimple elements.\footnote{In fact, we will see later that over an algebraically closed field of characteristic zero, a finite dimensional Lie algebra consisting of semisimple elements is automatically abelian.} 
\end{definition} 

\begin{proposition}\label{toral}
Let $\g$ be a semisimple Lie algebra, $\h\subset \g$ a toral subalgebra, and $B$ a nondegenerate invariant symmetric bilinear form on $\g$ (e.g., the Killing form). 

(i) We have a decomposition $\g=\oplus_{\alpha\in \h^*}\g_\alpha$,
where $\g_\alpha$ is the subspace of $x\in \g$ such that for $h\in \h$ we have 
$[h,x]=\alpha(h)x$, and $\g_0\supset \h$. 

(ii) We have $[\g_\alpha,\g_\beta]\subset\g_{\alpha+\beta}$. 

(iii) If $\alpha+\beta\ne 0$ then $\g_\alpha$ and $\g_\beta$ are orthogonal under $B$. 

(iv) $B$ restricts to a nondegenerate pairing $\g_\alpha\times \g_{-\alpha}\to {\bf k}$. 
\end{proposition} 

\begin{proof} (i) is just the joint eigenspace decomposition for $\h$ acting in $\g$. (ii) is a very easy special case of Lemma \ref{gra}. (iii) and (iv) follow from the fact that $B$ is nondegenerate and invariant. 
\end{proof} 

\begin{corollary}\label{cororedu} (i) The Lie subalgebra $\g_0\subset \g$ is reductive.

(ii) if $x\in \g_0$ then $x_s,x_n\in \g_0$. 
\end{corollary} 

\begin{proof} (i) This follows from Proposition \ref{reducri} and the fact that the form $(x,y)\mapsto {\rm Tr}|_\g(xy)$ on $\g_0$ is nondegenerate (Proposition \ref{toral}(iv) for the Killing form of $\g$). 

(ii) We have $[h,x]=0$ for $h\in \h$, so $[h,x_s]=0$, hence $x_s\in \g_0$. 
\end{proof} 

\subsection{Cartan subalgebras} 

\begin{definition} A {\bf Cartan subalgebra} of a semisimple Lie algebra $\g$ 
is a toral subalgebra $\h\subset \g$ such that $\g_0=\h$. 
\end{definition} 

\begin{example}\label{carsub} Let $\g={\mathfrak{sl}}_n(\bf k)$. Then the subalgebra $\h\subset \g$ 
of diagonal matrices is a Cartan subalgebra.
\end{example} 

It is clear that any Cartan subalgebra is a maximal toral subalgebra of $\g$. 
The following theorem, stating the converse, shows that Cartan subalgebras exist. 

\begin{theorem} Let $\h$ be a maximal toral subalgebra of $\g$. Then 
$\h$ is a Cartan subalgebra. 
\end{theorem} 

\begin{proof} Let $x\in \g_0$, then by Corollary \ref{cororedu}(ii) $x_s\in \g_0$, so 
$x_s\in \h$ by maximality of $\h$. Thus ${\rm ad}x|_{\g_0}={\rm ad}x_n|_{\g_0}$ 
is nilpotent. So by Engel's theorem $\g_0$ is nilpotent. But it is also reductive, hence abelian. 

Now let us show that every $x\in \g_0$ which is nilpotent in $\g$ must be zero. Indeed, in this case, 
for any $y\in \g_0$, the operator ${\rm ad}x\cdot {\rm ad}y: \g\to \g$ is nilpotent (as $[x,y]=0$), so ${\rm Tr}|_\g({\rm ad}x\cdot {\rm ad}y)=0$. But this form is nondegenerate on $\g_0$, which implies that $x=0$. 

Thus for any $x\in \g_0$, $x_n=0$, so $x=x_s$ is semisimple. Hence $\g_0=\h$ and $\h$ is a Cartan subalgebra. 
\end{proof} 

We will show in Theorem \ref{conjcar} 
that all Cartan subalgebras of $\g$ are conjugate under ${\rm Aut}(\g)$, 
in particular they all have the same dimension, which is called the {\bf rank} of $\g$. 

\subsection{Root decomposition} 

\begin{proposition}\label{toral1}
Let $\g$ be a semisimple Lie algebra, $\h\subset \g$ a Cartan subalgebra, and $B$ a nondegenerate invariant symmetric bilinear form on $\g$ (e.g., the Killing form). 

(i) We have a decomposition $\g=\h\oplus\bigoplus_{\alpha\in R}\g_\alpha$,
where $\g_\alpha$ is the subspace of $x\in \g$ such that for $h\in \h$ we have 
$[h,x]=\alpha(h)x$, and $R$ is the (finite) set of $\alpha\in \h^*$, $\alpha\ne 0$, such that 
$\g_\alpha\ne 0$. 

(ii) We have $[\g_\alpha,\g_\beta]\subset \g_{\alpha+\beta}$. 

(iii) If $\alpha+\beta\ne 0$ then $\g_\alpha$ and $\g_\beta$ are orthogonal under $B$. 

(iv) $B$ restricts to a nondegenerate pairing $\g_\alpha\times \g_{-\alpha}\to {\bf k}$. 
\end{proposition} 

\begin{proof} This immediately follows from Proposition \ref{toral}. 
\end{proof} 

\begin{definition} The set $R$ is called the {\bf root system} of $\g$ and its elements are called {\bf roots}. 
\end{definition} 

\begin{proposition} Let $\g_1,...,\g_n$ be simple Lie algebras and let $\g =\oplus_i\g_i$.

(i) Let $\h_i \subset \g_i$ be Cartan subalgebras of $\g_i$ and $R_i \subset \h_i^*$
the corresponding root systems of $\g_i$. Then $\h =\oplus_i \h_i$ is a Cartan subalgebra in $\g$ and the corresponding root system $R$ is the disjoint union of $R_i$.

(ii) Each Cartan subalgebra in $\g$ has the form $\h=\oplus_i \h_i$ where $\h_i \subset \g_i$
is a Cartan subalgebra in $\g_i$.
\end{proposition} 

\begin{proof} (i) is obvious. To prove (ii), given a Cartan subalgebra $\h\subset \g$, let 
$\h_i$ be the projections of $\h$ to $\g_i$. It is easy to see that $\h_i\subset \g_i$ are Cartan subalgebras. Also $\h\subset \oplus_i \h_i$ and the latter is toral, which implies that $\h=\oplus_i \h_i$ since $\h$ is a Cartan subalgebra. 
\end{proof} 

\begin{example}\label{Anm1} Let $\g=\mathfrak{sl}_n(\bf k)$. Then the subspace of diagonal matrices $\h$ is a Cartan subalgebra (cf. Example \ref{carsub}), and it can be naturally identified with the space of vectors $\bold x=(x_1,...,x_n)$ 
such that $\sum_i x_i=0$. Let ${\bold e}_i$ be the linear functionals on this space given by ${\bold e}_i(\bold x)=x_i$. 
We have $\g=\h\oplus\bigoplus_{i\ne j}{\bf k}E_{ij}$ and $[\bold x,E_{ij}]=(x_i-x_j)E_{ij}$. 
Thus the root system $R$ consists of vectors ${\bold e}_i-{\bold e}_j\in \h^*$ for $i\ne j$ (so there are $n(n-1)$ roots). 
\end{example} 

Now let $\g$ be a semisimple Lie algebra and $\h\subset \g$ a Cartan subalgebra. Let $(,)$ 
be a nondegenerate invariant symmetric bilinear form on $\g$, for example the Killing form. 
Since the restriction of $(,)$ to $\h$ is nondegenerate, it defines an isomorphism $\h\to \h^*$ given by $h\mapsto (h,?)$. The inverse of this isomorphism will be denoted by $\alpha\mapsto H_\alpha$. We also have the inverse form on $\h^*$ which we also will denote by $(,)$; it is given by $(\alpha,\beta):=\alpha(H_\beta)=(H_\alpha,H_\beta)$. 

\begin{lemma}\label{comef} For any $e\in \g_\alpha,f\in \g_{-\alpha}$ we have 
$$
[e,f]=(e,f)H_\alpha. 
$$
\end{lemma} 

\begin{proof} We have $[e,f]\in \h$ so it is enough to show that the inner product of both sides with any $h\in \h$ is the same. We have 
$$
([e,f],h)=(e,[f,h])=\alpha(h)(e,f)=((e,f)H_\alpha,h),
$$ 
as desired. 
\end{proof} 

\begin{lemma} (i) If $\alpha$ is a root then $(\alpha,\alpha)\ne 0$. 

(ii) Let $e\in \g_\alpha$, $f\in \g_{-\alpha}$ be such that $(e,f)=\frac{2}{(\alpha,\alpha)}$, and let 
$h_\alpha:=\frac{2H_\alpha}{(\alpha,\alpha)}$. Then $e,f,h_\alpha$ satisfy the commutation relations of the Lie algebra $\mathfrak{sl}_2$. 

(iii) $h_\alpha$ is independent of  the choice of $(,)$. 
\end{lemma} 

\begin{proof} (i) Pick $e\in \g_\alpha,f\in \g_{-\alpha}$ with $(e,f)\ne 0$. Let $h:= [e, f] = (e, f )H_\alpha$ 
(by Lemma \ref{comef}) and consider the Lie algebra $\mathfrak{a}$ 
generated by $e, f , h$. Then we see that 
$$
[h, e] =\alpha(h)e=(\alpha,\alpha)(e,f)e,\ [h,f]=-\alpha(h)f=-(\alpha,\alpha)(e,f)f.
$$
Thus if $(\alpha,\alpha)=0$ then $\mathfrak{a}$ is a solvable Lie algebra. 
By Lie's theorem, we can choose a basis in $\g$ such that operators ${\rm ad}e$, ${\rm ad}f$, ${\rm ad}h$ are upper triangular. Since $h = [e, f ]$, ${\rm ad}h$ will be strictly upper-triangular and
thus nilpotent. But since $h\in \h$, it is also semisimple. Thus, ${\rm ad}h = 0$, so $h=0$ as $\g$ is semisimple. On the other hand, $h = (e, f )H_\alpha\ne 0$. This contradiction proves the first part of the theorem.

(ii) This follows immediately from the formulas in the proof of (i). 

(iii) It's enough to check the statement for a simple Lie algebra, and in this case this is easy since $(,)$ is unique up to scaling by Corollary \ref{uniscal}. 
\end{proof} 

The Lie subalgebra of $\g$ spanned by $e,f,h_\alpha$, which we've shown to be isomorphic to 
${\mathfrak{sl}}_2(\bold k)$, will be denoted by ${\mathfrak{sl}}_2(\bold k)_\alpha$ (we will see that $\g_\alpha$ are 1-dimensional so it is independent of  the choices). 

\begin{proposition} Let $\mathfrak{a}_\alpha={\bf k}H_\alpha\oplus \bigoplus_{k\ne 0}\g_{k\alpha}\subset \g$. 
Then $\mathfrak{a}_\alpha$ is a Lie subalgebra of $\g$.   
\end{proposition} 

\begin{proof} This follows from the fact that for $e\in \g_{k\alpha},f\in \g_{-k\alpha}$ we have 
$[e,f]=(e,f)H_{k\alpha}=k(e,f)H_\alpha$. 
\end{proof} 

\begin{corollary} (i) The space $\g_\alpha$ is 1-dimensional for each root $\alpha$ of $\g$. 

(ii) If $\alpha$ is a root of $\g$ and $k\ge 2$ is an integer then $k\alpha$ is not a root of $\g$. 
\end{corollary} 

\begin{proof} For a root $\alpha$ the Lie algebra $\mathfrak{a}_\alpha$ contains $\mathfrak{sl}_2({\bf k})_\alpha$, so it is a finite dimensional representation of this Lie algebra. Also the kernel of $h_\alpha$ on this representation is spanned by $h_\alpha$, hence 1-dimensional, and eigenvalues of $h_\alpha$ are even integers since $\alpha(h_\alpha)=2$. 
Thus by the representation theory of $\mathfrak{sl}_2$ (Subsection \ref{sl2rep}), this representation is irreducible, 
i.e., eigenspaces of $h_\alpha$ (which are $\g_{k\alpha}$ and ${\bf k}H_\alpha$) are 1-dimensional. Therefore the map $[e,?]:  \g_{\alpha}\to \g_{2\alpha}$ is zero (as $\g_\alpha$ is spanned by $e$). So again by representation theory of $\mathfrak{sl}_2$ we have 
$\g_{k\alpha}=0$ for $|k|\ge 2$.  
\end{proof} 

\begin{theorem}\label{rootde} Let $\g$ be a semisimple Lie algebra with Cartan subalgebra
$\h$ and root decomposition $\g = \h\oplus\bigoplus_{\alpha \in R}\g_\alpha$. Let $( , )$ be a non-degenerate symmetric invariant bilinear form on $\g$.

(i) $R$ spans $\h^*$ as a vector space, and elements $h_
\alpha$, $\alpha \in R$ span
$\h$ as a vector space.

(ii) For any two roots $\alpha,\beta$, the number 
$a_{\alpha,\beta}:=\beta(h_\alpha)=\frac{2(\alpha,\beta)}{(\alpha,\alpha)}$ is an integer.

(iii) For $\alpha \in R$, define the {\bf reflection operator} $s_\alpha : \h^*\to 
\h^*$ by
$$
s_\alpha(\lambda) = \lambda-\lambda(h_\alpha)\alpha = \lambda - 2\frac{(\lambda, \alpha)}{
(\alpha,\alpha)}\alpha. 
$$
Then for any roots $\alpha$, $\beta$, $s_\alpha(\beta)$ is also a root.

(iv) For roots $\alpha, \beta \ne \pm \alpha$, the subspace
$V_{\alpha,\beta} =\oplus_{k\in \Bbb Z}\g_{\beta+k\alpha}\subset \g$ 
is an irreducible representation of $\mathfrak{sl}_2({\bf k})_\alpha$.
\end{theorem}

\begin{proof} (i) Suppose $h\in \h$ is such that $\alpha(h)=0$ for all roots $\alpha$. Then 
${\rm ad}h=0$, hence $h=0$ as $\g$ is semisimple. This implies both statements. 

(ii) $a_{\alpha,\beta}$ is the eigenvalue of $h_\alpha$ on $e_\beta$, hence an integer by the representation theory of $\mathfrak{sl}_2$ (Subsection \ref{sl2rep}). 

(iii) Let $x\in \g_\beta$ be nonzero. If $\beta(h_\alpha)\ge 0$ then let $y=f_\alpha^{\beta(h_\alpha)}x$. 
If $\beta(h_\alpha)\le 0$ then let $y=e_\alpha^{-\beta(h_\alpha)}x$. 
Then by representation theory of $\mathfrak{sl}_2$, $y\ne 0$. We also have 
$[h,y]=s_\alpha(\beta)(h)y$. This implies the statement. 

(iv) It is clear that $V_{\alpha,\beta}$ is a representation. Also all $h_\alpha$-eigenspaces in $V_{\alpha,\beta}$ are 1-dimensional, and the eigenvalues are either all odd or all even. 
This implies that it is irreducible.   
\end{proof} 

\begin{corollary} Let $\k=\Bbb C$ and $\h_{\Bbb R}$ be the $\Bbb R$-span of all $h_\alpha$. Then $\h=\h_{\Bbb R}\oplus i\h_{\Bbb R}$ and the restriction of 
the Killing form to $\h_{\Bbb R}$ is real-valued and positive definite. 
\end{corollary} 

\begin{proof} It follows from the previous theorem that the eigenvalues of ${\rm ad}h$, $h\in \h_{\Bbb R}$, are real. So $\h_{\Bbb R}\cap i\h_{\Bbb R}=0$, which implies the first statement. Now, $K(h,h)=\sum_i \lambda_i^2$ where $\lambda_i$ are the eigenvalues of ${\rm ad}h$ (which are not all zero if $h\ne 0$). Thus $K(h,h)>0$ if $h\ne 0$. 
\end{proof} 

\newpage

\section{\bf Structure of semisimple Lie algebras, II} 

\subsection{Strongly regular (regular semisimple) elements}\label{strreg} In this section we will discuss another way of constructing Cartan subalgebras. First consider an example. 

\begin{example} Let $\g=\mathfrak{sl}_n(\Bbb C)$ and $x\in \g$ be a diagonal matrix with distinct eigenvalues. Then the centralizer $\h=C(x)$ is the space of all diagonal matrices of trace $0$, which is a Cartan subalgebra.  Thus the same applies to any diagonalizable matrix with distinct eigenvalues, i.e., a generic matrix (one for which the discriminant of the characteristic polynomial is nonzero). 
\end{example}

So we may hope that if we take a generic element $x$ in a semisimple Lie algebra then its centralizer is a Cartan subalgebra. But for that we have to define what we mean by generic.  

\begin{definition} The {\bf nullity} $n(x)$ of an element $x\in \g$ is the multiplicity of the eigenvalue $0$ for the operator ${\rm ad}x$ (i.e., the dimension of the generalized $0$-eigenspace). The {\bf rank} ${\rm rank}\g$ of $\g$ is the minimal value of $n(x)$. An element $x$ is {\bf strongly regular} if $n(x)={\rm rank}\g$.
\end{definition} 

\begin{example} It is easy to check that for $\g={\mathfrak{sl}}_n$, $x$ is strongly regular if and only if 
its eigenvalues are distinct. 
\end{example} 

We will need the following auxiliary lemma. 

\begin{lemma}\label{cdo0} Let $P(z_1,...,z_n)$ be a nonzero complex polynomial, and $U\subset \Bbb C^n$ 
be the set of points $(z_1,...,z_n)\in \Bbb C^n$ such that $P(z_1,...,z_n)\ne 0$. Then $U$ is path-connected, dense and open. 
\end{lemma} 

\begin{proof} It is clear that $U$ is open, since it is the preimage of the open set $\Bbb C^\times\subset \Bbb C$ under a continuous map. It is also dense, as its complement, the hypersurface $P=0$, cannot contain a ball. Finally, to see that it is path-connected, take $\bold x,\bold y\in U$, and consider 
the polynomial $Q(t):=P((1-t)\bold x+t\bold y)$. It has only finitely many zeros, hence the entire complex line $\bold z=(1-t)\bold x+t\bold y$ except finitely many points is contained in $U$. Clearly, $\bold x$ and $\bold y$ can be connected by a path inside this line avoiding this finite set of points. 
\end{proof} 

\begin{lemma}\label{cdo} Let $\g$ be a complex semisimple Lie algebra. Then the set $\g^{\rm sr}$ of strongly regular elements is connected, dense and open in $\g$.
\end{lemma} 

\begin{proof} Consider the characteristic polynomial $P_x(t)$ of ${\rm ad}x$. We have 
$$
P_x(t)=t^{{\rm rank}\g}(t^m+a_{m-1}(x)t^{m-1}+...+a_0(x)), 
$$
where $m=\dim \g-{\rm rank}\g$ and $a_i$ are some polynomials of $x$, with $a_0\ne 0$. 
Then $x$ is strongly regular if and only if $a_0(x)\ne 0$. This implies the statement by Lemma \ref{cdo0}. 
\end{proof} 

\begin{proposition}\label{dimeq} Let $\g$ be a complex semisimple Lie algebra and $\h\subset \g$ a Cartan subalgebra. Then 

(i) $\dim\h={\rm rank}\g$; and 

(ii) the set $\h^{\rm reg}:=\h\cap \g^{\rm sr}$ coincides with the set 
$$
V:=\lbrace h\in \h: \alpha(h)\ne 0\ \forall \alpha\in R\rbrace.
$$ 
In particular, $\h^{\rm reg}$ is open and dense in $\h$. 
\end{proposition} 

\begin{proof} (i) Let $G$ be a connected Lie group with Lie algebra $\g$ (we know it exists, e.g. 
we can take $G$ to be the connected component of the identity in ${\rm Aut}(\g)$).  

\begin{lemma}\label{ope}
Let $\phi: G\times \h\to \g$ be the map defined by 
$\phi(g,x):={\rm Ad} g\cdot x$. Then the set $U:=\phi(G\times V)\subset \g$
is open.  
\end{lemma} 

\begin{proof} 
Let us compute the differential $\phi_*: \g\oplus \h\to \g$ 
at the point $(1,x)$ for $x\in \h$. 
We obtain 
$$
\phi_*(y,h)=[y,x]+h.
$$
The kernel of this map is identified with the set of $y\in \g$ such that $[y,x]\in \h$. But then 
$K([y,x],z)=K(y,[x,z])=0$ for all $z\in \h$, so $[y,x]=0$. Thus ${\rm Ker}\phi_*=C(x)$. 

Now let $x\in V$. Then $C(x)=\h$. Thus $\phi_*$ is surjective by dimension count, hence $\phi$ is a submersion 
at $(1,x)$. This means that $U:={\rm Im}\phi$ contains $x$ together with its neighborhood in $\g$. 
Hence the same holds for ${\rm Ad}g\cdot x$, which implies that $U$ is open. 
\end{proof} 

Since $\g^{\rm sr}$ is open and dense and $U$ is open by Lemma \ref{ope} and non-empty, we see that $U\cap \g^{\rm sr}\ne \emptyset$. But
$$
n({\rm Ad}g\cdot x)=n(x)=\dim C(x)=\dim \h
$$
for $x\in V$. This implies that ${\rm rank}\g=\dim\h$, which yields (i). 

(ii) It is clear that for $x\in \h$, we have 
$$
n(x)=\dim {\rm Ker}({\rm ad}x)=\dim \h+\# \lbrace \alpha\in R: \alpha(x)=0\rbrace.
$$ 
This implies the statement. 
\end{proof} 

\subsection{Conjugacy of Cartan subalgebras} \label{ccsub}

\begin{theorem}\label{conscar} (i) Let $\g$ be a complex semisimple Lie algebra and let $x\in \g$ be a strongly regular semisimple element (which exists by Proposition \ref{dimeq}). Then the centralizer $C(x)$ of $x$ in $\g$ is a Cartan subalgebra of $\g$.

(ii) Any Cartan subalgebra of $\g$ is of this form.  
\end{theorem} 

\begin{proof} Consider the eigenspace decomposition of ${\rm ad}x$: 
$\g=\oplus_\lambda \g_\lambda$. Since $\Bbb C x$ is a toral subalgebra, 
the Lie algebra $\g_0=C(x)$ is reductive, with $\dim(\g_0)={\rm rank}\g$. 

We claim that $\g_0$ is also nilpotent. 
By Engel's theorem, to establish this, it suffices to show that the restriction of ${\rm ad}y$ to $\g_0$ is nilpotent for $y\in \g_0$. But ${\rm ad} (x+ty)={\rm ad}x+t{\rm ad y}$ is invertible on $\g/\g_0$ 
for small $t$, since it is so for $t=0$ and the set of invertible matrices is open. Thus ${\rm ad}(x+ty)$ must be nilpotent on $\g_0$, as the multiplicity of the eigenvalue $0$ for this operator must be (at least) ${\rm rank}\g=\dim \g_0$. 
But ${\rm ad}(x+ty)=t{\rm ad}y$ on $\g_0$, which implies that ${\rm ad}y$ is nilpotent on $\g_0$, as desired. 

Thus $\g_0$ is abelian. Moreover, for $y,z\in \g_0$ the operator ${\rm ad}y_n\cdot {\rm ad}z$ is nilpotent on $\g$ (as the product of two commuting operators one of which is nilpotent), so $K_\g(y_n,z)=0$, which implies that $y_n=0$, as $K_\g$ restricts to a nondegenerate form on $\g_0$ and $z$ is arbitrary. It follows that any $y\in \g_0$ is semisimple, so $\g_0$ is a toral subalgebra. Moreover, it is maximal since any element commuting with $x$ is in $\g_0$. Thus $\g_0$ is a Cartan subalgebra. 

(ii) Let $\h\subset \g$ be a Cartan subalgebra. By Proposition \ref{dimeq} it contains a strongly regular element $x$, which is automatically semisimple. Then $\h=C(x)$.  
\end{proof} 

\begin{corollary}\label{resi} (i) Any strongly regular element $x\in \g$ is semisimple. 

(ii) Such $x$ is contained in a unique Cartan subalgebra, namely $\h_x=C(x)$. 
\end{corollary} 

\begin{proof} (i) It is clear that if $x$ is strongly regular then so is $x_s$. Since $x\in C(x_s)$ and as shown above $C(x_s)$ is a Cartan subalgebra, it follows that $x$ is semisimple. 

(ii) Let $\h\subset \g$ be a Cartan subalgebra containing $x$. Then $\h\subset \h_x$, thus by dimension count $\h=\h_x$. 
\end{proof} 

We note that there is also a useful notion of a {\bf regular element}, which is an $x\in \g$ for which the {\bf ordinary} (rather than generalized) $0$-eigenspace of ${\rm ad}x$ (i.e., the centralizer $C(x)$ of $x$) has dimension ${\rm rank}\g$. Such elements don't have to be semisimple, e.g. the nilpotent Jordan block in $\mathfrak{sl}_n$ is regular. It follows from Corollary \ref{resi}(i) that an element is strongly regular if and only if it is both regular and semisimple. For this reason, from now on we will follow standard terminology and call strongly regular elements {\bf regular semisimple}.

\begin{theorem}\label{conjcar} 
Any two Cartan subalgebras of a complex semisimple Lie algebra $\g$ are conjugate. 
I.e., if $\h_1,\h_2\subset \g$ are two Cartan subalgebras and $G$ a connected Lie group with Lie algebra $\g$ then there exists an element $g\in G$ such that ${\rm Ad}g\cdot \h_1=\h_2$. 
\end{theorem} 

\begin{proof} By Corollary \ref{resi}(ii), every element $x\in \g^{\rm sr}$ is contained in a unique Cartan subalgebra $\h_x$. Introduce an equivalence relation on $\g^{\rm sr}$ by setting $x\sim y$ if $\h_x$ is conjugate to $\h_y$. It is clear that if $x,y\in \h$ are regular elements in a Cartan subalgebra $\h$ then $\h_x=\h_y=\h$, so for any $g\in G$, ${\rm Ad}g\cdot x\sim y$, and any element equivalent to $y$ has this form. So by Lemma \ref{ope} the equivalence class $U_y$ of $y$ is open. However, by Lemma \ref{cdo}, $\g^{\rm sr}$ is connected. Thus there is only one equivalence class. Hence any two Cartan subalgebras of the form $\h_x$ for regular $x$ are conjugate. This implies the result, since by Theorem \ref{conscar} any Cartan subalgebra is of the form $\h_x$. 
\end{proof} 

\begin{remark} The same results and proofs apply over any algebraically closed field ${\bf k}$ of characteristic zero if we use the Zariski topology instead of the usual topology of $\Bbb C^n$ when working with the notions of a connected, open and dense set. 
\end{remark} 

\subsection{Root systems of classical Lie algebras}

\begin{example} Let $\g$ be the symplectic Lie algebra $\mathfrak{sp}_{2n}(\bf k)$. 
Thus $\g$ consists of square matrices $A$ of size $2n$ such that 
$$
AJ+JA^T=0
$$
where $J=\begin{pmatrix} 0& \bold 1\\ -\bold 1& 0\end{pmatrix}$, with blocks being of size $n$. 
So we get 
$A=\begin{pmatrix} a& b\\ c& -a^T\end{pmatrix}$, where 
$b,c$ are symmetric. A Cartan subalgebra $\h$ is then spanned by matrices $A$
such that $a={\rm diag}(x_1,...,x_n)$ and $b=c=0$. So $\h\cong {\bf k}^n$. 
In this case we have roots coming from the $a$-part, which are simply the roots ${\bold e}_i-{\bold e}_j$ of 
${\mathfrak{gl}}_n\subset {\mathfrak{sp}}_{2n}$ (defined by the condition that $b=c=0$)
and also the roots coming from the $b$-part, which are ${\bold e}_i+{\bold e}_j$ (including $i=j$, when we get 
$2{\bold e}_i$), and the $c$-part, which gives the negatives of these roots, 
$-{\bold e}_i-{\bold e}_j$, including $-2{\bold e}_i$. 

This is the {\bf root system of type $C_n$.}
\end{example} 
 
\begin{example}  Let $\g$ be the orthogonal Lie algebra $\mathfrak{so}_{2n}(\bf k)$, 
preserving the quadratic form $Q=x_1x_{n+1}+...+x_nx_{2n}$. Then the story is almost the same. 
The Lie algebra $\g$ consists of square matrices $A$ of size $2n$ such that 
$$
AJ+JA^T=0
$$
where $J=\begin{pmatrix} 0& \bold 1\\ \bold 1& 0\end{pmatrix}$, with blocks being of size $n$. 
So we get 
$A=\begin{pmatrix} a& b\\ c& -a^T\end{pmatrix}$, where 
$b,c$ are now skew-symmetric. A Cartan subalgebra $\h$ is again spanned by matrices $A$
such that $a={\rm diag}(x_1,...,x_n)$ and $b=c=0$. So $\h\cong {\bf k}^n$. 
In this case we again have roots coming from the $a$-part, which are simply the roots ${\bold e}_i-{\bold e}_j$ of 
${\mathfrak{gl}}_n\subset {\mathfrak{so}}_{2n}$ (defined by the condition that $b=c=0$)
and also the roots coming form the $b$-part, which are ${\bold e}_i+{\bold e}_j$ (but now excluding $i=j$, so only for $i\ne j$), and the $c$-part, which gives the negatives of these roots, 
$-{\bold e}_i-{\bold e}_j$, $i\ne j$. 

This is the {\bf root system of type $D_n$.} 
\end{example}

\begin{example} Let $\g$ be the orthogonal Lie algebra $\mathfrak{so}_{2n+1}(\bf k)$, 
preserving the quadratic form $Q=x_0^2+2x_1x_{n+1}+...+2x_nx_{2n}$. 
Then the Lie algebra $\g$ consists of square matrices $A$ of size $2n+1$ such that 
$$
AJ+JA^T=0
$$
where 
$$
J=\begin{pmatrix} \bold 1_1& 0 & 0\\ 0& 0 &\bold 1_n\\ 0& \bold 1_n& 0\end{pmatrix},
$$
So 
we get 
$$
A=\begin{pmatrix} 0& u& -w^T\\ w &a& b\\ -u^T & c & -a^T\end{pmatrix},
$$ 
where $b,c$ are skew-symmetric. A Cartan subalgebra $\h$ is spanned by matrices $A$
such that $a={\rm diag}(x_1,...,x_n)$ and $b=c=0$, $u=w=0$. So $\h\cong {\bf k}^n$. 
In this case we again have roots coming from the $a$-part, which are simply the roots ${\bold e}_i-{\bold e}_j$ of 
${\mathfrak{gl}}_n\subset {\mathfrak{so}}_{2n+1}$ (defined by the condition that $b=c=0$, $u=w=0$)
and also the roots coming form the $b$-part, which are ${\bold e}_i+{\bold e}_j$, $i\ne j$, 
and the $c$-part, which gives the negatives of these roots, $-{\bold e}_i-{\bold e}_j$, $i\ne j$. 
But we also have the roots coming from the $w$-part, which are ${\bold e}_i$, and from the $u$ part, which are $-{\bold e}_i$. 

This is the {\bf root system of type $B_n$.}
\end{example} 

\section{\bf Root systems}

\subsection{Abstract root systems} Let $E\cong \Bbb R^r$ be a Euclidean space with a positive definite inner product.  

\begin{definition} An {\bf abstract root system} is a finite set $R\subset E\setminus 0$ satisfying the following axioms: 

(R1) $R$ spans $E$; 

(R2) For all $\alpha,\beta\in R$ the number $n_{\alpha\beta}:=\frac{2(\alpha,\beta)}{(\alpha,\alpha)}$ is an integer; 

(R3) If $\alpha,\beta\in R$ then $s_\alpha(\beta):=\beta-n_{\alpha\beta}\alpha\in R$. 

Elements of $R$ are called {\bf roots}. The number $r=\dim E$ is called the {\bf rank} of $R$. 
\end{definition} 

In particular, taking $\beta=\alpha$ in R3 yields that $R$ is centrally symmetric, i.e., $R=-R$. 
Also note that $s_\alpha$ is the reflection with respect to the hyperplane $(\alpha,x)=0$, so 
R3 just says that $R$ is invariant under such reflections. 

Note also that if $R\subset E$ is a root system, $\overline E\subset E$ a subspace, and $R'=R\cap \overline E$ then $R'$ is also a root system inside $E'={\rm Span}(R')\subset \overline E$.  

For a root $\alpha$ the corresponding {\bf coroot} $\alpha^\vee\in E^*$ 
is defined by the formula $\alpha^\vee(x)=\frac{2(\alpha,x)}{(\alpha,\alpha)}$. 
Thus $\alpha^\vee(\alpha)=2$, $n_{\alpha\beta}=\alpha^\vee(\beta)$ and 
$s_\alpha(\beta)=\beta-\alpha^\vee(\beta)\alpha$. 

\begin{definition} A root system $R$ is {\bf reduced} if for $\alpha,c\alpha\in R$, we have $c=\pm 1$. 
\end{definition} 

\begin{proposition} If $\g$ is a semisimple Lie algebra and $\h\subset \g$ a Cartan subalgebra then the corresponding set of roots $R$ is a reduced root system, and $\alpha^\vee=h_\alpha$. 
\end{proposition} 

\begin{proof} This follows immediately from Theorem \ref{rootde}.
\end{proof} 

\begin{example} 1. The root system of $\mathfrak{sl}_n$ is called $A_{n-1}$. In this case, as we have seen in Example \ref{Anm1}, the roots are ${\bold e}_i-{\bold e}_j$, and $s_{{\bold e}_i-{\bold e}_j}=(ij)$, the transposition of the $i$-th and $j$-th coordinates. 

2. The subset $\lbrace 1,2,-1,-2\rbrace $ of $\Bbb R$ is a root system which is not reduced. 
\end{example} 

\begin{definition} Let $R_1\subset E_1,R_2\subset E_2$ be root systems. An {\bf isomorphism of root systems} $\phi: R_1\to R_2$ is an isomorphism $\phi: E_1\to E_2$ 
which maps $R_1$ to $R_2$ and preserves the numbers $n_{\alpha\beta}$. 
\end{definition} 

So an isomorphism does not have to preserve the inner product, e.g. it may rescale it. 

\subsection{The Weyl group}

\begin{definition} The {\bf Weyl group} of a root system $R$ is the group of automorphisms of $E$ 
generated by $s_\alpha$. 
\end{definition} 

\begin{proposition} $W$ is a finite subgroup of $O(E)$ which preserves $R$. 
\end{proposition} 

\begin{proof} Since $s_\alpha$ are orthogonal reflections, $W\subset O(E)$. By R3, $s_\alpha$ preserves $R$. By R1 an element of $W$ is determined by its action on $R$, hence $W$ is finite.  
\end{proof}   

\begin{example} For the root system $A_{n-1}$, $W=S_n$, the symmetric group. Note that 
for $n\ge 3$, the automorphism $x\mapsto -x$ of $R$ is not in $W$, so $W$ is, in general, a proper subgroup of ${\rm Aut}(R)$. 
\end{example} 

\subsection{Root systems of rank $2$} 

If $\alpha,\beta$ are linearly independent roots in $R$ and $E'\subset E$ is spanned by $\alpha,\beta$ then $R'=R\cap E'$ is a root system in $E'$ of rank $2$. So to classify reduced root systems, it is important to classify reduced root systems of rank $2$ first. 

\begin{theorem}\label{rank2rs} Let $R$ be a reduced root system and $\alpha,\beta\in R$ be two linearly independent roots with $|\alpha|\ge |\beta|$. Let $\phi$ be the angle between $\alpha$ and $\beta$. Then we have one of the following possibilities: 

(1) $\phi=\pi/2$, $n_{\alpha\beta}=n_{\beta\alpha}=0$; 

(2a) $\phi=2\pi/3$, $|\alpha|^2=|\beta|^2$, $n_{\alpha\beta}=n_{\beta\alpha}=-1$; 

(2b) $\phi=\pi/3$, $|\alpha|^2=|\beta|^2$, $n_{\alpha\beta}=n_{\beta\alpha}=1$; 

(3a) $\phi=3\pi/4$, $|\alpha|^2=2|\beta|^2$, $n_{\alpha\beta}=-1$, $n_{\beta\alpha}=-2$; 

(3b) $\phi=\pi/4$, $|\alpha|^2=2|\beta|^2$, $n_{\alpha\beta}=1$, $n_{\beta\alpha}=2$;

(4a) $\phi=5\pi/6$, $|\alpha|^2=3|\beta|^2$, $n_{\alpha\beta}=-1$, $n_{\beta\alpha}=-3$;

(4b) $\phi=\pi/6$, $|\alpha|^2=3|\beta|^2$, $n_{\alpha\beta}=1$, $n_{\beta\alpha}=3$.
\end{theorem} 

\begin{proof} We have $(\alpha,\beta)=|\alpha|\cdot |\beta|\cos\phi$, so $n_{\alpha\beta}=2\frac{|\beta|}{|\alpha|}\cos \phi$. Thus $n_{\alpha\beta}n_{\beta\alpha}=4\cos^2\phi$. Hence this number can only take values $0,1,2,3$ (as it is an integer by R2) and $\frac{n_{\alpha\beta}}{n_{\beta\alpha}}=\frac{|\beta|^2}{|\alpha|^2}$ if $n_{\alpha\beta}\ne 0$. The rest is obtained by analysis of each case.  
\end{proof} 

In fact, all these possibilities are realized. Namely, we have root systems $A_1\times A_1$, 
$A_2$, $B_2=C_2$ (the root system of the Lie algebras $\mathfrak{sp}_4$ and $\mathfrak{so}_5$, which are in fact isomorphic, consisting of the vertices and midpoints of edges of a square), and $G_2$, generated by $\alpha,\beta$ with $(\alpha,\alpha)=6$, $(\beta,\beta)=2$, $(\alpha,\beta)=-3$, and roots being $\pm \alpha,\pm \beta$, $\pm (\alpha+\beta)$, $\pm (\alpha+2\beta)$, $\pm (\alpha+3\beta)$, $\pm (2\alpha+3\beta)$. 

\begin{theorem} Any reduced rank $2$ root system $R$ is of the form $A_1\times A_1$, $A_2$, $B_2$ or $G_2$. 
\end{theorem} 

\begin{proof} Pick independent roots $\alpha,\beta\in R$ such that the angle $\phi$ is as large as possible. Then $\phi\ge \pi/2$  (otherwise can replace $\alpha$ with $-\alpha$), so we are in one of the cases $1,2a,3a,4a$. Now the statement follows by inspection of each case, giving 
$A_1\times A_1$, $A_2$, $B_2$ and $G_2$ respectively. 
\end{proof} 

\begin{corollary}\label{neg} If $\alpha,\beta\in R$ are independent roots with $(\alpha,\beta)<0$ then $\alpha+\beta\in R$. 
\end{corollary}

\begin{proof} This is easy to see from the classification of rank $2$ root systems. 
\end{proof} 

The root systems of rank $2$ are shown in the following picture. 

\includegraphics[scale=0.4]{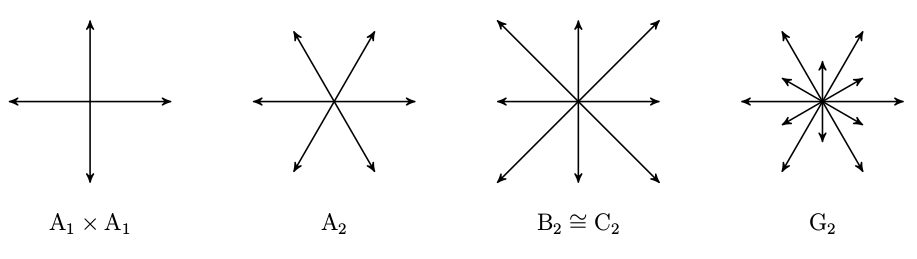}

\subsection{Positive and simple roots} 

Let $R$ be a reduced root system and 
$t\in E^*$ be such that $t(\alpha)\ne 0$ for any $\alpha\in R$. 
We say that a root is {\bf positive} (with respect to $t$) if $t(\alpha)>0$ and 
{\bf negative} if $t(\alpha)<0$. The set of positive roots is denoted by $R_+$ and of negative ones by $R_-$, so $R_+=-R_-$ and $R=R_+\cup R_-$ (disjoint union). This decomposition is called a {\bf polarization} of $R$; it depends on the choice of $t$. 

\begin{example} Let $R$ be of type $A_{n-1}$. Then for $t=(t_1,...,t_n)$ 
we have $t(\alpha)\ne 0$ for all $\alpha$ iff $t_i\ne t_j$ for any $i$, $j$. 
E.g. suppose $t_1>t_2>...>t_n$, then we have ${\bold e}_i-{\bold e}_j\in R_+$ iff $i<j$. We see that polarizations are in bijection with permutations in $S_n$, i.e., with elements of the Weyl group, which acts simply transitively on them. We will see that this is, in fact, the case for any 
reduced root system. 
\end{example} 

\begin{definition} A root $\alpha\in R_+$ is {\bf simple} if it is not a sum of two other positive roots. 
\end{definition} 

\begin{lemma} Every positive root is a sum of simple roots. 
\end{lemma} 

\begin{proof} If $\alpha$ is not simple then $\alpha=\beta+\gamma$ where $\beta,\gamma\in R_+$. 
We have $t(\alpha)=t(\beta)+t(\gamma)$, so $t(\beta),t(\gamma)<t(\alpha)$. If $\beta$ or $\gamma$ is not simple, we can continue this process, and it will terminate since $t$ has finitely many values on $R$. 
\end{proof} 

\begin{lemma}\label{nonpos} If $\alpha,\beta\in R_+$ are distinct simple roots then $(\alpha,\beta)\le 0$. 
\end{lemma} 

\begin{proof} Assume $(\alpha,\beta)>0$. Then $(-\alpha,\beta)<0$ so by Lemma \ref{neg} $\gamma:=\beta-\alpha$ is a root. 
 If $\gamma$ is positive then $\beta=\alpha+\gamma$ is not simple. If $\gamma$ is negative then 
 $-\gamma$ is positive so $\alpha=\beta+(-\gamma)$ is not simple. 
\end{proof} 

\begin{theorem}\label{simbasis} The set $\Pi\subset R_+$ of simple roots is a basis of $E$. 
\end{theorem} 

\begin{proof} We will use the following linear algebra lemma: 

\begin{lemma}\label{linnalg} Let $v_i$ be vectors in a Euclidean space $E$ such that $(v_i,v_j)\le 0$ 
when $i\ne j$ and $t(v_i)>0$ for some $t\in E^*$. Then $v_i$ are linearly independent. 
\end{lemma} 

\begin{proof} Suppose we have a nontrivial relation 
$$
\sum_{i\in I}c_iv_i=\sum_{i\in J}c_iv_i
$$
where $I,J$ are disjoint and $c_i>0$ (clearly, every nontrivial relation can be written in this form). 
Evaluating $t$ on this relation, we deduce that both sides are nonzero. Now let us compute the square of the left hand side: 
$$
0< |\sum_{i\in I}c_iv_i|^2=(\sum_{i\in I}c_iv_i,\sum_{j\in J}c_jv_j)\le 0.
$$
This is a contradiction. 
\end{proof} 

Now the result follows from Lemma \ref{nonpos} and Lemma \ref{linnalg}. 
\end{proof}  

Thus the set $\Pi$ of simple roots has $r$ elements: $\Pi=(\alpha_1,...,\alpha_r)$. 

\begin{example}\label{simroocla} Let us describe simple roots for classical root systems. Suppose the polarization is given by $t=(t_1,...,t_n)$ with decreasing coordinates. Then: 

1. For type $A_{n-1}$, i.e., $\g=\mathfrak{sl}_n$, the simple roots are $\alpha_i:={\bold e}_i-{\bold e}_{i+1}$, $1\le i\le n-1$. 

2. For type $C_n$, i.e., $\g=\mathfrak{sp}_{2n}$, the simple roots are 
$$
\alpha_1=\bold e_1-\bold e_2,...,\ \alpha_{n-1}=\bold e_{n-1}-\bold e_n,\ \alpha_n=2\bold e_n.
$$ 

3. For type $B_n$, i.e., $\g=\mathfrak{so}_{2n+1}$, we have the same story as for $C_n$ except $\alpha_n=\bold e_n$ rather than $2\bold e_n$. 
Thus the simple roots are 
$$
\alpha_1=\bold e_1-\bold e_2,...,\ \alpha_{n-1}=\bold e_{n-1}-\bold e_n,\ \alpha_n=\bold e_n.
$$ 

4. For type $D_n$, i.e., $\g=\mathfrak{so}_{2n}$, the simple roots are
$$
\alpha_1=\bold e_1-\bold e_2,...,\ \alpha_{n-2}=\bold e_{n-2}-\bold e_{n-1},\ \alpha_{n-1}=\bold e_{n-1}-\bold e_n,\ \alpha_n=\bold e_{n-1}+\bold e_n.
$$ 
\end{example}

We thus obtain 

\begin{corollary} Any root $\alpha\in R$ can be uniquely written as 
$\alpha=\sum_{i=1}^r n_i\alpha_i$, where $n_i\in \Bbb Z$. If $\alpha$ is positive then $n_i\ge 0$ for all $i$ and if $\alpha$ is negative then $n_i\le 0$ for all $i$. 
\end{corollary} 

For a positive root $\alpha$, its {\bf height} $h(\alpha)$ is the number $\sum n_i$. 
So simple roots are the roots of height $1$, and the height of ${\bold e}_i-{\bold e}_j$ in $R=A_{n-1}$ 
is $j-i$.  

\subsection{Dual root system} 

For a root system $R$, the set $R^\vee\subset E^*$ of $\alpha^\vee$ for all $\alpha\in R$
is also a root system, such that $(R^\vee)^\vee=R$. It is called the {\bf dual root system} to $R$. 
For example, $B_n$ is dual to $C_n$, while $A_{n-1}$, $D_n$ and $G_2$ are self-dual. 

Moreover, it is easy to see that any polarization of $R$ gives rise to a polarization 
of $R^\vee$ (using the image $t^\vee$ of $t$ under the isomorphism $E\to E^*$ 
induced by the inner product), and the corresponding system $\Pi^\vee$ of simple roots 
consists of $\alpha_i^\vee$ for $\alpha_i\in \Pi$. 

\subsection{Root and weight lattices} 

Recall that a {\bf lattice} in a real vector space $E$ is a subgroup $Q\subset E$ generated by a basis of $E$. Of course, every lattice is conjugate to $\Bbb Z^n\subset \Bbb R^n$ by an element of $GL_n(\Bbb R)$. Also recall that for a lattice 
$Q\subset E$ the {\bf dual lattice} $Q^*\subset E^*$ is the set of $f\in E^*$ such that $f(v)\in \Bbb Z$ 
for all $v\in Q$. If $Q$ is generated by a basis ${\bold e}_i$ of $E$ then $Q^*$ is generated by 
the dual basis ${\bold e}_i^*$. 

In particular, for a root system $R$ we can define the {\bf root lattice} $Q\subset E$, which is generated by the simple roots $\alpha_i$ with respect to some polarization of $R$. Since $Q$ is also generated by all roots in $R$, it is independent of  the choice of the polarization. Similarly, we can define the {\bf coroot lattice} $Q^\vee\subset E^*$ generated by $\alpha^\vee,\alpha\in R$, which is just the root lattice of $R^\vee$. 

Also we define the {\bf weight lattice} $P\subset E$ to be the dual lattice to $Q^\vee$: $P=(Q^\vee)^*$, and the {\bf coweight lattice} $P^\vee\subset E^*$ to be the dual lattice to $Q$: $P^\vee=Q^*$, so $P^\vee$ is the weight lattice of $R^\vee$. 
Thus 
$$
P=\lbrace \lambda\in E: (\lambda,\alpha^\vee)\in \Bbb Z\ \forall \alpha\in R\rbrace,\ P^\vee=\lbrace \lambda\in E^*: (\lambda,\alpha)\in \Bbb Z\ \forall \alpha\in R\rbrace. 
$$

Since for $\alpha,\beta\in R$ we have $(\alpha^\vee,\beta)=n_{\alpha\beta}\in \Bbb Z$, we have 
$Q\subset P$, $Q^\vee\subset P^\vee$. 

Given a system of simple roots $\Pi=\lbrace \alpha_1,...,\alpha_r\rbrace$, 
we define {\bf fundamental coweights} $\omega_i^\vee$ to be the dual basis 
to $\alpha_i$ and {\bf fundamental weights} $\omega_i$ to be the dual basis 
to $\alpha_i^\vee$: $(\omega_i,\alpha_j^\vee)=(\omega_i^\vee,\alpha_j)=\delta_{ij}$. 
Thus $P$ is generated by $\omega_i$ and $P^\vee$ by $\omega_i^\vee$. 

\begin{example} Let $R$ be of type $A_1$. Then $(\alpha,\alpha^\vee)=2$ 
for the unique positive root $\alpha$, so $\omega=\frac{1}{2}\alpha$, thus $P/Q=\Bbb Z/2$. 
More generally, if $R$ is of type $A_{n-1}$ and we identify $Q\cong Q^\vee, P\cong P^\vee$, then 
$P$ becomes the set of $\lambda=(\lambda_1,...,\lambda_n)\in \Bbb R^n$ 
such that $\sum_i \lambda_i=0$ and $\lambda_i-\lambda_j\in\Bbb Z$.  
So we have a homomorphism $\phi: P\to \Bbb R/\Bbb Z$ given by $\phi(\lambda)=\lambda_i$ mod $\Bbb Z$ (for any $i$). Since $\sum_i \lambda_i=0$, we have $\phi: P\to \Bbb Z/n$, and ${\rm Ker}\phi=Q$ (integer vectors with sum zero). Also it is easy to see that $\phi$ is surjective 
(we may take $\lambda_i=\frac{k}{n}$ for $i\ne n$ and $\lambda_n=\frac{k}{n}-k$, then $\phi(\lambda)=\frac{k}{n}$). Thus $P/Q\cong \Bbb Z/n$. 
\end{example} 

\section{\bf Properties of the Weyl group} 

\subsection{Weyl chambers} 

Suppose we have two polarizations of a root system $R$ defined by $t,t'\in E^*$, and $\Pi,\Pi'$ are the corresponding systems of simple roots. Are $\Pi,\Pi'$ equivalent in a suitable sense? The answer turns out to be yes. To show this, we will need the notion of a Weyl chamber. 

Note that the polarization defined by $t$ depends only on the signs of $t(\alpha)$, so does not change when $t$ is continuously deformed without crossing the hyperplanes $t(\alpha)=0$. This motivates the following definition: 

\begin{definition} A {\bf Weyl chamber} is a connected component of the complement of the root hyperplanes $L_\alpha$ given by the equations $(\alpha,x)=0$ in $E$ ($\alpha\in R$). 
\end{definition} 

Thus a Weyl chamber is defined by a system of strict homogeneous linear inequalities $\pm (\alpha,x)>0$, $\alpha\in R$. 
More precisely, the set of solutions of such a system is either empty or a Weyl chamber. 

Thus the polarization defined by $t$ depends only on the Weyl chamber containing $t$. 

The following lemma is geometrically obvious. 

\begin{lemma}  (i) The closure $\overline{C}$ of a Weyl chamber $C$ is a convex cone. 

(ii) The boundary of $\overline{C}$ is a union of codimension $1$ faces $F_i$ which are 
convex cones inside one of the root hyperplanes defined inside it 
by a system of non-strict homogeneous linear inequalities.   
\end{lemma} 

The root hyperplanes containing the faces $F_i$ are called the {\bf walls} of $C$. 

We have seen above that every Weyl chamber defines a polarization of $R$. 
Conversely, every polarization defines the corresponding {\bf positive Weyl chamber}
$C_+$ defined by the conditions $(\alpha,x)>0$ for $\alpha\in R_+$ (this set 
is nonempty since it contains $t$, hence is a Weyl chamber). Thus $C_+$ 
is the set of vectors of the form $\sum_{i=1}^r c_i\omega_i$ with $c_i>0$. So $C_+$ has $r$ faces 
$L_{\alpha_1}\cap \overline C_+,...,L_{\alpha_r}\cap \overline C_+$.

\begin{lemma}\label{mutinv} These assignments are mutually inverse bijections between polarizations of $R$ and Weyl chambers.
\end{lemma} 

\begin{exercise} Prove Lemma \ref{mutinv}.
\end{exercise} 

Since the Weyl group $W$ permutes the roots, it acts on the set of Weyl chambers. 

\begin{theorem}\label{trans} $W$ acts transitively on the set of Weyl chambers. 
\end{theorem} 

\begin{proof} Let us say that Weyl chambers $C,C'$ are {\bf adjacent} if they share a common face 
$F\subset L_\alpha$. In this case it is easy to see that $s_\alpha(C)=C'$. 
Now given any Weyl chambers $C,C'$, pick generic $t\in C,t'\in C'$ and connect them with a straight segment. This will define a sequence of Weyl chambers visited by this segment: $C_0=C,C_1,...,C_m=C'$, and $C_i,C_{i+1}$ are adjacent for each $i$. So $C_i,C_{i+1}$ lie in the same $W$-orbit. Hence so do $C,C'$.  
\end{proof} 

\begin{corollary} Every Weyl chamber has $r$ walls. 
\end{corollary} 

\begin{proof} This follows since it is true for the positive Weyl chamber and by Theorem \ref{trans} 
the Weyl group acts transitively on the Weyl chambers. 
\end{proof} 

\begin{corollary} Any two polarizations of $R$ are related by the action of an element $w\in W$. 
Thus if $\Pi,\Pi'$ are systems of simple roots corresponding to two polarizations then there is 
$w\in W$ such that $w(\Pi)=\Pi'$. 
\end{corollary} 

\subsection{Simple reflections} 

Given a polarization of $R$ and the corresponding system of simple roots $\Pi=\lbrace \alpha_1,...,\alpha_r\rbrace$, the {\bf simple reflections} are the reflections $s_{\alpha_i}$, denoted by $s_i$. 

\begin{lemma} For every Weyl chamber $C$ there exist $i_1,...,i_m$ such that 
$C=s_{i_1}...s_{i_m}(C_+)$. 
\end{lemma} 

\begin{proof} Pick $t\in C,t_+\in C_+$ generically and connect them with a straight segment as before. 
Let $m$ be the number of chamber walls crossed by this segment. The proof is by induction in $m$ (with obvious base). Let $C'$ be the chamber entered by our segment from $C$ and $L_\alpha$ the wall 
separating $C,C'$, so that $C=s_\alpha(C')$. By the induction assumption 
$C'=u(C_+)$, where $u=s_{i_1}...s_{i_{m-1}}$. So $L_\alpha=u(L_{\alpha_j})$ for some $j$. 
Thus $s_\alpha=u s_j u^{-1}$. Hence $C=s_\alpha(C')=s_\alpha u(C_+)=u s_j(C_+)$, so we get the result with $i_m=j$. 
\end{proof} 

\begin{corollary} (i) The simple reflections $s_i$ generate $W$; 

(ii) $W(\Pi)=R$.
\end{corollary}  

\begin{proof} (i) This follows since for any root $\alpha$, the hyperplane $L_\alpha$ is a wall of some Weyl chamber, so $s_\alpha$ is a product of $s_i$. 

(ii) Follows from (i).  
\end{proof} 

Thus $R$ can be reconstructed from $\Pi$ as $W(\Pi)$, where $W$ is the subgroup of $O(E)$ generated by $s_i$. 

\begin{example} For root system $A_{n-1}$ part (i) says that any element of $S_n$ is a product of transpositions of neighbors. 
\end{example} 

\subsection{Length of an element of the Weyl group} 

Let us say that a root hyperplane $L_\alpha$ {\bf separates} two Weyl chambers $C,C'$ if they lie on different sides of $L_\alpha$. 

\begin{definition} The {\bf length} $\ell(w)$ of $w\in W$ is the number of root hyperplanes separating $C_+$ and $w(C_+)$. 
\end{definition} 

We have $t\in C_+, w(t)\in w(C_+)$, so $\ell(w)$ is the number of roots 
$\alpha$ such that $(t,\alpha)>0$ but $(w(t),\alpha)=(t,w^{-1}\alpha)<0$. 
Note that if $\alpha$ is a root satisfying this condition then $\beta=-w^{-1}\alpha$ 
satisfies the conditions $(t,\beta)>0$, $(t,w\beta)<0$. Thus $\ell(w)=\ell(w^{-1})$ and $\ell(w)$ is the number of positive roots which are mapped by $w$ to negative roots. 
Note also that the notion of length depends on the polarization of $R$ (as 
it refers to the positive chamber $C_+$ defined using the polarization). 

\begin{example} Let $s_i$ be a simple reflection. Then $s_i(C_+)$ is adjacent to $C_+$, with the only separating hyperplane being $L_{\alpha_i}$. Thus $\ell(s_i)=1$. 
It follows that the only positive root mapped by $s_i$ to a negative root
is $\alpha_i$ (namely, $s_i(\alpha_i)=-\alpha_i$), and thus $s_i$ permutes 
$R_+\setminus \lbrace \alpha_i\rbrace$. 
\end{example} 

\begin{proposition} Let $\rho=\frac{1}{2}\sum_{\alpha\in R_+}\alpha$. Then 
$(\rho,\alpha_i^\vee)=1$ for all $i$. Thus $\rho=\sum_{i=1}^r \omega_i$. 
\end{proposition} 

\begin{proof} We have $\rho=\frac{1}{2}\alpha_i+\frac{1}{2}\sum_{\alpha\in R_+,\alpha\ne \alpha_i}\alpha$. Since $s_i$ permutes $R_+\setminus \lbrace \alpha_i\rbrace$, we get 
$s_i\rho=\rho-\alpha_i$. But for any $\lambda$, $s_i\lambda=\lambda-(\lambda,\alpha_i^\vee)\alpha_i$. This implies the statement. 
\end{proof} 

The weight $\rho$ plays an important role in representation theory of semisimple Lie algebras. For instance, it occurs in the Weyl character formula for these representations which we will soon derive. 

\begin{theorem} Let $w=s_{i_1}...s_{i_l}$ be a representation of $w\in W$ as a product of simple reflections that has minimal possible length. Then $l=\ell(w)$. 
\end{theorem} 

\begin{proof} First, if $l>0$, then $w$ sends $\alpha_{i_l}$ to a negative root. Otherwise, after $s_{i_l}\alpha_{i_l}=-\alpha_{i_l}$, a negative-to-positive step gives $v\alpha_{i_l}=\alpha_{i_j}$, where $v=s_{i_{j+1}}\cdots s_{i_{l-1}}$, since $s_{i_j}$ changes the sign only of $\pm\alpha_{i_j}$. Thus $s_{i_j}v s_{i_l}=v$, deleting two letters, a contradiction. Hence no nontrivial $w$ fixes $C_+$, and $W$ acts freely on chambers.

As before, define a chain of Weyl chambers 
$C_k=s_{i_1}...s_{i_k}(C_+)$, so that $C_0=C_+$ and $C_l=w(C_+)$. 
We have seen that $C_k$ and $C_{k-1}$ are adjacent. So 
there is a zigzag path from $C_+$ to $w(C_+)$ that intersects at most 
$l$ root hyperplanes (namely, the segment from $C_{k-1}$ to $C_k$ intersects only one hyperplane). Thus $\ell(w)\le l$. On the other hand, pick generic points 
in $C_+$ and $w(C_+)$ and connect them with a straight segment. This segment intersects every separating root hyperplane exactly once and does not intersect other root hyperplanes, so produces an expression of $w$ as a product of $\ell(w)$ simple reflections. This implies the statement.
\end{proof} 

An expression $w=s_{i_1}...s_{i_l}$ is called {\bf reduced} if $l=\ell(w)$. 

\begin{proposition} The Weyl group $W$ acts simply transitively on Weyl chambers. 
\end{proposition} 

\begin{proof} By Theorem \ref{trans} the action is transitive, so we just have to show that if $w(C_+)=C_+$ then $w=1$. But in this case $\ell(w)=0$, so $w$ has to be a product of zero simple reflections, i.e., indeed $w=1$. 
\end{proof} 

Thus we see that $\overline C_+$ is a {\it fundamental domain} of the action of $W$ on $E$. 

Moreover, we have 

\begin{proposition} $E/W=\overline C_+$, i.e., every $W$-orbit on $E$ has a unique representative in $\overline C_+$. 
\end{proposition} 

\begin{proof} Suppose $\lambda,\mu\in \overline C_+$ and $\lambda=w\mu$, where $w\in W$ is shortest possible. Assume the contrary, that $w\ne 1$. Pick a reduced decomposition $w=s_{i_l}...s_{i_1}$. Let $\gamma$ be the positive root 
which is mapped to a negative root by $w$ but not by $s_{i_l}w$, i.e., 
$\gamma=s_{i_1}...s_{i_{l-1}}\alpha_{i_l}$. Then $0\le (\mu,\gamma)=(\lambda,w\gamma)\le 0$. 
So $(\mu,\gamma)=0$. Thus 
$$
\lambda=w\mu=s_{i_l}...s_{i_1}\mu=s_{i_{l-1}}...s_{i_1}s_\gamma\mu=
s_{i_{l-1}}...s_{i_1}\mu
$$
which is a contradiction since $w$ was the shortest possible. 
\end{proof} 

\begin{corollary}\label{w0} Let $C_-=-C_+$ be the {\bf negative Weyl chamber}. Then there exists a unique $w_0\in W$ such that $w_0(C_+)=C_-$. We have $\ell(w_0)=|R_+|$ 
and for any $w\ne w_0$, $\ell(w)<\ell(w_0)$. Also $w_0^2=1$.  
\end{corollary} 

\begin{exercise} Prove Corollary \ref{w0}. 
\end{exercise} 

The element $w_0$ is therefore called the {\bf longest element} of $W$. 

\begin{example} For the root system $A_{n-1}$ the element $w_0$ is the order reversing 
involution: $w_0(1,2,...,n)=(n,...,2,1)$.  
\end{example} 

\section{\bf Dynkin diagrams} 

\subsection{Cartan matrices and Dynkin diagrams} 

Our goal now is to classify reduced root systems, which is a key step in the classification of semisimple Lie algebras. We have shown that classifying root systems is equivalent to classifying sets $\Pi$ of simple roots. So we need to classify such sets $\Pi$. 
Before doing so, note that we have a nice notion of {\bf direct product} of root systems.

Namely, let $R_1\subset E_1$ and $R_2\subset E_2$ be two root systems. 
Let $E=E_1\oplus E_2$ (orthogonal decomposition) and $R=R_1\sqcup R_2$
(with $R_1\perp R_2$). If $t_1\in E_1,t_2\in E_2$ define polarizations of $R_1,R_2$ with systems of simple roots $\Pi_1,\Pi_2$ then $t=t_1+t_2$ defines a polarization of $R$ with $\Pi=\Pi_1\sqcup \Pi_2$ (with $\Pi_1\perp \Pi_2$ and $\Pi_i=\Pi\cap R_i$). 

\begin{definition} A root system $R$ is {\bf irreducible} if it cannot be written (nontrivially) in this way.  
\end{definition} 
 
\begin{lemma} If $R$ is a root system with system of simple roots $\Pi=\Pi_1\sqcup \Pi_2$ with $\Pi_1\perp \Pi_2$ then $R=R_1\sqcup R_2$ where $R_i$ is the root system generated by $\Pi_i$. 
\end{lemma} 

\begin{proof} If $\alpha\in \Pi_1,\beta\in \Pi_2$ then $s_\alpha(\beta)=\beta$, $s_\beta(\alpha)=\alpha$ and $s_\alpha$ and $s_\beta$ commute. So if $W_i$ is the group generated by $s_\alpha,\alpha\in \Pi_i$ then $W=W_1\times W_2$, with $W_1$ acting trivially on $\Pi_2$ and $W_2$ on $\Pi_1$. Thus 
$$
R=W(\Pi)=W(\Pi_1\sqcup\Pi_2)=W_1(\Pi_1)\sqcup W_2(\Pi_2)=R_1\sqcup R_2. 
$$
\end{proof} 

\begin{proposition} Any root system is uniquely a union of irreducible ones. 
\end{proposition} 

\begin{proof} The decomposition is given by the maximal decomposition of $\Pi$ 
into mutually orthogonal systems of simple roots. 
\end{proof} 

Thus it suffices to classify irreducible root systems. 

As noted above, a root system is determined by pairwise inner products of positive roots. However, it is more convenient to encode them by the {\bf Cartan matrix} $A$ defined by   
$$
a_{ij}=n_{\alpha_i\alpha_j}=(\alpha_i^\vee,\alpha_j). 
$$
The following properties of the Cartan matrix follow immediately from Lemma \ref{nonpos}, Theorem \ref{rank2rs} and Theorem \ref{simbasis}: 

\begin{proposition}\label{carmat} (i) $a_{ii}=2$; 

(ii) $a_{ij}$ is a nonpositive integer for $i\ne j$; 

(iii) for any $i\ne j$, $a_{ij}a_{ji}=4\cos^2\phi\in \lbrace 0,1,2,3\rbrace$,
where $\phi$ is the angle between $\alpha_i$ and $\alpha_j$; 

(iv) Let $d_i=|\alpha_i|^2$. Then the matrix $d_ia_{ij}$ is symmetric and 
positive definite. 
\end{proposition} 

We will see later that conversely, any such matrix defines a root system. 

\begin{example} 1. Type $A_{n-1}$: $a_{ii}=2,a_{i,i+1}=a_{i+1,i}=-1$, $a_{ij}=0$ otherwise. 

2. Type $B_n$: $a_{ii}=2$, $a_{i,i+1}=a_{i+1,i}=-1$ except that $a_{n,n-1}=-2$. 

3. Type $C_n$: transposed to $B_n$. 

4. Type $D_n$: same as $B_n$ but $a_{n-1,n-2}=a_{n,n-2}=a_{n-2,n}=a_{n-2,n-1}=-1$, 
$a_{n,n-1}=a_{n-1,n}=0$. 

5. Type $G_2$: $A=\begin{pmatrix} 2 & -1\\ -3 & 2\end{pmatrix}$. 
\end{example} 

It is convenient to encode such matrices by {\bf Dynkin diagrams:}  

$\bullet$ Indices $i$ are vertices;  

$\bullet$ Vertices $i$ and $j$ are connected by $a_{ij}a_{ji}$ lines; 

$\bullet$ If $a_{ij}\ne a_{ji}$, i.e., $|\alpha_i|^2\ne |\alpha_j|^2$, then the arrow on the lines goes from long root to short root (``less than" sign). 

It is clear that such a diagram completely determines the Cartan matrix (if we fix the labeling of vertices), and vice versa. Also it is clear that the root system is irreducible if and only if 
its Dynkin diagram is connected. 

\begin{proposition} The Cartan matrix determines the root system uniquely.  
\end{proposition}

\begin{proof} We may assume the Dynkin diagram is connected. The Cartan matrix determines, for any pair of simple roots, the angle between them (which is right or obtuse) and the ratio of their lengths if they are not orthogonal. By the classification of rank 2 root systems, this determines the inner product on simple roots, up to scaling, which implies the statement. 
\end{proof} 

\subsection{Classification of Dynkin diagrams} 

The following theorem gives a complete classification of irreducible root systems. 

\begin{theorem}\label{classdyn} (i) Connected Dynkin diagrams are classified by the list given in the picture below, i.e., they are $A_n,B_n,C_n,D_n,G_2$ which we have already met, along with four more: $F_4,E_6,E_7,E_8$.  

(ii) Every matrix satisfying the conditions of Proposition \ref{carmat} is a Cartan matrix of some root system. 
\end{theorem} 

\includegraphics[scale=0.1]{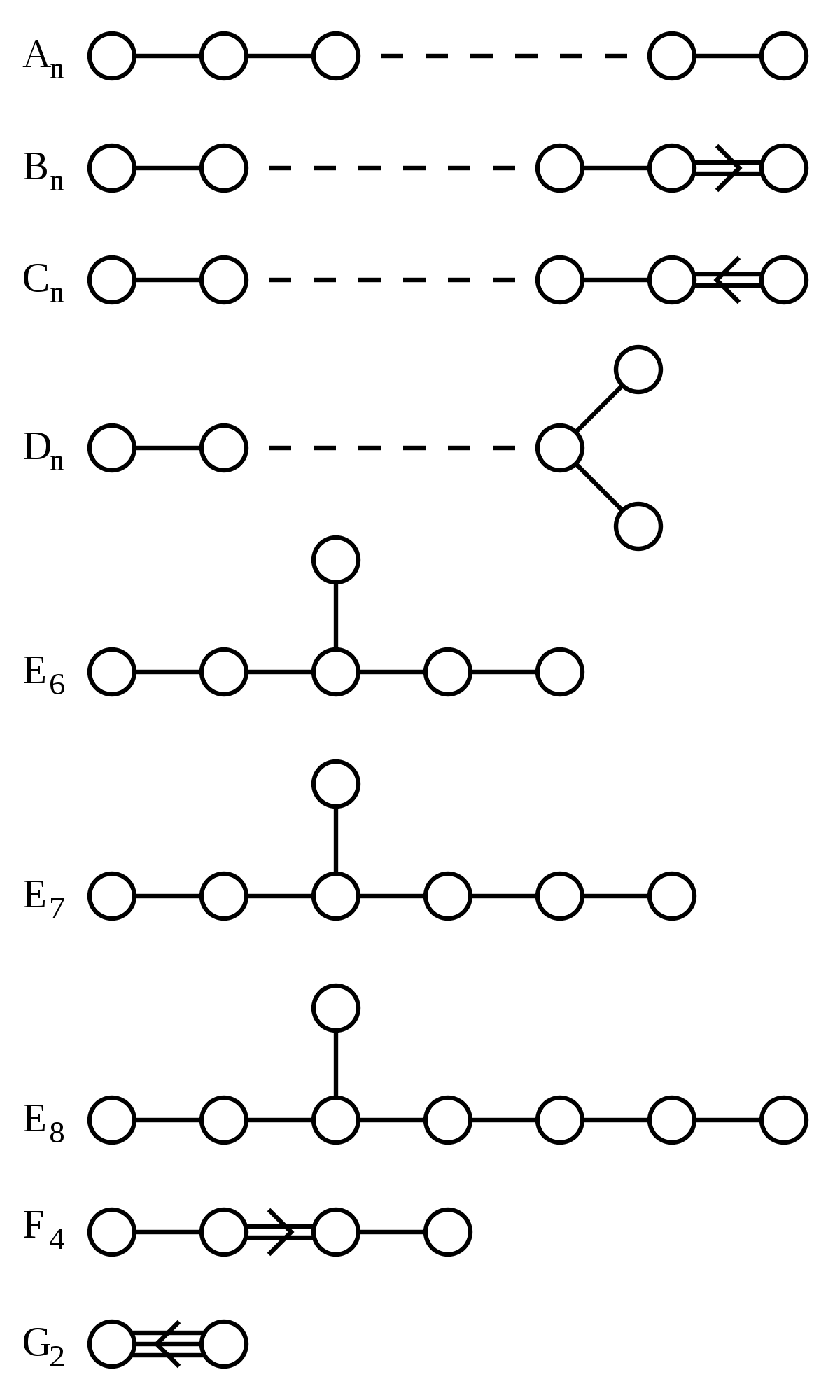}

The proof of Theorem \ref{classdyn} is rather long but direct. It consists of several steps. The first step is construction of the remaining root systems 
$F_4,E_6,E_7,E_8$. 

\subsection{The root system $F_4$}\label{F4r}

\begin{definition} The root system $F_4$ is the union of the root system $B_4\subset \Bbb R^4$ with the vectors 
$$
\left(\pm \tfrac{1}{2},\pm \tfrac{1}{2},\pm \tfrac{1}{2},\pm \tfrac{1}{2}\right)=\sum_{i=1}^4 (\pm \tfrac{1}{2}{\bold e}_i),
$$
for all choices of signs. 
\end{definition} 

Thus besides the roots of $B_4$, which are $\pm {\bold e}_i\pm {\bold e}_j$ (24 of them, squared length $2$) and $\pm {\bold e}_i$ (8 of them, squared length $1$), we have the 16 new roots $\sum_{i=1}^4 (\pm \frac{1}{2}{\bold e}_i)$ (squared length $1$); this gives a total of 48. 

\begin{exercise} Check that this is an irreducible root system.
\end{exercise} 

To give a polarization of the $F_4$ root system, pick $t=(t_1,t_2,t_3,t_4)$ 
with $t_1\gg t_2\gg t_3\gg t_4$. 

\begin{exercise} Check that for this polarization, the simple positive roots are, $\alpha_1=\frac{1}{2}({\bold e}_1-{\bold e}_2-{\bold e}_3-{\bold e}_4)$, $\alpha_2={\bold e}_4$,
$\alpha_3={\bold e}_3-{\bold e}_4$, $\alpha_4={\bold e}_2-{\bold e}_3$. Thus $\alpha_1^\vee={\bold e}_1-{\bold e}_2-{\bold e}_3-{\bold e}_4$, $\alpha_2^\vee=2{\bold e}_4$,
$\alpha_3^\vee={\bold e}_3-{\bold e}_4$, $\alpha_4^\vee={\bold e}_2-{\bold e}_3$.
So the Cartan matrix has the form 
$$
A=\begin{pmatrix} 2& -1 &0&0\\-1&2&-2&0\\0&-1&2&-1\\ 0&0 &-1&2
\end{pmatrix} 
$$
which gives the Dynkin diagram of $F_4$. 
\end{exercise}  

\subsection{The root system $E_8$} \label{E8r}

\begin{definition} The root system $E_8$ is the union of the root system $D_8\subset \Bbb R^8$ with the vectors $\sum_{i=1}^8 (\pm \frac{1}{2}{\bold e}_i)$,
for all choices of signs with even number of minuses. 
\end{definition} 

Thus besides the roots of $D_8$, $\pm {\bold e}_i\pm {\bold e}_j$ (112 of them), we have 
128 new roots $\sum_{i=1}^8 (\pm \frac{1}{2}{\bold e}_i)$. So in total we have $240$ roots. All roots have squared length $2$. 

\begin{exercise} Show that it is an irreducible root system. 
\end{exercise} 

To give a polarization of the $E_8$ root system, pick $t$ so that $t_i\gg t_{i+1}$. 

\begin{exercise} Check that for this polarization, the simple positive roots are, $\alpha_1=\frac{1}{2}({\bold e}_1+{\bold e}_8-\sum_{i=2}^7{\bold e}_i)$, $\alpha_2={\bold e}_7+{\bold e}_8$ and $\alpha_i={\bold e}_{10-i}-{\bold e}_{11-i}$ for $3\le i\le 8$. Thus the roots $\alpha_2,...,\alpha_8$ generate the root system $D_7$, while 
$a_{13}=-1$ and $a_{1i}=0$ for all $i\ne 1,3$. In other words, the Cartan matrix 
has the form 
$$
A=\begin{pmatrix} 2& 0 & -1& 0& 0& 0& 0& 0\\ 0& 2 & 0& -1& 0& 0& 0& 0\\ -1& 0 & 2& -1& 0& 0& 0& 0\\ 0& -1 & -1& 2& -1& 0& 0& 0\\0& 0& 0 & -1& 2& -1& 0& 0\\
0& 0& 0& 0 & -1& 2& -1& 0&\\ 0&0& 0& 0& 0 & -1& 2& -1\\ 0&0&0& 0& 0& 0 & -1& 2\end{pmatrix}
$$
This recovers the Dynkin diagram $E_8$. 
\end{exercise}  

\subsection{The root system $E_7$}\label{E7r}

\begin{definition} The root system $E_7$ is the subsystem of $E_8$
generated by $\alpha_1,...,\alpha_7$. 
\end{definition}  

Note that these roots (unlike $\alpha_8={\bold e}_2-{\bold e}_3$) satisfy the 
equation $x_1+x_2=0$. Thus $E_7$ is the intersection of 
$E_8$ with this subspace. So it includes the roots 
$\pm {\bold e}_i\pm {\bold e}_j$ with $3\le i,j\le 8$ distinct (60 roots), $\pm ({\bold e}_1-{\bold e}_2)$
(2 roots) and $\sum_{i=1}^8 (\pm \frac{1}{2}{\bold e}_i)$ with even number of minuses and the opposite signs for ${\bold e}_1$ and ${\bold e}_2$ (64 roots). 
Altogether we get 126 roots. The Cartan matrix is the upper left corner 7 by 7 submatrix of the Cartan matrix of $E_8$, so it is
$$
A=\begin{pmatrix} 2& 0 & -1& 0& 0& 0& 0\\ 0& 2 & 0& -1& 0& 0& 0\\ -1& 0 & 2& -1& 0& 0& 0\\ 0& -1 & -1& 2& -1& 0& 0\\0& 0& 0 & -1& 2& -1& 0\\
0& 0& 0& 0 & -1& 2& -1\\ 0&0& 0& 0& 0 & -1& 2\end{pmatrix}
$$

\subsection{The root system $E_6$}\label{E6r}

\begin{definition} The root system $E_6$ is the subsystem of $E_8$ and $E_7$ generated by $\alpha_1,...,\alpha_6$. 
\end{definition}  

Note that these roots (unlike $\alpha_8={\bold e}_2-{\bold e}_3$ and $\alpha_7={\bold e}_3-{\bold e}_4$) satisfy the 
equations $x_1+x_2=0, x_2-x_3=0$. Thus $E_6$ is the intersection of 
$E_8$ with this subspace. So it includes the roots 
$\pm {\bold e}_i\pm {\bold e}_j$ with $4\le i,j\le 8$ distinct (40 roots), and $\sum_{i=1}^8 (\pm \frac{1}{2}{\bold e}_i)$ with even number of minuses and the opposite signs for ${\bold e}_1$ and ${\bold e}_2$ and for ${\bold e}_1$ and ${\bold e}_3$ (32 roots). 
Altogether we get 72 roots. 
The Cartan matrix is the upper left corner 6 by 6 submatrix of the Cartan matrix of $E_8$, so it is
$$
A=\begin{pmatrix} 2& 0 & -1& 0& 0& 0\\ 0& 2 & 0& -1& 0& 0\\ -1& 0 & 2& -1& 0& 0\\ 0& -1 & -1& 2& -1& 0\\0& 0& 0 & -1& 2& -1\\
0& 0& 0& 0 & -1& 2\end{pmatrix}
$$
This recovers the Dynkin diagram $E_6$.

\subsection{The elements $\rho$ and $\rho^\vee$}\label{rhorhov} 

Recall that the elements $\rho\in \h^*$ and $\rho^\vee\in \h$
for a simple Lie algebra $\g$ are defined by the conditions 
$(\rho,\alpha_i^\vee)=(\rho^\vee,\alpha_i)=1$ for all $i$
(note that $\rho$ is not a root in general, and $\rho^\vee$ 
is not an instance of the assignment $\alpha\mapsto \alpha^\vee$ for roots $\alpha$). 
So for classical Lie algebras they can be computed from 
Example \ref{simroocla}. Namely, we get 
$$
\rho_{A_{n-1}}=\rho_{A_{n-1}}^\vee=(\tfrac{n-1}{2},\tfrac{n-3}{2},...,-\tfrac{n-1}{2}),
$$
$$
\rho_{B_n}=\rho_{C_n}^\vee=(\tfrac{2n-1}{2},...,\tfrac{3}{2},\tfrac{1}{2}),
$$
$$
\rho_{C_n}=\rho_{B_n}^\vee=(n,n-1,...,1),
$$
$$
\rho_{D_n}=\rho_{D_n}^\vee=(n-1,n-2,...,0).
$$

\begin{exercise} Show that the elements $\rho$ and $\rho^\vee$ 
for exceptional root systems (in the above realizations) are as follows: 
$$
\rho_{G_2}=3\alpha+5\beta,\ \rho^\vee_{G_2}=5\alpha^\vee+3\beta^\vee,
$$
$$
\rho_{F_4}=(\tfrac{11}{2},\tfrac{5}{2},\tfrac{3}{2},\tfrac{1}{2}),\ \rho_{F_4}^\vee=(8,3,2,1),
$$
$$
\rho_{E_8}=\rho_{E_8}^\vee=(23,6,5,4,3,2,1,0),
$$
$$
\rho_{E_7}=\rho_{E_7}^\vee=(\tfrac{17}{2},-\tfrac{17}{2},5,4,3,2,1,0),
$$
$$
\rho_{E_6}=\rho_{E_6}^\vee=(4,-4,-4,4,3,2,1,0).
$$
(recall that we realized $E_6,E_7,E_8$ inside $\Bbb R^8$). 
\end{exercise} 

\subsection{Proof of Theorem \ref{classdyn}}
Now that we have shown that there exist root systems attached to all Cartan matrices, it remains to classify Cartan matrices (or Dynkin diagrams), i.e. show that there are no others than those we have considered. For this purpose we consider Dynkin diagrams as graphs with certain kind of special edges (with one, two or three lines and a possible orientation). 
Note first that any subgraph of a Dynkin diagram must itself be a Dynkin diagram, since a principal submatrix of a positive definite symmetric matrix is itself positive definite. On the other hand, consider {\bf untwisted and twisted affine Dynkin diagrams} shown below. These are not Dynkin diagrams since the corresponding matrix $A$ is degenerate, hence not positive definite.

\begin{center}
\begingroup
\newcommand{\LGLAdot}[2]{%
  \draw[fill=white] (#1,#2) circle[radius=2.1pt];}
\newcommand{\LGLAlabel}[3]{%
  \node[font=\scriptsize,below=4pt] at (#1,#2) {$#3$};}
\newcommand{\LGLAdouble}[3]{%
  \draw (#1,0.055)--(#2,0.055) (#1,-0.055)--(#2,-0.055);
  \draw ({(#1+#2)/2-#3*0.13},0.16)--({(#1+#2)/2+#3*0.10},0)
    --({(#1+#2)/2-#3*0.13},-0.16);}
\newcommand{\LGLAtriple}[3]{%
  \draw (#1,0.095)--(#2,0.095) (#1,0)--(#2,0)
    (#1,-0.095)--(#2,-0.095);
  \draw ({(#1+#2)/2-#3*0.13},0.19)--({(#1+#2)/2+#3*0.10},0)
    --({(#1+#2)/2-#3*0.13},-0.19);}
\newcommand{\LGLAgap}[2]{%
  \draw (#1,0)--({(#1+#2)/2-0.34},0)
    ({(#1+#2)/2+0.34},0)--(#2,0);
  \node at ({(#1+#2)/2},0) {$\cdots$};}
\resizebox{0.98\textwidth}{!}{%
\begin{tikzpicture}[x=0.71cm,y=0.71cm,line width=0.55pt,
  every node/.style={inner sep=0pt}]

\begin{scope}
  \node[anchor=west] at (-0.2,1.35) {$\widetilde A_n\quad(n\ge2)$};
  \draw (0,0)--(2,0.8)--(4,0);
  \LGLAgap{0}{4}
  \LGLAdot{0}{0}\LGLAdot{2}{0.8}\LGLAdot{4}{0}
  \LGLAlabel{0}{0}{1}\LGLAlabel{4}{0}{n}
  \node[font=\scriptsize,below=4pt] at (2,0.8) {$0$};
\end{scope}
\begin{scope}[xshift=4.5cm]
  \node at (0,1.35) {$\widetilde D_4$};
  \draw (-0.65,0)--(0.65,0) (0,-0.65)--(0,0.65);
  \LGLAdot{0}{0}\LGLAdot{-0.65}{0}\LGLAdot{0.65}{0}
  \LGLAdot{0}{-0.65}\LGLAdot{0}{0.65}
\end{scope}
\begin{scope}[xshift=6.3cm]
  \node[anchor=west] at (-0.2,1.35) {$\widetilde D_n\quad(n\ge5)$};
  \draw (0,0.55)--(1,0)--(0,-0.55)
    (4,0.55)--(3,0)--(4,-0.55);
  \LGLAgap{1}{3}
  \LGLAdot{0}{0.55}\LGLAdot{0}{-0.55}\LGLAdot{1}{0}
  \LGLAdot{3}{0}\LGLAdot{4}{0.55}\LGLAdot{4}{-0.55}
  \node[font=\scriptsize,left=3pt] at (0,0.55) {$0$};
  \node[font=\scriptsize,left=3pt] at (0,-0.55) {$1$};
  \LGLAlabel{1}{0}{2}\LGLAlabel{3}{0}{n-2}
  \node[font=\scriptsize,right=3pt] at (4,0.55) {$n-1$};
  \node[font=\scriptsize,right=3pt] at (4,-0.55) {$n$};
\end{scope}

\begin{scope}[yshift=-1.9cm]
  \node[anchor=west] at (-0.2,1.35) {$\widetilde B_n\quad(n\ge4)$};
  \draw (0,0.55)--(1,0)--(0,-0.55);
  \LGLAgap{1}{3}\LGLAdouble{3}{4}{1}
  \LGLAdot{0}{0.55}\LGLAdot{0}{-0.55}\LGLAdot{1}{0}
  \LGLAdot{3}{0}\LGLAdot{4}{0}
  \node[font=\scriptsize,left=3pt] at (0,0.55) {$0$};
  \node[font=\scriptsize,left=3pt] at (0,-0.55) {$1$};
  \LGLAlabel{1}{0}{2}\LGLAlabel{3}{0}{n-1}\LGLAlabel{4}{0}{n}
\end{scope}
\begin{scope}[xshift=6.3cm,yshift=-1.9cm]
  \node[anchor=west] at (-0.2,1.35) {$A_{2n-1}^{(2)}\quad(n\ge4)$};
  \draw (0,0.55)--(1,0)--(0,-0.55);
  \LGLAgap{1}{3}\LGLAdouble{3}{4}{-1}
  \LGLAdot{0}{0.55}\LGLAdot{0}{-0.55}\LGLAdot{1}{0}
  \LGLAdot{3}{0}\LGLAdot{4}{0}
  \node[font=\scriptsize,left=3pt] at (0,0.55) {$0$};
  \node[font=\scriptsize,left=3pt] at (0,-0.55) {$1$};
  \LGLAlabel{1}{0}{2}\LGLAlabel{3}{0}{n-1}\LGLAlabel{4}{0}{n}
\end{scope}
\begin{scope}[xshift=3.3cm,yshift=-1.52cm]
  \node[anchor=east] at (0.7,0) {$\widetilde B_3$};
  \begin{scope}[xshift=0.72cm]
    \draw (0,0.4)--(0.7,0)--(0,-0.4);
    \LGLAdouble{0.7}{1.7}{1}
    \LGLAdot{0}{0.4}\LGLAdot{0}{-0.4}
    \LGLAdot{0.7}{0}\LGLAdot{1.7}{0}
  \end{scope}
\end{scope}
\begin{scope}[xshift=3.3cm,yshift=-2.31cm]
  \node[anchor=east] at (0.7,0) {$A_5^{(2)}$};
  \begin{scope}[xshift=0.72cm]
    \draw (0,0.4)--(0.7,0)--(0,-0.4);
    \LGLAdouble{0.7}{1.7}{-1}
    \LGLAdot{0}{0.4}\LGLAdot{0}{-0.4}
    \LGLAdot{0.7}{0}\LGLAdot{1.7}{0}
  \end{scope}
\end{scope}

\begin{scope}[yshift=-3.8cm]
  \node[anchor=west] at (-0.2,1.35) {$\widetilde C_n\quad(n\ge3)$};
  \LGLAdouble{0}{1}{1}\LGLAgap{1}{3}\LGLAdouble{3}{4}{-1}
  \LGLAdot{0}{0}\LGLAdot{1}{0}\LGLAdot{3}{0}\LGLAdot{4}{0}
  \LGLAlabel{0}{0}{0}\LGLAlabel{1}{0}{1}
  \LGLAlabel{3}{0}{n-1}\LGLAlabel{4}{0}{n}
\end{scope}
\begin{scope}[xshift=6.3cm,yshift=-3.8cm]
  \node[anchor=west] at (-0.2,1.35) {$D_{n+1}^{(2)}\quad(n\ge3)$};
  \LGLAdouble{0}{1}{-1}\LGLAgap{1}{3}\LGLAdouble{3}{4}{1}
  \LGLAdot{0}{0}\LGLAdot{1}{0}\LGLAdot{3}{0}\LGLAdot{4}{0}
  \LGLAlabel{0}{0}{0}\LGLAlabel{1}{0}{1}
  \LGLAlabel{3}{0}{n-1}\LGLAlabel{4}{0}{n}
\end{scope}
\begin{scope}[xshift=3.3cm,yshift=-3.43cm]
  \node[anchor=east] at (0.7,0) {$\widetilde C_2$};
  \begin{scope}[xshift=0.72cm]
    \LGLAdouble{0}{0.85}{1}\LGLAdouble{0.85}{1.7}{-1}
    \LGLAdot{0}{0}\LGLAdot{0.85}{0}\LGLAdot{1.7}{0}
  \end{scope}
\end{scope}
\begin{scope}[xshift=3.3cm,yshift=-4.22cm]
  \node[anchor=east] at (0.7,0) {$D_3^{(2)}$};
  \begin{scope}[xshift=0.72cm]
    \LGLAdouble{0}{0.85}{-1}\LGLAdouble{0.85}{1.7}{1}
    \LGLAdot{0}{0}\LGLAdot{0.85}{0}\LGLAdot{1.7}{0}
  \end{scope}
\end{scope}

\begin{scope}[yshift=-5.7cm]
  \node[anchor=west] at (-0.2,1.55) {$A_{2n}^{(2)}\quad(n\ge3)$};
  \LGLAdouble{0}{1}{-1}\LGLAgap{1}{3}\LGLAdouble{3}{4}{-1}
  \LGLAdot{0}{0}\LGLAdot{1}{0}\LGLAdot{3}{0}\LGLAdot{4}{0}
  \LGLAlabel{0}{0}{0}\LGLAlabel{1}{0}{1}
  \LGLAlabel{3}{0}{n-1}\LGLAlabel{4}{0}{n}
\end{scope}
\begin{scope}[xshift=3.3cm,yshift=-5.7cm]
  \node[anchor=east] at (0.7,0) {$A_4^{(2)}$};
  \begin{scope}[xshift=0.72cm]
    \LGLAdouble{0}{0.85}{-1}\LGLAdouble{0.85}{1.7}{-1}
    \LGLAdot{0}{0}\LGLAdot{0.85}{0}\LGLAdot{1.7}{0}
  \end{scope}
\end{scope}
\begin{scope}[xshift=6.3cm,yshift=-5.7cm]
  \node[anchor=west] at (-0.2,1.55) {$\widetilde E_6$};
  \draw (0,0)--(4,0) (2,0)--(2,1.15);
  \foreach \j in {0,1,2,3,4}{\LGLAdot{\j}{0}}
  \LGLAdot{2}{0.575}\LGLAdot{2}{1.15}
\end{scope}

\begin{scope}[yshift=-7.6cm,x=0.53cm]
  \node[anchor=west] at (-0.25,1.35) {$\widetilde E_7$};
  \draw (0,0)--(6,0) (3,0)--(3,0.8);
  \foreach \j in {0,1,2,3,4,5,6}{\LGLAdot{\j}{0}}
  \LGLAdot{3}{0.8}
\end{scope}
\begin{scope}[xshift=6.3cm,yshift=-7.6cm,x=0.49cm]
  \node[anchor=west] at (-0.3,1.35) {$\widetilde E_8$};
  \draw (0,0)--(7,0) (2,0)--(2,0.8);
  \foreach \j in {0,1,2,3,4,5,6,7}{\LGLAdot{\j}{0}}
  \LGLAdot{2}{0.8}
\end{scope}

\begin{scope}[yshift=-9.5cm]
  \node[anchor=west] at (-0.2,1.35) {$\widetilde F_4$};
  \draw (0,0)--(2,0) (3,0)--(4,0);
  \LGLAdouble{2}{3}{1}
  \foreach \j in {0,1,2,3,4}{\LGLAdot{\j}{0}}
\end{scope}
\begin{scope}[xshift=6.3cm,yshift=-9.5cm]
  \node[anchor=west] at (-0.2,1.35) {$E_6^{(2)}$};
  \draw (0,0)--(2,0) (3,0)--(4,0);
  \LGLAdouble{2}{3}{-1}
  \foreach \j in {0,1,2,3,4}{\LGLAdot{\j}{0}}
\end{scope}

\begin{scope}[yshift=-11.4cm]
  \node[anchor=west] at (-0.2,1.35) {$\widetilde G_2$};
  \draw (0,0)--(1.6,0);
  \LGLAtriple{1.6}{3.2}{1}
  \LGLAdot{0}{0}\LGLAdot{1.6}{0}\LGLAdot{3.2}{0}
\end{scope}
\begin{scope}[xshift=6.3cm,yshift=-11.4cm]
  \node[anchor=west] at (-0.2,1.35) {$D_4^{(3)}$};
  \draw (0,0)--(1.6,0);
  \LGLAtriple{1.6}{3.2}{-1}
  \LGLAdot{0}{0}\LGLAdot{1.6}{0}\LGLAdot{3.2}{0}
\end{scope}
\end{tikzpicture}}

\smallskip
\parbox{0.98\textwidth}{\small
Dots indicate a chain of simple edges. Arrows point toward the shorter roots.}
\endgroup
\end{center}

\begin{exercise} Prove this by showing that in each case there exists a nonzero vector $v$ such that $Av=0$. For example, in the simply laced case (only simple edges), this amounts to finding a labeling of the vertices by nonzero numbers such that the sum of labels of the neighbors to each vertex is twice the label of that vertex, and in the non-simply laced case it's a weighted version of that. 
\end{exercise}  

Thus they cannot occur inside a Dynkin diagram (as the restriction of a positive definite inner product to a subspace must be positive definite). 

We conclude that a Dynkin diagram is a tree. Indeed, 
it cannot have a loop with simple edges, since this is the affine diagram $\widetilde A_{n-1}$, which has a null vector $(1,...,1)$. If there is a loop with non-simple edges, this is even worse - this vector will have a negative inner product with itself. 

Further, it cannot have vertices with at least four simple edges coming out since it cannot have a subdiagram $\widetilde D_4$ (and for non-simple edges it is even worse, as before). Thus all the vertices of our tree 
are i-valent for $i\le 3$. 

Also we cannot have a subdiagram $\widetilde D_n$, $n\ge 5$, which implies that there is at most one trivalent vertex. 

Further, if there is a triple edge then the diagram is $G_2$. 
There is no way to attach any edge to the $G_2$ diagram because $D_4^{(3)}$ and $\widetilde G_2$ are forbidden. 

Next, if there is a trivalent vertex then there cannot be a non-simple edge anywhere in the diagram (as we have forbidden affine diagrams $A_{2k-1}^{(2)},\widetilde B_n$). So in this case the diagram is simply laced, so it must be on our list ($D_n,E_6,E_7,E_8$) since it cannot contain affine diagrams $\widetilde E_6,\widetilde E_7,\widetilde E_8$. 

It remains to consider chain-shaped diagrams. They can't contain two double edges (affine diagrams $A_{2k}^{(2)},D_{k+1}^{(2)},\widetilde C_n$). 
Thus if the double edge is at the end, we can only get $B_n$ and $C_n$.  

Finally, if the double edge is in the middle, we can't have affine subdiagram  
$\widetilde F_4$ or $E_6^{(2)}$, so our diagram must be $F_4$. 
Theorem \ref{classdyn} is proved. 

\begin{remark} Note that we have {\bf exceptional isomorphisms} $D_2\cong A_1\times A_1$, $D_3\cong A_3$, $B_2\cong C_2$. Otherwise the listed root systems are distinct. 
\end{remark} 

\subsection{Simply laced and non-simply laced diagrams} 

As we already mentioned, a Dynkin diagram (or the corresponding root system) is called {\bf simply laced} if all the edges are simple, i.e. $a_{ij}=0,-1$ for $i\ne j$. This is equivalent to the Cartan matrix being symmetric, or to all roots having the same length. The connected simply-laced diagrams are $A_n,n\ge 1; D_n, n\ge 4; E_6,E_7,E_8$. The remaining diagrams $B_n,C_n,F_4,G_2$ are not simply laced, but they contain roots of only two squared lengths, whose ratio is $2$ for double edge ($B_n,C_n,F_4$) and $3$ for triple edge 
($G_2$). The roots of the bigger length are called {\bf long} and of the smaller length are called {\bf short}. 

It is easy to see that long and short roots form a root system of the same rank (but not necessarily irreducible). For instance, in $G_2$ both form a root system of type $A_2$, and in $B_2$ both are $A_1\times A_1$. In $B_3$ long roots form $D_3$ and short ones form $A_1\times A_1\times A_1$. However, only long roots form a root subsystem, since a long positive root can be the sum of two short ones, but not vice versa. 

\section{\bf Construction of a semisimple Lie algebra from a Dynkin diagram} 

\subsection{Serre relations} 

Let $\bf k$ be an algebraically closed field of characteristic zero. We would like to show that any reduced root system gives rise to a semisimple Lie algebra over $\bf k$, and moreover a unique one. To this end, it suffices to show that any reduced {\it irreducible} root system gives rise to a unique (finite dimensional) {\it simple} Lie algebra. 

Let $\g$ be a finite dimensional simple Lie algebra over ${\bf k}$ with Cartan subalgebra $\h\subset \g$ and root system $R\subset \h^*$ (which is thus reduced and irreducible). Fix a polarization of $R$ with the set of simple roots $\Pi=(\alpha_1,...,\alpha_r)$, and let $A=(a_{ij})$ be the Cartan matrix of $R$. We have a decomposition $\g= \n_+\oplus \h\oplus \n_-$, where $\n_\pm:=\oplus_{\alpha\in R_{\pm}}\g_\alpha$ are the Lie subalgebras spanned by 
positive, respectively negative root vectors. 
Pick elements $e_i\in \g_{\alpha_i}$, $f_i\in \g_{-\alpha_i}$ so that $e_i,f_i,h_i=[e_i,f_i]$ form an ${\mathfrak{sl}}_2$-triple. 

\begin{theorem}\label{serrel} (Serre relations) (i) The elements $e_i,f_i,h_i$, $i=1,...,r$ generate $\g$. 

(ii) These elements satisfy the following relations: 
$$
[h_i,h_j]=0,\ [h_i,e_j]=a_{ij}e_j,\ [h_i,f_j]=-a_{ij}f_j,\ [e_i,f_j]=\delta_{ij}h_i,
$$
$$
 ({\rm ad}e_i)^{1-a_{ij}}e_j=0,\  ({\rm ad}f_i)^{1-a_{ij}}f_j=0,\ i\ne j.
 $$
\end{theorem} 

The last two sets of relations are called {\bf Serre relations}. Note that if $a_{ij}=0$ then the Serre relations just say that $[e_i,e_j]=[f_i,f_j]=0$. 

\begin{proof} (i) We know that $h_i$ form a basis of $\h$, so it suffices to show that 
$e_i$ generate $\n_+$ and $f_i$ generate $\n_-$. We only prove the first statement, the second being 
the same for the opposite polarization.  

Let $\n_+'\subset \n_+$ be the Lie subalgebra generated by $e_i$. It is clear that $\n_+'=\oplus_{\alpha\in R_+'} \g_\alpha$ where $R_+'\subset R_+$. Assume the contrary, that $R_+'\ne R_+$. Pick $\alpha\in R_+\setminus R_+'$ with the smallest height (it is not a simple root). 
Then $\g_{\alpha-\alpha_i}\subset \n_+'$, so $[e_i,\g_{\alpha-\alpha_i}]=0$. 
Let $x\in \g_{-\alpha}$ be a nonzero element. We have 
$$
([x,e_i],y)=(x,[e_i,y])=0
$$
for any $y\in \g_{\alpha-\alpha_i}$. Thus $[x,e_i]=0$ for all $i$, which implies, by the representation theory
of $\mathfrak{sl}_2$ (Subsection \ref{sl2rep}), that $(\alpha,\alpha_i^\vee)\le 0$ for all $i$, hence $(\alpha,\alpha_i)\le 0$ for all $i$. This would imply that $(\alpha,\alpha)\le 0$, a contradiction. This proves (i).   

(ii) All the relations except the Serre relations follow from the definition and properties of root systems. 
So only the Serre relations require proof. We prove only the relation involving $f_i$, the other one being the same for the opposite polarization. Consider the $(\mathfrak{sl}_2)_i$-submodule $M_{ij}$ of $\g$ 
generated by $f_j$. It is finite dimensional and 
we have $[h_i,f_j]=-a_{ij}f_j$, $[e_i,f_j]= 0$. 
Thus by the representation theory of $\mathfrak{sl}_2$ (Subsection \ref{sl2rep}) we must have $M_{ij}\cong V_{-a_{ij}}$.
Hence $({\rm ad}f_i)^{-a_{ij}+1}f_j=0$. 
\end{proof} 

\subsection{The Serre presentation for semisimple Lie algebras}

Now for any reduced root system $R$ let $\g(R)$ be the Lie algebra generated by $e_i,f_i,h_i,i=1,...,r$, 
with {\bf defining relations} being the relations of Theorem \ref{serrel}. Precisely, this means 
that $\g(R)$ is the quotient of the free Lie algebra $FL_{3r}$ with generators $e_i,f_i,h_i$ 
modulo the Lie ideal generated by the differences of the left and right hand sides of these relations. 

\begin{theorem} (Serre) (i) The Lie subalgebra $\n_+$ of $\g(R)$ generated by $e_i$ 
has the Serre relations  $({\rm ad}e_i)^{1-a_{ij}}e_j=0$ as the defining relations. Similarly, 
the Lie subalgebra $\n_-$ of $\g(R)$ generated by $f_i$ 
has the Serre relations  $({\rm ad}f_i)^{1-a_{ij}}f_j=0$ as the defining relations.
In particular, $e_i,f_i\ne 0$ in $\g(R)$. Moreover, $h_i$ are linearly independent. 

(ii) $\g(R)$ is a sum of finite dimensional modules over every simple root subalgebra $({\mathfrak{sl}_2})_i=(e_i,f_i,h_i)$. 

(iii) $\g(R)$ is finite dimensional. 

(iv) $\g(R)$ is semisimple and has root system $R$. 
\end{theorem} 

\begin{proof} It is easy to see that $\g(R_1\sqcup R_2)=\g(R_1)\oplus \g(R_2)$, 
so it suffices to prove the theorem for irreducible root systems. 

(i) Consider the (in general, infinite dimensional) Lie algebra $\widetilde {\g(R)}$ 
generated by $e_i,f_i,h_i$ with the defining relations of Theorem \ref{serrel} without the Serre relations. This Lie algebra is $\Bbb Z$-graded, with $\deg(e_i)=1$, $\deg(f_i)=-1$, $\deg(h_i)=0$. Thus we have a decomposition 
$$
\widetilde{\g(R)}=\widetilde{\n_+}\oplus \widetilde{\h}\oplus \widetilde{\n_-},
$$ 
where $\widetilde{\n_+}$, $\widetilde{\h}$ and $\widetilde{\n_-}$ are Lie subalgebras spanned by 
elements of positive, zero and negative degree, respectively. Moreover, it is easy to see that 
$\widetilde{\n_+}$ is generated by $e_i$, $\widetilde{\n_-}$ is generated by $f_i$, and $\widetilde{\h}$ is spanned by $h_i$ (indeed, any commutator can be simplified to have only $e_i$, only $f_i$, or only a single $h_i$). 

\begin{lemma}\label{free} (i) The Lie algebra $\widetilde{\n_+}$ is free on the generators $e_i$
and $\widetilde{\n_-}$ is free on the generators $f_i$. 

(ii) $h_i$ are linearly independent in $\widetilde \h$ (i.e., $\widetilde \h\cong \h$). 
\end{lemma} 

\begin{proof} (i) We prove only the second statement, the first one being the same for the opposite polarization. Let $\h'$ be a vector space with basis $h_i'$, $i=1,...,r$ 
and consider the Lie algebra $\mathfrak{a}:=\h'\ltimes FL_r$, where $FL_r$ 
is freely generated by $f_1',...,f_r'$ and 
$$
[h_i',f_j']=-a_{ij}f_j',\  [h_i',h_j']=0. 
$$
Consider the universal enveloping algebra 
$$
U=U(\mathfrak{a})={\bf k}[h_1',...,h_r']\ltimes {\bf k}\langle f_1',...,f_r'\rangle,
$$ 
which as a vector space is naturally identified with the tensor product 
${\bf k}\langle f_1,...,f_r\rangle\otimes {\bf k}[h_1',...,h_r']$, via $f\otimes h\mapsto fh$ 
(by Proposition \ref{freelieprop}). 
Now define an action of $\widetilde{\g(R)}$ on the space $U$ as follows. For $P\in {\bf k}[h_1',...,h_r']$ and $w$ a word in $f_i'$ of weight $-\alpha$, we set 
$$
h_i(w\otimes P)=w\otimes (h_i'-\alpha(h_i))P,\ f_i(w\otimes P)=f_i'w\otimes P,\
$$
$$
e_i(f_{j_1}'...f_{j_s}'\otimes P)=\sum_{k: j_k=i} f'_{j_1}....\widehat{f_{j_k}'}...f_{j_s}'\otimes (h_i'-(\alpha_{j_{k+1}}+...+\alpha_{j_s})(h_i))P
$$
(where the hat means that the corresponding factor is omitted). 
It is easy to check that this indeed defines an action, i.e., the relations of $\widetilde{\g(R)}$ are satisfied (check it!). 
Thus we have a linear map $\widetilde{\g(R)}\to U$ given by $x\mapsto x(1)$. The restriction of this map 
to the Lie subalgebra $\widetilde{\n_-}$ is a map $\phi: \widetilde{\n}_-\to FL_r$ which sends every iterated commutator of $f_i$ to 
itself. This implies that $\phi$ is an isomorphism, i.e., $\widetilde{\n_-}$ is free. 

(ii) The elements $h_i(1)=h_i'$ are linearly independent, hence so are $h_i$. 
\end{proof} 

Now consider the element $S_{ij}^+:=({\rm ad}e_i)^{1-a_{ij}}e_j$ in $\widetilde{\n_+}$ and 
$S_{ij}^-:=({\rm ad}f_i)^{1-a_{ij}}f_j$ in $\widetilde{\n_-}$. 
It is easy to check that $[f_k,S_{ij}^+]=0$ (this follows easily from the representation theory of $\mathfrak{sl}_2$, Subsection \ref{sl2rep},--check it!). Therefore, setting $I_+$ to be the ideal in the Lie algebra $\widetilde{\n_+}$ generated by $S_{ij}^+$, and $I_-$ to be the ideal in the Lie algebra $\widetilde{\n_-}$ generated by $S_{ij}^-$, we see that the ideal of Serre relations in $\widetilde{\g(R)}$ is $I_+\oplus I_-$. Lemma \ref{free} now implies (i). 

(ii) The Serre relations imply that $e_j$ generates the representation $V_{-a_{ij}}$ 
of $({\mathfrak{sl}_2})_i$ for $j\ne i$, and so does $f_j$. Also any element of $\h$ generates 
$V_0$ or $V_2$ or the sum of the two, and $e_i,f_i$ generate $V_2$. This implies (ii)  
since $\g(R)$ is generated by $e_i,f_i,h_i$, and if $x$ generates a representation $X$ of $({\mathfrak{sl}_2})_i$ and $y$ generates a representation $Y$ then $[x,y]$ generates a quotient of $X\otimes Y$. 

(iii) We have $\g(R)=\oplus_{\alpha\in Q} \g_\alpha$, where $\g_\alpha$ are the subspaces of $\g(R)$ of weight $\alpha$, and $\g_0=\h$.  Let $Q_+$ be the $\Bbb Z_{\ge 0}$-span 
of $\alpha_i$. Then $\g_\alpha$ is zero unless $\alpha\in Q_+$ or $-\alpha\in Q_+$, and 
is finite dimensional for any $\alpha$. 

We will now show that if $\g_\alpha\ne 0$ then $\alpha\in R$ or $\alpha=0$, which implies (iii). 
It suffices to consider $\alpha\in Q_+$. 
We prove the statement by induction in the height ${\rm ht}(\alpha)=\sum_i k_i$ where $\alpha=\sum_i k_i\alpha_i$. 
The base case (height $1$) is obvious, so we only need to justify the inductive step. 
We have $(\alpha,\omega_i^\vee)=k_i\ge 0$ for all $i$. If there is only one $i$ with $k_i>0$ 
then the statement is clear since $\g_{m\alpha_i}=0$ if $m\ge 2$ (as $\n_+$ is generated by $e_i$). 
So assume that there are at least two such indices $i$. Since $(\alpha,\alpha)>0$, 
there exists $i$ such that $(\alpha,\alpha_i^\vee)>0$. By the representation theory of $\mathfrak{sl}_2$ (Subsection \ref{sl2rep}), 
$\g_{s_i\alpha}\ne 0$. Clearly, $s_i\alpha=\alpha-(\alpha,\alpha_i^\vee)\alpha_i\notin -Q_+$ (since $k_j>0$ for at least two indices $j$), so $s_i\alpha\in Q_+$ but has height smaller than $\alpha$ (as $(\alpha,\alpha_i^\vee)>0$). 
So by the induction assumption $s_i\alpha\in R$, which implies $\alpha\in R$. This proves (iii). 

(iv) We see that $\g(R)=\h\oplus\bigoplus_{\alpha\in R}\g_\alpha$, where $\g_\alpha$ are 1-dimensional (this follows from (ii),(iii) since every root can be mapped to a simple root by a composition of simple reflections). Let $I$ be a nonzero ideal in $\g$. Since $R$ spans $\h^*$, if $0\ne h\in I\cap \h$ then there exists a root $\alpha$ with $\alpha(h)\ne 0$. So $[h,\g_\alpha]=\g_\alpha\subset I$. Thus in any case $I$ contains $\g_\alpha$ for some root $\alpha$.
Also, by the representation theory of 
$\mathfrak{sl}_2$, $I_\beta:=I\cap \g_\beta\ne 0$ implies $I_{w\beta}\ne 0$ for all $w\in W$. Thus $I_{\alpha_i}\ne 0$ for some $i$, i.e., $e_i\in I$. Hence $h_i,f_i\in I$. Now let $J$ be the set of indices $j$ for which $e_j,f_j,h_j\in I$ (or, equivalently, just $e_j\in I$); we have shown it is nonempty. Since $[h_j,e_k]=a_{jk}e_k$, we find that if $j\in J$ and $a_{jk}\ne 0$ (i.e., $k$ is connected to $j$ in the Dynkin diagram) then $k\in J$. Since the Dynkin diagram is connected, $J=[1,...,r]$ and $I=\g$. Thus $\g$ is simple and clearly has root system $R$. This proves (iv) and completes the proof of Serre's theorem. 
\end{proof} 

\begin{corollary} Isomorphism classes of simple Lie algebras over ${\bf k}$ are in bijection 
with Dynkin diagrams $A_n$, $n\ge 1$, $B_n$, $n\ge 2$, $C_n$, $n\ge 3$, 
$D_n$, $n\ge 4$, $E_6,E_7,E_8$, $F_4$ and $G_2$.  
\end{corollary} 

\section{\bf Representation theory of semisimple Lie algebras} 

\subsection{Representations of semisimple Lie algebras} 

We will now develop representation theory of complex semisimple Lie algebras. The representation theory of semisimple Lie algebras over an algebraically closed field of characteristic zero is completely parallel, so we will stick to the complex case. So all representations will be over $\Bbb C$. 
We will mostly be interested in finite dimensional representations; as we know, they can be exponentiated to holomorphic representations of the corresponding simply connected Lie group $G$, which defines a bijection between isomorphism classes of such representations of $\g$ and $G$. 

Let $\g$ be a semisimple Lie algebra. Recall that by Theorem \ref{comred}, every finite dimensional representation of $\g$ is completely reducible, so to classify finite dimensional representations it suffices to classify irreducible representations. 

As in the simplest case of $\mathfrak{sl}_2$, a crucial tool is the decomposition of a representation in a direct sum of eigenspaces of a Cartan subalgebra $\h\subset \g$. 

\begin{definition} Let $\lambda\in \h^*$, and $V$ a representation of $\g$ (possibly infinite dimensional). 
Then a vector $v\in V$ is said to have {\bf weight} $\lambda$ if $hv=\lambda(h)v$ for all $h\in \h$; such vectors are called {\bf weight vectors}. 
The subspace of such vectors is called the {\bf weight subspace of $V$ of weight $\lambda$} and denoted by $V[\lambda]$. If $V[\lambda]\ne 0$, we say that $\lambda$ is a weight of $V$, and the set of weights of $V$ is denoted by $P(V)$. 
\end{definition} 

It is easy to see that $\g_\alpha V[\lambda]\subset V[\lambda+\alpha]$. 

Let $V'\subset V$ be the span of all weight vectors in $V$. Then it is clear that 
$V'=\oplus_{\lambda\in \h^*}V[\lambda]$. 

\begin{definition} We say that $V$ {\bf has a weight decomposition} (with respect to a Cartan subalgebra $\h\subset \g$) if $V'=V$, i.e., if $V=\oplus_{\lambda\in \h^*}V[\lambda]$. 
\end{definition} 

Note that not every representation of $\g$ has a weight decomposition (e.g., for $V=U(\g)$ with $\g$ acting by left multiplication all weight subspaces are zero).  

\begin{proposition} Any finite dimensional representation $V$ of $\g$ has a weight decomposition. 
Moreover, all weights of $V$ are {\bf integral}, i.e., $P(V)$ is a finite subset of the weight lattice $P\subset \h^*$ of $\g$.  
\end{proposition} 

\begin{proof} For each $i=1,...,r$, $V$ is a finite dimensional representation of the root subalgebra $(\mathfrak{sl}_2)_i$, so its element $h_i$ acts semisimply on $V$. Thus $\h$ acts semisimply on $V$, hence $V$ has a weight decomposition. Also eigenvalues of $h_i$ are integers, so 
for any $\lambda\in P(V)$ we have $\lambda(h_i)=(\lambda,\alpha_i^\vee)\in \Bbb Z$, hence $\lambda\in P$.   
\end{proof} 

\begin{definition} A vector $v$ in $V[\lambda]$ is called a {\bf highest weight vector of weight $\lambda$} if $e_iv=0$ for all $i$, i.e., if $\n_+v=0$. A representation $V$ of $\g$ is a {\bf highest weight representation with highest weight $\lambda$} if it is generated by such a nonzero vector. 
\end{definition} 

\begin{proposition} Any finite dimensional representation $V\ne 0$ contains a nonzero highest weight vector of some weight $\lambda$. Thus every irreducible finite dimensional representation of $\g$ is a highest weight representation. 
\end{proposition} 

\begin{proof} Note that $P(V)$ is a finite set. Let $\rho^\vee=\sum_{i=1}^r\omega_i^\vee$. Pick $\lambda\in P(V)$ so that $(\lambda,\rho^\vee)$ is maximal. Then $\lambda+\alpha_i\notin P(V)$ for any $i$, since $(\lambda+\alpha_i,\rho^\vee)=(\lambda,\rho^\vee)+1$. Hence for any nonzero $v\in V[\lambda]$ (which exists as $\lambda\in P(V)$) we have 
$e_iv=0$. 

The second statement follows since an irreducible representation is generated by any of its nonzero vectors. 
\end{proof} 

\subsection{Verma modules} Even though we are mostly interested in finite dimensional representations of $\g$, it is useful to consider some infinite dimensional representations, which are called {\bf Verma modules}. 

The Verma module $M_\lambda$ is defined as ``the largest highest weight representation with highest weight $\lambda$". Namely, it is generated by a single highest weight vector $v_\lambda$ 
with {\bf defining relations} $hv=\lambda(h)v$ for $h\in \h$ and $e_iv=0$. More formally speaking, 
we make the following definition. 

\begin{definition} Let $I_\lambda\subset U(\g)$ be the left ideal generated by the elements $h-\lambda(h),h\in \h$ and 
$e_i$, $i=1,...,r$. Then the {\bf Verma module} $M_\lambda$ is the quotient 
$U(\g)/I_\lambda$. 
\end{definition} 

In this realization, the highest weight vector $v_\lambda$ is just the class of the unit $1$ of $U(\g)$. 
 
\begin{proposition} The map $\phi: U(\n_-)\to M_\lambda$ given by $\phi(x)=xv_\lambda$ 
is an isomorphism of left $U(\n_-)$-modules.  
\end{proposition}  

\begin{proof} 
By the PBW theorem, the multiplication map 
$$
\xi: U(\n_-)\otimes U(\h\oplus \n_+)\to U(\g)
$$ 
is a linear isomorphism. It is easy to see that $\xi^{-1}(I_\lambda)=U(\n_-)\otimes K_\lambda$, where 
$$
K_\lambda:=\sum_i U(\h\oplus \n_+)(h_i-\lambda(h_i))+\sum_i U(\h\oplus \n_+)e_i
$$
is the kernel of the homomorphism $\lambda_+: U(\h\oplus \n_+)\to \Bbb C$ 
given by $\lambda_+(h)=\lambda(h)$, $h\in \h$, $\lambda_+(e_i)=0$. 
Thus, we have a natural isomorphism of left $U(\n_-)$-modules
$$
U(\n_-)=U(\n_-)\otimes U(\h\oplus \n_+)/K_\lambda\to M_\lambda,
$$
as claimed.
\end{proof} 

\begin{remark} The definition of $M_\lambda$ means that 
it is the {\bf induced module} $U(\g)\otimes_{U(\h\oplus \n_+)}\Bbb C_\lambda$, where 
$\Bbb C_\lambda$ is the one-dimensional representation of $\h\oplus \n_+$ on which it acts via $\lambda_+$. 
\end{remark} 

Recall that $Q_+$ denotes the set of elements $\sum_{i=1}^r k_i\alpha_i$ where $k_i\in \Bbb Z_{\ge 0}$. We obtain

\begin{corollary} $M_\lambda$ has a weight decomposition with 
$P(M_\lambda)=\lambda-Q_+$, $\dim M_\lambda[\lambda]=1$, and weight subspaces of $M_\lambda$ are finite dimensional. 
\end{corollary} 

\begin{proposition} (i) (Universal property of Verma modules) 
If $V$ is a representation of $\g$ and $v\in V$ is a vector such that 
$hv=\lambda(h)v$ for $h\in h$ and $e_iv=0$ for $1\le i\le r$ then there is a unique homomorphism 
$\eta: M_\lambda\to V$ such that $\eta(v_\lambda)=v$. In particular, if $V$ is generated by such $v\ne 0$ (i.e., $V$ is a highest weight representation with highest weight vector $v$) then $V$ is a quotient 
of $M_\lambda$.  

(ii) Every highest weight representation  has a weight decomposition into finite dimensional weight subspaces. 
\end{proposition} 

\begin{proof} (i) Uniqueness follows from the fact that $v_\lambda$ generates $M_\lambda$. 
To construct $\eta$, note that we have a natural homomorphism of $\g$-modules 
$\widetilde \eta: U(\g)\to V$ given by $\widetilde \eta(x)=xv$. Moreover, $\widetilde\eta|_{I_\lambda}=0$
thanks to the relations satisfied by $v$, so $\widetilde \eta$ descends to a map 
$\eta: U(\g)/I_\lambda=M_\lambda\to V$. Moreover, if $V$ is generated by $v$ then this map is surjective, as desired. 

(ii) This follows from (i) since a quotient of any representation with a weight decomposition 
must itself have a weight decomposition. 
\end{proof} 

\begin{corollary} Every highest weight representation $V$ has a unique highest weight generator, up to scaling. 
\end{corollary} 

\begin{proof} Suppose $v,w$ are two highest weight generators of $V$ of weights $\lambda,\mu$. 
If $\lambda=\mu$ then they are proportional since $\dim V[\lambda]\le \dim M_\lambda[\lambda]=1$, as $V$ is a quotient of $M_\lambda$. On the other hand, if $\lambda\ne \mu$, then we can assume without loss of generality that $\lambda-\mu\notin Q_+$ (otherwise switch $\lambda,\mu$). Then 
$\mu\notin \lambda-Q_+$, hence $\mu\notin P(V)$, a contradiction. 
\end{proof} 

\begin{proposition} For every $\lambda\in \h^*$, the Verma module $M_\lambda$ has a unique 
irreducible quotient $L_\lambda$. Moreover, $L_\lambda$ is a quotient of every highest weight $\g$-module $V$ with highest weight $\lambda$. 
\end{proposition} 

\begin{proof} Let $Y\subset M_\lambda$ be a proper submodule. Then $Y$ has a weight decomposition, and cannot contain a nonzero multiple of $v_\lambda$ (as otherwise $Y=M_\lambda$), so $P(Y)\subset 
(\lambda-Q_+)\setminus \lbrace \lambda \rbrace$. Now let $J_\lambda$  be the sum of all 
proper submodules $Y\subset M_\lambda$. Then $P(J_\lambda)\subset (\lambda-Q_+)\setminus \lbrace \lambda \rbrace$, so $J_\lambda$ is also a proper submodule of $M_\lambda$  (the maximal one). 
Thus, $L_\lambda:=M_\lambda/J_\lambda$ is an irreducible highest weight module with highest weight $\lambda$. Moreover, if $V$ is any nonzero quotient of $M_\lambda$ then the kernel 
$K$ of the map $M_\lambda\to V$ is a proper submodule, hence contained in $J_\lambda$. 
Thus the surjective map $M_\lambda\to L_\lambda$ descends to a surjective map $V\to L_\lambda$. 
The kernel of this map is a proper submodule of $V$, hence zero if $V$ is irreducible. 
Thus in the latter case $V\cong L_\lambda$. 
\end{proof} 

\begin{corollary} Irreducible highest weight $\g$-modules are classified by their highest weight $\lambda\in \h^*$, via the bijection $\lambda\mapsto L_\lambda$.
\end {corollary}  

\subsection{Finite dimensional modules} 

Since every finite dimensional irreducible $\g$-module is highest weight, it is of the form $L_\lambda$ for $\lambda$ belonging to some subset $P_F\subset P$, the set of weights $\lambda$ such that $L_\lambda$ is finite dimensional. So to obtain a final classification 
of finite dimensional irreducible representations of $\g$, we should determine the subset $P_F$. 

Let $P_+\subset P$ be the intersection of $P$ with the closure of the dominant Weyl chamber $C_+$; 
i.e., $P_+$ is the set of nonnegative integer linear combinations of the fundamental weights
$\omega_i$. In other words, $P_+$ is the set of $\lambda\in P$ 
such that $(\lambda,\alpha_i^\vee)\in \Bbb Z_{\ge 0}$ for $1\le i\le r$. Weights belonging to $P_+$ are called {\bf dominant integral}. 

\begin{proposition} We have $P_F\subset P_+$. 
\end{proposition} 

\begin{proof} The vector $v_\lambda$ is highest weight for $(\mathfrak{sl}_2)_i$ with highest weight $\lambda(h_i)=(\lambda,\alpha_i^\vee)$. This must be a nonnegative integer for the corresponding $\mathfrak{sl}_2$-module to be finite dimensional. 
\end{proof} 

\begin{lemma} If $\lambda\in P_+$ then in $L_\lambda$, we have $f_i^{\lambda(h_i)+1}v_\lambda=0$. 
\end{lemma} 

\begin{proof} By the representation theory of $\mathfrak{sl}_2$ (Subsection \ref{sl2rep}), 
we have $e_if_i^{\lambda(h_i)+1}v_\lambda=0$. 
Also $e_jf_i^{\lambda(h_i)+1}v_\lambda=0$ for $j\ne i$ since $[e_j,f_i]=0$. 
Thus, $w:=f_i^{\lambda(h_i)+1}v_\lambda$ is a highest weight vector in $L_\lambda$. So $w$ cannot be a generator (as the highest weight generator is unique up to scaling). Thus $w$ generates a proper submodule in $L_\lambda$, which must be zero since $L_\lambda$ is irreducible. 
\end{proof} 

\begin{lemma} \label{locfin} Let $V$ be a $\g$-module with weight decomposition into finite dimensional weight subspaces. If $V$ is a sum of finite dimensional $(\mathfrak{sl}_2)_i$-modules for each $i=1,...,r$, then 
for each $\lambda\in P$ and $w\in W$, 
$\dim V[\lambda]=\dim V[w\lambda]$. In particular, $P(V)$ is $W$-invariant.  
\end{lemma} 

\begin{proof} Since the Weyl group $W$ is generated by 
the simple reflections $s_i$, it suffices to prove the statement for $w=s_i$, and in fact to prove 
that $\dim V[\lambda]\le \dim V[s_i\lambda]$ (as $s_i^2=1$). 

If $(\lambda,\alpha_i^\vee)=m\ge 0$ 
then consider the operator $f_i^m: V[\lambda]\to V[s_i\lambda]$. We claim that this operator is injective, which implies the desired inequality. Indeed, let $v\in V[\lambda]$ be a nonzero vector
and $E$ be the representation of $(\mathfrak{sl}_2)_i$ generated by $v$. 
Then $E$ is finite dimensional, and $v\in E[m]$, so by the representation theory of $\mathfrak{sl}_2$ (Subsection \ref{sl2rep}), 
$f_i^mv\ne 0$, as claimed. 

Similarly, if $(\lambda,\alpha_i^\vee)=-m\le 0$ then the operator $e_i^m: V[\lambda]\to V[s_i\lambda]$ is injective. This proves the lemma. 
 \end{proof} 

Now we are ready to state the main classification theorem. 

\begin{theorem} For any $\lambda\in P_+$, $L_\lambda$ is finite dimensional; i.e., $P_F=P_+$. Thus finite dimensional irreducible representations of $\g$ are classified, up to an isomorphism, by their highest weight $\lambda\in P_+$, via the bijection $\lambda\mapsto L_\lambda$. 
Moreover, for any $\mu\in P$ and $w\in W$, $\dim L_\lambda[\mu]=\dim L_\lambda[w\mu]$. 
\end{theorem} 

\begin{proof} 
Since $f_i^{\lambda(h_i)+1}v_\lambda=0$, we see that $v_\lambda$ generates the irreducible finite dimensional $(\mathfrak{sl}_2)_i$-module of highest weight $\lambda(h_i)$. Also, every nonzero 
element of $\g$ generates a finite dimensional $(\mathfrak{sl}_2)_i$-module. 
But every vector of $L_\lambda$ is a linear combination 
of vectors of the form $a_1...a_Nv_\lambda,a_i\in \g$. 
Hence every vector in $L_\lambda$ generates a finite dimensional $(\mathfrak{sl}_2)_i$-module. 
Thus by Lemma \ref{locfin}, $P(L_\lambda)$ is $W$-invariant. 

Now let $\mu\in P(L_\lambda)\cap P_+$. Then $\mu=\lambda-\beta$, $\beta\in Q_+$, so 
$$
(\mu,\rho^\vee)=(\lambda,\rho^\vee)-(\beta,\rho^\vee)\le (\lambda,\rho^\vee).
$$ 
So if $\mu=\sum_i m_i\omega_i$, $m_i\in \Bbb Z_{\ge 0}$ then 
$\sum_i m_i(\omega_i,\rho^\vee)\le (\lambda,\rho^\vee)$. 
 Since $(\omega_i,\rho^\vee)\ge \frac{1}{2}$, this implies that 
 $P(L_\lambda)\cap P_+$ is finite. But we know that $WP_+=P$, hence $W(P(L_\lambda)\cap P_+)=P(L_\lambda)$, as $P(L_\lambda)$ is $W$-invariant. It follows that $P(L_\lambda)$ is finite, hence 
 $L_\lambda$ is finite dimensional. 
 \end{proof} 
 
 \begin{example} For $\g=\mathfrak{sl}_2$ 
 the dominant integral weights are nonnegative integers $n\in \Bbb Z_{\ge 0}$, and it is easy to see that $L_n=V_n$. 
 \end{example} 
 
 \section{\bf The Weyl character formula} 
 
 \subsection{Characters} 
 
  Let $V$ be a finite dimensional representation of a semisimple Lie algebra $\g$. 
  Recall that the action of $\g$ on $V$ can be exponentiated to the action 
  of the corresponding simply connected complex Lie group $G$. 
  Recall also that the {\bf character} of a finite dimensional representation $V$ 
  of any group $G$ is the function 
  $$
  \chi_V(g)={\rm Tr}|_V(g). 
  $$
  
  Let us compute this character in our case. To this end, let $\h\subset \g$ be a Cartan subalgebra, 
 $h\in \h$, and let us compute $\chi_V(e^h)$. Note that this completely determines $\chi_V$ since 
 it determines $\chi_V(e^x)$ for any semisimple element $x\in \g$ (as such an element can be conjugated into $\h$), and semisimple elements form a dense open set in $\g$ (complement of zeros of some polynomial). So elements of the form 
 $e^x$ as above form a dense open set at least in some neighborhood of $1$ in $G$, and an analytic function on $G$ is determined by its values on any nonempty open set. 
 
We know that 
 $V$ has a weight decomposition: $V=\oplus_{\mu\in P}V[\mu]$. Thus 
we have 
$$
\chi_V(e^h)=\sum_{\mu\in P}\dim V[\mu] e^{\mu(h)}. 
$$  
Consider the group algebra $\Bbb Z[P]$. It sits 
naturally inside the algebra of analytic functions on $\h$ via 
$\lambda\mapsto e^\lambda$, where $e^\lambda(h):=e^{\lambda(h)}$, and we see that 
$\chi_V\in \Bbb Z [P]$, namely 
$$
\chi_V=\sum_{\mu\in P}\dim V[\mu] e^\mu. 
$$
We will call the element $\chi_V$ the {\bf character} of $V$. 

\subsection{Category $\mathcal O$} 
Note that the above definition of character is a purely formal algebraic definition, i.e., $\chi_V$ is simply the generating function of dimensions of weight subspaces of $V$. So it makes sense for any (possibly infinite dimensional) representation $V$ with a weight decomposition into finite dimensional weight subspaces, except we may obtain an infinite sum. More precisely, we make the following definition.

\begin{definition} The category $\mathcal O_{\rm int}$ is the category of representations $V$ of $\g$ with weight decomposition into finite dimensional weight spaces $V=\oplus_{\mu\in P}V[\mu]$, such that $P(V)$ is contained in the union of sets $\lambda^i-Q_+$ for a finite collection of weights $\lambda^1,...,\lambda^N\in P$ (depending on $V$).\footnote{Usually one also adds the condition that $V$ is a finitely generated $U(\g)$-module, but we don't need this condition here, so we won't impose it.} 
\end{definition}

Here the subscript ``int" indicates that we consider only integral weights (i.e., ones in $P$). However, for brevity we will drop this subscript in this section and just denote this category by $\mathcal O$. 

For example, any highest weight module with integral highest weight belongs to $\mathcal O$. 

Let $\mathcal R$ be the ring of series $a:=\sum_{\mu\in P} a_\mu e^\mu$ ($a_\mu\in \Bbb Z$) such that 
the set $P(a)$ of $\mu$ with $a_\mu\ne 0$ is contained in the union of sets $\lambda^i-Q_+$ for a finite collection of weights $\lambda^1,...,\lambda^N\in P$. Then for every $V\in \mathcal{O}$ we can define the character $\chi_V\in \mathcal R$. Moreover, it is easy to see that 
if 
$$
0\to X\to Y\to Z\to 0
$$ 
is a short exact sequence in $\mathcal O$ then $\chi_Y=\chi_X+\chi_Z$, and that for any $V,U\in \mathcal O$ we have $V\otimes U\in \mathcal O$ and 
$\chi_{V\otimes U}=\chi_V\chi_U$. 

\begin{example} Let $V=M_\lambda$ be the Verma module. Recall that as a vector space 
$M_\lambda=U(\n_-)v_\lambda$, and that $U(\n_-)=\otimes_{\alpha\in R_+} \Bbb C[e_{-\alpha}]$ (using the PBW theorem). 
Thus 
$$
\sum_\mu U(\n_-)[\mu] e^\mu=\frac{1}{\prod_{\alpha\in R_+} (1-e^{-\alpha})}
$$
and hence 
$$
\chi_{M_\lambda}=\frac{e^\lambda}{\prod_{\alpha\in R_+} (1-e^{-\alpha})}.
$$
It is convenient to rewrite this formula as follows: 
$$
\chi_{M_\lambda}=\frac{e^{\lambda+\rho}}{\Delta},\ 
\Delta:=\prod_{\alpha\in R_+} (e^{\alpha/2}-e^{-\alpha/2}).
$$
The (trigonometric) polynomial $\Delta$ is called the {\bf Weyl denominator}. 
\end{example} 

Note that we have a homomorphism $\varepsilon: W\to \Bbb Z/2$ given by 
the formula $w\mapsto \det(w|_\h)$, i.e. $w\mapsto (-1)^{\ell(w)}$; it is defined on simple reflections by 
$s_i\mapsto -1$. This homomorphism is called the {\bf sign character}. For example, for type $A_{n-1}$ this is the sign of a permutation in $S_n$. We will say that an element 
$f\in \Bbb C[P]$ is {\bf anti-invariant} under $W$ if $w(f)=(-1)^{\ell(w)}f$ for all $w\in W$. 
 
 \begin{proposition} 
 The Weyl denominator $\Delta$ is anti-invariant under $W$.  
 \end{proposition} 
 
 \begin{proof} Since $s_i$ permutes positive roots not equal to $\alpha_i$ and sends $\alpha_i$ to $-\alpha_i$, it follows that $s_i\Delta=-\Delta$. 
 \end{proof} 
  
\subsection{The Weyl character formula} 

\begin{theorem} (Weyl character formula) For any $\lambda\in P_+$ the character $\chi_\lambda:=\chi_{L_\lambda}$ of the irreducible finite dimensional representation $L_\lambda$ is given by 
$$
\chi_\lambda=\frac{\sum_{w\in W}(-1)^{\ell(w)}e^{w(\lambda+\rho)}}{\Delta}.
$$
\end{theorem} 

The proof of this theorem is in the next subsection. 

\begin{corollary}\label{Wdf} (Weyl denominator formula) One has 
$$
\Delta=\sum_{w\in W}(-1)^{\ell(w)}e^{w\rho}.
$$
\end{corollary} 

\begin{proof} This follows from the Weyl character formula by setting 
$\lambda=0$ (as $L_0=\Bbb C$ is the trivial representation).  
\end{proof} 

For example, for $\g=\mathfrak{sl}_n$ Corollary \ref{Wdf} reduces to the usual product formula for the Vandermonde determinant. 

\subsection{Proof of the Weyl character formula} 

Consider the product $\Delta \chi_\lambda\in \Bbb Z[P]$. We know that $\chi_\lambda$ is $W$-invariant, so this product is $W$-anti-invariant. Thus, 
$$
\Delta \chi_\lambda=\sum_{\mu\in P} c_\mu e^\mu,
$$
where $c_{w\mu}=(-1)^{\ell(w)}c_\mu$. Moreover, $c_\mu=0$ unless $\mu\in \lambda+\rho-Q_+$, and $c_{\lambda+\rho}=1$. Thus to prove the Weyl character formula, we need to show 
that $c_\mu=0$ if $\mu\in P_+\cap (\lambda+\rho-Q_+)$ and $\mu\ne \lambda+\rho$. 

To this end, we will construct the above decomposition $\Delta \chi_\lambda$ using representation theory, so that this vanishing property is apparent from the construction. 

First recall from Subsection \ref{vani} that we have the Casimir element $C$ of $U(\g)$ given by the formula 
$C=\sum_i a_ia^i$ for a basis $a_i\in \g$ with dual basis $a^i$ of $\g$ under the Killing form. This element is central, so acts by a scalar on every highest weight (in particular, finite dimensional irreducible) representation. 
We can write $C$ in the form 
$$
C=\sum_j x_j^2+\sum_{\alpha\in R_+}\frac{(\alpha,\alpha)}2(e_{-\alpha}e_\alpha+e_\alpha e_{-\alpha}),
$$
for an orthonormal basis $x_j$ of $\h$ and $e_\alpha$ are normalized in such a way that  
$[e_\alpha,e_{-\alpha}]=h_\alpha$. Hence
$$
C=\sum_j x_j^2+\sum_{\alpha\in R_+}(\alpha,\alpha)e_{-\alpha}e_\alpha+\sum_{\alpha\in R_+}\frac{(\alpha,\alpha)}2h_\alpha. 
$$
Thus we get 

\begin{lemma} If $V$ is a highest weight representation with highest weight $\lambda$ then 
$C|_V=(\lambda,\lambda+2\rho)=|\lambda+\rho|^2-|\rho|^2$. 
\end{lemma} 

Now we will define a sequence of modules $K(b)$ from category $\mathcal O$ parametrized by some binary strings $b$. This is done inductively. We set $K(\emptyset)=L_\lambda$. Now suppose $K(b)$ is already defined. If $K(b)=0$ then we set $K(b0)=K(b1)=0$. Otherwise, pick a nonzero vector $v_b\in K(b)$, of some weight $\nu(b)\in \lambda-Q_+$ such that the height of $\lambda-\nu(b)$ takes the minimal possible value. Then $v_b$ is a highest weight vector,  and we can consider the corresponding homomorphism 
$$
\xi_b: M_{\nu_b}\to K(b).
$$ 
Let $K(b1),K(b0)$ be the kernel and cokernel of $\xi_b$. We have 
$$
\chi_{K(b1)}-\chi_{M_{\nu(b)}}+\chi_{K(b)}-\chi_{K(b0)}=0.
$$
Thus we have 
$$
\chi_{K(b)}=\chi_{M_{\nu(b)}}-\chi_{K(b1)}+\chi_{K(b0)}.
$$

Now, it is clear that for every $\mu$, every sufficiently long sequence $b$ satisfies $K(b)[\mu]=0$.  
So iterating this formula starting with $b=\emptyset$, we will get 
\begin{equation}\label{seri}
\chi_\lambda=\sum_b (-1)^{\Sigma(b)}\chi_{M_{\nu(b)}}
\end{equation}
where $\Sigma(b)$ is the sum of digits of $b$ (which could a priori be an infinite sum). So 
$$
\Delta \chi_\lambda=\sum_b (-1)^{\Sigma(b)}e^{\nu(b)+\rho}.
$$
Also note that by induction in the length of $b$ we can conclude that the eigenvalue 
of $C$ on $M_{\nu(b)}$ is $|\lambda+\rho|^2-|\rho|^2$ regardless of $b$, which implies 
that 
$$
|\nu(b)+\rho|^2=|\lambda+\rho|^2
$$
for all $b$; in particular, this shows that the sum \eqref{seri} is finite. 

So it remains to show that if $\mu=\lambda+\rho-\beta \in P_+$ with $\beta\in Q_+$ and 
$\beta\ne 0$ then $|\mu|^2<|\lambda+\rho|^2$. 
Indeed,  
$$
|\lambda+\rho|^2-|\mu|^2=|\lambda+\rho|^2-|\lambda-\beta+\rho|^2=
$$
$$
2(\lambda+\rho,\beta)-|\beta|^2>(\lambda+\rho,\beta)-|\beta|^2=(\lambda+\rho-\beta,\beta)\ge 0.
$$
This completes the proof of the Weyl character formula. 

\begin{exercise} Let $Q$ be the root lattice of a simple Lie algebra $\g$, $Q_+$ its positive part. Define the {\bf Kostant partition function} to be the function $p: Q\to \Bbb Z_{\ge 0}$ 
which attaches to $\beta\in Q_+$ the number of ways to write $\beta$ as a sum of positive roots 
of $\g$ (where the order does not matter), and $p(\beta)=0$ if $\beta\notin Q_+$. 

(i) Show that 
$$
\sum_{\beta\in Q_+}p(\beta)e^{-\beta}=\frac{1}{\prod_{\alpha\in R_+}(1-e^{-\alpha})}.
$$

(ii) Prove the {\bf Kostant multiplicity formula}  
$$
\dim L_\lambda[\gamma]=\sum_{w\in W}(-1)^{\ell(w)}p(w(\lambda+\rho)-\rho-\gamma).
$$

(iii) Compute $p(k_1\alpha_1+k_2\alpha_2)$ for $\g=\mathfrak{sl}_3$ and $\g=\mathfrak{sp}_4$. 

(iv) Use (iii) to compute explicitly the weight multiplicities of the irreducible representations 
$L_\lambda$ for $\g=\mathfrak{sl}_3$ and $\g=\mathfrak{sp}_4$. (You should get a sum of 6, respectively 8 terms, not particularly appealing, but easily computable in each special case). 
\end{exercise} 

\subsection{The Weyl dimension formula} 

Recall that the Weyl character formula can be written as a trace formula: for $h\in \h$
$$
\chi_{\lambda}(e^h)=\Tr|_{L_\lambda}(e^h)=\frac{\sum_{w\in W}(-1)^{\ell(w)}e^{(w(\lambda+\rho),h)}}{\prod_{\alpha\in R_+}(e^{\frac{1}{2}(\alpha,h)}-e^{-\frac{1}{2}(\alpha,h)})}. 
$$
The dimension of $L_\lambda$ should be obtained from this formula
when $h=0$. However, we do not immediately get the answer since 
this formula gives the character as a ratio of two trigonometric polynomials which both vanish at $h=0$, giving an indeterminacy. We know the limit exists since the character is a trigonometric polynomial, but 
we need to compute it. This can be done as follows. 

Let us restrict attention to $h=2t\rho^\vee$ where $t\in \Bbb R$. We have 
$$
\chi_{\lambda}(e^{2t\rho^\vee})=\frac{\sum_{w\in W}(-1)^{\ell(w)}e^{2t(w(\lambda+\rho),\rho^\vee)}}{\prod_{\alpha\in R_+}(e^{t(\alpha,\rho^\vee)}-e^{-t(\alpha,\rho^\vee)})}. 
$$
The key idea is that for this specialization the numerator can also be factored using the denominator formula for the dual root system, since $(w(\lambda+\rho),\rho^\vee)=(\lambda+\rho,w^{-1}\rho^\vee)$. Setting $\lambda=0$ gives the same factorization of the denominator. Thus we have 
\begin{equation}\label{princ}
\chi_{L_\lambda}(e^{2t\rho^\vee})=\frac{\prod_{\alpha\in R_+}(e^{t(\alpha^\vee,\lambda+\rho)}-e^{-t(\alpha^\vee,\lambda+\rho)})}
{\prod_{\alpha\in R_+}(e^{t(\alpha^\vee,\rho)}-e^{-t(\alpha^\vee,\rho)})}. 
\end{equation} 
Now sending $t\to 0$, we obtain 

\begin{proposition} We have 
$$
\dim L_\lambda=\frac{\prod_{\alpha\in R_+}(\alpha^\vee,\lambda+\rho)}{\prod_{\alpha\in R_+}(\alpha^\vee,\rho)}.
$$
\end{proposition} 

Note that this number is an integer, but this is not obvious without its interpretation as the dimension of a representation. 

Formula \eqref{princ} has a meaning even before taking the limit. Namely, the eigenvalues of the element $2\rho^\vee$ define a $\Bbb Z$-grading on the representation $L_\lambda$ called the {\bf principal grading}, and we obtain a product formula for the Poincar\'e polynomial of this grading.  

\newpage

\centerline{\Large\bf Lie groups and Lie algebras II}

\section{\bf Representations of $GL_n$, I} 

We begin with a more detailed study of finite dimensional representations of semisimple Lie algebras and the corresponding complex Lie groups. 

\subsection{Tensor products of fundamental representations}

The following result shows that if we understand fundamental representations of a semisimple Lie algebra $\g$ (i.e., irreducible representations with fundamental highest weights $\omega_i$), we can gain some insight into general finite dimensional representations. 

\begin{proposition}\label{funda} Let $\lambda=\sum_{i=1}^r m_i\omega_i$ be a dominant integral weight for $\g$. 
Consider the tensor product $T_\lambda:=\otimes_i L_{\omega_i}^{\otimes m_i}$, and let 
$v:=\otimes_i v_{\omega_i}^{\otimes m_i}$ be the tensor product of the highest weight vectors. 
Let $V$ be the subrepresentation of $T_\lambda$ generated by $v$. Then $V\cong L_\lambda$. 
\end{proposition} 

\begin{proof} We have $V=L_\lambda\oplus \bigoplus_{\mu\in (\lambda-Q_+)\cap P_+}N_{\lambda\mu}L_\mu$ where $N_{\lambda\mu}$ are nonnegative integers. Let $C\in U(\g)$ be the Casimir element for  $\g$. Recall that $C|_{L_\mu}=(\mu,\mu+2\rho)$. Thus 
$C|_V=(\lambda,\lambda+2\rho)$. But we have seen in the proof of the Weyl character formula that for any $\mu\in (\lambda-Q_+)\cap P_+$ such that $\mu\ne \lambda$, we have $(\mu,\mu+2\rho)<(\lambda,\lambda+2\rho)$. Therefore we see that $N_{\lambda\mu}=0$ for $\mu\ne \lambda$. 
\end{proof} 

\subsection{Representations of $SL_n(\Bbb C)$} 

Let us now discuss more explicitly the representation theory of $SL_n(\Bbb C)$. We will consider its finite dimensional complex analytic representations as a complex Lie group. We have shown that this is equivalent to considering finite dimensional representations of the Lie algebra $\mathfrak{sl}_n(\Bbb C)$. We have also seen that these are completely reducible and 
the irreducible representations are $L_\lambda$, where $\lambda=\sum_{i=1}^{n-1}m_i\omega_i$, $\omega_i$ are the fundamental weights, and $m_i\in \Bbb Z_{\ge 0}$.  

First let us compute $\omega_i$. Recall that the standard Cartan subalgebra $\h$ is the space $\Bbb C_0^n$ of vectors in $\Bbb C^n$ with zero sum of coordinates (diagonal matrices with trace zero). So elements of $\h^*$ can be viewed as vectors $(x_1,...,x_n)\in \Bbb C^n$ modulo simultaneous shift of all coordinates by the same number (i.e., $\h^*=\Bbb C^n/\Bbb C_{\rm diag}$).  

Recall that the simple coroots are $\alpha_i^\vee=\bold e_i-\bold e_{i+1}$. Thus $\omega_i$ are determined by the conditions 
$$
(\omega_i,\bold e_j-\bold e_{j+1})=\delta_{ij}.  
$$
This means that $\omega_i=(1,...,1,0,...,0)$ where 
there are $i$ copies of $1$. Thus a dominant integral weight $\lambda$ has the form 
$$
\lambda=(m_1+...+m_{n-1},m_2+...+m_{n-1},....,m_{n-1},0). 
$$
So dominant integral weights are parametrized by non-increasing sequences 
$\lambda_1\ge...\ge \lambda_{n-1}$ of nonnegative integers. This agrees with the representation theory of $SL_2(\Bbb C)$ that we worked out before: in this case the sequence has just one term.  

Let us now describe explicitly the fundamental representations $L_{\omega_i}$. 
Consider first the representation $V=\Bbb C^n$ with the usual action of matrices. 
It is called the {\bf vector representation} or the {\bf tautological representation} (as every matrix goes to itself). 
It is irreducible and has a standard basis $v_1,...,v_n$. To find its highest weight, we have to find a vector 
$v\ne 0$ such that $e_iv=0$. As $e_i=E_{i,i+1}$, we have $v=v_1$. It is easy to see that $hv=\omega_1(h)v$, so we see that $v$ has weight $\omega_1$, hence $L_{\omega_1}=V$. 

To construct $L_{\omega_m}$ for $m>1$, consider the exterior power $\wedge^m V$. 
It is easy to show that it is irreducible. A basis of $\wedge^m V$ 
consists of wedges $v_{i_1}\wedge...\wedge v_{i_m}$ where $i_1<...<i_m$. 
The highest weight vector is clearly $v_1\wedge...\wedge v_m$, and it has weight $\omega_m$. 
Thus $L_{\omega_m}=\wedge^m V$. 

Note that $\wedge^n V=\Bbb C$ (the trivial representation) since every matrix in $SL_n(\Bbb C)$ acts by its determinant, which is $1$, and $\wedge^mV=0$ for $m>n$. Also $V^*\cong \wedge^{n-1}V$ since the wedge pairing $V\otimes \wedge^{n-1}V\to \wedge^nV=\Bbb C$ is invariant and nondegenerate. Similarly, 
$\wedge^mV^*\cong \wedge^{n-m}V$. 

We now see from Proposition \ref{funda} that the irreducible representation $L_\lambda$ for $\lambda=\sum_i m_i\omega_i$ is generated inside $\otimes_{i=1}^{n-1} (\wedge^i V)^{\otimes m_i}$ by the tensor product of the highest weight vectors. 

\begin{example} $L_{N\omega_1}=S^NV$, generated by the vector $v_1^{\otimes N}\in V^{\otimes N}$. 
\end{example} 

\subsection{Representations of $GL_n(\Bbb C)$} 

Let us now explain how to extend these results to $GL_n(\Bbb C)$. This is easy to do since $GL_n(\Bbb C)$ is not very different from the direct product $\Bbb C^\times\times SL_n(\Bbb C)$. Namely, 
$GL_n(\Bbb C)=(\Bbb C^\times\times SL_n(\Bbb C))/\mu_n$ where $\mu_n$ is the group of roots of unity of order $n$ embedded as $z\mapsto (z^{-1},z\bold 1_n)$. Indeed, the corresponding covering homomorphism
$\Bbb C^\times\times SL_n(\Bbb C)\to GL_n(\Bbb C)$ is given by $(z,A)\mapsto zA$. 
So it suffices to classify irreducible holomorphic representations 
of the complex Lie group $\Bbb C^\times\times SL_n(\Bbb C)$; the irreducible holomorphic representations of $GL_n(\Bbb C)$ are a subset of them. 

For $n=1$ this is just the problem of describing the holomorphic representations of $\Bbb C^\times$. 
This is easy. The Lie algebra is spanned by a single element $h$ such that $e^{2\pi ih}=1$. 
This element must act in a representation by an operator $H$ such that $e^{2\pi iH}=1$. 
It follows that $H$ is diagonalizable with integer eigenvalues. Thus representations of $\Bbb C^\times$ are completely reducible, with irreducibles $\chi_N$ one-dimensional and labeled by integers $N\in \Bbb Z$, $\chi_N(z)=z^N$.  

The same argument leads to a similar answer for $\Bbb C^\times \times SL_n$: representations are completely reducible with irreducibles being $L_{\lambda,N}=\chi_N\otimes L_\lambda$.  Moreover, the ones factoring through $GL_n$ just have $N=nr+\sum_{i=1}^{n-1} \lambda_i$ for some integer $r$. 

Recall that $GL_n$ has reductive Lie algebra $\mathfrak{gl}_n$ with Cartan subalgebra $\h=\Bbb C^n$. 
The highest weight of $L_{\lambda,nm_n+\sum_{i=1}^{n-1} \lambda_i}$ is easily computed and equals 
$(m_1+...+m_{n-1}+m_n,...,m_{n-1}+m_n,m_n)$. Thus highest weights of finite dimensional representations 
are non-increasing sequences $(\lambda_1,...,\lambda_n)$ of integers which don't have to be positive. 
The fundamental representations are still $L_{\omega_m}=\wedge^m V$, and the only difference with $SL_n$ is that now the top exterior power $\wedge^n V$ is not trivial but rather is a 1-dimensional {\bf determinant character} with highest weight $\omega_n=(1,...,1)$. The highest weight of a finite dimensional representation then has the form $\lambda=\sum_{i=1}^n m_i\omega_i$, where $m_i\ge 0$ for $i\ne n$, while $m_n$ is an arbitrary integer. Consequently, $L_\lambda$ is found inside $\otimes_{i=1}^n (\wedge^i V)^{\otimes m_i}$ 
as the representation generated by the product of highest weight vectors. Note that it makes sense to take $m_n<0$, as for a one-dimensional representation and $k<0$ it is natural to define $\chi^{\otimes k}:=(\chi^*)^{\otimes -k}$. 

The representations with $m_n\ge 0$ are especially important; it is easy to see that these are exactly the ones that occur inside $V^{\otimes N}$ for some $N$ (check it!). These representations are called {\bf polynomial} since their matrix coefficients are polynomial functions of the matrix entries $x_{ij}$ of $X\in GL_n(\Bbb C)$, and consequently they extend by continuity to representations of the semigroup ${\rm Mat}_n(\Bbb C)\supset GL_n(\Bbb C)$. Note that any irreducible representation is a polynomial one tensored with a non-positive power of the determinant character $\wedge^nV$. 

\subsection{Schur-Weyl duality} 
Note that highest weights of polynomial representations are non-increasing sequences of nonnegative integers $(\lambda_1,...,\lambda_n)$, i.e. {\bf partitions} with $\le n$ parts. Namely, they are partitions of $|\lambda|=\sum_i \lambda_i$, which is just the eigenvalue of 
$\bold 1_n\in {\mathfrak{gl}}_n$ on $L_\lambda$ and can also be defined as the number $N$ such that 
$L_\lambda$ occurs in $V^{\otimes N}$. 

Traditionally partitions are encoded by {\bf Young diagrams.} Namely, the Young diagram of a partition $\lambda=(\lambda_1,...,\lambda_n)$ 
consists of $n$ rows of boxes, the $i$-th row consisting of $\lambda_i$ boxes, so that row $i$ is placed directly under row $i-1$ and all rows start on the same vertical line. For example, here are the Young diagrams of the partitions $(4,3,2)$ (left) and $(3,3,2,1)$ (right): 

\includegraphics[scale=0.7]{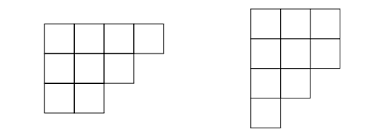}

Thus we have 
$$
V^{\otimes N}=\oplus_{\lambda: |\lambda|=N} L_\lambda\otimes \pi_\lambda, 
$$
where $\pi_\lambda:=\Hom_{GL_n(\Bbb C)}(L_\lambda,V^{\otimes N})$ are multiplicity spaces. Here the summation is over partitions of $N$, and $L_\lambda=0$ if $\lambda$ has more than $n$ parts. To understand the spaces $\pi_\lambda$, note that 
the symmetric group $S_N$ acts on $V^{\otimes N}$ and commutes with $GL_n(\Bbb C)$, so it gets to act on each $\pi_\lambda$. 

Let $A$ be the image of $U(\mathfrak{gl}_n)$ in ${\rm End}_{\Bbb C}(V^{\otimes N})$, and 
$B$ be the image there of $\Bbb CS_N$. The algebras $A,B$ commute.    

\begin{theorem} (Schur-Weyl duality) (i) The centralizer of $A$ is $B$ and vice versa. 

(ii) If $\lambda$ has at most $n$ parts then the representation $\pi_\lambda$ of $B$ (hence $S_N$) is irreducible, and such representations are pairwise non-isomorphic.  

(iii) If $\dim V\ge N$ then $\pi_\lambda$  exhaust all irreducible representations of $S_N$.
\end{theorem}  

\begin{proof} We start with 

\begin{lemma}\label{span} If $U$ is a $\Bbb C$-vector space then 
$S^NU$ is spanned by elements $x\otimes...\otimes x$, $x\in U$. 
\end{lemma} 

\begin{proof} It suffices to consider the case when $U$ is finite dimensional. 
Then the span of these vectors is a nonzero subrepresentation in the irreducible $GL(U)$-representation $S^NU$, which implies the statement.
\end{proof} 

\begin{lemma}\label{gene}  For any associative algebra $R$ over $\Bbb C$, the algebra $S^NR:=(R^{\otimes N})^{S_N}$ 
is generated by elements 
$$
\Delta_N(x):=x\otimes 1\otimes...\otimes 1+1\otimes x\otimes...\otimes 1+...+1\otimes...\otimes 1\otimes x
$$
for $x\in R$. 
\end{lemma} 
 
\begin{proof} Let $P_N$ be the Newton polynomial expressing $z_1...z_N$ via $p_k:=\sum_{i=1}^N z_i^k$, $k=1,...,N$ (it exists and is unique by the fundamental theorem on symmetric functions). 
Then we have 
$$
x\otimes...\otimes x=P_N(\Delta_N(x),...,\Delta_N(x^N)).
$$ 
Hence the lemma follows from Lemma \ref{span}.  
\end{proof} 

Let us now show that $A$ is the centralizer $Z_B$ of $B$. Note that $Z_B=S^N(\End V)$. 
Thus the statement follows from Lemma \ref{gene}. 

We will now use the following easy but important lemma (which actually holds over any field). 

\begin{lemma}\label{dc} (Double centralizer lemma) Let $V$ be a finite dimensional vector space and $A,B\subset {\rm End}V$ be subalgebras such that $B$ is isomorphic to a direct sum of matrix algebras and $A$ is the centralizer of $B$. Then $A$ is also isomorphic to a direct sum of matrix algebras, 
and moreover 
$$
V=\oplus_{i=1}^n W_i\otimes U_i,
$$ 
where $W_i$ run through all irreducible $A$-modules and $U_i$ through irreducible $B$-modules. In particular, $B$ is the centralizer of $A$ and we have a natural bijection between irreducible $A$-modules and irreducible $B$-modules which matches $W_i$ and $U_i$.    
\end{lemma} 

\begin{proof} We have $V=\oplus_{i=1}^n W_i\otimes U_i$ 
where $U_i$ run through irreducible representations of $B$ 
and $W_i=\Hom_B(U_i,V)\ne 0$ are multiplicity spaces. Thus 
$A=\oplus_{i=1}^n \End W_i$ and $B=\oplus_{i=1}^n\End U_i$, which 
implies the statement. 
\end{proof} 

Since the algebra $B$ is a direct sum of matrix algebras (by complete reducibility of representations of finite groups), Lemma \ref{dc} yields (i).\footnote{This also gives another proof of the fact that $A$ is a direct sum of matrix algebras, i.e. complete reducibility of $V^{\otimes N}$.}

To prove (ii), it suffices to note that if $\lambda$ has $\le n$ parts then $L_\lambda$ occurs 
in $V^{\otimes N}$, so $\pi_\lambda\ne 0$. The rest follows from (i) and Lemma \ref{dc}. 

(iii) If $\dim V\ge N$ then pick $N$ linearly independent vectors $v_1,...,v_N\in V$. 
It is easy to see that the map $\Bbb CS_N\to V^{\otimes N}$ defined by 
$s\mapsto s(v_1\otimes...\otimes v_N)$ is injective. Thus $B=\Bbb CS_N$. 
This implies the statement. 
\end{proof} 

\begin{remark} The algebra $A$ is called the {\bf Schur algebra} and $B$ the {\bf centralizer algebra}. 
\end{remark}

Thus we see that representations of $S_N$ are labeled by partitions $\lambda$ of $N$, and those that occur in $V^{\otimes N}$ correspond to the partitions that have $\le \dim V$ parts. Moreover, we claim that this labeling of representations by partitions does not depend on $\dim V$. To show this, suppose $\lambda$ has $\le n$ parts and 
$V=\Bbb C^n$. We have the Schur-Weyl decomposition of $GL_{n+1}(\Bbb C)\times S_N$-modules
$$
(V\oplus \Bbb C)^{\otimes N}=\oplus_\mu L_\mu^{(n+1)}\otimes \pi_\mu^{(n+1)},
$$
Let us restrict this sum to $GL_n(\Bbb C)\times S_N$, and consider what happens to the summand $L_\lambda^{(n+1)}\otimes \pi_\lambda^{(n+1)}$. The highest weight vector $v$ in $L_\lambda^{(n+1)}$ tensored with any element $w$ of $\pi_\lambda^{(n+1)}$ sits in $V^{\otimes N}\subset (V\oplus \Bbb C)^{\otimes N}$, since the $n+1$-th component of its weight is  zero. Hence $v\otimes w$ generates a copy of $L_\lambda^{(n)}\otimes \pi_\lambda^{(n)}$ as a $GL_n(\Bbb C)\times S_N$-module. This implies that $\pi_\lambda^{(n+1)}\cong \pi_\lambda^{(n)}$.   

\begin{exercise} Let $R=\Bbb C[x_1,...,x_N,y_1,...,y_N]^{S_N}$ (the algebra of invariant polynomials). Show that $R$ is generated by the elements 
$Q_{rs}:=\sum_{i=1}^N x_i^ry_i^s$ where $r,s\ge 0$, $1\le r+s\le N$. 
\end{exercise} 

\begin{exercise}\label{contentex} Let $\lambda=(\lambda_1,...,\lambda_n)$ be a partition. 
Let us fill the Young diagram of $\lambda$ with numbers, placing $c(i,j):=j-i$ in the $j$-th box in the $i$-th row. Thus the number written in each box depends only on its position $(i,j)$; it is called the {\bf content} of this box. The {\bf content of $\lambda$} is the sum $c(\lambda)$ of contents of all its boxes: 
$$
c(\lambda)=\sum_{(i,j)\in \lambda}c(i,j).
$$ 

(i) Show that 
$$
c(\lambda)=\sum_{i=1}^n\frac{\lambda_i(\lambda_i-2i+1)}{2}.
$$

(ii) Let $\bold c=\sum_{1\le i<j\le N}(ij)\in \Bbb CS_N$ 
be the sum of all transpositions. 
Show that $\bold c$ is a central element of $\Bbb CS_N$ which acts 
on the irreducible representation $\pi_\lambda$ of $S_N$ 
by the scalar $c(\lambda)$. ({\bf Hint:} Consider the action 
of $\bold c$ on $V^{\otimes N}$ and use Schur-Weyl duality 
to relate it to the diagonal action of the quadratic Casimir of $\mathfrak{gl}_n$). 
\end{exercise} 

\section{\bf Representations of $GL_n$, II}

\subsection{Schur functors} 

\begin{definition} For a partition $\lambda$ of $N$ we define the {\bf Schur functor} $S^\lambda$ on the category of complex vector spaces (or complex representations of any group or Lie algebra) by $S^\lambda V=\Hom_{S_N}(\pi_\lambda,V^{\otimes N})$. 
\end{definition} 

Thus we have 
$$
V^{\otimes N}=\oplus_{\lambda}S^\lambda V\otimes \pi_\lambda,
$$
and if $\lambda$ has $\le n$ parts and $V=\Bbb C^n$ then $S^\lambda V=L_\lambda$ as a representation of $GL(V)=GL_n(\Bbb C)$. 

\begin{example} 1. We have $S^{(n)}V=S^nV$, $S^{(1^n)}V=\wedge^nV$.  

2. We have 
$$
V\otimes V= S^{(2)}V\otimes \Bbb C_+\oplus  S^{(1,1)}V\otimes \Bbb C_{-}=S^2V\oplus \wedge^2V
$$
where $S_2$ acts in the first summand trivially and in the second one by sign. 

Consider now the decomposition of $V\otimes V\otimes V$. We have 
$$
V\otimes V\otimes V= S^{(3)}V\otimes\Bbb C_+\oplus  S^{(2,1)}V\otimes\Bbb C^2\oplus S^{(1,1,1)}V\otimes\Bbb C_- 
$$
$$
=S^3V\oplus S^{(2,1)}V\otimes\Bbb C^2\oplus \wedge^3V.
$$
Thus 
$$
S^2V\otimes V=S^3V\oplus S^{(2,1)}V,\ \wedge^2 V\otimes V=\wedge^3V\oplus S^{(2,1)}V.
$$
We conclude that $S^{(2,1)}V$ can be described as the space of tensors symmetric in the first two components whose full symmetrization is zero, or tensors antisymmetric 
on the first two components whose full antisymmetrization is zero. 
\end{example} 

\begin{exercise} 1. Let $V=\Bbb C^n$, $n\ge 4$. Decompose $V^{\otimes 4}$ 
as a direct sum of irreducible representations of $GL_n(\Bbb C)\times S_4$. 
Characterize the occurring Schur functors as spaces of tensors with certain symmetry properties, similarly to the above description of $S^{(2,1)}V$. 
Compute the decompositions of $V\otimes S^3V$, $V\otimes \wedge^3V$, 
$S^2V\otimes S^2V, S^2V\otimes \wedge^2V$ 
and $\wedge^2V\otimes \wedge^2V$ into Schur functors. 

2. Decompose $V\otimes V^*$, $V\otimes V\otimes V^*$ into a direct sum of irreducible representations. Describe the algebra ${\rm End}_{GL_n(\Bbb C)}(V\otimes V^*\otimes V^*)$. 
\end{exercise} 

Let us compute the dimension of $S^\lambda V$ when $\dim V=N$ and $\lambda$ has $k$ parts. We have $\rho=(N-1,N-2,...,1,0)$ 
(for $SL_N$), so the Weyl dimension formula tells us that 
$$
\dim S^\lambda V=\prod_{1\le i<j\le N}\frac{\lambda_i-\lambda_j+j-i}{j-i}=
$$
$$
\prod_{1\le i<j\le k}\frac{\lambda_i-\lambda_j+j-i}{j-i}\prod_{1\le i\le k<j\le N}\frac{\lambda_i+j-i}{j-i}=
$$
$$
\prod_{1\le i<j\le k}\frac{\lambda_i-\lambda_j+j-i}{j-i}\prod_{i=1}^k \frac{(N+1-i)...(N+\lambda_i-i)}{(k+1-i)...(k+\lambda_i-i)}.
$$
We obtain 

\begin{proposition} $\dim S^\lambda V=P_\lambda(N)$
where $P_\lambda$ is a polynomial of degree $|\lambda|$ with rational coefficients 
and integer roots. Moreover, the roots of $P_\lambda$ are all the integers in the interval $[1-\lambda_1,k-1]$ (occurring with multiplicities). 
\end{proposition} 

Moreover, we see that $P_\lambda(N)$ is an integer-valued polynomial, i.e., it takes integer values at integer points (this is equivalent to being an integer linear combination of $\binom{N}{j}$). 

\begin{example} 
$$
P_{(n)}(N)=\dim S^nV=\binom{N+n-1}{n},\ P_{(1^n)}(N)=\dim \wedge^n V=\binom{N}{n}.
$$
Also
$$
P_{(a,b)}(N)=(a-b+1)\frac{N...(N+a-1)\cdot (N-1)...(N+b-2)}{(a+1)!b!}
=
$$
$$
\frac{a-b+1}{a+1}\binom{N+a-1}{a}\binom{N+b-2}{b}
$$
E.g., $P_{(2,1)}(N)=\dim S^{(2,1)}V=\frac{N(N+1)(N-1)}{3}$.
Also, 
$$
P_{(a,a)}(N)=\frac{1}{a+1}\binom{N+a-1}{a}\binom{N+a-2}{a}=
$$
$$
\frac{1}{N+a-1}\binom{N+a-1}{N-1}\binom{N+a-2}{N-2}=
{\rm Nar}(N+a-1,N-1),
$$
the {\bf Narayana numbers}. 
\end{example}

\begin{exercise} Let $g_q$ be the diagonal matrix with diagonal elements $1,q,q^2,...,q^{n-1}$. 
Compute the trace of $g_q$ in $S^\lambda V$ in the product form. Write the answer explicitly (as a polynomial in $q$) with positive coefficients in the case $|\lambda|\le 3$.   
\end{exercise} 

\begin{exercise} Draw the weights of the representation $S^{(2,2)}\Bbb C^3$ of 
$SL(3)$ on the hexagonal lattice, and indicate their multiplicities.   
\end{exercise} 

\subsection{The fundamental theorem of invariant theory}

Suppose we have a finite dimensional vector space $V$ and a collection of tensors 
$T_i\in V^{\otimes m_i}\otimes V^{*\otimes n_i}$, $i=1,...,k$. 
An important problem is to describe ``coordinate free" invariants of such a collection of tensors, i.e., polynomial functions 
$F(T_1,...,T_k)$ which are invariant under the action of 
$GL(V)$. How can we classify such functions? This sounds formidably hard in such generality, but turns out to be very easy using Schur-Weyl duality. 

It suffices to study such functions that have homogeneity degree $d_i$ with respect to each $T_i$. To do so, we will depict each $T_i$ by a vertex with $m_i$ incoming and $n_i$ outgoing arrows. We should think of incoming arrows as $V$-components and outgoing ones as $V^*$-components. Let us draw $d_i$ such vertices for each $i$. To construct an invariant, let us connect the arrows preserving orientation so that all the arrows are used (this will only be possible if the number of incoming arrows equals the number of outgoing ones; otherwise every invariant of the multidegree $(d_1,...,d_k)$ will be zero). To the obtained graph $\Gamma$ we can assign the {\bf convolution} of tensors, which gives an invariant function $F_\Gamma$ of the correct multidegree. 

\begin{theorem}\label{ftit} The functions $F_\Gamma$ for various $\Gamma$ span the space of invariant functions. 
\end{theorem} 
 
\begin{proof} An invariant function may be viewed as an invariant element 
of the space $\bigotimes_{i=1}^k (V^{*\otimes m_i}\otimes V^{\otimes n_i})^{\otimes d_i}$, which we may write as the space of linear maps $V^{\otimes M}\to V^{\otimes N}$, where 
$M=\sum d_im_i$ is the number of incoming arrows and $N=\sum d_in_i$ the number of outgoing arrows. 
If $M\ne N$, there are no nonzero invariant maps. Otherwise, by the Schur-Weyl duality, the space of such maps is spanned by maps defined by permutations. But any such permutation defines a graph $\Gamma$, so the corresponding invariant is just the convolution $F_\Gamma$, which implies the statement.  
\end{proof} 

\begin{remark}\label{indepe} 
Note that this proof also implies that if 
$$
\dim V\ge N=\sum_i m_id_i=\sum_i n_id_i,
$$ 
then the functions $F_\Gamma$ for non-isomorphic graphs $\Gamma$ with $N$ edges are linearly independent, so they form a basis in the degree $N$ part $A_N$ of the algebra $A$ of invariant functions. 
(Here the vertices of $\Gamma$ are colored by $k$ colors corresponding to the types of tensors, and at every vertex of color $i$ the outgoing edges are labeled by $[1,n_i]$ and incoming edges by $[1,m_i]$. Isomorphisms are required to preserve these colorings and labelings). 
\end{remark} 

\begin{example}\label{uniqness} Assume that $m_i=n_i=1$, i.e., $T_1,...,T_k$ are just matrices with $GL_n$ acting by conjugation. Then all graphs that we can get are unions of cycles, so Theorem \ref{ftit} implies that the algebra $A_{k,n}$ of such invariants (where $n=\dim V$) is generated by traces of cyclic words 
$$
F_{j_1,...,j_r}={\rm Tr}(T_{j_1}...T_{j_r})
$$
(here ``cyclic" means that words differing by a cyclic permutation are considered to be the same). Moreover, by Remark \ref{indepe}, these elements are ``asymptotically algebraically independent", i.e. there is no nonzero polynomial of them that vanishes for all sizes of matrices $n$. 

This implies that there are no universal polynomial identities for matrices of all sizes. Indeed, if $P(T_1,...,T_k)=0$ for square matrices $T_1,...,T_k$ of any size $n$ (where $P$ is a fixed nonzero noncommutative polynomial) then adding another matrix $T_{k+1}$, we get 
$$
{\rm Tr}(P(T_1,...,T_k)T_{k+1})=0,
$$ 
which contradicts linear independence 
of $F_{j_1,...,j_r}$.   

In particular, this implies that the universal Lie polynomials $\mu_n(x,y)$ 
of degree $n$ occurring in the Baker-Campbell-Hausdorff formula, i.e., such that 
$$
\log(\exp(x)\exp(y))\sim \sum_{m\ge 1}\frac{\mu_m(x,y)}{m!}
$$
for $x\in {\rm Lie}(G)$ for any Lie group $G$, are unique 
(in fact, they are already unique for the family of groups $GL_n(\Bbb C)$ 
for all $n$). 
 
This is false, however, if the size of matrices is fixed; in this case there are plenty of polynomial identities for each matrix size. For example, for matrices of size $1$ we have $[X,Y]=0$ and for matrices of size $2$ we have $[Z,[X,Y]^2]=0$. For general $n$ there is the Amitsur-Levitzki identity given in Exercise \ref{ali}. 
\end{example} 

\begin{exercise}\label{ali} Let $X_1,...,X_{2n}$ be complex $n$ by $n$ matrices. 
Let $\Lambda=\wedge(\xi_1,...,\xi_{2n})$ be the exterior algebra generated by $\xi_i$ with relations $\xi_i\xi_j=-\xi_j\xi_i,\xi_i^2=0$. Let $X$ be the matrix over $\Lambda$ 
given by 
$$
X:=X_1\xi_1+...+X_{2n}\xi_{2n}.
$$

(i) Let $Y=X^2$. Show that $Y\in {\rm Mat}_n(\Lambda_+)$ 
where $\Lambda_+$ is the commutative subalgebra of $\Lambda$ spanned by 
the elements of even degrees. Compute $Y^n$. 

(ii) Show that ${\rm Tr}(Y^k)=0\in \Lambda_+$ for $k=1,...,n$. 

(iii) Deduce that $Y^n=0$. This should yield the {\bf Amitsur-Levitzki identity} 
$$
\sum_{\sigma\in S_{2n}}{\rm sign}(\sigma)X_{\sigma(1)}...X_{\sigma(2n)}=0.
$$

(iv) Deduce the same identity over any commutative ring $R$. 
\end{exercise}

\section{\bf Representations of $GL_n$, III}

\subsection{Schur polynomials and characters of representations of the symmetric group} 

Using Schur-Weyl duality and the character formula for representations of $GL_n$, we can obtain information about characters of the symmetric group. Namely, 
it follows from the Weyl character formula that the characters of representations of $GL_n$ are given by the formula 
$$
s_\lambda(x_1,...,x_n)=\frac{\sum_{\sigma\in S_n}{\rm sign}(\sigma)x_{\sigma(1)}^{\lambda_1+n-1}...x_{\sigma(n)}^{\lambda_n}}{\prod_{i<j}(x_i-x_j)}=\frac{\det(x_i^{\lambda_j+n-j})}{\prod_{i<j}(x_i-x_j)}.
$$
These symmetric polynomials are called {\bf Schur polynomials}. 
For example, the character of $S^mV$ is 
$$
s_{(m)}(x_1,...,x_n)=\sum_{1\le j_1\le ...\le j_m\le n}x_{j_1}...x_{j_m}=h_m(x_1,...,x_n),
$$
the $m$-th {\bf complete symmetric function}, and the character of $\wedge^mV$ is  
$$
s_{(1^m)}(x_1,...,x_n)=\sum_{1\le j_1<...<j_m\le n}x_{j_1}...x_{j_m}=e_m(x_1,...,x_n),
$$
the $m$-th {\bf elementary symmetric function}. 

Let us now compute the trace in $V^{\otimes N}$ of $x^{\otimes N}\sigma$, where $x={\rm diag}(x_1,...,x_n)$ is a diagonal matrix and $\sigma\in S_N$ a permutation. 
Let $\sigma$ have $m_i$ cycles of length $i$. Then we have 
$$
{\rm Tr}|_{V^{\otimes N}}(x^{\otimes N}\sigma)=\prod_i (x_1^i+...+x_n^i)^{m_i}.
$$
On the other hand, using Schur-Weyl duality, we get 
$$
{\rm Tr}|_{V^{\otimes N}}(x^{\otimes N}\sigma)=\sum_\lambda \chi_\lambda(\sigma)s_\lambda(x),
$$
where $\chi_\lambda(\sigma)={\rm Tr}|_{\pi_\lambda}(\sigma)$ is the character of the representation $\pi_\lambda$ of $S_N$. Thus we have 
$$
\sum_\lambda \chi_\lambda(\sigma)s_\lambda(x)=\prod_i (x_1^i+...+x_n^i)^{m_i}.
$$
Multiplying this by the discriminant, we get 
$$
\sum_\lambda \chi_\lambda(\sigma)\det(x_i^{\lambda_j+n-j})=\prod_{i<j}(x_i-x_j)\cdot \prod_i (x_1^i+...+x_n^i)^{m_i}.
$$
Thus we get 

\begin{theorem}(Frobenius character formula) 
The character value $\chi_\lambda(\sigma)$ is 
the coefficient of $x_1^{\lambda_1+n-1}...x_n^{\lambda_n}$ 
in the polynomial 
$$
\prod_{i<j}(x_i-x_j)\cdot \prod_i (x_1^i+...+x_n^i)^{m_i}.
$$
\end{theorem} 

\begin{exercise} Let $V=\Bbb C^2$ be the 2-dimensional tautological representation
of $GL_2(\Bbb C)$. Decompose $V^{\otimes N}$ into a direct sum of irreducible representations of $GL_2(\Bbb C)\times S_N$ and compute the characters and dimensions 
of all the irreducible representations of $GL_2$ and $S_N$ that occur. 
\end{exercise} 

\subsection{Howe duality} \label{howd}

Howe duality is another instance when we have a double centralizer property. Consider two finite dimensional complex vector spaces $V,W$, and consider the symmetric power $S^N(V\otimes W)$ as a representation of 
$GL(V)\times GL(W)$. 

\begin{theorem} (Howe duality) We have a decomposition 
$$
S^N(V\otimes W)=\oplus_{\lambda: |\lambda|=N}S^\lambda V\otimes S^\lambda W.
$$ 
\end{theorem} 

Note that if $\lambda$ has more parts than $\dim V$ or $\dim W$ then the corresponding summand is zero. 

\begin{proof} We have 
$$
S^N(V\otimes W)=((V\otimes W)^{\otimes N})^{S_N}=
(V^{\otimes N}\otimes W^{\otimes N})^{S_N}
$$
So using the Schur-Weyl duality, we get 
$$
S^N(V\otimes W)=((\oplus_{\lambda: |\lambda|=N}S^\lambda V\otimes \pi_\lambda) \otimes (\oplus_{\mu: |\mu|=N}S^\mu W\otimes \pi_\mu))^{S_N}=
$$
$$
\oplus_{\lambda,\mu: |\lambda|=|\mu|=N}S^\lambda V\otimes S^\mu W\otimes (\pi_\lambda\otimes \pi_\mu)^{S_N}. 
$$
But the character of $\pi_\lambda$ is integer-valued, so $\pi_\lambda\cong \pi_\lambda^*$. 
Thus by Schur's lemma $(\pi_\lambda\otimes \pi_\mu)^{S_N}=\Bbb C^{\delta_{\lambda\mu}}$, and we get 
$$
S^N(V\otimes W)=\oplus_{\lambda: |\lambda|=N}
S^\lambda V\otimes S^\lambda W,
$$
as claimed. 
\end{proof} 

Note that we never used that $V,W$ were finite dimensional, so the statement is valid for any complex vector spaces $V,W$. 

\begin{corollary} (Cauchy identity) 
If $x=(x_1,...,x_r)$ and $y=(y_1,...,y_s)$ then one has 
$$
\sum_\lambda s_\lambda(x)s_\lambda(y)z^{|\lambda|}=\prod_{i=1}^r\prod_{j=1}^s\frac{1}{1-zx_iy_j}.
$$
\end{corollary} 

\begin{proof} 

\begin{lemma} (Molien formula). Let $A: V\to V$ be a linear operator on a finite dimensional vector space $V$. Denote by $S^NA$ the induced linear operator 
$A^{\otimes N}$ on $S^NV$. Then 
$$
\sum_{N=0}^\infty {\rm Tr}(S^NA)z^N=\frac{1}{\det(1-zA)}. 
$$
\end{lemma} 

\begin{proof} Let $\dim V=r$ and $A$ have eigenvalues $x_1,...,x_r$. Then 
the eigenvalues of $S^N A$ are all possible monomials in $x_i$ of degree $N$. 
Thus ${\rm Tr}(S^N A)$ is the sum of these monomials, which is the complete symmetric function $h_N(x_1,...,x_r)$. So 
 $$
\sum_{N=0}^\infty {\rm Tr}(S^NA)z^N=\sum_{N\ge 0}h_N(x_1,...,x_r)z^N=\prod_{i=1}^r \frac{1}{1-zx_i}=
\frac{1}{\det(1-zA)}. 
$$
\end{proof} 

Now let $X\in GL(V)$ with eigenvalues $x_1,...,x_r$ and $Y\in GL(W)$ 
with eigenvalues $y_1,...,y_s$. Then by Howe duality
$$
{\rm Tr}(S^N(X\otimes Y))=\sum_{\lambda: |\lambda|=N}s_\lambda(x)s_\lambda(y). 
$$
On the other hand, by Molien's formula  
$$
\sum_{N\ge 0}{\rm Tr}(S^N(X\otimes Y))z^N=\frac{1}{\det(1-z(X\otimes Y))}=\prod_{i,j}\frac{1}{1-zx_iy_j}.
$$
Comparing the two formulas, we obtain the statement. 
\end{proof} 

\section{\bf Fundamental and minuscule weights} 

\subsection{Minuscule weights}

Let $\g$ be a simple complex Lie algebra. Minuscule weights for $\g$ are highest weights 
for which irreducible representations are especially simple. 

\begin{definition} A dominant integral weight $\omega$ for $\g$ is called {\bf minuscule} 
if $(\omega,\beta)\le 1$ for all positive coroots $\beta$. 
\end{definition} 

Equivalently, $|(\omega,\beta)|\le 1$ for any coroot $\beta$. 

Obviously, $\omega=0$ is minuscule, but there may exist other minuscule weights. For example, 
for $\g=\mathfrak{sl}_n$, all fundamental weights are minuscule, since 
$(\omega_i, \bold e_j-\bold e_k)=0$ if $j,k\le i$ or $j,k>i$ and 
$(\omega_i, \bold e_j-\bold e_k)=1$ if $j\le i<k$. 

It is easy to see that any minuscule weight $\omega\ne 0$ is fundamental. Indeed, 
we can have $(\omega,\alpha_i^\vee)=1$ only for one $i$, and for all other simple coroots this inner product must be zero. Otherwise we will have $(\omega,\theta^\vee)\ge 2$, where $\theta^\vee$ is the maximal coroot (the maximal root of the dual root system $R^\vee$).\footnote{The maximal coroot $\theta^\vee$ should not be confused with 
the coroot $\widetilde \theta^\vee$ corresponding to the maximal root $\theta$ (highest weight of the adjoint representation) under a $W$-invariant identification $\h^*\cong \h$. In the non-simply-laced case they are not even proportional: e.g., for the root system $B_2$, $\theta^\vee=(2,0)$ while $\widetilde\theta^\vee=(1,1)$. This may be confusing since according to the general coroot notation, $\widetilde\theta^\vee$ should be denoted by $\theta^\vee$. \label{foo}} 

On the other hand, not all fundamental weights are minuscule. In fact, we will see that the simple Lie algebras of types $G_2$, $F_4$ and $E_8$ do not have any nonzero minuscule weights. To formulate a criterion for a fundamental weight to be minuscule, 
recall that $\theta^\vee=\sum_i k_i\alpha_i^\vee$, where $k_i=(\omega_i,\theta^\vee)$ are strictly positive integers.  

\begin{lemma} A fundamental weight $\omega_i$ is minuscule if 
and only if $k_i=1$.  
\end{lemma} 

\begin{proof} The definition of minuscule means that $k_i\le 1$. 
On the other hand, if $k_i=1$ then given a positive 
coroot $\beta=\sum_j n_j\alpha_j^\vee$, we have $n_j\le k_j$, in particular
$n_i\le 1$, so $\omega_i$ is minuscule.   
\end{proof}

\begin{lemma}\label{minex} Let $\omega\in Q$ and $|(\omega,\beta)|\le 1$ for all coroots $\beta$. 
Then $\omega=0$. 
\end{lemma} 

\begin{proof} Assume the contrary. Choose a counterexample $\omega=\sum_i m_i\alpha_i$ so that $\sum_i |m_i|$ is minimal possible. We have 
$$
(\omega,\omega)=\sum_i m_i(\omega,\alpha_i)>0.
$$
So there exists $j$ such that $m_j$ and $(\omega,\alpha_j^\vee)$ 
are nonzero and have the same sign. Replacing $\omega$ with $-\omega$ if needed, we may assume that both are positive, then $(\omega,\alpha_j^\vee)=1$. Then $s_j\omega=\omega-\alpha_j=\sum_i m_i'\alpha_i$ where 
$m_j'=m_j-1$ and $m_i'=m_i$ for all $i\ne j$ is another counterexample. But we have $\sum_i |m_i'|=\sum_i |m_i|-1$, a contradiction. 
\end{proof} 

Why are minuscule weights interesting? It is because of the following result.

\begin{proposition} The following conditions on a dominant integral weight $\omega$ are equivalent: 

(1) $\omega$ is minuscule; 

(2) all weights of the representation $L_\omega$ belong to the orbit $W\omega$; 

(3) if $\lambda$ is a dominant integral weight such that $\omega-\lambda\in Q_+$  
then $\lambda=\omega$. 
\end{proposition} 

\begin{proof} Let us prove that (1) implies (3). If $\omega=0$, there is nothing to prove, since then $-\lambda\in Q_+$, so $(\lambda,\rho)\le 0$, hence $\lambda=0$. So suppose that $\omega=\omega_i$ is minuscule. We have $\omega_i-\lambda=\sum_k m_k\alpha_k$ with $m_k\ge 0$. If $m_k=0$ for some $k\ne i$ then the problem reduces to smaller rank by deleting the vertex $k$ from the Dynkin diagram. So we may assume 
$m_k>0$ for all $k\ne i$. Let $\beta$ be a positive coroot. Then 
$$
(\omega_i-\lambda,\beta)=(\omega_i,\beta)-(\lambda,\beta)\le (\omega_i,\beta)\le 1
$$ 
and if $\alpha_i^\vee$ does not occur in $\beta$ then it is $\le 0$. 
So in particular we have $(\omega_i-\lambda,\alpha_j^\vee)\le 0$ if $j\ne i$. 
If also $(\omega_i-\lambda,\alpha_i^\vee)\le 0$ then $(\omega_i-\lambda,\omega_i-\lambda)\le 0$, so $\omega_i=\lambda$, as claimed. Thus we may assume that $(\omega_i-\lambda,\alpha_i^\vee)=1$, i.e., $m_i>0$, so $m_j>0$ for all $j$. Thus, $(\omega_i-\lambda,\theta^\vee)\ge 1$ (as $\theta^\vee$ is a dominant coweight). Hence
$(\lambda,\theta^\vee)\le 0$, i.e., $\lambda=0$, as $\theta^\vee$ contains all $\alpha_j^\vee$ 
with positive coefficients. Thus $\omega_i\in Q$. But this is impossible by Lemma \ref{minex}.

To see that (3) implies (2), note that if $\mu$ is any weight of $L_\omega$ then for some $w\in W$ the weight $\lambda=w\mu$ is dominant and $\omega-\lambda\in Q_+$, so $\lambda=\omega$ 
and $\mu=w^{-1}\omega$. 

Finally, we show that (2) implies (1). Assume (2) holds. If $\omega$ is not minuscule then there is a positive root $\alpha$ such that $(\omega,\alpha^\vee)>1$, hence $2(\omega,\alpha)>(\alpha,\alpha)$. Then $\omega-\alpha$ is a weight of $L_\omega$ (the weight of the nonzero vector $f_\alpha v_\omega$), and it is not $W$-conjugate to $\omega$, as 
$$
(\omega-\alpha,\omega-\alpha)=(\omega,\omega)-2(\omega,\alpha)+(\alpha,\alpha)<(\omega,\omega).
$$ 
\end{proof}

This immediately implies 

\begin{corollary}\label{charfor} The character of $L_\omega$ with minuscule $\omega$ is 
$$
\chi_\omega=\sum_{\gamma\in W\omega}e^\gamma. 
$$
\end{corollary} 

\begin{proposition} $\omega\in P_+$ is minuscule if and only if the restriction of $L_\omega$ to any root $\mathfrak{sl}_2$-subalgebra 
of $\g$ is the direct sum of $1$-dimensional and 2-dimensional representations.
\end{proposition} 

\begin{proof} Let $\omega$ be minuscule and 
$v\in L_\omega$ be a weight vector which is a highest weight vector for $(\mathfrak{sl}_2)_\alpha$. 
Then $h_\alpha v=(w\omega,\alpha^\vee)v=(\omega,w^{-1}\alpha^\vee)v$ for some $w\in W$. Thus $h_\alpha v=0$ or $h_\alpha v=v$, as claimed. 

On the other hand, if $\omega$ is not minuscule then there is a positive root $\alpha$ 
such that $(\omega,\alpha^\vee)=m>1$. So $h_\alpha v_\omega=mv_\omega$
and $v_\omega$ generates the irreducible $m+1$-dimensional representation of 
$(\mathfrak{sl}_2)_\alpha$. 
\end{proof} 

\subsection{Tensor product with a minuscule representation} 

\begin{corollary}\label{tenpromin} If $\omega$ is minuscule then for any dominant integral weight $\lambda$ of $\g$
we have 
$$
L_\omega\otimes L_\lambda=\oplus_{\gamma\in W\omega}L_{\lambda+\gamma},
$$
where if $\lambda+\gamma$ is not dominant then we agree that 
$L_{\lambda+\gamma}=0$. 
\end{corollary} 

\begin{proof} By the Weyl character formula and Corollary \ref{charfor}, the character of $L_\omega \otimes L_\lambda$ is 
$$
\chi_{L_\omega\otimes L_\lambda}=
\frac{\sum_{\mu\in W\omega}\sum_{w\in W}(-1)^{\ell(w)}e^{w(\lambda+\rho)+\mu}}
{\prod_{\alpha\in R_+}(e^{\alpha/2}-e^{-\alpha/2})}=
$$
$$
\frac{\sum_{\gamma\in W\omega}\sum_{w\in W}(-1)^{\ell(w)}e^{w(\lambda+\gamma+\rho)}}{\prod_{\alpha\in R_+}(e^{\alpha/2}-e^{-\alpha/2})}. 
$$
If $\lambda+\gamma\notin P_+$ then for some $i$ we have 
$(\lambda+\gamma,\alpha_i^\vee)<0$. But $(\gamma,\alpha_i^\vee)\ge -1$. So 
$(\lambda+\gamma,\alpha_i^\vee)=-1$ and thus $(\lambda+\gamma+\rho,\alpha_i^\vee)=0$. 
So for such $\gamma$, for any $w\in W$ 
the summand for $w$ cancels with the summand for $ws_i$. Thus we get 
$$
\chi_{L_\omega\otimes L_\lambda}=
\frac{\sum_{\gamma\in W\omega: \lambda+\gamma\in P_+}\sum_{w\in W}(-1)^{\ell(w)}e^{w(\lambda+\gamma+\rho)}}{\prod_{\alpha\in R_+}(e^{\alpha/2}-e^{-\alpha/2})}=\sum_{\gamma\in W\omega: \lambda+\gamma\in P_+}\chi_{L_{\lambda+\gamma}}.
$$
\end{proof} 

\begin{example}\label{tensgl} 1. Let $V$ be the vector representation of $GL_n$. 
Then for a partition $\lambda$, 
$V\otimes L_\lambda=\bigoplus_{\mu\in \lambda+\square} L_\mu$, where 
$\mu$ runs over all partitions obtained by adding one {\bf addable} box to the Young diagram of $\lambda$, i.e., such that it remains a Young diagram. 
For example, 
$$
V\otimes S^{(3,3,2,1)}V=S^{(4,3,2,1)}V\oplus S^{(3,3,3,1)}V\oplus S^{(3,3,2,2)}V\oplus S^{(3,3,2,1,1)}V.
$$

2. More generally, $\wedge^mV\otimes L_\lambda=\bigoplus_{\mu\in \lambda+m\square} L_\mu$, 
where we sum over partitions obtained by adding $m$ addable boxes to different rows of the Young diagram of $\lambda$ (going from top to bottom), i.e. a collection of $m$ boxes in different rows after adding which we still have a Young diagram. This follows immediately from Corollary \ref{tenpromin}.
For example, 
$$
\wedge^2V\otimes S^{(3,1)}V=S^{(4,2)}V\oplus S^{(4,1,1)}V\oplus S^{(3,2,1)}V\oplus S^{(3,1,1,1)}V.
$$
\end{example} 

\begin{proposition} (i) Let $\lambda$ be a partition of $N$. Then we have 
$$
\Bbb CS_{N+1}\otimes_{\Bbb CS_N}\pi_\lambda=\bigoplus_{\mu\in \lambda+\square}\pi_\mu.
$$

(ii) Let $\mu$ be a partition of $N+1$. Then we have 
$$
\pi_\mu|_{S_N}=\bigoplus_{\lambda\in \mu-\square}\pi_\lambda.
$$
\end{proposition} 

Here in (ii) we sum over all ways to delete a {\bf removable box} from the Young 
diagram of $\mu$, i.e., such that the remaining collection of boxes is still a Young diagram. 

\begin{proof} (i) Let $V$ be a vector space of sufficiently large dimension. 
Using Frobenius reciprocity and Schur-Weyl duality, we have 
$$
\Hom_{S_{N+1}}(\Bbb CS_{N+1}\otimes_{\Bbb CS_N}\pi_\lambda,V^{\otimes N+1})=
\Hom_{S_{N}}(\pi_\lambda,V\otimes V^{\otimes N})=V\otimes S^\lambda V.
$$
On the other hand, again by the Schur-Weyl duality, 
$$
\Hom_{S_{N +1}}(\bigoplus_{\mu\in \lambda+\square}\pi_\mu,V^{\otimes N+1})=
\bigoplus_{\mu\in \lambda+\square}S^\mu V.
$$
So the statement follows from Example \ref{tensgl}(1).

(ii) follows from (i) and Frobenius reciprocity. 
\end{proof} 

Let $\lambda^\dagger$ be the {\bf conjugate partition} to $\lambda$, 
which consists of the boxes $(j,i)$ where $(i,j)\in \lambda$. In other words, 
the Young diagram of $\lambda^\dagger$ is obtained by transposing the Young 
diagram of $\lambda$. For example, 
$(3,3,2,1)^\dagger=(4,3,2)$. 

\begin{corollary}\label{sign} Let $\Bbb C_-$ be the sign representation 
of $S_N$. Then 
$$
\pi_\lambda\otimes \Bbb C_-\cong \pi_{\lambda^\dagger}.
$$ 
\end{corollary}

\begin{proof} We argue by induction in $N=|\lambda|$, with obvious base $N=1$. 
Suppose the statement is known for $N$ and let us prove it for $N+1$. 
Given a partition $\nu$ of $N+1$, let $\lambda$ 
be obtained from $\nu$ by deleting a removable box $(i,j)$. 
Note that we have a natural isomorphism 
$$
\xi: (\Bbb CS_{N+1}\otimes_{\Bbb CS_N}\pi_\lambda)\otimes \Bbb C_-\to \Bbb CS_{N+1}\otimes_{\Bbb CS_N}(\pi_\lambda\otimes \Bbb C_-)=
\Bbb CS_{N+1}\otimes_{\Bbb CS_N}\pi_{\lambda^\dagger}.
$$
This can be written as an isomorphism 
$$
\bigoplus_{\mu\in \lambda+\square}\pi_\mu\otimes \Bbb C_-\cong \bigoplus_{\eta\in \lambda^\dagger+\square}\pi_\eta.
$$
Suppose $\pi_\nu\otimes \Bbb C_-=\pi_{\overline \nu}$. Then $\overline\nu\in  \lambda^
\dagger+\square$. But by Exercise \ref{contentex}, $\pi_\nu$ is the eigenspace of 
 the element $
\bold c\in \Bbb CS_{N+1}$ in $\Bbb CS_{N+1}\otimes_{\Bbb CS_N}\pi_\lambda$ with eigenvalue 
$c(\nu)$ (as $c(\mu)$ are all distinct for $\mu\in \lambda+\square$). Hence the eigenvalue of $\bold c$ on $\pi_{\bar\nu}$ is $-c(\nu)$. This implies that $\bar\nu=\nu^\dagger$, which justifies the induction step.
\end{proof} 

\begin{proposition}  \label{skewhowe} ({\bf Skew Howe duality}) 
Let $V,W$ be complex vector spaces. Then
$$
\wedge^n(V\otimes W)\cong \bigoplus_{\lambda: |\lambda|=n} S^\lambda V\otimes S^{\lambda^\dagger}W
$$
as $GL(V)\times GL(W)$-modules. 
\end{proposition} 

\begin{exercise} Prove Proposition \ref{skewhowe}. 

{\bf Hint:} Repeat the proof of the usual Howe duality (Subsection \ref{howd}), using Corollary \ref{sign}. 
\end{exercise} 

\begin{exercise} Compute characters and dimensions of irreducible representations 
$L_{a+b,b,0}$ of $SL_3(\Bbb C)$, where $a,b\ge 0$. Compute the weight multiplicities 
and draw the weights on the hexagonal lattice for \linebreak $a+b\le 3$, indicating the multiplicities. What are the special features of the case $b=0$?   

{\bf Hint.} The best way to do this exercise is to compute the characters 
recursively, using that $V\otimes L_{a+b,b,0}=L_{a+b+1,b,0}\oplus L_{a+b,b+1,0}\oplus
L_{a+b-1,b-1,0}$ (if $a=0$, the second summand drops out and if $b=0$ then the third one drops out), by the ``addable boxes" rule. This allows one to express the characters for $b+1$ in terms of the characters for $b$ and $b-1$. And we know the characters of $L_{a,0,0}$ - they are the complete symmetric functions $h_a$.  
\end{exercise} 

\begin{exercise} Compute the decomposition of 
$\wedge^mV\otimes S^kV$, \linebreak $\wedge^m V\otimes \wedge^k V$, $S^2(\wedge^m V)$, 
$\wedge^2(\wedge^m V)$ 
into irreducible representations of $GL(V)$.
\end{exercise} 

\begin{exercise}\label{intertw} Let $\g$ be a finite dimensional simple complex Lie algebra, and $V$ a finite dimensional representation of $\g$. Given a homomorphism 
$\Phi: L_\lambda\to V\otimes L_\mu$, let $\langle \Phi\rangle:=({\rm Id}\otimes v_\mu^*,\Phi v_\lambda)\in V$, where $v_\lambda$ is a highest weight vector of $L_\lambda$ and 
$v_\mu^*$ the lowest weight vector of $L_\mu^*$. In other words, we have 
$$
\Phi v_\lambda=\langle \Phi\rangle \otimes v_\mu+\text{lower terms}
$$ 
where the lower terms have lower weight than $\mu$ in the second component. 

(i) Show that $\langle \Phi\rangle$ has weight $\lambda-\mu$. 

(ii) Show that $f_i^{(\lambda,\alpha_i^\vee)+1}\langle \Phi\rangle=0$ for all $i$. 

(iii) Let $V[\nu]_\lambda$ be the subspace of vectors $v\in V[\nu]$ of weight $\nu$ which satisfy 
the equalities $f_i^{(\lambda,\alpha_i^\vee)+1}v=0$ for all $i$. Show that the map 
\linebreak $\Phi\mapsto \langle \Phi\rangle$ defines an isomorphism of vector spaces
$\Hom_\g(L_\lambda, V\otimes L_\mu)\cong V[\lambda-\mu]_\lambda$.

{\bf Hint.} Let 
$M_\lambda$ be the Verma module with highest weight $\lambda$, and $\overline M_{-\mu}$ be the {\bf lowest weight} Verma module 
with lowest weight $-\mu$, i.e., generated by a vector $v_{-\mu}$ with defining relations 
$hv_{-\mu}=-\mu(h)v_{-\mu}$ for $h\in \h$ and $f_iv_{-\mu}=0$.  
Show first that the map $\Phi\mapsto \langle \Phi\rangle$ defines an isomorphism 
$\Hom_\g(M_\lambda, V\otimes \overline M_{-\mu}^*)\cong V[\lambda-\mu]$. Next, show that 
$\Phi\in \Hom_\g(M_\lambda, V\otimes \overline M_{-\mu}^*)$ factors through $L_\lambda$ 
iff $\langle \Phi\rangle\in V[\lambda-\mu]_\lambda$, i.e., $f_i^{(\lambda,\alpha_i^\vee)+1}\langle \Phi\rangle=0$ (for this, use that $e_jf_i^{(\lambda,\alpha_i^\vee)+1}v_\lambda=0$, and that the kernel of $M_\lambda\to L_\lambda$ is generated by the vectors $f_i^{(\lambda,\alpha_i^\vee)+1}v_\lambda$). 
This implies that the above map defines an isomorphism 
$\Hom_\g(L_\lambda, V\otimes \overline M_{-\mu}^*)\cong V[\lambda-\mu]_\lambda$.
Finally, show that every homomorphism $L_\lambda\to V\otimes \overline M_{-\mu}^*$
in fact lands in $V\otimes L_\mu\subset V\otimes \overline M_{-\mu}^*$.  

(iv) Let $V$ be the vector representation of $SL_n(\Bbb C)$. Determine the weight subspaces of $S^mV$, and compute the decomposition of $S^mV\otimes L_\mu$ into irreducibles for all $\mu$
(use (iii)). 

(v) For any $\g$, compute the decomposition of $\g\otimes L_\mu$, where 
$\g$ is the adjoint representation of $\g$ (again use (iii)). 

In both (iv) and (v) you should express the answer in terms of the numbers $k_i$ such that $\mu=\sum_i k_i\omega_i$ and the Cartan matrix entries.  
\end{exercise} 

\begin{proposition}\label{minwei} Every coset in $P/Q$ contains a unique minuscule weight. 
This gives a bijection between $P/Q$ and minuscule weights. So the number of 
minuscule weights equals $\det A$, where $A$ is the Cartan matrix. 
\end{proposition}  

\begin{proof} Let $C:=a+Q\in P/Q$ be a coset, and consider the intersection 
$C\cap P_+$. Let $\omega\in C\cap P_+$ be an element with smallest 
$(\omega,\rho^\vee)$. If $\lambda$ is a dominant weight of $L_\omega$ 
then $\lambda\in C\cap P_+$, so $(\lambda,\rho^\vee)\ge (\omega,\rho^\vee)$, hence 
$(\omega-\lambda,\rho^\vee)\le 0$. But $\omega-\lambda\in Q_+$, so $\lambda=\omega$. 
Thus $\omega$ is minuscule. On the other hand, if $\omega_1,\omega_2\in C$ are minuscule and distinct then $\omega_1-\omega_2\in Q$, so by Lemma \ref{minex}, there is a coroot $\beta$ such that 
$(\omega_1-\omega_2,\beta)\ge 2$. So $(\omega_1,\beta)=1$ and $(\omega_2,\beta)=-1$. 
The first identity implies $\beta>0$ and the second one $\beta<0$, a contradiction. 
\end{proof} 

\subsection{Fundamental weights of classical Lie algebras} 

Let us now determine the fundamental weights of classical Lie algebras of types $B_n,C_n,D_n$. 

{\bf Type $C_n$.} Then $\g=\mathfrak{sp}_{2n}$. The positive roots are $\bold e_i\pm \bold e_j,2\bold e_i$, the simple roots 
$\alpha_1=\bold e_1-\bold e_2,...,\alpha_n=2\bold e_n$, so $\alpha_i^\vee=\alpha_i$ for 
$i\ne n$ and $\alpha_n^\vee=\bold e_n$. So 
$\omega_i=(1,...,1,0,...,0)$ ($i$ ones) for $1\le i\le n$. 

{\bf Type $B_n$.} Then $\g=\mathfrak{so}_{2n+1}$, so we have the same story as for $C_n$ except $\alpha_n=\bold e_n$ and $\alpha_n^\vee=2\bold e_n$, so we have the same $\omega_i$ for $i<n$ but $\omega_n=(\frac{1}{2},...,\frac{1}{2})$. 

{\bf Type $D_n$.} Then $\g=\mathfrak{so}_{2n}$, so the positive roots are $\bold e_i\pm \bold e_j$, the simple roots $\alpha_1=\bold e_1-\bold e_2,...,\alpha_{n-2}=\bold e_{n-2}-\bold e_{n-1}$, $\alpha_{n-1}=\bold e_{n-1}-\bold e_n$, $\alpha_n=\bold e_{n-1}+\bold e_n$. 
So $\omega_i=(1,...,1,0,...,0)$ ($i$ ones) for $i=1,...,n-2$, 
but $\omega_{n-1}=(\frac{1}{2},...,\frac{1}{2},-\frac{1}{2})$, 
$\omega_n=(\frac{1}{2},...,\frac{1}{2},\frac{1}{2})$. 

\subsection{Minuscule weights outside type $A$}

Proposition \ref{minwei} immediately tells us how many minuscule weights we have. 
For type $A$ we saw that all fundamental weights are minuscule. 
For $G_2,F_4,E_8$, $\det A=1$, so the only minuscule weight is $0$. 
For type $B_n$ we have $\det A=2$, so we should have one nonzero minuscule weight, 
and this is the weight $(\frac{1}{2},...,\frac{1}{2})$. The corresponding representation has weights $(\pm\frac{1}{2},...,\pm\frac{1}{2})$, so it has dimension $2^n$. It is called the {\bf spin representation}, denoted $S$. 

For $C_n$ we also have $\det A=2$, so we again have a unique nonzero minuscule weight.
Namely, it is the weight $(1,0,...,0)$ (so the minuscule representation is the tautological 
representation of $\mathfrak{sp}_{2n}$, of dimension $2n$). For $D_n$ we have $\det A=4$, so 
we have three nontrivial minuscule representations, with highest weights
$\omega_1,\omega_{n-1},\omega_n$, of dimensions $2n,2^{n-1},2^{n-1}$. The first one is the tautological representation and the remaining two are the {\bf spin representations} $S_-,S_+$, whose weights are $(\pm\frac{1}{2},...,\pm\frac{1}{2})$ with odd, respectively even number of minuses. 

For $E_6$ there are two nontrivial minuscule representations $V,V^*$ of dimension $27$. 
For $E_7$ there is just one of dimension $56$. These dimensions are computed easily by counting elements in the corresponding Weyl group orbits. 

\section{\bf Fundamental representations of classical Lie algebras} 

\subsection{Type $C_n$} Since the fundamental weights for $\g=\mathfrak{sp}_{2n}$ are 
$\omega_i=(1,...,1,0,..,0)$ ($i$ ones), same as for $\mathfrak{gl}_n$, one may 
think that the fundamental representations are also ``the same", i.e. $\wedge^i V$, where $V$ is the $2n$-dimensional vector representation. 
Indeed, a Cartan subalgebra in $\g$ is the space of matrices 
${\rm diag}(a_1,...,a_n,-a_1,...,-a_n)$, so $L_{\omega_1}=V$, with highest weight vector $e_1$. 
However, the representation $\wedge^2V$ is not irreducible, even though it has the correct highest weight $\omega_2$. Indeed, we have $\wedge^2V=\wedge^2_0V\oplus \Bbb C$, where $\Bbb C$ is the trivial representation spanned by the inverse $B^{-1}=\sum_i e_{i+n}\wedge e_i$ 
of the invariant nondegenerate skew-symmetric form $B=\sum_i e_i^*\wedge e_{i+n}^*\in \wedge^2V^*$ preserved by $\g$, and $\wedge^2_0V$ is the orthogonal complement of $B$. 

It turns out that $\wedge^2_0V$ is irreducible. (You can show it directly or using the Weyl dimension formula). Thus we have $L_{\omega_2}=\wedge^2_0V$ (if $n\ge 2$). 

So what happens for $L_{\omega_j}$ with any $j\ge 2$? To determine this, note that we have 
a homomorphism of representations $\iota_B: \wedge^{i} V\to \wedge^{i-2}V$, which is just the contraction with $B$ (we agree that $\wedge^jV=0$ for $j<0$). So we may consider the subrepresentation $\wedge^i_0V={\rm Ker}(\iota_B|_{\wedge^iV})\subset \wedge^iV$. 

\begin{exercise}
(i) Let $m_B: \wedge^{i-1}V\to \wedge^{i+1}V$ be the operator defined by $m_B(u):=B^{-1}\wedge u$. 
Show that the operators $m_B,\iota_B$ generate a representation of the Lie algebra $\mathfrak{sl}_2$ on $\wedge V:=\oplus_{i=0}^{2n} \wedge^iV$ where they are proportional to the operators $e,f$, such that $h$ acts on $\wedge^i V$ by multiplication by $i-n$. 

(ii) Show that $\iota_B$ is injective when $i\ge n+1$ and surjective when $i\le n+1$ (so an isomorphism for $i=n+1$).    

(iii) Show that ${\rm Ker}(\iota_B|_{\wedge^jV})$ is irreducible for $j\le n$, and 
is isomorphic to $L_{\omega_j}$, where we agree that $\omega_0=0$. Deduce that 
$$
\wedge V=\oplus_{i=0}^n L_{\omega_i}\otimes L_{n-i} 
$$
as a representation of $\mathfrak{sp}_{2n}\oplus \mathfrak{sl}_2$, where $L_m$ is the $m+1$-dimensional irreducible representation of $\mathfrak{sl}_2$ of highest weight $m$. 

(iv) Show that every irreducible representation of $\mathfrak{sp}_{2n}$ occurs in $V^{\otimes N}$ for some $N$.  

Thus we see another instance of the double centralizer property. 
\end{exercise} 

\subsection{Type $B_n$} We have $\g=\mathfrak{so}_{2n+1}$, preserving the quadratic form 
$Q=\sum_{i=1}^n x_ix_{i+n}+x_{2n+1}^2$. A Cartan subalgebra consists of matrices 
${\rm diag}(a_1,...,a_n,-a_1,...,-a_n,0)$. So the representations $\wedge^iV$, $1\le i\le n$, where $V$ is the $2n+1$-dimensional vector representation, have highest weight $(1,...,1,0,...0)$ ($i$ ones), which is $\omega_i$ if $i\le n-1$. 

\begin{exercise} Show that the representation $\wedge^i V$ is irreducible for $0\le i\le n$. 
\end{exercise} 

Thus for $1\le i\le n-1$ we have $\wedge^i V=L_{\omega_i}$. On the other hand, the representation $\wedge^nV$, even though irreducible, is not fundamental. Indeed, its highest weight is $(1,...,1)=2\omega_n$, as $\omega_n=(\frac{1}{2},...,\frac{1}{2})$. 
In fact, we see that the representation $L_{\omega_n}$ does not occur in $V^{\otimes N}$ for any $N$, since coordinates of its highest weight are not integer. As mentioned above, this representation is called the {\bf spin representation} $S$. Vectors in $S$ are called {\bf spinors}. The weights of $S$ are Weyl group translates of $\omega_n$, so they are $(\pm\frac{1}{2},...,\pm\frac{1}{2})$ for any choices of signs, so 
$\dim S=2^n$, and the character of $S$ is given by the formula 
$$
\chi_S(x_1,...,x_n)=(x_1^{\frac{1}{2}}+x_1^{-\frac{1}{2}})...(x_n^{\frac{1}{2}}+x_n^{-\frac{1}{2}}).
$$ 
This is supposed to be the trace of $\diag(x_1,...,x_n,x_1^{-1},...,x_n^{-1},1)\in SO_{2n+1}(\Bbb C)$, which does not make sense since the square roots on the right hand side are defined only up to sign. This shows that the spin representation $S$ {\bf does not lift} to the group $SO_{2n+1}(\Bbb C)$. Namely, the group $SO_{2n+1}(\Bbb C)$ is not simply connected, and 
the representation $S$ only lifts to the universal covering group $\widetilde{SO}_{2n+1}(\Bbb C)$, which is called the {\bf spin group}, and is denoted ${\rm Spin}_{2n+1}(\Bbb C)$. 

\begin{example} Let $n=1$. Then $\g=\mathfrak{so}_3(\Bbb C)=\mathfrak{sl}_2(\Bbb C)$ and 
$S$ is the 2-dimensional irreducible representation. We know that this representation does not lift to $SO_3(\Bbb C)$ but only to its double cover $SL_2(\Bbb C)$, which is simply connected
(so $\pi_1(SO_3(\Bbb C))=\Bbb Z/2$, demonstrated by the famous {\bf belt trick}). 
So we have ${\rm Spin}_3(\Bbb C)=SL_2(\Bbb C)$. This is related to the spin phenomenon 
in quantum mechanics which we will discuss later. This explains the terminology. 
\end{example} 

\begin{proposition}\label{fung} For $n\ge 3$ we have $\pi_1(SO_n(\Bbb C))=\Bbb Z/2$.  
\end{proposition} 

\begin{proof} 

\begin{lemma}\label{retra} Let $X_n$ be the hypersurface in $\Bbb C^n$ given by the equation 
$z_1^2+...+z_n^2=1$. Then for any $1\le k\le n-2$ we have $\pi_k(X_n)=0$, i.e., every continuous map $S^k\to X_n$ contracts to a point. E.g., $X_n$ is connected (=0-connected) for $n\ge 2$, simply connected (=1-connected) for $n\ge 3$, 2-connected for $n\ge 4$, etc.
\end{lemma} 

\begin{proof} The surface $X_n$ is the complexification of the $n-1$-sphere, $X_n^{\Bbb R}:=X_n\cap \Bbb R^n=S^{n-1}$. We will define 
a continuous family of maps \linebreak $f_t: X_n\to X_n$ such that $f_1={\rm Id}$ and $f_0$ 
lands in $X_n^{\Bbb R}$, with $f_t|_{X_n^{\Bbb R}}={\rm Id}$. This will show that $X_n^{\Bbb R}$ is a retract of $X_n$, so $X_n$ has the required properties since so does $X_n^{\Bbb R}$ (indeed, any map $\gamma=f_1\circ \gamma: S^k\to X_n$ is homotopic to the map $f_0\circ \gamma$ in $X_n^{\Bbb R}$, the homotopy being $f_t\circ \gamma$). 

Let $z=x+iy\in X_n$, where $x,y\in \Bbb R^n$. Then $z^2=1$, so we have $x^2-y^2=1, xy=0$. Hence
$$
(x+tiy)^2=x^2-t^2y^2=1+(1-t^2)y^2\ge 1. 
$$
So we may define 
$$
f_t(z):=\frac{x+tiy}{\sqrt{x^2-t^2y^2}}.
$$
Then $f_t(z)^2=1$, $f_1(z)=z$, and $f_0(z)=\frac{x}{|x|}$ lands in the sphere $S^{n-1}$, as needed. 
\end{proof}  

In particular, for $n=4$, changing coordinates, we see that the surface $ad-bc=1$ is doubly connected, i.e., $SL_2(\Bbb C)$ is doubly connected and thus $\pi_1(SO_3(\Bbb C))=\Bbb Z/2$ (which we already knew). 

Now, the group $SO_n(\Bbb C)$ acts on $X_n$ transitively with stabilizer 
$SO_{n-1}(\Bbb C)$, so we have a fibration 
$SO_n\to X_n$ with fiber $SO_{n-1}$. Therefore, we have an 
exact sequence 
$$
\pi_2(X_n)\to \pi_1(SO_{n-1}(\Bbb C))\to \pi_1(SO_{n}(\Bbb C))\to \pi_1(X_n)
$$
(a portion of the long exact sequence of homotopy groups). 
By Lemma \ref{retra}, the first and the last group in this sequence are trivial for $n\ge 4$
which implies that in this case $\pi_1(SO_{n-1}(\Bbb C))\cong \pi_1(SO_{n}(\Bbb C))$, so we conclude by induction that $\pi_1(SO_{n}(\Bbb C))=\Bbb Z/2$ for all $n\ge 3$ (using the case $n=3$ as the base). 
\end{proof} 

\begin{corollary} For $n\ge 1$ the simply connected group ${\rm Spin}_{2n+1}(\Bbb C)$ is a double cover 
of $SO_{2n+1}(\Bbb C)$. 
\end{corollary} 

\begin{exercise} (i) Use a similar argument to show that the groups $SL_{n+1}(\Bbb C)$ 
and $Sp_{2n}(\Bbb C)$ are simply connected for $n\ge 1$ (consider their action on nonzero vectors in the vector representation and compute the stabilizer). 

(ii) Generalize this argument to show that for any $k\ge 1$ the higher homotopy group $\pi_k$ 
for the classical groups $SL_{n+1}(\Bbb C)$, $SO_n(\Bbb C)$, $Sp_{2n}(\Bbb C)$ 
stabilizes (i.e., becomes independent of  $n$) when $n$ is large enough. How large does 
$n$ have to be for that?   
\end{exercise} 

\subsection{Type $D_n$} We have $\g=\mathfrak{so}_{2n}$, preserving the quadratic form 
$$
Q=\sum_{i=1}^n x_ix_{i+n}.
$$
A Cartan subalgebra consists of matrices 
${\rm diag}(a_1,...,a_n,-a_1,...,-a_n)$. So the representations $\wedge^iV$, $1\le i\le n$, where $V$ is the $2n$-dimensional vector representation, have highest weight $(1,...,1,0,...0)$ ($i$ ones), which is $\omega_i$ if $i\le n-2$. 

\begin{exercise} Show that the representation $\wedge^i V$ is irreducible for $0\le i\le n-1$. 
\end{exercise} 

Thus $L_{\omega_i}=\wedge^i V$ for $i\le n-2$. On the other hand, while the representation $L_{(1,...,1,0)}$ is irreducible, it is not fundamental, as 
$(1,...,1,0)=\omega_{n-1}+\omega_n$, where  
$\omega_{n-1}=(\frac{1}{2},...,\frac{1}{2}, -\frac{1}{2})$ and 
$\omega_n=(\frac{1}{2},...,\frac{1}{2}, \frac{1}{2})$. 
The fundamental representations $L_{\omega_{n-1}},L_{\omega_n}$ are 
called the {\bf spin representations} and denoted $S_-,S_+$; their elements are called {\bf spinors}. 
Similarly to the odd dimensional case, they have dimensions 
$2^{n-1}$ and characters 
$$
\chi_{S_\pm}=\left((x_1^{\frac{1}{2}}+x_1^{-\frac{1}{2}})...(x_n^{\frac{1}{2}}+x_n^{-\frac{1}{2}})\right)_\pm
$$
where the subscript $\pm$ means that we take the 
monomials with odd (for --), respectively even (for +) number of minuses.
This shows that, similarly to the odd dimensional case, 
$S_+,S_-$ don't occur in $V^{\otimes N}$ and don't lift 
to $SO_{2n}(\Bbb C)$ but require the universal covering ${\rm Spin}_{2n}(\Bbb C)=\widetilde{SO}_{2n}(\Bbb C)$, called the {\bf spin group}. Proposition \ref{fung} implies 

\begin{corollary} For $n\ge 2$ the group ${\rm Spin}_{2n}(\Bbb C)$ is a double cover of $SO_{2n}(\Bbb C)$. 
\end{corollary} 

\begin{example} Consider the spin groups and representations for small dimensions. 
We have seen that ${\rm Spin}_3=SL_2$, $S=\Bbb C^2$. We also have ${\rm Spin}_4=SL_2\times SL_2$, with $S_+,S_-$ being the 2-dimensional representations of the factors. 
We have ${\rm Spin}_5={\rm Sp}_4$, with $S$ being the $4$-dimensional vector representation. So $SO_5={\rm Sp}_4/(\pm 1)$. Finally, ${\rm Spin}_6=SL_4$, 
with $S_+,S_-$ being the $4$-dimensional representation $V$ and its dual $V^*$. 
Thus $SO_6=SL_4/(\pm 1)$. 
\end{example} 

\begin{exercise}\label{orthgr} Let $V$ be a finite dimensional vector space with a nondegenerate inner product. Consider the algebra $SV$ of polynomial functions on $V^*$. 
Let $x_1,...,x_n$ be an orthonormal basis of $V$, so that $SV\cong \Bbb C[x_1,...,x_n]$, 
and let $R^2:=\sum_{i=1}^n x_i^2\in S^2V$ be the ``squared radius".
Also let $\Delta=\sum_{i=1}^n\frac{\partial^2}{\partial x_i^2}$ be the {\bf Laplace operator.} 
Note that the Lie algebra $\mathfrak{so}(V)$ acts on $SV$ by automorphisms and 
$R^2$ and $\Delta$ are $\mathfrak{so}(V)$-invariant. A polynomial $P\in SV$ is called {\bf harmonic} if $\Delta P=0$.  

(i) Show that the operator of multiplication by $R^2$ and the Laplace operator $\Delta$ define an action of $\mathfrak{sl}_2$ on $SV$ which commutes with $\mathfrak{so}(V)$. Namely, they are proportional to $f,e$ respectively. Compute the operator $h$ (it will be a first order differential operator in $x_i$). 

(ii) Let $H_m\subset S^mV$ be the space of harmonic polynomials of degree $m$ 
(a representation of $\mathfrak{so}(V)$). 
Show that as an $\mathfrak{so}(V)\oplus\mathfrak{sl}_2$-module, $SV$ decomposes as 
$$
SV=\oplus_{m=0}^\infty H_m\otimes W_m,
$$
where $W_m$ are irreducible (infinite dimensional) representations of $\mathfrak{sl}_2$.   
Find the dimensions of $H_m$. 

(iii) Show that $H_m$ is irreducible, in fact $H_m=L_{m\omega_1}$. 
Decompose $S^mV$ into a direct sum of irreducible representations 
of $\mathfrak{so}(V)$. 

(iv) Show that $W_m$ are Verma modules and compute their highest weights.  

(v) For $s\in \Bbb C$ consider the algebra 
$$
A_s:=\Bbb C[x_1,...,x_n]/(x_1^2+...+x_n^2-s), 
$$
the algebra of polynomial functions on the hypersurface $x_1^2+...+x_n^2=s$ (here $(f)$ denotes the principal ideal generated by $f$). This algebra has a natural action of $\mathfrak{so}(V)$. Decompose $A_s$ into a direct sum of irreducible representations of $\mathfrak{so}(V)$.
\end{exercise}

\subsection{The Clifford algebra} 

It is important to be able to realize the spin representations explicitly. The reason it is somewhat tricky 
is that these representations don't occur in tensor powers of $V$ (as they have half-integer weights). However, the tensor product of a spin representation with its dual, $S\otimes S^*$, has integer weights and does express in terms of $V$. So we need to extract "the square root" from this representation, in the sense that ``the space of vectors of size $n$ is the square root of the space of square matrices of size $n$". This is the idea behind the Clifford algebra construction. 

\begin{definition} Let $V$ be a finite dimensional vector space over a field ${\bf k}$ of characteristic $\ne 2$ with a nondegenerate symmetric inner product $(,)$. The {\bf Clifford algebra} ${\rm Cl}(V)$ is the algebra generated 
by vectors $v\in V$ with defining relations 
$$
v^2=\tfrac{1}{2}(v,v), v\in V.
$$ 
\end{definition} 
Thus for $a,b\in V$ we have 
$$
ab+ba=(a+b)^2-a^2-b^2=\tfrac{1}{2}((a+b,a+b)-(a,a)-(b,b))=(a,b).
$$ 
This is a deformation of the exterior algebra $\wedge V$ which is defined in the same way 
but $v^2=0$. More precisely, ${\rm Cl}(V)$ has a filtration (defined by setting $\deg(v)=1$, $v\in V$) such that the associated graded algebra receives a surjective map 
$\phi: \wedge V\to {\rm gr}{\rm Cl}(V)$. We will show that this is a nice (``flat") deformation, in the sense that $\dim {\rm Cl}(V)=\dim \wedge V=2^{\dim V}$, so that $\phi$ is an isomorphism. 
This is a kind of Poincar\'e-Birkhoff-Witt theorem (namely, it is similar to the PBW theorem for Lie algebras, and in fact a special case of one if you pass from Lie algebras to more general Lie superalgebras). Namely, we have the following theorem. 

\begin{theorem}\label{repcli} If $\k$ is algebraically closed then the algebra ${\rm Cl}(V)$ is isomorphic to ${\rm Mat}_{2^n}(\bold k)$ if $\dim V=2n$ and to ${\rm Mat}_{2^n}(\bold k)\oplus {\rm Mat}_{2^n}(\bold k)$ if $\dim V=2n+1$.
\end{theorem} 

\begin{proof} Let us start with the even case. Pick a basis 
$a_1,...,a_n,b_1,...,b_n$ of $V$ so that the inner product is given by 
$$
(a_i,a_j)=(b_i,b_j)=0,\ (a_i,b_j)=\delta_{ij}. 
$$ 
We have $a_ia_j+a_ja_i=0$, $b_ib_j+b_jb_i=0$, $b_ia_j+a_jb_i=\delta_{ij}$. 
Define the ${\rm Cl}(V)$-module $M=\wedge (a_1,...,a_n)$ 
with the action of ${\rm Cl}(V)$ defined by 
$$
\rho(a_i)w=a_iw,\ \rho(b_i)w=\tfrac{\partial w}{\partial a_i},
$$
where 
$$
\tfrac{\partial}{\partial a_i}a_{k_1}...a_{k_r}=(-1)^{j-1} a_{k_1}...\widehat{a_{k_j}}...a_{k_r}
$$
if $i=k_j$ for some $j$ (where hat means that the term is omitted), and otherwise 
the result is zero. It is easy to check that this is indeed a representation. 

Now for $I=(i_1<...<i_k),J=(j_1<...<j_m)$ 
consider the elements $c_{IJ}=a_{i1}...a_{i_k}b_{j_1}...b_{j_m}\in {\rm Cl}(V)$. 
It is easy to see that these elements span ${\rm Cl}(V)$. 
Also it is not hard to do the following exercise. 

\begin{exercise} Show that the operators $\rho(c_{IJ})$ are linearly 
independent.
\end{exercise} 

 Thus $\rho: {\rm Cl}(V)\to \End M$ is an isomorphism, which 
proves the proposition in even dimensions. 

Now, if $\dim V=2n+1$, we pick a basis as above plus an additional element $z$ such that 
$(z,a_i)=(z,b_i)=0$, $(z,z)=2$. So we have 
$$
za_i+a_iz=0,\ zb_i+b_iz=0, \ z^2=1. 
$$
Now we can define the module $M_\pm$ 
on which $a_i,b_i$ act as before and $zw=\pm (-1)^{\deg w}w$. 
It is easy to see as before that the map 
$$
\rho_+\oplus \rho_-: {\rm Cl}(V)\to \End M_+\oplus \End M_-
$$
is an isomorphism. This takes care of the odd case. 
\end{proof} 

We will now construct an inclusion of the Lie algebra $\mathfrak{so}(V)$ into the Clifford algebra. This will allow us to regard representations of the Clifford algebra as representations of $\mathfrak{so}(V)$, which will give us a construction of the spin representations. 

Consider the linear map $\xi: \wedge^2V=\mathfrak{so}(V)\to {\rm Cl}(V)$ given by the formula 
$$
\xi(a\wedge b)=\tfrac{1}{2}(ab-ba)=ab-\tfrac{1}{2}(a,b).
$$
Then 
$$
[\xi(a\wedge b),\xi(c\wedge d)]=[ab,cd]=abcd-cdab=(b,c)ad-acbd-cdab=
$$
$$
(b,c)ad-(b,d)ac+acdb-cdab=
$$
$$
(b,c)ad-(b,d)ac+(a,c)db-cadb-cdab=
$$
$$
(b,c)ad-(b,d)ac+(a,c)db-(a,d)cb=
$$
$$
(b,c)\xi(a\wedge d)-(b,d)\xi(a\wedge c)+(a,c)\xi(d\wedge b)-(a,d)\xi(c\wedge b)=
\xi([a\wedge b,c\wedge d]). 
$$
Thus $\xi$ is a homomorphism of Lie algebras and we can define the representations
$\xi^*M$ for even $\dim V$ and $\xi^*M_{\pm}$ for odd $\dim V$ by $\rho_{\xi^*M}(a):=\rho_M(\xi(a))$. 

The representation $\xi^*M$ is reducible, namely 
$$
\xi^*M=(\xi^*M)_0\oplus (\xi^*M)_1,
$$
where subscripts $0$ and $1$ indicate the even and odd degree parts. 

\begin{exercise} (i) Show that for even $\dim V$, the representations $(\xi^*M)_0,(\xi^*M)_1$ are isomorphic to $S_+,S_-$ respectively.

(ii) Show that for odd $\dim V$, the representations $\xi^*M_+$ and $\xi^*M_-$ are both isomorphic to $S$.  

{\bf Hint.} Find the highest weight vector 
for each of these representations and compute the weight of this vector. 
Then compare dimensions. 
\end{exercise} 

\section{\bf Maximal root, exponents, Coxeter numbers, dual representations} 

\subsection{Duals of irreducible representations}

Now let $\g$ be any complex semisimple Lie algebra. How to compute the dual of the irreducible representation $L_\lambda$? It is clear that the highest weight of $L_\lambda^*$ equals 
$-\mu$, where $\mu$ is the lowest weight of $L_\lambda$, so we should compute the latter. 
For this purpose, recall that the Weyl group $W$ of $\g$ contains a unique element $w_0$ which maps dominant weights to antidominant weights, i.e., maps positive roots to negative roots. 
This is the maximal element, which is the unique element whose length is $|R_+|$. For example, if $-1\in W$ then clearly $w_0=-1$. It is easy to see that the lowest weight of $L_\lambda$ is $w_0\lambda$.  

Thus we get 

\begin{proposition} $L_\lambda^*=L_{-w_0\lambda}$. 
\end{proposition} 

The map $-w_0$ permutes fundamental (co)weights and simple (co)roots, so it is induced by an automorphism of the Dynkin diagram of $\g$. So if $\g$ is simple and its Dynkin diagram has no nontrivial automorphisms, we have $w_0=-1$, so $-w_0=1$ and thus $L_\lambda^*=L_\lambda$ for all $\lambda$. 
This happens for $A_1$, $B_n$, $C_n$, $G_2$, $F_4$, $E_7$ and $E_8$. 
In general, note that $s_i$ and hence the whole Weyl group $W$ acts trivially on $P/Q$, which implies that $-w_0$ acts on $P/Q$ by inversion.  
Thus we see that for $A_{n-1}$, $n\ge 3$, when $P/Q=\Bbb Z/n$, the map
$-w_0$ is the flip of the chain. Another way to see it is to note that $L_{\omega_1}^*=V^*=\wedge^{n-1}V=L_{\omega_{n-1}}$
(as $\dim V=n$). For $E_6$, $P/Q=\Bbb Z/3$, so $-w_0$ must exchange the two nonzero minuscule weights and thus must also be the flip.

\begin{exercise} (i) Show that for $D_{2n+1}$ we have $S_+^*=S_-$ while for 
$D_{2n}$ we have $S_+^*=S_+$, $S_-^*=S_-$. ({\bf Hint:} Show that in the first case 
$P/Q\cong\Bbb Z/4$ while in the second case $P/Q\cong(\Bbb Z/2)^2$.) 

(ii) Show that the restriction of the spin representation $S$ of $\mathfrak{so}_{2n+1}$ 
to $\mathfrak{so}_{2n}$ is $S_+\oplus S_-$. 

(iii) Show that there exist unique up to scaling nonzero {\bf Clifford multiplication} homomorphisms 
$$
V\otimes S\to S,\ V\otimes S_+\to S_-,\ V\otimes S_-\to S_+.
$$ 

(iv) Compute the decomposition of the tensor products 
$$
S\otimes S^*,\ S_+\otimes S_+^*,\
S_-\otimes S_-^*,\ S_+\otimes S_-^*
$$ 
into irreducible representations. 

{\bf Hint.} In the odd dimensional case, use that ${\rm Cl}(V)=2S\otimes S^*$ as an $\mathfrak{so}(V)$-module, that ${\rm gr}{\rm Cl}(V)=\wedge V$, and that representations of $\mathfrak{so}(V)$ 
are completely reducible. 

The even case is similar: 
$$
{\rm Cl}(V)=S_+\otimes S_+^*\oplus S_-\otimes S_-^*\oplus S_-\otimes S_+^*\oplus S_+\otimes S_-^*.
$$ 
If $\dim V=2n$ and $n$ is even, use that all representations of 
$\mathfrak{so}(V)$ are selfdual to conclude that the last two summands are isomorphic. 
(If $n$ is odd, they will not be isomorphic). 

Also in this case you need to pay attention to the middle exterior power - it should split into two parts. 
Namely, if $\dim V=2n$ then on $\wedge^n V$ we have two 
invariant bilinear forms: one symmetric coming from the one on $V$, 
denoted $B(\xi,\eta)$, and the other given by wedge product 
$\wedge : \wedge^nV\times \wedge^nV\to \wedge^{2n}V=\Bbb C$, which is symmetric  
for even $n$ and skew-symmetric for odd $n$. Since the wedge product form is nondegenerate, there is a unique linear operator $*:\wedge^nV\to \wedge^nV$ called the {\bf Hodge *-operator} such that $B(\xi,\eta)=\xi\wedge *\eta$. You should show that $*^2=1$ in the even case and $*^2=-1$ in the odd case (use an orthonormal basis of $V$). Thus we have an eigenspace decomposition $\wedge^nV=\wedge^n_+V\oplus \wedge^n_-V$,  into eigenspaces of $*$ with eigenvalues $\pm 1$ in the even case (called {\bf selfdual} and {\bf anti-selfdual} forms respectively) and $\pm i$ in the odd case. You will see that these pieces are irreducible and non-isomorphic, and that for even $n$ one of them (which?) goes into $S_+\otimes S_+^*$ and the other into $S_-\otimes S_-^*$; for odd $n$ they go into $S_+\otimes S_-^*$ and $S_-\otimes S_+^*$ instead. 
\end{exercise} 

\subsection{The maximal root} \label{maxroo}

Let $\g$ be a complex simple Lie algebra and $\theta$ be the maximal root of $\g$, 
i.e., the highest weight of the adjoint representation. For example, for $\g=\mathfrak{sl}_n$ the adjoint representation is generated by the highest weight vector of $V\otimes V^*$, where $V=\Bbb C^n$ is the vector representation. Thus we have 
$$
\theta=\omega_1+\omega_{n-1}=(2,1,...,1,0)=(1,0,...,0,-1),
$$
the sum of the highest weights of $V$ and $V^*$ (recall that weights for $\mathfrak{sl}_n$ are $n$-tuples of complex numbers modulo simultaneous translation by the same number). Thus, $\theta$ is not fundamental. Similarly, for $\g=\mathfrak{sp}_{2n}$, 
we have $\g=S^2V$ where $V$ is the vector representation, so 
$\theta=2\omega_1$ is again not fundamental. 
Nevertheless, we have the following proposition.

\begin{proposition} 
For any simple Lie algebra $\g\ne \mathfrak{sl}_n,\mathfrak{sp}_{2n}$, $\theta$ is a fundamental weight. 
\end{proposition}  

\begin{proof} If $\g=\mathfrak{so}_N$, $N\ge 7$ (i.e. of type $B$ or $D$ but not $A$ or $C$) then $\g=\wedge^2V=L_{\omega_2}$, so $\theta=\omega_2$. 

If $\g=G_2$, $\alpha_1=\alpha$ is the long simple root and $\alpha_2=\beta$ is the short one, then we easily see that $\theta=2\alpha_1+3\alpha_2=\omega_1$. 

If $\g=F_4$ then using the conventions of Subsection \ref{F4r}, 
we have $\theta=\bold e_1+\bold e_2=\omega_4$. 

If $\g=E_8$ then using the conventions of Subsection \ref{E8r}, 
we have $\theta=\bold e_1+\bold e_2=\omega_8$. 

If $\g=E_7$ then using the conventions of Subsection \ref{E7r}, 
we have $\theta=\bold e_1-\bold e_2=\omega_1$.

If $\g=E_6$ then using the conventions of Subsection \ref{E6r}, we have
$$
\theta=\tfrac{1}{2}(\bold e_1-\bold e_2-\bold e_3+\sum_{i=4}^8 \bold e_i)=\omega_2.
$$
\end{proof}

\subsection{Principal $\mathfrak{sl}_2$, exponents}\label{expone} 

Let $\g$ be a simple Lie algebra and let $e=\sum_i e_i$ and $h\in \h$ be such that 
$\alpha_i(h)=2$ for all $i$ (i.e., $h=2\rho^\vee$). We have $[h,e]=2e$ and 
$h=\sum_i (2\rho^\vee,\omega_i)h_i$. So defining 
$f:=\sum_i (2\rho^\vee,\omega_i)f_i$, we have $[h,f]=-2f$, $[e,f]=h$. So $e,f,h$ 
span an $\mathfrak{sl}_2$-subalgebra of $\g$ called the {\bf principal $\mathfrak{sl}_2$-subalgebra.} 

\begin{exercise} Let $\g=\mathfrak{sl}_{n+1}$. Show that the restriction of 
the $n+1$-dimensional vector representation $V$ of $\g$ to the principal $\mathfrak{sl}_2$-subalgebra 
is the irreducible representation $L_n$. 
\end{exercise} 

Consider now $\g$ as a module over its principal $\mathfrak{sl}_2$-subalgebra. 
How does it decompose? To see this, we can look at the weight decomposition of $\g$ under $h$. 
We have $\g=\n_-\oplus \h\oplus \n_+$, and these summands correspond to negative, zero and positive weights, respectively. Moreover, all weights are even, and for $m>0$, 
$\dim \g[2m]=r_m$ is the number of positive roots of height $m$, i.e., representable as a sum of $m$ simple roots, while $\g[0]=\h$ (as $\rho^\vee$ is a regular coweight), so $\dim \g[0]=r$, the rank of $\g$. 

\begin{definition} $m$ is called an {\bf exponent} of $\g$ if $r_m>r_{m+1}$. 
The multiplicity of $m$ is $r_m-r_{m+1}$. 
\end{definition}  

Since $r_m$ is zero for large $m$ while $r_0=r$,  there are $r$ exponents counting multiplicities. The exponents of $\g$ are denoted $m_i$ and are arranged in non-decreasing order: $m_1\le m_2\le...\le m_r$ (including multiplicities). Note that roots of height $2$ are $\alpha_i+\alpha_j$ where $i,j$ are connected by an edge. Thus we have $r_0=r_1=r$, $r_2=r-1$ (as the Dynkin diagram of $\g$ is a tree), so $m_1=1$ and $m_2>1$. 
We also have $m_r=(\rho^\vee,\theta):={\rm h}_\g-1$, where 
$\theta$ is the maximal root. The number ${\rm h}_\g$ is called the {\bf Coxeter number} of  $\g$. Finally, we have $\sum_{i=1}^r m_i=|R_+|$.  

\begin{proposition} The restriction of $\g$ to the principal $\mathfrak{sl}_2$-subalgebra 
decomposes as $\oplus_{i=1}^rL_{2m_i}$. 
\end{proposition} 

\begin{proof} This easily follows from the representation theory of $\mathfrak{sl}_2$ (Subsection \ref{sl2rep}) and the definition of $m_i$.  
\end{proof}

\begin{example} The exponents of $\mathfrak{sl}_n$ are $1,2,...,n-1$. 
\end{example} 

\begin{exercise} (i) Show that the exponents of $\mathfrak{so}_{2n+1}$ and $\mathfrak{sp}_{2n}$ are $1,3,...,2n-1$, and the exponents of $\mathfrak{so}_{2n+2}$ are $1,3,...,2n-1$ and $n$ (so in the latter case, when $n$ is odd, the exponent $n$ has multiplicity $2$). 

(ii) Show that the exponents of $G_2$ are $1$ and $5$. 
\end{exercise} 

\begin{exercise} Show that the exponents of $F_4$ are $1,5,7,11$, the exponents of 
$E_6$ are $1,4,5,7,8,11$, the exponents of $E_7$ are 
$1,5,7,9,11,13,17$, and the exponents of $E_8$ are 
$1,7,11,13,17,19,23,29$. 

{\bf Hint:} For $m\ge 1$, use the data from Subsections \ref{F4r},\ref{E8r},\ref{E7r},\ref{E6r} to count roots satisfying the equation $(\rho^\vee,\alpha)=m$, and find $m$ where the number of such roots drops as $m$ is increased. 
\end{exercise} 

\begin{exercise} Use the Weyl character formula for the adjoint representation and the Weyl denominator formula to prove the following identity for a simple Lie algebra $\g$: 
$$
\sum_{i=1}^r \frac{q^{2m_i+1}-q^{-2m_i-1}}{q-q^{-1}}=
\prod_{\alpha\in R_+: (\theta,\alpha^\vee)>0}\frac{q^{(\theta+\rho,\alpha^\vee)}-q^{-(\theta+\rho,\alpha^\vee)}}{q^{(\rho,\alpha^\vee)}-q^{-(\rho,\alpha^\vee)}}.
$$
({\bf Hint:} Compute the character of $\g$ as a module over the principal $\mathfrak{sl}_2$-subalgebra in two different ways.)
\end{exercise} 

\subsection{The Coxeter number and the dual Coxeter number} 

We have defined the Coxeter number of a simple complex Lie algebra $\g$ (or a reduced irreducible root system $R$) to be ${\rm h}_R={\rm h}_\g:=(\theta,\rho^\vee)+1=m_r+1$, where $m_r$ is the largest exponent of $\g$. 
One can also define the {\bf dual Coxeter number} of $\g$ (or $R$) 
as ${\rm h}_R^\vee={\rm h}_\g^\vee:=(\widetilde\theta^\vee,\rho)+1$, cf. footnote \ref{foo} (clearly, ${\rm h}_R^\vee={\rm h}_R$ if $R$ is simply laced). So the dual Coxeter number is the eigenvalue $\frac{1}{2}(\theta,\theta+2\rho)$ of $\frac{1}{2}C$ on the adjoint representation $\g$, where $C\in U(\g)$ is the quadratic Casimir element defined using the inner product in which $(\theta,\theta)=2$ (or, equivalently, long roots have squared length $2$). Indeed, if we identify $\h$ and $\h^*$ using this inner product then $\theta$ gets identified with $\widetilde\theta^\vee$.  

Using the formulas from Subsections \ref{rhorhov} and \ref{maxroo}, 
we get
$$
{\rm h}_{A_{n-1}}=n,
$$
$$
{\rm h}_{B_n}=2n,\ {\rm h}_{B_n}^\vee=2n-1,
$$
$$
{\rm h}_{C_n}=2n,\ {\rm h}_{C_n}^\vee=n+1,
$$
$$
{\rm h}_{D_n}=2n-2,
$$
$$
{\rm h}_{G_2}=(2\alpha+3\beta,5\alpha^\vee+3\beta^\vee)+1=6,\   h_{G_2}^\vee=\frac{1}{3}(2\alpha+3\beta,3\alpha+5\beta)+1=4,
$$
$$
{\rm h}_{F_4}=(8,3,2,1)\cdot (1,1,0,0)+1=12,\ 
{\rm h}_{F_4}^\vee=(\tfrac{11}{2},\tfrac{5}{2},\tfrac{3}{2},\tfrac{1}{2})\cdot (1,1,0,0)+1=9, 
$$
$$
{\rm h}_{E_8}=(23,6,5,4,3,2,1,0)\cdot (1,1,0,0,0,0,0,0)+1=30,
$$
$$
{\rm h}_{E_7}=(\tfrac{17}{2},-\tfrac{17}{2},5,4,3,2,1,0)\cdot (1,-1,0,0,0,0,0,0)+1=18, 
$$
$$
{\rm h}_{E_6}=(4,-4,-4,4,3,2,1,0)\cdot \tfrac{1}{2}(1,-1,-1,1,1,1,1,1)+1=12.
$$
Note that we always have ${\rm h}_R={\rm h}_{R^{\vee}}$, but if $R$ is not simply laced then, as we see, the numbers
${\rm h}_R$, ${\rm h}_{R^\vee}^\vee$, ${\rm h}_R^\vee$ are different, in general. 

\subsection{Representations of complex, real and quaternionic type}

\begin{definition} An irreducible finite dimensional $\Bbb C$-representation $V$ of a group $G$ or Lie algebra $\g$   
is {\bf complex type} when $V\ncong V^*$, {\bf real type} if there is a symmetric isomorphism 
$V\to V^*$ (i.e., an invariant symmetric inner product on $V$), and {\bf quaternionic type} 
if there is  a skew-symmetric isomorphism 
$V\to V^*$ (i.e., an invariant skew-symmetric inner product of $V$). 
\end{definition} 

It is easy to see that any irreducible finite dimensional representation is of exactly one of these three types (check it!). 

\begin{exercise} Let $V$ be an irreducible finite dimensional representation of a finite group $G$. 

(i) Show that ${\rm End}_{\Bbb R G}V$ is $\Bbb C$ for complex type, ${\rm Mat}_2(\Bbb R)$ 
for real type and the quaternion algebra $\Bbb H$ for quaternionic type. This explains the terminology. 

(ii) Show that $V$ is of real type if and only if in some basis of $V$ the matrices of all elements of $G$ 
have real entries. 

You may find helpful to look at \cite{E}, Problem 5.1.2 (it contains a hint). 
\end{exercise}

\begin{example}\label{realquatsl2} Let $L_n$ be the irreducible representation of 
$\mathfrak{sl}_2(\Bbb C)$ with highest weight $n$ (i.e., of dimension $n+1$). 
Then $L_n$ is of real type for even $n$ and quaternionic type for odd $n$. Indeed, $L_n=S^nV$, where $V=L_1=\Bbb C^2$, so the invariant form on $L_n$ is $S^nB$, where $B$ is the invariant form on $V$, which is skew-symmetric. 
\end{example} 

Now let $\g$ be any simple Lie algebra and $\lambda\in P_+$ be such that $\lambda=-w_0\lambda$, so that $L_\lambda$ is selfdual. How to tell if it is of real or quaternionic type?  

\begin{proposition}\label{typ} $L_\lambda$ is of real type if $(2\rho^\vee,\lambda)$ is even and of quaternionic type if it is odd. 
\end{proposition} 

\begin{proof} The number $n:=(2\rho^\vee,\lambda)$ is the eigenvalue of the element $h$ of the principal 
$\mathfrak{sl}_2$-subalgebra on the highest weight vector $v_\lambda$. All the other eigenvalues 
are strictly less. Thus the restriction of $L_\lambda$ to the principal $\mathfrak{sl}_2$-subalgebra 
is of the form $L_n\oplus \bigoplus_{m<n}k_mL_m$, i.e., $L_n$ occurs with multiplicity $1$. Hence the nondegenerate invariant form on $L_\lambda$ restricts to a nondegenerate invariant form on $L_n$, so by Example \ref{realquatsl2} it is skew-symmetric if $n$ is odd and symmetric if $n$ is even.  
\end{proof} 

\begin{example} Consider $\g=\mathfrak{so}_{2n}$. 
Then we have 
$$
\rho^\vee=\rho=\sum_i \omega_i=(n-1,n-2,...,1,0).
$$ 
So $(2\rho^\vee,\omega_{n-1})=(2\rho^\vee,\omega_n)=\frac{n(n-1)}{2}$. 
This is odd if $n=2,3$ modulo $4$ and even if $n=0,1$ modulo $4$. 
Thus $S_\pm$ carry a symmetric form when $n=0$ mod $4$ and a skew-symmetric form 
if $n=2$ mod $4$, while for $n=1,3$ mod $4$ we have $S_+^*=S_-$, so $S_+,S_-$ are of complex type.  

Consider now $\g=\mathfrak{so}_{2n+1}$. Then $\rho^\vee=\sum_i \omega_i^\vee=(n,n-1,...,1)$. 
So $(2\rho^\vee,\omega_n)=\frac{n(n+1)}{2}$. So $S$ carries a skew-symmetric form if $n=1,2$ mod $4$ and 
a symmetric form if $n=0,3$ mod $4$.   
\end{example} 

We obtain the following result. 

\begin{theorem} (Bott periodicity for spin representations) The behavior of the spin representations of the orthogonal Lie algebra $\mathfrak{so}_m$
is determined by the remainder $r$ of $m$ modulo $8$. 
Namely: 

For $r=1,7$, $S$ is of real type.

For $r=3,5$, $S$ is of quaternionic type. 

For $r=0$, $S_+,S_-$ are of real type.  

For $r=2,6$, $S_+^*=S_-$ (complex type).

For $r=4$, $S_+,S_-$ are of quaternionic type. 
\end{theorem} 

\section{\bf Differential forms, partitions of unity} 

Now we want to develop an integration theory on Lie groups. First we need to recall the basics about integration on manifolds. 

\subsection{Locally compact spaces} 

A Hausdorff topological space $X$ is called {\bf locally compact} if every point has a neighborhood whose closure is compact. For example, $\Bbb R^n$ and thus every manifold is locally compact. 

\begin{lemma}\label{loccom} If $X$ is a locally compact topological space with a countable base then it can be represented as a nested union of compact subsets: 
$X=\cup_{n\in \Bbb N}K_n$, $K_i\subset K_{i+1}$, such that every point $x\in X$ 
has a neighborhood $U_x$ contained in some $K_n$. 
\end{lemma} 

\begin{proof} For each $x\in X$ fix a neighborhood $U_x$ of $x$ such that $\overline{U}_x$ is compact. By Lemma \ref{counsub} the open cover $\lbrace U_x\rbrace$ of $X$ has a countable subcover $\lbrace W_i,i\in \Bbb N\rbrace$. Then the sets $K_n=\cup_{i=0}^n \overline{W_i}$
form a desired nested sequence of compact subsets of $X$. 
\end{proof} 

An open cover of a topological space $X$ is said to be {\bf locally finite} if every point of $X$ has a neighborhood intersecting only finitely many members of this cover. 

\begin{lemma}\label{locfi} Let $X$ be a locally compact topological space with a countable base. Then every base of $X$ has a countable, locally finite subcover. 
\end{lemma} 

\begin{proof} Use Lemma \ref{loccom} to write $X$ as a nested union of compact sets $K_n$ such that every point is contained in some $K_n$ together with its neighborhood. 
We construct the required subcover inductively as follows. Choose finitely many sets $U_1,...,U_{N_0}$ of the base covering $K_0$, and remove all other members of the base which meet $K_0$. The remaining collection of open sets is no longer a base but still an open cover of $X$. So add finitely many new sets $U_{N_0+1},...,U_{N_1}$ from this cover (all necessarily disjoint from $K_0$) to our list 
so that it now covers $K_1$, and remove all other members that meet $K_1$, and so on. 
The remaining sequence $U_1,U_2,...$ has only finitely many members which meet 
every $K_n$, so every point of $X$ has a neighborhood meeting only finitely many $U_i$. 
\end{proof} 

\subsection{Reminder on differential forms} 

Let $M$ be a real smooth $n$-dimensional manifold. Recall that a differential $k$-form on $M$ 
is a smooth section of the vector bundle $\wedge^kT^*M$, i.e., a skew-symmetric $(0,k)$-tensor field (see Subsection \ref{diffor}). Thus, for example, 
a 1-form is a section of $T^*M$. If $x_1,...,x_n$ are local coordinates on $M$ near some point $p\in M$ 
then the differentials $dx_1,...,dx_n$ form a basis in fibers of $T^*M$ near this point, so a general $1$-form 
in these coordinates has the form 
$$
\omega=\sum_{i=1}^nf_i(x_1,...,x_n)dx_i.
$$
If we change the coordinates $x_1,...,x_n$ to $y_1,...,y_n$ then $x_i$ are smooth functions of $y_1,...,y_n$ and in the new coordinates $\omega$ looks like
$$
\omega=\sum_{i,j=1}^n f_i(x_1,...,x_n)\frac{\partial x_i}{\partial y_j}dy_j. 
$$
Similarly, a differential $k$-form in the coordinates $x_i$ looks like 
$$
\omega=\sum_{1\le i_1<...< i_k\le n}f_{i_1,...,i_k}(x_1,...,x_n)dx_{i_1}\wedge...\wedge dx_{i_k}
$$
where $f_{i_1,...,i_k}$ are smooth functions, and in the coordinates $y_j$ it looks like 
$$
\omega=\sum_{1\le i_1<...<i_k\le n}\sum_{1\le j_1<...<j_k\le n}f_{i_1,...,i_k}(x_1,...,x_n)\det\left(\frac{\partial x_{i_r}}{\partial y_{j_s}}\right)dy_{j_1}\wedge...\wedge dy_{j_k}. 
$$
The space of differential $k$-forms on $M$ is denoted 
$\Omega^k(M)$. For instance, $\Omega^0(M)=C^\infty(M)$ and $\Omega^k(M)=0$ for $k>n$.  
Consider now the extremal case $k=n$. The bundle $\wedge^n T^*M$ is a line bundle (a vector bundle of rank $1$), so locally any differential $n$-form in coordinates $x_i$ has the form 
$$
\omega=f(x_1,...,x_n)dx_1\wedge...\wedge dx_n, 
$$
which in coordinates $y_j$ takes the form 
$$
\omega=f(x_1,...,x_n)\det\left(\frac{\partial x_i}{\partial y_j}\right)dy_1\wedge...\wedge dy_n.
$$

We have a canonical differentiation operator $d: \Omega^0(M)\to \Omega^1(M)$ 
given in local coordinates by 
$$
df=\sum_{i=1}^n \frac{\partial f}{\partial x_i}dx_i.
$$
It is easy to check that this operator does not depend on the choice of coordinates (this becomes obvious if 
you define it without coordinates, $df(v)=\partial_v f$ for $v\in T_pM$). Also $\Omega^\bullet(M):=\oplus_{k=0}^n \Omega^k(M)$ is a graded algebra under wedge product, and $d$ naturally extends to a degree $1$ derivation $d: \Omega^\bullet(M)\to \Omega^\bullet (M)$ defined in coordinates by 
$$
d(fdx_{i_1}\wedge...\wedge dx_{i_k})=df\wedge dx_{i_1}\wedge ...\wedge dx_{i_k}.
$$ 
Namely, this is independent of  choices and gives rise to a derivation in the ``graded" sense: 
$$
d(a\wedge b)=da\wedge b+(-1)^{\deg a}a\wedge db.
$$
A form $\omega$ is {\bf closed} if $d\omega=0$ and {\bf exact} if $\omega=d\eta$ for some $\eta$. 
It is easy to check that $d^2=0$, so any exact form is closed. However, not every closed form is exact: 
on the circle $S^1=\Bbb R/\Bbb Z$ the form $dx$ is closed but the function $x$ is defined only up to adding integers, so $dx$ is not exact. The space $\Omega^k_{\rm closed}(M)/\Omega^k_{\rm exact}(M)$ is 
called the $k$-th {\bf de Rham cohomology} of $M$, denoted $H^k(M)$. 

If $f: M\to N$ is a differentiable mapping then for a differential form $\omega\in \Omega^k(N)$ 
we can define the pullback $f^*\omega\in \Omega^k(M)$, given by 
$(f^*\omega)(v_1,...,v_k)=\omega(f_*v_1,...,f_*v_k)$ for $v_1,...,v_k\in T_pM$. 
This operation commutes with wedge product and the differential, and $(f\circ g)^*=g^*\circ f^*$. 

\subsection{Partitions of unity} 

Let $M$ be a manifold and $\lbrace{U_i, i\in I\rbrace}$ 
be an open cover of $M$. 

\begin{definition} A smooth {\bf partition of unity} subordinate to 
$\lbrace{U_i, i\in I\rbrace}$ is a collection $\lbrace f_s, s\in S\rbrace$ of smooth nonnegative functions on $M$ such that

(i) for all $s$ the support of $f_s$ is contained in $U_i$ for some $i=i(s)$;

(ii) Any $y\in M$ has a neighborhood in which all but finitely many $f_s$ are zero;  

(iii) $\sum_s f_s=1$. 
\end{definition} 

Note that the sum in (iii) makes sense because of condition (ii). 

Note also that given any partition of unity $\lbrace f_s\rbrace$ subordinate to $\lbrace U_i\rbrace$, we can define 
$$
F_i:=\sum_{s: i(s)=i}f_s,
$$
and this is a new partition of unity subordinate to the same cover now labeled by the set $I$, with the support of $F_i$ contained in $U_i$. 

Finally, note that in every partition of unity on $M$, the set of $s$ such that $f_s$ is not identically zero is countable, and moreover finite if $M$ is compact. This follows from the fact that by Lemma \ref{counsub}, any open cover of a manifold $M$ has a countable subcover, and moreover a finite one if $M$ is compact (applied to the neighborhoods from condition (ii)). 

\begin{proposition}\label{subor} Any open cover $\lbrace{U_i,i\in I\rbrace}$ of a manifold $M$ admits a partition of unity subordinate to this cover. 
\end{proposition} 

\begin{proof} 
Define a function $h: \Bbb [0,\infty)\to\Bbb R$ given 
by $h(t)=0$ for $t\ge 1$ and $h(t)=\exp(\frac{1}{t-1})$ for $t<1$. 
It is easy to check that $h$ is smooth. Thus we can define 
the smooth {\bf hat function} $H(x):=h(|x|^2)$ on $\Bbb R^n$, supported on the closed unit ball $\overline{B(0,1)}$. 

If $\phi: \overline{B(0,1)}\to M$ is a $C^\infty$-map which is a diffeomorphism 
onto the image, we will say that the image of $\phi$ is a {\bf closed ball} in $M$. 
Thus given a closed ball $\overline B$ on $M$ (equipped with a diffeomorphism $\phi: \overline{B(0,1)}\to \overline B$), we have a hat function $H_B(y):=H(\phi^{-1}(y))$ on $\overline B$, which we extend by zero to a smooth function 
on $M$ whose support is $\overline B$ and which is strictly positive in its interior $B\subset \overline{B}$.  

Now let $\lbrace\overline B_s, s\in J\rbrace$ be the collection of all closed balls in $M$ such that their interiors $B_s$ are contained in some $U_i$. Then $\lbrace B_s, s\in J\rbrace$ is clearly a base for $M$. Thus by Lemma \ref{locfi}, this base has a countable, locally finite subcover $\lbrace B_s, s\in S\rbrace$. Picking diffeomorphisms $\phi_s: \overline{B(0,1)}\to \overline B_s, s\in S$, we can define the smooth function $F(y):=\sum_{s\in S} H_{B_s}(y)$, which is strictly positive on $M$ since $B_s$ cover $M$ (this makes sense by the local finiteness). Now define the smooth functions 
$f_s(y):=\frac{H_{B_s}(y)}{F(y)}$. This collection 
is a partition of unity subordinate to the cover $\lbrace U_i\rbrace$, as desired. 
\end{proof} 

\section{\bf Integration on manifolds}

\subsection{Integration of top differential forms on oriented manifolds} 
An important operation with top degree differential forms is {\bf integration}. Namely, if 
$\omega$ is a differential $n$-form on an open set $U\subset \Bbb R^n$ (with the usual orientation), $\omega=f(x_1,...,x_n)dx_1\wedge...\wedge dx_n$, then we can set
$$
\int_{U}\omega:=\int_{U}f(x_1,...,x_n)dx_1...dx_n.
$$
(provided this integral is absolutely convergent). 
This, however, is not completely canonical: if we change coordinates 
(so that $U$ maps diffeomorphically to $U'$), 
the change of variable formula in a multiple integral tells us that 
$$
\int_{U}f(x_1,...,x_n)dx_1\wedge...\wedge dx_n=\int_{U'} f(x_1(\bold y),...,x_n(\bold y))\bigg|\det\left(\tfrac{\partial x_i}{\partial y_j}\right)\bigg| dy_1\wedge...\wedge dy_n,
$$
while the transformation law for $\omega$ is the same but without the absolute value. This shows that our definition is invariant only under orientation preserving transformations of coordinates, i.e., ones 
whose Jacobian $\det\left(\frac{\partial x_i}{\partial y_j}\right)$ is positive. 
Consequently, we will only be able to define integration of top differential forms 
on {\bf oriented manifolds}, i.e., ones equipped with an atlas of charts in which transition maps have a positive Jacobian; such an atlas defines an {\bf orientation} on $M$. To fix an orientation, we just need to say which local coordinate systems (or bases of tangent spaces) are right-handed, and do so in a consistent way. But this cannot always be done globally (the classic counterexamples are M\"obius strip and Klein bottle).
 
 Now let us proceed to define integration of a continuous top form $\omega$ over an oriented manifold $M$. For this pick an atlas of local charts $\lbrace U_i,i\in I\rbrace$ on $M$ and pick a partition of unity $\lbrace f_s\rbrace$ subordinate to this cover, which is possible by Proposition \ref{subor}. First assume 
that $\omega$ is nonnegative, i.e., $\omega(v_1,...,v_n)\ge 0$ for a right-handed basis 
$v_i$ of any tangent space of $M$. Then define 
\begin{equation}\label{defint}  
 \int_M \omega:=\sum_{s} \int_{U_{i(s)}}f_s\omega
 \end{equation} 
where in each $U_i$ we use a right-handed coordinate system 
to compute the corresponding integral. This makes sense (as a nonnegative real number or $+\infty$), and is also independent of  the choice of a partition of unity. Indeed, it is easy to see that for two atlases 
$\lbrace U_i\rbrace$, $\lbrace V_j\rbrace$ and two partitions of unity $\lbrace f_s\rbrace ,\lbrace g_t\rbrace$ the answer is the same, by comparing both to the answer for the atlas $\lbrace U_i\cap V_j\rbrace$ and partition of unity $\lbrace f_sg_t\rbrace$. In fact, this makes sense for any measurable $\omega$ (i.e., given by a measurable function in every local chart) if we use Lebesgue integration. 

Now, if $\omega$ is not necessarily nonnegative, we may define 
the nonnegative form $|\omega|$ which is $\omega$ at points where $\omega$ is nonnegative and $-\omega$ otherwise. Then, if 
$$
\int_M|\omega|<\infty, 
$$
we can define $\int_M\omega$ by the same formula \eqref{defint} which will now be 
a not necessarily positive but absolutely convergent series (a finite sum in the compact case).  

Importantly, the same definition works for manifolds $M$ with boundary $\partial M$ (an $n-1$-manifold); the only difference is that at boundary points the manifold locally looks like $\Bbb R^n_+$ (the space of vectors with nonnegative last coordinate) rather than $\Bbb R^n$. Note that the boundary of an oriented manifold carries a canonical orientation as well (a basis of $T_p\partial M$ is right-handed if adding at the beginning a tangent vector directed outside $M$ produces a right-handed basis of $T_pM$). 

\begin{remark} If the manifold $M$ is non-orientable, we cannot integrate 
top differential forms on $M$. However, we can integrate {\bf densities} on $M$, which are sections of the line bundle $|\wedge^nT^*M|$, the absolute value of the orientation bundle. This bundle is defined by transition functions $|g_{ij}(x)|$, where $g_{ij}(x)$ are the transition functions of $\wedge^nT^*M$. Thus its sections, called densities on $M$, transform under changes of coordinates according to the rule
$$
f(x_1,...,x_n)|dx_1\wedge...\wedge dx_n|=f(x_1(\bold y),...,x_n(\bold y))\vert\det\left(\tfrac{\partial x_i}{\partial y_j}\right)\vert\cdot |dy_1\wedge...\wedge dy_n|,
$$
i.e., exactly the one needed for the integral to be defined canonically. This procedure actually makes sense for any manifold, and in the oriented case reduces to integration of top forms described above. 

Using partitions of unity, it is not hard to show that the bundle $|\wedge^nT^*M|$ 
is trivial (check it!). A positive smooth section of this bundle (i.e., positive in every chart) therefore exists and is nothing but a {\bf positive smooth measure} on $M$, and any two such measures 
differ by multiplication by  a positive smooth function. Moreover, given such a measure $\mu$ 
and a measurable function $f$ on $M$ such that $\int_M |f|d\mu<\infty$ 
(i.e., $f\in L^1(M,\mu)$), we can define $\int_Mfd\mu$ as usual. 
\end{remark}

\subsection{Nonvanishing forms}  
Let us say that a top degree continuous differential form $\omega$ on $M$ is {\bf non-vanishing} 
if for any $x\in M$, $\omega_x\in \wedge^nT_x^*M$ is nonzero. 
In this case, $\omega$ defines an orientation on $M$ by 
declaring a basis $v_1,...,v_n$ of $T_xM$ right-handed 
if $\omega(v_1,...,v_n)>0$ (in particular, there are no non-vanishing top forms on non-orientable manifolds). Thus we can integrate top differential forms on $M$, and in particular $\omega$ defines a positive measure $\mu=\mu_\omega$ on $M$, namely
$$
\mu(U)=\int_U\omega
$$
for an open set $U\subset M$ (this integral may be $+\infty$, but is finite 
if $U$ is a small enough neighborhood of any point $x\in M$). 
Thus we can integrate functions on $M$ with respect to this measure: 
$$
\int_M fd\mu=\int_M f\omega. 
$$
This, of course, only makes sense if $f$ is measurable and 
$\int_M |f|d\mu<\infty$, i.e., if $f\in L^1(M,\mu)$. 
Note also that if $\lambda\in \Bbb R^\times$ then $\mu_{\lambda\omega}=|\lambda|\mu_\omega$. 
 
\begin{example} If $M$ is an open set in $\Bbb R^n$ with the usual orientation 
and $\omega=dx_1\wedge...\wedge dx_n$ then $\int_M\omega=\int_Mdx_1...dx_n$ is just the volume of $M$. For this reason top differential forms are often called {\bf volume forms}, especially when they are non-vanishing and thus define an orientation and a measure on $M$, and in the latter case $\int_M\omega$, if finite, is called the {\bf volume of $M$} with respect to $\omega$.  
\end{example}   
  
\begin{proposition} If $M$ is compact and $\omega$ is non-vanishing 
then $M$ has finite volume under the measure $\mu=\mu_\omega$, and every bounded measurable (in particular, any continuous) function on $M$ is in $L^1(M,\mu)$. 
\end{proposition}   
  
\begin{proof} For each $x\in M$ choose a neighborhood $U_x$ of $x$ such that $\mu(U_x)<\infty$. The collection of sets $U_x$ forms an open cover of $M$, so it has a finite subcover $U_1,...,U_N$, and $\mu(M)\le \mu(U_1)+...+\mu(U_N)<\infty$. 
Then $\int_M |f|d\mu\le\mu(M){\rm sup} |f| <\infty$ for bounded measurable $f$. 
\end{proof}   

\subsection{Stokes formula} 
  
A central result about integration of differential forms is  
 
\begin{theorem}(Stokes formula) If $M$ is an $n$-dimensional oriented manifold with boundary 
 and $\omega$ a compactly supported differential $n-1$-form on $M$ of class $C^1$ then 
 $$
 \int_M d\omega=\int_{\partial M}\omega.
 $$ 
 \end{theorem}  
 
In particular, if $M$ is closed (has no boundary) then $\int_M d\omega=0$, and 
if $\omega$ is closed ($d\omega=0$) then $\int_{\partial M}\omega=0$. 
 
 When $M$ is an interval in $\Bbb R$, this reduces to the fundamental theorem of calculus. If $M$ is a region in $\Bbb R^2$, this reduces to Green's formula. 
 If $M$ is a surface in $\Bbb R^3$, this reduces to the classical Stokes formula from vector calculus. Finally, if $M$ is a region in $\Bbb R^3$ then this reduces to 
 the Gauss formula (Divergence theorem). 
 
The proof of the Stokes formula is not difficult. Namely, by writing $\omega$ as $\sum_s f_s\omega$ for some partition of unity, it suffices to prove the formula for $M$ being a box in $\Bbb R^n$, which easily follows from the fundamental theorem of calculus.  
 
\subsection{Integration on Lie groups} Now let $G$ be a real Lie group of dimension $n$. In this case given any $\xi\in \wedge^n\g^*$, we can extend it to a left-invariant skew-symmetric tensor field (i.e., top differential form) $\omega_\xi$ on $G$. 
Also, if $\xi\ne 0$ then $\omega=\omega_\xi$ is non-vanishing and thus defines an orientation and a left-invariant positive measure $\mu_\omega$ on $G$.  Note that $\xi$ is unique up to scaling by a real number $\lambda\in \Bbb R^\times$. So, since $\mu_{\lambda \omega}=|\lambda|\mu_\omega$, we see that $\mu_\omega$ is defined uniquely up to scaling by positive numbers. 
This measure is called the {\bf left-invariant Haar measure} 
and we'll denote it just by $\mu_L$ (assuming that the normalization has been chosen somehow). 

In a similar way we can define the {\bf right invariant Haar measure} $\mu_R$ on $G$. One may ask if these measures coincide (or, rather, are proportional, since they are defined only up to normalization). This question is answered by the following proposition. 

Given a 1-dimensional real representation $V$ of a group $G$, let 
$|V|$ be the representation of $G$ on the same space 
with $\rho_{|V|}(g)=|\rho_V(g)|$, where $\rho: G\to {\rm Aut}(V)=\Bbb R^\times$. 

\begin{proposition} $\mu_L=\mu_R$ if and only if $|\wedge^n\g^*|$ (or, equivalently, $|\wedge^n\g|$) is a trivial representation of $G$. 
\end{proposition} 

\begin{proof} It is clear that $\mu_L=\mu_R$ if and only if 
the left-invariant top volume form $\omega$ on $G$ is also right invariant up to sign. 
This is equivalent to saying that $\omega$ is conjugation invariant up to sign, 
i.e., that $\omega_1\in \wedge^n\g^*$ is invariant up to sign under the action of $G$. 
This implies the statement. 
\end{proof} 

If $\mu_L=\mu_R$ then $G$ is called {\bf unimodular}. 
In this case we have a {\bf bi-invariant Haar measure} $\mu=\mu_L=\mu_R$ on $G$ (under some normalization).  

In particular, we see that if $G$ has no nontrivial continuous characters $G\to \Bbb R^+$ 
then it is unimodular.

\begin{example} If $G$ is a discrete countable group then $G$ is unimodular and $\mu$ 
is the counting measure: $\mu(U)=|U|$ (number of elements in $U$). 
\end{example} 

\begin{exercise} (i) Let us say that a finite dimensional real Lie algebra $\g$ 
of dimension $n$ is unimodular if $\wedge^n\g$ is a trivial representation of $\g$. 
Show that a connected Lie group $G$ is unimodular if and only if so is ${\rm Lie}G$. 

(ii) Show that a perfect Lie algebra (such that $\g=[\g,\g]$) is unimodular. In particular, a semisimple Lie algebra is unimodular. 

(iii) Show that a nilpotent (in particular, abelian) Lie algebra is unimodular. 

(iv) Show that if $\g_1,\g_2$ are unimodular then so is $\g_1\oplus \g_2$. Deduce that a reductive Lie algebra is unimodular.   

(v) Show that the Lie algebra of upper triangular matrices of size $n$ 
is {\bf not} unimodular for $n>1$. Give an example of 
a Lie algebra $\g$ and ideal $I$ such that $I$ and $\g/I$ are unimodular but $\g$ is not. 

(vi) Give an example of a non-unimodular Lie group $G$ such that 
its connected component of the identity $G^\circ$ is unimodular (try groups of the form $\Bbb Z\ltimes \Bbb R$). 
\end{exercise} 

For a unimodular Lie group $G$, we will sometimes denote the integral of a function $f$
with respect to the Haar measure by 
$$
\int_G f(g)dg. 
$$ 

\begin{proposition} A compact Lie group is unimodular. 
\end{proposition} 

\begin{proof} The representation of $G$ on $|\wedge^n\g^*|$  defines a continuous homomorphism $\rho: G\to \Bbb R^+$. Since $G$ is compact, the image $\rho(G)$ of $\rho$ is a compact subgroup of $\Bbb R^+$. But the only such subgroup is the trivial group. This implies the statement. 
\end{proof} 

Thus, on a compact Lie group we have a (bi-invariant) Haar measure $\mu$. 
Moreover, in this case $\int_Gd\mu={\rm Volume}(G)<\infty$, so we have a canonical normalization of $\mu$ by the condition that it is a probability measure: 
$$
\int_G d\mu=1.
$$
E.g., for finite groups this normalization is the averaging measure, 
which is $|G|^{-1}$ times the counting measure. This is the normalization we will use if $G$ is compact.   

\section{\bf Representations of compact Lie groups} 

\subsection{Unitary representations} \label{unrep}

Now we can extend to compact groups the result that representations of finite groups are unitary. Namely, let $V$ be a finite dimensional (continuous) complex representation of a compact Lie group $G$.  

\begin{proposition} $V$ admits a $G$-invariant unitary structure. 
\end{proposition} 

\begin{proof} Fix a positive Hermitian form $B$ on $V$ and define a new Hermitian form on $V$ by 
$$
B_{\rm av}(v,w)=\int_G B(\rho_V(g)v,\rho_V(g)w)dg.
$$
This form is well defined since $G$ is compact and is $G$-invariant by construction (since the measure $dg$ is invariant). Also $B_{\rm av}(v,v)>0$  for $v\ne 0$ 
since $B(w,w)>0$ for any $w\ne 0$. 
\end{proof} 

\begin{corollary} Every finite dimensional representation $V$ of a compact Lie group $G$ 
is completely reducible. 
\end{corollary} 

\begin{proof} Let $W\subset V$ be a subrepresentation and $B$ be an invariant positive Hermitian form on $V$. Let $W^\perp\subset V$ be the orthogonal complement of $W$ 
under $B$. Then $V=W\oplus W^\perp$, which implies the statement.  
\end{proof} 

In particular, this applies to the special unitary group $SU(n)$. Recall that 
$SU(n)/SU(n-1)=S^{2n-1}$, which implies that $SU(n)$ is simply connected. 
Thus (smooth) representations of $SU(n)$ are the same thing as representations
of the Lie algebra $\mathfrak{su}(n)$ or its complexification
$\mathfrak{sl}_n$. Thus we get a new, analytic proof that finite dimensional representations 
of $\mathfrak{sl}_n$ are completely reducible (this is called {\bf Weyl's unitary trick}). 
In fact, we will see that complete reducibility of finite dimensional representations of all semisimple Lie algebras can be proved in this way. 

\subsection{Matrix coefficients} Let $V$ be a finite dimensional continuous complex representation of a Lie group $G$. A {\bf matrix coefficient} of $V$ is a function $G\to \Bbb C$ of the form $(f,\rho_V(g)v)$ 
for some $v\in V$ and $f\in V^*$. Obviously, such a function is continuous. 

\begin{proposition} Matrix coefficients are smooth. 
\end{proposition} 

\begin{proof} Let us say that $v\in V$ is smooth if the function $f(\rho_V(g)v)$ is smooth for any $f\in V^*$; it is clear that such vectors form a subspace $V_{\rm sm}$ of $V$. Our job is to show that, in fact, $V_{\rm sm}=V$. To this end let us first construct some smooth vectors. For this let $\phi: G\to \Bbb C$ be a smooth function with compact support, and let 
$$
w=w(\phi,v):=\int_G \phi(g)\rho_V(g)vdg,
$$ 
where $dg$ is a left-invariant Haar measure on $G$ and $v\in V$. 
We claim that $w$ is a smooth vector. Indeed, 
$$
f(\rho_V(h)w)=f\left(\rho_V(h)\int_G \phi(g)\rho_V(g)vdg\right)=
$$
$$
\int_G f(\phi(g)\rho_V(hg)v)dg=
\int_G f(\phi(h^{-1}g)\rho_V(g)v)dg,
$$
and this is manifestly smooth in $h$ (we can differentiate indefinitely under the integral sign). 

Define a {\bf delta-like sequence} (or a {\bf Dirac sequence}) around a point $x_0\in M$ on a manifold $M$ with a smooth measure $dx$  to be a sequence of continuous functions $\phi_n$ on $M$ such that for every neighborhood $U$ of $x_0$ the supports of almost all $\phi_n$ are contained in $U$, and $\int_M \phi_n(x)dx=1$. The ``hat" function construction implies that delta-like sequences exist and can be chosen non-negative and smooth. Namely, we can pick a sequence of non-negative smooth functions satisfying the first condition and then normalize it to satisfy the second one. 

Now let $\phi_n$ be a smooth delta-like sequence around $1$ on $G$ with left-invariant Haar measure. Let $w_n:=w(\phi_n,v)$. It is obvious that $w_n\to v$ as $n\to \infty$. Thus $V_{\rm sm}$ is dense in $V$. Since $V$ is finite dimensional, it follows that $V_{\rm sm}=V$, as claimed. 
\end{proof} 

Now let $V$ be an irreducible representation of a compact Lie group $G$. 
As shown above, it has an invariant positive Hermitian inner product, which we'll denote by $(,)$. Moreover, this product is unique up to scaling. Pick an orthonormal basis $v_1,...,v_n$
of $V$ under this inner product, and let $v_1^*,...,v_n^*$ be the dual basis of $V^*$. Now consider the matrix coefficients of  
$V$ in this basis: 
$$
\psi_{V,ij}(g):=v_j^*(\rho_V(g)v_i)=(\rho_V(g)v_i,v_j). 
$$
Note that these functions are independent of  the normalization of $(,)$. 

Suppose now that we also have another such representation $W$ with orthonormal basis $w_i$.  
 
\begin{theorem} (Orthogonality of matrix coefficients) We have 
$$
\int_G\psi_{V,ij}(g)\overline{\psi_{W,kl}(g)}dg=0 
$$
if $V$ is not isomorphic to $W$. Also 
$$
\int_G\psi_{V,ij}(g)\overline{\psi_{V,kl}(g)}dg=\frac{\delta_{ik}\delta_{jl}}{\dim V}.
$$ 
\end{theorem} 

\begin{proof} We have 
$$
\int_G\psi_{V,ij}(g)\overline{\psi_{W,kl}(g)}dg=\int_G((\rho_V(g)\otimes \rho_{\overline W}(g))(v_i\otimes w_k),v_j\otimes w_l)dg=
$$
$$
(P(v_i\otimes w_k),v_j\otimes w_l)
$$
where
$$
P:=\int_G \rho_V(g)\otimes \rho_{\overline W}(g)dg=\int_G \rho_{V\otimes \overline{W}}(g)dg.
$$
Since $W$ is unitary, $\overline{W}\cong W^*$, so we have 
$$
P=\int_G \rho_{V\otimes W^*}(g)dg :V\otimes W^*\to V\otimes W^*. 
$$
By construction, ${\rm Im}(P)\subset (V\otimes W^*)^G$, which is zero if $V\ncong W$. 
Thus we have proved the proposition in this case. 

It remains to consider the case $V=W$. In this case $V\otimes W^*=V\otimes V^*=V\otimes \overline{V}$, and the only invariant 
in this space up to scaling is $\bold u:=\sum_k v_k\otimes v_k$. Also $P$ is conjugation invariant under $G$, so by decomposing $V\otimes V^*$ into irreducibles we see that it is the orthogonal projector to $\Bbb C\bold u$: 
$$
P\bold x=\frac{(\bold x,\bold u)}{(\bold u,\bold u)}\bold u=\frac{(\bold x,\bold u)\bold u}{\dim V}.
$$ 
In particular, 
$$
(P(v_i\otimes w_k),v_j\otimes w_l)=\frac{\delta_{ik}\delta_{jl}}{\dim V},
$$
as claimed.
\end{proof} 

\subsection{The Peter-Weyl theorem} 

Thus we see that the functions $\psi_{V,ij}$ for various $V,i,j$ form an orthogonal system in the Hilbert space $L^2(G)=L^2(G,dg)$ of measurable functions 
$f: G\to \Bbb C$ such that 
$$
||f||^2=\int_G |f(g)|^2dg<\infty. 
$$
A fundamental result about compact Lie groups is that this system is, in fact, complete: 

\begin{theorem}\label{pweyl} (Peter-Weyl theorem) The functions $\psi_{V,ij}$ form an orthogonal basis of $L^2(G)$. 
\end{theorem} 

Theorem \ref{pweyl} will be proved in Section \ref{ppw}. 

\subsection{An alternative formulation of the Peter-Weyl theorem} 
Given a finite dimensional irreducible representation $V$ of $G$, consider the space $\Hom_G(V,L^2(G))$ of $G$-homomorphisms for the action of $G$ on $L^2(G)$ by left translations. 
We have an obvious inclusion 
$$
\iota_V: V^*\hookrightarrow \Hom_G(V,L^2(G))
$$ 
via the matrix coefficient map $f\mapsto [v\mapsto (\rho_{V^*}(-)f)(v)]$. Clearly, this is a map of $G$-modules, where now $G$ acts on $L^2(G)$ by right translations.  
We claim that $\iota_V$ is surjective, i.e., an isomorphism. For this, note that an element $\phi\in \Hom_G(V,L^2(G))$ can be viewed a left $G$-equivariant $L^2$-function 
$\phi: G\to V^*$, i.e. such that for almost all $g\in G$ (with respect to the Haar measure) we have 
\begin{equation}\label{invaa}
\phi(x)=\rho_{V^*}(xg^{-1})\phi(g)
\end{equation} for almost all $x\in G$. But then by changing $\phi$ on a set of measure zero if needed, we may replace it by a continuous function (the right hand side of \eqref{invaa}). Then, setting $g=1$, we have $\phi(x)=\rho_{V^*}(x)\phi(1)$, as claimed.  

Let \scriptsize
$$
\iota=\oplus_{V\in {\rm Irrep}(G)}\sqrt{\dim V}{\rm Id}_V\otimes \iota_V: \oplus_{V\in {\rm Irrep}(G)}V\otimes V^*\to \oplus_{V\in {\rm Irrep}(G)}V\otimes \Hom_G(V,L^2(G)). 
$$ \normalsize
Then $\iota$ defines an isometric embedding of $G\times G$-modules
$$
\xi: \oplus_{V\in {\rm Irrep}(G)}V\otimes V^*\hookrightarrow L^2(G).
$$
We will denote the image of $\xi$ by $L^2_{\rm alg}(G)$ (the ``algebraic part" of $L^2(G)$). Note that if $\psi\in L^2(G)$ generates a finite dimensional representation $V$ under the action of $G$ by left translations then $\psi$ belongs to the image of a homomorphism $V\to L^2(G)$, hence to $L^2_{\rm alg}(G)$. Thus $L^2_{\rm alg}(G)$ is just the subspace of $\psi\in L^2(G)$ which generate 
a finite dimensional representation under left translations by $G$. We also see that 
it may be equivalently characterized as the subspace of $\psi\in L^2(G)$ which generate 
a finite dimensional representation under right translations by $G$.

\begin{theorem} (Peter-Weyl theorem, alternative formulation) The space $L^2_{\rm alg}(G)$ is dense in $L^2(G)$. 
In other words, the map $\xi$ gives rise to an isomorphism 
$$
\widehat\oplus_{V\in {\rm Irrep}(G)}V\otimes V^*\to L^2(G)
$$
where the first copy of $G$ acts on $V$ and the second one on $V^*$ and the hat denotes the Hilbert space completion of the direct sum. 
\end{theorem} 

Note that this is again an instance of the double centralizer property! Namely, it expresses representation-theoretically 
the fact that the centralizer of the group of left translations on $G$ is the group of right translations on $G$, and vice versa. 

For example, let $G=S^1$. Then the irreducible representations of $G$ are the characters 
$\psi_n(\theta)=e^{in\theta}$. So the Peter-Weyl theorem in this case says that 
$\lbrace e^{in\theta}\rbrace$ is an orthonormal basis of $L^2(S^1)$ with norm
$$
||f||^2:=\frac{1}{2\pi}\int_0^{2\pi}|f(\theta)|^2d\theta, 
$$
which is the starting point for Fourier analysis. So the Peter-Weyl theorem is similarly a starting point for {\bf nonabelian Fourier (or harmonic) analysis}. 

\begin{exercise}\label{homospace} Let $G$ be a compact Lie group and $H\subset G$ a closed subgroup. 
Then we have a compact homogeneous space $G/H$ and the Haar measure on $G$ defines a probability measure on $G/H$. So we can define 
the infinite dimensional unitary representation $L^2(G/H)$ of $G$. 

(i) Show that we have a decomposition 
$$
L^2(G/H)=\widehat\oplus_{V\in {\rm Irrep}G}N_H(V)V,
$$
where $N_H(V)=\dim V^H$, the dimension of the space of $H$-invariants of 
$V$. 

(ii) Let $G=SO(3)$, so the irreducible representations are $L_{2m}$ for $m\ge 0$. 
Thus 
$$
L^2(G/H)=\widehat\oplus_{m\ge 0}N_H(m)L_{2m}.
 $$
Compute this decomposition (i.e., the numbers $N_H(m)$) for $H=\Bbb Z/n\Bbb Z$ acting by rotations around an axis by angles $2\pi k/n$ (rotations of a regular $n$-gon).  

(iii) Do the same for the dihedral group $H=\bold D_n$ of symmetries of the regular $n$-gon (where reflections in the plane are realized as rotations around a line in this plane).

(iv) Do the same for the groups $H=SO(2)$ and $H=O(2)$ of rotations and symmetries of the circle. 
 
(v) Do the same for $H$ being the group of symmetries of a platonic solid (tetrahedron, cube, icosahedron).  

It may be more convenient to give $N_H(m)$ in the form of the generating function 
$\sum_m N_H(m)t^m$. 
\end{exercise} 

\begin{exercise} Let $G=GL_n(\Bbb C)$. A {\bf regular algebraic function} on $G$ 
is a polynomial of $X_{ij}$ and $\det(X)^{-1}$ for $X\in G$. Denote by 
$\mathcal{O}(G)$ the algebra of regular algebraic functions on $G$. 

(i) Show that $G\times G$ acts on $\mathcal{O}(G)$ by left and right multiplication. 

(ii) (Algebraic Peter-Weyl theorem) Show that as a $G\times G$-module, we have 
$$
\mathcal O(G)=\oplus_{V\in {\rm Irrep}(G)}V\otimes V^*,
$$
where ${\rm Irrep}G$ is the set of isomorphism classes of irreducible algebraic representations of $G$. 

{\bf Hint.} Compute $\Hom_G(V,\mathcal O(G))$ where $G$ acts on $\mathcal O(G)$ by right translations. For this, interpret elements of this space as equivariant functions $G\to V^*$ and show that such functions are automatically regular algebraic. 

(iii) Generalize (i) and (ii) to orthogonal and symplectic groups.  
\end{exercise} 

\subsection{Orthogonality and completeness of characters} 
\begin{corollary} Let $\chi_V(g)={\rm Tr}(\rho_V(g))$ be the character of $V$. 
Then $\lbrace{\chi_V(g),V\in {\rm Irrep}G\rbrace}$ is an orthonormal basis of $L^2(G)^G$, 
the space of conjugation-invariant functions in $L^2(G)$ (i.e., such that $f(gxg^{-1})=f(x)$).  
\end{corollary} 

\begin{proof} We have $\chi_V(g)=\sum_i \psi_{V,ii}(g)$, so 
by orthogonality of matrix coefficients $\chi_V$ 
are orthonormal in $L^2(G)^G$. So it remains 
to show that they are complete. For this observe 
that $L^2_{\rm alg}(G)^G=\xi(\oplus_V (V\otimes V^*)^G)=\oplus_V \Bbb C\chi_V$. 
Thus our job is to show that $L^2_{\rm alg}(G)^G$ is dense in $L^2(G)^G$. 
To this end, for $\psi\in L^2(G)^G$ fix a sequence $\psi_n\in L^2_{\rm alg}(G)$
such that $\psi_n\to \psi$ as $n\to \infty$. Such a sequence exists 
by the Peter-Weyl theorem. Let 
$$
\psi_n^{\rm av}(x)=\int_G\psi_n(gxg^{-1})dg.
$$
It is easy to see that $\psi_n^{\rm av}\in L^2_{\rm alg}(G)$. 
Also 
$||\psi_n^{\rm av}-\psi||\le ||\psi_n-\psi||\to 0$, $n\to \infty$, as claimed. 
\end{proof} 

\section{\bf Proof of the Peter-Weyl theorem} \label{ppw} 

\subsection{Compact operators and the Hilbert-Schmidt theorem} 

To prove the Peter-Weyl theorem, we will use the Hilbert-Schmidt theorem -- the spectral theorem for compact self-adjoint operators in a Hilbert space. 

Recall that a {\bf bounded} operator $A: H\to H$ on a Hilbert space $H$ is a linear operator such that for some $C\ge 0$ we have $||A\bold v||\le C||\bold v||$, $\bold v\in H$. The smallest constant $C$ with this property is called the {\bf norm} of $A$ and denoted $||A||$.  Recall also that $A$ is {\bf compact} if there is a sequence of finite rank operators $A_n: H\to H$ such that $||A_n-A||\to 0$ as $n\to \infty$. In other words, the space $K(H)$ of compact operators on $H$  is the closure of the space $K_f(H)$ of finite rank operators under the norm $A\mapsto ||A||$ on the space of bounded operators $B(H)$. 

\begin{lemma}\label{conve} If $A$ is compact then it maps bounded sets to pre-compact sets (i.e., ones whose closure is compact). In other words, for every bounded sequence 
$\bold v_n\in H$, the sequence $A\bold v_n$ has a convergent subsequence.\footnote{The converse statement also holds, but we will not need it.}  
\end{lemma} 

\begin{proof} Let $\bold v_n\in H$, $||\bold v_n||\le 1$. 
Pick a sequence of finite rank operators $A_n$ such that $||A_n-A||<\frac{1}{n}$. 
Let $\bold v_n^1$ be a subsequence of $\bold v_n$ such that $A_1\bold v_n^1$ is convergent. 
Let $\bold v_n^2$ be a subsequence of $\bold v_n^1$ such that $A_2\bold v_n^2$ is convergent, and so on. Finally, let $\bold w_n=\bold v_n^n$. Note that 
$$
||A\bold v_i^k-A\bold v_j^k||\le ||A_k\bold v_i^k-A_k\bold v_j^k||+||A-A_k||\cdot ||\bold v_i^k-\bold v_j^k||
$$
$$
\le ||A_k\bold v_i^k-A_k\bold v_j^k||+\tfrac{2}{k}-\varepsilon_k.
$$
for some $\varepsilon_k>0$. 
Since $A_k\bold v_i^k, i\ge 1$ is convergent, it is a Cauchy sequence, so
 there is $M_k$ such that for $i,j\ge M_k$, $||A_k\bold v_i^k-A_k\bold v_j^k||<\varepsilon_k$, hence
$$
||A\bold v_i^k-A\bold v_j^k||< \tfrac{2}{k}. 
$$
But $\bold w_n$ is a subsequence of $\bold v_n^k$ starting from the $k$-th term. So 
there is $N_k$ such that 
$$
||A\bold w_i-A\bold w_j||< \tfrac{2}{k},\ i,j\ge N_k.
$$
In other words, the sequence $A\bold w_n$ is Cauchy. Hence it is convergent, as desired. 
\end{proof} 

\begin{proposition}\label{compa} Let $M$ be a compact manifold with positive smooth probability measure $d\bold x$
and $K(\bold x,\bold y)$ a continuous function on $M\times M$. Then the operator 
$$
(A\psi)(\bold y):=\int_M K(\bold x,\bold y)\psi(\bold x)d\bold x.
$$
on $L^2(M)$ is compact.
\end{proposition} 

\begin{proof} By using a partition of unity, the problem can be reduced to the case when $M$ is replaced by the hypercube $[0,1]^n$. Let us split it in $m^n$ pixels of sidelength $\frac{1}{m}$ and approximate $K(\bold x,\bold y)$ by its value in the midpoint of each of the $m^{2n}$ pixels in 
$[0,1]^{2n}$. Denote the corresponding approximation by $K_m(\bold x,\bold y)$ and the corresponding operator by $A_m$; it has rank $\le m^n$. Let $\varepsilon_m:={\rm sup}|K-K_m|$, then $||A-A_m||\le \varepsilon_m$. Finally, by Cantor's theorem,\footnote{Cantor's theorem says that any continuous function on a compact set $X$ is uniformly continuous.}  
$K$ is uniformly continuous, which implies that $\varepsilon_m\to 0$ as $m\to \infty$, hence the statement. 
\end{proof} 
 
 Recall that a bounded operator $A$ is {\bf self-adjoint} if $(A\bold v,\bold w)=(\bold v,A\bold w)$ for $\bold v,\bold w\in H$.

\begin{theorem} (Hilbert-Schmidt) Let $A: H\to H$ be a compact self-adjoint operator. 
Then there is an orthogonal decomposition 
$$
H={\rm Ker}A\oplus \widehat \bigoplus_\lambda H_\lambda,  
$$
where $\lambda$ runs over non-zero eigenvalues of $A$, and $A|_{H_\lambda}=\lambda\cdot {\rm Id}$. Moreover, the spaces $H_\lambda$ are finite dimensional and the eigenvalues $\lambda$ are real and either form a finite set or a sequence going to $0$. 
\end{theorem} 

Note that for finite rank operators, this obviously reduces to the standard theorem in linear algebra: a self-adjoint (Hermitian) operator on a finite dimensional space $V$ with a positive Hermitian form has an orthogonal eigenbasis, and its eigenvalues are real.  

\begin{proof} We first prove the theorem for the operator $A^2$. 
Let $\beta:=||A||^2={\rm sup}_{||\bold v||=1}(A^2\bold v,\bold v)\ge 0$. We may assume without loss of generality that $\beta\ne 0$. Let $A_n$ be a sequence of self-adjoint finite rank operators converging to $A$, and let $\beta_n=||A_n||^2$, which is also the maximal eigenvalue of $A_n^2$. We have $\beta_n\to \beta$. Let $\bold v_n$ 
be a sequence of unit vectors in $H$ such that $A_n^2\bold v_n=\beta_n \bold v_n$. 
By Lemma \ref{conve}, the sequence $A^2\bold v_n$ has a convergent subsequence, so passing to this subsequence we may assume that $A^2\bold v_n$ is convergent to some $\bold w\in H$. Hence $A_n^2\bold v_n\to \bold w$, so $\bold v_n\to \beta^{-1}\bold w$. Thus $A^2\bold w=\beta \bold w$. 
We can now replace $H$ with the orthogonal complement of $\bold w$ and iterate this procedure. 

As a result we'll get a sequence of numbers $\beta_1> \beta_2>....> 0$, 
which is either finite (in which case the theorem is obvious) or tends to $0$ (by compactness of $A^2$), and the corresponding sequence of finite dimensional orthogonal eigenspaces $H_{\beta_k}$ (also by compactness of $A^2$). Let $\bold v$ be a vector orthogonal to all $H_{\beta_k}$. Then $||A\bold v||^2\le \beta_k ||\bold v||^2$ for all $k$, so if $\beta_k$ is an infinite sequence going to $0$, it follows that $A\bold v=0$, as desired. 
 
Now, we have $H={\rm Ker}A^2\oplus \widehat \bigoplus_n H_{\beta_n}$, and 
$A$ preserves this decomposition, acting by $0$ on  ${\rm Ker}A^2$ and with eigenvalues 
$\pm \sqrt{\beta_n}$ on $H_{\beta_n}$. This implies the theorem. 
\end{proof} 

\subsection{Proof of the Peter-Weyl theorem} 

Let $G$ be a compact Lie group and $h_N$ a delta-like sequence around $1$ on $G$. 
By replacing $h_N(x)$ with $\frac{1}{2}(h_N(x)+h_N(x^{-1}))$, we may assume that 
$h_N$ is invariant under inversion. 
Define the {\bf convolution operators} $B_N$ on $L^2(G)$ by 
$$
(B_N\psi)(y)=\int_G h_N(x)\psi(x^{-1}y)dx=\int_G h_N(yz^{-1})\psi(z)dz.
$$
By Proposition \ref{compa}, these operators are compact (as the kernel $K(y,z):=h_N(yz^{-1})$ is continuous). Moreover, they are clearly self-adjoint (as $h_N(x)=h_N(x^{-1})$ and $h_N$ is real) and commute with right translations by $G$. So by the Hilbert-Schmidt theorem, we have the corresponding spectral decomposition 
$$
L^2(G)={\rm Ker}B_N\oplus \widehat\bigoplus_\lambda H_{N,\lambda}
$$
invariant under right translations. Since $H_{N,\lambda}$ are finite dimensional and invariant under right translations, they are contained in $L^2_{\rm alg}(G)$ (this is the key step of the proof). 
Thus the closure $\overline{L^2_{\rm alg}(G)}$ contains the image of $B_N$. So for any $\psi\in L^2(G)$ we can find $\psi_N\in L^2_{\rm alg}(G)$ 
such that $||B_N\psi-\psi_N||<\frac{1}{N}$. 

Now let $\psi\in C(G)$. By Cantor's theorem, $\psi$ is uniformly continuous. 
It follows that $B_N\psi$ uniformly converges to $\psi$ as $N\to \infty$ (check it!). Thus
$$
||\psi-\psi_N||\le ||\psi-B_N\psi||+||B_N\psi-\psi_N||<||\psi-B_N\psi||+\tfrac{1}{N}\to 0
$$ 
as $N\to \infty$. So $\overline{L^2_{\rm alg}(G)}$ contains $C(G)$. 
But $C(G)$ is dense in $L^2(G)$ (namely, by using a partition of unity
this reduces to the case of a box in $\Bbb R^n$, where it is well known). 
Thus $\overline{L^2_{\rm alg}(G)}=L^2(G)$. This completes the proof 
of the Peter-Weyl theorem. 

\subsection{Existence of faithful representations} 

\begin{lemma}\label{nes} Let $G$ be a compact Lie group and $G=G_0\supset G_1\supset...$ 
be a nested sequence of closed subgroups without repetitions. Then this sequence is finite.\end{lemma} 

\begin{proof}  Assume the contrary, i.e. that it is infinite. The dimensions must stabilize, so we may assume that $\dim G_n$ are all the same. Then $K=G_n^\circ$ is independent of  $n$, and we have a nested sequence 
$$
G_0/K\supset G_1/K\supset...
$$ 
of finite groups, without repetitions. But such a sequence can't have length bigger than $|G_0/K|$, contradiction. 
\end{proof}  

\begin{corollary}\label{exfaith} Any compact Lie group has a faithful finite dimensional representation, 
so it is isomorphic to a closed subgroup of the unitary group $U(n)$. 
\end{corollary} 

\begin{proof} Pick a nontrivial finite dimensional representation $V_1$ of $G=G_0$, and let $G_1$ be the kernel of this representation. Now pick another representation $V_2$ of $G$ which is nontrivial 
as a $G_1$-representation, and let $G_2$ be the kernel of $V_2$ in $G_1$, and so on. By Lemma \ref{nes}, at some point we will have a subgroup $G_k\subset G$ such that every finite dimensional representation of $G$ is trivial when restricted to $G_k$. But then by the Peter-Weyl theorem, $G_k$ acts trivially on $L^2(G)$, so $G_k=1$. Thus $V_1\oplus...\oplus V_k$ is a faithful $G$-representation. 
\end{proof} 

\begin{remark} Conversely, any closed subgroup of $U(n)$ is a compact Lie group, see Exercise \ref{closedlieex} below. 
\end{remark} 

\begin{remark} Corollary \ref{exfaith} is false for non-compact Lie groups, even for connected ones. For example, 
let $G$ be the universal cover of $SL_2(\Bbb R)$ (it has fiber $\Bbb Z=\pi_1(SL_2(\Bbb R)))$. Indeed, any finite dimensional continuous representation $V$ of $G$ is smooth, so gives a finite dimensional representation of the Lie algebra $\mathfrak{sl}_2(\Bbb R)$, hence of $\mathfrak{sl}_2(\Bbb C)$, which is therefore a direct sum of $L_n$. So $V$ exponentiates to $SL_2(\Bbb C)$, and thus its restriction to $\mathfrak{sl}_2(\Bbb R)$ exponentiates to $SL_2(\Bbb R)$, so is not faithful for $G$. 
\end{remark} 

\begin{exercise} Show that any compact Lie group admits a structure 
of a metric space such that the metric is invariant under left and right translations.
\end{exercise} 

\subsection{Density in continuous functions} 

In fact, we can now prove an even stronger version of the Peter-Weyl theorem. For this note that $L^2_{\rm alg}(G)$ is a unital algebra.

\begin{theorem}\label{densco} The algebra $L^2_{\rm alg}(G)$ is dense in the algebra of continuous functions
$C(G)$ in the supremum norm
$$
||f||=\max_{g\in G}|f(g)|.
$$  
\end{theorem} 

\begin{proof} Consider the closure $\mathcal A$ of $L^2_{\rm alg}(G)$ inside $C(G)$ (under the supremum norm).  Then $\mathcal A$ is a closed subalgebra invariant under complex conjugation, and by Corollary \ref{exfaith} it separates points on $G$.  Therefore, by the Stone-Weierstrass theorem, $\mathcal A=C(G)$.  
\end{proof} 

\begin{remark} If $G=S^1$, this is the usual theorem of uniform approximation of continuous functions on the circle by trigonometric polynomials. If we restrict to even functions, this will be just the usual Weierstrass theorem on approximation of continuous functions on an interval by polynomials. 
\end{remark} 

\begin{corollary}\label{corooo} Let $A\subset L^2_{\rm alg}(G)$ be a left-invariant subalgebra stable under complex 
conjugation and separating points on $G$. Then $A=L^2_{\rm alg}(G)$. 
\end{corollary} 

\begin{proof} By the Stone-Weierstrass theorem, $A$ is dense in $C(G)$ in uniform metric, hence in $L^2(G)$ 
in the Hilbert norm. Thus for every irreducible representation $V$ of $G$, 
$\Hom_G(V,A)$ must be dense in the space $\Hom_G(V,L^2(G)_{\rm left})=V^*$. So $\Hom_G(V,A)=V^*$, hence $A=L^2_{\rm alg}(G)$.   
\end{proof} 

Let us call a finite dimensional representation $V$ of a group $G$ {\bf unimodular} if $\wedge^{\dim V}V\cong \Bbb C$ is the trivial representation. 

\begin{proposition}\label{dirsum} Let $V$ be a faithful finite dimensional representation of a compact Lie group $G$. Then:

(i) If $V$ is unimodular then the subalgebra $A\subset C(G)$ generated by matrix coefficients $f(\rho_V(g)v)$, $v\in V$, $f\in V^*$, coincides with $L^2_{\rm alg}(G)$. 

(ii) If $Y$ is an irreducible finite dimensional representation of $G$, then for some $n,m$, the representation $Y$ is contained as a direct summand in $V^{\otimes n}\otimes V^{*\otimes m}$. 
Moreover, if $V$ is unimodular then one may take $m=0$. 
\end{proposition} 

\begin{proof} (i) Let $d:=\dim(V)$. It is clear that $A\subset L^2_{\rm alg}(G)$ is $G$-invariant and $A$  
separates points on $G$, since $V$ is faithful.  Also choosing a $G$-invariant unitary structure on $V$ we can realize $G$ as a closed subgroup of $SU(V)\subset V\otimes V^*$ (as $V$ is unimodular), and for a unitary matrix with determinant $1$ one has $g^\dagger=g^{-1}=\wedge^{d-1}g$. Thus $A$ is invariant under complex conjugation. So by Corollary \ref{corooo} 
$A=L^2_{\rm alg}(G)$. 

(ii) It suffices to establish the unimodular case since in general we may replace 
$V$ with the unimodular representation $V\oplus V^*$. But then by (i), $L^2_{\rm alg}(G)$ is a quotient 
of $S(V\otimes V^*)$, which implies the statement.    
\end{proof} 

\begin{exercise}\label{closedlieex} In this exercise you will show that a closed subgroup of a Lie group $G$ 
is a closed Lie subgroup (Theorem \ref{closedlie}). 

Clearly, it suffices to assume that $G$ is connected. 
Let $\g={\rm Lie}G$ and $H\subset G$ be a closed subgroup. 

(i) Let $\h$ be the set of vectors $a\in \g$ such that 
there is a sequence $h_n\in H$, $h_n\to 1$, and nonzero real numbers $c_n$ such that 
$$
c_n\log h_n\to a,\ n\to \infty. 
$$
This is clearly a subset of $\g$ invariant under scalar multiplication (since we can rescale $c_n$). Show that $\h$ consists of all 
$a\in \g$ for which the \linebreak 1-parameter subgroup $\exp(ta)$ is contained in $H$. 
(Consider the elements $h_n^{[c_n]}$, where $[c]$ is the floor of $c$). 

(ii) Show that $\h$ is a subspace of $\g$. (For $a,b\in \h$ consider the elements 
$h_N:=\exp(\frac{a}{N})\exp(\frac{b}{N})$ to show that $a+b\in \h$). 

(iii) Show that $\h$ is a Lie subalgebra of $\g$. (For $a,b\in \h$ consider 
the elements 
$$
h_N:=\exp(\tfrac{a}{N})\exp(\tfrac{b}{N})\exp(-\tfrac{a}{N})\exp(-\tfrac{b}{N})
$$
to show that $[a,b]\in \h$). 

(iv) Let $H_0\subset G$ be the connected Lie subgroup with Lie algebra $\h$. 
Given a sequence $h_N\in H$, $h_N\to 1$, show that $h_N\in  H_0$ for $N\gg 1$. 
To this end, pick a transverse slice $S\subset G$ to $H_0$ near $1$, and write
$h_N=s_Nh_{N,0}$, where $h_{N,0}\in H_0$, $s_N\in S$. Look at the asymptotics of $\log s_N$ 
as $N\to \infty$, and deduce that $s_N=1$ for large enough $N$. 

(v) Conclude that $G/H$ is a manifold, and $S$ defines a local chart on this manifold 
near $1$. Deduce that $H$ is a closed Lie subgroup of $G$, and $H_0=H^\circ$. 
\end{exercise} 

\section{\bf Representations of compact topological groups} 

\subsection{Existence of the Haar measure} 
One can generalize integration theory to arbitrary compact and even to locally compact topological groups. For simplicity we will describe this generalization in the case of compact topological groups with a countable base. 

Namely, let $X$ be a compact Hausdorff topological space with a countable base. For compact Hausdorff spaces this is equivalent to being metrizable. Let $C(X,\Bbb R)$ be the space of continuous real-valued functions on $X$. This is a real Banach space with norm 
$$||f||=\max_{x\in X}|f(x)|.$$ Recall that by the {\bf Riesz-Markov-Kakutani representation theorem}, 
a finite Borel measure $\mu$ on $X$ is the same thing as a positive continuous linear functional $I: C(X,\Bbb R)\to \Bbb R$ (i.e., such that $I(f)\ge 0$ for $f\ge 0$), namely, 
$$
I(f)=\int_X fd\mu.
$$
Moreover, $\mu$ is a probability measure if and only if $I(1)=1$, and any $\mu\ne 0$ 
has positive volume and so can be normalized to be a probability measure. 

Now let $G$ be a compact topological group with a countable base. It acts on $C(G,\Bbb R)$ by left and right translations, so acts on nonnegative probability measures of $G$. 

\begin{theorem} (Haar, von Neumann) $G$ admits a unique left-invariant probability measure.  
\end{theorem} 

This measure is also automatically right-invariant (since it is unique) and 
is called the {\bf Haar measure} on $G$. 

\begin{remark} A unique up to scaling left-invariant regular Haar measure (albeit of infinite volume and not always right-invariant in the non-compact case) exists more generally for any locally compact group $G$ (not necessarily having a countable base).\footnote{Note that a finite Borel measure on a compact Hausdorff space with a countable base is necessarily regular.} We will not prove this here, but we remark that Haar measures on Lie groups that we have constructed using top differential forms are a special case of this. 
\end{remark} 

\begin{proof} Let $g_i, i\ge 1$ be a dense sequence in $G$ (it exists since $G$ has a countable base, hence is separable, as you can pick a point in every open set of this base). Let $p_i$ be a sequence of positive numbers such that $\sum_i p_i=1$. To this data attach the {\bf averaging operator} \linebreak $A: C(G,\Bbb R)\to C(G,\Bbb R)$ given by 
$$
(Af)(x)=\sum_i p_if(xg_i).
$$
This operator can be interpreted as follows: we have a Markov chain with states being points of 
$G$ and the transition probability from $x$ to $xg_i$ equal to $p_i$, then $(Af)(x)$ 
is the expected value of $f$ after one transition starting from $x$. 
It is clear that $A$ is a left-invariant bounded operator (of norm $1$). Moreover, 
$A$ acts by the identity on the line $L\subset C(G,\Bbb R)$ 
of constant functions. 

For $f\in C(G,\Bbb R)$ denote by 
$\nu(f)$ the distance from $f$ to $L$, i.e., 
$$
\nu(f)=\tfrac{1}{2}(\max f-\min f). 
$$
Then $\nu(Af)<\nu(f)$ unless $f\in L$. Indeed, if $f$ is not constant and $x\in G$, pick 
$j$ such that $f(xg_j)<\max f$ (exists since the sequence $xg_i$ is dense in $G$), then 
$$
(Af)(x)=\sum_i p_if(xg_i)\le (1-p_j)\max f+p_jf(xg_j)<\max f.
$$ 
So $\max(Af)<\max f$.
Similarly, $\min(Af)>\min f$. 

Now fix $f\in C(G,\Bbb R)$ and consider the sequence $f_n:=A^nf$, $n\ge 0$. 
This means that we let our Markov chain run for $n$ steps. We know 
that for finite Markov chains there is an asymptotic distribution, 
and we'll show that this is also the case in the situation at hand, giving rise to a construction of the invariant integral. 

Obviously, the sequence $f_n$ is uniformly bounded by $\max |f|$. 
Also it is {\bf equicontinuous}: for any 
$\varepsilon>0$ there exists a neighborhood \linebreak $1\in U\subset G$
such that for any $x\in G$ and $u\in U$, 
$$
|f_n(x)-f_n(ux)|<\varepsilon.
$$ 
Indeed, it suffices to show that $f$ is uniformly continuous, i.e., 
for any $\varepsilon$ find $U$ such that for all $x\in G,u\in U$ we have 
$|f(x)-f(ux)|<\varepsilon$; this $U$ will then work for all $f_n$. 
But this is guaranteed by Cantor's theorem. Namely, 
assume the contrary, that there is no such $U$. Then there are two sequences 
$x_i,u_i\in G$, $u_i\to 1$, with $|f(x_i)-f(u_ix_i)|\ge \varepsilon$.  
The sequence $x_i$ has a convergent subsequence, so 
we may assume without loss of generality that $x_i\to x\in G$. 
Then taking the limit $i\to \infty$, we get that $\varepsilon\le 0$, a contradiction.

Therefore, by the {\bf Ascoli-Arzela theorem} the sequence 
$f_n$ has a convergent subsequence. Let us remind the proof of this theorem. 
We construct subsequences $f_n^k$ of $f_n$ inductively by picking $f_n^k$ from 
$f_n^{k-1}$ so that $f_n^k(g_k)$ converges (with $f_n^0=f_n$), which can be done by the boundedness assumption, and then set $h_m:=f_m^m=f_{n(m)}$. Then $h_m(g_i)$ converges, hence Cauchy, for all $i$, which by equicontinuity implies that $h_m(x)$ is a Cauchy sequence in $C(G,\Bbb R)$, hence converges to some $h\in \Bbb C(G,\Bbb R)$. 

We claim that $h\in L$. Indeed, we have 
$$
\nu(f_{n(m)})\ge \nu(f_{n(m)+1})=\nu(Af_{n(m)})\ge \nu(f_{n(m+1)}),
$$
so taking the limit when $m\to \infty$, we get 
$$
\nu(h)\ge \nu(Ah)\ge \nu(h),
$$
i.e., $\nu(Ah)=\nu(h)$. Moreover, since $\max A^nf$ decreases, $\min A^nf$ increases, and any subsequential limit of $A^nf$ is constant, $h$ is independent of the chosen convergent subsequence, i.e., the whole sequence $A^n f$ converges uniformly to $h$. The assignment $f\mapsto h$ is therefore 
a continuous left-invariant positive linear functional $I: C(G,\Bbb R)\to L=\Bbb R$, 
and $I(1)=1$, as claimed. 

Similarly, we may construct a right-invariant integral 
$$
I_*: C(G,\Bbb R)\to L=\Bbb R
$$ 
with $I_*(1)=1$, and 
by construction for any left invariant integral $J$ we have $J(f)=J(I_*(f))$. Thus for every 
left invariant integral $J$ with $J(1)=1$ we have $J(f)=I_*(f)$; in particular $I(f)=I_*(f)$.  
This shows that $I$ is unique, invariant on both sides and independent of  the choice of $g_i,p_i$, and hence that $A^nf\to I(f)$ as $n\to \infty$. 
\end{proof} 

\begin{example} A basic example of a compact topological group with countable base which is, in general, not a Lie group, is a {\bf profinite group}. Namely, let 
$G_0,G_1,...$ be finite groups and $\phi_i: G_{i+1}\to G_i$ be surjective homomorphisms.  
Then the {\bf inverse limit} $G:=\underleftarrow{\rm lim} G_n$ is the group 
consisting of sequences $g_0\in G_0,g_1\in G_1,...$ where 
$\phi_i(g_{i+1})=g_i$. This group $G$ has projections 
$p_n: G\to G_n$ and a natural topology, for which 
a base of neighborhoods of $1$ consists of ${\rm Ker}(p_n)$. 
(This topology can be defined by a bi-invariant metric: 
$d(\bold a,\bold b)=C^{n(\bold a,\bold b)}$, where $n(\bold a,\bold b)$ 
is the first position at which $\bold a,\bold b$ differ, and $0<C<1$). 
A sequence $\bold a^n$ converges to $\bold a$ in this topology if 
for each $k$, $a^n_k$ eventually stabilizes to $a_k$. 
It is easy to show that $G$ is compact. 

Profinite groups are ubiquitous in mathematics. For example, 
the $p$-{\bf adic integers} $\Bbb Z_p$ for a prime $p$ form a profinite group, namely the inverse limit of $\Bbb Z/p^n\Bbb Z$; in fact, it is a profinite ring. The multiplicative group of this ring $\Bbb Z_p^\times$ is also a profinite group. One may also consider non-abelian profinite groups $GL_n(\Bbb Z_p)$, $O_n(\Bbb Z_p)$, $Sp_{2n}(\Bbb Z_p)$, etc. 
Finally, absolute Galois groups, such as ${\rm Gal}(\overline{\Bbb Q}/\Bbb Q)$, are (very complicated) profinite groups. 

Note that infinite profinite groups are uncountable and {\bf totally disconnected}, i.e., $G^\circ=1$. 

More generally, the inverse limit makes sense if $G_i$ are compact Lie groups. 
In this case $G$ is equipped with the product topology, so also compact (by Tychonoff's theorem). 
For example, consider the sequence of Lie groups $G_n=\Bbb R/\Bbb Z$ and maps 
$\phi_i: G_{i+1}\to G_i$ given by $\phi_i(x)=px$ for a prime $p$. We can realize $G_n$ as $\Bbb R/p^n\Bbb Z$, then  
$\phi_i(y)=y\text{ mod }p^i$. Let $G:= \underleftarrow{\rm lim} G_n$. We have projections $p_n: G\to G_n$, 
and an element $a\in {\rm Ker}(p_1)$ is a sequence of elements $a_n\in \Bbb Z/p^n$ such that $a_{n+1}$ 
projects to $a_n$, i.e., ${\rm Ker}(p_1)=\Bbb Z_p$. Thus we have a short exact sequence of compact topological groups
$$
0\to \Bbb Z_p\to G\to \Bbb R/\Bbb Z\to 0
$$
(non-split, as $G$ is connected). In fact, we can obtain $G$ as a quotient $(\Bbb R\times \Bbb Z_p)/\Bbb Z$
where $\Bbb Z$ is embedded diagonally. 
\end{example} 

\begin{corollary} Finite dimensional (continuous) representations of a compact topological group $G$ with a countable base are unitary and completely reducible. 
\end{corollary} 

The proof is the same as for Lie groups, once we have the integration theory, which we now do. 

\subsection{The Peter-Weyl theorem for compact topological groups} 

\begin{theorem} (i) (Peter-Weyl theorem) Let $G$ be a compact topological group with a countable base. 
Then the set ${\rm Irrep}G$ is countable, and 
$$
L^2(G)=\widehat\oplus_{V\in {\rm Irrep}(G)}V\otimes V^*
$$
as a $G\times G$-module. 

(ii) The subspace $L^2_{\rm alg}(G)=\oplus_{V\in {\rm Irrep}(G)}V\otimes V^*$ is dense in $C(G)$ in the supremum norm. 
\end{theorem} 

Again, the proof is analogous to Lie groups, using a delta-like sequence of continuous hat functions. Namely, we may take 
$$
h_N(x)=c_N\max (\tfrac{1}{N}-d(x,1),0),
$$ 
where $d$ is some metric defining the topology of $G$, and $c_N>0$ are normalization constants such that $\int_G h_N(x)dx=1$. 

\begin{remark} If $G$ is profinite then finite dimensional representations of $G$ 
are just representations of $G_n$ for various $n$: 
$$
{\rm Irrep}G=\cup_{n\ge 1} {\rm Irrep}G_n
$$
(nested union). 
\end{remark}  

\begin{corollary} Any compact topological group with countable base is 
an inverse limit of a sequence of compact Lie groups $...\to G_1\to G_0$, where the maps
$G_{i+1}\to G_i$ are surjective.
\end{corollary} 

\begin{proof} Let $V_1,V_2,...$ be the irreducible representations of $G$. Let $K_m={\rm Ker}(\rho_{V_1}\oplus...\oplus \rho_{V_m})\subset G$, a closed normal subgroup. Then 
$G/K_m\subset U(V_1\oplus...\oplus V_m)$ is a compact Lie group, and $\cap_m K_m=1$, so $G$ is the inverse limit of $G/K_m$.   
\end{proof}  

\begin{exercise} (i) Let $\Bbb Q_p=\Bbb Z_p[1/p]$ be the field of $p$-adic numbers, i.e., 
the field of fractions of $\Bbb Z_p$. Construct the Haar measure $|dx|$ on the additive group of $\Bbb Q_p$ in which the volume of $\Bbb Z_p$ is $1$ using the Haar measure on $\Bbb Z_p$. 

(ii) Show that $\Bbb Q\subset \Bbb Q_p$ and $\Bbb Q_p=\Bbb Q+\Bbb Z_p$, and 
use this to define an embedding $\Bbb Q_p/\Bbb Z_p\to \Bbb Q/\Bbb Z$. 
Show that $\Bbb Q/\Bbb Z=\oplus_{p\text{ prime}} \Bbb Q_p/\Bbb Z_p$. 

(iii) Define the additive character $\psi: \Bbb Q_p\to U(1)\subset \Bbb C^\times$ 
by $\psi(x):=\exp(2\pi i\overline{x})$, where $\overline{x}$ 
is the image of $x$ in $\Bbb Q/\Bbb Z$. Use $\psi$ to label 
the characters (=irreducible representations) of $\Bbb Z_p$
by $\Bbb Q_p/\Bbb Z_p$. 

(iv) Let $|x|$ be the $p$-adic norm of $x\in \Bbb Q_p$ 
($|x|=p^{-n}$ if $x\in p^n\Bbb Z_p$ but $x\notin p^{n+1}\Bbb Z_p$, and $|0|=0$). 
For which $s\in \Bbb C$ is the function $|x|^s$ in $L^2(\Bbb Z_p)$? 

(v) The Peter-Weyl theorem in particular implies that any $L^2$ function $f$
on a compact abelian group $G$ with a countable base can be expanded in a Fourier series 
$$
f(x)=\sum_j c_j\psi_j(x),
$$
where $\psi_j$ are the characters of $G$. Write the Fourier expansion 
of $|x|^s$ when it is in $L^2(\Bbb Z_p)$. 

(vi) Show that $\frac{|dx|}{|x|}$ is a Haar measure on the multiplicative group $\Bbb Q_p^\times=GL_1(\Bbb Q_p)$. 
More generally, show that $|dX|:=\frac{\prod_{1\le i,j\le n}|dx_{ij}|}{|\det(X)|^n}$ 
is a Haar measure on $GL_n(\Bbb Q_p)$ (where $X=(x_{ij})$). 

(vii) Classify characters of $\Bbb Z_p^\times$. 

(viii) Let $S$ be the space of locally constant functions 
on $\Bbb Q_p$ with compact support (i.e., linear combinations of indicator functions of sets 
of the form $a+p^n\Bbb Z_p$, $a\in \Bbb Q_p$). Show that the Fourier transform operator 
$$
\mathcal F(f)=\int_{\Bbb Q_p}\psi(xy)f(y)|dy|
$$
maps $S$ to itself, and $(\mathcal{F}^2f)(x)=f(-x)$. 
Show that $\mathcal F$ preserves the integration 
pairing on $S$, $(f,g)=\int_{\Bbb Q_p}f(x)\overline{g(x)}|dx|$, 
and therefore extends to a unitary operator $L^2(\Bbb Q_p)\to L^2(\Bbb Q_p)$. 
\end{exercise} 

\section{\bf The hydrogen atom, I} 

\subsection{The Schr\"odinger equation}  
Let us now apply our knowledge of non-abelian harmonic analysis to solve a basic problem in quantum mechanics -- describe the dynamics of the hydrogen atom. 

The mechanics of the hydrogen atom is determined by motion of a charged quantum particle (electron) in  a rotationally invariant attracting electric field. The potential of such a field is 
$-\frac{1}{r}$, where $r^2=x^2+y^2+z^2$ (since this theory does not have nontrivial dimensionless quantities, we may choose the units of measurement so that all constants are equal to $1$). Thus, the wave function $\psi(x,y,z,t)$ for our particle obeys the {\bf Schr\"odinger equation}
$$
i\partial_t\psi=H\psi, 
$$
where $H$ is the {\bf quantum Hamiltonian}
$$
H:=-\frac{1}{2}\Delta-\frac{1}{r},
$$
and $\Delta=\partial_x^2+\partial_y^2+\partial_z^2$ is the Laplace operator. 
Recall also that for each $t$, the function $\psi(-,-,-,t)$ is in $L^2(\Bbb R^3)$ and $||\psi||=1$. 
The problem is to solve this equation given the initial value $\psi(x,y,z,0)$.\footnote{Recall that $\psi$ determines the probability $p(U,t)$ to find the electron in a region $U\subset \Bbb R^3$ 
at a time $t$, which is given by the formula $p(U,t)=\int_U |\psi(x,y,z,t)|^2dxdydz$.}

The Schr\"odinger equation can be solved by separation of variables as follows. 
Suppose we have an orthonormal basis $\psi_N$ of $L^2(\Bbb R^3)$ such that 
$H\psi_N=E_N\psi_N$. Then if 
$$
\psi(x,y,z,0)=\sum_N c_N\psi_N(x,y,z)
$$
(i.e., $c_N=(\psi,\psi_N)$)
then 
$$
\psi(x,y,z,t)=\sum_N c_Ne^{-iE_Nt}\psi_N(x,y,z),
$$
So our job is to find such basis 
$\psi_N$, i.e., diagonalize the self-adjoint operator $H$. 

Note that the operator $H$ is unbounded and defined only on a dense subspace of $L^2(\Bbb R^3)$, and although it is symmetric ($(H\psi,\eta)=(\psi,H\eta)$ for compactly supported functions), 
it is very nontrivial to say what precisely it means that $H$ is self-adjoint. Also, 
this operator turns out to have both discrete and continuous spectrum, which means that there is actually {\it no basis} with the desired properties -- eigenfunctions of $H$ which lie in $L^2(\Bbb R^3)$ span 
a {\it proper} closed subspace of this Hilbert space. However, this will not be a problem for our calculation. 

\subsection{Bound states} 
We first focus on {\bf bound states}, i.e., solutions of the {\bf stationary Schr\"odinger equation} 
$$
H\psi=E\psi
$$
which belong to $L^2(\Bbb R^3)$ and thus decay at infinity in the sense of $L^2$-norm (this is the situation when the electron does not have enough energy to escape from the nucleus, i.e., it is ``bound" to it and thus unlikely to be found far from the origin, which explains the terminology). In particular, such eigenfunctions must have negative energy, $E<0$. 
To do so, let us utilize the rotational symmetry and 
write this equation in spherical coordinates. For this we just need to write the Laplacian $\Delta$ in spherical coordinates. Let us write $\bold r=r\bold u$, where 
$\bold u\in S^2$ (i.e., $|\bold u|=1$). 
We have  
$$
\Delta=\Delta_r+\tfrac{1}{r^2}\Delta_{\rm sph} 
$$
where 
$$
\Delta_{\rm sph}=\tfrac{1}{\sin^2\phi}\partial_\theta^2+\tfrac{1}{\sin \phi}\partial_\phi \sin\phi \partial_\phi
$$
 is a differential operator on $S^2$ (the {\bf spherical Laplacian}, or the {\bf Laplace-Beltrami operator}) and 
$$
\Delta_r=\partial_r^2+\tfrac{2}{r}\partial_r
$$
is the {\bf radial part} of $\Delta$ (check it!). 
So our equation looks like 
$$
\partial_r^2\psi+\tfrac{2}{r}\partial_r\psi+\tfrac{2}{r}\psi+\tfrac{1}{r^2}\Delta_{\rm sph}\psi=-2E\psi.
$$
This equation can be solved by again applying separation of variables. Namely, we look for solutions in the form 
$$
\psi(r,\bold u)=f(r)\xi(\bold u), 
$$
where 
\begin{equation}\label{sphe} 
\Delta_{\rm sph} \xi+\lambda \xi=0.
\end{equation} 
Then we obtain the following equation for $f$: 
\begin{equation}\label{radialeq} 
f''(r)+\tfrac{2}{r}f'(r)+(\tfrac{2}{r}-\tfrac{\lambda}{r^2}+2E)f(r)=0.
\end{equation} 
So now we have to solve equation \eqref{sphe} and in particular determine 
which values of $\lambda$ occur. 

To this end, recall that the operator $\Delta_{\rm sph}$ is rotationally invariant, so it preserves 
the space $L^2_{\rm alg}(S^2)$ of functions on $S^2$ belonging to finite dimensional representations of $SO(3)$. Moreover,  it preserves the decomposition $L^2_{\rm alg}(S^2)=\oplus_{\ell\ge 0}L_{2\ell}$ of this space  into irreducible representations of $SO(3)$ (Exercise \ref{homospace}(ii)), and on each $L_{2\ell}$ it acts by a certain scalar $-\lambda_\ell$. To compute this scalar, consider the vector $Y_\ell^0$ in $L_{2\ell}$ of weight zero. This vector is invariant under $SO(2)$ changing $\theta$, so it depends only on $\phi$; in fact, it is a polynomial of degree $\ell$ in $\cos\phi$: $Y_\ell^0=P_\ell(\cos\phi)$. Also orthogonality of the decomposition implies that 
$$
\int_{-1}^1 P_k(z)P_n(z)dz=0,\ k\ne n.
$$
This means that $P_n$ are the {\bf Legendre polynomials}. Also 
$$
\Delta_{\rm sph}P_\ell(z)=\partial_z(1-z^2)\partial_zP_\ell(z)=-\lambda_\ell P_\ell(z),
$$
which shows (by looking at the leading term) that 
$$
\lambda_\ell=\ell(\ell+1),\ \ell\in \Bbb Z_{\ge 0},
$$
and the space of solutions of \eqref{sphe} with $\lambda=\lambda_\ell$ 
is $2\ell+1$-dimensional and is isomorphic to $L_{2\ell}$ as an $SO(3)$-module. 

Consider now the vector $Y_\ell^m\in L_{2\ell}$ of any integer weight 
$-\ell\le m\le \ell$. We will be interested in these vectors up to scaling. 
We have 
$$
Y_\ell^m(\theta,\phi)=e^{im\theta}P_\ell^m(\cos \phi),
$$
where $P_{\ell}^m$ are certain functions. These functions are called {\bf spherical harmonics}. Moreover, it follows from representation theory of $SO(3)$ that $Y_\ell^m$ are trigonometric polynomials which are even for even $m$ and odd for odd $m$ (check it!), so $P_\ell^m(z)$ are polynomials in $z$ when $m$ is even and are of the form $(1-z^2)^{1/2}$ times a polynomial 
in $z$ when $m$ is odd. 
 
Let us calculate the functions $P_\ell^m$. Since they are eigenfunctions of the spherical Laplacian, 
we  obtain that $P_\ell^m$ satisfy the {\bf Legendre differential equation} 
$$
\partial_z (1-z^2)\partial_zP-\frac{m^2}{1-z^2}P+\ell(\ell+1)P=0.
$$

\begin{exercise} Show that this equation has a unique up to scaling continuous solution on $[-1,1]$ when $-\ell\le m\le \ell$ and $m$ is an integer, given by the formula 
$$
P_\ell^m(z)=(1-z^2)^{m/2}\partial_z^{\ell+m}(1-z^2)^\ell. 
$$
\end{exercise} 

These functions are called {\bf associated Legendre polynomials} (even though they are not quite polynomials when $m$ is odd). 

Now we can return to equation \eqref{radialeq}. It now has the form 
\begin{equation}\label{radialeq1} 
f''(r)+\tfrac{2}{r}f'(r)+(\tfrac{2}{r}-\tfrac{\ell(\ell+1)}{r^2}+2E)f(r)=0.
\end{equation} 

To simplify this equation, write 
$$
f(r)=r^\ell e^{-\frac{r}{n}}h(\tfrac{2r}{n}),
$$ 
where $n$ can be chosen at our convenience. Then for $h$ we get 
the equation 
$$
\rho h''(\rho)+(2\ell+2-\rho)h'(\rho)+(n-\ell-1+\tfrac{1}{4}(1+2En^2)\rho)h(\rho)=0. 
$$
We see that the equation simplifies when $n=\frac{1}{\sqrt{-2E}}$, i.e., $E=-\frac{1}{2n^2}$, so let us make this choice. 
Then we have 
$$
\rho h''(\rho)+(2\ell+2-\rho)h'(\rho)+(n-\ell-1)h(\rho)=0, 
$$
which is the {\bf generalized Laguerre equation}. Moreover, we have $||\psi||^2<\infty$, 
which translates to 
\begin{equation} \label{integrability}
\int_0^\infty \rho^{2\ell+2}e^{-\rho}|h(\rho)|^2d\rho<\infty
\end{equation} 
(the factor $\rho^2$ comes from the Jacobian of the spherical coordinates). 

How do solutions of the generalized Laguerre equation behave at $\rho=0$? 
Let us look for a solution of the form $\rho^s(1+o(1))$. 
The characteristic equation for $s$ then has the form 
$$
s(s+2\ell+1)=0,
$$
which gives $s=0$ or $s=-2\ell-1$. 
Thus, for $\ell\ge 1$ the solution 
$\rho^{-2\ell-1}(1+o(1))$ does not satisfy \eqref{integrability}, 
so we are left with a unique solution $h_n(\rho)$ 
which is regular at $\rho=0$ and $h_n(0)=1$. On the other hand, if 
$\ell=0$, the solution $\rho^{-1}(1+o(1))$, even though it satisfies \eqref{integrability}, 
gives rise to a rotationally invariant function $\psi\sim \frac{1}{r}$ as $r\to 0$, 
so we don't get $H\psi=E\psi$, but rather get $H\psi=E\psi+C\delta_0$, 
where $\delta_0$ is the delta function concentrated at zero. So $\psi$ does not really satisfy 
the stationary Schr\"odinger equation as a distribution and has to be discarded, leaving us, as before, with the unique solution $h_n(\rho)$ such that $h_n(0)=1$. 

Using the power series method, we obtain 
$$
h_n(\rho)=\sum_{k=0}^\infty \frac{(1+\ell-n)...(k+\ell-n)}{(2\ell+2)...(2\ell+1+k)}\frac{\rho^k}{k!}.
$$
It is easy to see that this series converges for all $\rho$ and
$$
\lim_{\rho\to +\infty}\frac{\log|h_n(\rho)|}{\rho}=1
$$
 {\bf unless the series terminates}, which happens   
iff $n-\ell-1$ is a nonnegative integer.
(To check the latter, show that the Taylor coefficients $a_k$ of $h_n$ 
are eventually of one sign, and that $|a_k|$ is bounded below by $\frac{1}{(k+N)!}$ for some $N$).  
So it fails \eqref{integrability} unless $n-\ell-1\in \Bbb Z_{\ge 0}$. In this case, 
$$
h_n(\rho)=\sum_{k=0}^{n-\ell-1} \frac{(1+\ell-n)...(k+\ell-n)}{(2\ell+2)...(2\ell+1+k)}\frac{\rho^k}{k!}=
L^{2\ell+1}_{n-\ell-1}(\rho),
$$
the $n-\ell-1$-th {\bf generalized Laguerre polynomial} with parameter $\alpha=2\ell+1$, a polynomial of degree $n-\ell-1$. Namely, the generalized Laguerre polynomials $L^\alpha_N$ are defined by the formula 
$$
L_N^\alpha(\rho):=\sum_{k=0}^{N} (-1)^k\frac{N...(N-k+1)}{(\alpha+1)...(\alpha+k)}\frac{\rho^k}{k!}.
$$
Thus we obtain the following theorem. 

\begin{theorem}\label{boundst} The bound states of the hydrogen atom, up to scaling, are 
$$
\psi_{n\ell m}(r,\theta,\phi)=r^\ell e^{-\frac{r}{n}}L_{n-\ell-1}^{2\ell+1}(\tfrac{2r}{n})Y_\ell^m(\theta,\phi),
$$
where $Y_\ell^m(\theta,\phi)=e^{im\theta}P_\ell^m(\cos\phi)$ are spherical harmonics, 
$n\in \Bbb Z_{>0}$, $\ell$ an integer between $0$ and $n-1$, and $-\ell\le m\le \ell$ is an integer.  
The energy of the state $\psi_{n\ell m}$ is $E_n=-\frac{1}{2n^2}$. 
\end{theorem} 

\begin{exercise}\label{angmom}
Let $\bold r=(x,y,z)$ and 
$\bold p=(-i\partial_x,-i\partial_y,-i\partial_z)$ be the position and momentum operators in $\Bbb R^3$ (these are actually vectors whose components are operators on functions in $\Bbb R^3$). 
Let $\bold L=\bold r\times \bold p$ be the {\bf angular momentum operator}. We have 
$\bold L=(L_x,L_y,L_z)$ where 
$$
L_x=-i(y\partial_z-z\partial_y),\ L_y=-i(z\partial_x-x\partial_z),\ L_z=-i(x\partial_y-y\partial_x)
$$
Let $r=|\bold r|=\sqrt{x^2+y^2+z^2}$ (the operator of multiplication by this function) and $H=\frac{1}{2}\bold p^2+U(r)=-\frac{1}{2}\Delta+U(r)$ be a rotationally symmetric Schr\"odinger operator on $\Bbb R^3$ with potential $U(r)$ (smooth for $r>0$).
Show that the components of $i\bold L$ are vector fields that 
define the action of the Lie algebra ${\rm Lie}(SO(3))$ 
on functions on $\Bbb R^3$ induced by rotations. 
Deduce that $[\bold L,\bold p^2]=0$ (componentwise).
\end{exercise}

\section{\bf The hydrogen atom, II} 

\subsection{Quantum numbers} 
The number $n$ in Theorem \ref{boundst} is called the {\bf principal quantum number}; it characterizes the energy of the state. 
The number $\ell$ is called the {\bf azimuthal quantum number}; it characterizes the eigenvalue 
of the spherical Laplacian $\Delta_{\rm sph}$, which has the physical interpretation as (minus) the {\bf orbital angular momentum operator} $\bold L^2=L_x^2+L_y^2+L_z^2$. By Exercise \ref{angmom}, the operators 
$iL_x$, $iL_y$ and $iL_z$ are just the generators of the Lie algebra ${\rm Lie}(SO(3))$ acting on $\Bbb R^3$, i.e., we have 
$$
[L_x,L_y]=iL_z,\ [L_y,L_z]=iL_x,\ [L_z,L_x]=iL_y
$$
Thus, $\bold L^2$ is simply a Casimir of ${\rm Lie}(SO(3))$. Namely, recall that the standard Casimir $C$ acts 
on $L_{2\ell}$ as $\frac{2\ell(2\ell+2)}{4}=\ell(\ell+1)$, so $\bold L^2=C$. 

Finally, $m$ is called the {\bf magnetic quantum number}, and it is the eigenvalue of $L_z=-i\partial_\theta$ (in spherical coordinates). 

\begin{corollary} The space $W_n$ of states with principal quantum number $n$ has dimension
$n^2$. 
\end{corollary} 

\begin{proof} By Theorem \ref{boundst}, this dimension is $\sum_{\ell=0}^{n-1}(2\ell+1)=n^2$. 
\end{proof} 

In fact, this analysis applies not just to hydrogen but to other chemical elements whose 
nucleus has charge $>1$, if we neglect interaction between electrons. Thus it can potentially 
be used to explain patterns of the periodic table. 

\subsection{Coulomb waves} 
We note, however, that $\psi_{n\ell m}$ {\bf do not} form a basis of $L^2(\Bbb R^3)$. Instead, they span (topologically) a proper closed subspace $L^2_0(\Bbb R^3)$ of $L^2(\Bbb R^3)$ on which the operator $H$ is bounded and negative definite. So if 
a smooth function $\varphi$ on $\Bbb R^3$ (say, with compact support away from the origin) satisfies 
$(H\varphi,\varphi)\ge 0$ then $\varphi\notin L^2_0(\Bbb R^3)$. It is easy to construct such examples: let $\varphi$ be a hat function and $\varphi_s(\bold r)=\varphi(\bold r+s\bold a)$, where $\bold a$ is any nonzero vector. We then have 
$$
(H\varphi_s,\varphi_s)=\frac{1}{2}\int_{\Bbb  R^3} |\nabla \varphi(\bold r)|^2dV-\int_{\Bbb R^3}\frac{|\varphi(\bold r)|^2}{|\bold r-s\bold a|}dV,
$$
and we observe that the first term is positive and the second one goes to zero as $s\to \infty$, so for large $s$ this expression is positive. This happens because besides bound states the hydrogen atom also has {\bf continuous spectrum} $[0,\infty)$ corresponding to free electrons which are not bound by the nucleus. This part of the spectrum can be computed similarly to the discrete (bound state) spectrum, except that the energy will take arbitrary {\bf nonnegative} values. The corresponding wavefunctions are not normalizable (i.e., not in $L^2$), and are given by similar formulas to bound states but with imaginary $n$. Their continuous linear combinations satisfying appropriate boundary conditions are called {\bf Coulomb waves}. 

\subsection{Spin}
Also, the answer $n^2$ for the number of states in the $n$-th energy level does not quite agree with the periodic table, which suggests it should rather be $2n^2$: the numbers of electrons 
at each level are $2,8,18,32...$. This is because the Schr\"odinger model which we computed is not quite right, as it does not take into account an additional degree of freedom called {\bf spin} (a sort of intrinsic angular momentum). Namely, it turns out that the space of states of an electron is not $L^2(\Bbb R^3)$ but rather $L^2(\Bbb R^3)\otimes \Bbb C^2$, with the same Hamiltonian as before but the Lie algebra ${\rm Lie}(SO(3))$ acting diagonally (where $\Bbb C^2$ is the 2-dimensional irreducible representation of this Lie algebra). Thus the space of states of the $n$-th energy level taking spin into account is 
$$
V_n=(L_0\oplus L_2\oplus...\oplus L_{2n-2})\otimes L_1=2L_1\oplus 2L_3\oplus...\oplus 2L_{2n-3}\oplus L_{2n-1}
$$
and $\dim V_n=2n^2$. 
In other words, we have the additional spin-$z$ operator
\[
S_z=\begin{pmatrix}
\frac12&0\\
0&-\frac12
\end{pmatrix}
\]
acting on the $\mathbb C^2$ factor in the standard basis $\bold e_+,\bold e_-$. Thus the
$z$-component of total angular momentum is
\[
J_z=L_z+S_z,
\]
and the product states
\[
\psi_{n\ell m,+}:=\psi_{n\ell m}\otimes \bold e_+,\qquad
\psi_{n\ell m,-}:=\psi_{n\ell m}\otimes \bold e_-
\]
are eigenvectors of $J_z$ with eigenvalues $m+\frac12$ and $m-\frac12$,
respectively.

Note also that $V_n$ is {\bf not} a representation of $SO(3)$ but is only a representation of its double cover $SU(2)$ where $-{\rm Id}$ acts by $-1$. However, this {\bf anomaly} does not mean a violation of the $SO(3)$ symmetry, since true quantum states are unit vectors in the Hilbert space {\bf up to a phase factor}. 

\subsection{The Pauli exclusion principle} 
Suppose now that we have $k$ electrons, each at the $n$-th energy level. If the electrons had been marked, the space of states for them would have been $V_n^{\otimes k}$. But in real life they are indistinguishable, 
so we need to mod out by permutations. So we might think the space of states is $S^kV_n$. 
However, as electrons are {\bf fermions}, this answer turns out to be not correct: the correct answer is 
$\wedge^k V_n$ rather than $S^k V_n$. In other words, when two identical electrons are switched, the corresponding vector changes sign. This is another example of a sign which does not violate symmetry  
since states are well defined only up to a phase factor.

In particular, this implies that if $k>2n^2$ then the space of states is zero, i.e., there cannot be more than $2n^2$ electrons at the $n$-th energy level (the {\bf Pauli exclusion principle}). This is exactly the kind of pattern we see in the periodic table.

\includegraphics[scale=0.1]{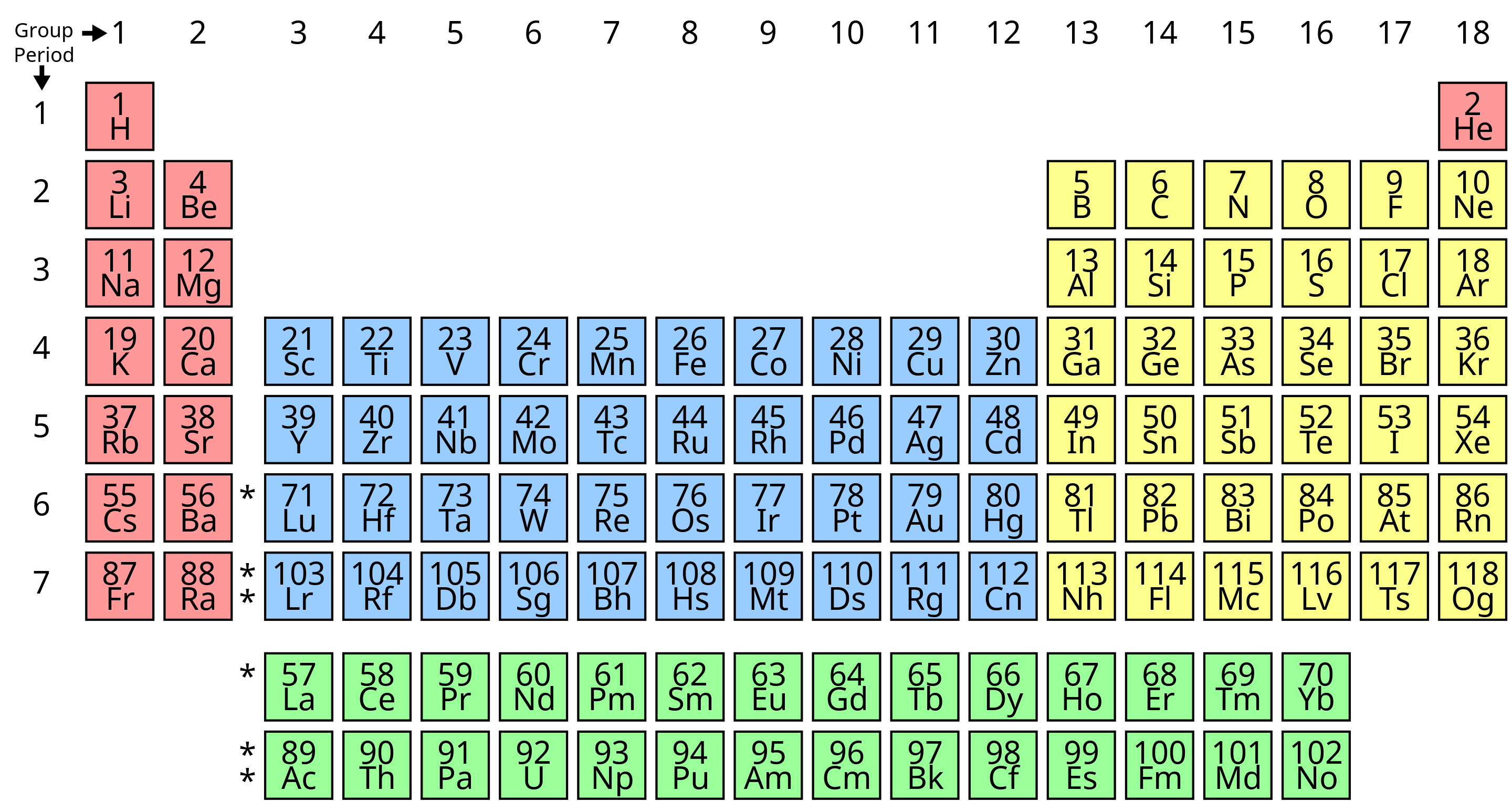}

 Namely, the first energy level has two slots (the first row, or period, of the table), and the second one has 8 slots (the second period of the table). Further down interactions between electrons start to matter and the picture is modified (giving still 8 slots in the next period instead of 18), but we still see a similar pattern: 8 slots in the third period, 18 in periods 4,5, and 32 in periods 6,7. This arrangement is justified by the fact that the columns (groups) of elements, which have the same number of electrons at the last level, have similar chemical properties. For example, in the first column we have alkali metals (except hydrogen) and in the last one we have inert gases.

\begin{exercise} \label{rungelenz} Keep the notation of Exercise \ref{angmom}.

(i) Let $\bold A_0=\frac{1}{2}(\bold p\times \bold L-\bold L\times \bold p)$. 
Show that $[\bold A_0,\bold p^2]=0$ (componentwise). 

(ii) Let $\bold A:=\bold A_0+\phi(r)\bold r$. Show that there exists a function $\phi$ such that 
$[\bold A,H]=0$ if and only if $U$ is the Coulomb potential $\frac{C}{r}+D$, and then $\phi$ is uniquely determined, and compute $\phi$. The corresponding operator $\bold A$ is called the {\bf quantum Laplace-Runge-Lenz vector}.\footnote{In the classical mechanics setting, the existence of this conservation law is the reason why orbits for Coulomb potential are periodic (Kepler's law), while this is not so for other rotationally invariant potentials, except harmonic oscillator. It was discovered many times over the last 300 years. This is one of the most basic examples of ``hidden symmetry".}  

(iii) (Hidden symmetry of the hydrogen atom). By virtue of (ii), the components of $\bold A$ act (by second order differential operators) on functions on $\Bbb R^3$ commuting with $H$. In particular, they act on each $W_n$ (note that in this problem we ignore spin). Use these components to define an action of $\mathfrak{so}_4\cong \mathfrak{so}_3\oplus \mathfrak{so}_3\cong \mathfrak{sl}_2\oplus \mathfrak{sl}_2$  on $W_n$ so that the geometric one (generated by the components of $\bold L$) is the diagonal copy.  

(iv) Show that $W_n=L_{n-1}\boxtimes L_{n-1}$ as a representation of $\mathfrak{sl}_2\oplus \mathfrak{sl}_2$.  

(v) Now include spin by tensoring with the representation $\Bbb C^2$ of $SU(2)$ 
and show that $V_n=L_{n-1}\boxtimes L_{n-1}\boxtimes L_1$ 
as a representation of $\mathfrak{so}_4\oplus \mathfrak{su}_2=\mathfrak{sl}_2\oplus \mathfrak{sl}_2\oplus \mathfrak{sl}_2$. 
This representation is irreducible, which explains why the $n$-th energy level of $H$ 
is degenerate, with multiplicity (i.e., dimension) $2n^2$. 
\end{exercise} 

\begin{exercise} Let $H=-\frac{1}{2}\Delta+\frac{1}{2}r^2$ be the Hamiltonian of the quantum harmonic oscillator in $\Bbb R^n$, where $r=\sqrt{x_1^2+...+x_n^2}$. Compute the eigenspaces of $H$ in $L^2(\Bbb R^n)$ as representations of $SO(n)$ and find the eigenvalues of $H$ with multiplicities and an orthogonal eigenbasis. 

{\bf Hint.} Show that the operator $e^{r^2/2}\circ H\circ e^{-r^2/2}$ preserves the space of polynomials $\Bbb C[x_1,...,x_n]$, and find an eigenbasis $P_{i_1i_2...i_n}$ for this operator in this space (these should express via Hermite polynomials; use that $H=H_1+...+H_n$ is the sum of operators $H_i$ depending only on $x_i$). This will give orthogonal eigenfunctions 
$$
\psi_{i_1...i_n}(\bold r)=P_{i_1...i_n}(\bold r)e^{-r^2/2}
$$ 
in $L^2(\Bbb R^n)$. Using properties of Hermite polynomials, conclude that these are complete. Then use Exercise \ref{orthgr}. 
\end{exercise} 

\section{\bf Forms of semisimple Lie algebras over an arbitrary field} 

\subsection{Automorphisms of semisimple Lie algebras} \label{autom}

We showed in Corollary \ref{auto} that for a complex semisimple $\g$, the group ${\rm Aut}(\g)$ is a Lie group with Lie algebra $\g$. We also showed in Theorem \ref{conjcar} that 
its connected component of the identity ${\rm Aut}(\g)^\circ$ acts transitively on 
the set of Cartan subalgebras in $\g$. This group is called the {\bf adjoint group} attached to $\g$, and we will denote it by $G_{\rm ad}$. 

Let $\h\subset \g$ be a Cartan subalgebra, and $H\subset G_{\rm ad}$ be the corresponding connected Lie subgroup. This subgroup can be viewed as the group of linear operators 
$\g\to \g$ which act by $1$ on $\h$ and by $e^{\alpha(x)}$, $x\in \h$, on each $\g_\alpha$. 
Thus the exponential map $\h\to H$ defines an isomorphism 
$\h/2\pi iP^\vee\cong H$. The group $H$ is called the {\bf maximal torus} of $G_{\rm ad}$ corresponding to $\h$.

\begin{proposition}\label{norma} The normalizer $N(H)$ of $H$ in $G_{\rm ad}$ 
coincides with the stabilizer of $\h$ and contains $H$ as a normal subgroup, so that 
$N(H)/H$ is naturally isomorphic to the Weyl group $W$. 
\end{proposition} 

\begin{proof} First note that since $SL_2(\Bbb C)$ is simply connected, 
for any simple root $\alpha_i$ we have a homomorphism 
$\eta_i: SL_2(\Bbb C)\to G_{\rm ad}$ which identifies ${\rm Lie}(SL_2(\Bbb C))$ with the $\mathfrak{sl}_2$-subalgebra of $\g$ corresponding to this simple root. Let 
\begin{equation}\label{Si}
S_i:=\eta_i\left(\begin{pmatrix} 0 & 1\\ -1 & 0\end{pmatrix}\right).
\end{equation}  
Given $w\in W$, pick a decomposition $w=s_{i_1}...s_{i_n}$, and let $\widetilde w:=S_{i_1}...S_{i_n}\in  G_{\rm ad}$.\footnote{The element $\widetilde w$ in general depends on the decomposition of $w$ as a product of simple reflections. One can show it does not if we take only reduced decompositions, but we will not need this.} Note that $\widetilde w$ acts on $\h$ by $w$. 
So if $w=w_1w_2\in W$ then $\widetilde w=\widetilde w_1\widetilde w_2 h$, where $h$ preserves the root decomposition and acts trivially on $\h$. Thus if $h|_{\g_{\alpha_j}}=\exp(b_j)$ then 
$h=\exp(\sum_j b_j\omega_j^\vee)\in H$. So the elements $\widetilde w$ and $H$ generate 
a subgroup $N\subset N(H)$ of $G_{\rm ad}$ such that $N/H\cong W$. 

It remains to show that $N(H)=N$. To this end, for $x\in N(H)$, let $\alpha_i'=x(\alpha_i)$. 
Then $\alpha_i'$ form a system of simple roots, so there exists $w\in W$ such that $w(\alpha_i')=\alpha_{p(i)}$, where $p$ 
is some permutation. Then $\widetilde w x(\alpha_i)=\alpha_{p(i)}$. So $\widetilde w x$ defines a Dynkin diagram automorphism of $\g$.
Since this automorphism is defined by an element of $G_{\rm ad}$, 
it stabilizes all fundamental representations, so $p={\rm id}$, hence $\widetilde w x\in H$, as claimed.  
\end{proof} 

In particular, we see that $H$ is a maximal commutative subgroup 
of $G_{\rm ad}$, hence the terminology ``maximal torus".

\begin{remark} Note that in general $N(H)$ is {\bf not} isomorphic to $W\ltimes H$: it can be a non-split extension of $W$ by $H$. 
\end{remark}

Another obvious subgroup of ${\rm Aut}(\g)$ is the finite group ${\rm Aut}(D)$ of automorphisms 
of the Dynkin diagram of $\g$, which just permutes the generators $e_i,f_i,h_i$ in the Serre presentation. Thus we have a natural homomorphism 
$$
\xi: {\rm Aut}(D)\ltimes G_{\rm ad}\to {\rm Aut}(\g), 
$$
which is the identity map on the connected components of $1$. This homomorphism is clearly injective, since any nontrivial element of ${\rm Aut}(D)$ nontrivially permutes fundamental representations of $\g$. 

\begin{proposition}\label{auto1} $\xi$ is an isomorphism.  
\end{proposition} 

\begin{proof} Our job is to show that $\xi$ is surjective, i.e. for $a\in {\rm Aut}(\g)$ show that
$a\in {\rm Im}\xi$. By Theorem \ref{conjcar}, we may assume without loss of generality 
that $a$ preserves a Cartan subalgebra $\h\subset \g$ (indeed, this can be arranged by multiplying by an element of $G_{\rm ad}$, since $G_{\rm ad}$ acts transitively on Cartan subalgebras of $\g$). Then by multiplying by an element of ${\rm Aut}(D)\cdot N(H)$ we can make sure that $a$ acts trivially on $\h$ and $\g_{\alpha_i}$. Then $a=1$, which implies the proposition. 
\end{proof}

\subsection{Forms of semisimple Lie algebras} 

We have classified semisimple Lie algebras over $\Bbb C$, but what about other fields (say of characteristic zero), notably $\Bbb R$ (the case relevant to the theory of Lie groups)? 

To address this question, note that the Serre presentation of a semisimple Lie algebra is defined over $\Bbb Q$, so it defines a Lie algebra of the same dimension over any such field, by imposing the same generators 
and relations. Such a Lie algebra is called {\bf split}. So for example, over an algebraically closed field of characteristic zero, any semisimple Lie algebra is automatically split.  

Now let $\g$ be a semisimple Lie algebra over a field $K$ of characteristic zero which splits over a Galois extension 
$L$ of $K$, i.e., $\g\otimes_K L=\g_L$ is split (corresponds to a Dynkin diagram via Serre's presentation). Can we classify such $\g$? 

To this end, let $\Gamma={\rm Gal}(L/K)$ be the Galois group of $L$ over $K$ and observe that we can recover $\g$ as the subalgebra of invariants $\g_L^\Gamma$. So $\g$ is determined by the action of $\Gamma$ on the split semisimple Lie algebra $\g_L$. Note that this action is {\bf twisted-linear}, i.e., additive and $g(\lambda x)=g(\lambda)g(x)$ for $x\in \g_L$, $\lambda\in L$, $g\in \Gamma$. The simplest example of such an action is the action $\rho_0(g)$ which preserves all the generators $e_i,f_i,h_i$ and just acts on the scalars, which corresponds to the split form of $\g$. So any twisted-linear action $\rho$ can be written as 
$$
\rho(g)=\eta(g)\rho_0(g)
$$
for some map 
$$
\eta: \Gamma\to {\rm Aut}(\g_L). 
$$
In order that $\rho$ be a homomorphism, we need 
$$
\eta(gh)\rho_0(gh)=\eta(g)\rho_0(g)\eta(h)\rho_0(h), 
$$
which is equivalent to 
$$
\eta(gh)=\eta(g)\cdot g(\eta(h)),
$$
where for $a\in {\rm Aut}(\g_L)$, $g(a):=\rho_0(g)a\rho_0(g)^{-1}$.
In other words, $\eta$ is a {\bf $1$-cocycle}.  
We will denote the Lie algebra attached to such cocycle $\eta$ by $\g_\eta$. 

It remains to determine when $\g_{\eta_1}$ is isomorphic to $\g_{\eta_2}$. This will happen exactly when the corresponding representations $\rho_1$ and $\rho_2$ are isomorphic, i.e., there is 
$a\in {\rm Aut}(\g_L)$ such that $\rho_1(g)a=a\rho_2(g)$, i.e., 
$$
\eta_1(g)\rho_0(g)a=a\eta_2(g)\rho_0(g), 
$$
or 
$$
\eta_1(g)=a\eta_2(g)g(a)^{-1}.
$$
Two 1-cocycles related in this way are called {\bf cohomologous} (obviously, an equivalence relation), and the set of equivalence classes of cohomologous cocycles is called the {\bf first Galois cohomology} of $\Gamma$ with coefficients in ${\rm Aut}(\g_L)$ and denoted by $H^1(\Gamma,{\rm Aut}(\g_L))$. 
Note that this is cohomology with coefficients in a nonabelian group, so it is just a set and not a group. 

So we obtain 

\begin{proposition} 
Semisimple Lie algebras $\g$ over $K$ which split over a Galois extension $L$ of $K$ are classified by  the first Galois cohomology 
$H^1(\Gamma,{\rm Aut}(\g_L))$. 
\end{proposition} 

\begin{remark} There is nothing special about semisimplicity or about Lie algebras here -- this works for any kind of linear algebraic structures, such as associative algebras, algebraic varieties, schemes, etc. 
\end{remark}

\subsection{Real forms of a semisimple Lie algebra} \label{realfo1}
Let us now make this classification more concrete in the case $K=\Bbb R$, $L=\Bbb C$, which is relevant to classification of real semisimple Lie groups. In this case, 
$\Gamma=\Bbb Z/2$ generated by complex conjugation $s\mapsto \overline s$ and, as we have shown, ${\rm Aut}(\g_L)={\rm Aut}(D)\ltimes G_{\rm ad}$, where $D$ is the Dynkin diagram of $\g$ and $G_{\rm ad}$ is the corresponding connected adjoint complex Lie group. Also since we always have $\eta(1)=1$, the cocycle $\eta$ 
is determined by the element $s=\eta(-1)\in {\rm Aut}(D)\ltimes G_{\rm ad}$. Moreover, 
$s$ must satisfy the cocycle condition
$$
s\overline s=1
$$
and the corresponding real Lie algebra, up to isomorphism, depends only 
on the cohomology class of $s$, which is the equivalence class 
modulo transformations $s\mapsto as \overline{a}^{-1}$. 
We thus obtain the following theorem. 

\begin{theorem} 
Real semisimple Lie algebras whose complexification is $\g$ (i.e., {\bf real forms} 
of $\g$) are classified by $s\in {\rm Aut}(D)\ltimes G_{\rm ad}$ such that 
$s \overline{s}=1$ modulo equivalence $s\mapsto as \overline{a}^{-1}$, 
$a\in {\rm Aut}(\g)$, where complex conjugation acts trivially on ${\rm Aut}(D)$.
\end{theorem} 

We denote the real form of $\g$ corresponding to $s$ by $\g_{(s)}$. 
Namely, $\g_{(s)}=\lbrace x\in \g: \overline{x}=s(x)\rbrace$. 
For example, $\g_{(1)}$ is the split form, consisting of real $x\in \g$, i.e., such that $\overline x=x$. 

Alternatively, one may define the {\bf antilinear involution} $\sigma_s(x)=\overline{s(x)}$, and 
$\g_{(s)}$ is the set of fixed points of $\sigma_s$ in $\g$.   

In particular, such $s$ defines an element $s_0\in {\rm Aut}(D)$ 
such that $s_0^2=1$. Note that the conjugacy class of $s_0$
is invariant under equivalences. The element $s_0$ permutes connected components of $D$, 
preserving some and matching others into pairs. Thus every semisimple real Lie algebra 
is a direct sum of simple ones, and each simple one either has a connected Dynkin diagram 
$D$ (i.e., the complexified Lie algebra $\g$ is still simple) or 
consists of two identical components (i.e., the complexified Lie algebra 
is $\g=\a\oplus \a$ for some simple complex $\a$). In the latter case 
$s=(g,\overline g^{-1})s_0$ where $s_0$ is the transposition and 
$g\in {\rm Aut}(\a)$, so $s$ is cohomologous to $s_0$ by taking $a=(g,1)$. 
Thus in this case $\g_{(s)}=\g_{(s_0)}=\a$, a complex simple Lie algebra regarded as a real Lie algebra. 

It remains to consider the case when $D$ is connected, i.e., $\g$ is simple. 

\begin{definition} (i) A real form $\g_{(s)}$ of a complex simple Lie algebra $\g$ 
is said to be {\bf inner} to $\g_{(s')}$ if $s'=gs$ up to equivalence, where 
$g\in {\rm G}_{\rm ad}$ (i.e., $s$ and $s'$ differ by an inner automorphism). 
The {\bf inner class} of $\g_{(s)}$ is the collection of all real forms inner to $\g_{(s)}$. 
In particular, an {\bf inner form} is a form inner to the split form. 

(ii) $\g_{(s)}$ is called {\bf quasi-split} if $s=s_0\in {\rm Aut}(D)$ (modulo equivalence).  
\end{definition} 

So in particular any real form is inner to a unique quasi-split form, and a real form that is both inner and quasi-split is split. 

\begin{exercise}\label{cartancla} Let $\g_{\Bbb R}$ be a real semisimple Lie algebra and $\h_{\Bbb R}\subset \g_{\Bbb R}$ a Cartan subalgebra (the centralizer of a regular semisimple element of $\g_{\Bbb R}$). Let $\h\subset \g$ be their complexifications, 
and $H\subset G_{\rm ad}$ the corresponding complex Lie groups. Let $\bold K$ be the kernel of the natural map of Galois cohomology sets $H^1(\Bbb Z/2,N(H))\to H^1(\Bbb Z/2,G_{\rm ad})$ (i.e., the preimage of the unit element), where $\Bbb Z/2$ acts on $G_{\rm ad}$ by complex conjugation associated to the real form $\g_{\Bbb R}$ of $\g$. 

(i) Show that conjugacy classes of Cartan subalgebras in $\g_{\Bbb R}$ are bijectively labeled by elements of $\bold K$, with the unit element corresponding to $\h_{\Bbb R}$.

(ii) Show that $\bold K$ is a finite set.\footnote{For classical Lie algebras the set $\bold K$ will be computed explicitly in Exercise \ref{carclas}. The explicit answer is known for exceptional Lie algebras as well, but we will not discuss it here.} 
\end{exercise} 

\section{\bf Classification of real forms of semisimple Lie algebras} 

\subsection{The compact real form} 

An important example of a real form of simple complex Lie algebra $\g$ is the {\bf compact real form}. 
It is determined by the automorphism $\tau$ (called the {\bf Cartan involution}) defined by the formula 
$$
\tau(h_j)=-h_j,\ \tau(e_j)=-f_j,\ \tau(f_j)=-e_j.
$$
Let us denote this real form $\g_{(\tau)}$ by $\g^c$. 

\begin{proposition} The Killing form of $\g^c$ is negative definite.  
\end{proposition} 

\begin{proof} We have an orthogonal decomposition
$$
\g^c=(\h\cap \g^c)\oplus \bigoplus_{\alpha\in R_+} (\g_\alpha\oplus \g_{-\alpha})\cap \g^c.
$$
Moreover, the Killing form is clearly negative definite 
on $\h\cap \g^c$, since the inner product on the coroot lattice is positive definite, and 
$\lbrace i\alpha_j^\vee\rbrace$ is a basis of $\h\cap \g^c$. 
So it suffices to show that the Killing form is negative definite on 
$(\g_\alpha\cap \g_{-\alpha})\cap \g^c$ for any $\alpha\in R_+$. 
 
First consider the case $\g=\mathfrak{sl}_2$. 
Then $\g^c$ is spanned by the Pauli matrices $ih$, $e-f$, $i(e+f)$, 
so $\g^c=\mathfrak{su}(2)$. It follows that the trace form 
of any finite dimensional representation of $\g^c$ is negative definite. 
  
 Thus for a general $\g$, the elements $S_i$ given by \eqref{Si} preserve $\g^c$; this follows 
 since the matrix $S:=\begin{pmatrix} 0 & 1\\ -1 & 0\end{pmatrix}$ belongs to $SU(2)$, and 
 ${\rm Lie}(SU(2)_i)\subset \g^c$.  It follows that for any $w\in W$ the element $\widetilde w$ preserves $\g^c$. Thus the restriction of the Killing form of $\g^c$ to $\g^c\cap (\mathfrak{sl}_2)_\alpha$ is negative definite for any root $\alpha$ (since it is so for simple roots, as follows from the case of $\mathfrak{sl}_2$). 
This implies the statement.  
\end{proof} 

Now consider the group ${\rm Aut}(\g^c)$. Since the Killing form 
on $\g^c$ is negative definite, it is a closed subgroup in the orthogonal group $O(\g^c)$, hence is compact. Moreover, it is a Lie group with Lie algebra $\g^c$. Thus we obtain

\begin{corollary} Let $G_{\rm ad}^c={\rm Aut}(\g^c)^\circ$. Then $G_{\rm ad}^c$ is a connected compact Lie group with Lie algebra 
 $\g^c$. 
\end{corollary} 

In particular, this gives a new proof that representations of a finite dimensional semisimple Lie algebra are completely 
reducible (by using Weyl's unitary trick, see Subsection \ref{unrep}). Indeed, let $G$ be the universal cover of $G_{\rm ad}$, $Z$ the center of $G$, and $U,V$ two irreducible representations of $G$ (equivalently, of $\mathfrak g$). If $Z$ acts on $U,V$ by different characters then obviously ${\rm Ext}^1_{\mathfrak g}(U,V)=0$ (see Subsection \ref{extee}). Otherwise, ${\rm Ext}_{\mathfrak g}^1(U,V)={\rm Ext}^1_{\mathfrak g}(\Bbb C,{\rm Hom}_{\Bbb C}(U,V))$, which is zero since both $\Bbb C$ and ${\rm Hom}_{\Bbb C}(U,V)$ carry a trivial action of $Z$ and therefore are representations of $G_{\rm ad}$, hence of the compact group $G_{\rm ad}^c$. 

\begin{exercise} (i) Show that if $\g=\mathfrak{sl}_n$ then $G^c_{\rm ad}=PSU(n)=SU(n)/\mu_n$, where $\mu_n$ is the group of roots of unity of order $n$. 

(ii) Show that if $\g=\mathfrak{so}_n$ then $G^c_{\rm ad}=SO(n)$ for odd $n$ and 
$SO(n)/\pm 1$ for even $n$. 

(iii) Show that if $\g=\mathfrak{sp}_{2n}$ then $G^c_{\rm ad}=U(n,\Bbb H)/\pm 1$, where 
$U(n,\Bbb H)$ is the quaternionic unitary group $Sp_{2n}(\Bbb C)\cap U(2n)$ (see Exercise \ref{quatgrou}). 
\end{exercise} 

\begin{exercise} (i) Compute the signature of the Killing form of the split form $\g^{\rm spl}$ of a complex simple Lie algebra $\g$ in terms of its dimension and rank, and show that the compact form is never split. 

(ii) Show that the compact form is inner to the quasi-split form defined by 
the flip of the Dynkin diagram corresponding to taking the dual representation (i.e., induced by $-w_0$), but is never quasi-split itself (show that the quasi-split form contains nonzero nilpotent elements). For which simple Lie algebras is the compact form inner? 
\end{exercise} 

\subsection{Other examples of real forms} 

So let us list real forms of simple Lie algebras that we know so far. 

1. Type $A_{n-1}$. We have the split form $\mathfrak{sl}_n(\Bbb R)$, 
the compact form $\mathfrak{su}(n)$, and also for $n>2$ the quasi-split form 
associated to the automorphism $s(A)=-JA^TJ^{-1}$, where $J_{ij}=(-1)^i\delta_{i,n+1-j}$
(this automorphism sends $e_i,f_i,h_i$ to $e_{n+1-i},f_{n+1-i},h_{n+1-i}$). So the corresponding real Lie algebra is the Lie algebra of traceless matrices preserving the hermitian or skew-hermitian form defined by the matrix $J$, which has signature $(p,p)$ if $n=2p$ and $(p+1,p)$ or $(p,p+1)$ if $n=2p+1$. Thus 
in the first case we have $\mathfrak{su}(p,p)$ and in the second case we have
$\mathfrak{su}(p+1,p)$. Note that for $n=2$ we have $\mathfrak{su}(1,1)=\mathfrak{sl}_2(\Bbb R)$, so in this special case this form is not new. 
We also observe that for $n\ge 4$ there are other forms, e.g. $\mathfrak{su}(n-p,p)$ with $1\le p\le \frac{n}{2}-1$.  

2. Type $B_n$. We have the split form $\mathfrak{so}(n+1,n)$, the compact form 
$\mathfrak{so}(2n+1)$. The Dynkin diagram has no nontrivial automorphisms, 
so there are no non-split quasi-split forms. In particular, since $A_1=B_1$, 
we have $\mathfrak{so}(3)=\mathfrak{su}(2)$ and 
$\mathfrak{so}(2,1)=\mathfrak{su}(1,1)$. 

3. Type $C_n$. We have the split form $\mathfrak{sp}_{2n}(\Bbb R)$ and compact form 
$\mathfrak{u}(n,\Bbb H)$. The Dynkin diagram has no nontrivial automorphisms, 
so there are no non-split quasi-split forms. The equality $B_2=C_2$ 
implies that $\mathfrak{so}(3,2)=\mathfrak{sp}_4(\Bbb R)$ 
and $\mathfrak{so}(5)=\mathfrak{u}(2,\Bbb H)$. 

4. Type $D_n$. We have the split form $\mathfrak{so}(n,n)$, the compact form 
$\mathfrak{so}(2n)$. Moreover, in this case we have a unique nontrivial involution
of the Dynkin diagram. More precisely, this is true for $n\ne 4$, while for $n=4$ we have ${\rm Aut}(D)=S_3$, but there is still a unique non-trivial involution up to conjugation.  
So we also have a non-split quasi-split form. To compute it, recall that the split form is defined 
by the equation $A=-JA^TJ^{-1}$ where $J_{ij}=\delta_{i,2n+1-j}$. 
The quasi-split form is obtained by replacing $J$ by $J'=gJ$, where 
$g$ permutes $e_n$ and $e_{n+1}$ (this is the automorphism that switches $\alpha_{n-1}$ and $\alpha_n$ while keeping other simple roots fixed). The signature of the form defined by $J'$ is 
$(n+1,n-1)$, so we get that the non-split quasi-split form is $\mathfrak{so}(n+1,n-1)$. 
In particular, since $D_2=A_1+A_1$, for $n=2$ we get 
$$
\mathfrak{so}(4)=\mathfrak{su}(2)\oplus \mathfrak{su}(2),\ \mathfrak{so}(2,2)=\mathfrak{su}(1,1)\oplus \mathfrak{su}(1,1),\ 
\mathfrak{so}(3,1)=\mathfrak{sl}_2(\Bbb C)
$$ 
(the Lie algebra of the Lorentz group of special relativity). Also, since $D_3=A_3$, 
for $n=3$ we get  $\mathfrak{so}(6)=\mathfrak{su}(4)$, $\mathfrak{so}(3,3)=\mathfrak{sl}_4(\Bbb R)$, 
and $\mathfrak{so}(4,2)=\mathfrak{su}(2,2)$. 

5. Type $G_2$. We have the split and compact forms $G_2(\Bbb R),G_2^c$. 

6. Type $F_4$. We have the split and compact forms $F_4(\Bbb R),F_4^c$. 

7. Type $E_6$. We have the split and compact forms $E_6(\Bbb R),E_6^c$ and the quasi-split form 
$E_6^{qs}$ attached to the non-trivial automorphism. 

8. Type $E_7$. We have the split and compact forms $E_7(\Bbb R),E_7^c$.

9. Type $E_8$. We have the split and compact forms $E_8(\Bbb R),E_8^c$. 

\subsection{Classification of real forms}

However, we are not done with the classification of real forms yet, as we still need to find all real forms and show there are no others. To this end, consider a complex simple Lie algebra $\g=\g^c\otimes_{\Bbb R}\Bbb C$.  
We have the compact antilinear involution $\omega=\sigma_\tau$ of $\g$ whose set of fixed points is $\g^c$. 
 Another real structure on $\g$ is then defined by the antilinear involution 
 $\sigma=\omega\circ g$, where $g\in {\rm Aut}(\g)$ is such that ${\omega(g)}g=1$. 
 But it is easy to see that 
 $$
 \omega(g)=(g^\dagger)^{-1},
 $$ 
 where $x^\dagger$ is the adjoint to $x\in \End(\g)$ under the negative definite Hermitian 
 form $(X,Y)={\rm Tr}({\rm ad}X{\rm ad}\omega(Y))$ (the Hermitian extension of the Killing form on $\g^c$ to $\g$). It follows that the operator $g$ is self-adjoint. Thus it is diagonalizable 
 with real eigenvalues, and we have a decomposition 
 $$
 \g=\oplus_{\gamma\in \Bbb R} \g(\gamma),
 $$ 
 where $\g(\gamma)$ is the $\gamma$-eigenspace of $g$, such that
 $[\g(\beta),\g(\gamma)]\subset\g(\beta\gamma)$. Now consider 
 the operator $|g|^t$ for any $t\in \Bbb R$. It acts on $\g(\gamma)$ 
 by $|\gamma|^t$, so $|g|^t=\exp(t\log|g|)\in G_{\rm ad}$ is a 1-parameter subgroup. Now 
 define $\theta:=g|g|^{-1}$. We have $\theta\circ \omega=\omega\circ \theta$ and  $\theta^2=1$. Also $g$ and $\theta$ define the same real structure since $\theta=|g|^{-1/2}g\omega(|g|^{1/2})$.  This shows that without loss of generality we may assume that $g=\theta$ with 
 $\theta\circ\omega=\omega\circ\theta$ (i.e., $\theta\in \Aut(\g^c)$) 
and $\theta^2=1$.\footnote{The advantage of passing from $g$ to $\theta$ 
is that the equation $\theta^2=1$ is much easier to solve than $g\omega(g)=1$, as it just means that we have a decomposition of $\g$ into the $+1$- and $-1$-eigenspaces of $\theta$.}

Moreover, another such element $\theta'$ defines 
the same real form if and only if $\theta'=x\theta\omega(x)^{-1}$ for some $x\in {\rm Aut}(\g)$. 
So we get 
$$
x\theta\omega(x)^{-1}=\omega(x)\theta x^{-1},
$$
so setting $z:=\omega(x)^{-1}x$, we get $\omega(z)=z^{-1}$, $\theta z=z^{-1}\theta$. 
Note that $z=x^\dagger x$ is positive definite. 
So setting $y=xz^{-1/2}$, we have 
$$
\omega(y)=\omega(x)z^{1/2}=xz^{-1/2}=y
$$
 i.e., $y\in {\rm Aut}(\g^c)$ 
and 
$$
\theta'=x\theta\omega(x)^{-1}=x\theta zx^{-1}=xz^{-1/2}\theta z^{1/2}x^{-1}=y\theta y^{-1}.
$$ 
Thus we obtain 

\begin{theorem}\label{realfo2} Real forms of $\g$ are in bijection with conjugacy classes 
of involutions $\theta\in {\rm Aut}(\g^c)$, via $\theta\mapsto \omega_\theta:=\theta\circ \omega=\omega\circ\theta$.  
\end{theorem} 

Theorem \ref{realfo2} provides a different classification of real forms 
from the one given in Subsection \ref{realfo1}, obtained by ``counting" from the compact form rather than the split form (as we did in Subsection \ref{realfo1}). We denote the real form of $\g$ assigned in Theorem \ref{realfo2} to an involution $\theta: \g\to \g$ by $\g_\theta$. 
For example, $\g_1=\g^c=\g_{(\tau)}$. 

Thus we have a canonical (up to automorphisms of $\g^c$) decomposition $\g=\mathfrak{k}\oplus \p$, into the eigenspaces of $\theta$ with eigenvalues $1$ and $-1$, such that $\mathfrak{k}$ is a Lie subalgebra, 
$\p$ is a module over $\mathfrak{k}$ and $[\p,\p]\subset \mathfrak{k}$. 
We also have the corresponding decomposition for the underlying real Lie algebra 
$\g^c=\mathfrak{k}^c\oplus \p^c$. Moreover, 
the corresponding real form $\g_\theta$ 
is just $\g_\theta=\mathfrak{k}^c\oplus \p_\theta$, where $\p_\theta:=i\p^c$. 

\begin{exercise} Show that $\mathfrak{k}$ is a reductive Lie algebra. Does it have to be semisimple? 
\end{exercise}

\begin{proposition}\label{maxcomp} There exists a Cartan subalgebra $\h$ in $\g$ invariant under $\theta$, such that $\h\cap \mathfrak k$ is a Cartan subalgebra in $\mathfrak k$. 
\end{proposition} 

\begin{proof} Take a generic $t\in \mathfrak{k}^c$; as $\mathfrak{k}$ is reductive, it is regular semisimple. Let $\h_+^c$ 
be the centralizer of $t$ in $\mathfrak{k}^c$. Then $\h_+:=\h_+^c\otimes_{\Bbb R}\Bbb C\subset \mathfrak{k}$ is a Cartan subalgebra. Let $\h_-^c$ be a maximal subspace of 
$\p^c$ for the property that $\h^c:=\h_+^c\oplus \h_-^c$ 
is a commutative Lie subalgebra of $\g^c$. 

We claim that $\h:=\h^c\otimes_{\Bbb R}\Bbb C$ 
is a Cartan subalgebra in $\g$. Indeed, it obviously consists of semisimple elements (as all elements 
in $\g^c$ are semisimple, being anti-hermitian operators on $\g^c$). Now, if 
$z\in \g$ commutes with $\h$ then $z=z_++z_-$, $z_+\in \mathfrak{k}$ and $z_-\in \p$, and 
both $z_+,z_-$ commute with $\h$. Thus $z_+\in \h_+$ and 
$z_-=x+iy$, where $x,y\in \p^c$ and both commute with $\h$. 
Hence $x,y\in \h_-^c$ by the definition of $\h_-^c$. Thus $z\in \h$, as claimed. 
It is clear that $\h$ is $\theta$-stable, so the proposition is proved.  
\end{proof} 

Thus we have a decomposition $\h=\h_+\oplus \h_-$, and 
$\theta$ acts by $1$ on $\h_+$ and by $-1$ on $\h_-$. 

\begin{lemma}\label{nocor} The space $\h_-$ does not contain any coroots of $\g$. 
\end{lemma}

\begin{proof} Suppose that $\alpha^\vee\in \h_-$ is a coroot. Thus $\theta(\alpha^\vee)=-\alpha^\vee$, 
so $\theta(e_\alpha)=e_{-\alpha}$ and $\theta(e_{-\alpha})=e_\alpha$ for some 
nonzero $e_{\pm \alpha}\in \g_{\pm \alpha}$. Let $x=e_\alpha+e_{-\alpha}$. We have 
$\theta(x)=x$, so $x\in \mathfrak{k}$. On the other hand, $x\notin \h_+$ (as $x$ is orthogonal to $\h_+$ and nonzero) and $[\h_+,x]=0$ since $\alpha$ vanishes on $\h_+$. This is a contradiction, since $\h_+$ 
is a maximal commutative subalgebra of $\mathfrak{k}$. 
\end{proof} 

By Lemma \ref{nocor}, a generic element $t\in \h_+$ is regular in $\g$. So let us pick one for which ${\rm Re}(t,\alpha^\vee)$ is nonzero for any coroot $\alpha^\vee$ of $\g$, and use 
it to define a polarization of $R$: set $R_+:=\lbrace\alpha\in R:{\rm Re}(t,\alpha^\vee)>0\rbrace$. 
Then $\theta(R_+)=R_+$. So $\theta(\alpha_i)=\alpha_{\theta(i)}$, where 
$\theta(i)$ is the action of $\theta$ on the Dynkin diagram $D$ of 
$\g$. Thus if $\theta(i)=i$ then $\theta(e_i)=\pm e_i$, $\theta(h_i)=h_i$, $\theta(f_i)=\pm f_i$ while 
if $\theta(i)\ne i$, we can normalize $e_i,e_{\theta(i)},f_i,f_{\theta(i)}$ so that $\theta(e_i)=e_{\theta(i)}$, $\theta(f_i)=f_{\theta(i)}$, $\theta(h_i)=h_{\theta(i)}$. Thus $\theta$ can be encoded in a marked Dynkin diagram of $\g$: we connect vertices $i$ and $\theta(i)$ if $\theta(i)\ne i$ 
and paint a $\theta$-stable vertex $i$ white if $\theta(e_i)=e_i$ (i.e., $e_i\in \mathfrak{k}$, a {\bf compact root}), and black if $\theta(e_i)=-e_i$ (i.e., $e_i\in \p$, a {\bf non-compact root}). Such a decorated Dynkin diagram is called a {\bf Vogan diagram}. So we see that every Vogan diagram gives rise to a real form, and every real form is defined by some Vogan diagram. 

\begin{exercise} (i) Show that the signature of the Killing form of 
a real form $\g_\theta$ of a complex semisimple Lie algebra $\g$ corresponding to involution $\theta$ equals $(\dim{\mathfrak p},\dim \mathfrak{k})$. In particular, the Killing form of $\g_\theta$ is negative definite if and only if $\theta=1$, i.e., 
$\g_\theta=\g^c$ is the compact form.  

(ii) Deduce that for the split form $\dim \mathfrak{k}=|R_+|$, the number of positive roots of $
\g$. 

(iii) Show that for a real form of $\g$ in the compact inner class, we have 
${\rm rank}(\mathfrak{k})={\rm rank}\g$. 
\end{exercise} 

\subsection{Real forms of classical Lie algebras} 

We are not finished yet with the classification of real forms since different Vogan diagrams can define the same real form (they could arise from different choices of $R_+$ coming from different choices of the element $t$). 
However, we are now ready to classify real forms of classical Lie algebras. 

{\bf 1. Type $A_{n-1}$, compact inner class.} In this case $\theta$ is an inner automorphism, conjugation by an element of order $\le 2$ in $PSU(n)$. Obviously, such an element can be lifted 
to $g\in U(n)$ such that $g^2=1$, so $\theta(x)=gxg^{-1}$. Thus
$g={\rm Id}_p\oplus (-{\rm Id}_q)$ where $p+q=n$ and we may assume that $p\ge q$. 
It is easy to see that this defines the real form $\g_\theta=\mathfrak{su}(p,q)$, and 
$\mathfrak{k}=\mathfrak{sl}_p\oplus \mathfrak{gl}_q$. These are all pairwise non-isomorphic since the corresponding automorphisms $\theta$ are not conjugate to each other. So we get $[\frac{n}{2}]+1$ real forms. Note that for $n=2$ this exhausts all real forms, so we have only two -- $\mathfrak{su}(2)$ and $\mathfrak{su}(1,1)=\mathfrak{sl}_2(\Bbb R)$ with $\mathfrak{k}=\mathfrak{gl}_1$. 

{\bf 2. Type $A_{n-1}$, $n>2$, the split inner class.} If $n$ is odd, there is no choice as all the vertices of the Vogan diagram are connected into pairs, so we only get the split form $\g_\theta=\mathfrak{sl}_n(\Bbb R)$. However,  
if $n=2k$ is even, there is one unmatched vertex in the middle of the Vogan diagram, which can be either white or black. It is easy to check that in the first case (white vertex) $\mathfrak{k}=\mathfrak{sp}_{2k}$ and in the second one (black vertex) $\mathfrak{k}=\mathfrak{so}_{2k}$. So the first case is $\g_\theta=\mathfrak{sl}(k,\Bbb H)$, the Lie algebra of quaternionic matrices of size $k$ whose trace has zero real part (See Subsection \ref{morecla}), while the second case is the split form $\g_\theta=\mathfrak{sl}_n(\Bbb R)$. 

{\bf 3. Type $B_n$.} Then $\theta$ is an inner automorphism, given by an element 
of order $\le 2$ in $SO(2n+1)$. So $\theta={\rm Id}_{2p+1}\oplus (-{\rm Id}_{2q})$ where $p+q=n$. 
Thus all the real forms are $\mathfrak{so}(2p+1,2q)$ (all distinct), $\mathfrak{k}=\mathfrak{so}_{2p+1}\oplus \mathfrak{so}_{2q}$. 

{\bf 4. Type $C_n$.} Then $\theta$ is an inner automorphism, given by an element 
$g\in {\rm Sp}_{2n}(\Bbb C)$ such that $g^2=1$ or $g^2=-1$. 
In the first case the $1$-eigenspace of $g$ has dimension $2p$ and the $-1$-eigenspace has dimension $2q$ (since they are symplectic), where $p+q=n$, and we may assume $p\ge q$ (replacing $g$ by $-g$ if needed). So the real form we get is $\g_\theta=\mathfrak{u}(p,q,\Bbb H)$, the quaternionic pseudo-unitary Lie algebra for a quaternionic Hermitian form (see Subsection \ref{morecla}). In this case $\mathfrak{k}=\mathfrak{sp}_{2p}\oplus \mathfrak{sp}_{2q}$. 
On the other hand, if $g^2=-1$ then $\Bbb C^{2n}=V(i)\oplus V(-i)$ 
(eigenspaces of $g$, which in this case are Lagrangian subspaces), so $\mathfrak{k}=\mathfrak{gl}_n(\Bbb C)$. 
The corresponding real form is the split form $\g_\theta=\mathfrak{sp}_{2n}(\Bbb R)$.  

{\bf 5. Type $D_n$, compact inner class.} We again have an inner automorphism $\theta$ 
given by $g\in SO(2n)$ such that $g^2=\pm 1$. If $g^2=1$ then 
$\Bbb C^{2n}=V(1)\oplus V(-1)$, the direct sum of eigenspaces, 
and since $\det(g)=1$, the eigenspaces are even-dimensional, of dimensions 
$2p$ and $2q$ where $p+q=n$, and, as in the case of type $C_n$, we may assume $p\ge q$. So the corresponding real form 
is $\g_\theta=\mathfrak{so}(2p,2q)$ with $\mathfrak{k}=\mathfrak{so}_{2p}\oplus \mathfrak{so}_{2q}$. On the other hand, if $g^2=-1$ then we have
$\Bbb C^{2n}=V(i)\oplus V(-i)$, and these are Lagrangian subspaces of dimension $n$. 
So $\mathfrak{k}=\mathfrak{gl}_n(\Bbb C)$. The corresponding real form 
is the quaternionic orthogonal Lie algebra (symmetries of a quaternionic skew-Hermitian form), $\g_\theta=\mathfrak{so}^*(2n)$ (see Subsection \ref{morecla}). 
 
{\bf 6. Type $D_n$, the other inner class.} In this case $\theta$ is given by an element $g$ 
of $O(2n)$ such that $\det(g)=-1$ and $g^2=\pm 1$. Note that if $g^2=-1$ then , as shown above, $\det(g)=1$, so in the case at hand we always have $g^2=1$. Then $\Bbb C^{2n}=V(1)\oplus V(-1)$, but now the dimensions of these spaces are odd, $2p+1$ and $2q+1$ where 
$p+q=n-1$, and we may assume that $p\ge q$. So the real form is $\g_\theta=\mathfrak{so}(2p+1,2q+1)$, with $\mathfrak{k}=\mathfrak{so}_{2p+1}\oplus\mathfrak{so}_{2q+1}$. 
Note that for $n=3$, $D_3=A_3$, so we have $\mathfrak{so}(5,1)=\mathfrak{sl}(2,\Bbb H)$. Note also that this agrees with what we found before: the split form $\mathfrak{so}(n,n)$ is in the compact inner class 
for even $n$ and in the other one for odd $n$, and the quasi-split form $\mathfrak{so}(n+1,n-1)$ 
the other way around. 
 
\begin{exercise} Compute the subalgebras $\mathfrak{k}$ for all the real forms of classical simple Lie algebras. 
\end{exercise}  
 
\begin{exercise} Compute the correspondence between Vogan diagrams and real forms for classical simple Lie algebras. 
\end{exercise}  
 
\section{\bf Real forms of exceptional Lie algebras} 

\subsection{Equivalence of Vogan diagrams} 
For exceptional Lie algebras, it is convenient to make a more systematic use of Vogan diagrams (we could do this also for classical Lie algebras, but there we can also do everything explicitly 
using linear algebra). Recall that any real form comes from a certain Vogan diagram, but 
different Vogan diagrams may be equivalent, i.e., define the same real form. 
So our job is to describe this equivalence relation. 

First consider the case of the compact inner class. In this case 
the Vogan diagram is just the Dynkin diagram with black and white vertices 
(i.e., no matched vertices). Moreover, the case of all white vertices 
corresponds to the compact form, while the case when there are black vertices 
to noncompact forms. So let us focus on the latter case. 
Thus we have an element $\theta\in H\subset G_{\rm ad}$ such that 
$\theta\ne 1$ but $\theta^2=1$, but we are allowed to conjugate $\theta$ 
by elements of $N(H)$, i.e., transform it by elements of the Weyl group $W$. So how do simple reflections 
$s_i$ act on $\theta$ (in terms of its Vogan diagram)?

The Vogan diagram of $\theta$ is determined by the numbers $\alpha_j(\theta)=\pm 1$: 
if this number is $1$ then $j$ is white, and if it is $-1$ then $j$ is black. Now, we have 
$$
\alpha_j(s_i(\theta))=(s_i\alpha_j)(\theta)=(\alpha_j-a_{ij}\alpha_i)(\theta)=\alpha_j(\theta)\alpha_i(\theta)^{-a_{ij}}.
$$
This equals $\alpha_j(\theta)$ unless $\alpha_i(\theta)=-1$ and $a_{ij}$ is odd. 
Thus we obtain the following lemma. 

\begin{lemma} Suppose the Vogan diagram of $\theta$ contains a black vertex $i$. 
Then changing the colors of all neighbors $j$ of $i$ such that $a_{ij}$ is odd gives an equivalent Vogan diagram. 
\end{lemma} 

The same lemma holds, with the same proof, in the case of the other inner class (which for exceptional Lie algebras is possible only for $E_6$), except we should ignore the vertices
matched into pairs (so $i$ and $j$ should be $\theta$-stable vertices).    

\subsection{Classification of real forms} We are now ready to classify real forms of exceptional Lie algebras. 

{\bf 1. Type $G_2$.} We have two color configurations up to equivalence: $\circ\circ$ and $(\bullet\circ,\circ\bullet,\bullet\bullet)$. The first corresponds to the compact form $G_2^c$ and the second to the split form $G_2^{\rm spl}$. It is easy to check that in the second case $\mathfrak{k}=\mathfrak{sl}_2\oplus \mathfrak{sl}_2$ (indeed, it has dimension $6$ and rank $2$). So we don't have other real forms. 

{\bf 2. Type $F_4$.} Let $\alpha_1,\alpha_2$ be short roots and $\alpha_3,\alpha_4$ long roots. Then all nonzero off-diagonal $a_{ij}$ are odd except $a_{23}=-2$. So we may change the colors 
of the neighbors of any black vertex, except that if the black vertex is $2$ then we should not change the color of $3$. By such changes, we can bring the colors at $3,4$ into the form 
$\circ\circ$ or $\circ\bullet$, and then bring the colors at $1,2$ to the form $\circ\circ$ or $\bullet\circ$. So we are down to four configurations: 
$$
\circ\circ\circ\circ,\ \bullet\circ\circ\circ,\ \circ\circ\circ\bullet,\ \bullet\circ\circ\bullet  
$$ 

Moreover, the fourth case, $\bullet\circ\circ\bullet$, is actually equivalent to the third one, $\circ\circ\circ\bullet$.
This is seen from the chain of equivalences
$$
\circ\circ\circ\bullet=\circ\circ\bullet\bullet=\circ\bullet\bullet\circ=
\bullet\circ\bullet\circ=\bullet\bullet\bullet\bullet=\bullet\bullet\circ\bullet=\bullet\circ\circ\bullet
$$   
Thus we are left with three variants, 
$$
\circ\circ\circ\circ, \bullet\circ\circ\circ, \circ\circ\circ\bullet.
$$
The first configuration, $\circ\circ\circ\circ$,
corresponds to the compact form $F_4^c$. 

In the second case, $\bullet\circ\circ\circ$, $\alpha(\theta)=-1$
exactly when the root $\alpha$ has half-integer coordinates (recall that there are 16 such roots, see Subsection \ref{F4r}). 
Thus the Lie algebra $\mathfrak{k}$ is comprised by the root subspaces for roots with integer coordinates and the Cartan subalgebra, i.e., $\mathfrak{k}=\mathfrak{so}_9$ 
(type $B_4$). Also in this case $\p=S$, the spin representation 
of $\mathfrak{so}_9$. This is not the split form, since for the split form $\dim \mathfrak{k}$ 
should be $24$ and here it is $36$. Let us denote this form $F_4^1$.   

Thus, the third case, $\circ\circ\circ\bullet$, must be the split form, $F_4^{\rm spl}$.
We see that $\mathfrak{k}$ contains the 21-dimensional Lie algebra 
$\mathfrak{sp}_6=C_3$ (generated by the simple roots $\alpha_1,\alpha_2,\alpha_3$), so 
given that $\mathfrak{k}$ has rank $4$ and dimension $24$, we have 
$\mathfrak{k}=\mathfrak{sp}_6\oplus \mathfrak{sl}_2$. 

{\bf 3. Type $E_6$, split inner class.} In this case in the Vogan diagram two pairs of vertices are connected, so we can only color the two remaining vertices. So we have two equivalence classes of colorings -- $\circ\circ$ and $(\bullet\bullet,\bullet\circ,\circ\bullet)$. Let us show that they correspond to two different real forms. Consider first the $\circ\circ$ case. In this case $\theta$ is simply the diagram automorphism, so we have $\mathfrak{k}=F_4$, as the Dynkin diagram of $F_4$ is obtained by folding the Dynkin diagram of $E_6$ (check it!). This is not the split form since $\dim \mathfrak{k}=52$, but for the split form it is $36$; denote this form by $E_6^1$. 
So the split form $E_6^{\rm spl}$ corresponds to the second 
equivalence class $(\bullet\bullet,\bullet\circ,\circ\bullet)$. One can show that in this case 
$\mathfrak{k}=\mathfrak{sp}_8$, i.e., type $C_4$ (check it!).  

{\bf 4. $E_6,E_7,E_8$, compact inner class.} In this case the Vogan diagram 
has no arrows and just is the usual Dynkin diagram with vertices colored black and white.
One option is that all vertices are white, this corresponds to the compact forms $E_6^c, E_7^c,E_8^c$ ($\theta=1$). If there is at least one black vertex, then by using equivalence transformations we can make sure that the nodal vertex is black. Then flipping the color of its neighbors if needed, we can make sure that the vertex on the shortest leg is also black. This allows us to change the color of the nodal vertex whenever we want (as long as the vertex on the shortest leg remains black). 

We now want to unify the coloring of the long leg. 
We can bring the long leg to the following normal forms: 

$E_6$: $\circ\circ, \bullet\circ=\bullet\bullet=\circ\bullet$. But by flipping the colors on the neighbors of the nodal vertex, 
we see that $\bullet\circ$ and $\circ\circ$ are equivalent, so all patterns are equivalent to $\bullet\bullet$. 

$E_7$: $\circ\circ\circ, \bullet\circ\circ=\bullet\bullet\circ=\circ\bullet\bullet=\circ\circ\bullet$, 
$\bullet\circ\bullet=\bullet\bullet\bullet=\circ\bullet\circ$. But by flipping the colors on the neighbors of the nodal vertex, we see that all patterns are equivalent to $\bullet\bullet\bullet$. 

$E_8$: $
\circ\circ\circ\circ, \bullet\circ\circ\circ=\bullet\bullet\circ\circ=\circ\bullet\bullet\circ=\circ\circ\bullet\bullet=\circ\circ\circ\bullet,$ 
$$ \bullet\circ\circ\bullet=\bullet\circ\bullet\bullet=\bullet\bullet\bullet\circ=\begin{cases}\bullet\circ\bullet\circ\\ \circ\bullet\circ\circ\end{cases}
$$
$$
=\bullet\bullet\bullet\bullet=\bullet\bullet\circ\bullet=
\circ\bullet\bullet\bullet=\begin{cases}\circ\bullet\circ\bullet \\ \circ\circ\bullet\circ\end{cases}.
$$
But by flipping the colors on the neighbors of the nodal vertex, we see that all patterns are equivalent to $\bullet\bullet\bullet\bullet$. 

Thus we can always arrange all vertices on the long leg except possibly 
the neighbor of the node to be black, while the short leg and the node also remain black. 
In addition, as seen from the pictures above, in the cases $E_6$ and $E_8$ these two configurations are equivalent by transformations 
inside the leg. 

Now we can consider the configurations on the remaining leg (of length 2). 
The equivalence classes are $\circ\circ$ and $\bullet\circ=\circ\bullet=\bullet\bullet$. 

So in the case of $E_6$ and $E_8$ we get 
just two cases. It turns out that both for $E_6$ and $E_8$ 
these give two different real forms, one of which is split in the case of $E_8$. 
 
Consider first the $E_6$ case. One option is to take the Vogan diagram with just one black vertex, 
at the end of the long leg:
$$
\boxed{\begin{matrix}\circ & \circ & \circ &\circ &\bullet\\  &  & \circ &   &   &\end{matrix}}
$$
 Then $\mathfrak{k}=\mathfrak{so}_{10}\oplus \mathfrak{so}_2$ (as the black vertex corresponds to a minuscule weight). We denote this real form by $E_6^2$. 
On the other hand, if there is only one black vertex on the short leg, 
$$
\boxed{\begin{matrix}\circ & \circ & \circ &\circ &\circ \\  &  & \bullet &   &   &\end{matrix}}
$$
then 
$\mathfrak{k}$ contains $\mathfrak{sl}_6$, so this real form is different (as $\mathfrak{sl}_6$ is not a Lie subalgebra of $\mathfrak{so}_{10}$).   
It's not difficult to show that in this case $\mathfrak{k}=\mathfrak{sl}_6\oplus \mathfrak{sl}_2$. 
We denote this real form by $E_6^3$. 

Now consider the $E_8$ case. Again one option is the Vogan diagram with just one black vertex, 
at the end of the long leg: 
$$
\boxed{\begin{matrix}\circ & \circ & \circ &\circ &\circ &\circ &\bullet\\  &  & \circ &   &   &\end{matrix}}
$$
Then $\mathfrak{k}$ contains $E_7$, so this is not the split form since $\dim \mathfrak{k}\ge 133$ but for the split form it should be 120. In fact, it is not hard to see that $\mathfrak{k}=E_7\oplus \mathfrak{sl}_2$. We denote this real form by $E_8^1$. The second form is the split one, $E_8^{\rm spl}$. It can, for example, be obtained if we color black only one vertex, at the end of the middle leg:
$$
\boxed{\begin{matrix}\bullet & \circ & \circ &\circ &\circ &\circ &\circ \\  &  & \circ &   &   &\end{matrix}}
$$
 In fact, it's not hard to show that the algebra $\mathfrak{k}$ in this case is $\mathfrak{so}_{16}$. 

Finally, consider the $E_7$ case. In this case we have four options, but two of them end up being equivalent. Namely, 
we have 
$$
\boxed{\begin{matrix}\bullet & \circ & \bullet &\circ &\bullet&\bullet\\  &  & \bullet &   &   &\end{matrix}}=
\boxed{\begin{matrix}\bullet & \bullet & \bullet &\bullet &\bullet&\bullet\\  &  & \circ &   &   &\end{matrix}}=
\boxed{\begin{matrix}\circ & \bullet & \circ &\bullet &\bullet&\bullet\\  &  & \circ &   &   &\end{matrix}}=
$$
$$
\boxed{\begin{matrix}\circ & \bullet & \bullet &\bullet &\circ&\bullet\\  &  & \circ &   &   &\end{matrix}}=
\boxed{\begin{matrix}\circ & \circ & \bullet &\circ &\circ&\bullet\\  &  & \bullet &   &   &\end{matrix}}=
\boxed{\begin{matrix}\circ & \circ & \bullet &\circ &\bullet&\bullet\\  &  & \bullet &   &   &\end{matrix}}.
$$
So we are left with three cases, which all turn out different. The first one is just one black vertex 
at the end of the long leg:
$$
\boxed{\begin{matrix}\circ & \circ & \circ &\circ &\circ &\bullet\\  &  & \circ &   &   &\end{matrix}}
$$ 
In this case $\mathfrak{k}$ contains $E_6$, so this is not the split form, as $\dim \mathfrak{k}\ge 78$ but for the split form it is $63$. It is easy to see that 
$\mathfrak{k}=E_6\oplus \mathfrak{so}_2$ in this case (the black vertex corresponds to the minuscule weight). We denote this real form by $E_7^1$. The second option is a black vertex at the end of the middle leg:
$$
\boxed{\begin{matrix}\bullet & \circ & \circ &\circ &\circ &\circ \\  &  & \circ &   &   &\end{matrix}}
$$ 
Then $\mathfrak{k}$ contains $\mathfrak{so}_{12}$, of dimension 
$66$, so again not the split form. One can show that for this form $\mathfrak{k}=\mathfrak{so}_{12}\oplus \mathfrak{sl}_2$. We denote it by $E_7^2$. Finally, the split form $E_7^{\rm spl}$ is obtained when one colors black just the end of the short leg:
$$
\boxed{\begin{matrix}\circ & \circ & \circ &\circ &\circ &\circ \\  &  & \bullet &   &   &\end{matrix}}
$$ 
 Then $\mathfrak{k}$ contains $\mathfrak{sl}_7$ and one can show that $\mathfrak{k}=\mathfrak{sl}_8$. 

\begin{exercise} Work out the details of computation of $\mathfrak{k}$ 
for real forms of exceptional Lie algebras. 
\end{exercise} 

\begin{exercise}\label{G2ex} Let $\g$ be the complex Lie algebra of type $G_2$, and $G$ the corresponding Lie group. Let $\mathfrak{sl}_3\subset \g$ be the Lie subalgebra generated by long root elements and $SU(3)\subset G^c$ be the corresponding subgroup. 
Show that $G^c/SU(3)\cong S^6$. Use this to construct embeddings $G^c\hookrightarrow SO(7)$ and $G^c\hookrightarrow {\rm Spin}(7)$. 

{\bf Hint.} Consider the 7-dimensional irreducible representation of $G^c$. Show that it is of real type
(obtained by complexifying a real representation $V$) and then consider the action of $G^c$  on the set of unit vectors in $V$ under a positive invariant inner product. Then compute the Lie algebra of the stabilizer and use that the sphere is simply connected.
\end{exercise} 

\begin{exercise} Keep the notation of Exercise \ref{G2ex}. Show that one has ${\rm Spin}(7)/G^c=S^7$ and $SO(7)/G^c=\Bbb R\Bbb P^7$. 

{\bf Hint.} Let $S$ be the spin representation of ${\rm Spin}(7)$. Use that it is of real type (this can be deduced from Proposition \ref{typ}) and then consider the action of ${\rm Spin}(7)$ on vectors of norm $1$ in $S_{\Bbb R}$. Compute the Lie algebra of the stabilizer and use that the sphere is simply connected.
\end{exercise} 

\begin{remark} More generally, one can classify automorphisms 
of a simple complex Lie algebra $\g$ of arbitrary finite order. This was done by V. Kac using diagrams now known as {\bf Kac diagrams}, see \cite{OV}, Subsection 4.7. In particular, 
this approach can be applied to classify automorphisms 
of order $2$ which correspond to real forms of $\g$, see \cite{OV}, Subsection 5.5.
\end{remark} 

\section{\bf Classification of connected compact and complex reductive groups} 

\subsection{Connected compact Lie groups} 

We are now ready to classify connected compact Lie groups. We start with the following exercise.

\begin{exercise}\label{redu} Show that if $K^c$ is a compact Lie group then 
$\mathfrak{k}:={\rm Lie}(K^c)_{\Bbb C}$ is a reductive Lie algebra. 

{\bf Hint.} First use integration over $K^c$ to show that $\mathfrak{k}$ has a $K^c$-invariant 
positive definite Hermitian form. Then show that if $I$ is an ideal in $\mathfrak{k}$ 
then its orthogonal complement $I^\perp$ is also an ideal.  
\end{exercise} 

Now we can proceed. We already know many examples of compact connected Lie groups -
namely tori $(S^1)^r$ and also groups $G_{\rm ad}^c$ where $G_{\rm ad}={\rm Aut}(\g)^\circ$ for a semisimple Lie algebra $\g$. We can also consider products 
$(S^1)^r\times G_{\rm ad}^c$. Exercise \ref{redu} shows that the Lie algebra of any compact Lie group is isomorphic to one of such a product, so this should be an exhaustive list up to taking coverings and quotients by finite central subgroups. It thus remains to understand the nature of these coverings, which reduces to understanding $\pi_1(G_{\rm ad}^c)$. So our next task is to compute this group. In particular, we will show that it is finite. 

So let $\g$ 
be a semisimple complex Lie algebra and $G$ 
the corresponding simply connected complex Lie group 
(the universal cover of $G_{\rm ad}$). Let $Z$ be the kernel of the 
covering map $G\to G_{\rm ad}$, which is also $\pi_1(G_{\rm ad})$ and the center of $G$. 
The finite dimensional representations of $G$ are the same as those of $\g$, so the irreducible ones are $L_\lambda$, $\lambda\in P_+$. The center $Z$ acts by a certain character $\bchi_\lambda: Z\to \Bbb C^\times$ on each $L_\lambda$. Since $L_{\lambda+\mu}$ is contained in $L_\lambda\otimes L_\mu$, we have $\bchi_{\lambda+\mu}=\bchi_\lambda\bchi_\mu$, so $\bchi$ uniquely extends to a homomorphism $\bchi: P\to \Hom(Z,\Bbb C^\times)$. Also, by definition $\bchi_\theta=1$
(since the maximal root $\theta$ is the highest weight of the adjoint representation on which $Z$ acts trivially). 

Now, by Exercise \ref{intertw}, if $\lambda(h_i)$ are sufficiently large then for every root 
$\alpha$ of $\g$ we have $L_{\lambda+\alpha}\subset L_\lambda\otimes \g$. Thus 
$\bchi_{\lambda+\alpha}=\bchi_\lambda$, hence $\bchi_\alpha=1$. So $\bchi$ 
is trivial on the root lattice $Q$, i.e., defines a homomorphism $P/Q\to \Hom(Z,\Bbb C^\times)$, or, equivalently, $Z\to P^\vee/Q^\vee$. 

Note that the same argument works for $G_{\rm ad}^c$, its 
universal cover $G^c$, and its center $Z^c$  instead of $G_{\rm ad}$, $G$, $Z$. 

\begin{proposition}\label{repGad} 
A representation $L_\lambda$ of $\g$ of highest weight $\lambda\in P_+$
lifts to a representation of $G_{\rm ad}$ (or, equivalently, $G_{\rm ad}^c$) 
if and only if $\lambda \in P_+\cap Q$.   
\end{proposition} 

\begin{proof} We have just shown that if $\lambda\in P_+\cap Q$ then $L_\lambda$ lifts. 
The converse follows from Proposition \ref{dirsum} applied to $V=\g$.  
\end{proof} 

Now we can proceed with the classification of semisimple compact connected Lie groups. 
We begin with the following lemma from topology (see e.g. \cite{M}, Supplementary exercises to Chapter 13, p.500, Exercise 4). 

\begin{lemma}\label{fund} If $X$ is a connected compact manifold then the fundamental group 
$\pi_1(X)$ is finitely generated. 
\end{lemma} 

\begin{proof} (sketch) Cover $X$ by small balls, pick a finite subcover, connect the centers. 
We get a finite graph whose fundamental group maps surjectively to $\pi_1(X)$.  
\end{proof} 

\begin{theorem} Let $\g$ be a semisimple complex Lie algebra and $G^c_{\rm ad}$
the corresponding adjoint compact group. Then $\pi_1(G^c_{\rm ad})=P^\vee/Q^\vee$.
Thus the universal cover $G^c$ of $G^c_{\rm ad}$
is a compact Lie group. 
\end{theorem} 

\begin{proof} Let ${G^c_*}$ be a finite cover of $G_{\rm ad}^c$, 
and $Z_{G^c_*}\subset {G^c_*}$ be the kernel of the projection ${G^c_*}
\to G_{\rm ad}^c$. Then finite dimensional irreducible representations 
of ${G^c_*}$ are a subset of finite dimensional irreducible representations of $\g$, labeled by a subset $P_+({G^c_*})\subset P_+$ containing $P_+\cap Q$ (as by Proposition \ref{repGad} these are highest weights of representations of $G_{\rm ad}^c$). Let $P({G^c_*})\subset P$ be generated by $P_+({G^c_*})$. Let $\bchi_\lambda$ be the character by which $Z_{G^c_*}$ acts on the irreducible representation $L_\lambda$ of ${G^c_*}$. By Proposition \ref{repGad}, $\bchi$ defines an  injective homomorphism $\xi: P({G^c_*})/Q\to Z_{G^c_*}^\vee$. Since ${G^c_*}$ is compact, by the Peter-Weyl theorem this homomorphism is surjective, hence is an isomorphism.  

It remains to show that $\pi_1(G_{\rm ad}^c)$ is finite (then we can take ${G^c_*}$ to be the universal cover of $G_{\rm ad}^c$, in which case $P({G^c_*})=P$, so we get $P/Q\cong Z^\vee$, hence $Z=\pi_1(G_{\rm ad})\cong P^\vee/Q^\vee$). To this end, note that by Lemma \ref{fund}, $\pi_1(G_{\rm ad}^c)$ is a finitely generated abelian group. Take a subgroup of finite index $N$ in $\pi_1(G_{\rm ad}^c)$ and let ${G^c_*}$ be the corresponding cover. As we have shown, then $N=|Z_{G^c_*}|\le |P({G^c_*})/Q|\le |P/Q|$. But for finitely generated abelian groups this implies that 
the group is finite.  
\end{proof} 

This immediately implies the following corollary. 

\begin{corollary} (i) If $\g$ is a simple complex Lie algebra then the simply connected Lie group 
$G^c$ corresponding to the Lie algebra $\g^c$ is compact, and its center is 
$P^\vee/Q^\vee$, which also equals $\pi_1(G_{\rm ad}^c)$. 

(ii) Let $\Gamma\subset P^\vee/Q^\vee$ be a subgroup. Then the 
irreducible representations of $G^c/\Gamma$ are $L_\lambda$ 
such that $\lambda$ defines the trivial character of $\Gamma$. 

(iii) Let $G_i^c$ be the simply connected compact Lie group corresponding to a simple summand $\g_i$ of a semisimple Lie algebra $\g=\oplus_{i=1}^n\g_i$. Then any connected Lie group with Lie algebra $\g^c$ is compact and has the form $(\prod_{i=1}^nG_i^c)/Z$, where $Z=\pi_1(G^c)$ is a subgroup of $\prod_i Z_i$, and $Z_i=P_i^\vee/Q_i^\vee$ are the centers of 
$G_i^c$. Moreover, every semisimple connected compact Lie group has this form. 
\end{corollary} 

In particular, it follows that simply connected semisimple compact Lie groups are of the form 
$\prod_{i=1}^nG_i^c$, where $G_i^c$ are simply connected and simple.\footnote{We say that a connected Lie group $G$ is {\bf simple} if so is its Lie algebra. Thus this does not quite mean that $G$ is simple as an abstract group: it may have a finite center (e.g., $G=SU(2)$ or $SL_2(\Bbb C)$). For this reason such ``simple" groups are sometimes called {\bf almost simple}. However, the corresponding adjoint group $G_{\rm ad}$ is indeed simple as an abstract group.} 

\begin{corollary} Any connected compact Lie group is the quotient of $T\times C$ by a finite central subgroup, where $T=(S^1)^m$ is a torus and $C$ is compact, semisimple and simply connected.  
\end{corollary} 

\begin{proof} Let $L$ be such a group, $\mathfrak{l}$ its Lie algebra. It is reductive, so we can uniquely decompose $\mathfrak{l}$ as $\mathfrak{t}\oplus\mathfrak{c}$ where $\mathfrak{t}$ is the center 
and $\mathfrak{c}$ is semisimple. Let $T,C\subset L$ be the connected Lie subgroups 
corresponding to $\mathfrak{t},\mathfrak{c}$. It is clear that 
${\rm Lie}\overline T=\mathfrak{t}={\rm Lie}T$, so $T$ is closed, hence compact, hence a torus. Also since $\mathfrak{c}$ is semisimple with negative Killing form, $C$ is compact, hence closed. Thus we have a surjective 
homomorphism $T\times C\to L$ whose kernel is finite, as desired. 
\end{proof} 

\subsection{Polar decomposition} 

Now let us study the structure of the Lie subgroup $G_{{\rm ad},\theta}\subset G_{\rm ad}$ corresponding to the real form $\g_\theta\subset \g$ of a semisimple complex Lie algebra $\g$, namely, the group of fixed points 
of the antiholomorphic involution $\omega_\theta=\omega\circ \theta$ in $G_{\rm ad}$. 
It is clear that this subgroup is closed (${\rm Lie}\overline{G_{{\rm ad},\theta}}=\g_\theta={\rm Lie}G_{{\rm ad},\theta}$), but it may be disconnected: e.g. if $\g_\theta=\mathfrak{sl}_2(\Bbb R)$ then $G_{\rm ad}=PGL_2(\Bbb C)$, so $G_{{\rm ad},\theta}=PGL_2(\Bbb R)$, the quotient of 
$GL_2(\Bbb R)$ by scalars, which has two components. However,  
the results below apply mutatis mutandis to the connected group $G_{{\rm ad},\theta}^\circ$. 

Let $K^c\subset G_{{\rm ad},\theta}$ be the subgroup of elements acting on $\g$ by unitary operators; namely, $K^c$ is the set of fixed points of $\omega_\theta$ 
on $G_{\rm ad}^c$.\footnote{Of course, the group $K^c$ depends on $\theta$, but for simplicity we will not indicate this dependence in the notation.} This is a closed (possibly disconnected) subgroup of $G_{\rm ad}^c$ since ${\rm Lie}\overline{K^c}=\mathfrak{k}^c={\rm Lie}K^c$, hence it is compact. Also let $P_\theta:=\exp(\p_\theta)\subset G_{{\rm ad},\theta}$ (note that it is not a subgroup!). Since $\p_\theta$ acts on $\g$ by Hermitian operators, the exponential map 
$\exp: \p_\theta\to P_\theta$ is a diffeomorphism, so $P_\theta\subset G_{{\rm ad},\theta}$ 
is a closed embedded submanifold (the set of elements acting on $\g$ by positive Hermitian operators). 

\begin{theorem}\label{pdec} (Polar decomposition for $G_{{\rm ad},\theta}$) The multiplication map $\mu: K^c\times P_\theta\to G_{{\rm ad},\theta}$ is a diffeomorphism. Thus $G_{{\rm ad},\theta}\cong K^c\times \Bbb R^{\dim\p}$ as a manifold (in particular, $G_{{\rm ad},\theta}$ is homotopy equivalent to $K^c$). 
\end{theorem} 

\begin{proof} Recall that every invertible complex matrix $A$ can be uniquely written as a product 
$A=U_AR_A$, where $U=U_A$ is a unitary matrix and $R=R_A$ a positive Hermitian matrix, namely 
$R=(A^\dagger A)^{1/2}$, $U=A(A^\dagger A)^{-1/2}$ (the classical polar decomposition). 
Let us consider this decomposition for $g\in G_{{\rm ad},\theta}\subset {\rm Aut}(\g)\subset GL(\g)$. Since $g^\dagger g$ is an automorphism of $\g$ with positive eigenvalues, so is $(g^\dagger g)^{1/2}=R_g$, so $R_g\in P_\theta$ (a positive self-adjoint element in $G_{{\rm ad},\theta}$). Also since $U_g$ is unitary, it belongs to $K^c$. Thus 
the regular map $g\mapsto (U_g,R_g)$ is the inverse to $\mu$ (using the uniqueness of the polar decomposition).  
\end{proof} 

In particular, applying Theorem \ref{pdec} to complex Lie groups, we get 

\begin{corollary} The multiplication map defines a diffeomorphism 
$$
G_{\rm ad}^c\times\bold P\cong G_{\rm ad},
$$
 where $\bold P$
is the set of elements of $G_{\rm ad}$ acting on $\g$ by positive Hermitian operators.  
In particular, $\pi_1(G_{\rm ad})=\pi_1(G_{\rm ad}^c)=P^\vee/Q^\vee$. 
\end{corollary} 

\begin{corollary}\label{conta} If $G$ is a semisimple complex Lie group then the center $Z$ of $G$ is contained in $G^c$, i.e., coincides with the center $Z^c$ of $G^c$. Thus 
the restriction of finite dimensional representations from $G$ to $G^c$ 
is an equivalence of categories. 
\end{corollary}

This also implies that by taking coverings the polar decomposition applies verbatim 
to the real form $G_\theta=G^{\omega_\theta}\subset G$ of any connected complex semisimple Lie group $G$ instead of $G_{\rm ad}$. We note, however, that if $G$ is simply connected, then $G_\theta^\circ$ need not be. 
In fact, its fundamental group could be infinite. The simplest example is 
$G=SL_2(\Bbb C)$, then for the split form $G_\theta=SL_2(\Bbb R)$, which as we showed is homotopy equivalent to $SO(2)=S^1$, i.e. its fundamental group is $\Bbb Z$. 

\begin{example} 1. For $G_\theta=SL_n(\Bbb C)$ we have $K^c=SU(n)$ and 
$P_\theta$ is the set of positive Hermitian matrices of determinant $1$, so the polar decomposition  
in this case is the usual polar decomposition of complex matrices. 

2. For $G_\theta=SL_n(\Bbb R)$ we have $K^c=SO(n)$ 
and $P_\theta$ is the set of positive symmetric matrices of determinant $1$, so the polar decomposition in this case is the usual polar decomposition of real matrices. 
\end{example} 

\subsection{Connected complex reductive groups} \label{ccrg}

\begin{definition} A connected complex Lie group $G$ is {\bf reductive} if it is of the form 
$((\Bbb C^\times)^m\times G_{\rm ss})/Z$ where $G_{\rm ss}$ is semisimple and $Z$ is a finite central subgroup. 
A complex Lie group $G$ is reductive if $G^\circ$ is reductive and $G/G^\circ$ is finite. 
\end{definition}

\begin{example} $GL_n(\Bbb C)=(\Bbb C^\times\times SL_n(\Bbb C))/\mu_n$ is reductive. 
\end{example} 

It is clear that the Lie algebra ${\rm Lie}G$ of any complex reductive Lie group $G$ is reductive, and any complex reductive Lie algebra is the Lie algebra of a connected complex reductive Lie group. 
However, a simply connected complex Lie group with a reductive Lie algebra need not be reductive
(e.g. $G=\Bbb C$). 

If $G=((\Bbb C^\times)^m\times G_{\rm ss})/Z$ is a connected complex reductive Lie group 
then by Corollary \ref{conta}, $Z\subset (S^1)^m\times G_{\rm ss}^c\subset (\Bbb C^\times)^m\times G_{\rm ss}$, so we can define the compact subgroup $G^c\subset G$ by $G^c:=((S^1)^m\times G_{\rm ss}^c)/Z$, and vice versa. It is easy to see that this gives rise to a {\bf complexification functor} $K\mapsto K_{\Bbb C}$
\scriptsize
$$
\lbrace\text{connected compact Lie groups}\rbrace \to \lbrace\text{connected reductive complex Lie groups}\rbrace
$$
\normalsize such that $(G^c)_{\Bbb C}=G$, which defines a bijection between isomorphism classes. 
Moreover, restriction of finite dimensional representations from $G$ to $G^c$ is an equivalence, so representations of $G$ are completely reducible. The irreducible representations are parametrized by collections 
$(n_1,...,n_m,\lambda)$, $\lambda\in P_+(G_{\rm ss})$, $n_i\in \Bbb Z$, which define the trivial character of $Z$. 

We also see that the preimage $H\subset G$ of the standard maximal torus $H_{\rm ad}\subset G_{\rm ad}$ (see Subsection \ref{autom}) is a torus $H=\exp(\h)\cong (\Bbb C^\times)^r\subset G$ which contains the center $Z$ of $G$, and if $G$ is semisimple simply connected then the exponential map defines an isomorphism 
$\h/2\pi i Q^\vee\cong H$. 

\subsection{Linear groups} 

A connected Lie group $G$ (real or complex) is called {\bf linear} if it can be realized as a Lie subgroup of $GL_n(\Bbb R)$, respectively $GL_n(\Bbb C)$. We have seen that any complex semisimple group is linear. 
However, for real semisimple groups this is not so (e.g. the universal cover of $SL_2(\Bbb R)$ 
is not linear, see Exercise \ref{unico}). In fact, we see that we can characterize connected real semisimple linear groups as follows.

\begin{proposition} Suppose $\g_\theta$ 
is a real form of a semisimple complex Lie algebra $\g$, $G$ a connected complex Lie group with Lie algebra $\g$, and 
$G_\theta=G^{\omega_\theta}$. Then 
$G_\theta, G_\theta^\circ$ are linear groups. Moreover, every connected real semisimple linear Lie group is of the form $G_\theta^\circ$ for some connected complex semisimple Lie group $G$ and real form $G_\theta$.
\end{proposition} 

\begin{exercise} Classify simply connected real semisimple {\bf linear} Lie groups. 
\end{exercise} 

\subsection{Strongly regular elements in connected complex reductive groups}

In this subsection we would like to generalize the results of
Subsections \ref{strreg} and \ref{ccsub} on strongly regular elements
from Lie algebras to Lie groups. We repeat the arguments from these
subsections with suitable changes.

Let $G$ be a connected complex reductive group of rank $r$.

\begin{lemma}\label{ger} For every $g\in G$,
\[
\dim{\rm Ker}(\Ad_g-1)\ge r.
\]
\end{lemma}

\begin{proof} This holds on a nonempty open subset in a sufficiently small
neighborhood of the identity: if $g=\exp x$ with $x$ regular semisimple and
sufficiently small, then
\[
{\rm Ker}(\Ad_g-1)={\rm Ker}(\ad x),
\]
which is a Cartan subalgebra and hence has dimension $r$. Equivalently,
all $(\dim\g-r+1)$-minors of $\Ad_g-1$ vanish on this subset. Since these
minors are holomorphic functions of $g\in G$, they vanish identically by
the identity theorem, hence the statement holds for all $g\in G$.
\end{proof}

Now let $P_g(t)$ be the characteristic polynomial of the operator $\Ad_g$ on
$\g$. By Lemma \ref{ger}, $P_g(t+1)$ is divisible by $t^r$ for all
$g\in G$. Thus
\begin{equation}\label{ptp1}
P_g(t+1)=t^r(t^m+b_{m-1}(g)t^{m-1}+\cdots+b_0(g)),
\end{equation}
where $b_i$ are regular functions on $G$, in fact polynomials in the matrix
coefficients of the adjoint representation, and $m=\dim\g-r$.

\begin{definition}
An element $g\in G$ is called {\bf strongly regular} if $b_0(g)\ne 0$, i.e.,
if the generalized eigenspace of $\Ad_g$ with eigenvalue $1$ has dimension
$r$. The set of strongly regular elements of $G$ is denoted by $G^{\rm sr}$.
\end{definition}

\begin{proposition}\label{dimker}
(i) If the $1$-eigenspace of $\Ad_g$ for $g\in G$ coincides with its generalized
$1$-eigenspace, then $\g^g$ is a reductive Lie algebra.

(ii) If $g\in G^{\rm sr}$, then the $1$-eigenspace of $\Ad_g$ coincides
with its generalized $1$-eigenspace.

(iii) If $g\in G^{\rm sr}$, then $\g^g$ is a reductive Lie algebra.

(iv) $G^{\rm sr}$ is a dense open subset of $G$.
\end{proposition}

\begin{proof}
(i) Without loss of generality, we may assume that $G$ is semisimple.
Indeed, the center of $\g$ is contained in $\g^g$, and after quotienting by
it the assertion reduces to the semisimple case.

By the assumption on $g$, we have an orthogonal decomposition
\[
\g=\g^g\oplus \g',
\]
with respect to the Killing form of $\g$, where $\g'$ is the sum of the
generalized eigenspaces of $\Ad_g$ with eigenvalues different from $1$.
Hence the Killing form is nondegenerate on $\g^g$, so $\g^g$ is reductive
by Proposition \ref{reducri}.

Part (ii) immediately follows from Lemma \ref{ger}, and part (iii) follows
from (i) and (ii).

(iv) If $U$ is a small enough ball centered at $0$ in $\g$, then
\[
\exp(U)\cap G^{\rm sr}=\exp(U\cap\g^{\rm sr}),
\]
so $b_0$ is not identically zero. Since $b_0$ is a holomorphic function on
$G$, it follows that $G^{\rm sr}$ is open and dense.
\end{proof}

\begin{remark}
The set $G^{\rm sr}$ is connected, since the nonvanishing locus of a
nonzero holomorphic function in a connected complex manifold is connected
(this is not hard to deduce from the Weierstrass preparation theorem).
\end{remark}

\begin{proposition}\label{caaa}
If $g\in G^{\rm sr}$, then $\g^g\subset\g$ is a Cartan subalgebra.
\end{proposition}

\begin{proof}
It suffices to consider the case when $G$ is semisimple. We have the
orthogonal decomposition
\[
\g=\g^g\oplus \g'
\]
under the Killing form. Let $y\in\g^g$. Since $\g^g$ is a Lie subalgebra,
$\ad y$ preserves $\g^g$; by invariance of the Killing form it also
preserves the orthogonal complement $\g'$. Moreover, $\Ad_g$ commutes with
$\exp(t\ad y)$. Hence the operator
\[
\Ad_{g\exp(ty)}-1,
\]
where $t\in\Bbb C$, preserves this decomposition and is invertible on
$\g'$ for small $t$, since it is so for $t=0$.

Thus by \eqref{ptp1}, applied to $g\exp(ty)$, the zero-eigenvalue
contribution to $\Ad_{g\exp(ty)}-1$ has algebraic multiplicity at least
$r=\dim\g^g$. Since the operator is invertible on $\g'$, its restriction
to $\g^g$ has only the eigenvalue $0$, hence is nilpotent. Since $g$ acts
trivially on $\g^g$, this says that
\[
\Ad_{\exp(ty)}-1=\exp(t\ad y)-1
\]
is nilpotent on $\g^g$. Therefore $\ad y$ is nilpotent on $\g^g$. By
Engel's theorem, $\g^g$ is a nilpotent Lie algebra. Since $\g^g$ is
reductive, it follows that $\g^g$ is abelian.

Let $x\in\g^g$ and let $x=x_s+x_n$ be its Jordan decomposition inside the
reductive Lie algebra $\g^g$. Then for every $y\in\g^g$ we have
\[
{\rm tr}_\g(\ad x_n\,\ad y)=0,
\]
since $\ad x_n$ is nilpotent and commutes with $\ad y$. Since the Killing
form of $\g$ is nondegenerate on $\g^g$, this implies $x_n=0$. Hence
$\g^g$ consists of semisimple elements. As $\dim\g^g=r$, it is a Cartan
subalgebra of $\g$.
\end{proof}

\section{\bf Maximal tori in compact groups, Cartan decomposition}

\subsection{Maximal tori in connected compact Lie groups} 

Let $\g$ be a complex semisimple Lie algebra, $\g^c$ its compact form, $G$ 
a connected Lie group with Lie algebra $\g$, $G^c\subset G$ its compact part (the connected Lie subgroup with Lie algebra $\g^c$), as above. 

A {\bf Cartan subalgebra} $\h^c\subset \g^c$ is a maximal commutative Lie subalgebra 
(note that it automatically consists of semisimple elements since all elements of $\g^c$ are semisimple). 
In other words, it is a subspace such that $\h^c\otimes_{\Bbb R}\Bbb C$ is a Cartan subalgebra of $\g$. 

Recall that all Cartan subalgebras of $\g$ are conjugate, 
even if equipped with a system of simple roots (Theorem \ref{conjcar}). Namely, given two such subalgebras 
$(\h,\Pi)$ and $(\h',\Pi')$, there is $g\in G$ such that 
${\rm Ad}_g (\h,\Pi)=(\h',\Pi')$. It turns out that the same result holds for $\g^c$. 

\begin{lemma}\label{compconj} 
Any two Cartan subalgebras in $\g^c$ equipped with systems of simple roots 
are conjugate under $G^c$. 
\end{lemma} 

\begin{proof} Given $(\h^c,\Pi)$ 
and $(\h^{c'},\Pi')$, there is $g\in G$ such that 
${\rm Ad}_g (\h^c,\Pi)=(\h^{c'},\Pi')$. Then we also have 
${\rm Ad}_{\overline g}(\h^c,\Pi)=(\h^{c'},\Pi')$, 
where $\overline g:=\omega(g)$. So 
$\overline g^{-1}g$ commutes with $\h^c$ and preserves $\Pi$, i.e., 
$\overline g h=g$, $h\in H:= \exp(\h^c_{\Bbb C})$. 
Writing $g=kp$, where $k\in G^c$, 
$p\in \bold P$, we have 
$kp^{-1}h=kp$, so $h=p^2$. Since $p$ is positive, 
$p=h^{1/2}$, so it commutes with $\h^c$ and preserves $\Pi$, thus 
${\rm Ad}_k (\h^c,\Pi)=(\h^{c'},\Pi')$, as claimed. 
\end{proof} 

Note that for every Cartan subalgebra $\h^c\subset \g^c$, 
$H^c=\exp(\h^c)\subset G^c$ is a torus, which is clearly a {\bf maximal torus}: the Lie algebra of any torus $T\subset G^c$ containing $H^c$ would be a commutative subalgebra of $\mathfrak g^c$ containing $\h^c$, so it must coincide with $\mathfrak h^c$, hence $T=H^c$. 
Conversely, if $H^c\subset G^c$ is a maximal torus then 
${\rm Lie}(H^c)$ can be included in a Cartan subalgebra, hence it is itself a Cartan subalgebra. 
So the exponential map defines a bijection between Cartan subalgebras in $\g^c$ and maximal tori in 
$G^c$. The same statements apply to Cartan subalgebras in $\g$ and maximal tori in 
$G$.  Moreover, by the last paragraph of Subsection \ref{ccrg}, $H^c,H$ are maximal abelian subgroups of $G^c,G$, since the same is true for $G_{\rm ad}$ (Subsection \ref{autom}). 

\begin{corollary}\label{conjuu} (i) Any two maximal tori in $G$ or $G^c$ equipped with 
systems of simple roots are conjugate.  

(ii) Any element $g\in G^{\rm rs}$ is contained in a maximal torus $H\subset G$. 
\end{corollary} 

\begin{proof} Part (i) follows from the discussion preceding the corollary. 
For (ii), let $\h=\g^g$. By Proposition \ref{caaa} it is a Cartan subalgebra
of $\g$, hence $H:=\exp(\h)$ is a maximal torus of $G$ which commutes with $g$. 
Since $H$ is a maximal abelian subgroup of $G$, it follows that $g\in H$. 
\end{proof} 

We also have 
 
\begin{theorem}\label{thmco} (i) Every element of a connected compact Lie group $K$ is contained in a maximal torus.

(ii) All maximal tori in $K$ are conjugate (even when equipped with 
systems of simple roots).  
\end{theorem}
\begin{proof} We may assume without loss of generality that $K$ is semisimple, i.e., $K=G^c$ for a connected semisimple complex Lie group $G$, which implies (ii). To prove (i), note that since the preimage in $G^c$ of a maximal torus of $G^c_{\rm ad}$ is a maximal torus of $G^c$, it suffices to consider the case $K=G_{\rm ad}^c$. Let $K'\subset K$ be the set of elements contained in a maximal torus. 
Fix a maximal torus $T\subset K$ and consider the
map $f: K\times T\to K$ given by $f(k,t)=ktk^{-1}$, whose image is $K'$. 
This implies that $K'$ is compact, hence closed, so $K\setminus K'$ is open. 

On the other hand, by Corollary \ref{conjuu}(ii), $(K\setminus K')\cap G^{\rm sr}=\emptyset$, so 
for all $g\in K\setminus K'$ we have $b_0(g)=0$. Thus $K\setminus K'$ is empty and $K'=K$.
\end{proof} 

This immediately implies

\begin{corollary}
The exponential map $\exp: \g^c\to G^c$ is surjective.\footnote{Here is another proof of this corollary. Let $B(x,y)$ be the Killing form of $\g^c$. Since $B$ is negative definite, the form $-B$ extends to a bi-invariant Riemannian metric on $G_c$. Since $G^c$ is compact, the Hopf-Rinow theorem guarantees that for any $g\in G^c$ there is a geodesic on $G^c$ in this metric connecting $1$ and $g$. But it is easy to see that this geodesic is a segment of a one-parameter subgroup of $G^c$, which implies the statement.}
\end{corollary} 

\begin{exercise} Is the exponential map surjective for the group $SL_2(\Bbb C)$? 
\end{exercise} 

\subsection{Semisimple and unipotent elements}

Let $G$ be a connected reductive complex Lie group. 
An element $g\in G$ is called {\bf semisimple} if it acts in every finite dimensional representation of $G$ by a semisimple (=diagonalizable) operator, and {\bf unipotent} if it acts in every finite dimensional representation of $G$ by a unipotent operator (all eigenvalues are $1$). For example, Corollary \ref{conjuu}(ii) implies that 
every regular semisimple element is semisimple. 

\begin{exercise} Let $Y$ be a faithful finite dimensional representation of $G$ 
(it exists by Corollary \ref{exfaith}). Show that $g\in G$ is semisimple if and only if it acts semisimply 
on $Y$, and unipotent if and only if it acts unipotently on $Y$. 

{\bf Hint:} Use Proposition \ref{dirsum}.  
\end{exercise}

\begin{exercise} Show that if $G$ is semisimple then 
the exponential map defines a homeomorphism between 
the set of nilpotent elements in $\g={\rm Lie} G$ and the set of unipotent elements in $G$. 
\end{exercise} 

\begin{exercise}\label{redu1} Let $Z$ be the center of a connected complex reductive group $G$. 

(i) Show that the homomorphism $\pi: G\to G/Z$ defines a bijection between unipotent elements 
of $G$ and unipotent elements of $G/Z$. 

(ii) Show that the set of semisimple elements of $G$ is the preimage under $\pi$ 
of the set of semisimple elements of $G/Z$. 
\end{exercise} 

\begin{proposition}\label{Jo} (i) (Jordan decomposition in $G$). Every element 
\linebreak $g\in G$ has a unique factorization $g=g_sg_u$, where 
$g_s\in G$ is semisimple, $g_u\in G$ is unipotent 
and $g_sg_u=g_ug_s$. 

(ii) $g\in G$ is contained in a maximal torus if and only if it is semisimple.
\end{proposition} 

\begin{proof} The proof of (i) is Exercise \ref{jordec}. 
For (ii) we only need to prove the ``if" direction. So suppose $g\in G$ is semisimple, 
then $\g^g$ is reductive by Proposition \ref{dimker}(i). 
 Taking a regular semisimple $y\in \g^g$, we find that 
for small $t$ the Lie subalgebra $\h:=\g^{g\exp(ty)}=(\g^g)^{\exp(ty)}$ 
is a Cartan subalgebra of $\g^g$. But $\dim\h=\dim\g^{g\exp(ty)}\ge r$. 
Thus $\dim \h=r$ and $\h$ is a Cartan subalgebra of $\g$. 
Thus the maximal torus $H:=\exp(\h)$ is a maximal abelian subgroup of $G$
which commutes with $g$. Hence $g\in H$. 
\end{proof} 

Note that the same argument applied to $G^c,\g^c$ instead of $G,\g$ gives another proof of Theorem \ref{thmco}(ii). 

\begin{exercise}\label{jordec} Prove Proposition \ref{Jo}(i). 

{\bf Hint.} Use Exercise \ref{redu1} to reduce to the case when $G=G_{\rm ad}$ is a semisimple adjoint group. In this case, write ${\rm Ad}_g$ as $su$, where $s$ is a semisimple and $u$ a unipotent operator 
with $su=us$ (Jordan decomposition for matrices). Show that $s={\rm Ad}_{g_s}$ and $u={\rm Ad}_{g_u}$ 
for some commuting $g_s,g_u\in G_{\rm ad}$. Then establish uniqueness using the uniqueness of Jordan decomposition of matrices. 
\end{exercise} 

\begin{proposition}\label{srss} Any strongly regular element $g\in G$ is semisimple. 
\end{proposition} 

\begin{proof}
Let $g=g_sg_u$
be the Jordan decomposition of $g$. Write
$g_u=\exp x$
with \(x\in\g\) nilpotent. Since \(g\) commutes with \(g_u\), it fixes \(x\);
hence
$x\in \g^g$.
But by Proposition \ref{caaa}, \(\g^g\) is a Cartan subalgebra, hence
consists of semisimple elements. Thus \(x\) is both nilpotent and
semisimple, so \(x=0\). Therefore \(g_u=\exp x=1\), and \(g\) is
semisimple.
\end{proof} 

In view of Proposition \ref{srss}, strongly regular elements are also called {\bf regular semisimple}.

\subsection{Maximal abelian subspaces of $\p_\theta$}

Let $G$ be a connected complex semisimple group, $G_\theta\subset G$ a real form, 
$\g_\theta\subset \g$ their Lie algebras. We have the polar decomposition 
$G_\theta=K^c P_\theta$ and the additive version 
$\g_\theta=\mathfrak{k}^c\oplus \p_\theta$, with $\p_\theta=i\p^c$.
Also $\g^c=\mathfrak{k}^c\oplus \p^c$. 

\begin{proposition}\label{maxab} (i) Let $\a$ be a maximal abelian subspace of $\p_\theta$. Then the centralizer $\mathfrak z$ of $\a$ in $\g^c$ has the form $\mathfrak{m}\oplus i\mathfrak{a}$, where $\mathfrak{m}$ is a reductive Lie algebra contained in $\mathfrak{k}^c$. Moreover, if $\mathfrak{t}$ is a Cartan subalgebra 
of $\mathfrak{m}$ then $\mathfrak{t}\oplus i\mathfrak{a}$ is a Cartan subalgebra of $\g^c$ and $\mathfrak{t}\oplus \mathfrak{a}$ is a Cartan subalgebra of $\g_\theta$.

(ii) If $a\in \a$ is sufficiently generic then the centralizer of $a$ in $\p_\theta$ is $\a$. 

(iii) For any $p\in \p_\theta$ there exists $k\in K^c$ such that ${\rm Ad}_k(p)\in \a$.

(iv) All maximal abelian subspaces of $\p_\theta$ are conjugate by $K^c$. 
\end{proposition} 

\begin{proof} (i)
Let $x\in \g^c$, $[x,\a]=0$. Write $x=x_++x_-$, $x_+\in \mathfrak{k}^c,x_-\in \p^c$. Then $[x_\pm,\a]=0$, thus $x_-\in i\a$ by maximality of $\a$. So 
$x\in \mathfrak{k}^c\oplus i\a$. Thus $\mathfrak z=\mathfrak{m}\oplus i\a$ where $\mathfrak m\subset \mathfrak k^c$ 
is a reductive Lie algebra. Moreover, if $\mathfrak{t}\subset \mathfrak{m}$ 
is a Cartan subalgebra then $\mathfrak{t}\oplus i\mathfrak{a}$ is a maximal 
abelian subalgebra of $\g^c$, hence is a Cartan subalgebra. Similarly, $\mathfrak{t}\oplus \mathfrak{a}$ is a Cartan subalgebra of $\g_\theta$.

(ii) Consider the group $T_\a:=\exp(i\a)\subset G^c$. It is clear from (i) that this is a compact torus. Thus for a generic enough $a\in \a$, the 1-parameter subgroup 
$e^{ita}$ is dense in $T_\a$. So if $p\in \p_\theta$ and $[p,a]=0$ 
then $e^{ita}$ commutes with $p$, hence so do $T_\a$ and $\a$. 
So by maximality of $\a$ we have $p\in \a$.  

(iii) Let $a\in \a$ be generic enough as in (ii). Then by (ii), ${\rm Ad}_k(p)\in \a$ 
if and only if $[{\rm Ad}_k(p),a]=0$. 

Consider the function 
$f: K^c\to \Bbb R$ given by $f(b):=({\rm Ad}_b(p),a)$. 
This function is continuous, so attains a maximum
on the compact group $K^c$. Suppose $k$ is a maximum point of $f$. 
Let $p_0:={\rm Ad}_k(p)$. Differentiating $f$ at $k$, we get 
$([x,p_0],a)=0$ for all $x\in \mathfrak{k}^c$. Thus 
$(x,[p_0,a])=0$ for all $x\in \mathfrak{k}^c$. But $[p_0,a]\in \mathfrak{k}^c$ and 
the inner product on $\mathfrak{k}^c$ is nondegenerate. 
Thus $[p_0,a]=0$, as desired. 

(iv) Let $\a,\a'$ be maximal abelian subspaces of $\p_\theta$. 
Pick a generic element $p\in \a'$ as in (ii). By (iii) we can find $k\in K^c$ such that ${\rm Ad}_k(p)=a\in \a$. Moreover, $a$ is generic in  ${\rm Ad}_k(\a')$. So for every $b\in \a$ we have $[b,{\rm Ad}_k(\a')]=0$ (as $[b,a]=0$). By maximality
of $\a'$ this implies that $b\in {\rm Ad}_k(\a')$, i.e., $\a\subset {\rm Ad}_k(\a')$. 
Thus $\dim \a\le \dim \a'$. Switching $\a,\a'$, we also get $\dim \a'\le \dim \a$, hence $\dim \a=\dim \a'$ and $\a={\rm Ad}_k(\a')$, as claimed. 
\end{proof} 

\subsection{The Cartan decomposition of semisimple linear groups}

Let $\a\subset \p_\theta$ be a maximal abelian subspace and $A=\exp(\a)\subset P_\theta \subset G_\theta$. This is a subgroup isomorphic to $\Bbb R^n$, where $n=\dim \a$. 

\begin{theorem}\label{cardec} (The Cartan decomposition) 
We have $G_\theta=K^c AK^c$. 
In other words, every element $g\in G_\theta$
has a factorization $g=k_1ak_2$, $k_1,k_2\in K^c$, 
$a\in A$.\footnote{This factorization is not unique.} 
\end{theorem} 

\begin{proof} Recall that we have the polar decomposition 
$G_\theta=K^cP_\theta$. Thus it suffices to show that every $K^c$-orbit on  
$P_\theta$ intersects $A$. To do so, take $Y\in P_\theta$ 
and let $y=\log Y\in \p_\theta$. By Proposition 
\ref{maxab} there is $k\in K^c$ 
such that ${\rm Ad}_k(y)\in\a$. 
Then ${\rm Ad}_k(Y)\in A$, as claimed.  
\end{proof} 

\begin{remark} Theorem \ref{cardec} has a straightforward generalization to reductive groups. 
\end{remark} 

\begin{example} 1. For $G_\theta=GL_n(\Bbb C)$, Theorem \ref{cardec} reduces to a classical theorem in linear algebra: any invertible complex matrix can be written as $U_1DU_2$, where $U_1,U_2$ are unitary and 
$D$ is diagonal with positive entries. 

2. Similarly, for $G_\theta=GL_n(\Bbb R)$, Theorem \ref{cardec} says that 
any invertible real matrix can be written as $O_1DO_2$, where $O_1,O_2$ are orthogonal and $D$ is diagonal with positive entries. 
\end{example} 

\subsection{Maximal compact subgroups}

\begin{theorem}\label{conjcom} (E. Cartan) Let $G_\theta$ be a real form of a connected semisimple complex group $G$. Then any compact subgroup $L$ of $G_\theta$ is conjugate to a subgroup of $K^c$ by an element of $P_\theta$. Also every compact subgroup of $G_\theta$ is contained in a maximal one. Thus all maximal compact subgroups of $G_\theta$ are conjugate (to $K^c$). 
\end{theorem} 

\begin{proof} We give a simplified version of Cartan's proof, due to G. D. Mostow. 

First note that $K^c$ is a maximal compact subgroup of $G_\theta$. Indeed, if $L\supset K^c$ 
is a compact subgroup then the polar decomposition implies that $L=K^c\cdot (P_\theta\cap L)$. But if 
$Y\in P_\theta\cap L$ and $Y\ne 1$ then the sequence $Y^n\in L$ has no convergent subsequence (which is clear by looking at the eigenvalues of $Y^n$ on $\g_\theta$).
Thus $L=K^c$. 

It remains to prove that every compact subgroup $L\subset G_\theta$ can be conjugated into $K^c$ by an element 
of $P_\theta$. The idea of proof is to define an $L$-invariant continuous real-valued function $f$ on $P_\theta$
and show that it has a unique minimum $Y$ using a convexity argument. Then a required conjugating 
element is $Y^{-\frac{1}{2}}$. 

So let us proceed with this plan. Recall that we have a decomposition of the Lie algebra 
$\g_\theta:={\rm Lie}(G_\theta)$ given by $\g_\theta=\mathfrak{k}^c\oplus \p_\theta$, which is the eigenspace decomposition of $\theta$, and 
that the Killing form $B=B_\g$ is positive on $\p_\theta$, negative 
on $\mathfrak{k}^c$, and $\theta$-invariant. Thus we have a positive definite inner product on the real vector space $\g_\theta$ given by 
$$
B_\theta(x,y):=-B(x,\theta(y)).
$$ 
Denote by $A^\dagger$ the adjoint operator to $A\in {\rm End}(\g_\theta)$ 
under this inner product. Then $A:={\rm Ad}_{g}$ is orthogonal ($A^\dagger=A^{-1}$) 
for $g\in K^c$, while for $g\in P_\theta$ it is self-adjoint ($A^\dagger=A$), unimodular and positive definite as its eigenvalues are positive). So if $g=kp$ with $k\in K^c$, $p\in P_\theta$ then $\overline g=kp^{-1}$, hence 
\begin{equation}\label{Adkp}
{\rm Ad}_g^\dagger={\rm Ad}_{kp}^\dagger={\rm Ad}_{p}^\dagger{\rm Ad}_{k}^\dagger=
{\rm Ad}_{p}{\rm Ad}_{k}^{-1}={\rm Ad}_{pk^{-1}}={\rm Ad}{\overline g^{-1}}.
\end{equation}

Let 
$$
S:=\int_{L} {\rm Ad}_h^\dagger {\rm Ad}_hdh\in {\rm End}(\g_\theta).
$$
Then $S$ is a self-adjoint positive definite operator. So it admits 
an orthonormal eigenbasis $v_i$ with eigenvalues $\lambda_i>0$.
Let $\lambda_{\rm min}$ be the smallest of these eigenvalues. 

Consider the function $f: P_\theta\to \Bbb R$ given by 
$$
f(X):={\rm Tr}({\rm Ad}_X\cdot S)=\sum_i \lambda_iB_{\theta}({\rm Ad}_Xv_i,v_i).
$$
So, since ${\rm Ad}_X$ is positive definite, we have 
\begin{equation}\label{inequa} 
f(X)\ge \lambda_{\rm min}{\rm Tr}({\rm Ad}_X).
\end{equation} 
Note also that the group $G_\theta$ acts on $P_\theta$ by 
$g\circ X=gX\overline{g}^{-1}$, and by \eqref{Adkp} 
the function $f$ is $L$-invariant. 

Recall that for any $R>0$ the set of unimodular positive symmetric matrices 
$A$ with ${\rm Tr}(A)\le R$ is compact, since so is its subset of diagonal matrices, and any such matrix can be diagonalized by an orthogonal transformation. Since ${\rm Ad}_X$ is a positive self-adjoint operator on $\g_\theta$ with respect to $B_\theta$, it follows from \eqref{inequa} that 
the set of $X\in P_\theta$ with $f(X)\le R$ is compact. This implies that $f$, being continuous, attains a minimum on $P_\theta$. Suppose it attains a minimum at 
the point $Y=\exp(y)$, $y\in \p_\theta$. 

\begin{proposition} This minimum point is unique. 
\end{proposition} 

\begin{proof} Suppose $Z=\exp(z)$, $z\in \p_\theta$ is another minimum point. Consider the polar decomposition of the element $\exp(-\tfrac{z}{2})\exp(\tfrac{y}{2})\in G_\theta$: 
$$
\exp(\tfrac{z}{2})\exp(-\tfrac{y}{2})=k\exp(\tfrac{x}{2}),
$$
$k\in K^c$, $x\in \mathfrak p_\theta$. It follows that 
$$
\exp(\tfrac{x}{2})=\exp(-\tfrac{y}{2})\exp(\tfrac{z}{2})k=k^{-1}\exp(\tfrac{z}{2})\exp(-\tfrac{y}{2}),
$$
so multiplying, we get
$$
\exp(x)=\exp(-\tfrac{y}{2})\exp(z)\exp(-\tfrac{y}{2})
$$
and thus 
\begin{equation}\label{t1}
\exp(z)=\exp(\tfrac{y}{2})\exp(x)\exp(\tfrac{y}{2}).
\end{equation}
Consider the function 
$$
F(t)=f(\exp(\tfrac{y}{2})\exp(tx)\exp(\tfrac{y}{2})),\ t\in \Bbb R. 
$$
This function has a global minimum at $t=0$, and also at $t=1$ in view of \eqref{t1}. 
Thus the function $F$ is not strictly convex. On the other hand, we have the following lemma. 

\begin{lemma}\label{lina} Let $a,M$ be symmetric real matrices such that $M$ is positive definite. Then the function 
$$
\phi(t):={\rm Tr}(\exp(ta)M),\ t\in \Bbb R
$$
is convex, and is strictly convex if $a\ne 0$. 
\end{lemma} 

\begin{proof} Conjugating $a,M$ simultaneously by an orthogonal matrix, we may assume that $a$ is diagonal, with diagonal entries $a_i$. Then we have  
$$
\phi(t):=\sum_i M_{ii}\exp(ta_i).
$$
Since $M$ is positive definite, $M_{ii}>0$ and the statement follows. 
\end{proof} 

Using Lemma \ref{lina} for $a:={\rm ad}x$ and $M:=\exp(\frac{{\rm ad}y}{2})S\exp(\frac{{\rm ad}y}{2})$ and the fact that $F(t)$ is not strictly convex, we get that ${\rm ad}x=0$, hence $x=0$ (as $\g$ is semisimple) and $y=z$, as claimed. 
\end{proof} 

Now, since the function $f$ has a unique minimum point and is $L$-invariant, this minimum point must also be $L$-invariant. Thus we have 
$h\exp(y)=\exp(y)\overline h$ for all $h\in L$. It follows that 
$$
\exp(-\tfrac{y}{2})h\exp(\tfrac{y}{2})=\exp(\tfrac{y}{2})\overline h\exp(-\tfrac{y}{2})=
\overline{\exp(-\tfrac{y}{2})h\exp(\tfrac{y}{2})}.
$$
Thus the element $p:=\exp(-\tfrac{y}{2})=Y^{-\frac{1}{2}}$ conjugates $L$ into $K^c$. 
\end{proof} 

\subsection{Cartan subalgebras in real semisimple Lie algebras} 

Recall that a Cartan subalgebra of a real semisimple Lie algebra $\g_\theta$ is a maximal commutative subalgebra consisting of semisimple elements. We have seen that Cartan subalgebras in a complex semisimple Lie algebra are conjugate, but this is not so for real semisimple Lie algebras, as demonstrated by the following exercise. 

\begin{exercise}\label{carclas} (i) Let $\g=\mathfrak{sl}_n(\Bbb R)$. For $0\le m\le \frac{n}{2}$, let $\h_m$ be the space of matrices of the form 
$$
A=\bigoplus_{i=1}^m \begin{pmatrix} a_i& b_i\\ -b_i & a_i\end{pmatrix}\oplus {\rm \diag}(c_1,...,c_{n-2m})
$$
such that ${\rm Tr}(A)=0$. Show that $\h_m$ is a Cartan subalgebra of $\g$ and that 
$\h_m$ is not conjugate to $\h_{m'}$ when $m\ne m'$ (look at eigenvalues of elements of $\h_m$ in the vector representation). Conclude that Lemma \ref{compconj} does not necessarily hold for non-compact forms of $\g$. 

(ii) Show that every Cartan subalgebra in $\g$ is conjugate to one of the form $\h_m$ for some $m$. 

(iii) Classify Cartan subalgebras in other classical real simple Lie algebras (up to conjugacy). 
\end{exercise} 

Let us say that a semisimple element of $\g_\theta$ is {\bf split} if it acts on $\g_\theta$ with real eigenvalues, and say that a commutative Lie subalgebra 
of $\g_\theta$ is a {\bf split subalgebra} if it consists of split elements. An invariant of a Cartan subalgebra $\h\subset \g_\theta$ under conjugation is the dimension 
$s(\h)$ of the largest split subalgebra of $\h$ (consisting of all split elements of $\h$). For example, a split real form $\g_\theta$ has a split Cartan subalgebra with $s(\h)=r={\rm rank}\g$, and conversely, a real form that admits a split Cartan subalgebra is split. Also, in Exercise \ref{carclas}, $s(\h_m)=n-1-m$. 

Let us say that $\h$ is {\bf maximally split} if $s(\h)$ is the largest possible, and {\bf maximally compact} if $s(\h)$ is the smallest possible. For example, in 
Exercise \ref{carclas}, $\h_0$ is maximally split and $\h_{[n/2]}$ is maximally compact (where $[n/2]$ is the floor of $n/2$). Also, a split Cartan subalgebra 
is maximally split and a compact one (i.e., one for which $\exp(\h)$ is a compact torus) is maximally compact, if they exist. Finally, the Cartan subalgebra 
$\h_+^c\oplus i\h_-^c$, where $\h_+^c,\h_-^c$ are as in the proof of Proposition \ref{maxcomp}, is maximally compact. 

Note that $s(\h)$ may also be interpreted as the signature of the Killing form restricted to $\h$, which equals 
$(s(\h),r-s(\h))$.  

\begin{theorem} 
(i) A $\theta$-stable Cartan subalgebra $\h\subset \g_\theta$ is maximally split iff $\h_-:=\h\cap \p_\theta$ is a maximal abelian subspace in $\p_\theta$.

(ii) A $\theta$-stable Cartan subalgebra $\h\subset \g_\theta$  is maximally compact 
iff $\h_+:=\h\cap \mathfrak{k}^c$ is a Cartan subalgebra in $\mathfrak{k}^c$, and in this case 
$s(\h)={\rm rank}\g-{\rm rank}(\mathfrak{k})$. 

(iii) Any two maximally split $\theta$-stable Cartan subalgebras are conjugate by $K^c$.  

(iv) Any two maximally compact $\theta$-stable Cartan subalgebras 
are conjugate by $K^c$.

(v) Any Cartan subalgebra in $\g_\theta$ is conjugate to a 
$\theta$-stable one by an element of $G_\theta$ (or, equivalently, $P_\theta$). 
\end{theorem} 

\begin{proof} (i) It is clear that if $\h_-$ is a maximal abelian subspace 
of $\p_\theta$ then $\h$ is maximally split, since 
by Proposition \ref{maxab} any abelian subspace  
of $\p_\theta$ can be conjugated into $\h_-$. Conversely, 
if $\h$ is maximally split, suppose that $a\in \p_\theta, a\notin \h_-$ 
with $[a,\h_-]=0$. Then $\h_-'=\h_-\oplus \Bbb R a$, and let $\h'$ 
be a Cartan subalgebra of $\g_\theta$ containing $\h'_-$. 
Then $s(\h')>s(\h)$, a contradiction. 

(ii) It is clear that if $\h_+$ is a Cartan subalgebra of $\mathfrak{k}^c$ 
then $\h$ is maximally compact. Also given a Cartan 
subalgebra $\mathfrak{\h_+}\subset \mathfrak{k}^c$, take a 
Cartan subalgebra $\h$ of $\g_\theta$ containing $\h_+$. 
Then $s(\h)\le {\rm rank}\g-{\rm rank}(\mathfrak{k})$. 
This implies that for any maximally compact $\h$, we have that 
$\h\cap \mathfrak{k}^c$ is a Cartan subalgebra in  $\mathfrak{k}^c$, and 
$s(\h)={\rm rank}\g-{\rm rank}(\mathfrak{k})$. 

(iii) Let $\h,\h'$ be maximally split $\theta$-stable Cartan subalgebras in $\g_\theta$. Then $\h_-,\h_-'$ are maximal abelian subspaces of $\p_\theta$. 
So they are conjugate by $K^c$ by Proposition \ref{maxab}, thus we may assume that $\h_-=\h_-'$. Let $Z_-^c$ be the centralizer of $\h_-$ in $K^c$. 
It is a compact group,  and it is clear that $\h_+,\h_+'\subset {\rm Lie}(Z_-^c)$  
are Cartan subalgebras. Hence they are conjugate by an element of $Z_-^c$, as desired. 

(iv) Let $\h,\h'$ be maximally compact $\theta$-stable Cartan subalgebras in $\g_\theta$. Then $\h_+,\h_+'$ are Cartan subalgebras of $\mathfrak{k}^c$, so they are conjugate by $K^c$ and we may assume that $\h_+=\h_+'$. Let $Z_+$ be the centralizer of $\h_+$ in $G_\theta$ and $\z_+={\rm Lie}(Z_+)$. This is a $\theta$-stable reductive subalgebra of $\g_\theta$ containing $\h,\h'$ whose center contains $\h_+$. Thus $\h_-,\h_-'\subset {\rm Lie}(Z_+)/\h_+$ are $\theta$-stable split Cartan subalgebras, so they are conjugate by $Z_+^c:=Z_+\cap K^c$ owing to (iii). This implies the statement.  

(v) The proof is by induction in the rank $r$ of $\g_\theta$, with obvious base $r=0$. 
Suppose the statement is known for rank $<r$ and let us prove it for rank $r$. 
Let $\h\subset \g_\theta$ be a Cartan subalgebra. We have $\h=\h_+\oplus \h_-$ where $\h_+,\h_-$ are 
the subspaces of elements with imaginary and real eigenvalues on the adjoint representation, respectively. The Lie group $H_+=\exp(\h_+)$ is a compact torus, so it is contained in a maximal compact subgroup. Hence by Theorem \ref{conjcom} $H_+$ is conjugate to a subgroup of $K^c$. We may thus assume that $\h_+\subset \mathfrak{k}^c$. 

As in (iv), let $Z_+\subset G_\theta$ be the centralizer of $\h_+$ and $\z_+={\rm Lie}(Z_+)$. It suffices to show that $\h$ is conjugate to a $\theta$-stable Cartan subalgebra under $Z_+$. This is equivalent to saying that 
$\h_-$ is conjugate to a $\theta$-stable Cartan subalgebra of $\z_+/\h_+$
under $Z_+/H_+$. So if $\h_+\ne 0$ then the statement follows by the induction assumption, since the rank of $\z_+/\h_+$ is smaller than $r$. On the other hand, if $\h_+=0$ then $\h$ is split, so $\g_\theta$ is split. In this case, let $\h_0$ be the standard Cartan subalgebra of $\g_\theta$. Fixing systems of simple roots 
$\Pi$ for $\h$ and $\Pi_0$ for $\h_0$, there exists an isomorphism 
$\phi: (\g_\theta,\h,\Pi)\to (\g_\theta,\h_0,\Pi_0)$ which is given by an inner 
automorphism of $\g_\theta$, i.e., an element $g\in G_{{\rm ad},\theta}$, which completes the induction step and the proof. 
\end{proof}  

\subsection{Integral form of the Weyl character formula}

\begin{proposition}\label{wcf} Let $f$ be a conjugation-invariant continuous function 
on a compact connected Lie group $K$ with a maximal torus $T\subset K$
and Haar probability measure $dk$. 
Then 
$$
\int_K f(k)dk=\frac{1}{|W|}\int_T f(t)|\Delta(t)|^2dt, 
$$
where $\Delta(t)$ is the Weyl denominator,\footnote{Note that the function $\rho(t)$ may be multivalued, but its branches differ from each other by a root of unity, so 
the function $|\Delta(t)|$ is well defined. Namely, $|\Delta(t)|=|\Delta_0(t)|$ where 
$\Delta_0(t)=\prod_{\alpha\in R^+}(\alpha(t)-1)$.} 
$$
\Delta(t)=\rho(t)^{-1}\prod_{\alpha\in R^+}(\alpha(t)-1).
$$
\end{proposition} 

\begin{proof} Since characters of irreducible representations span a dense subspace in the space of conjugation-invariant continuous functions on $K$, 
it suffices to check this for $f=\chi_\lambda$, the character of the irreducible representation $L_\lambda$. Then the left hand side is $\delta_{0\lambda}$ by orthogonality of characters. 
On the other hand, the Weyl character formula implies that the right hand side also equals 
$\delta_{0\lambda}$. 
\end{proof} 

\begin{example} Let $f$ be a conjugation-invariant continuous function on $U(n)$. Then 
$$
\int_{U(n)}f(k)dk=
$$
$$
\frac{1}{(2\pi)^nn!}\int_{|z_1|=...=|z_n|=1} f({\rm diag}(z_1,...,z_n))\prod_{m<j}|z_m-z_j|^2d\theta_1....d\theta_n
$$
where $z_j=e^{i\theta_j}$. 
\end{example} 

Thus we see that the orthogonality of characters can be written 
as 
$$
\frac{1}{|W|}\int_T \chi_\lambda(t)\overline {\chi_\mu(t)}|\Delta(t)|^2dt=\delta_{\lambda,\mu}.
$$

\begin{exercise} (i) Let $\mathfrak{k}={\rm Lie}K$ with Cartan subalgebra 
$\mathfrak{t}$ and $f$ be a compactly 
supported $K$-invariant continuous function on $\mathfrak{k}$. Show that 
$$
\int_{\mathfrak{k}} f(a)da=\frac{1}{|W|}\int_{\mathfrak{t}} f(u)|\Delta_{\rm rat}(u)|^2du
$$
(for suitable normalization of 
$da$, $du$), where $\Delta_{\rm rat}(u)=\prod_{\alpha\in R_+}\alpha(u)$ is the rational version 
of the Weyl denominator. 

{\bf Hint.} In Proposition \ref{wcf}, make a change of variable 
$k=\exp(\varepsilon a)$, $t=e^{\varepsilon u}$ for small $\varepsilon>0$ and then send $\varepsilon$ to zero. 

(ii) Write explicitly the identity you get if you set $f(a):=e^{B_{\mathfrak k}(a,a)}$ and compute the mutual normalization of $da,du$ in (i). 

{\bf Hint.} For the right hand side, use 
that $\Delta_{\rm rat}(u)=\lim_{\varepsilon\to 0}\varepsilon^{-|R_+|}\Delta(\varepsilon u)$ and 
the Weyl denominator formula. Then compute the Gaussian integral on both sides, use the Weyl 
denominator formula again, and take the limit $\varepsilon\to 0$. 
\end{exercise} 

\section{\bf Topology of Lie groups and homogeneous spaces, I} 

\subsection{The Chevalley-Eilenberg complex of a compact connected Lie group} 

We would now like to study topology of connected Lie groups. The Cartan decomposition implies that 
any real semisimple Lie group $G_\theta$ is diffeomorphic to the product of its 
maximal compact subgroup $K^c$ and a Euclidean space. This combined with weak Levi decomposition (Theorem \ref{wld})
implies that topology of connected Lie groups essentially reduces to topology of compact ones, as any simply-connected solvable Lie group has a filtration by normal subgroups with successive quotients being the 
1-dimensional group $\Bbb R$, hence is diffeomorphic to $\Bbb R^n$ (cf. Theorem \ref{solvthir}, Corollary \ref{homtype} below).  

So let us study cohomology of compact connected Lie groups. 

We first recall some generalities on cohomology of manifolds. As we mentioned before, the cohomology 
of an $n$-dimensional manifold $M$ can be computed by the {\bf de Rham complex}
$$
0\to \Omega^0(M)\to \Omega^1(M)\to...\to \Omega^n(M)\to 0,
$$
where $\Omega^i(M)$ is the space of smooth (complex-valued) differential $i$-forms on $M$. 
The maps in this complex are given by the differential $d: \Omega^i(M)\to \Omega^{i+1}(M)$, which satisfies the equation $d^2=0$. Namely, we define the $i$-th {\bf de Rham cohomology} 
of $M$ as the quotient 
$$
H^i(M,\Bbb C):=\Omega^i_{\rm closed}(M)/\Omega^i_{\rm exact}(M)
$$
where $\Omega^i_{\rm closed}(M)\subset \Omega^i(M)$ 
is the space of {\bf closed forms} (such that $d\omega=0$) and 
$\Omega^i_{\rm exact}(M)\subset \Omega^i(M)$ is the space of {\bf exact forms} 
(such that $\omega=d\eta$ for some $\eta\in \Omega^{i-1}(M)$). 

If $M$ is compact then the spaces $H^i(M,\Bbb C)$ are known to be finite dimensional, 
so we can define the {\bf Betti numbers} of $M$, $b_i(M):=\dim H^i(M,\Bbb C)$. 
Note that $b_0(M)$ is the number of connected components of $M$, 
so if $M$ is connected then $b_0(M)=1$. 

The wedge product of differential forms descends to the cohomology, which makes 
$H^{\bullet}(M,\Bbb C):=\oplus_{i=0}^n H^i(M,\Bbb C)$ into a graded algebra. This algebra is associative and {\bf graded-commutative}: $ab=(-1)^{\deg(a)\deg(b)}ba$ (since the wedge product of differential forms 
has these properties). Moreover, if $f: M\to N$ is a differentiable map of manifolds then 
we have the pullback map $f^*:\Omega^i(N)\to \Omega^i(M)$ which commutes with $d$ 
and hence descends to the cohomology. Also $f^*$ preserves the wedge product, hence defines 
a graded algebra homomorphism $f^*: H^{\bullet}(N,\Bbb C)\to H^{\bullet}(M,\Bbb C)$. 

\begin{exercise} Let $f: [0,1]\times M\to N$ be a differentiable map and 
$f_t: M\to N$ be given by $f_t(x)=f(t,x)$. Then $f_0^*=f_1^*$ on $H^{\bullet}(N,\Bbb C)$. 
In other words, $f^*$ is invariant under (smooth) homotopies of $f$. 
\end{exercise} 

Recall that for a vector field $v$ on $M$, the {\bf Lie derivative}
$$
L_v: \Omega^\bullet(M)\to \Omega^\bullet(M)
$$ 
is the unique derivation of the algebra of differential forms 
which commutes with the de Rham differential and equals the usual derivative 
of a function along $v$ on $\Omega^0(M)$. 

\begin{lemma} (Cartan's magic formula) Let $v$ 
be a vector field on $M$, $L_v: \Omega^i(M)\to \Omega^i(M)$ the Lie derivative 
and $\iota_v: \Omega^i(M)\to \Omega^{i-1}(M)$ the contraction operator. Then 
$$
L_v=\iota_v d+d\iota_v.
$$
\end{lemma} 

\begin{proof} It suffices to check this identity on local charts. 
It is easy to see that both sides are derivations, so it suffices to check 
the equation on functions ($0$-forms) and on $1$-forms of the form 
$df$ where $f$ is a function. For functions we have $L_vf=\iota_vdf$, which is essentially the definition of $L_v$, while for $\omega=df$ we have 
$$
L_v(df)=d(L_vf)=d\iota_v(df)=(\iota_v d+d\iota_v)(df),
$$
since $d^2=0$. 
\end{proof} 

\begin{corollary} $L_v$ maps closed forms to exact forms, hence acts trivially in cohomology.  
\end{corollary} 

\begin{corollary}\label{trivac} If a connected Lie group $G$ acts on a manifold $M$ 
then $G$ acts trivially on $H^{\bullet}(M,\Bbb C)$. 
\end{corollary} 

Suppose now that a compact connected Lie group $G$ acts on a manifold $M$. 
Then we have the averaging operator $P: \Omega^{\bullet}(M)\to \Omega^{\bullet}(M)$ over $G$
which commutes with $d$ and satisfies the equation $P^2=P$, so we have a decomposition of complexes
$$
\Omega^{\bullet}(M)=\Omega^{\bullet}(M)^G\oplus \Omega^{\bullet}(M)_0
$$
where the first summand is the image of $P$ and the second one is the kernel of $P$.  

\begin{theorem}\label{exactness}  The complex $\Omega^{\bullet}(M)_0$ is exact. Thus the 
cohomology $H^{\bullet}(M,\Bbb C)$ is computed by the complex 
of invariant differential forms $\Omega^{\bullet}(M)^G$.
\end{theorem} 

\begin{proof} If $\omega\in \Omega^i(M)_0$ is closed then by Corollary \ref{trivac} the cohomology class $[\omega]$ 
of $\omega$ coincides with the cohomology class of 
$[g\omega]$ for all $g\in G$. Thus 
$$
[\omega]=\int_G [g\omega]dg=\big[\int_G g\omega dg\big]=0.
$$
It follows that $\omega=d\eta$ for some $\eta\in \Omega^i(M)$. Then $\omega=(1-P)\omega=d(1-P)\eta$, and $(1-P)\eta\in \Omega^i(M)_0$. So the complex $\Omega^{\bullet}(M)_0$ is exact, which implies the statement. 
\end{proof} 

\begin{corollary} If $G$ is a compact connected Lie group then 
$H^{\bullet}(G,\Bbb C)$ is computed by the complex $\Omega^{\bullet}(G)^G$ 
of left-invariant differential forms on $G$. 
\end{corollary} 

The complex $\Omega^{\bullet}(G)^G$ is called the {\bf Chevalley-Eilenberg complex} of $G$.

\subsection{Cohomology of Lie algebras} 
It turns out that the Chevalley-Eilenberg complex of $G$ can be described purely algebraically in terms of the Lie algebra $\g={\rm Lie}(G)_{\Bbb C}$. To this end, we will need another lemma from basic differential geometry. 

\begin{lemma} (Cartan differentiation formula) Let $\omega\in \Omega^m(M)$ 
and $v_0,...,v_m$ be vector fields on $M$. Then 
$$
d\omega(v_0,...,v_m)=\sum_i (-1)^i L_{v_i}(\omega(v_0,...,\widehat v_i,...,v_m))+
$$
$$
\sum_{i<j}(-1)^{i+j}\omega([v_i,v_j],v_0,...,\widehat v_i,...,\widehat{v_j},...,v_m)
$$
(where the hats indicate the omitted terms). 
\end{lemma} 

\begin{proof} It is easy to show that the right hand side is linear over functions 
on $M$ with respect to each $v_i$ (the first derivatives of the function cancel out). 
Therefore, it suffices to assume that $v_i=\frac{\partial}{\partial x_{k_i}}$ (in local coordinates), and 
$\omega=fdx_{j_1}\wedge...\wedge dx_{j_m}$. Then the second summand on the RHS vanishes and the verification is straightforward. 
\end{proof} 

\begin{corollary} Let $G$ be a Lie group and $\omega\in \Omega^m(G)^G$ 
be a left-invariant differential form. Then for any left-invariant vector fields
$v_0,...,v_m$ we have 
\begin{equation}\label{carfor} 
d\omega(v_0,...,v_m)=
\sum_{i<j}(-1)^{i+j}\omega([v_i,v_j],v_0,...,\widehat v_i,...,\widehat{v_j},...,v_m).
\end{equation}
\end{corollary} 

\begin{proof} This follows since the functions $\omega(v_0,...,\widehat v_i,...,v_m)$ are constant. 
\end{proof} 

Now observe that $\Omega^m(G)^G=\wedge^m\g^*$. Thus we get 

\begin{corollary}\label{ce} For any Lie group $G$ the complex $\Omega^{\bullet}(G)^G$
coincides with the complex 
$$
0\to \Bbb C\to \g^*\to (\wedge^2\g)^*\to...(\wedge^m\g)^*\to...
$$
with differential defined by \eqref{carfor}, where $\g={\rm Lie}(G)_{\Bbb C}$. \end{corollary} 

This purely algebraic complex can be defined for any Lie algebra $\g$ 
over any field (the equality $d^2=0$ follows from the Jacobi identity).\footnote{Note that if $\g$ is finite dimensional then $\wedge^i\g^*=(\wedge^i\g)^*$.} It is called the {\bf standard complex} or the {\bf Chevalley-Eilenberg complex} of $\g$, denoted $CE^{\bullet}(\g)$, and its cohomology is called the {\bf Lie algebra cohomology} of $\g$, denoted $H^{\bullet}(\g)$.\footnote{Note that $H^1(\g)$ already appeared earlier in Section 18.}

Also note that the complex $CE^{\bullet}(\g)$ has wedge product multiplication, 
which descends to the cohomology. Thus $H^{\bullet}(\g)$ is a graded-commutative associative algebra. 
Furthermore, if $\g={\rm Lie}(G)_{\Bbb C}$ for a compact connected Lie group $G$ 
then $H^{\bullet}(\g)\cong H^{\bullet}(G,\Bbb C)$ as a graded algebra. However, this may fail even at the level of vector spaces (i.e., Betti numbers) if $G$ is not compact.  

\begin{example} Let $\g$ be abelian, $\dim\g<\infty$. Then $CE^{\bullet}(\g)=\wedge^{\bullet} \g^*$, with zero differential, so 
$H^{\bullet}(\g)=\wedge^{\bullet} \g^*$. So if $G=(S^1)^n$ is a torus then we get 
$H^{\bullet}(G,\Bbb C)=\wedge^{\bullet} \g^*=\wedge^{\bullet}(\xi_1,...,\xi_n)$ where $\xi_i$ have degree $1$. 
In particular, $H^{\bullet}(S^1)=\wedge^{\bullet}(\xi)$. However, for the universal cover $\Bbb R$ of $S^1$ this is clearly false. 
\end{example} 

\begin{remark}\label{warn} Corollary \ref{ce} implies that for compact Lie groups $K_1,K_2$
the map $\Omega^{\bullet}(K_1)\otimes \Omega^{\bullet}(K_2)\to \Omega^{\bullet}(K_1\times K_2)$ 
(i.e., in components, $\Omega^i(K_1)\otimes \Omega^j(K_2)\to \Omega^{i+j}(K_1\times K_2)$)
defines an isomorphism of cohomology rings $H^{\bullet}(K_1,\Bbb C)\otimes H^{\bullet}(K_2,\Bbb C)\to H^{\bullet}(K_1\times K_2,\Bbb C)$.
This is a special case of the {\bf K\"unneth theorem}, which actually holds for any manifolds (and more generally for sufficiently nice topological spaces) with finite dimensional cohomology in each degree, which need not have any group structure. We warn the reader, however, that the {\bf tensor product of algebras here is in the graded sense}, i.e. 
$$
(a\otimes b)(a'\otimes b')=(-1)^{\deg(b)\deg(a')}(aa'\otimes bb').
$$
\end{remark} 

\begin{theorem} If $G$ is a connected compact Lie group with 
${\rm Lie}(G)_\Bbb C=\g$ then $H^{\bullet}(G,\Bbb C)\cong (\wedge^{\bullet} \g^*)^\g$ as a ring. 
\end{theorem}

\begin{proof} We have an action of $G\times G$ on $G$, so 
the cohomology of $G$ is computed by the complex of invariants 
$\Omega^{\bullet}(G)^{G\times G}=(\wedge^{\bullet} \g^*)^G$. So our job is to show that the differential in this complex is actually zero. But this follows immediately from the definition of the differential in $\wedge^{\bullet} \g^*$. 
\end{proof} 

We also have

\begin{proposition} If $G$ is a connected Lie group, $\Gamma\subset G$ 
a finite subgroup, and $\pi: G\to G/\Gamma$ is the canonical map then $\pi^*$ defines an isomorphism $H^{\bullet}(G/\Gamma,\Bbb C)\to H^{\bullet}(G,\Bbb C)$. 
\end{proposition} 

\begin{proof} The map $\pi^*$ is an isomorphism $H^{\bullet}(G/\Gamma,\Bbb C)\to H^{\bullet}(G,\Bbb C)^\Gamma$, but $\Gamma$, being a subgroup of $G$, acts trivially on $H^{\bullet}(G,\Bbb C)$. 
\end{proof} 

Thus it suffices to determine the cohomology of simple, simply connected compact Lie groups. 

\section{\bf Topology of Lie groups and homogeneous spaces, II}\label{top2}

\subsection{The coproduct on the cohomology ring} 
To understand the algebra $R:=H^{\bullet}(G)=H^{\bullet}(G,\Bbb C)$ better, note that the multiplication 
map $G\times G\to G$ induces the graded algebra homomorphism 
$\Delta: H^{\bullet}(G)\to H^{\bullet}(G\times G)=H^{\bullet}(G)\otimes H^{\bullet}(G)$, which is coassociative: 
$$
(\Delta\otimes \id)\circ \Delta=(\id\otimes \Delta)\circ \Delta.
$$ 
(Note that the warning in Remark \ref{warn} about tensor product in the graded sense still applies here!)
Such a map $\Delta$ is called a {\bf coproduct} since 
it defines an algebra structure on the dual space $R^*$ (see Subsection \ref{copro}). We also have the augmentation map 
$\varepsilon: R\to \Bbb C$ such that 
$$
(\varepsilon\otimes 1)(\Delta(x))=(1\otimes \varepsilon)(\Delta(x))=x
$$
for all 
$x\in R$. Such a structure is called a {\bf graded bialgebra}.\footnote{Moreover, we have an algebra homomorphism $S: R\to R$ induced by the inversion map $G\to G$ called the {\bf antipode}. This makes $R$ into what is called a {\bf graded Hopf algebra}.} 

\begin{exercise} (Hopf theorem) Let $R$ be a finite dimensional graded-commutative bialgebra over a field $\bold k$ of characteristic zero, and $R[0]=\bold k$ (where the grading is by nonnegative integers). Show that $R$ 
is a {\bf free} graded commutative algebra on some homogeneous generators of odd degrees, i.e., $R=\wedge_{\bold k}^{\bullet}(\xi_1,...,\xi_r)$ 
with $\deg \xi_i=2m_i+1$ for some nonnegative integers $m_i$. 
Thus $\dim R=2^r$. 

{\bf Hint.} Recall from Subsection \ref{primi} that an element $x\in R$ is {\bf primitive} if $\Delta(x)=x\otimes 1+1\otimes x$. Show that any homogeneous primitive $x$ has odd degree (use that $\dim R<\infty$), thus $x^2=0$, and that 
$R$ is generated by homogeneous primitive elements. Then show that linearly independent primitive elements in $R$ cannot satisfy any nontrivial relation (take a relation of lowest degree, compute its coproduct and find a relation of even lower degree, getting a contradiction). 

For more hints see \cite{C}, Subsection 2.4. 
\end{exercise} 

Let us now determine the number $r$. We have $2^r=\dim (\wedge^{\bullet}\g^*)^\g$. 
But this dimension can be computed using the Weyl character formula. Namely, the character 
of $\wedge^{\bullet} \g^*$ is 
$$
\chi_{\wedge^{\bullet}\g^*}(t)=2^{{\rm rank}\g}\prod_{\alpha>0}(1+\alpha(t))(1+\alpha(t)^{-1}), 
$$
where $T\subset G$ is a maximal torus and $t\in T$. 
So 
$$
\dim (\wedge^{\bullet}\g^*)^\g=\frac{2^{{\rm rank}\g}}{|W|}\int_T\prod_{\alpha>0}(\alpha(t^2)-1)(1-\alpha(t^{-2}))dt
=2^{{\rm rank}\g}.
$$
So $r={\rm rank}\g$. 

Thus we have  
$$
H^{\bullet}(G)=H^{\bullet}(\g)=(\wedge^{\bullet} \g^*)^\g=\wedge^{\bullet} (\xi^{(1)},...,\xi^{(r)}),
$$
where $r={\rm rank}\g$ and $\deg(\xi^{(i)})=2m_i+1$. Moreover, it suffices to consider 
the case when $\g$ is simple. 
What are the numbers $m_i$ in this case? 

Let us order $m_i$ as follows: $m_1\le m_2\le...\le m_r$. 
We know that $r+2\sum m_i=\dim \g$, so $\sum_i m_i=|R_+|$. 
Also it is not hard to see that $m_1=1$, $m_2>1$:

\begin{exercise} Show that for a simple Lie algebra $\g$ we have 
$(\wedge^3\g^*)^\g=\Bbb C$, spanned by the triple product $([x,y],z)$. 

{\bf Hint.} Let $\omega\in (\wedge^3\g^*)^\g$. 

1. Show that
$$
\omega(e_i,[f_i,h_i],h)+\omega(e_i,h_i,[f_i,h])=0
$$
for $h\in \h$ and deduce that 
$$
\omega(e_i,f_i,h)=\tfrac{1}{2}\alpha_i(h)\omega(e_i,f_i,h_i). 
$$

2. Take $y,z\in \h$ and show that
$$
\omega(h_i,y,z)+\omega(f_i,[e_i,y],z)+\omega(f_i,y,[e_i,z])=0.
$$
Deduce that $\omega(x,y,z)=0$ for $x,y,z\in \h$. Conclude that 
$\omega$ is completely determined by $\omega(e_\alpha,e_{-\alpha},h)$ for all roots $\alpha$ and $h\in \h$. Use the Weyl group to reduce to $\omega(e_i,f_i,h)$ and then to $\omega(e_i,f_i,h_i)$. 

3. Finally, use that 
$$
\omega([e_i,e_j],f_i,f_j)=\omega(e_j,f_j,h_i)=\omega(e_i,f_i,h_j)
$$
to show that all possible $\omega$ are proportional.
\end{exercise} 

In particular, we see that for a simple compact connected 
Lie group $G$, one has $H^3(G,\Bbb C)\cong \Bbb C$. Thus, the sphere $S^n$ admits a Lie group structure if and only if $n=0,1,3$. 

\begin{example} We get $m_2=2$ for $A_2$, $m_2=3$ for $B_2=C_2$, $m_2=5$ for $G_2$. Thus the Poincar\'e polynomials 
$P_\g(q):=\sum_{n\ge 0}\dim H^n(G,\Bbb C)q^n$ 
for compact simple Lie groups of rank $\le 2$ are: 
$$
P_{A_1}(q)=1+q^3,\ P_{A_2}(q)=(1+q^3)(1+q^5),\ 
$$
$$
P_{B_2}(q)=(1+q^3)(1+q^7),\
P_{G_2}(q)=(1+q^3)(1+q^{11}).
$$
\end{example}

\subsection{The cohomology ring of a simple compact connected Lie group}

In fact, we have the following classical theorem, which we will not prove in general, but will prove below for type $A$ and also in exercises for classical groups and $G_2$. 

\begin{theorem}\label{coho} Let $G$ be a simple compact Lie group with complexified Lie algebra $\g$. Then the numbers $m_i$ are the exponents of $\g$ defined in Subsection \ref{expone}. In other words, the degrees $2m_i+1$ of generators of the cohomology ring are the dimensions 
of simple modules occurring in the decomposition of $\g$ over its principal $\mathfrak{sl}_2$-subalgebra. Thus the cohomology ring 
$H^\bullet(G,\Bbb C)$ is the exterior algebra 
$\wedge^\bullet(\xi_{2m_1+1},...,\xi_{2m_r+1})$, where 
$\xi_j$ has degree $j$. 
\end{theorem} 

A modern general proof of this theorem can be found in \cite{R}.  

\begin{remark} The Poincar\'e polynomial $P_\g(q)$ of $(\wedge^{\bullet} \g^*)^\g$ 
is given by the formula 
$$
P_\g(q)=\frac{(1+q)^r}{|W|}\int_T \prod_{\alpha\in R} (1+q\alpha(t))\prod_{\alpha>0}|\alpha(t)^{\frac{1}{2}}-\alpha(t)^{-\frac{1}{2}}|^2dt. 
$$ 
So Theorem \ref{coho} is equivalent to the statement that this integral equals 
$\prod_i (1+q^{2m_i+1})$. 
\end{remark} 

We will prove Theorem \ref{coho} in the case of type $A$. 

\begin{corollary} For $\g=\mathfrak{sl}_n$ we have 
$m_i=i$. Equivalently, the same is true for $\g=\mathfrak{gl}_n$ if 
we add $m_0=0$. 
\end{corollary}

\begin{proof} Let $\g=\mathfrak{gl}_n$, $V=\Bbb C^n$. We need to compute the Poincar\'e polynomial of 
$\wedge^{\bullet}(V\otimes V^*)^\g$. The skew Howe duality (Proposition \ref{skewhowe}) implies that this Poincar\'e polynomial is 
$$
P(q)=\sum_{\lambda=\lambda^t}q^{|\lambda|}, 
$$
where the summation is over $\lambda$ with $\le n$ parts. 
But there are exactly $2^n$ such symmetric partitions $\lambda$: 
they consist of a sequence of hooks $(k,1^{k-1})$ 
with decreasing values of $k$, with each of them either present or not. The degree of such a hook is $2k-1$, which implies that 
\begin{equation}\label{ppgl}
P_{\mathfrak{gl}_n}(q)=(1+q)(1+q^3)(1+q^5)...(1+q^{2n-1}). 
\end{equation}  
\end{proof} 

Thus we get that the cohomology $H^{\bullet}(U(n),\Bbb C)=H^{\bullet}(GL_n(\Bbb C),\Bbb C)$ is $\wedge^{\bullet}(\xi_1,\xi_3,...,\xi_{2n-1})$ (where subscripts are degrees) with Poincar\'e polynomial \eqref{ppgl}, and 
$H^{\bullet}(SU(n),\Bbb C)=H^{\bullet}(SL_n(\Bbb C),\Bbb C)=
\wedge^{\bullet}(\xi_3,...,\xi_{2n-1})$ with Poincar\'e polynomial 
$(1+q^3)(1+q^5)...(1+q^{2n-1})$. 

In the next exercise and the following subsections 
we will use the notions of a {\bf cell complex} and 
its {\bf cellular homology and cohomology} with coefficients in any commutative ring, 
and the fact that if a manifold is equipped with a 
cell decomposition (i.e., represented as a disjoint union of cells) then its cellular cohomology with $\Bbb C$-coefficients (=dual to the cellular homology) is canonically isomorphic to the de Rham cohomology via the integration pairing (the {\bf de Rham theorem}). More details can be found, for instance, in \cite{H}.  

\begin{exercise}\label{classcoh} (i) Give another proof of Theorem \ref{coho} for type $A_{n-1}$ as follows. Use that $SU(n)/SU(n-1)=S^{2n-1}$ to construct a cellular decomposition of 
$SU(n)$ into $2^{n-1}$ cells (use the decomposition of $S^{2n-1}$ into a point and its complement). Then show that the differential in the corresponding cochain complex with $\Bbb C$-coefficients is zero (compare its dimension to the dimension of the cohomology). Derive Theorem \ref{coho} for $SU(n)$ by induction in $n$.

(ii) Use the same idea and the fact that $U(n,\Bbb H)/U(n-1,\Bbb H)=S^{4n-1}$ 
to establish Theorem \ref{coho} in type $C_n$. Conclude that the cohomology ring of $U(n,\Bbb H)$ (and ${\rm Sp}_{2n}(\Bbb C)$) is $\wedge(\xi_3,\xi_7,...,\xi_{4n-1})$ with Poincar\'e polynomial
 $(1+q^3)(1+q^7)...(1+q^{4n-1})$.

(iii) Show that these Poincar\'e polynomials are valid for cohomology of the same Lie groups with any coefficients.\footnote{A similar idea can be used 
to find the cohomology of ${\rm Spin}(n)$ (see Exercise \ref{stief} below) but it is a bit more complicated since there is no cell decomposition with zero boundary map, and thus any cell decomposition has strictly more than $2^r$ cells for sufficiently large $n$
(as there is 2-torsion in the integral cohomology).}
\end{exercise} 

\subsection{Cohomology of homogeneous spaces} 
Let $G$ be a connected compact Lie group, $\g={\rm Lie}(G)_{\Bbb C}$, $K\subset G$ a closed subgroup, $\mathfrak{k}={\rm Lie}(K)_\Bbb C$, and consider the homogeneous space $G/K$. How to compute the cohomology $H^{\bullet}(G/K,\Bbb C)$?

Since the group $G$ acts on $G/K$, this cohomology is computed by 
the complex $\Omega^{\bullet}(G/K)^G=(\wedge^{\bullet} (\g/\mathfrak{k})^*)^K$. 
Let us denote this complex by $CE^{\bullet}(\g,K)$. It is called the {\bf relative 
Chevalley-Eilenberg complex}. 

For example, if $K=\Gamma$ is finite, this is just the $\Gamma$-invariant part of the usual Chevalley-Eilenberg complex. But $\Gamma$ acts trivially on the cohomology, so we get $H^{\bullet}(G/\Gamma)=H^{\bullet}(G)$ (as already noted above). 

But what happens if $\dim K>0$? Can we describe the differential in this complex algebraically as we did for $K=1$? 

This question is answered by the following proposition. Let $\mathfrak{k}\subset \g$ be a pair of Lie algebras (not necessarily finite dimensional, over any field). Denote by $CE^i(\g,\mathfrak{k})$ the spaces $(\wedge^i(\g/\mathfrak{k})^*)^\mathfrak{k}$. 

\begin{proposition}\label{subcom}  $CE^{\bullet}(\g,\mathfrak{k})$ is a subcomplex 
of $CE^{\bullet}(\g)$. 
\end{proposition} 

\begin{exercise} Prove Proposition \ref{subcom}. 
\end{exercise} 

\begin{definition} 
The complex $CE^{\bullet}(\g,\mathfrak{k})$ is called the {\bf relative Chevalley-Eilenberg complex}, and its cohomology  is called the {\bf relative Lie algebra cohomology}, denoted by $H^{\bullet}(\g,\mathfrak{k})$. 
\end{definition} 

Now note that, going back to the setting of compact Lie groups,
we have $CE^\bullet(\g,K)=CE^\bullet(\g,\mathfrak{k})^{K/K^\circ}$, so we obtain 

\begin{corollary} $H^\bullet(G/K,\Bbb C)\cong H^\bullet(\g,\mathfrak{k})^{K/K^\circ}$ as algebras. 
\end{corollary} 

Thus, the computation of the cohomology of $G/K$ reduces to the computation of the relative Lie algebra cohomology, which is again a purely algebraic problem. 

\begin{corollary} Suppose $s \in K$ is an element that acts by $-1$ on $\g/\mathfrak{k}$. 
Then $(\wedge^i(\g/\mathfrak{k})^*)^{K}=0$ for odd $i$. Hence the differential 
in $CE^\bullet(\g,K)$ vanishes and thus $H^\bullet(G/K,\Bbb C)\cong (\wedge^\bullet (\g/\mathfrak{k})^*)^{K}$, with cohomology present only in even degrees. 
\end{corollary} 

\begin{exercise}\label{stief} 
The real {\bf Stiefel manifold} ${\rm St}_{n,k}(\Bbb R)$, $k<n$,  
is the manifold of all orthonormal $k$-tuples of vectors in $\Bbb R^n$. 
For example, ${\rm St}_{n,1}(\Bbb R)=S^{n-1}$ and ${\rm St}_{n,n-1}(\Bbb R)=SO(n)$. 

(i) Show that ${\rm St}_{n,k}(\Bbb R)=SO(n)/SO(n-k)$ and hence 
$\dim {\rm St}_{n,k}(\Bbb R)=k(n-k)+\frac{k(k-1)}{2}$. 

(ii) Show that for $n\ge 3$, the manifold ${\rm St}_{n,2}(\Bbb R)$ is a fiber bundle over $S^{n-1}$ with fiber $S^{n-2}$. Conclude that ${\rm St}_{n,2}(\Bbb R)$ has a cell decomposition with four cells of dimensions $0,n-2,n-1,2n-3$. Show that 
the boundary of the $n-1$-dimensional cell is zero if $n$ is even and 
twice the $n-2$-dimensional cell if $n$ is odd. Compute the cohomology groups
of ${\rm St}_{n,2}(\Bbb R)$ with any coefficient ring. In particular, show that if $n$ is odd then the cohomology groups with coefficients in any field of characteristic $\ne 2$ are the same as for the sphere $S^{2n-3}$. 

(iii) Use the relative Chevalley-Eilenberg complex to compute 
the cohomology $H^*({\rm St}_{n,2}(\Bbb R),\Bbb C)$ in another way. Compare to (ii).  
\end{exercise} 

\begin{exercise} (i) Prove Theorem \ref{coho} for type $B_n$ 
using the method of Exercise \ref{classcoh}. Namely, use that 
$SO(2n+1)/SO(2n-1)={\rm St}_{2n+1,2}(\Bbb R)$ and Exercise \ref{stief}(ii) or (iii). 
Conclude that the cohomology ring of $SO(2n+1)$ (and $SO_{2n+1}(\Bbb C)$) over $\Bbb C$ is $\wedge^\bullet(\xi_3,\xi_7,...,\xi_{4n-1})$ with Poincar\'e polynomial
is $(1+q^3)(1+q^7)...(1+q^{4n-1})$.

(ii) Use the conclusion of (i) for $B_{n-1}$ and that $SO(2n)/SO(2n-1)=S^{2n-1}$ to prove Theorem \ref{coho} for type $D_{n}$ (again
using the method of Exercise \ref{classcoh}). 
Conclude that the cohomology ring of $SO(2n)$ (and $SO_{2n}(\Bbb C)$) over $\Bbb C$ is $\wedge^\bullet(\xi_3,\xi_7,...,\xi_{4n-5},\eta_{2n-1})$ with Poincar\'e polynomial
having the form $(1+q^3)(1+q^7)...(1+q^{4n-5})\cdot (1+q^{2n-1})$.

(iii) Show that these Poincar\'e polynomials are valid for cohomology of the same Lie groups with coefficients in any ring containing $\frac{1}{2}$. 
\end{exercise} 

\section{\bf Topology of Lie groups and homogeneous spaces, III}

\subsection{Grassmannians} Let $G=U(m+n), K=U(n)\times U(m)$, so that 
$G/K$ is the {\bf Grassmannian} ${\rm G}_{m+n,n}(\Bbb C)\cong {\rm G}_{m+n,m}(\Bbb C)$ (the manifold of $m$-dimensional or $n$-dimensional subspaces of $\Bbb C^{m+n}$). 
The element $s=I_n\oplus (-I_m)$ acts by $-1$ on $\g/\mathfrak{k}=V\otimes W^*\oplus W\otimes V^*$, where $V,W$ are the tautological representations 
of $U(n)$ and $U(m)$. So we get that the Grassmannian has cohomology only 
in even degrees, and 
$$
H^{2i}({\rm G}_{m+n,m}(\Bbb C))=\wedge^{2i}(V\otimes W^*\oplus W\otimes V^*)^{U(n)\times U(m)}. 
$$
We can therefore use the skew Howe duality (Proposition \ref{skewhowe}) to see that 
$$
\dim H^{2i}({\rm G}_{m+n,m}(\Bbb C))=N_i(n,m),
$$
where $N_i(n,m)$ is the number of partitions $\lambda=(\lambda_1,...,\lambda_k)$ whose Young diagram has $i$ boxes and fit into the rectangle $m\times n$ (i.e., such that $k\le m,\lambda_1\le n$).

To compute $N_i(m,n)$, consider the generating function 
$$
f_{n,m}(q)=\sum_i N_i(n,m)q^{i}. 
$$
Then, denoting by $p_i$ the jumps $\lambda_i-\lambda_{i+1}$ of $\lambda$ (with $p_0=n-\lambda_1$), we have 
$$
\sum_{n\ge 0}f_{n,m}(q)z^n=
$$
$$
\sum_{p_0,p_1,...,p_m\ge 0} z^{p_0+p_1+...+p_m}q^{p_1+2p_2+...+mp_m}=\prod_{j=0}^m \frac{1}{1-q^jz}.
$$
So the Betti numbers of Grassmannians are the coefficients of this series. 
For example, if $m=1$ we get 
$$
\sum_{n\ge 0}f_{n,m}(q)z^n=\frac{1}{(1-z)(1-qz)}=\sum_n (1+q+...+q^n)z^n.
$$
So we recover the Poincar\'e polynomial $1+q+...+q^n$ of the complex projective space $\Bbb C\Bbb P^n$. More precisely, this is the Poincar\'e polynomial 
evaluated at $q^{\frac{1}{2}}$, which is actually a polynomial in $q$ since we have nontrivial 
cohomology only in even degrees. 

The polynomials $f_{n,m}(q)$ are called the {\bf Gaussian binomial coefficients} 
and they can be computed explicitly. Namely, we have 

\begin{proposition} \label{gausbin}
$$
f_{m,n}(q)=\binom{m+n}{n}_q=\binom{m+n}{m}_q=\frac{[m+n]_q!}{[m]_q![n]_q!},
$$
where $[m]_q:=\frac{q^m-1}{q-1}$ and $[m]_q!:=[1]_q...[m]_q$. 
\end{proposition} 

\begin{proof} This follows immediately from the {\bf $q$-binomial theorem}\footnote{Note that setting $q=1$ in the $q$-binomial theorem, we get the familiar formula from calculus, often called the binomial theorem: 
$$
\sum_{n\ge 0}\binom{m+n}{m}z^n= \frac{1}{(1-z)^{m+1}}. 
$$
}
\begin{equation}\label{qbt} 
\sum_{n\ge 0}\binom{m+n}{n}_qz^n=\prod_{j=0}^m \frac{1}{1-q^jz}. 
\end{equation} 
\end{proof} 

\begin{exercise} Prove \eqref{qbt}. 

{\bf Hint.} Let $F(z)$ be the RHS of this identity. Write a $q$-difference equation expressing  $F(qz)$ in terms of $F(z)$. Show that this equation has a unique solution 
such that $F(0)=1$. Then prove that the LHS satisfies the same equation. 
\end{exercise} 

\begin{exercise} Compute the Betti numbers of ${\rm G}_{N,2}(\Bbb C)$. 
\end{exercise}

\subsection{Schubert cells} 
There is actually a more geometric way to obtain the same result. This way is based on decomposing the Grassmannians into {\bf Schubert cells}. Namely, let $F_i\subset \Bbb C^{m+n}$ be spanned by the first $i$ basis vectors $e_1,...,e_i$; thus 
$$
0=F_0\subset F_1\subset...\subset F_{m+n}=\Bbb C^{m+n}.
$$  
Given an $m$-dimensional subspace $V\subset \Bbb C^{m+n}$, let $\ell_j$ be the smallest integer for which $\dim (F_{\ell_j}\cap V)=j$. Then 
$$
1\le \ell_1<\ell_2<...<\ell_m\le m+n,
$$ 
which defines a partition with parts 
$$
\lambda_1=\ell_m-m, \lambda_2=\ell_{m-1}-m+1,...,\lambda_m=\ell_1-1
$$ 
fitting in the $m\times n$ box. Let $S_\lambda\subset {\rm G}_{m+n,m}(\Bbb C)$ be the 
set of $V$ giving such numbers $\lambda_i$. 

\begin{exercise} Show that $S_\lambda$ is a locally closed embedded complex submanifold of the Grassmannian isomorphic to the affine space $\Bbb C^{|\lambda|}$ of dimension $|\lambda|=\sum_i \lambda_i$ (i.e., a closed embedded submanifold in an open subset of the Grassmannian).   

{\bf Hint.} Show that for $V\in S_\lambda$, the elements $f_k:=e_{\ell_k}^*|_V$ form a basis of $V^*$. For $\ell_{j}+1\le i\le \ell_{j+1}$ (with $\ell_{m+1}:=m+n$), show that $e_i^*|_V$ is a linear combination of $f_k$, $j+1\le k\le m$, and denote the corresponding coefficients by $a_{ik}(V)$. 
Show that the assignment $V\mapsto (a_{ik}(V))$ is an isomorphism $S_\lambda\cong \Bbb C^{|\lambda|}$. 
\end{exercise} 

\begin{definition} The subset $S_\lambda$ of the Grassmannian is called the {\bf Schubert cell} corresponding to $\lambda$. 
\end{definition} 

So we see that ${\rm G}_{m+n,m}(\Bbb C)$ has a {\bf cell decomposition} into a disjoint union of Schubert cells. 

Now we can rederive the same formula for the Poincar\'e polynomial of the Grassmannian
from the following well-known fact from algebraic topology: 

\begin{proposition}
If $X$ is a connected cell complex which only has even-dimensional cells, with finitely many in each dimension, then  
the cohomology of $X$ vanishes in odd degrees, and 
the groups $H^{2i}(X,\Bbb Z)$ are free abelian groups of ranks $b_{2i}(X)$, where 
the Betti number $b_{2i}(X)$ is just the number of cells in $X$ of dimension $2i$. 
Moreover, $X$ is simply connected. 
\end{proposition} 

Indeed, the boundary map in this cell complex has to be zero, and its fundamental group must be trivial, as it is a quotient of the fundamental group of the 1-skeleton of $X$, which is a single point (why?).  

So we obtain an even stronger statement than before: 

\begin{corollary}
 $H^{2i}({\rm G}_{m+n,n}(\Bbb C),\Bbb Z)$ are free abelian groups of ranks given by coefficients of $\binom{m+n}{m}_q$, and the odd cohomology groups are zero.  
 Moreover, Grassmannians are simply connected. 
\end{corollary} 

In particular, this gives Betti numbers over any field (including positive characteristic), not just $\Bbb C$. 

\subsection{Flag manifolds} \label{flman}
The {\bf flag manifold} $\mathcal F_n(\Bbb C)$ is the space of all {\bf complete flags} 
$0=V_0\subset V_1\subset...\subset V_n=\Bbb C^n$, where 
$\dim V_i=i$. Note that the flag manifold is a homogeneous space: 
${\mathcal F}_n=G/T$, where $G=U(n)$ and $T=U(1)^n$ is a maximal torus in $G$.
It can also be written as $G_{\Bbb C}/B$, where $G_{\Bbb C}=GL_n(\Bbb C)$ 
and $B=B_n$ is the subgroup of upper triangular matrices.  

We have fibrations $\pi: \mathcal F_n(\Bbb C)\to \Bbb C\Bbb P^{n-1}$ 
sending $(V_1,...,V_{n-1})$ to $V_{n-1}$, whose fiber is the space of flags in $V_{n-1}$, i.e., 
$\mathcal F_{n-1}(\Bbb C)$. This shows, by induction, that flag manifolds can be decomposed into even-dimensional cells isomorphic to $\Bbb C^k$.

More precisely, to define actual cells, we need to trivialize 
the fibration $\pi$ over each cell in $\Bbb C\Bbb P^{n-1}$. These cells 
are $C_{in}$, $i=1,...,n$, where $C_{in}$ is the set of hyperplanes 
$E\subset \Bbb C^n$ defined by an equation 
$a_1x_1+...+a_nx_n=0$ where the first nonzero coefficient is $a_i$ (so  $C_{in}\cong \Bbb C^{n-i}$). 
This means that for $(x_1,...,x_n)\in E$, 
the coordinates $x_j,j\ne i$ can be chosen arbitrarily, and then $x_i$ 
is uniquely determined. So we may identify $E$ with $\Bbb C^{n-1}$ 
by sending $(x_1,...,x_n)$ to $(x_1,...,x_{i-1},x_{i+1},...,x_n)$, 
which defines the required trivialization. 
 
Thus we obtain a stratification of $\mathcal{F}_n$ into cells $C_w$ labeled 
by permutations $w\in S_n$, which we'll represent as orderings of $1,2,...,n$. 
Namely, this stratification and labeling are defined by induction in $n$: 
for $w\in S_{n-1}$, $C_w\times C_{in}=C_{w_i'}$, where $w_i'\in S_n$ 
is obtained from $w$ by inserting $n$ in the $i$-th place (namely, 
$w_i'=w\circ (i,i+1,...,n))$. By analogy with the Grassmannian, the cells $C_w$ are  called {\bf Schubert cells}. 
  
It follows that the Betti numbers of $\mathcal{F}_n$ vanish in odd degrees, and in even degrees are given by the generating function
$$
\sum b_{2i}(\mathcal{F}_n)q^i=[n]_q!=(1+q)(1+q+q^2)...(1+q+...+q^{n-1}). 
$$
Moreover, it is easy to see that $\dim_{\Bbb C} C_w=\ell(w)$, so we get the identity
$$
\sum_{w\in S_n}q^{\ell(w)}=[n]_q!
$$
Finally, note that the group $B_n$ of upper triangular matrices preserves each $C_w$. In fact, it is easy to check by induction in $n$ that $C_w$ are simply  $B_n$-orbits on ${\mathcal F}_n$. 

\begin{remark} We have a map $\pi_m: \mathcal{F}_{m+n}(\Bbb C)\to {\rm G}_{m+n,m}(\Bbb C)$ 
sending $(V_1,...,V_{m+n-1})$ to $V_m$. This is a fibration with fiber 
$\mathcal{F}_m(\Bbb C)\times \mathcal F_{n}(\Bbb C)$. This gives another proof 
of the formula for Betti numbers of the Grassmannian (Proposition \ref{gausbin}). 
\end{remark}

We can also define the {\bf partial flag manifold} $\mathcal F_S(\Bbb C)$, where 
$S\subset [1,n-1]$ is a subset, namely the space of {\bf partial flags} $(V_s, s\in S)$, $V_s\subset \Bbb C^n$, $\dim V_s=s$, $V_s\subset V_t$ if $s<t$. 

\begin{exercise} Let $S=\lbrace n_1,n_1+n_2,...,n_1+...+n_{k-1}\rbrace$, and $n_k=n-n_1-...-n_{k-1}$. Show that the even Betti numbers of the partial flag manifold are the coefficients of the polynomial 
$$
P_S(q):=\frac{[n]_q!}{[n_1]_q!...[n_k]_q!}
$$
called the {\bf Gaussian multinomial coefficient} (and the odd Betti numbers vanish). Show that the partial flag manifold is simply connected. 
\end{exercise} 

\section{\bf Levi decomposition} 

\subsection{Cohomology of Lie algebras with coefficients}

The definition of cohomology of Lie algebras may be generalized 
to define the cohomology with coefficients in a module, so that the cohomology considered above is the one for the trivial module.

Let $\g$ be a Lie algebra and $V$ a $\g$-module. The {\bf Chevalley-Eilenberg (or standard) complex of $\g$ with coefficients in $V$} is defined by
$$
CE^\bullet(\g,V):=\Hom(\wedge^\bullet\g,V)
$$
with differential defined by the full Cartan formula (without dropping the first term): 
$$
d\omega(a_0,...,a_m)=\sum_i (-1)^i a_i\omega(a_0,...,\widehat a_i,...,a_m)+
$$
$$
\sum_{i<j}(-1)^{i+j}\omega([a_i,a_j],a_0,...,\widehat a_i,...,\widehat{a_j},...,a_m).
$$
The cohomology of this complex is called the {\bf cohomology of $\g$ with coefficients in $V$} and denoted $H^\bullet(\g,V)$. Note that the previously defined cohomology $H^\bullet(\g)$ is $H^\bullet(\g,\Bbb C)$. 

If $\g$ is the Lie algebra of a Lie group $G$ (or its complexification) and $V$ is finite dimensional, then we simply have $CE^\bullet(\g,V):=(\Omega^\bullet(G)\otimes V)^G$ (and the differential is just the de Rham differential). So in particular by Theorem \ref{exactness} we have (using that if $\g$ is semisimple then the smallest $i>0$ such that $H^i(\g,\Bbb C)\ne 0$ is $3$): 

\begin{proposition} (i) If $G$ is compact and $V$ is a nontrivial irreducible representation then 
$$
H^i(\g,V)=0,\ i>0.
$$ 
In particular, this is so for any non-trivial irreducible finite dimensional representation $V$ of a semisimple Lie algebra $\g$. 

(ii) (Whitehead's theorem) For semisimple $\g$ and any finite dimensional $V$ we have $H^1(\g,V)=H^2(\g,V)=0$.\footnote{Note that $H^1(\g,V)$ appeared earlier in Section 18 and Whitehead's theorem in the case of $H^1$ was proved in Subsection \ref{Whi}.}
\end{proposition} 

However, this cohomology is non-trivial in general if $\g$ is not semisimple or $V$ is infinite dimensional. 

Let us explore the meaning of $H^i(\g,V)$ for small $i$. 

{\bf 1.} We have $H^0(\g,V)=V^\g$, the $\g$-invariants in $V$. 

{\bf 2.}  $H^1(\g,V)$ is the quotient of the space $Z^1(\g,V)$ of 1-cocycles \linebreak
$\omega: \g\to V$, i.e., linear maps satisfying
$$
\omega([x,y])=x\omega(y)-y\omega(x)
$$
by the space of 1-coboundaries $B^1(\g,V)$, of the form $\omega(x)=xv$ for some $v\in V$. 

\begin{proposition} (i) If $V,W$ are representations of $\g$ then 
${\rm Ext}^1(V,W)=H^1(\g,\Hom_{\bf k}(V,W))$. 

(ii) Consider the action of the additive group of $V$ on the Lie algebra 
$\g\ltimes V$ (with trivial commutator on $V$) by 
$$
v\circ(x,w)=(x,w+xv).
$$
Then $H^1(\g,V)$ classifies Lie algebra homomorphisms $\g\to \g\ltimes V$ of the form 
$x\mapsto (x,\omega(x))$ modulo this action. 
\end{proposition} 

\begin{proof} (i) Suppose the space $W\oplus V$ is equipped with the action 
of $\g$ so that $W$ is a submodule and $V$ the quotient. Thus the action of $\g$ 
on $W\oplus V$ is given by 
$$
\rho(x)=\begin{pmatrix} \rho_W(x) & \omega(x) \\ 0 & \rho_V(x)\end{pmatrix},
$$
where $\omega: \g\to \Hom_{\bf k}(V,W)$. So the identity $\rho([x,y])=[\rho(x),\rho(y)]$ 
translates into 
$$
\omega([x,y])=\rho_W(x)\omega(y)-\omega(y)\rho_V(x)-\rho_W(y)\omega(x)+\omega(x)\rho_V(y).
$$
i.e., $\rho\in Z^1(\g,\Hom_{\bf k}(V,W))$. Also it is easy to check that 
for two such representations $\rho_1,\rho_2$ there is an isomorphism
$\rho_1\to \rho_2$ acting trivially on $W$ and $V/W$ if and only if 
the corresponding maps $\omega_1,\omega_2$ differ by a coboundary: 
$\omega_1-\omega_2\in B^1(\g,\Hom_{\bf k}(V,W))$. This implies the statement. 

(ii) We leave this to the reader as an exercise. 
\end{proof} 

{\bf 3.} $Z^1(\g,\g)$ is the Lie algebra of derivations of $\g$, and 
$B^1(\g,\g)$ is the ideal of inner derivations. So $H^1(\g,\g)$ 
is the Lie algebra of {\bf outer derivations}, the quotient 
of all derivations by inner derivations. In  particular, we rederive the 
fact proved earlier that all derivations of a semisimple complex Lie algebra $\g$ are inner ($H^1(\g,\g)=0$).

{\bf 4.} Suppose we want to define an {\bf abelian extension} $\widetilde \g$ 
of $\g$ by $V$, i.e., a Lie algebra which can be included in the short exact sequence 
$$
0\to V\to \widetilde \g\to \g\to 0
$$
where $V$ is an abelian ideal. To classify such extensions, pick a vector space splitting 
$\widetilde \g=\g\oplus V$, then the commutator looks like 
$$
[(x,v),(y,w)]=([x,y],xw-yv+\omega(x,y)),
$$
where $\omega: \wedge^2\g\to V$ is a linear map. The Jacobi identity is then equivalent 
to $\omega$ being in the space $Z^2(\g,V)$ of 2-cocycles. Moreover, it is easy to check 
that for two such extensions $\widetilde \g_1,\widetilde \g_2$ there is an isomorphism 
$\phi: \widetilde \g_1\to \widetilde \g_2$ which acts trivially on $V$ and $\g$ if and only if the corresponding cocycles $\omega_1,\omega_2$ differ by a coboundary: 
$\omega_1-\omega_2\in B^2(\g,V)$. Thus, we get 

\begin{proposition} Abelian extensions  of $\g$ by $V$ modulo isomorphisms 
which act trivially on $V$ and $\g$ are classified by $H^2(\g,V)$. 
For example, the space $H^2(\g,\Bbb C)$ classifies 
{\bf $1$-dimensional central extensions} of $\g$: 
$$
0\to \Bbb C\to \widetilde \g\to \g\to 0.
$$
\end{proposition} 

\begin{example} Let $\g=\Bbb C^2$ be the 2-dimensional abelian Lie algebra. 
Then we have seen that the Poincar\'e polynomial of the cohomology of $\g$ 
is $1+2q+q^2$ (cohomology of the 2-torus). So $H^2(\g,\Bbb C)=\Bbb C$. 
The only cocycle up to scaling is given by $\omega(x,y)=1$, where $x,y$ 
is a basis of $\g$, and all coboundaries are zero. So we have a central extension 
of $\g$ defined by this cocycle with basis $x,y,c$ and 
$[x,y]=c$, $[x,c]=[y,c]=0$. This is the {\bf Heisenberg Lie algebra}, 
which is isomorphic to the Lie algebra of strictly upper-triangular $3$ by $3$ matrices.     
\end{example} 

{\bf 5.} Let us now study deformations of Lie algebras. Suppose $\g$ is a Lie algebra over a field $\bold k$ and we want to deform the bracket, with deformation parameter $t$. So the new bracket will be 
$$
[x,y]_t=[x,y]+tc_1(x,y)+t^2c_2(x,y)+...,
$$
where $c_i: \wedge^2\g\to \g$ are linear maps. This bracket should satisfy the Jacobi identity, i.e., define a new Lie algebra structure on $\g[[t]]$ (over $\bold k[[t]]$). Such deformations are distinguished up to 
linear isomorphisms
$$
a=1+ta_1+t^2a_2+...
$$
where $a_i\in \End_{\bf k}(\g)$. 

In particular, in first order, i.e., modulo $t^2$, 
we get a new Lie algebra structure on $\g[t]/t^2\g[t]=\g\oplus t\g$ such that this Lie algebra can be included in the short exact sequence 
$$
0\to t\g\to \g\oplus t\g\to \g\to 0
$$
where $t\g\cong \g$ is an abelian ideal with adjoint action of $\g$ (note that this Lie algebra structure is automatically $\bold k[t]/t^2$-linear). So this is an abelian extension 
of $\g$ by $t\g$, and we know that such extensions are classified by $H^2(\g,\g)$. 
So we obtain 

\begin{proposition} 
First-order deformations of $\g$ as a Lie algebra are classified by $H^2(\g,\g)$. 
\end{proposition} 

Thus if $H^2(\g,\g)=0$, every deformation is isomorphic to the trivial one, with $c_1=c_2=...=0$. Indeed, applying automorphisms \newline $a=1+ta_1+t^2a_2+...$, we can kill successively $c_1$, then $c_2$, then $c_3$, and so on. Thus from Whitehead's theorem we obtain

\begin{corollary} If $\g$ is semisimple then it is rigid, i.e., has no nontrivial Lie algebra deformations. 
\end{corollary} 

\begin{example} Let $\g$ be the 2-dimensional abelian Lie algebra over $\Bbb C$. 
Then $H^2(\g,\g)=\Bbb C^2$, and we get a 2-parameter family of deformations
with bracket $[x,y]=tx+sy$. These, however, turn out to be all equivalent (for $(t,s)\ne (0,0)$)
under the action of $GL_2(\Bbb C)$: they are all isomorphic to the Lie algebra with basis 
$x,y$ and commutator $[x,y]=y$. 
\end{example} 

However, not all first order deformations of a Lie algebra lift to second order, i.e., modulo $t^3$.
Namely, the Jacobi identity in the second order tells us that $dc_2=[c_1,c_1]$, where 
$[c_1,c_1]$ is the {\bf Schouten bracket} of $c_1$ with itself:
$$
[c_1,c_1](x,y,z)=c_1(c_1(x,y),z)+c_1(c_1(y,z),x)+c_1(c_1(z,x),y).
$$
This expression is automatically a cocycle (check it!), but we need it to be a coboundary. 
So the cohomology class of $[c_1,c_1]$ in $H^3(\g,\g)$ is an obstruction to lifting the deformation modulo $t^3$. Thus the space $H^3(\g,\g)$ is the home for {\bf obstructions to deformations}. For example, 
if $\g$ is abelian then $H^2(\g,\g)=\Hom_{\bold k}(\wedge^2\g,\g)$, 
and the obstruction to extending $c=tc_1$ modulo $t^3$ 
is 
$$
{\rm Jacobi}(c_1):=[c_1,c_1]\in H^3(\g,\g)=\Hom_{\bold k}(\wedge^3\g,\g). 
$$

{\bf 6.} In a similar way we can study deformations $V[[t]]$ of a module $V$ over $\g$:
$$
\rho_t(x)=\rho(x)+t\rho_1(x)+t^2\rho_2(x)+...
$$
Modulo $t^2$ we get a $\g$-module structure on $V[t]/t^2V[t]=V\oplus tV$ such that 
we have a short exact sequence 
$$
0\to tV\to V\oplus tV\to V\to 0.
$$
Thus first order deformations of $V$ are classified by ${\rm Ext}^1_\g(V,V)=H^1(\g,\End_{\bf k}V)$. Again, lifting of this deformation modulo $t^3$ is not automatic, and we get an obstruction in $\Ext^2_\g(V,V)=H^2(\g,\End_\k(V))$. 

\begin{exercise} (i) Let $\a,\g$ be Lie algebras and $\phi: \a\to \g$ a homomorphism. 
Show that first order deformations of $\phi$ are classified by $H^1(\a,\g)$, where
$a\in \a$ acts on $\g$ by ${\rm ad}\phi(a)$. 

(ii) Show that if $\a$ is semisimple and $\g$ finite dimensional over $\Bbb C$ then 
$H^1(\a,\g)=0$. 

(iii) Show that if $\a,\g$ are semisimple complex Lie algebras then 
there are only finitely many homomorphisms $\a\to \g$ up to conjugation by $G_{\rm ad}$. ({\bf Hint}: Consider the affine algebraic variety $X\subset {\rm Hom}_{\Bbb C}(\a,\g)$ of all homomorphisms and show that the tangent space $T_\phi X$ is $Z^1(\a,\g)$, the space of 1-cocycles. Then use (ii) to deduce that $X$ is the union of finitely many orbits of $G_{\rm ad}$.) 

(iv) How many conjugacy classes do we have in (iii) if $\a=\mathfrak{sl}_2$ 
and $\g=\mathfrak{sl}_n,\mathfrak{so}_n,\mathfrak{sp}_{2n}$?  
\end{exercise} 

\subsection{Levi decomposition} \label{levide}

\begin{theorem}\label{levi1} (Levi decomposition, Theorem \ref{levi}) Over real or complex numbers  
we have $\g\cong {\rm rad}(\g)\oplus \g_{\rm ss}$, 
where $\g_{\rm ss}\subset \g$ is a semisimple subalgebra (but not necessarily an ideal); i.e., 
$\g$ is isomorphic to the semidirect product $\g_{\rm ss}\ltimes {\rm rad}(\g)$. 
In other words, the projection $p: \g\to \g_{\rm ss}$ admits an (in general, 
non-unique) splitting $q: \g_{\rm ss}\to \g$, i.e., a Lie algebra map such that $p\circ q={\rm Id}$. 
\end{theorem} 
  
\begin{proof} We can write $\g=\g_{ss}\oplus {\rm rad}(\g)$ as a vector space. Then the commutator looks like 
$$
[(a,x),(b,y)]=([x,b]-[y,a]+[a,b]+\omega(x,y),[x,y]),\ x,y\in \g_{ss}, a,b\in {\rm rad}(\g).  
$$

Let ${\rm rad}(\g)=D^0\supset D^1\supset...$ be the derived series of ${\rm rad}(\g)$, i.e., $D^{i+1}=[D^i,D^i]$. Suppose $D^n\ne 0$ but $D^{n+1}=0$ (so $D^n$ is an abelian ideal). 
Using induction in dimension of $\g$ and replacing $\g$
by $\g/D^n$, we may assume that $\omega(x,y)\in D^n$. But then  
$\omega\in Z^2(\g_{ss},D^n)$, which equals $B^2(\g_{ss},D^n)$ 
by Whitehead's theorem, i.e., $\omega=d\eta$. Using $\eta$, we can modify 
the splitting $\g=\g_{ss}\oplus {\rm rad}(\g)$ to make sure that $\omega=0$.    
This implies the statement.\footnote{In other words, we have reduced to the case when ${\rm rad}(\g)=V$ is abelian, and we have shown above that abelian extensions are classified by $H^2(\g_{ss},V)$, which is zero by Whitehead's theorem.}    
\end{proof} 

\section{\bf The third fundamental theorem of Lie theory} \label{thirdlie}

\subsection{Exponentiating nilpotent and solvable Lie algebras and the third fundamental theorem of Lie theory}

The following theorem implies the third fundamental theorem of Lie theory for solvable Lie algebras. Let $\g$ be a finite dimensional solvable Lie algebra over $\Bbb K=\Bbb R$ or $\Bbb C$ of dimension $n$. 

\begin{theorem}\label{solvthir} There is a simply connected Lie group $G$ over $\Bbb K$ with ${\rm Lie}(G)=\g$, diffeomorphic to $\Bbb K^n$. Moreover, if $\g$ is nilpotent then the exponential map $\exp: \g\to G$ is a diffeomorphism, and if we use it to identify $G$ with $\g$ then 
the multiplication map $\mu: \g\times \g\to \g$ is polynomial.  
 \end{theorem} 

\begin{proof} The proof is by induction in $n$, with trivial base $n=0$. Namely, 
fix a nonzero homomorphism $\chi: \g\to \Bbb K$ (which exists since $\g$ is solvable), 
and let $\g_0={\rm Ker}\chi$.  Then we have $\g=\Bbb K\bold d\ltimes \g_0$, the semidirect product, where $\bold d\in \g$ acts as a derivation $d$ on $\g_0$. Let $G_0$ be the simply connected Lie group corresponding to $\g_0$, which is defined by the induction assumption. So we have a 1-parameter group of automorphisms $e^{td}: \g_0\to \g_0$ which by the second fundamental theorem of Lie theory gives rise to a 1-parameter group of automorphisms $e^{td}: G_0\to G_0$. 
Thus we can define a group structure on $G:=G_0\times \Bbb K$ by the formula 
$$
(x,t)\cdot (y,s)=(x\cdot e^{td}(y),t+s),\ x,y\in G_0,\ t,s\in \Bbb K.
$$
Otherwise formulated, $G=\Bbb K\ltimes G_0$. 
This gives a desired group $G$ with Lie algebra $\g$. 

Moreover, if $\g$ is nilpotent then by the induction assumption the exponential map 
$\g_0\to G_0$ is a diffeomorphism, and if we use it to identify $\g_0$ with $G_0$ then the multiplication $\mu_0: \g_0\times \g_0\to \g_0$ is polynomial. So we may realize $G$ as $\g=\g_0\times \Bbb K$ with multiplication law 
$$
(X,t)*(Y,s)=\mu((X,t),(Y,s))=(\mu_0(X,e^{td}(Y)),t+s),\ X,Y\in \g_0,\ t,s\in \Bbb K.
$$
By nilpotency $d^N=0$ for some $N$, so 
$$
e^{td}(Y)=\sum_{n=0}^{N-1}
\frac{t^nd^n(Y)}{n!},
$$
so we see that $\mu$ is polynomial. Also, the BCH series of Subsection \ref{bchf} truncates to a polynomial group law on $\g$, associative by associativity of formal multiplication of exponentials. Its Lie algebra is $\g$, and its one-parameter subgroups are $t\mapsto tX$, so its exponential map is the identity. By Theorem \ref{second}, this simply connected group is isomorphic to $G$. Thus the exponential map for $\g$ is a diffeomorphism, and multiplication in exponential coordinates is polynomial. 
\end{proof}  

\begin{example} Let $\g$ be the Heisenberg Lie algebra, i.e. the Lie algebra of strictly upper triangular 3-by-3 matrices. Then under such identification the multiplication map in the corresponding Heisenberg group $G$ has the form 
$$
(x,y,z)*(x',y',z')=(x+x',y+y',z+z'+\tfrac{1}{2}(xy'-x'y)).
$$
\end{example} 

\begin{exercise} Show that if $\g$ is the 2-dimensional non-abelian complex 
Lie algebra and $G$ the corresponding simply connected Lie group then 
$\exp: \g\to G$ is not injective. 
\end{exercise}

\begin{definition} The simply connected Lie group whose Lie algebra is nilpotent is called {\bf unipotent}.\footnote{The reason for this terminology is that these groups act by unipotent operators on the adjoint representation.}  
\end{definition} 

\begin{corollary} (Third fundamental theorem of Lie theory, Theorem \ref{third}) For any finite dimensional Lie algebra $\g$ over $\Bbb R$ or $\Bbb C$ there is a simply connected Lie group $G$ with ${\rm Lie}(G)=\g$. 
\end{corollary} 

\begin{proof} By Theorem \ref{solvthir}, we have such a group $A$ for $\mathfrak{a}={\rm rad}(\g)$. Moreover, by the Levi decomposition theorem, the simply connected semisimple group 
$G_{ss}$ corresponding to $\g_{ss}$ acts on ${\rm rad}(\g)$. Hence by 
the second fundamental theorem of Lie theory, $G_{ss}$ acts on $A$, and the simply connected Lie group $G_{ss}\ltimes A$ has the Lie algebra $\g_{ss}\ltimes {\rm rad}(\g)=\g$. 
\end{proof} 

\begin{corollary}\label{homtype} A simply connected complex Lie group $G$ is of the form $G_{ss}\ltimes A$, where $A$ is solvable simply connected, hence 
diffeomorphic to $\Bbb C^n$, and $G_{ss}$ is a simply connected semisimple 
complex Lie group. Thus $G$ has the homotopy type of 
$G_{ss}^c$. 
\end{corollary} 

\subsection{Formal groups} 

The third fundamental theorem of Lie theory assigns a simply connected Lie group $G$ to any finite dimensional Lie algebra $\g$ over $\Bbb R$ or $\Bbb C$, such that ${\rm Lie}G=\g$. But what about infinite dimensional Lie algebras? There are some examples when this is possible, for instance for $\g={\rm Vect}(M)$, the Lie algebra of vector fields for a smooth manifold $M$, we can take $G$ to be the universal cover of ${\rm Diff}_0(M)$, the group of diffeomorphisms of $M$ homotopic to the identity, 
and for $\g=C^\infty(S^1,\mathfrak{k})$ for a finite dimensional Lie algebra $\mathfrak{k}$ we can take $G=C^\infty(S^1,K)$, where $K$ is the simply connected 
Lie group corresponding to $\mathfrak{k}$ (although we would need to explain in what sense $G$ is a Lie group and ${\rm Lie}G=\g$). However, for a general infinite dimensional $\g$, such an assignment is typically impossible and a suitable group $G$ does not exist. 

However, this assignment becomes possible (and in fact not just over $\Bbb R$ and $\Bbb C$ but over any field of characteristic zero) if we replace the notion of a Lie group with a purely algebraic notion of a {\bf formal group}. Roughly speaking, the notion of a formal group is the analog of the notion of a real or complex analytic Lie group where analytic functions are replaced by formal power series, and we don't worry about their convergence. This allows us to work with infinite dimensional Lie algebras and over arbitrary fields of characteristic zero. 
 
Let us give a precise definition. Given a vector space $V$ 
over a field ${\bf k}$ of characteristic zero, define the algebra   
${\bf k}[[V]]$ of {\bf formal regular functions} on $V$ to be $(SV)^*$, the dual of the symmetric algebra of $V$. Since $SV$ has a bialgebra structure $\Delta_0: SV\to SV\otimes SV$ defined by $\Delta_0(v)=v\otimes 1+1\otimes v$ for $v\in V$, 
the dual map $\Delta_0^*$ gives a commutative associative product 
on ${\bf k}[[V]]$, which is continuous in the weak topology of the dual space.\footnote{Recall that if $E$ is a vector space then the dual space $E^*$ carries the {\it weak topology} whose basis of neighborhoods of zero is given by orthogonal complements of 
finite dimensional subspaces of $E$.}
If $x_i,i\in I$ is a linear coordinate system on $V$ corresponding to a basis $v_i,i\in I$, then we have a natural identification ${\bf k}[[V]]\cong {\bf k}[[x_i,i\in I]]$ of ${\bf k}[[V]]$ with the algebra of formal power series in $x_i$. Note that here $I$ can be a set of any cardinality, not necessarily finite or countable. Moreover, if $\dim V<\infty$ then ${\bf k}[[V]]=\prod_{n\ge 0}S^nV^*$. 

Finally, note that we have the augmentation homomorphism (counit) $\varepsilon: {\bf k}[[V]]\to {\bf k}$ given by $\varepsilon(f)=f(0)$, i.e., obtained by taking the quotient by the maximal ideal $\mathfrak{m}\subset {\bf k}[[V]]$. 

\begin{definition} A {\bf formal group structure} on $V$ is a (topological) coproduct 
$\Delta: {\bf k}[[V]]\to {\bf k}[[V\oplus V]]$, i.e., a continuous\footnote{Note that if $\dim V<\infty$, any such homomorphism is automatically continuous.} homomorphism which is coassociative and compatible with the counit: 
$$
(\Delta\otimes {\rm Id})\circ \Delta(f)=({\rm Id}\otimes \Delta)\circ \Delta(f),\ (\varepsilon\otimes {\rm Id})\circ \Delta(f)=({\rm Id}\otimes\varepsilon)\circ \Delta(f)=f.
$$
 A {\bf formal group} over ${\bf k}$ is a pair $G=(V,\Delta)$. 
We will denote ${\bf k}[[V]]$ by $\O(G)$ and call it the {\bf algebra of regular functions} on $G$. We define the {\bf dimension} of $G$ by $\dim G:=\dim V$.  
\end{definition} 

A {\bf (homo)morphism of formal groups} $\phi: G_1\to G_2$ 
is a (continuous)  algebra homomorphism $\O(G_2)\to \O(G_1)$ preserving the coproduct.\footnote{Thus we forget the linear structure on $V$ (it does not have to be preserved by homomorphisms). In other words, to specify a formal group, we don't need to specify a vector space $V$ but only need to specify a (topological) ring isomorphic to ${\bf k}[[V]]$ for some $V$ and equipped with a coproduct.}  

For example, a 1-dimensional formal group is defined by a power series 
$F(x,y)\in {\bf k}[[x,y]]$, $F(x,y)=x+y+...$, where ... denotes quadratic and higher terms, which is associative:
$$
F(F(x,y),z)=F(x,F(y,z)).
$$
Such a series $F$ is called a {\bf formal group law}. Namely, the map 
$\Delta: {\bf k}[[x]]\to {\bf k}[[x_1,x_2]]$ 
is defined by the formula 
$$
\Delta(f)(x_1,x_2)=f(F(x_1,x_2)).
$$  
Higher-dimensional formal groups $G$ can also be presented in this way, with $F,x,y$ being vectors with $\dim G$ entries rather than scalars.

\begin{example} 1. The {\bf additive formal group}: $\Delta(f)=f\otimes 1+1\otimes f$, $f\in V^*$ (extended by multiplicativity and continuity to $\k[[V]]$). In other words, $F(x,y)=x+y$ and $\Delta(f)(x,y)=f(x+y)$. 

2. Let $G$ be a real or complex Lie group. Then the multiplication map 
$G\times G\to G$ is smooth. So we can take its Taylor expansion 
at the unit element, which defines a formal group $G_{\rm formal}$ 
called the {\bf formal completion of $G$ at the identity.} Its coproduct is defined by the formula $\Delta(f)(x,y)=f(x\circ y)$ where $(x,y)\mapsto x\circ y$ denotes the group law of $G$. The same construction is valid for an algebraic group over any field. 
\end{example} 

So what does it have to do with groups? In fact, a lot: if $G$ is a formal group 
then it defines a functor from the category of local commutative finite dimensional ${\bf k}$-algebras to the category of groups, 
$$
R\mapsto G(R)=\Hom_{\rm continuous}(\O(G),R) 
$$
(where the topology on $R$ is discrete).\footnote{Again, continuity is automatic if $\dim G<\infty$.}  Namely, the group law on such homomorphisms is defined by 
$$
(a\circ b)(f)=(a\otimes b)(\Delta(f)). 
$$ 
This makes sense even though $\Delta(f)$ does not belong 
to ${\bf k}[[V]]\otimes {\bf k}[[V]]$ but only to its completion 
${\bf k}[[V\oplus V]]$ since $R$ is finite dimensional. 

\begin{exercise} Show that $(G(R),\circ)$ is a group. 
\end{exercise} 

Moreover, any (homo)morphism of formal groups 
$G_1\to G_2$ defines a morphism 
of functors $G_1(?)\to G_2(?)$, and this assignment is compatible with 
composition. Furthermore, it is not hard to show that this assignment 
can be inverted, which allows us to define formal groups as representable 
functors from local finite dimensional commutative algebras to groups. 

Any formal group $G$ defines a Lie algebra ${\rm Lie}G$, 
which as a vector space is the continuous dual $\g:=({\mathfrak{m}}/{\mathfrak{m}^2})^*$. In other words, it is the underlying vector space $V$ of $G=(V,\Delta)$. Note that by compatibility of $\Delta$ with $\varepsilon$, for $f\in \mathfrak m$ the element $\Delta(f)-f\otimes 1-1\otimes f$ belongs to the completed tensor product $\mathfrak{m}\widehat \otimes \mathfrak{m}$, thus projects to a well defined element of 
$(\g\otimes \g)^*$. Thus the same is true for the element 
$\Delta(f)-\Delta^{\rm op}(f)$ (where $\Delta^{\rm op}$ is obtained from $\Delta$ by swapping components); in fact, it defines an element 
of $(\wedge^2\g)^*$. Moreover,  this element only depends 
on the residue $\overline f$ of $f$ in $\g^*= {\mathfrak{m}}/{\mathfrak{m}^2}$ (check it!). Denote the projection of $\Delta(f)-\Delta^{\rm op}(f)$ to $(\wedge^2\g)^*$
by $\delta(\overline f)$. Then $\delta: \g^*\to (\wedge^2\g)^*$ 
is continuous, so it is dual to the map $[,]=\delta^*: \wedge^2\g\to \g$, and it is easy to show that $[,]$ is a Lie bracket on $\g$; namely, the Jacobi identity follows from the coassociativity of $\Delta$ (check it!). 

Conversely, given a Lie algebra $\g$ over ${\bf k}$ (not necessarily finite dimensional), we can use the Baker-Campbell-Hausdorff formula (Subsection \ref{bchf}) to assign a formal group to $\g$. Namely, 
take $V=\g$ and define $\Delta: {\bf k}[[\g]]\to {\bf k}[[\g\oplus \g]]$
by 
$$
\Delta(f)(x,y)=f(\mu(x,y)),
$$
where $\mu(x,y)=x+y+\frac{1}{2}[x,y]+...$ is the Baker-Campbell-Hausdorff series. 
Then the coassociativity of $\Delta$ follows from the associativity of $\mu$. 
In other words, we define $G$ by setting its formal group law $F$ to be equal to $\mu$. 

\begin{example} Let $\g$ be a Lie algebra and $G$ be the corresponding formal group. Let $R$ be a finite dimensional local commutative algebra with maximal ideal $\mathfrak{m}_R$. Then $G(R)=\mathfrak{m}_R\otimes \g$ with group law 
$$
(x,y)\mapsto \mu(x,y)
$$
(which makes sense since the series terminates). 
\end{example} 

\begin{theorem} (The fundamental theorems of Lie theory for formal groups) These assignments are mutually inverse equivalences between the category 
of formal groups over ${\bf k}$ and the category of Lie algebras over ${\bf k}$.  
\end{theorem}

\begin{proof} The proof is analogous to the proof of the first two fundamental theorems for usual Lie groups (but without the analytic details), and we leave it as an exercise. Note that the third theorem, which was the hardest for usual Lie groups, assigning a group to a Lie algebra, has already been proved above by using the series $\mu(x,y)$.   
\end{proof} 

\begin{corollary}\label{alliso} Every $1$-dimensional formal group $G$ over a field of characteristic zero is isomorphic 
to the additive formal group, with $\Delta(f)(x,y)=f(x+y)$. 
\end{corollary} 

Over a field of positive characteristic (or over a commutative ring, such as $\Bbb Z$), much, but not all, of this story extends; let us for simplicity consider the finite dimensional case over a field. Namely, the definition of a formal group structure (say, on a finite dimensional space) is the same: it's a coproduct on ${\bf k}[[x_1,...,x_n]]$ with the same properties as above.\footnote{More precisely, instead of $SV$ we should take the {\bf symmetric algebra with divided powers} $\Gamma V$, defined by $\Gamma^mV:=(S^mV^*)^*$. Note that in characteristic $p$, 
$\Gamma^mV$ is not naturally isomorphic to $S^mV$ for $m\ge p$.} 
 The definition of the Lie algebra of a formal group also goes along for the ride. However, the reverse assignment fails, since the series $\mu(x,y)$ is only defined over $\Bbb Q$ and has all primes occurring in denominators of its coefficients. As a result, not any Lie algebra gives rise to a formal group, and the fundamental theorems of Lie theory for formal groups don't hold. 

In particular, there are many non-isomorphic 1-dimensional formal groups.  For example, we have the 
additive group law $F(x,y)=x+y$ as above, but also 
the {\bf multiplicative group law} $F(x,y)=x+y+xy$, 
which is called so because this means that 
$1+F(x,y)=(1+x)(1+y)$. In characteristic zero these are isomorphic 
by the map 
$$
x\mapsto e^x-1=\sum_{n\ge 1}\frac{x^n}{n!},
$$
(not surprisingly in view of Corollary \ref{alliso}), but in positive characteristic this series does not make sense and in fact the additive and multiplicative formal groups are not isomorphic (check it!). There are also many other 1-dimensional formal group laws, commutative and not. Such (commutative) formal group laws are very important in algebraic topology, since they parametrize (complex-oriented) cohomology theories. For example, the additive group law corresponds to ordinary cohomology and the multiplicative one to $K$-theory. In characteristic zero the isomorphism between the additive and multiplicative formal groups leads to the {\bf Chern character map} which identifies cohomology and $K$-theory of a topological space with $\Bbb Q$-coefficients.

\section{\bf Ado's theorem} \label{adothm} 

\subsection{The nilradical} 

Consider now a solvable Lie algebra $\a$ over $\Bbb C$ and its adjoint representation. By Lie's theorem, in some basis $\a$ acts in this representation by upper triangular matrices. Let $\n\subset \a$ be the subset of nilpotent elements 
(the {\bf nilradical} of $\a$). Thus $\n$ is the set of $x\in \a$ that act in this basis by strictly upper triangular matrices. In particular, $\n\supset [\a,\a]$, so $\a/\n$ is abelian. 

\begin{proposition}\label{trivact} If $d: \a\to \a$ is a derivation then $d(\a)\subset \n$. Thus if $\a={\rm rad}(\g)$ is the radical of $\g$ then $\g$ acts trivially on $\a/\n$.
\end{proposition} 
 
\begin{proof}  The derivation $d$ defines a solvable Lie algebra $\widetilde \a:=\Bbb Cd\ltimes \a$, so $[\widetilde \a,\widetilde \a]\subset \a$ consists of nilpotent elements. In particular it lies in $\n$.\footnote{Here is another proof of this proposition.  
The one-parameter group $e^{td}$ of automorphisms of $\a$ preserves the set of characters of $\a$ occurring in its adjoint representation. Hence it must preserve each of them individually, as there are finitely many and this group is connected. But by definition of $\n$ these characters span $(\a/\n)^*$. Thus $d$ acts trivially on $\a/\n$.}
\end{proof}

\subsection{Algebraic Lie algebras} 

Let us say that a finite dimensional complex Lie algebra $\g$ is {\bf algebraic} if 
$\g$ is the Lie algebra of a group $G=K\ltimes N$, where $K$ is a reductive group and $N$ a unipotent group. It turns out that this is equivalent to being the Lie algebra of an {\bf affine algebraic group} over $\Bbb C$ (i.e., a closed subgroup in $GL_n(\Bbb C)$ defined by polynomial equations), which motivates the terminology.  

A finite dimensional complex Lie algebra need not be algebraic: 

\begin{example}\label{nonal} Let $\g_1$ be a 3-dimensional Lie algebra 
with basis $d,x,y$ and $[x,y]=0$, $[d,x]=x$, $[d,y]=\sqrt{2}y$. 
Similarly, let $\g_2$ have basis $d,x,y$ with $[x,y]=0$, $[d,x]=x$, $[d,y]=y+x$. 
Then $\g_1,\g_2$ are not algebraic (check it!). 
\end{example} 

Nevertheless, we have the following proposition. 

\begin{proposition}\label{embe} Any finite dimensional complex Lie algebra 
is a Lie subalgebra of an algebraic one. 
\end{proposition} 

\begin{proof} Let us say that $\g$ is {\bf $n$-algebraic} if it is the Lie algebra 
of a group $G:=K\ltimes A$, where $K$ is reductive and $\a={\rm Lie}(A)$ is solvable with $\dim(\a/\n)\le n$, where $\n$ is the nilradical of $\a$. Thus $0$-algebraic is the same as algebraic. Note that for any $\g$ we have the Levi decomposition $\g=\g_{ss}\ltimes \a$, where 
$\a={\rm rad}(\g)$, which shows that any $\g$ is $n$-algebraic for some $n$. 
So it suffices to show that any $n$-algebraic Lie algebra for $n>0$ 
embeds into an $n-1$-algebraic one. 

To this end, let $\g={\rm Lie}(G)$ be $n$-algebraic, with $G=
K\ltimes A$ and $A$ simply connected. Let $\a={\rm Lie}(A)$, so $\dim(\a/\n)=n$. Pick $d\in \a$, $d\notin \n$ such that $d$ is $K$-invariant. 
This can be done since by Proposition \ref{trivact} $K$ acts trivially on $\a/\n$ and its representations are completely reducible. We have a decomposition $\a={\oplus_{i=1}^r} \a[\beta_i]$ of $\a$ into generalized eigenspaces of $d$. It is clear that $K$ preserves each $\a[\beta_i]$. Pick a character $\chi: \a\to \Bbb C$ such that $\chi(\mathfrak{n})=0$ and $\chi(d)=1$. 

Consider the subgroup $\Gamma$ of $\Bbb C$ generated by $\beta_i$ and let $\alpha_1,...,\alpha_m$ be a basis of $\Gamma$, so that $\beta_i=\sum_j b_{ij}\alpha_j$ for $b_{ij}\in \Bbb Z$.  
Let $T=(\Bbb C^\times)^m$ and make $T$ act on $G$ so that it commutes with $K$ and acts on $\a[\beta_i]$ by $(z_1,...,z_m)\mapsto \prod_j z_j^{b_{ij}}$. Now consider the group $\widetilde G:=(K\times T)\ltimes A$.
Let $\a'\subset {\rm Lie}(T)\ltimes \a\subset {\rm Lie}(\widetilde G)$ be spanned by ${\rm Ker}\chi$ and $d-\alpha$ where $\alpha=(\alpha_1,...,\alpha_m)\in {\rm Lie}(T)$. Then the nilradical 
$\n'$ of $\a'$ is spanned by $\n$ and $d-\alpha$ (as the latter is nilpotent). Moreover, if $A'$ is the simply connected group corresponding to $\a'$, then $(K\times T)\ltimes A\cong (K\times T)\ltimes A'$. Thus, the Lie algebra $\widetilde \g:={\rm Lie}(\widetilde G)$ is $n-1$-algebraic (as $\dim(\a'/\n')=n-1$), and it contains $\g$, as claimed. 
\end{proof} 
  
\begin{example} The Lie algebras $\g_1,\g_2$ in the Example \ref{nonal} are 
$1$-algebraic. 

To embed $\g_1$ into an algebraic Lie algebra, 
add element $\delta$ with $[\delta,x]=0$, $[\delta,y]=y$, 
$[\delta,d]=0$. Then the Lie algebra $\g_1'$ spanned by 
$\delta,d,x,y$ is $\b\oplus \b$, where $\b$ is the non-abelian 
2-dimensional Lie algebra (so it is algebraic). Namely, the first copy of $\b$ 
is spanned by $\delta,y$ and the second by $d-\sqrt{2}\delta,x$.  

To embed $\g_2$ into an algebraic Lie algebra, 
add element $\delta$ with $[\delta,x]=0$, $[\delta,y]=x$, 
$[\delta,d]=0$. Then the Lie algebra $\g_2'$ spanned by 
$\delta,d,x,y$ is $\Bbb C\ltimes \mathcal H$, where $\mathcal H$ is the 3-dimensional  Heisenberg Lie algebra with basis $\delta,x,y$, and $\Bbb C$ is spanned by $d-\delta$ (so it is algebraic, as $d-\delta$ acts diagonalizably with integer eigenvalues). 
\end{example} 

\subsection{Faithful representations of nilpotent Lie algebras} 

Let $\n$ be a finite dimensional nilpotent Lie algebra over $\Bbb C$. In this subsection we will show that $\n$ has a finite dimensional faithful representation. 

To this end, recall that by Theorem \ref{solvthir}, $\n={\rm Lie}(N)$ where $N$ is a simply connected Lie group, and the exponential map $\exp: \n\to N$ is bijective. Moreover, 
the multiplication law of $N$, when rewritten on $\n$ using the exponential map,  
is given by polynomials. 

\begin{proposition} Let $\O(N)$ be the space of polynomial functions on $N\cong \n$ (identified using the exponential map). Then $\O(N)$ is invariant 
under the action of $\n$ by left-invariant vector fields. Moreover, we have a canonical filtration $\O(N)=\cup_{n\ge 1}V_n$, where $V_n\subset \O(N)$ are finite dimensional subspaces such that $V_1\subset V_2\subset...$ and 
$\n V_n\subset V_{n-1}$.  
\end{proposition} 

\begin{proof} Let $\mu: \n\times \n\to \n$ be the polynomial multiplication law. 
Let $x\in \n$ and $L_x$ be the corresponding left-invariant 
vector field. Let $f\in \O(N)=S\n^*$. Then for $y\in \n$ we have 
$$
(L_xf)(y)=\frac{d}{dt}|_{t=0}f(\mu(y,tx)).
$$ 
Since $f$ and $\mu$ are polynomials, this is clearly a polynomial in $y$. 
Thus $L_x: \O(N)\to \O(N)$. 

We have a lower central series filtration on $\n$: 
$$
\n=D_0(\n)\supset [\n,\n]=D_1(\n)\supset...\supset D_m(\n)=0.
$$
This gives an ascending filtration 
$$
0=D_0(\n)^\perp\subset....\subset D_m(\n)^\perp=\n^*.
$$ 
We assign to $D_j(\n)^\perp$ filtration degree $d^j$, where 
$d$ is a sufficiently large positive integer. This gives rise to an ascending filtration 
$F^\bullet$ on $S\n^*=\O(N)$. Note that 
$$
\mu(x,y)=x+y+\sum_{i\ge 1}Q_i(x,y), 
$$
where $Q_i:\n\times \n\to [\n,\n]$ has degree $i$ in $y$. 
Thus 
$$
(L_xf)(y)=(\partial_x f)(y)+(\partial_{Q_1(y,x)}f)(y).
$$
The first term clearly lowers the degree, and so does the second one if 
$d$ is large enough. So we may take $V_n=F_n(S\n^*)$ to be the space of polynomials of degree $\le n$, then $L_xV_n\subset V_{n-1}$, as claimed. 
\end{proof} 

\begin{example} We illustrate this proof on the example of the Heisenberg algebra $\mathcal H=\langle x,y,c\rangle$ with $[x,y]=c$ and $[x,c]=[y,c]=0$. 
In this case 
$$
e^{tx}e^{sy}=e^{tx+sy+\frac{1}{2}tsc}, 
$$
so writing $u=px+qy+rc\in \mathcal H$, we get 
$$
\mu((p_1,q_1,r_1),(p_2,q_2,r_2))=(p_1+p_2,q_1+q_2,r_1+r_2+\tfrac{1}{2}(p_1q_2-p_2q_1)).
$$
Thus 
$$
L_c=\partial_r,\ L_x=\partial_p-\tfrac{1}{2}q\partial_r,\ L_y=\partial_q+\tfrac{1}{2}p\partial_r.
$$
We have $D_1(\mathcal H)=\Bbb Cc$, so $D_1(\mathcal H)^\perp$ is spanned by $p,q$. 
Thus we have $\deg(p)=\deg(q)=d$, $\deg(r)=d^2$. So for any $d>1$, 
$L_c,L_x,L_y$ lower the degree. So setting $V_n=F_n(S\mathcal H^*)$ to be the (finite dimensional) space of polynomials of degree $\le n$, we see that $L_c,L_x,L_y$ map $V_n$ to $V_{n-1}$. 
\end{example} 

\begin{corollary}\label{fai} Every finite dimensional nilpotent Lie algebra $\n$ over $\Bbb C$ has a faithful finite dimensional representation where all its elements act by nilpotent operators. Thus $\n$ is isomorphic to a subalgebra of 
the Lie algebra of strictly upper triangular matrices of some size.   
\end{corollary} 

\begin{proof} By definition, $\O(N)$ is a faithful $\n$-module. 
Hence so is $V_n$ for some $n$. 
\end{proof} 

\subsection{Faithful representations of general finite dimensional Lie algebras} 

\begin{theorem} (Ado's theorem) Every finite dimensional Lie algebra over $\Bbb C$ has a finite dimensional faithful representation. 
\end{theorem}

\begin{proof} Let $\g$ be a finite dimensional complex Lie algebra. 
By Proposition \ref{embe}, $\g$ can be embedded into an algebraic Lie algebra, so we may assume without loss of generality that $\g$ is algebraic. Thus 
$\g={\rm Lie}(G)$ where $G=K\ltimes N$ 
for reductive $K$ and unipotent $N$. Also we may assume that $\g\ne \g'\oplus \g''$ for $\g',\g''\ne 0$, otherwise the problem reduces to a smaller algebraic Lie algebra (indeed if $V',V''$ are faithful representations of $\g',\g''$ then $V'\oplus V''$ is a faithful representation of $\g'\oplus \g''$). 
If $\mathfrak n=0$, then $\g$ is reductive, hence has a faithful finite-dimensional representation, so we are done. So we may assume that $\mathfrak{n}\ne 0$, in  which case $\mathfrak{k}={\rm Lie}(K)$ acts faithfully on $\n={\rm Lie}(N)$.

Now, $\g$ acts on $\O(N)$ preserving the subspaces $V_n$ ($\n={\rm Lie}(N)$ acts by left invariant vector fields and $\mathfrak{k}$ by the adjoint action). As we have shown in the proof of Corollary \ref{fai}, $\n$ acts faithfully on $V_n$ for some $n$. We claim that this $V_n$ is, in fact, a faithful representation of the whole $\g$, which implies the theorem. Indeed, let $\a\subset \g$ be the ideal of elements acting by zero on $V_n$, and let $\overline{\a}$ 
be the projection of $\a$ to $\mathfrak{k}$ (an ideal in $\mathfrak{k}$). 
Since $\n$ acts faithfully on $V_n$, we have $\a\cap \n=0$. Given $a\in \a$, we have $a=\overline{a}+b$ where $\overline{a}\in \overline{\a}$ is the projection of $a$ and $b\in \n$. For $x\in \n$ we have $[a,x]\in \a\cap \n=0$. Thus $[\overline{a},x]=-[b,x]$. Hence the operator $x\mapsto [\overline{a},x]$ on $\n$ is nilpotent. So 
$\overline{\a}$ acts on $\n$ by nilpotent operators. Since $K$ is reductive and $\overline{\a}\subset \mathfrak{k}$ is an ideal, this means that $\overline{\a}$ acts on $\n$ by zero. Thus $\overline{\a}=0$ and $\a\subset \n$. Hence $\a=0$.   
\end{proof} 

\section{\bf Borel subgroups and the flag manifold of a complex reductive Lie group} 

\subsection{Borel subgroups and subalgebras} Let $G$ be a connected complex reductive Lie group, $\g={\rm Lie}(G)$. Fix a Cartan subalgebra $\h\subset \g$ with a system of simple positive roots $\Pi$, and consider the corresponding triangular decomposition $\g=\mathfrak{n}_-\oplus\h\oplus \mathfrak{n}_+$, where 
$\n_+$ is spanned by positive root elements and $\n_-$ by negative root elements. 
Let $H$ be the maximal torus in $G$ corresponding to $\h$, $N_+$ the unipotent subgroup of $G$ corresponding to $\n_+$, and $B_+=HN_+$ 
the solvable subgroup with ${\rm Lie}(B_+)=\b_+:=\h\oplus \n_+$; these are all closed Lie subgroups. 

\begin{definition} A {\bf Borel subalgebra} of $\g$ is a Lie subalgebra conjugate to $\b_+$. A {\bf Borel subgroup} of $G$ is a Lie subgroup conjugate to $B_+$. 
\end{definition}  

Since all pairs $(\h,\Pi)$ are conjugate, this definition does not depend on the choice of $(\h,\Pi)$. 

\begin{lemma}\label{ownnorm} $B_+$ is its own normalizer in $G$. 
\end{lemma} 

\begin{proof} Let $\gamma\in G$ be such that $\gamma B_+ \gamma^{-1}=B_+$. 
Let $H'={\rm Ad}\gamma(H)\subset B_+$. It is easy to show that we can conjugate 
$H'$ back into $H$ inside $B_+$, so we may assume without loss of generality 
that $H'=H$. Then $\gamma\in N(H)$, and it preserves positive roots. 
Hence the image of $\gamma$ in $W$ is $1$, so $\gamma\in H\subset B_+$, as claimed.  
\end{proof} 

\subsection{The flag manifold of a connected complex reductive group}\label{flam} 
Thus the set of all Borel subalgebras (or subgroups) in $G$ 
is the homogeneous space $G/B_+$, a complex manifold. 
It is called the {\bf flag manifold} of $G$. Note that it only depends 
on the semisimple part $\g_{ss}\subset \g$ and does not depend on the choice of 
the Cartan subalgebra and triangular decomposition. 

Let $G^c\subset G$ be the compact form of $G$, with Lie algebra $\g^c\subset \g$. 
It is easy to see that $\g^c+\b_+=\g$. Thus the $G^c$-orbit $G^c\cdot 1$ of $1\in G/B_+$ contains a neighborhood of $1$ in $G/B_+$. Hence the same holds for any point of this orbit, i.e., $G^c\cdot 1\subset G/B_+$ is an open subset. But it is also compact, since $G^c$ is compact, hence closed. As $G/B_+$ is connected, we get that $G^c\cdot 1=G/B_+$, i.e., $G^c$ acts transitively 
on $G/B_+$.   

Also the Cartan involution $\omega$ maps positive root elements to negative ones, so $G^c\cap B_+\subset w_0(B_+)\cap B_+=H$. Thus $G^c\cap B_+=H^c$, a maximal torus in $G^c$. So we get 

\begin{proposition}\label{iwas} We have $G/B_+=G^c/H^c$. In particular, $G/B_+$ 
is a compact complex manifold of dimension $|R_+|=\frac{1}{2}(\dim \g-{\rm rank}\g)$. 
\end{proposition}  

\begin{example} 1. For $G=SL_2$ we have $G/B_+=SU(2)/U(1)=S^2$, the Riemann sphere.    

2. For $G=GL_n$ we have $G/B_+=U(n)/U(1)^n=\mathcal F_n$, the set of flags in $\Bbb C^n$ that we considered in Subsection \ref{flman}. 
\end{example} 

Another realization of the flag manifold is  
as the $G$-orbit of the line spanned by the highest weight vector 
in an irreducible representation with a regular highest weight. Namely, let $\lambda\in P_+$ be a dominant integral weight with $\lambda(h_i)\ge 1$ for all $i$ (i.e., $\lambda=\mu+\rho$ for $\mu\in P_+$). Let $L_\lambda$ be the corresponding irreducible representation with highest weight vector $v_\lambda$. We have $\b_+\cdot \Bbb Cv_\lambda=\Bbb Cv_\lambda$, but $e_{-\alpha} v_\lambda\ne 0$ for any $\alpha\in R_+$ (as $e_\alpha e_{-\alpha}v_\lambda=h_\alpha v_\lambda=(\lambda,\alpha^\vee)v_\lambda$, and $(\lambda,\alpha^\vee)>0$). Moreover, these vectors have different weights, so are linearly independent. Thus $\b_+$ is the stabilizer of $\Bbb Cv_\lambda$ in $\g$. Hence any $g\in G$ which preserves $\Bbb Cv_\lambda$ 
belongs to the normalizer of $\b_+$ (or, equivalently, $B_+$), i.e., $g\in B_+$. 
Thus $\mathcal O:=G\cdot \Bbb Cv_\lambda\subset \Bbb P L_\lambda$ is 
identified with $G/B_+$. This shows that $\mathcal O$ is compact, hence closed, i.e., $\mathcal O=G/B_+$ is a smooth complex projective variety.  

Let $A=\exp(i\h^c)\subset H$, $K=G^c$, $N=N_+$. Proposition \ref{iwas} immediately implies

\begin{corollary}\label{iwas2} (The {\bf Iwasawa decomposition} of $G$)  The multiplication map 
$K\times A\times N\to G$ is a diffeomorphism. In particular, we have 
$G=KAN$.
\end{corollary}  

A similar theorem holds for {\it real} reductive groups (Theorem \ref{iwas1}).  

\subsection{The Borel fixed point theorem} Let $V$ be a finite dimensional representation of a finite dimensional $\Bbb C$-Lie algebra $\a$, and $X\subset \Bbb PV$ be a subset. We will say that $X$ is $\a$-invariant (or fixed by $\a$) if it is $\exp(\a)$-invariant. 

\begin{theorem} Let $\a$ be a solvable Lie algebra over $\Bbb C$, 
$V$ a finite dimensional $\a$-module. Let $X\subset \Bbb P V$ be a nonempty closed $\a$-invariant subset. 
Then there exists $x\in X$ fixed by $\a$.  
\end{theorem} 

\begin{proof} The proof is by induction in $n=\dim \a$. 
The base $n=0$ is trivial, so we only need to justify the induction step. 
Since $\a$ is solvable, it has an ideal $\a'$ of codimension $1$. 
By the induction assumption, $Y:=X^{\a'}$ (the set of $\exp(\a')$-fixed points in $X$)  is a nonempty closed subset 
of $X$, so it suffices to show that the 1-dimensional Lie algebra 
$\a/\a'$ has a fixed point on $Y$. Thus it suffices to prove the theorem for $n=1$. 

So let $\a$ be 1-dimensional, spanned by $a\in \a$. We can choose the normalization of $a$ so that distinct eigenvalues of $a$ on $V$ have different real parts. 
Fix $x_0\in X$ and consider the curve $e^{ta}x_0$ for $t\in \Bbb R$. 
It is easy to see that there exists $x:=\lim_{t\to \infty}e^{ta}x_0\in \Bbb PV$. 
Then $x\in X$ as $X$ is closed, and $x$ is fixed by $\a$, as desired. 
\end{proof} 

Note that for $X=\Bbb PV$, the Borel fixed point theorem reduces to Lie's theorem (Theorem \ref{liethm}). 

\subsection{Parabolic and Levi subalgebras} 

A Lie subalgebra $\p\supset \b$ of a reductive Lie algebra $\g$ containing some Borel subalgebra $\b\subset \g$ 
is called a {\bf parabolic subalgebra} of $\g$. The corresponding connected 
Lie subgroup $P\subset G$ is called a {\bf parabolic subgroup}. It is easy to see that 
$P\subset G$ is necessarily closed (check it!). 

\begin{exercise}\label{par} Show that parabolic subalgebras $\p$ containing $\b_+$ 
are in bijection with subsets $S\subset \Pi$ of the set of simple roots
of $\b_+$, namely, $\p$ is sent to the set $S_\p$ of $i\in \Pi$ such that
$f_i\in \p$, and $S$ is sent to the Lie subalgebra $\p_S$ of $\g$ generated by $\b_+$ and $f_i,i\in S$.   
\end{exercise}  

Let $P\subset G$ be a parabolic subgroup with Lie algebra $\p$. 
Let $U$ be the largest connected normal Lie subgroup of $P$ consisting of unipotent elements (the {\bf unipotent radical} of $P$), and let $\u={\rm Lie}(U)$; for instance, if $\p\supset \b_+$ then $\u$ is the Lie subalgebra spanned by $e_\alpha$ such that $e_{-\alpha}\notin \p$. It is easy to see that there exists a (non-unique) Lie subalgebra $\mathfrak l\subset \p$ complementary to $\u$, which therefore projects isomorphically to $\p/\u$; indeed, if $\p\supset \b_+$ then we can take $\mathfrak l$ to be the Lie subalgebra spanned by $\h$ and $e_\alpha,e_{-\alpha}$ where $\alpha$ runs through positive roots for which $e_{-\alpha}\in \p$. Such a subalgebra $\mathfrak l$ is called a {\bf Levi subalgebra} of $\p$, and we have $\p=\mathfrak l\ltimes \u$, which is $\mathfrak l\oplus \u$ as a vector space. 

Then $U=\exp(\u)$. The quotient $P/U$ is a reductive group with Lie algebra $\p/\u$. 
A {\bf Levi subgroup} of $P$ is a subgroup $L$ in $P$ such that $\mathfrak l:={\rm Lie}(L)$ 
is a Levi subalgebra of $\p$; equivalently, $L$ projects isomorphically to $P/U$, so 
we have $P=L\ltimes U$, written shortly as $P=LU$. It is not difficult to show that all Levi subgroups 
of $P$ (or, equivalently, all Levi subalgebras of $\p$) are conjugate by the action of $U$ (check it!). 

For example, $L$ is a maximal torus if and only if $P$ is a Borel subgroup, and $L=G$ if and only if $P=G$.  

\begin{example} Let $n=n_1+...+n_k$ where $n_i$ are positive integers. 
Then the subgroup $P$ of block upper triangular matrices with diagonal 
blocks of size $n_1,...,n_k$ is a parabolic subgroup of $GL_n(\Bbb C)$, 
and the subgroup $L$ of block diagonal matrices in $P$ is a Levi subgroup. 
The unipotent radical $U$ of $P$ is the subgroup of block upper triangular matrices with identity matrices on the diagonal. 
\end{example} 

\subsection{Maximal solvable and maximal nilpotent subalgebras}

Note that $\b_+$ is a maximal solvable subalgebra of $\g$; indeed, any bigger parabolic subalgebra contains a negative root vector, hence the corresponding root $\mathfrak{sl}_2$-subalgebra, so it is not solvable. Moreover, 
$B_+$ is a maximal solvable subgroup of $G$: if $P\supset B_+$ then 
some element $g\in P$ does not normalize $\b_+$, so 
${\rm Lie}(P)$ has to be larger than $\b_+$, hence not solvable. 
Thus any Borel subalgebra (subgroup) is a maximal solvable one. 
It turns out that the converse also holds. 

\begin{proposition}\label{solsu} Any solvable Lie subalgebra of $\g$ (respectively, connected solvable subgroup of $G$)
is contained in a Borel subalgebra (subgroup).   
\end{proposition} 

\begin{proof} Let $\a\subset \g$ be a solvable Lie subalgebra. 
By the Borel fixed point theorem, $\a$ has a fixed point $\b\in G/B_+$. 
Thus $\a$ normalizes $\b$. Hence $
\a\subset \b$, as claimed. 
\end{proof} 

\begin{corollary} Any element of $\g$ is contained in a Borel subalgebra $\b\subset \g$.   
\end{corollary} 

Let us say that a Lie subalgebra $\a\subset \g$ is {\bf a nilpotent subalgebra}
if every element of $\a$ acts nilpotently in every finite dimensional representation of $\g$.
Equivalently, $\a\subset[\g,\g]$ and ${\rm ad}x$ is nilpotent for every $x\in\a$.
Note that this is a stronger condition than just being nilpotent as a Lie algebra; for example, a nonzero Cartan subalgebra is a nilpotent Lie algebra (since it is abelian) but it is not a nilpotent subalgebra of $\g$. 

\begin{corollary} Any nilpotent subalgebra of $\g$ is conjugate to a Lie subalgebra of $\n_+$. Thus $\n_+$ is a maximal nilpotent subalgebra of $\g$, and any maximal nilpotent subalgebra of $\g$ is conjugate to $\n_+$. 
\end{corollary} 

\begin{proof} By Proposition \ref{solsu} there is $g\in G$ such that ${\rm Ad}_g\a\subset \b_+$, 
but since $\a$ is nilpotent we actually have ${\rm Ad}_g\a\subset \n_+$. 
\end{proof} 

A similar result holds for unipotent subgroups (i.e., connected Lie subgroups consisting of unipotent elements), with the same proof: 

\begin{corollary} Any unipotent subgroup of $G$ is conjugate to a (closed) Lie subgroup of $N_+$. Thus $N_+$ is a maximal unipotent subgroup of $G$, and any maximal unipotent subgroup of $G$ is conjugate to $N_+$. 
\end{corollary} 

We also have 

\begin{proposition}
The normalizer of $\n_+$ and $N_+$ in $G$ is $B_+$. Thus 
every maximal nilpotent subalgebra (unipotent subgroup) is contained in a unique Borel subgroup. Hence such subalgebras (subgroups) are parametrized by the flag manifold $G/B_+$. 
\end{proposition} 

\begin{proof} Clearly $B_+$ is contained in the normalizer of $N_+$, 
so this normalizer is a parabolic subgroup. We have seen that 
such a subgroup, if larger than $B_+$, 
must have a Lie algebra larger than $\b_+$, so it must be  $\p_S$ for some $S\ne \emptyset$, hence contains 
some root $\mathfrak{sl}_2$-subalgebra. But the group corresponding to such a subalgebra does not normalize $\n _+$, a contradiction.  
\end{proof} 

\subsection{Iwasawa decomposition of a real semisimple linear group}  

Let $G_\theta=K^cP_\theta$ 
be the polar decomposition of a real form of a complex semisimple group $G$,  
$\g_\theta=\mathfrak{k}^c\oplus \p_\theta$ the additive version, 
$\a\subset \p_\theta$ a maximal abelian subspace. Let $A=\exp(\a)\subset P_\theta$ be the corresponding abelian subgroup of $G_\theta$. Pick a generic element $a\in \a$. Let $\z=\g_\theta^a$ be the centralizer of $a$ in $\g_\theta$ and let $\n_{a,\pm}$ be the (nilpotent) Lie subalgebras of $\g_\theta$ spanned by eigenvectors of ${\rm ad}a$ with positive, respectively negative eigenvalues, so that $\g_\theta=\n_{a-}\oplus \z\oplus \n_{a+}$. Let $N_{a\pm}=\exp(\n_{a\pm})$.

The following theorem is a generalization of Proposition \ref{iwas2}.

\begin{theorem}\label{iwas1} (Iwasawa decomposition)
The multiplication map $K^c\times A\times N_{a+}\to G_\theta$ 
is a diffeomorphism. 
\end{theorem}

Theorem \ref{iwas1} is proved in the following exercise. 

\begin{exercise}  (i) Let $\m=\z\cap \mathfrak{k}^c$. Show that $\z=\m\oplus \a$ 
(use Proposition \ref{maxab}(ii)). 
  
(ii) Given $x\in \p_\theta$, write $x=x_-+x_0+x_+$, $x_\pm\in \n_{a\pm}$, $x_0\in \z$. 
Show that $\theta(x_\pm)=-x_\mp$, $\theta(x_0)=-x_0$. 
Deduce the {\bf additive Iwasawa decomposition}
$\g_\theta=\mathfrak{k}^c\oplus \a\oplus \n_{a+}$ (write $x$ as $(x_--x_+)+x_0+2x_+$). 

(iii) Show that $\z\oplus \n_{a+}=\m\oplus \a\oplus \n_{\a+}$ 
is a parabolic subalgebra in $\g_\theta$ with Levi subalgebra $\z$ (i.e., their complexifications are a parabolic subalgebra in $\g$ and its Levi subalgebra) 
and its unipotent radical is $\n_{a+}$. 

(iv) Let $M$ be the centralizer of $a$ in $K^c$. Show that $\Bbb P:=MAN_{a+}$ is a subgroup of $G_\theta$ and $X:=G_\theta/\Bbb P$ is a compact homogeneous space. 

(v) Show that $K^c$ acts transitively on $X$, and $X\cong K^c/M$ as a homogeneous space for $K^c$ (generalize the argument in Subsection \ref{flam}). Deduce Theorem \ref{iwas1}. 
\end{exercise} 

\subsection{The Bruhat decomposition} 

Let $G$ be a connected complex reductive group, $H\subset G$ a maximal torus, 
$B=B_+\supset H$ a Borel subgroup. The Bruhat decomposition 
is the decomposition of $G$ into double cosets of $B$. 

Let $N(H)$ be the normalizer of $H$ in $G$ and $W=N(H)/H$ be the Weyl group. 
Given $w\in W$, let $\widetilde w$ be a lift of $w$ to $N(H)$ and consider the double coset $B\widetilde w B\subset G$. Since any two lifts of $w$ differ by an element of $H$ which is contained in $B$, the set $B\widetilde w B$ does not depend on the choice of $\widetilde w$, so we will denote it by $BwB$. 

\begin{proposition}\label{disj} The double cosets $BwB$, $w\in W$ are disjoint.
\end{proposition} 

\begin{proof} Let $w_1,w_2\in N(H)$ be such that $Bw_1B=Bw_2B$. 
Then there exist elements $b_1,b_2\in B$ such that 
$b_1w_1=w_2b_2$. Let us apply this identity to a highest weight vector $v_\lambda$ of an irreducible representation $L_\lambda$ of $G$, where
$\lambda\in P_+$ is regular. We have $w_2b_2v_\lambda=Cv_{w_2\lambda}$ 
for some $C\in \Bbb C^\times$, where $v_{w_2\lambda}$ is an extremal vector of weight $w_2\lambda$. On the other hand, 
$b_1w_1v_\lambda=C'b_1v_{w_1\lambda}$ for some $C'\in \Bbb C^\times$. 
Thus $Cv_{w_2\lambda}=C'b_1v_{w_1\lambda}$.
But $b_1v_{w_1\lambda}$ equals $C''v_{w_1\lambda}$ 
plus terms of weight $>w_1\lambda$, where $C''\in \Bbb C^\times$. 
It follows that $w_1\lambda=w_2\lambda$, hence $w_1=w_2h$, $h\in H$.
\end{proof} 

\begin{theorem}\label{bd} (Bruhat decomposition)  The
union of the double cosets $BwB$, $w\in W$ is the entire group $G$.
Thus they define a partition of $G$ into double cosets of $B$.  
\end{theorem} 

Theorem \ref{bd} can be reformulated as a classification of $B$-orbits on the flag manifold $G/B$. Namely, given $w\in W$, the set $BwB/B$ 
is an orbit of $B$ on $G/B$, which we will denote by $C_w$. By Theorem 
\ref{disj}, $C_w$ are disjoint, and Theorem \ref{bd} is equivalent to 

\begin{theorem}\label{bd1} (Schubert decomposition)  $C_w,w\in W$ give the partition of $G/B$ into $B$-orbits. 
\end{theorem} 

The sets $BwB$ are called {\bf Bruhat cells} and the sets $C_w$ are called {\bf Schubert cells.}\footnote{We note that Bruhat cells, unlike Schubert cells, are not literally cells in the topological sense -- they are not homeomorphic to an affine space, but are homeomorphic to the product of an affine space and a torus.} 

Note that for type $A_{n-1}$ ($G=SL_n(\Bbb C)$ or its quotient), we have already proved Theorem \ref{bd1} in Subsection \ref{flman}, where we decomposed the flag manifold $\mathcal F_n$ into Schubert cells labeled by permutations. 

A proof of Theorem \ref{bd1} can be found, for example, in the textbook \cite{CG}. 
It is also sketched in the following exercise. 

\begin{exercise} 
(i) Let $B=B_+$ and $w\in W$. Consider the multiplication map 
$\mu_{i,w}: Bs_iB\times_B C_w\to G/B$. 
Show that if $\ell(s_iw)=\ell(w)+1$ 
then $\mu_{i,w}$ is an isomorphism onto $C_{s_iw}$, 
while if $\ell(s_iw)=\ell(w)-1$ then the image 
of $\mu_{i,w}$ consists of $C_w$ and $C_{s_iw}$.  

{\bf Hint:} Reduce to the $SL_2$-case.  

(ii) For $i\in \Pi$ let $P_i$ be the {\bf minimal parabolic} 
subgroup of $G$ generated by $B$ and the 1-parameter subgroup $\exp(tf_i)$. Show that $P_i/B=C_{s_i}\cup C_1\cong \Bbb C\Bbb P^1\subset G/B$
(where $C_1$ is a point and $C_{s_i}\cong \Bbb C$). 

(iii) Let $w=s_{i_1}...s_{i_l}$ be a reduced decomposition of 
$w\in W$ (so $l=\ell(w)$); denote this decomposition by $\overline{w}$. 
The product $\prod_{k=1}^l P_{i_k}$ carries 
a free action of $B^l$ via 
$$
(b_1,...,b_l)\circ (p_1,...,p_l):=(p_1b_1^{-1},b_1p_2b_2^{-1},...,b_{l-1}p_lb_l^{-1}).
$$
Define the {\bf Bott-Samelson variety} $X_{\overline w}:=(\prod_{k=1}^l P_{i_k})/B^l$. Use (ii) to show that if $\overline w=s_i\overline u$ then 
$X_{\overline w}$ fibers over $\Bbb C\Bbb P^1$ with fiber $X_{\overline u}$.
Deduce that $X_{\overline w}$ is a smooth projective variety of dimension $\ell(w)$. 

(iv) Define the {\bf Bott-Samelson map} 
$$
\mu_{\overline w}: X_{\overline w}\to G/B 
$$
given by multiplication. Use (i) to show that the image of $\mu_{\overline w}$ 
is the {\bf Schubert variety} $\overline C_w$, the closure of $C_w$ 
in $G/B$. Moreover, show that $\overline{C_w}\setminus C_w$ is the union of $C_u$ over some $u\in W$ with $\ell(u)<\ell(w)$. 

(v) Apply (iv) to the maximal element $w=w_0\in W$. In this case, show that 
$\mu_{\overline w}$ is surjective, and deduce Theorem \ref{bd1}.  
\end{exercise}

Let us derive some corollaries of Theorem \ref{bd1}.

\begin{corollary} (i) Any pair of Borel subgroups of $G$
is conjugate to the pair $(B,w(B))$
 for a unique $w\in W$. In particular, 
 any two Borel subgroups of $G$ share a maximal torus.
 
 (ii) The cell $C_w$ is isomorphic to $\Bbb C^{\ell(w)}$.  
 \end{corollary} 
 
 \begin{proof} (i) Let $(B_1,B_2)$ be a pair of Borel subgroups in $G$. 
 Then we can conjugate $B_1$ to $B$, and $B_2$ will be conjugated 
 to some Borel subgroup $B_3$. This subgroup is conjugate to $B$, i.e., is of the form $gBg^{-1}$ for some $g\in G$. By Bruhat decomposition, we can 
 write $g$ as $g=b_1\widetilde w b_2$, $b_1,b_2\in B$, $\widetilde w\in N(H)$. So conjugating by $b_1^{-1}$, we will bring our pair to the required form $(B,w(B))$, where $w$ is the image of $\widetilde w$ in $W$. Uniqueness follows from Proposition \ref{disj}.  
 
 (ii) By (i) we have $C_w\cong B/(B\cap w(B))$. 
 Since $B=NH$, where $N=[B,B]$ and $B\cap w(B)\supset H$, 
 we get $C_w=N/(N\cap w(B))=N/(N\cap w(N))$. This is a complex affine space 
whose dimension is the number of positive roots mapped to negative roots 
by $w$, i.e., $\ell(w)$.   
 \end{proof} 

\begin{corollary} The Poincar\'e polynomial of the flag manifold $G/B$ is 
$\sum_{i\ge 0}b_{2i}(G/B)q^i=\sum_{w\in W}q^{\ell(w)}.$
\end{corollary} 

\begin{remark} Similarly to the type $A$ case, one can show that 
this polynomial can also be written as $\prod_{i=1}^r [m_i+1]_q$, where $m_i$ are the exponents of $G$, but we will not give a proof of this identity.   
\end{remark}

\end{document}